\documentclass[a4paper]{amsart}

\usepackage{amsmath, amssymb, amsthm}
\usepackage[margin=3cm]{geometry}
\usepackage{mathtools}
\usepackage{color}
\usepackage{comment}
\usepackage{amsfonts, latexsym}
\usepackage{extpfeil}
\usepackage{tikz-cd}
\usepackage{mathrsfs} %\mathscrのため
\usepackage{array, booktabs, float}
\usepackage{hyperref}
\usepackage{adjustbox}

\usepackage[all]{xy}
\xyoption{line}
\xyoption{arrow}
\SelectTips{cm}{}

\usepackage{xcolor}

\theoremstyle{plain} %見出し太字, 本文斜体, 上下スペース
\newtheorem{theorem}{Theorem}[subsection]
\newtheorem{proposition}[theorem]{Proposition}
\newtheorem{lemma}[theorem]{Lemma}
\newtheorem{corollary}[theorem]{Corollary}

\theoremstyle{definition} %見出し太字, 本文普通, 上下スペース
\newtheorem{definition}[theorem]{Definintion}

\theoremstyle{remark} %見出し斜体, 本文普通
\newtheorem{remark}[theorem]{Remark}
\newtheorem{example}[theorem]{Example}

\numberwithin{subsection}{section}
\numberwithin{equation}{subsection} %数式番号

\newcommand{\gMod}[1]{{#1}\textup{-gMod}}
\newcommand{\gmod}[1]{{#1}\textup{-gmod}}
\newcommand{\gproj}[1]{{#1}\textup{-gproj}}
\newcommand{\quantum}[1]{U_q^{#1}(\mathfrak{g})}
\newcommand{\qcoordinate}{A_q(\mathfrak{n})}
\newcommand{\catquantum}[1]{\mathcal{U}_q({#1})} 
\newcommand{\dotcatquantum}[1]{\dot{\mathcal{U}}_q({#1})}
\newcommand{\renormalizedR}[2]{{\mathsf{R}}_{{#1}, {#2}}^{\mathrm{norm}}}
\newcommand{\fpdgMod}[1]{{#1}\textup{-gMod}_{\mathrm{f.p.d.}}}

\DeclareMathOperator{\id}{id}
\DeclareMathOperator{\Id}{Id}
\DeclareMathOperator{\height}{ht}
\DeclareMathOperator{\Hom}{Hom}
\DeclareMathOperator{\HOM}{HOM}
\DeclareMathOperator{\End}{End}
\DeclareMathOperator{\END}{END}
\DeclareMathOperator{\EXT}{EXT}
\DeclareMathOperator{\Aut}{Aut}

\DeclareMathOperator{\Ind}{Ind}
\DeclareMathOperator{\Res}{Res}
\DeclareMathOperator{\Cok}{Cok}
\DeclareMathOperator{\Ext}{Ext}
\DeclareMathOperator{\ad}{ad}
\DeclareMathOperator{\Ker}{Ker}
\DeclareMathOperator{\qdim}{qdim}
\DeclareMathOperator{\wt}{wt}

\newcommand{\ucross}[2]{
 \xybox{
    (-4,6)*{}; (4,6)*{};
    (-4,-4)*{};(4,4)*{} **\crv{(-4,-1) & (4,1)}?(1)*\dir{>};
    (4,-4)*{};(-4,4)*{} **\crv{(4,-1) & (-4,1)}?(1)*\dir{>};
    (-4,-6)*{\scriptstyle #1};
     (4,-6)*{\scriptstyle #2};
     (-8,0)*{};(8,0)*{};
     }}
\newcommand{\dcross}[2]{
 \xybox{ 
    (-4,6)*{}; (4,6)*{};
        (-4,-4)*{};(4,4)*{} **\crv{(-4,-1) & (4,1)}?(0)*\dir{<} ;
    (4,-4)*{};(-4,4)*{} **\crv{(4,-1) & (-4,1)}?(0)*\dir{<};
    (-4,-6)*{\scriptstyle #1};
     (4,-6)*{\scriptstyle #2};
     (-8,0)*{};(8,0)*{};
     }}     
\newcommand{\lcross}[2]{
 \xybox{
    (-4,6)*{}; (4,6)*{};
    (-4,-4)*{};(4,4)*{} **\crv{(-4,-1) & (4,1)}?(0)*\dir{<} ;
    (4,-4)*{};(-4,4)*{} **\crv{(4,-1) & (-4,1)}?(1)*\dir{>};
    (-4,-6)*{\scriptstyle #1};
     (4,-6)*{\scriptstyle #2};
     (-8,0)*{};(8,0)*{};
     }}
\newcommand{\rcross}[2]{
 \xybox{
    (-4,6)*{}; (4,6)*{};
    (-4,-4)*{};(4,4)*{} **\crv{(-4,-1) & (4,1)}?(1)*\dir{>};
    (4,-4)*{};(-4,4)*{} **\crv{(4,-1) & (-4,1)}?(0)*\dir{<};
    (-4,-6)*{\scriptstyle #1};
     (4,-6)*{\scriptstyle #2};
     (-8,0)*{};(8,0)*{};
     }}
\newcommand{\rcap}[1]{\xybox{
  (-6,0)*{};
  (6,0)*{};
  (-4,-.7)*{}="t1";
  (4,-.7)*{}="t2";
  "t1";"t2" **\crv{(-4,5.3) & (4,5.3)}; ?(1)*\dir{>}
  ?(.5)*\dir{}+(0,2)*{\scriptstyle{#1}}; 
}}
\newcommand{\lcap}[1]{\xybox{
  (-6,0)*{};
  (6,0)*{};
  (-4,-.7)*{}="t1";
  (4,-.7)*{}="t2";
  "t2";"t1" **\crv{(4,5.3) & (-4,5.3)}; ?(1)*\dir{>} ?(.5)*\dir{}+(0,2)*{\scriptstyle{#1}};
}}
\newcommand{\rcup}[1]{\xybox{
  (-6,0)*{};
  (6,0)*{};
  (-4,-.7)*{}="t1";
  (4,-.7)*{}="t2";
  "t1";"t2" **\crv{(-4,-6.7) & (4,-6.7)}; ?(1)*\dir{>}
   ?(.5)*\dir{}+(0,-2)*{\scriptstyle{#1}}; 
}}
\newcommand{\lcup}[1]{\xybox{
  (-6,0)*{};
  (6,0)*{};
  (-4,-.7)*{}="t1";
  (4,-.7)*{}="t2";
  "t2";"t1" **\crv{(4,-6.7) & (-4,-6.7)}; ?(1)*\dir{>}
  ?(.5)*\dir{}+(0,-2)*{\scriptstyle{#1}};
}}
\newcommand{\slineu}[1]{\xybox{
  (-2,0)*{};
  (2,0)*{};
  (0,2)*{};
  (0,0)*{}; (0,-8)*{} **\dir{-}; ?(0)*\dir{<};
  (0,-10)*{\scriptstyle  #1}; 
}}
\newcommand{\slined}[1]{\xybox{
  (-2,0)*{};
  (2,0)*{};
  (0,2)*{}; 
  (0,0)*{}; (0,-8)*{} **\dir{-}; ?(1)*\dir{>};
  (0,-10)*{\scriptstyle  #1};
}}
\newcommand{\sline}[1]{\xybox{
  (-2,0)*{};
  (2,0)*{};
  (0,2)*{}; 
  (0,0)*{}; (0,-8)*{} **\dir{-}; 
  (0,-10)*{\scriptstyle  #1};
}}
\newcommand{\sdotu}[1]{\xybox{
  (-2,0)*{};
  (2,0)*{};
  (0,2)*{};
  (0,0)*{}; (0,-8)*{} **\dir{-}?(.5)*{\scriptstyle\bullet} ?(0)*\dir{<};
  (0,-10)*{\scriptstyle  #1};
}}
\newcommand{\sdotd}[1]{\xybox{
  (-2,0)*{};
  (2,0)*{};
  (0,2)*{};
  (0,0)*{}; (0,-8)*{} **\dir{-}?(.5)*{\scriptstyle\bullet} ?(1)*\dir{>};
  (0,-10)*{\scriptstyle  #1};
}}

\tikzset{cross/.style={preaction={-,draw=white,line width=6pt}}}

\title{A diagrammatic approach to reflection functors}
\author{Haruto Murata}
\address{Graduate School of Mathematical Sciences, the University of Tokyo, 3-8-1 Komaba, Meguro-ku, Tokyo 153-8914, Japan.}
\email{muraharu@ms.u-tokyo.ac.jp}

\makeindex

\begin{document}
\begin{abstract}
We construct reflection functors for quiver Hecke algebras associated with arbitrary symmetrizable Kac-Moody algebras, 
from a higher representation-theoretic viewpoint. 
These functors provide a categorification of Lusztig's braid group action on the quantum group. 
Similar functors were recently constructed independently by Kashiwara--Kim--Oh--Park via a different approach. 
Moreover, we prove that our reflection functors satisfy the braid relations as natural isomorphisms. 
\end{abstract}

\maketitle

\tableofcontents

\section{Introduction}

\subsection{Categorification of quantum groups and braid group symmetry}

Categorical representations of quantum groups play a foundational role in modern representation theory. 
At the heart of this subject is the categorified quantum group (or 2-quantum group) $\mathcal{U}_q(\mathfrak{g})$ \cite{rouquier20082kacmoodyalgebras,MR2628852} associated with a symmetrizable generalized Cartan matrix $\mathsf{A} = (a_{i,j})_{i,j \in I}$. 
In this framework, categorifying the braid group symmetries constitutes a fundamental theme. 

To make this precise, recall that a categorical representation of the quantum group $U_q(\mathfrak{g})$ (or a $\mathcal{U}_q(\mathfrak{g})$-module) 
is an additive category $\mathcal{V}$ equipped with:
\begin{itemize}
\item a decomposition $\mathcal{V} = \bigoplus_{\lambda} \mathcal{V}_{\lambda}$ with respect to integral weights,
\item functors $E_i \colon \mathcal{V}_{\lambda} \to \mathcal{V}_{\lambda+\alpha_i}, \ F_i \colon \mathcal{V}_{\lambda} \to \mathcal{V}_{\lambda-\alpha_i} \ (i \in I)$, and  
\item specific natural transformations between compositions of these functors subject to several defining relations.
\end{itemize}
When $\mathcal{V}$ is integrable, 
Chuang--Rouquier \cite{MR2373155} established explicit triangulated equivalences between homotopy categories 
\[
K^b(\mathcal{V}_{\lambda}) \simeq K^b(\mathcal{V}_{s_i \lambda}), 
\]
categorifying Lusztig's braid group symmetries on integrable modules of $U_q(\mathfrak{g})$ \cite{MR2759715}. 
This groundbreaking discovery enabled them to resolve Brou\'{e}'s abelian defect group conjecture for symmetric groups. 

The quantum group $U_q(\mathfrak{g})$ itself possesses a braid group symmetry that is universal among the braid group actions on integrable modules \cite{MR2759715}. 
Categorifying this symmetry on the homotopy category of $\mathcal{U}_q(\mathfrak{g})$ is a major open problem, 
despite some recent progress \cite{MR4285453, MR4732757}.

\subsection{Quiver Hecke algebras and reflection functors}

The braid group symmetry also governs the internal structure of the negative half $U_q^-(\mathfrak{g})$. 
Specifically, each simple reflection $T_i$ of the braid group symmetry restricts to an isomorphism of subalgebras
\begin{equation} \label{eq:Ti}
T_i \colon {}_i U \to U_i, 
\end{equation}
where ${}_i U = U_q^-(\mathfrak{g}) \cap T_i^{-1} U_q^-(\mathfrak{g}), U_i = U_q^-(\mathfrak{g}) \cap T_i U_q(\mathfrak{g})$. 
These isomorphisms play a crucial role in constructing the PBW bases, 
which are used to characterize the canonical bases, or global bases \cite{MR1035415, MR2066942, MR2914878}. 

The negative half of the categorified quantum group $\mathcal{U}_q(\mathfrak{g})$ is essentially captured by the quiver Hecke algebra, 
also referred to as Khovanov--Lauda--Rouquier algebra \cite{rouquier20082kacmoodyalgebras,MR2525917,MR2763732}. 
Since its advent, it has become a central object in representation theory through its connections with Iwahori--Hecke algebras \cite{MR2551762}, quantum affine algebras \cite{MR3352041} and cluster algebras \cite{MR3758148}. 
Let $\gMod{R}$ denote the category of finitely generated graded modules over the quiver Hecke algebra, 
which is a monoidal category via the convolution product. 
Its Grothendieck ring is isomorphic to $U_q^-(\mathfrak{g})$ (after base change). 

Inside the category $\gMod{R}$, there are monoidal subcategories $\gMod{{}_iR}$ and $\gMod{R_i}$ corresponding to subalgebras ${}_iU$ and $U_i$, respectively. 
Constructing an equivalence of monoidal categories 
\begin{equation} \label{eq:reflection}
S_i \colon \gMod{{}_iR} \xrightarrow{\sim} \gMod{R_i}
\end{equation}
that categorifies the isomorphism (\ref{eq:Ti}) is therefore a problem of fundamental importance. 
We call such an equivalence a reflection functor.

The main purpose of this paper is to construct reflection functors for arbitrary $\mathsf{A}$ over any field. 
Prior constructions were constrained by specific technical assumptions. 
For symmetric $\mathsf{A}$ and a base field of characteristic zero, 
Kato \cite{MR3165425, MR4216698} constructed reflection functors using constructible sheaves on the representation spaces of the quiver via the decomposition theorem. 
McNamara \cite{mcnamara2020representationtheorygeometricextension} extended this construction to positive characteristic for finite $ADE$ types using parity sheaves. 
However, these geometric methods face technical challenges for infinite types in positive characteristic,
making a purely algebraic approach desirable.    
On the algebraic side, for finite types (not necessarily symmetric), 
Kashiwara--Kim--Oh--Park \cite{MR4717658} established an equivalence between localized categories using $R$-matrices.
However, whether their functor restricts to an equivalence between the original categories as in (\ref{eq:reflection}) remains open. 

Our construction is entirely algebraic.  
In addition, we require neither homotopy categories nor localization procedures. 
We expect our perspective and explicit computations to be useful toward categorifying the full braid group symmetry on the homotopy category $K^b(\mathcal{U}_q(\mathfrak{g}))$. 

As an application, 
we show that our reflection functors provide an alternative realization of the standard modules of \cite{MR3205728, MR3165425, MR3694676, MR3612467, murata2025affinehighestweightstructures} in Section \ref{sec:stratification}. 
Furthermore, in \cite{murata2026quantumimaginaryschurweylduality}, we employ the reflection functors to study the imaginary strata of quiver Hecke algebras of affine types. 
In a subsequent paper, we will show that our reflection functor essentially coincides with Kato's geometric reflection functor when $\mathsf{A}$ is symmetric, 
thereby uncovering the hidden structure behind the geometric construction. 
It would be interesting to explore its implications for the geometry of quiver representation spaces. 

While this paper was in preparation, a preprint by Kashiwara--Kim--Oh--Park \cite{MR5045632} appeared, 
which independently constructs reflection functors in full generality.  
Our approach is conceptually and technically distinct from theirs: 
while their proof relies on $R$-matrices and the localization of monoidal categories to eventually deduce an equivalence between unlocalized categories,
ours requires no localization procedures at any stage.
As an advantage of our approach, we are able to prove that our reflection functors satisfy the braid relations as natural isomorphisms.  

In the remainder of the introduction, we outline our approach in more detail. 

\subsection{Characterization of the isomorphism} 

Fix $i \in I$. 
A standard approach to obtaining a monoidal functor $S_i \colon \gMod{{}_iR} \to \gMod{R_i}$ would be to find a presentation of $\gMod{{}_iR}$ as a monoidal category. 
However, this strategy is not feasible because, even at the decategorified level, the subalgebra ${}_iU$ lacks a concise presentation by generators and relations as an algebra. 
To circumvent this difficulty, 
we shift our perspective and view ${}_iU$ as a module over a suitable auxiliary algebra rather than as an algebra itself. 

The auxiliary algebra we use is the parabolic quantum group $U_q(\mathfrak{p}_i)$, 
defined as the subalgebra of $U_q(\mathfrak{g})$ generated by $e_i, f_j \ (j \in I)$ and $q^h \ (h \in \mathsf{P}^{\lor})$, 
where $\mathsf{P}^{\lor}$ is the coweight lattice. 
We prove by explicit formulas that ${}_iU$ has a right $U_q(\mathfrak{p}_i)$-module structure (Proposition \ref{prop:actionlift}).
Since $U_q^-(\mathfrak{g}) = {}_iU \oplus f_i U_q^-(\mathfrak{g})$ \cite{MR2759715}, 
we have an isomorphism of vector spaces
\[
{}_i U \simeq {}_iV(0) \coloneqq U_q(\mathfrak{p}_i) \bigg/ \left(e_iU_q(\mathfrak{p}_i) + f_iU_q(\mathfrak{p}_i) + \sum_{h \in \mathsf{P}^{\lor}} (q^h - 1)U_q(\mathfrak{p}_i)\right)
\]
via the canonical projection. 
In fact, we prove that this map is an isomorphism of right $U_q(\mathfrak{p}_i)$-modules, 
which provides a presentation of ${}_iU$ as a cyclic right $U_q(\mathfrak{p}_i)$-module. 

Furthermore, we show by explicit formulas that the other subalgebra $U_i$ also admits a right $U_q(\mathfrak{p}_i)$-module structure (Proposition \ref{prop:bimodule}).  
By the universal property of ${}_iU \simeq {}_iV(0)$, 
there exists a unique homomorphism of right $U_q(\mathfrak{p}_i)$-modules from ${}_iU$ to $U_i$ that sends the unit element $1$ to itself. 
Crucially, this homomorphism is precisely the isomorphism $T_i \colon {}_i U \to U_i$ induced by Lusztig's braid group symmetry. 
In other words, this construction provides a direct characterization of the isomorphism $T_i \colon {}_iU \to U_i$,
rather than defining it as the restriction of an automorphism of the entire algebra $U_q(\mathfrak{g})$. 

\subsection{Categorifying the procedure} \label{sec:categorifying}

To construct a reflection functor, we categorify the procedure developed in the previous section. 
We employ the categorified parabolic quantum group $\mathcal{U}_q(\mathfrak{p}_i)$ introduced in \cite{MR3709726}, 
which is described diagrammatically as a parabolic analogue of the categorified quantum group $\mathcal{U}_q(\mathfrak{g})$. 

We first establish a right $\mathcal{U}_q(\mathfrak{p}_i)$-module structure on $\gMod{{}_iR}$ that categorifies ${}_iV(0)$ (Theorem \ref{thm:cyclotomic2rep}). 
This generalizes the categorification of highest weight integrable modules \cite{MR2995184, MR3709726} to the parabolic setting. 
We note that the action of the Levi part of $\mathcal{U}_q(\mathfrak{p}_i)$ was already established by Vera \cite{MR4285453}. 
In the spirit of Chuang--Rouquier's minimal categorifications, 
the category $\gMod{{}_iR}$ satisfies a universal property as a right $\mathcal{U}_q(\mathfrak{p}_i)$-module.

Next, we prove in Theorem \ref{thm:anotheraction} that the other category $\gMod{R_i}$ is also a right $\mathcal{U}_q(\mathfrak{p}_i)$-module categorifying the right $U_q(\mathfrak{p}_i)$-module structure on $U_i$, 
which is the most crucial step of this paper.  
While determining the functors $E_i$ and $F_j \ (j \in I)$ from the decategorified formulas is straightforward, 
identifying the correct natural transformations is nontrivial.  

Once these module structures are established, 
the universality of $\gMod{{}_iR}$ immediately yields a morphism of right $\mathcal{U}_q(\mathfrak{p}_i)$-modules
\[
S_i \colon \gMod{{}_iR} \to \gMod{R_i}
\]
that sends the unit object to itself. 
Similarly, we construct a functor 
\[
S'_i \colon \gMod{R_i} \to \gMod{{}_iR}
\]
in the opposite direction as follows. 
We show that $\gMod{R_i}$ is a left $\mathcal{U}_q(\mathfrak{p}_i)$-module possessing a universal property, 
and that $\gMod{{}_iR}$ also carries a left $\mathcal{U}_q(\mathfrak{p}_i)$-module structure. 
By universality, we obtain a morphism of left $\mathcal{U}_q(\mathfrak{p}_i)$-modules $S_i'$ that sends the unit object to itself. 

The main result of this paper is stated as follows: 

\begin{theorem}[{Theorem \ref{thm:reflectionfunctor}, Theorem \ref{thm:braidrel}}] \label{thm:main}
For each $i \in I$, $S_i$ and $S'_i$ are mutually quasi-inverse monoidal equivalences. 
Furthermore, the collection of functors $\{S_i \}_{i \in I}$ satisfies the braid relations as natural isomorphisms in the following sense:
for any distinct $i, j \in I$ such that $a_{i,j} a_{j,i} \leq 3$, let $h(i,j) = 2,3,4,6$ according to whether $a_{i,j}a_{j,i} = 0,1,2,3$, respectively. 
Then, the compositions of $h(i,j)$ reflection functors $S_iS_j \cdots$ and $S_jS_i\cdots$ induce equivalences
\begin{equation} \label{eq:braid}
\gMod{{}_iR} \cap \gMod{{}_jR} \xrightarrow{\sim} \gMod{R_i}\cap \gMod{R_j}, 
\end{equation}
which are naturally isomorphic to each other. 
\end{theorem}

\subsection{Outline of the proof}

The proof of the first assertion of Theorem \ref{thm:main} proceeds in two steps. 

First, we prove in Proposition \ref{prop:monoidality} that the functors $S_i$ and $S'_i$ are monoidal, that is, 
there are natural isomorphisms
\[
S_i(X \circ Y) \simeq S_i(X) \circ S_i(Y), \ S'_i (X' \circ Y') \simeq S'_i (X') \circ S'_i(Y')
\]
for $X, Y \in \gMod{{}_iR}$ and $X',Y' \in \gMod{R_i}$. 
We establish this by induction on weight. 

Next, we prove in Proposition \ref{prop:linear} that $S_i$ is a morphism of left $\mathcal{U}_q(\mathfrak{p}_i)$-modules
(in addition to being a morphism of right $\mathcal{U}_q(\mathfrak{p}_i)$-modules by definition).
Similarly, $S'_i$ is a morphism of right $\mathcal{U}_q(\mathfrak{p}_i)$-modules. 
Consequently, the composition $S'_i \circ S_i$ is an endofunctor of $\gMod{{}_iR}$ as a right $\mathcal{U}_q(\mathfrak{p}_i)$-module sending the unit object to itself. 
Since $\gMod{{}_iR}$ is generated by the unit object as a right $\mathcal{U}_q(\mathfrak{p}_i)$-module, 
$S'_i \circ S_i$ is naturally isomorphic to the identity functor. 
The same argument shows that $S_i \circ S'_i$ is naturally isomorphic to the identity functor, 
confirming that $S_i$ and $S'_i$ are quasi-inverse equivalences. 

To establish the braid relation for $i,j \in I$, 
we first verify that the two compositions induce equivalences as in (\ref{eq:braid}). 
To prove that these compositions are naturally isomorphic, 
we employ the categorified parabolic quantum group $\mathcal{U}_q(\mathfrak{p}_{\{i,j\}})$, 
which categorifies the subalgebra $U_q(\mathfrak{p}_{\{i,j\}})$ of $U_q(\mathfrak{g})$ generated by $e_i,e_j, f_k \ (k \in I)$ and $q^h \ (h \in \mathsf{P}^{\lor})$.
We show that $\gMod{{}_iR} \cap \gMod{{}_jR}$ is a right $\mathcal{U}_q(\mathfrak{p}_{\{i,j\}})$-module possessing a universal property. 
Transporting this structure via the two equivalences (\ref{eq:braid}) yields 
two actions of $\mathcal{U}_q(\mathfrak{p}_{\{i,j\}})$ on $\gMod{R_i} \cap \gMod{R_j}$.
It then suffices to prove that these two actions coincide up to natural isomorphism, 
which is carried out in Section \ref{chap:braidrel}.

While the conceptual framework is transparent, 
the verification requires precise control over the relations between natural transformations.  
To make the exposition fully accessible, we have included detailed computations.

\subsection{Other results} \label{sec:otherresults}
In establishing our main theorems, we also prove several results that are of independent interest: 
\begin{enumerate}
\item The Grothendieck ring $K(\gMod{R}) \otimes_{\mathbb{Z}[q,q^{-1}]} \mathbb{Q}(q)$ of the category of finitely-generated (not finite-dimensional) graded modules over the quiver Hecke algebra, is isomorphic to $U_q^-(\mathfrak{g})$ (Theorem \ref{thm:categorification2}). 
The key is that the quiver Hecke algebra is finitely-generated over its center, which is isomorphic to a polynomial ring. 
\item We establish a parabolic generalization of the categorification of highest weight integrable modules \cite{MR2995184, MR3709726}, 
for any standard parabolic subalgebra $\mathfrak{p}_J$ of $\mathfrak{g}$, where $J$ is a subset of $I$. 
Specifically, we introduce categorified parabolic quantum group $\mathcal{U}_q(\mathfrak{p}_J)$ (Definition \ref{def:catquantum}), 
and prove that it acts on the category $\gMod{R^{J,\Lambda}}$ of modules over a parabolic generalization of cyclotomic quiver Hecke algebras for any $J$-dominant integral weight $\Lambda$ (Section \ref{sec:cyclotomic}).
It categorifies the $U_q(\mathfrak{p}_J)$-module $V_J(\Lambda)$ (Definition \ref{def:integrablemodule}). 
The result used in Section \ref{sec:categorifying} is a special case of this result for $J = \{i \}$ and $\Lambda = 0$.  
\item For $w,v \in W$ and $i \in I$ satisfying $s_iw > w, s_iv > v$, we prove equivalences between several subcategories of $\gMod{R}$ that categorify the following isomorphisms: 
\begin{itemize}
\item $T_{(s_iw)^{-1}}^{-1}U_q^-(\mathfrak{g}) \cap U_q^-(\mathfrak{g}) \cap T_vU_q^-(\mathfrak{g}) \xrightarrow{T_i} T_{w^{-1}}^{-1} U_q^-(\mathfrak{g}) \cap U_q^-(\mathfrak{g}) \cap T_{s_iv} U_q^-(\mathfrak{g})$. 
\item $T_w(U_q^0(\mathfrak{g})U_q^+(\mathfrak{g})) \cap U_q^-(\mathfrak{g}) \cap T_vU_q^-(\mathfrak{g}) \xrightarrow{T_i} T_{s_iw}(U_q^0(\mathfrak{g})U_q^+(\mathfrak{g})) \cap U_q^-(\mathfrak{g}) \cap T_{s_iv}U_q^-(\mathfrak{g})$. 
\end{itemize}
\end{enumerate}

\subsection{Notations and Conventions} \label{sec:notation}
Throughout this paper, $\mathbf{k}$ is a field of arbitrary characteristic.
Dimension of a $\mathbf{k}$-vector space is denoted by $\dim$, and tensor product over $\mathbf{k}$ is denoted by $\otimes$. 
For a graded $\mathbf{k}$-vector space $V = \bigoplus_{d \in \mathbb{Z}}V_d$, we define a formal series
\[
\qdim V = \sum_{d \in \mathbb{Z}}(\dim V_d)q^d. 
\]
If every homogeneous component of $V$ is finite dimensional and $V_d = 0$ for sufficiently small $d$, 
it gives a Laurentian series: $\qdim V \in \mathbb{Z}((q))$. 

A graded $\mathbf{k}$-linear category $\mathcal{A}$ is a $\mathbf{k}$-linear category endowed with a $\mathbf{k}$-linear autoequivalence $q$ called the grading shift functor. 
For $X, Y \in \mathcal{A}$, a $\mathbb{Z}$-graded $\mathbf{k}$-vector space $\HOM_{\mathcal{A}}(X,Y)$ is defined by 
\[
\HOM_{\mathcal{A}}(X,Y)_d = \Hom_{\mathcal{A}}(q^dX,Y) \ (d \in \mathbb{Z}). 
\]
Homogeneous elements of $\HOM_{\mathcal{A}}(X,Y)$ are called homogeneous morphisms from $X$ to $Y$. 
%For $X \in \mathcal{A}$ and $f(q) = \sum_k c_k q^k \in \mathbb{Z}[q,q^{-1}]$, we abbreviate 
%\[
%f(q) X = \bigoplus_k (q^k X)^{\oplus c_k} \in \mathcal{A}. 
%\]
Let $K_{\oplus}(\mathcal{A})$ denote the split Grothendieck group of $\mathcal{A}$. 
It is a $\mathbb{Z}[q,q^{-1}]$-module by the grading shift functor $q$, 
and we define $K_{\oplus}(\mathcal{A})_{\mathbb{Q}(q)} = K_{\oplus}(\mathcal{A}) \otimes_{\mathbb{Z}[q,q^{-1}]} \mathbb{Q}(q)$. 
When $\mathcal{A}$ is an abelian category, let $K(\mathcal{A})$ denote its Grothendieck group and let $K(\mathcal{A})_{\mathbb{Q}(q)} = K(\mathcal{A}) \otimes_{\mathbb{Z}[q,q^{-1}]} \mathbb{Q}(q)$. 

Let $A$ be a $\mathbb{Z}$-graded $\mathbf{k}$-algebra. 
Let $\gMod{A}$ (resp.~$\gproj{A}, \gmod{A}$) denote the category of finitely-generated graded left $A$-modules (resp. finitely-generated projective graded left $A$-modules, finite-dimensional graded left $A$-modules) whose morphisms are degree-preserving A-module homomorphisms. 
For a graded $A$-module $X$, we define its grading shift $qX$ by $(qX)_d = X_{d-1}$. 
Then, $\gMod{A}, \gproj{A}$ and $\gmod{A}$ are graded categories. 
For $X, Y \in \gMod{A}$, we define a graded $\mathbf{k}$-vector space $\EXT_A(X,Y)$ by 
\[
\EXT_A(X,Y)_d = \Ext_{\gMod{A}}(q^dX,Y) \ (d \in \mathbb{Z}). 
\]

Let $\mathfrak{C}$ be a 2-category.
We define a 2-category $\mathfrak{C}^{\mathrm{op}}$ with the same objects as $\mathfrak{C}$, and the hom-category $\mathfrak{C}^{\mathrm{op}}(a,b) = \mathfrak{C}(b,a)$. \index{$\mathfrak{C}^{\mathrm{op}}$}
We define another 2-category $\mathfrak{C}^{\mathrm{co}}$ with the same objects as $\mathfrak{C}$, and the hom-category $\mathfrak{C}^{\mathrm{co}}(a,b)$ is the opposite category of $\mathfrak{C}(a,b)$.  \index{$\mathfrak{C}^{\mathrm{co}}$}

A graded $\mathbf{k}$-linear 2-category is a 2-category enriched in graded $\mathbf{k}$-linear categories. 

For $n \in \mathbb{Z}_{\geq 1}$, let $\mathfrak{S}_n$ be the symmetric group of degree $n$. 
Let $e_n$ (resp. $w_n$) denote the unit element (resp. the longest element) of $\mathfrak{S}_n$. 
When $n = l+m$ with $l,m \geq 1$, let $\mathfrak{S}_{n}^{l,m} \subset \mathfrak{S}_n$ be the minimal length coset representative for $\mathfrak{S}_n/(\mathfrak{S}_l \times \mathfrak{S}_m)$. 
For $w \in \mathfrak{S}_l$ and $v \in \mathfrak{S}_m$, we define $w \star v$ as the image of $(w,v) \in \mathfrak{S}_l \times \mathfrak{S}_m$ in $\mathfrak{S}_n$ under the canonical embedding $\mathfrak{S}_l \times \mathfrak{S}_m \subset \mathfrak{S}_n$. 

\subsection{Acknowledgement}
I am deeply grateful to my supervisor, Noriyuki Abe, for his continuous support and invaluable feedback throughout this research. 
I would like to thank Myungho Kim and the anonymous referees for their helpful comments on the manuscript. 
This work was supported by the WINGS-FMSP program at the University of Tokyo, and JSPS KAKENHI Grant Number 25KJ1132.

\section{Quiver Hecke algebras and categorified quantum groups}
\subsection{Quantum groups} \label{sub:quantumgroups}

We mainly follow the conventions in \cite{MR3758148}. 
Throughout this paper, let $(\mathsf{A}, \mathsf{P}, \Pi, \Pi^{\lor}, (\cdot, \cdot))$ be a fixed root datum, 
where $\mathsf{A} = (a_{i, j})_{i, j \in I}$ is a symmetrizable generalized Cartan matrix, 
$\mathsf{P}$ is a free abelian group called the weight lattice,  
$\Pi = \{ \alpha_i \}_{i\in I}$ is a subset of $\mathsf{P}$,  
$\Pi^{\lor} = \{ h_i \}_{i \in I}$ is a subset of $\mathsf{P}^{\lor} = \Hom_{\mathbb{Z}}(\mathsf{P}, \mathbb{Z})$,  
and $(\cdot, \cdot)$ is a $\mathbb{Q}$-valued symmetric bilinear form on $\mathsf{P}$, satisfying the following conditions:  
\begin{enumerate}
  \item $a_{i, j} = \langle h_i, \alpha_j \rangle$ for $i, j \in I$, 
  \item $(\alpha_i, \alpha_i) \in 2 \mathbb{Z}_{> 0}$ for $i \in I$, 
  \item $\langle h_i, \lambda \rangle = 2 (\alpha_i, \lambda)/(\alpha_i, \alpha_i)$ for $i\in I$ and $\lambda \in \mathsf{P}$, 
  \item $\Pi$ is linearly independent and 
  \item for any $i \in I$, there exists $\Lambda_i \in \mathsf{P}$ such that $\langle h_j, \Lambda_i \rangle = \delta_{i, j}$ for all $j \in I$. 
\end{enumerate}

For each $i \in I$, we call $\alpha_i$ the simple root, $h_i$ the simple coroot, and $\Lambda_i$ the fundamental weight. 
We put $q_i = q^{(\alpha_i, \alpha_i)/2}, [n] = (q^n-q^{-n})/(q-q^{-1}), [n]! = [n] [n-1] \cdots [1], [n]_i = (q_i^n - q_i^{-n})/(q_i-q_i^{-1})$, and $[n]_i! = [n]_i [n-1]_i \cdots [1]_i$. 
Let $W$ be the Weyl group, which is generated by the simple reflections $s_i \ (i \in I)$. 
The root lattice is defined as $\mathsf{Q}=\sum_{i \in I} \mathbb{Z} \alpha_i \subset \mathsf{P}$, the positive root lattice is $\mathsf{Q}_+ = \sum_{i \in I }\mathbb{Z}_{\geq 0}\alpha_i$, and the negative root lattice is $\mathsf{Q}_- = - \mathsf{Q}_+$. 
We define $\height \colon \mathsf{Q}_+ \to \mathbb{Z}_{\geq 0}$ to be a morphism of monoids given by $\height (\alpha_i) = 1 \ (i \in I)$. 

\begin{definition}
Let $J$ be a subset of $I$.  
The parabolic quantum group $U_q(\mathfrak{p}_J)$ is a $\mathbb{Q}(q)$-algebra \index{$U_q(\mathfrak{p}_J)$}
on generators $e_i \ (i \in J), f_i \ (i \in I), q^h \ (h \in \mathsf{P}^{\lor})$, subject to the following relations: 
\begin{align*}
&q^0 = 1, \ q^h q^k = q^{h+k} \ (h,k \in \mathsf{P}^{\lor}),  \\
&q^h e_i q^{-h} = q^{\langle h, \alpha_i \rangle} e_i \ (i \in J, h \in \mathsf{P}^{\lor}), \\ 
&q^h f_i q^{-h} = q^{-\langle h, \alpha_i \rangle} f_i \ (i \in I, h \in \mathsf{P}^{\lor}), \\
&[e_i, f_j] = \delta_{i,j} \frac{t_i - t_i^{-1}}{q_i-q_i^{-1}} \ (i \in J, j \in I), \\
&\sum_{s=0}^{1-a_{i,j}} (-1)^s e_i^{(s)} e_j e_i^{(1-a_{i,j}-s)} = 0 \quad \text{if $i \neq j$} \ (i,j \in J), \\
&\sum_{s=0}^{1-a_{i,j}} (-1)^s f_i^{(s)} f_j f_i^{(1-a_{i,j}-s)} = 0 \quad \text{if $i \neq j$} \ (i,j \in I), 
\end{align*}
where $t_i = q^{\frac{(\alpha_i,\alpha_i)}{2} h_i}, e_i^{(s)} = e_i^s/([s]_i!), f_i^{(s)} = f_i^s/([s]_i!)$. 
\end{definition}

When $J = I$, $U_q(\mathfrak{p}_J)$ is the ordinary quantum group $U_q(\mathfrak{g})$. 
%When $J = \emptyset$, we write $U_q(\mathfrak{p}_{\emptyset}) = U_q(\mathfrak{b}^-)$. 

Let $U_q^+(\mathfrak{p}_J)$ (resp. $U_q^-(\mathfrak{p}_J), U_q^0(\mathfrak{p}_J)$) be the algebra generated by $e_i \ (i \in J)$ (resp. $f_i \ (i \in I)$ or $q^h \ (h \in \mathsf{P}^{\lor})$) with the same defining relations as $U_q(\mathfrak{p}_J)$.
Note that $U_q^-(\mathfrak{p}_J)$ and $U_q^0(\mathfrak{p}_J)$ are independent of $J$. 

\begin{lemma} \label{lem:pJtriangular}
The canonical homomorphism $U_q(\mathfrak{p}_J) \to U_q(\mathfrak{g})$ is injective, and 
the multiplication induces an isomorphism
\[
U_q(\mathfrak{p}_J) \simeq U_q^-(\mathfrak{p}_J) \otimes U_q^0(\mathfrak{p}_J) \otimes U_q^+(\mathfrak{p}_J). 
\]
\end{lemma}

\begin{proof}
By the defining relations, it is easy to see that the homomorphism $U_q^-(\mathfrak{p}_J) \otimes U_q^0(\mathfrak{p}_J) \otimes U_q^+(\mathfrak{p}_J) \to U_q(\mathfrak{p}_J)$ induced by the multiplication is surjective. 
By the definition, we have $U_q^0(\mathfrak{p}_J) \simeq U_q^0(\mathfrak{g})$ and $U_q^-(\mathfrak{p}_J) \simeq U_q^-(\mathfrak{g})$. 
It is well-known that $U_q^+(\mathfrak{p}_J) \to U_q^+(\mathfrak{g})$ is injective. 
In fact, the nondegenerate bilinear form on $U_q^+(\mathfrak{g})$ of \cite[Section 1]{MR2759715} is pulled back to the nondegenerate bilinear form on $U_q(\mathfrak{p}_J)$.  
Hence, the assertion follows from the triangular decomposition of $U_q(\mathfrak{g})$. 
\end{proof}

%Let $\varphi$ be $\mathbb{Q}(q)$-algebra antiautomorphisms of $\quantum{}$ defined by
%\[
% \varphi(e_i) = f_i, \ \varphi(f_i) = e_i, \ \varphi(q^h) = q^h, \\
%\]

Let $\sigma$ be a $\mathbb{Q}(q)$-algebra antiautomorphism of $U_q(\mathfrak{p}_J)$ defined by \index{$\sigma$ on $U_q(\mathfrak{p}_J)$}
\[
\sigma(e_i) = e_i, \ \sigma(f_i) = f_i, \ \sigma(q^h) = q^{-h}.
\] 

Let $\overline{(\cdot)}$ be a $\mathbb{Q}$-algebra automorphism of $U_q(\mathfrak{p}_J)$ defined by \index{$\overline{(\cdot)}$ on $U_q(\mathfrak{p}_J)$}
\[
\overline{e_i} = e_i,\ \overline{f_i} = f_i,\ \overline{q^h} = q^{-h}, \overline{q} = q^{-1}. 
\]

\begin{definition} \label{def:integrablemodule}
Let $J \subset I$. 
Let $\Lambda \in P$ be a $J$-dominant weight, that is, it satisfies $\langle h_j, \Lambda \rangle \geq 0$ for all $j \in J$. 
We define a left $U_q(\mathfrak{p}_J)$-module $V_J(\Lambda)$ as 
\[
U_q(\mathfrak{p}_J)\bigg/ \left(\sum_{j\in J} \left(U_q(\mathfrak{p}_J)e_j + U_q(\mathfrak{p}_J)f_j^{\langle h_j, \Lambda \rangle + 1}\right) + \sum_{h \in \mathsf{P}^{\lor}} U_q(\mathfrak{p}_J) (q^h - q^{\langle h, \Lambda \rangle}) \right). \index{$V_J(\Lambda)$}
\]
Let $v_{\Lambda}^J \in V_J(\Lambda)$ denote the image of $1 \in U_q(\mathfrak{p}_J)$. 
We write $v_{\Lambda}$ instead of $v_{\Lambda}^J$ when there is no risk of ambiguity. 

We define a right $U_q(\mathfrak{p}_J)$-module ${}_JV(-\Lambda)$ as 
\[
U_q(\mathfrak{p}_J)\bigg/ \left(\sum_{j\in J} \left(e_j U_q(\mathfrak{p}_J) + f_j^{\langle h_j, \Lambda \rangle +1}U_q(\mathfrak{p}_J)\right) + \sum_{h \in \mathsf{P}^{\lor}}  (q^h - q^{-\langle h, \Lambda \rangle})U_q(\mathfrak{p}_J) \right). \index{${}_JV(-\Lambda)$}
\] 
Let $v_{-\Lambda} \in {}_JV(-\Lambda)$ denote the image of $1 \in U_q(\mathfrak{p}_J)$. 
\end{definition} 

When $J=I$, $V_J(\Lambda) = V_I(\Lambda)$ is isomorphic to the integrable highest weight module of highest weight $\Lambda$. 
If $V_J(\Lambda)$ is regarded as a right $U_q(\mathfrak{p}_J)$-module using the anti-automorphism $\sigma$, it coincides with ${}_JV(-\Lambda)$. 

The automorphism $\overline{(\cdot)}$ of $U_q(\mathfrak{p}_J)$ induces automorphisms of $V_J(\Lambda)$ and ${}_JV(-\Lambda)$ by the definition, 
which are also denoted by $\overline{(\cdot)}$. 

\subsection{$q$-Boson algebras}

Let $J \subset I$ and let $\Lambda \in \mathsf{P}$ be a $J$-dominant weight. 
In this section, we introduce parabolic q-boson algebra, a hybrid of $U_q(\mathfrak{g})$ and the $q$-Boson algebra. 
and explain that $V_J(\Lambda)$ is a simple module over it. 

\begin{definition}[{\cite[3.1]{kwon2025infinitelevelfockspacescrystal}}] \label{def:qboson}
Let $J \subset I$. 
We define $B_q^J(\mathfrak{g})$ to be a $\mathbb{Q}(q)$-algebra with generators $e_i \ (i \in I), f_i \ (i \in I), q^h \ (h \in \mathsf{P}^{\lor})$ subject to the following relations: \index{$B_q^J(\mathfrak{g})$}
\begin{align*}
&q^0 = 1, \ q^h q^k = q^{h+k} \ (h,k \in \mathsf{P}^{\lor}),  \\
&q^h e_i q^{-h} = q^{\langle h, \alpha_i \rangle} e_i \ (i \in I, h \in \mathsf{P}^{\lor}), \\
&q^h f_i q^{-h} = q^{-\langle h, \alpha_i \rangle} f_i \ (i \in I, h \in \mathsf{P}^{\lor}), \\
&[e_i, f_j] = \delta_{i,j} \frac{t_i - t_i^{-1}}{q_i-q_i^{-1}} \ (i \in J, j \in I), \\
&e_if_j = q^{-(\alpha_i,\alpha_j)}f_je_i + \delta_{i,j} \ (i \in I \setminus J, j \in I), \\
&f_{i,j} = e_{i,j} = 0 \ (i,j \in I, i \neq j), \\
\end{align*}
where 
\begin{align*}
f_{i,j} &= \sum_{s=0}^{1-a_{i,j}} (-1)^s f_i^{(s)} f_j f_i^{(1-a_{i,j}-s)}, \\
e_{i,j} &= \begin{cases} 
\sum_{s=0}^{1-a_{i,j}} (-1)^s e_i^{(1-a_{i,j}-s)} e_j e_i^{(s)} & \text{if $i,j \in J$ or $i,j \in I \setminus J$}, \\
\sum_{s=0}^{1-a_{i,j}} (-1)^s q_i^{sa_{i,j}} e_i^{(1-a_{i,j}-s)} e_j e_i^{(s)} & \text{if $i \in I \setminus J, \ j \in J$}, \\
\sum_{s=0}^{1-a_{i,j}} (-1)^s q_i^{-sa_{i,j}}e_i^{(1-a_{i,j}-s)} e_j e_i^{(s)} & \text{if $i \in J, \ j \in I \setminus J$.} \\
\end{cases}
\end{align*}
Let $B_q^J(\mathfrak{g})^-$ (resp. $B_q^J(\mathfrak{g})^0, B_q^J(\mathfrak{g})^+$) be the $\mathbb{Q}(q)$-algebra generated by $f_i \ (i \in I)$ (resp. $q^h \ (h \in \mathsf{P}^{\lor})$ or $e_i \ (i \in I)$) with the same relations as $B_q^J(\mathfrak{g})$.  
\end{definition}

We have a canonical algebra homomorphism $U_q(\mathfrak{p}_J) \to B_q^J(\mathfrak{g})$. 

\begin{remark}
Our $B_q^J(\mathfrak{g})$ coincides with $U_q(\mathfrak{g}, \mathfrak{p}_J)$ of \cite{kwon2025infinitelevelfockspacescrystal}, except that we extended the Cartan part.

Note that $B_q^I(\mathfrak{g}) = U_q(\mathfrak{g})$. 
When $J = \emptyset$, let $B_q(\mathfrak{g})$ be the subalgebra of $B_q^{\emptyset}(\mathfrak{g})$ generated by $e_i, f_i \ (i \in I)$. 
It is the $q$-Boson algebra defined in \cite[Section 3.3]{MR1159265}, 
and $B_q^{\emptyset}(\mathfrak{g})$ is a smash product of $\mathbb{Q}(q)[\mathsf{P}^{\lor}]$ and $B_q(\mathfrak{g})$.  

The algebra $B_q^J(\mathfrak{g})$ is isomorphic to the generalization of $q$-oscillator algebra $U_q^{I \setminus J, \emptyset}(\mathfrak{g})$ introduced in \cite[Section 2.1]{MR4609778}, via the correspondence 
\[
B_q^J(\mathfrak{g}) \to U_q^{I \setminus J, \emptyset} (\mathfrak{g}), \ e_i \mapsto \begin{cases}
e_i & \text{if $i \in J$}, \\
-(q_i - q_i^{-1}) t_ie_i & \text{if $i \in I \setminus J$}, 
\end{cases} \ 
f_i \mapsto f_i, \ q^h \mapsto q^h. 
\] 
(Although only the case of finite-dimensional $\mathfrak{g}$ is considered in \cite{MR4609778}, the definition applies verbatim to arbitrary symmetrizable Kac-Moody algebras.)
Furthermore, $B_q^J(\mathfrak{g})$ is closely related to the degenerate quantized universal enveloping algebra of \cite[Definition 2.7]{MR3376147}, as discussed in \cite[Remark 5.20]{hoshino2025semisimplemodulecategoriesfusion}. 
\end{remark}

\begin{lemma}[{\cite[Lemma 3.1]{kwon2025infinitelevelfockspacescrystal}}] \label{lem:triangular}
\begin{comment}
(1) There exists a $\mathbb{Q}(q)$-algebra anti-involution $\varphi$ of $B_q^J(\mathfrak{g})$ given by \index{$\varphi$ on $B_q^J(\mathfrak{g})$}
\begin{align*}
\varphi(f_i) = \begin{cases}
\frac{1}{1-q_i^2} e_i & \text{if $i \in I \setminus J$}, \\
q_ie_it_i & \text{if $i \in J$}, 
\end{cases} \ \varphi(e_i) = \begin{cases}
(1-q_i^2) f_i & \text{if $i \in I \setminus J$}, \\
q_if_it_i^{-1} & \text{if $i \in J$}, 
\end{cases},\ \varphi(q^h) = q^h. 
\end{align*}
\end{comment}

We have canonical isomorphisms 
\[
B_q^J(\mathfrak{g})^- \simeq U_q^-(\mathfrak{g}),\ B_q^J(\mathfrak{g})^0 \simeq U_q^0(\mathfrak{g}). 
\]
Furthermore, the multiplication induces an isomorphism
\[
B_q^J(\mathfrak{g}) \simeq B_q^J(\mathfrak{g})^- \otimes B_q^J(\mathfrak{g})^0 \otimes B_q^J(\mathfrak{g})^+.  
\]
\end{lemma}

\begin{lemma} \label{lem:presentation}
The canonical homomorphisms 
\begin{align*}
&U_q^-(\mathfrak{g}) \bigg/ \sum_{j \in J} U_q^-(\mathfrak{g})f_j^{\langle h_j,\Lambda \rangle + 1} \to V_J(\Lambda) \\
&\to B_q^J(\mathfrak{g})\bigg/ \left(\sum_{i \in I}B_q^J(\mathfrak{g})e_i + \sum_{j \in J} B_q^J(\mathfrak{g})f_j^{\langle h_j,\Lambda \rangle +1} + \sum_{h \in \mathsf{P}^{\lor}}B_q^J(\mathfrak{g})(q^h-q^{\langle h,\Lambda \rangle}) \right). 
\end{align*}
are both isomorphisms. 
Hence, the left $U_q(\mathfrak{p}_J)$-module structure on $V_J(\Lambda)$ uniquely extends to a left $B_q^J(\mathfrak{g})$-module structure. 
\end{lemma}

\begin{proof}
Regarding the first homomorphism, observe that the triangular decomposition implies $U_q^-(\mathfrak{g}) \simeq M_J(\Lambda)$, where 
\[
M_J(\Lambda) = U_q(\mathfrak{p}_J) \bigg/ \left( \sum_{j \in J} U_q(\mathfrak{p}_J)e_j + \sum_{h \in \mathsf{P}^{\lor}} U_q(\mathfrak{p}_J) (q^h - q^{\langle h,\Lambda \rangle}) \right). 
\]
Hence, it suffices to prove that for each $j \in J$ the image of $U_q^-(\mathfrak{p}_J)f_j^{\langle h_j, \Lambda \rangle + 1}$ in $M_J(\Lambda)$ is a $U_q(\mathfrak{p}_J)$-submodule. 
By the triangular decomposition of $U_q(\mathfrak{p}_J)$ (Lemma \ref{lem:pJtriangular}), it suffices to prove that $e_i f_j^{\langle h_j,\Lambda \rangle + 1}v_{\Lambda} = 0 \ (i \in I)$ in $M_J(\Lambda)$, which is well-known: see \cite[3.5.6]{MR2759715} for instance. 
The second isomorphism is proved in a similar way using the triangular decomposition of $B_q^J(\mathfrak{g})$ (Lemma \ref{lem:triangular} (2)). 

Hence, the left $U_q(\mathfrak{p}_J)$-module structure on $V_J(\Lambda)$ extends to a left $B_q^J(\mathfrak{g})$-module structure. 
Since $V_J(\Lambda) = U_q^-(\mathfrak{g})v_{\Lambda}$, the commutation relation between $e_i \ (i \in I \setminus J)$ and $f_i \ (i \in I)$ of Definition \ref{def:qboson} imply that such an extension is unique. 
\end{proof}

\begin{theorem}[{\cite[Theorem 3.13, Proposition 3.18]{kwon2025infinitelevelfockspacescrystal}}] \label{thm:bosonsimple}
$V_J(\Lambda)$ is simple as a left $B_q^J(\mathfrak{g})$-module. 
\begin{comment}
and it admits a nondegenerate symmetric bilinear form $(,)$ such that 
\[
(v_{\Lambda},v_{\Lambda}) = 1, \ (ux, y) = (x,\varphi(u)y) \ (x,y \in V_J(\Lambda), u \in B_q^J(\mathfrak{g})). 
\]
(2) The right $U_q(\mathfrak{p}_J)$-module structure on $V_J^{\mathrm{r}}(\Lambda)$ extends to a right $B_q^j(\mathfrak{g})$-module structure by letting $e_i \in B_q^J(\mathfrak{g})$ acts by $\overline{r_i}$ for each $i \in I \setminus J$.
Furthermore, $V_J^{\mathrm{r}}(-\Lambda)$ is simple as a right $B_q^J(\mathfrak{g})$-module, and it admits a nondegenerate symmetric bilinear form $(,)$ such that 
\[
(v_{-\Lambda},v_{-\Lambda}) = 1, (xf_i, y) = (x,ye_i) \ (x,y \in V_J^{\mathrm{r}}(\Lambda), i \in I). 
\]
\end{comment}
\end{theorem}

\begin{definition}\label{def:qboson'}
Let $J \subset I$. 
We define ${B'}_q^J(\mathfrak{g})$ to be a $\mathbb{Q}(q)$-algebra with generators $e_i \ (i \in I), f_i \ (i \in I), q^h \ (h \in \mathsf{P}^{\lor})$ subject to the following relations: \index{${B'}_q^J(\mathfrak{g})$}
\begin{align*}
&q^0 = 1,\ q^h q^k = q^{h+k} \ (h,k \in \mathsf{P}^{\lor}), \\
&q^h e_i q^{-h} = q^{\langle h, \alpha_i \rangle} e_i, (h \in \mathsf{P}^{\lor}, i \in I) \\
&q^h f_i q^{-h} = q^{-\langle h, \alpha_i \rangle} f_i, \ (h \in \mathsf{P}^{\lor}, i\in I) \\
&[e_i, f_j] = \delta_{i,j} \frac{t_i - t_i^{-1}}{q_i-q_i^{-1}} \ (i\in J, j \in I), \\
&f_je_i = q^{-(\alpha_i,\alpha_j)}e_if_j + \delta_{i,j} \ (i \in I \setminus J, j \in I), \\
&e_{i,j} = f_{i,j} = 0 \ (i,j \in I, i \neq j). 
\end{align*}
\end{definition}

Note that we have an $\mathbb{Q}(q)$-algebra anti-automorphism $\sigma \colon B_q^J(\mathfrak{g}) \to {B'}_q^J(\mathfrak{g})$ given by
\[
\sigma(e_i) = e_i,\ \sigma(f_i )= f_i,\ \sigma(q^h) = q^{-h}. 
\] 
Also note that we have a canonical homomorphism $U_q(\mathfrak{p}_J) \to {B'}_q^J(\mathfrak{g})$. 

\begin{theorem} \label{thm:bosonsimple'}
The right $U_q(\mathfrak{p}_J)$-module structure on ${}_JV(-\Lambda)$ uniquely extends to a right ${B'}_q^J(\mathfrak{g})$-module structure. 
Furthermore, ${}_JV(\Lambda)$ is simple as a right ${B'}_q^J(\mathfrak{g})$-module. 
\begin{comment}
and it admits a nondegenerate symmetric bilinear form $(,)$ such that 
\begin{align*}
&(v_{-\Lambda},v_{-\Lambda}) = 1, \\ 
&(xf_i, y) = \frac{1}{1-q_i^2}(x,ye_i) \ (x,y \in {}_JV(-\Lambda), i \in I \setminus J), \\
&(xf_i,y) = q_i(x,yt_i^{-1}e_i) \ (x, y \in {}_JV(-\Lambda), i \in J). 
\end{align*}
\end{comment}
\end{theorem}

\begin{proof}
It follows from Theorem \ref{thm:bosonsimple} by applying the anti-isomorphism 
\[
\sigma \colon B_q^J(\mathfrak{g}) \to {B'}_q^J(\mathfrak{g}).
\] 
\end{proof}

\subsection{Braid group action} \label{sec:braidgroupaction}

\begin{definition}
Let $i \in I$. 
We define $T_i$ to be the $\mathbb{Q}(q)$-algebra automorphism of $U_q(\mathfrak{g})$ given by  \index{$T_i$}
\begin{align*}
  T_i(q^h) = q^{s_i h},\ T_i (e_i) = - t_i^{-1}f_i,\ T_i(f_i) = - e_i t_i, \\
  T_i(e_j) = \sum_{r+s = -a_{i,j}}(-1)^r q_i^{-r} e_i^{(r)}e_je_i^{(s)} \ (j \neq i), \\
  T_i(f_j) = \sum_{r+s = -a_{i,j}}(-1)^r q_i^{r}f_i^{(s)}f_jf_i^{(r)} \ (j \neq i).  
\end{align*}
\end{definition}

\begin{remark} \label{rem:substitution}
Our $T_i$ above coincides with $T''_{i,1}$ of \cite[37.1.3]{MR2759715} by substituting 
\[
q \mapsto v^{-1}, e_i \mapsto F_i, f_i \mapsto E_i, q^h \mapsto K_h. 
\]
\end{remark}

The inverse $T_i^{-1}$ coincides with $\sigma T_i \sigma$ \cite[37.2.4]{MR2759715}: explicitly
\begin{align*}
  T_i^{-1}(q^h) = q^{s_i h}, T_i^{-1} (e_i) = - f_it_i, T_i^{-1}(f_i) = - t_i^{-1}e_i, \\
  T_i^{-1}(e_j) = \sum_{r+s = -a_{i,j}}(-1)^r q_i^{-r} e_i^{(s)}e_je_i^{(r)} \ (j \neq i), \\
  T_i^{-1}(f_j) = \sum_{r+s = -a_{i,j}}(-1)^r q_i^{r}f_i^{(r)}f_jf_i^{(s)} \ (j \neq i).  
\end{align*}
The automorphisms $\{T_i\}_{i \in I}$ satisfy the braid relations \cite[Theorem 39.4.3]{MR2759715}. 
For each $w \in W$, we define $T_w$ to be the automorphism given by 
\[
T_w = T_{i_1} \cdots T_{i_l}, 
\]
where $(i_1, \ldots, i_l)$ is a reduced word of $w$. 
$T_w$ is independent of the choice of the reduced word. 

In the rest of this section, we fix $i \in I$. 

\begin{definition}
We define two subalgebras of $U_q^-(\mathfrak{g})$
\[
U_i = U_q^-(\mathfrak{g}) \cap T_iU_q^-(\mathfrak{g}),\ {}_iU = U_q^-(\mathfrak{g}) \cap T_i^{-1}U_q^-(\mathfrak{g}).  
\] \index{$U_i, {}_iU$}
\end{definition}

Note that our $U_i$ (resp. ${}_iU$) coincides with ${}^{\sigma}\mathbf{f}[i]$ (resp. $\mathbf{f}[i]$) of \cite[38.1]{MR2759715}, by the substitution of Remark \ref{rem:substitution}. 
By the definition, the automorphism $T_i$ induces an isomorphism ${}_iU \to U_i$. 

By \cite[Proposition 3.1.6]{MR2759715}, for any $u \in U_q^-(\mathfrak{g})$, there uniquely exist elements $r_i(u), {}_ir(u) \in U_q^-(\mathfrak{g})$ such that \index{$r_i, {}_ir$}
\[
e_iu - ue_i = \frac{r_i(u)t_i - t_i^{-1}{}_ir(u)}{q_i - q_i^{-1}}. 
\]
They yield $\mathbb{Q}(q)$-linear endomorphisms of $U_q^-(\mathfrak{g})$, $r_i$ and ${}_ir$. 

\begin{lemma}[{\cite[18.1.6]{MR2759715}}] \label{lem:boson}
We have $U_i = \Ker r_i,\ {}_iU = \Ker {}_ir$. 
\end{lemma}

\begin{lemma}[{\cite[Section 38]{MR2759715}}] \label{lem:PBW}
The multiplication induces the following isomorphisms:
\[
U_i \otimes_{\mathbb{Q}(q)} \langle f_i \rangle \to U_q^-(\mathfrak{g}), \langle f_i \rangle \otimes_{\mathbb{Q}(q)} {}_iU \to U_q^-(\mathfrak{g}), 
\]
where $\langle f_i \rangle$ is the $\mathbb{Q}(q)$-subalgebra of $U_q(\mathfrak{g})$ generated by $f_i$. 
\end{lemma}

We write $U_q(\mathfrak{p}_i)$ instead of $U_q(\mathfrak{p}_{\{i\}})$.
Recall the left (resp. right) $U_q(\mathfrak{p}_i)$-module $V_{\{i\}}(0)$ (resp. ${}_{\{i\}}V(0)$). 
We simply write $V_i (0)$ (resp. ${}_iV(0)$) for it. 

\begin{lemma} \label{lem:isom}
The following morphisms are isomorphisms:
\begin{align*}
U_i \to V_i(0),\ & u \mapsto uv_0, \\
{}_iU \to {}_iV(0),\ & u \mapsto v_0u. 
\end{align*}
\end{lemma}

\begin{proof}
It immediately follows from Lemma \ref{lem:presentation} and Lemma \ref{lem:PBW}.  
\end{proof}

By Lemma \ref{lem:isom}, $U_i$ (resp. ${}_iU$) inherits a left (resp. right) $U_q(\mathfrak{p}_i)$-module structure. 
We shall explicitly described it below. 

\begin{definition}
We define $\mathbb{Q}(q)$-linear endomorphisms $\ad_{f_i}, \ad_{e_i}, \ad_{f_i}^*, \ad_{e_i}^*$ of $U_q(\mathfrak{g})$ as \index{$\ad_{f_i}, \ad_{e_i}, \ad_{f_i}^*, \ad_{e_i}^*$}
\begin{align*}
\ad_{f_i}(u) = f_i u - t_i u t_i^{-1} f_i, \ \ad_{e_i}(u) = e_iut_i - ue_it_i,  \\ 
\ad^*_{f_i}(u) = uf_i - f_i t_i u t_i^{-1},\ \ad^*_{e_i}(u) = t_i^{-1}ue_i - t_i^{-1}e_iu. 
\end{align*}
\end{definition}

Note that $\ad^*_{f_i} = \sigma \ad_{f_i} \sigma, \ad^*_{e_i} = \sigma\ad_{e_i}\sigma$. 

\begin{lemma} \label{lem:Tvsad}
We have 
\[
\ad_{f_i} = T_i \ad_{e_i}^* T_i^{-1}, \ad_{e_i} = T_i \ad_{f_i}^* T_i^{-1}.
\]
\end{lemma}

\begin{proof}
For $u \in U_q(\mathfrak{g})$, 
\begin{align*}
T_i \ad_{e_i}^* (u)& = T_i (t_i^{-1} ue_i- t_i^{-1}e_iu)  \\
&= -t_iT_i(u)t_i^{-1}f_i + t_it_i^{-1}f_iT_i(u) \\
&= \ad_{f_i}T_i(u). 
\end{align*}
By applying $\sigma$, we obtain
\[
T_i^{-1} \ad_{e_i} = \ad^*_{f_i} T_i^{-1}. 
\]
\end{proof}

\begin{lemma} \label{lem:stable}
$U_i$ is stable under $\ad_{f_i}$ and $\ad_{e_i}$, while ${}_iU$ is stable under $\ad^*_{f_i}$ and $\ad^*_{e_i}$. 
\end{lemma}

\begin{proof}
Note that $U_q^-(\mathfrak{g})$ is stable under $\ad_{f_i}$ and $\ad^*_{f_i}$. 
Let $u \in U_i$. 
Lemma \ref{lem:boson} shows $r_i(u) = 0$, 
hence $\ad_{e_i}(u) = (t_i^{-1}{}_ir(u)t_i)/(q_i-q_i^{-1})$, which belongs to $U_q^-(\mathfrak{g})$. 
On the other hand, Lemma \ref{lem:Tvsad} shows $T_i^{-1}\ad_{e_i}(u) = \ad_{f_i}^*T_i^{-1}(u)$, hence $T_i^{-1}\ad_{e_i}(u)$ also belongs to $U_q^-(\mathfrak{g})$. 
It means that $\ad_{e_i}(u) \in U_i$. 
By applying $\sigma$, we see that ${}_iU$ is stable under $\ad_{e_i}^*$. 
By using Lemma \ref{lem:Tvsad} again, we deduce that $U_i$ is stable under $\ad_{f_i}$ and that ${}_iU$ is stable under $\ad_{f_i}^*$. 
\end{proof}

\begin{proposition} \label{prop:actionlift}
(1) $U_i$ is a left $U_q(\mathfrak{p}_i)$-module by 
\[
f_i \cdot u = \ad_{f_i} (u), e_i \cdot u = \ad_{e_i}(u), f_j \cdot u = f_ju \ (j \neq i), q^h \cdot u = q^h u q^{-h} \ (h \in \mathsf{P}^{\lor}).  
\] 
Furthermore, the first isomorphism of Lemma \ref{lem:isom} is left $U_q(\mathfrak{p}_i)$-linear. 

(2) ${}_iU$ is a right $U_q(\mathfrak{p}_i)$-module by 
\[
u\cdot f_i = \ad_{f_i}^* (u), u \cdot e_i = \ad_{e_i}^*(u), u \cdot f_j = uf_j \ (j \neq i), u \cdot q^h = q^{-h}uq^h \ (h \in \mathsf{P}^{\lor}).  
\] 
Furthermore, the first isomorphism of Lemma \ref{lem:isom} is right $U_q(\mathfrak{p}_i)$-linear. 
\end{proposition}

\begin{proof}
(1) Since $f_i v_0 = e_iv_0 = 0$, we have 
\[
\ad_{f_i}(u)v_0 = f_i (uv_0), \ad_{e_i}(u)v_0 = e_i(uv_0), (q^huq^{-h})v_0 = q^h(uv_0). 
\]
Hence, the assertion follows from Lemma \ref{lem:stable}. 
(2) is similar. 
\end{proof}

Recall that we have an isomorphism $T_i \colon {}_iU \to U_i$. 
Hence, the left $U_q(\mathfrak{p}_i)$-module structure on $U_i$ yields a left $U_q(\mathfrak{p}_i)$-module structure ${}_iU$.
Similarly, the right $U_q(\mathfrak{p}_i)$-module structure on ${}_iU$ yields a right $U_q(\mathfrak{p}_i)$-module structure on $U_i$.  
They are explicitly given by the following formulas. 

\begin{proposition} \label{prop:bimodule}
(1) $U_i$ is a right $U_q(\mathfrak{p}_i)$-module by 
\[
u \cdot f_i = \ad_{e_i}(u), u \cdot e_i = \ad_{f_i}(u), u \cdot f_j = u u_j \ (j \neq i), 
\]
where $u_j = T_i(f_j)$. \index{$u_j = T_i(f_j)$}

(2) ${}_iU$ is a left $U_q(\mathfrak{p}_i)$-module by 
\[
f_i \cdot u = \ad_{e_i}^*(u), e_i \cdot u = \ad_{f_i}^*(u), f_j \cdot u = u'_j u \ (j \neq i), 
\]
where $u'_j = T_i^{-1}(f_j)$. \index{$u'_j = T_i^{-1}(f_j)$}

(3) $T_i \colon {}_iU \to U_i$ is both left $U_q(\mathfrak{p}_i)$-linear and right $U_q(\mathfrak{p}_i)$-linear. 
\end{proposition}

\begin{proof}
Using Lemma \ref{lem:Tvsad}, it immediately follows from Proposition \ref{prop:actionlift}.
\end{proof}

\begin{remark}
The left and the right $U_q(\mathfrak{p}_i)$-action do not commute.  
\end{remark}

\begin{lemma} \label{lem:uj}
For $j \neq i$, we have 
\[
u_j = \ad_{f_i}^{(-a_{i,j})}(f_j), u'_j = (\ad_{f_i}^*)^{(-a_{i,j})}(f_j). 
\]
\end{lemma}

\begin{proof}
It is straightforward from the definition, see \cite[Lemma 1.1.1]{MR1265471}. 
\end{proof}

\begin{remark}
Since ${}_iU \simeq {}_iV(0)$ is generated by $1$ as a right $U_q(\mathfrak{p}_i)$-module, 
the isomorphism $T_i \colon {}_iU \to U_i$ can be characterized as the right $U_q(\mathfrak{p}_i)$-module homomorphism that sends $1$ to $1$. 
Similarly, the isomorphism $T_i^{-1} \colon U_i \to {}_iU$ can be characterized as the left $U_q(\mathfrak{p}_i)$-module homomorphism that sends $1$ to $1$. 
We will construct functors that categorify $T_i$ and $T_i^{-1}$ based on these characterizations. 
\end{remark} 

\subsection{Quiver Hecke algebras}

\begin{definition}
  A choice of scalars $Q$ consists of elements $t_{i,j} \in \mathbf{k}^{\times} \ (i, j \in I)$ 
  and $s_{i,j}^{p,q} \in \mathbf{k} \ (i,j \in I, p,q \in \mathbb{Z}_{>0}, p(\alpha_i,\alpha_i) + q(\alpha_j,\alpha_j) = -2(\alpha_i,\alpha_j))$ 
  subject to the following conditions: \index{$Q = (t_{i,j}, s_{i,j}^{p,q})$: choice of scalars}
  \begin{enumerate}
  \item $t_{i,i} = 1$, 
  \item $t_{i,j} = t_{j,i}$ if $a_{i,j} = 0$, 
  \item $s_{i,j}^{p,q} = s_{j,i}^{q,p}$. 
  \end{enumerate} 
  For such scalars, we define polynomials $Q_{i,j}(u,v) \in \mathbf{k}[u,v] \ (i,j \in I)$ as 
  \[
  Q_{i,j}(u,v) = \begin{cases}
  t_{i,j}u^{-a_{i,j}} + t_{j,i} v^{-a_{j,i}} + \sum_{p,q \in \mathbb{Z}_{> 0}} s_{i,j}^{p,q} u^p v^q & \text{if $a_{i,j} < 0$}, \\
  t_{i,j} = t_{j,i} & \text{if $a_{i,j} = 0$}, \\
  0 & \text{if $a_{i,j} = 2$}.
  \end{cases}
  \] \index{$Q_{i,j}$}
  \end{definition}
  
Note that we have $Q_{i,j}(u,v) = Q_{j,i} (v,u)$. 

\begin{definition}
Fix a choice of scalars $Q$. 
Let $\beta \in \mathsf{Q}_+$. 
Put $n = \height \beta$ and $I^{\beta} = \{ \nu \in I^n \mid \alpha_{\nu_1} + \cdots + \alpha_{\nu_n} = \beta \}$. 
The quiver Hecke algebra $R(\beta)$ is a graded $\mathbf{k}$-algebra defined by the following generators and relations:  \index{$R(\beta)$}
\begin{itemize}
\item The generators are 
\[
e(\nu) \ (\nu \in I^{\beta}), x_k \ (1 \leq k \leq n), \tau_k \ (1 \leq k \leq n-1). 
\]
\item The relations are 
\begin{align*}
  & e(\nu)e(\nu') = \delta_{\nu, \nu'} e(\nu), \  \sum_{\nu \in I^{\beta}} e(\nu) = 1, \\
  & x_k e(\nu) = e(\nu) x_k, \ x_k x_{l} = x_{l} x_k, \\
  & \tau_k e(\nu) = e(s_k(\nu)) \tau_k \ (1 \leq k \leq n-1), \ \tau_k \tau_{l} = \tau_{l} \tau_k \ (1 \leq k, l \leq n-1, \lvert k-l \rvert \geq 2), \\
  & (\tau_k x_{k+1} - x_k \tau_k)e(\nu) = (x_{k+1} \tau_k - \tau_k x_k)e(\nu) = \delta_{\nu_k, \nu_{k+1}} e(\nu) \; (1 \leq k \leq n-1), \\
  & \tau_k^2 e(\nu) = Q_{\nu_k, \nu_{k+1}} (x_k, x_{k+1}) e(\nu) \; (1 \leq k \leq n-1), \\
  & (\tau_{k+1} \tau_k \tau_{k+1} - \tau_k \tau_{k+1} \tau_k) e(\nu) = \overline{Q}_{\nu_k, \nu_{k+1}, \nu_{k+2}}(x_k, x_{k+1}, x_{k+2})e(\nu) \; (1 \leq k \leq n-2), 
\end{align*}  
where 
\[
\overline{Q}_{i,i',i''}(u,u',u'') = \begin{cases} 
\dfrac{Q_{i,i'}(u,u')-Q_{i,i'}(u'',u')}{u-u''} & \text{if $i= i'' \neq i'$}, \\
0 & \text{otherwise}. 
\end{cases}
\] \index{$\overline{Q}_{i,i',i''}$}
\item The degree is given by 
\[
\deg e(\nu) = 0, \ \deg x_k e(\nu) = (\alpha_{\nu_k},\alpha_{\nu_k}), \ \deg \tau_k e(\nu) = -(\alpha_{\nu_k},\alpha_{\nu_{k+1}}). 
\]
\end{itemize}
\end{definition}

In this paper, every $R(\beta)$-module is assumed to be a graded left module unless otherwise specified. 
For each $w \in \mathfrak{S}_n$, fix a reduced expression $w = s_{i_1} \cdots s_{i_l}$ and define 
\[
\tau_w = \tau_{i_1} \cdots \tau_{i_l}. 
\]
In general, it depends on the choice of the reduced expression. 

\begin{lemma} \label{lem:tau}
Let $\nu \in I^{\beta}$. 
If $w$ satisfies 
\[
\text{there exists no $1 \leq a < b < c \leq n$ such that $w(a) > w(b) > w(c)$ and $\nu_a = \nu_c \neq \nu_b$,}
\] 
then $\tau_w e(\nu)$ is independent of the reduced expression of $w$. 
\end{lemma}

\begin{proof}
Consider passing from one reduced expression of $w$ to another by braid moves. 
The lemma follows from the defining relations of $R(\beta)$.
\end{proof}

For instance, if $w \in \mathfrak{S}_n^{l,m}$ for some $l + m = n$, then $\tau_w$ is independent of the choice.  

There is a $\mathbf{k}$-algebra anti-involution $\varphi$ of $R(\beta)$ that fixes all the generators $e(\nu), x_k$ and $\tau_k$. \index{$\varphi$ on $R(\beta)$}
Using it, we get a duality functor $D$ on $\gmod{R(\beta)}$ given by $D(M) = \Hom_{\mathbf{k}}(M, k)$, on which $R(\beta)$ acts by 
\[
\text{$(a f)(m) = f(\varphi(a)m)$ for $a \in R(\beta), f \in D(M), m \in M$. } 
\]
The $d$-th homogeneous component of $D(M)$ is $D(M)_d = \Hom_{\mathbf{k}}(M_{-d}, \mathbf{k})$. 
%Similarly, there is a duality functor $D'$ on $\gproj{R(\beta)}$ defined by $D'(P) = \Hom_{R(\beta)} (P, R(\beta))$. \index{$D'$}
A finite-dimensional module $M \in \gmod{R(\beta)}$ is said to be self-dual if $DM \simeq M$. 

Let $Z(\beta) = \left( \bigoplus_{\nu \in I^{\beta}} \mathbf{k}[x_1, \ldots, x_{\height \beta}]e(\nu) \right)^{\mathfrak{S}_{\height \beta}}$ be the center of $R(\beta)$. \index{$Z(\beta)$}
It is isomorphic to 
\[
\bigotimes_{i \in I}\mathbf{k}[z_{i,1}, \ldots, z_{i,k_i}]^{\mathfrak{S}_{k_i}},
\]
where $k_i$ is given by $\beta = \sum_{i \in I} k_i \alpha_i$ and $\deg z_{i,p} = (\alpha_i,\alpha_i)$.  

Let $\beta,\gamma \in Q_+$ and put $m = \height (\beta), n = \height (\gamma)$. 
We define an idempotent $e(\beta, \gamma)$ of $R(\beta + \gamma)$ by
\[
e(\beta, \gamma) = \sum_{\nu \in I^{\beta}, \nu' \in I^{\gamma}} e(\nu, \nu'). 
\]
We sometimes write $e(\beta,*) = e(\beta,\gamma) = e(*,\gamma)$ and $e(i,*) = e(\alpha_i,*), e(*,i) = e(*,\alpha_i)$.   

Then, $R(\beta+ \gamma)e(\beta,\gamma)$ is a right $(R(\beta) \otimes R(\gamma))$-module as follows:
\begin{align*}
 &ue(\beta,\gamma) (e(\nu)\otimes 1) = ue(\nu,\gamma) \ (\nu \in I^{\beta}), \\
 &ue(\beta, \gamma) (1 \otimes e(\nu)) = ue(\beta, \nu) \ (\nu \in I^{\gamma}), \\
 &ue(\beta,\gamma) (x_k \otimes 1) = ue(\beta,\gamma)x_k \ (1 \leq k \leq m), \\
 &ue(\beta,\gamma) (1 \otimes x_k) = ue(\beta,\gamma) x_{k+m} \ (1 \leq k \leq n), \\
 &ue(\beta,\gamma) (\tau_k \otimes 1) = ue(\beta,\gamma)\tau_k \ (1 \leq k \leq m-1), \\ 
 &ue(\beta, \gamma) (1 \otimes \tau_k) = ue(\beta,\gamma) \tau_{k+m} \ (1\leq k \leq n-1). 
\end{align*}
It is both left $R(\beta + \gamma)$-projective and right $(R(\beta) \otimes R(\gamma))$-projective. 
Similar property holds for $e(\beta, \gamma)R(\beta + \gamma)$. 
They produce two exact functors
\begin{align*}
  \Ind_{\beta, \gamma} &= R(\beta + \gamma)e(\beta, \gamma) \otimes_{R(\beta) \otimes R(\gamma)} (\cdot)\colon \gMod{(R(\beta) \otimes R(\gamma))} \to \gMod{R(\beta + \gamma)}, \\
  \Res_{\beta, \gamma} &= \Hom_{R(\beta + \gamma)}(R(\beta + \gamma)e(\beta, \gamma), \cdot)\colon \gMod{R(\beta + \gamma)} \to \gMod{(R(\beta) \otimes R(\gamma))}. 
\end{align*}
We have an adjoint pair $(\Ind_{\beta, \gamma}, \Res_{\beta, \gamma})$. 
For multiple $(\beta_1, \ldots, \beta_m) \in \mathsf{Q}_+^m$, we define $\Ind_{\beta_1, \ldots, \beta_m}$ and $\Res_{\beta_1, \ldots, \beta_m}$ in the same manner. 
We usually write $M \circ N$ instead of $\Ind_{\beta, \gamma} (M \otimes N)$ and call it the convolution product of $M$ and $N$.
It gives a monoidal structure on $\gMod{R} = \bigoplus_{\beta \in Q_+} \gMod{R(\beta)}$ with the unit object $\mathbf{k} \in \gMod{R(0)}$, which is denoted by $\mathbf{1}$. 
Additionally, $\gmod{R} = \bigoplus_{\beta \in Q_+} \gmod{R(\beta)}$ and $\gproj{R} = \bigoplus_{\beta \in Q_+} \gproj{R(\beta)}$ are closed under the convolution products.
When $M \otimes N$ is regarded as a subspace of $M \circ N$, it is denoted by $M \boxtimes N$. 
For $u \in M, v\in N$, the element $u \otimes v \in M \boxtimes N$ is denoted by $u \boxtimes v$. 
Similarly, when $R(\alpha) \otimes R(\beta)$ is regarded as a subspace of $R(\alpha+\beta)$, it is denoted by $R(\alpha) \boxtimes R(\beta)$. 

Let $\beta_1, \ldots, \beta_m, \gamma_1, \ldots, \gamma_n \in \mathsf{Q}_+$ with $\sum_k \beta_k = \sum_l \gamma_l$.
Put $b_k = \height \beta_k, c_l = \height \gamma_l$ and $a = \sum_k b_k = \sum_l c_l$. 
Let $A = A(\beta_1, \ldots, \beta_m; \gamma_1, \ldots, \gamma_n)$ be the set of $\boldsymbol{\alpha} = (\alpha_{k,l})_{1 \leq k \leq m, 1 \leq l \leq n} \in \mathsf{Q}_+^{mn}$ satisfying  
\[
\sum_k \alpha_{k,l} = \gamma_l \ (1 \leq l \leq n), \  \sum_l \alpha_{k,l} = \beta_k \ (1 \leq k \leq m). 
\] 
Let ${}^{(c_l)}\mathfrak{S}^{(b_k)}$ be the set of minimal length double-coset representatives for 
\[
(\mathfrak{S}_{c_1} \times \cdots \times \mathfrak{S}_{c_n}) \backslash \mathfrak{S}_a / (\mathfrak{S}_{b_1} \times \cdots \times \mathfrak{S}_{b_m}). 
\]
For $\boldsymbol{\alpha} \in A$, we define $w(\boldsymbol{\alpha}) \in {}^{(c_l)}\mathfrak{S}^{(b_k)}$ as the element satisfying
\[
\lvert w(\boldsymbol{\alpha}) [b_1 + \cdots + b_{k-1} + 1, b_1 + \cdots + b_k] \cap [c_1 + \cdots + c_{l-1}+1, c_1 + \cdots + c_l] \rvert = \height \alpha_{k,l}
\]
for any $1 \leq k \leq m, 1 \leq l \leq n$. 

Let $M_k \in \gMod{R(\beta_k)} \ (1 \leq k \leq m)$. 
Put $V = \Res_{\gamma_1, \ldots, \gamma_n} \Ind_{\beta_1, \ldots, \beta_m} (M_1 \otimes \cdots \otimes M_m)$. 
It is decomposed into a direct sum of graded $\mathbf{k}$-vector spaces
\[
V = \bigoplus_{w \in {}^{(c_l)}\mathfrak{S}^{(b_k)}} (R(\gamma_1) \otimes \cdots \otimes R(\gamma_n))\tau_w (M_1 \otimes \cdots \otimes M_m). 
\]
For $w \in {}^{(c_l)}\mathfrak{S}^{(b_k)}$, we define two subspaces of $V$ 
\begin{align*}
F_{\leq w}V &= \bigoplus_{v \in {}^{(c_l)}\mathfrak{S}^{(b_k)}, v \leq w} (R(\gamma_1) \otimes \cdots \otimes R(\gamma_n))\tau_v (M_1 \otimes \cdots \otimes M_m), \\
F_{< w}V &=  \bigoplus_{v \in {}^{(c_l)}\mathfrak{S}^{(b_k)}, v < w} (R(\gamma_1) \otimes \cdots \otimes R(\gamma_n))\tau_v (M_1 \otimes \cdots \otimes M_m). 
\end{align*}

\begin{proposition}[Mackey filtration, {\cite[Proposition 2.18]{MR2525917}}] \label{prop:Mackey}
We use the notation above. 
For $w \in {}^{(c_l)}\mathfrak{S}^{(b_k)}$, $F_{\leq w}V$ and $F_{< w}V$ are $R(\gamma_1) \otimes \cdots \otimes R(\gamma_n)$-submodules of $V$. 
Furthermore, $F_{\leq w}V / F_{< w}V$ is isomorphic to a direct sum of
\begin{align*}
&q^{m(\boldsymbol{\alpha})}(\Ind_{\alpha_{1,1}, \ldots, \alpha_{m,1}} \otimes \cdots \otimes \Ind_{\alpha_{1,n}, \ldots, \alpha_{m,n}}) \cdot \\  
&(\Res_{\alpha_{1,1},\ldots,\alpha_{1,n}}M_1 \otimes \cdots \otimes \Res_{\alpha_{m,1},\ldots, \alpha_{m,n}}M_m) \\
&\text{where $m(\boldsymbol{\alpha}) = -\sum_{1 \leq k < k' \leq m, n \geq l > l' \geq 1} (\alpha_{k,l}, \alpha_{k',l'})$}, 
\end{align*}
parametrized by $\boldsymbol{\alpha} \in A$ satisfying $w(\boldsymbol{\alpha}) = w$. 
The isomorphism is given by 
\begin{align*}
q^{m(\boldsymbol{\alpha})}\Res_{\alpha_{1,1}, \ldots,\alpha_{1,n}}M_1 \otimes \cdots \otimes \Res_{\alpha_{m,1}, \ldots,\alpha_{m,n}} M_m &\to F_{\leq w}V/F_{<w} V,  \\
v_1 \otimes \cdots \otimes v_m &\mapsto \tau_w(v_1 \otimes \cdots \otimes v_n) + F_{< w}V, 
\end{align*}
and it is natural in $M_1, \ldots, M_m$. 
\end{proposition}

\begin{comment}
\begin{proposition}[Mackey filtration, {\cite[Proposition 2.18]{MR2525917}}] \label{prop:Mackey}
Let $\beta_1, \ldots, \beta_m, \gamma_1, \ldots, \gamma_n \in \mathsf{Q}_+$ with $\sum_k \beta_k = \sum_l \gamma_l$, and $M_k \in \gMod{R(\beta_k)}$. 
Then, 
\[ 
\Res_{\gamma_1,\ldots,\gamma_n} \Ind_{\beta_1, \ldots, \beta_m} (M_1 \otimes \cdots \otimes M_m)  
\] 
has a filtration whose successive quotients are
\begin{align*}
&q^{m(\alpha_{k,l})}(\Ind_{\alpha_{1,1}, \ldots, \alpha_{m,1}} \otimes \cdots \otimes \Ind_{\alpha_{1,n}, \ldots, \alpha_{m,n}}) (\Res_{\alpha_{1,1},\ldots,\alpha_{1,n}}M_1 \otimes \cdots \otimes \Res_{\alpha_{m,1},\ldots, \alpha_{m,n}}M_m), \\
&m(\alpha_{k,l}) = -\sum_{1 \leq k < k' \leq m, n \geq l > l' \geq 1} (\alpha_{k,l}, \alpha_{k',l'}), 
\end{align*}
where $(\alpha_{k,l}) \in \mathsf{Q}_+^{mn}$ runs under the condition 
\[
\sum_k \alpha_{k,l} = \gamma_l \ (1 \leq l \leq n), \  \sum_l \alpha_{k,l} = \beta_k \ (1 \leq k \leq m). 
\] 
\end{proposition}
\end{comment}

Put $n = \height \beta$. 
We define an algebra involution $\sigma$ of $R(\beta)$ by 
\begin{align*}
&\sigma(e(\nu)) = e(\nu_n, \ldots, \nu_1),\ \sigma(x_k) = x_{n+1-k},\\
&\sigma(\tau_k e(\nu)) = (-1)^{\delta_{\nu_k,\nu_{k+1}}}\tau_{n-k}e(\nu_n,\ldots,\nu_1). 
\end{align*}
It yields an autofunctor $\sigma_*$ of $\gMod{R(\beta)}$. 

Let $i \in I$ and $n \in \mathbb{Z}_{\geq 0}$. 
In $R(n\alpha_i)$, we have the following identity: 
\begin{equation} \label{eq:demazure}
\tau_{w_n} f \tau_k = \tau_{w_n} \partial_k(f), \ \tau_k f \tau_{w_n} = \partial_k(f) \tau_{w_n}, 
\end{equation}
where $f \in \mathbf{k}[x_1, \ldots,x_n], 1 \leq k \leq n-1$, and $\partial_k$ is the Demazure operator defined by \index{$\partial_k$: Demazure operator}
\[
\partial_k (f) = \frac{s_k(f) - f}{x_k-x_{k+1}}. 
\]
The operators $\partial_k \ (1 \leq k \leq n-1)$ satisfy the braid relations. 
Hence, we have an operator $\partial_w$ for each $w \in \mathfrak{S}_n$. 

We define 
\begin{align*}
\mathbf{x}_n = x_2 x_3^2 \cdots x_n^{n-1}, \ & \mathbf{x}'_n = x_1^{n-1}x_2^{n-2} \cdots x_{n-1}, \\
b_+(i^n) = \mathbf{x}_n\tau_{w_n}, \ & b_-(i^n) = \tau_{w_n}\mathbf{x}_n , \\
b'_+(i^n) = (-1)^{n(n-1)/2} \mathbf{x}'_n \tau_{w_n}, \ & b'_-(i^n) = (-1)^{n(n-1)/2} \tau_{w_n} \mathbf{x}'_n. 
\end{align*} \index{$\mathbf{x}_n, \mathbf{x}'_n$} \index{$b_+(i^n),b'_+(i^n), b_-(i^n),b'_-(i^n)$}
Note that 
\[
\varphi(b_+(i^n)) = b_-(i^n), \varphi(b'_+(i^n)) = b'_-(i^n), \sigma(b_+(i^n)) = b'_+(i^n), \sigma(b_-(i^n)) = b'_-(i^n). 
\]
Using (\ref{eq:demazure}), we obtain
\[
\tau_{w_n} b_+(i^n) = b_-(i^n) \tau_{w_n} = \tau_{w_n} b'_+(i^n) = b'_-(i^n) \tau_{w_n} = \tau_{w_n}. 
\]
It follows that $b_+(i^n), b_-(i^n), b'_+(i^n)$ and $b'_-(i^n)$ are all idempotents. 
Let $L(i^n)$ be the unique self-dual simple $R(n\alpha_i)$-module, and let $P(i^n)$ be its projective cover. \index{$L(i^n)$} \index{$P(i^n)$}
All the four modules 
\begin{align*}
q_i^{n(n-1)/2} R(n\alpha_i)b_+(i^n), &\ q_i^{-n(n-1)/2} R(n\alpha_i)b_-(i^n), \\ 
q_i^{n(n-1)/2}R(n\alpha_i)b'_+(i^n), &\ q_i^{-n(n-1)/2}R(n\alpha_i)b'_-(i^n), 
\end{align*}
are projective covers of $L(i^n)$. 
Hence, they are mutually isomorphic. 

\begin{lemma} \label{lem:projisom}
The left multiplication by $\mathbf{x}_n$ gives an isomorphism 
\[
q_i^{n(n-1)/2}b'_-(i^n)R(n\alpha_i) \to q_i^{-n(n-1)/2} b_+(i^n)R(n\alpha_i), 
\]
and the left multiplication by $\tau_{w_n}$ gives an isomorphism 
\[
q_i^{-n(n-1)/2}b_+(i^n)R(n\alpha_i) \to q_i^{n(n-1)/2} b'_-(i^n)R(n\alpha_i).
\] 
Furthermore, they are inverse to each other.  
\end{lemma}

\begin{proof}
We suppress degree shifts. 
Note that 
\[
b_+(i^n)R(n\alpha_i) = \mathbf{x}_n \tau_{w_n} \mathbf{k}[x_1, \ldots, x_n], b'_-(i^n) R(n\alpha_i) = \tau_{w_n} \mathbf{k}[x_1,\ldots,x_n], 
\]
where the second equality follows from 
\[
b'_-(i^n)\tau_w f(x_1, \ldots, x_n) = (-1)^{n(n-1)/2} \tau_{w_n} \partial_w(\mathbf{x}_n) f.
\] 
Using $\tau_{w_n} \mathbf{x}_n \tau_{w_n} = \tau_{w_n}$, we see that the two morphisms are well-defined and mutually inverse. 
\end{proof}

\subsection{Categorification theorem} \label{subsec:categorification}

\begin{theorem}[\cite{MR2525917, MR2763732}] \label{thm:categorification1}
We have an isomorphism of $\mathbb{Q}(q)$-algebras
\[
K(\gmod{R})_{\mathbb{Q}(q)} \to U_q^-(\mathfrak{g}), \ [L(i)] \mapsto (1-q_i^2) f_i \ (i \in I). 
\]
\end{theorem}

We may extend this result to $K(\gMod{R})$ as follows:  

\begin{theorem} \label{thm:categorification2}
We have an isomorphism of $\mathbb{Q}(q)$-algebras
\[
\chi \colon K(\gMod{R})_{\mathbb{Q}(q)} \to U_q^-(\mathfrak{g}), 
\] \index{$\chi \colon K(\gMod{R})_{\mathbb{Q}(q)} \to U_q^-(\mathfrak{g})$}
such that $\chi(R(\alpha_i)) = f_i$. 
Furthermore, the inclusions induce isomorphisms
\[
K(\gmod{R})_{\mathbb{Q}(q)} \xrightarrow{\sim} K(\gMod{R})_{\mathbb{Q}(q)} \xleftarrow{\sim} K_{\oplus}(\gproj{R})_{\mathbb{Q}(q)}. 
\]
\end{theorem}

\begin{proof}
Step 1. Surjectivity of $K(\gmod{R(\beta)})_{\mathbb{Q}(q)} \to K(\gMod{R(\beta)})_{\mathbb{Q}(q)} \ (\beta \in \mathsf{Q}_+)$.
We need to show that for any $M \in \gMod{R(\beta)}$, the element $[M]$ lies in the image of $K(\gmod{R(\beta)})_{\mathbb{Q}(q)}$. 
Recall that 
\[
Z(\beta) \simeq \bigotimes_{i \in I} \mathbf{k}[e_{i,1}, \ldots, e_{i,k_i}], 
\]
where $k_i$ is the coefficient of $\alpha_i$ in $\beta$, and $\deg e_{i,r} = r(\alpha_i,\alpha_i)$.
We fix a labeling $I = \{ i_1, \ldots, i_n\}$. 
Consider the following exact sequence: 
\[
0 \to \Ker \to q_{i_1}^2 M \xrightarrow{e_{i_1,1} \times} M \to \Cok \to 0. 
\]
In $K(\gMod{R(\beta)})_{\mathbb{Q}(q)}$, we have 
\[
[M] = \frac{[\Cok] - [\Ker]}{1-q_{i_1}^2}. 
\]
Hence, the surjectivity is reduced to proving that $[\Cok]$ and $[\Ker]$ are in the image of $K(\gmod{R(\beta)})_{\mathbb{Q}(q)}$. 
Repeating this procedure for all $e_{i,r}$, we may assume that the center $Z(\beta)$ acts on $M$ trivially. 
Since $R(\beta)$ is finitely generated over $Z(\beta)$, $M$ is finite-dimensional. 
Now, it is obvious that $[M]$ is in the image of $K(\gmod{R(\beta)})$. 

Step 2. Injectivity of $K(\gmod{R(\beta)})_{\mathbb{Q}(q)} \to K(\gMod{R(\beta)})_{\mathbb{Q}(q)} \ (\beta \in \mathsf{Q}_+)$.
Let $\{ L_1, \ldots, L_r \}$ be a complete set of isomorphism classes of simple graded $R(\beta)$-modules up to grading shifts. 
For $1 \leq s \leq r$, let $P_s$ be the projective cover of $L_s$. 
Then, $\{ [L_1], \ldots, [L_r] \}$ (resp. $\{ [P_1], \ldots, [P_r]\}$) is a free $\mathbb{Z}[q,q^{-1}]$-basis of $K(\gmod{R(\beta)})$ (resp. $K_{\oplus} (\gproj{R(\beta)})$). 
Since the functors $\HOM_{R(\beta)} (P_s,*)$ are exact, we have a $\mathbb{Z}[q,q^{-1}]$-linear map
\begin{align*}
K(\gMod{R(\beta)}) &\to K(\gmod{R(\beta)})_{\mathbb{Z}((q))} \coloneq K(\gmod{R(\beta)}) \otimes_{\mathbb{Z}[q,q^{-1}]} \mathbb{Z}((q)), \\
[M] &\mapsto \sum_{1 \leq s \leq r} (\qdim \HOM_{R(\beta)}(P_s, M)) [L_s]. 
\end{align*}
Under this homomorphism, $[L_s]$ is sent to $[L_s]$ for every $1 \leq s \leq r$. 
Hence, the composition $K(\gmod{R}) \to K(\gMod{R}) \to K(\gmod{R(\beta)})_{\mathbb{Z}((q))}$ is injective, 
which proves the injectivity of $K(\gmod{R}) \to K(\gMod{R})$. 

Step 3. Isomorphism $K(\gproj{R})_{\mathbb{Q}(q)} \simeq K(\gMod{R})_{\mathbb{Q}(q)}$. 
For $i \in I$, we have a short exact sequence 
\[
0 \to q_i^2 R(\alpha_i) \to R(\alpha_i) \to L(i) \to 0. 
\]
Hence, $[R(\alpha_i)] = [L(i)]/(1-q_i^2)$ in $K(\gMod{R})_{\mathbb{Q}(q)}$. 
On the other hand, Theorem \ref{thm:categorification1} and Step 1 imply that the $\mathbb{Q}(q)$-algebra $K(\gMod{R})_{\mathbb{Q}(q)}$ is generated by $[L(i)] \ (i \in I)$.
Therefore, the morphism $K_{\oplus}(\gproj{R})_{\mathbb{Q}(q)} \to K(\gMod{R})_{\mathbb{Q}(q)}$ is surjective. 
By step 1 and 2, we have 
\begin{align*}
\dim_{\mathbb{Q}(q)} K_{\oplus}(\gproj{R(\beta)})_{\mathbb{Q}(q)} &= \dim_{\mathbb{Q}(q)} K(\gmod{R(\beta)})_{\mathbb{Q}(q)} \\
&= \dim_{\mathbb{Q}(q)} K(\gMod{R(\beta)}).
\end{align*}
Hence, the assertion follows. 

Step 4. Combined with Theorem \ref{thm:categorification1}, we deduce the theorem. 
\end{proof}

Note that the morphism $K(\gMod{R}(\beta)) \to K(\gmod{R(\beta)})_{\mathbb{Z}((q))}$ in Step 2 of the proof above yields the inverse of the isomorphism $K(\gmod{R(\beta)})_{\mathbb{Q}(q)} \to K(\gMod{R(\beta)})_{\mathbb{Q}(q)}$. 
It implies the following corollary.  

\begin{corollary}
Let $\beta \in \mathsf{Q}_+$, and $M \in \gMod{R(\beta)}$.  
We use the notation of Step 2 in the proof above. 
We have $\qdim \HOM_{R(\beta)}(P_s, M) \in \mathbb{Q}(q)$ for any $1 \leq s \leq r$, and 
\[
\chi(M) = \sum_{1 \leq s \leq r} (\qdim \HOM_{R(\beta)}(P_s,M)) \chi(L_s). 
\]
\end{corollary} 

\begin{definition}
For $M \in \gmod{R(\beta)}$, we define $R(\beta-\alpha_i)$-modules
\begin{align*}
E_i' M = e(i,\beta-\alpha_i) M, \ {E_i'}^* M = e(\beta -\alpha_i,i)M
\end{align*}
\end{definition}

\begin{lemma} \label{lem:res}
For $M \in \gmod{R(\beta)}$, we have 
\[
\chi(E_i'M) = \frac{1}{1-q_i^2} {}_ir(\chi (M)), \ \chi({E_i'}^*M) = \frac{1}{1-q_i^2} r_i(\chi(M)). 
\]
\end{lemma}

\begin{proof}
By the Mackey-filtration (Proposition \ref{prop:Mackey}), we have a short exact sequence
\[
0 \to M \simeq L(i) \otimes M \to E_i'(L(i) \circ M) \to q^{-(\alpha_i,\alpha_i)} L(i) \circ E_i'M \to 0,  
\]
and isomorphisms
\[
E_i'(L(j) \circ M) \simeq q^{-(\alpha_i,\alpha_j)} L(j) \circ E_i'(M) \ (j \neq i). 
\]
Hence, 
\[
\chi(E_i'(L(j) \circ M)) = q^{-(\alpha_i,\alpha_j)} \chi(L(j)) \chi(E_i'M) + \delta_{i,j} \chi(M) \ (j \in I). 
\]

On the other hand, we have  
\[
{}_ir(f_jx) = q^{-(\alpha_i,\alpha_j)}f_j ({}_ir(x)) + \delta_{i,j}x 
\]
for $j \in I$ and $x \in U_q^-(\mathfrak{g})$ \cite[1.2.13]{MR2759715}. 
It implies 
\[
\frac{1}{1-q_i^2} {}_ir((1-q_j^2)f_jx) = q^{-(\alpha_i,\alpha_j)}(1-q_j^2)f_j \frac{1}{1-q_i^2} {}_ir(x) + \delta_{i,j} x. 
\]
Since $\chi(L(j)) = (1-q_j^2)f_j$ and $U_q^-(\mathfrak{g})$ is generated by $f_j \ (j \in I)$ as a $\mathbb{Q}(q)$-algebra, 
the first assertion follows. 
The second assertion is similar. 
\end{proof}

By \cite[Section 1]{MR2759715}, there exists a nondegenerate symmetric $\mathbb{Q}(q)$-bilinear form $(,)$ on $U_q^-(\mathfrak{g})$ determined by 
\[
(1,1) = 1, \ (f_ix, y) = \frac{1}{1-q_i^2}(x, {}_ir(y)), 
\]
for $x,y \in U_q^-(\mathfrak{g})$. 
Let $c$ be the $\mathbb{Q}$-linear automorphism of $U_q^-(\mathfrak{g})$ defined by 
\[
(c(x),y) = \overline{(x,\overline{y})} \ (x, y \in U_q^-(\mathfrak{g})).  
\]
For $x \in U_q^-(\mathfrak{g})_{-\beta}$ with $\beta = \sum_{i \in I} k_i \alpha_i \ (k_i \geq 0)$, 
our $c(x)$ is $\prod_{i \in I}(-q_i^2)^{k_i}$-multiple of $\sigma(x)$ in \cite[3.1]{MR2914878}. 
By \cite[Proposition 3.6]{MR2914878}, we have 
\[
c(xy) = q^{(\beta,\gamma)}c(y)c(x) \ (x \in U_q^-(\mathfrak{g})_{-\beta}, y \in U_q^-(\mathfrak{g})_{-\gamma}). 
\]
By definition, we have 
\[
c(1) = 1, c((1-q_i^2)f_i) = (1-q_i^2)f_i \ (i \in I).  
\]
Note that $c$ is uniquely determined by these properties. 

\begin{comment}
\begin{lemma} \label{lem:twist}
Let $\beta \in \mathbf{Q}_+$. 

(1) For $P \in \gproj{R(\beta)}$ and $M \in \gMod{R(\beta)}$, we have 
\[
(\overline{\chi{P}}, \chi(M)) = \qdim \HOM_{R(\beta)}(P,M). 
\]

(2) For $M \in \gmod{R(\beta)}$, we have 
\[
\chi(DM) = c(\chi(M)). 
\]
\end{lemma}

\begin{proof}
(1) For $P \in \gproj{R(\beta)}, M \in \gmod{R(\beta+\alpha_i)}$, we have 
\begin{align*}
\qdim \HOM_{R(\beta+\alpha_i)}(R(\alpha_i)\circ P,M) &= \qdim \HOM_{R(\alpha_i) \otimes R(\beta)} (R(\alpha_i) \otimes P, \Res_{\alpha_i,\beta}M) \\
&= \qdim \HOM_{R(\beta)} (P, E'_iM). 
\end{align*}
On the other hand, we have 
\begin{align*}
(\overline{\chi(R(\alpha_i) \circ P)}, \chi(M)) &= (f_i \overline{\chi(P)}, \chi(M))\\
&= \frac{1}{1-q_i^2}(\overline{\chi(P)}, {}_ir (\chi(M))) \\
&= (\overline{\chi(P)}, \overline{E'_iM}) \quad \text{by Lemma \ref{lem:res}}. 
\end{align*}
Hence, the assertion follows by induction from 
\[
(1,1) = 1 = \qdim \HOM_{R(0)}{\mathbf{1},\mathbf{1}}. 
\]

(2) It suffices to prove the equality for a self-dual simple module $M$. 
By definition, we have $DM \simeq M$. 
For any $P \in \gproj{R(\beta)}$, we have 
\begin{align*}
(\overline{\chi(P)}, \chi(DM)) &= \qdim \HOM_{R(\beta)}(P,M) \quad \text{by (1) and $DM \simeq M$}  \\
&= (\overlin{\chi(P)}, \chi(M))
\end{align*}
\end{proof}

\end{comment}

\begin{lemma} \label{lem:twist}
For $M \in \gmod{R}$, we have 
\[
\chi(DM) = c (\chi(M)). 
\]
\end{lemma}

\begin{proof}
It suffices to prove that $D$ satisfies properties that characterizes $c$ discussed above. 
For $M \in \gmod{R(\beta)}, N \in \gmod{R(\gamma)}$, we have 
\[
D(M \circ N) \simeq q^{(\beta,\gamma)} DN \circ DM, 
\] 
by \cite[Theorem 2.2]{MR2822211}. 
Furthermore, we have $DL(i) \simeq L(i)$ and $\chi(L(i)) = (1-q_i^2)f_i$ for any $i \in I$. 
Hence, the lemma follows. 
\end{proof}

\subsection{Categorified quantum groups}

In this section, we introduce the categorified parabolic quantum group $\mathcal{U}_q(\mathfrak{p}_J)$ for any subset $J$ of $I$. 
This is a parabolic analogue of $\mathcal{U}_q(\mathfrak{g})$ defined in \cite{MR3461059}.
It is known that $\catquantum{\mathfrak{g}}$ is isomorphic to both Khovanov-Lauda's 2-category \cite{MR2628852,MR3300416} by \cite[Theorem 2.1]{MR3461059}, 
and Rouquier's 2-category \cite{rouquier20082kacmoodyalgebras} by \cite{MR3451390}. 
When $J$ consists of a single element, 
$\mathcal{U}_q(\mathfrak{p}_J)$ is introduced in \cite[Section 2.1]{MR3709726}. 

\begin{definition} \label{def:bubbleparameter}
Let $Q$ be a choice of scalars. 
A choice of bubble parameters $C$ compatible with $Q$ consists of elements $c_{i,\lambda} \in \mathbf{k}^{\times} \ (i \in I, \lambda \in \mathsf{P})$ satisfying \index{$C = (c_{i,\lambda})_{i \in I, \lambda \in \mathsf{P}}$: bubble parameters}
\[
c_{i,\lambda+\alpha_j} / c_{i,\lambda} = t_{i,j}. 
\]
\end{definition}

Given $c_{i,\lambda}$ for every $i \in I$ and a representative $\lambda$ of every coset of $\mathsf{Q}$ in $\mathsf{P}$, 
we can extend it to a unique choice of bubble parameters compatible with $Q$.

\begin{definition} \label{def:catquantum}
 Let $J \subset I$. 
 Fix a choice of scalars $Q$ and a choice of bubble parameters $C$ compatible with $Q$. 
 Then the graded $\mathbf{k}$-linear 2-category $\mathcal{U}_q(\mathfrak{p}_J) = \mathcal{U}_q(\mathfrak{p}_J;Q,C)$ is defined as follows:  \index{$\catquantum{\mathfrak{p}_J}$}
 \begin{itemize}
 \item Objects are $\lambda \in \mathsf{P}$.
 \item 1-morphisms are formal direct sums of shifts of compositions of the generating 1-morphisms: \[
  1_\lambda, F_i 1_{\lambda} = 1_{\lambda -\alpha_i}F_i = 1_{\lambda-\alpha_i} F_i 1_{\lambda}, E_j1_{\lambda} = 1_{\lambda+\alpha_j} E_j = 1_{\lambda+\alpha_j}E_j 1_{\lambda}
  \]
 for $\lambda \in P, i \in I, j \in J$. 
 \item 2-morphisms are generated over $\mathbf{k}$ by compositions of shifts of the decorated tangle-like diagrams:
\begin{align*}
\xy 0;/r.17pc/: 
(0,0)*{\sdotu{j}}; (5,0)*{\lambda}; 
\endxy \colon q_j^2 E_j 1_{\lambda} \to E_j 1_{\lambda}, \quad 
&
\xy 0;/r.17pc/: (0,0)*{\sdotd{i}}; (5,0)*{\lambda};
\endxy \colon q_i^2 F_i 1_{\lambda} \to F_i 1_{\lambda}, \\
\xy 0;/r.17pc/: 
(0,0)*{\ucross{j}{j'}}; (9,0)*{\lambda}; 
\endxy \colon q^{-(\alpha_j,\alpha_{j'})} E_j E_{j'}1_{\lambda} \to E_{j'}E_j1_{\lambda}, \quad 
& 
\xy 0;/r.17pc/: 
(0,0)*{\dcross{i}{i'}}; (9,0)*{\lambda}; 
\endxy \colon q^{-(\alpha_i,\alpha_{i'})} F_i F_{i'}1_{\lambda} \to F_{i'}F_i1_{\lambda}, \\
\xy 0;/r.17pc/:
(0,0)*{\rcup{j}}; (9,0)*{\lambda}; 
\endxy \colon q_j^{1+\langle h_j, \lambda \rangle} 1_{\lambda} \to F_jE_j 1_{\lambda}, \quad
& 
\xy 0;/r.17pc/:
(0,0)*{\rcap{j}}; (9,0)*{\lambda}; 
\endxy \colon q_j^{1-\langle h_j, \lambda \rangle} E_jF_j1_{\lambda} \to 1_{\lambda}, \\
\xy 0;/r.17pc/:
(0,0)*{\lcup{j}}; (9,0)*{\lambda}; 
\endxy \colon q_j^{1-\langle h_j, \lambda \rangle} 1_{\lambda} \to E_jF_j 1_{\lambda}, \quad
& 
\xy 0;/r.17pc/:
(0,0)*{\lcap{j}}; (9,0)*{\lambda}; 
\endxy \colon q_j^{1+\langle h_j, \lambda \rangle} F_jE_j1_{\lambda} \to 1_{\lambda}, \\
\end{align*}
for $i,i' \in I, j,j' \in J, \lambda \in \mathsf{P}$. 
\end{itemize}

Note that our shift functor $q$ is $\langle -1 \rangle$ in \cite{MR3461059}. 
We read 1-morphisms from right to left, and 2-morphisms from bottom to top. 
We write 
\begin{align*}
\xy 0;/r.17pc/: (0,0)*{\slineu{j}}; (5,0)*{\lambda} \endxy = \id_{E_j1_{\lambda}}, \quad 
&
\xy 0;/r.17pc/: (0,0)*{\slined{i}}; (5,0)*{\lambda} \endxy = \id_{F_i1_{\lambda}}, \\ 
\xy 0;/r.17pc/: (0,0)*{\sdotu{j}}; (-3,0)*{\scriptstyle n}; (5,0)*{\lambda}; \endxy = \left( \xy 0;/r.17pc/: (0,0)*{\sdotu{j}}; (5,0)*{\lambda} \endxy \right)^n,  \quad 
& 
\xy 0;/r.17pc/: (0,0)*{\sdotd{i}}; (-3,0)*{\scriptstyle n}; (5,0)*{\lambda}; \endxy = \left( \xy 0;/r.17pc/: (0,0)*{\sdotd{i}}; (5,0)*{\lambda} \endxy \right)^n, \\
\xy 0;/r.17pc/: (0,0)*{\lcross{i}{j}}; (9,0)*{\lambda} \endxy = \xy 0;/r.17pc/:
(0,0)*{\xybox{
(0,0)*{\dcross{}{}};
(8,7)*{\lcap{}};
(-8,-7)*{\lcup{}};
(-4,8)*{\slined{}};
(4,-8)*{\slined{i}};
(12,-12); (12,4) **\dir{-} ?(1)*\dir{>};
(12,-14)*{\scriptstyle j};
(-12,-4); (-12,12) **\dir{-} ?(1)*\dir{>}; 
(17,0)*{\lambda};
}}\endxy, \quad
&
\xy 0;/r.17pc/: (0,0)*{\rcross{j}{i}}; (9,0)*{\lambda}; \endxy = \xy 0;/r.17pc/:
(0,0)*{\xybox{
(0,0)*{\dcross{}{}};
(-8,7)*{\rcap{}};
(8,-7)*{\rcup{}};
(4,8)*{\slined{}};
(-4,-8)*{\slined{i}};
(-12,-12); (-12,4) **\dir{-} ?(1)*\dir{>};
(-12,-14)*{\scriptstyle j};
(12,-4); (12,12) **\dir{-} ?(1)*\dir{>}; 
(17,0)*{\lambda};
}}\endxy 
\end{align*}
for $i \in I, j \in J, \lambda \in \mathsf{P}, n \in \mathbb{Z}_{\geq 0}$. 
We omit 1-morphisms ($\lambda$) from the diagram when it is obvious from the context, or when it is irrelevant to the computation. 

The following local relations are imposed on the 2-morphisms: 
\begin{enumerate}
\item Right and left adjunction ($j \in J, \lambda \in \mathsf{P}$): 
\begin{align*}
\xy 0;/r.17pc/:
(-4,3)*{\rcap{}};
(-8,-4)*{\slineu{j}};
(4,-3)*{\rcup{}};
(8,4)*{\slineu{}};
(12,0)*{\lambda};
\endxy \quad 
= 
\xy 0;/r.17pc/:
(0,-4)*{\sline{j}}; 
(0,4)*{\slineu{}}; 
(5,0)*{\lambda}; 
\endxy \quad 
= 
\xy 0;/r.17pc/:
(-8,4)*{\slineu{}};
(-4,-3)*{\lcup{}};
(4,3)*{\lcap{}};
(8,-4)*{\slineu{j}};
(12,0)*{\lambda};
\endxy, \quad 
\xy 0;/r.17pc/:
(-4,3)*{\lcap{}};
(-8,-4)*{\slined{j}};
(4,-3)*{\lcup{}};
(8,4)*{\slined{}};
(12,0)*{\lambda};
\endxy \quad 
= 
\xy 0;/r.17pc/:
(0,-4)*{\slined{j}}; 
(0,4)*{\sline{}}; 
(5,0)*{\lambda}; 
\endxy \quad 
= 
\xy 0;/r.17pc/:
(-8,4)*{\slined{}};
(-4,-3)*{\rcup{}};
(4,3)*{\rcap{}};
(8,-4)*{\slined{j}};
(12,0)*{\lambda};
\endxy. 
\end{align*}
\item Dot cyclicity($j \in J, \lambda \in \mathsf{P}$): 
\[
\xy 0;/r.17pc/:
(-4,3)*{\lcap{}};
(-8,-4)*{\slined{j}};
(4,-3)*{\lcup{}};
(8,4)*{\slined{}};
(12,0)*{\lambda};
(0,0)*{\bullet};
\endxy \quad 
= 
\xy 0;/r.17pc/:
(0,-4)*{\slined{j}}; 
(0,4)*{\sline{}}; 
(5,0)*{\lambda}; 
(0,0)*{\bullet};
\endxy \quad 
= 
\xy 0;/r.17pc/:
(-8,4)*{\slined{}};
(-4,-3)*{\rcup{}};
(4,3)*{\rcap{}};
(8,-4)*{\slined{j}};
(12,0)*{\lambda};
(0,0)*{\bullet};
\endxy. 
\] 
Hence, we can freely move the dots along the strands until they meet a crossing. 

\item Crossing cyclicity ($j,j' \in J, \lambda \in \mathsf{P}$): 
\[
\xy 0;/r.17pc/: 
(0,0)*{\dcross{j}{j'}}; 
(8,0)*{\lambda}
\endxy 
= 
\xy 0;/r.17pc/: 
(0,0)*{\xybox{(0,0)*{\ucross{}{}}; 
(-8,-7)*{\rcup{}}; 
(8,7)*{\rcap{}}; 
(-4,4); (20,4) **\crv{(-4,20) & (20,20)}; ?(1)*\dir{>};
(-20,-4); (4,-4) ** \crv{(-20,-20) & (4,-20)}; ?(1)*\dir{>}; 
(-20,22); (-20,-4) **\dir{-} ?(1)*\dir{>};
(-12,22); (-12,-4) **\dir{-} ?(1)*\dir{>};
(12,4); (12,-22) **\dir{-} ?(1)*\dir{>}+(0,-2)*{\scriptstyle j};
(20,4); (20,-22) **\dir{-} ?(1)*\dir{>}+(0,-2)*{\scriptstyle j'};
(25,0)*{\lambda}; 
(-25,0)*{}; 
}};
\endxy 
= 
\xy 0;/r.17pc/: 
(0,0)*{\xybox{(0,0)*{\ucross{}{}}; 
(8,-7)*{\lcup{}}; 
(-8,7)*{\lcap{}}; 
(4,4); (-20,4) **\crv{(4,20) & (-20,20)}; ?(1)*\dir{>};
(20,-4); (-4,-4) ** \crv{(20,-20) & (-4,-20)}; ?(1)*\dir{>}; 
(20,22); (20,-4) **\dir{-} ?(1)*\dir{>};
(12,22); (12,-4) **\dir{-} ?(1)*\dir{>};
(-12,4); (-12,-22) **\dir{-} ?(1)*\dir{>}+(0,-2)*{\scriptstyle j'};
(-20,4); (-20,-22) **\dir{-} ?(1)*\dir{>}+(0,-2)*{\scriptstyle j};
(25,0)*{\lambda}; 
(-25,0)*{}; 
}};
\endxy. 
\]
Note that (1) and (3) imply
\[
  \xy 0;/r.17pc/: (0,0)*{\lcross{j}{j'}}; (9,0)*{\lambda} \endxy = \xy 0;/r.17pc/:
  (0,0)*{\xybox{
  (0,0)*{\ucross{}{}};
  (8,-7)*{\lcup{}};
  (-8,7)*{\lcap{}};
  (-4,-8)*{\slineu{j'}};
  (4,8)*{\slineu{}};
  (12,12); (12,-4) **\dir{-} ?(1)*\dir{>};
  (-12,-14)*{\scriptstyle j};
  (-12,4); (-12,-12) **\dir{-} ?(1)*\dir{>}; 
  (17,0)*{\lambda};
  }}\endxy, \quad
  \xy 0;/r.17pc/: (0,0)*{\rcross{j}{j'}}; (9,0)*{\lambda}; \endxy = \xy 0;/r.17pc/:
  (0,0)*{\xybox{
  (0,0)*{\ucross{}{}};
  (-8,-7)*{\rcup{}};
  (8,7)*{\rcap{}};
  (4,-8)*{\slineu{j}};
  (-4,8)*{\slineu{}};
  (-12,12); (-12,-4) **\dir{-} ?(1)*\dir{>};
  (12,-14)*{\scriptstyle j'};
  (12,4); (12,-12) **\dir{-} ?(1)*\dir{>}; 
  (17,0)*{\lambda};
  }}\endxy.
\]
\item Quadratic KLR ($i, i' \in I, \lambda \in \mathsf{P}$): 
\[
\xy 0;/r.17pc/:
(0,4)*{\dcross{}{}};
(0,-4)*{\dcross{i}{i'}};
(9,0)*{\lambda};
\endxy
= Q_{i,i'} \left( 
\xy 0;/r.17pc/: 
(-4,8); (-4,-8) **\dir{-} ?(1)*\dir{>}+(0,-2)*{\scriptstyle i}; 
(4,8); (4,-8) **\dir{-} ?(1)*\dir{>}+(0,-2)*{\scriptstyle i'};
(9,0)*{\lambda};
(-4,0)*{\bullet};
\endxy, \quad 
\xy 0;/r.17pc/:
(-4,8); (-4,-8) **\dir{-} ?(1)*\dir{>}+(0,-2)*{\scriptstyle i}; 
(4,8); (4,-8) **\dir{-} ?(1)*\dir{>}+(0,-2)*{\scriptstyle i'};
(9,0)*{\lambda}; 
(4,0)*{\bullet}; 
\endxy
\right).
\]
\item Dot slide ($i, i' \in I, \lambda \in \mathsf{P}$): 
\begin{align*}
\xy 0;/r.17pc/: 
(0,0)*{\dcross{i}{i'}}; 
(2,1.8)*{\bullet}; 
(8,0)*{\lambda};
\endxy 
- 
\xy 0;/r.17pc/: 
(0,0)*{\dcross{i}{i'}}; 
(-2,-1.3)*{\bullet};
(8,0)*{\lambda};
\endxy
&= 
\xy 0;/r.17pc/: 
(0,0)*{\dcross{i}{i'}}; 
(2,-1.3)*{\bullet}; 
(8,0)*{\lambda};
\endxy
-
\xy 0;/r.17pc/: 
(0,0)*{\dcross{i}{i'}}; 
(-2,1.8)*{\bullet}; 
(8,0)*{\lambda}; 
\endxy \\
&= \begin{cases}
\xy 0;/r.17pc/: 
(-4,0)*{\slined{i}};
(4,0)*{\slined{i}}; 
(8,0)*{\lambda}; 
\endxy & \text {if $i = i'$}, \\
0 & \text {if $i \neq i'$}. 
\end{cases}
\end{align*}
\item Cubic KLR ($i, i', i'' \in I, \lambda \in \mathsf{P}$):
\[
\xy 0;/r.17pc/: 
(4,-8)*{\dcross{i'}{i''}}; 
(-4,0)*{\dcross{}{}};
(4,8)*{\dcross{}{}};
(-8,-8)*{\slined{i}};
(8,0)*{\slined{}};
(-8,8)*{\slined{}};
(12,0)*{\lambda};
\endxy
- 
\xy 0;/r.17pc/: 
(-4,-8)*{\dcross{i}{i'}}; 
(4,0)*{\dcross{}{}};
(-4,8)*{\dcross{}{}};
(8,-8)*{\slined{i''}};
(-8,0)*{\slined{}};
(8,8)*{\slined{}};
(12,0)*{\lambda};
\endxy
= 
\overline{Q}_{i,i',i''}\left( 
\xy 0;/r.17pc/: 
(-4,0)*{\slined{i}}; 
(0,0)*{\slined{i'}};
(4,0)*{\slined{i''}}; 
(-4,0)*{\bullet}; 
\endxy, 
\xy 0;/r.17pc/: 
(-4,0)*{\slined{i}}; 
(0,0)*{\slined{i'}};
(4,0)*{\slined{i''}}; 
(0,0)*{\bullet}; 
\endxy, 
\xy 0;/r.17pc/: 
(-4,0)*{\slined{i}}; 
(0,0)*{\slined{i'}};
(4,0)*{\slined{i''}}; 
(4,0)*{\bullet}; 
\endxy
\right) \lambda. 
\]
\item Mixed $EF$ ($i \in I, j \in J, i \neq j, \lambda \in \mathsf{P}$): 
\[
\xy 0;/r.17pc/:
(0,-4)*{\rcross{j}{i}};
(0,4)*{\lcross{}{}};
(8,0)*{\lambda}; 
\endxy 
= 
\xy 0;/r.17pc/:
(-4,-8); (-4,8) **\dir{-} ?(1)*\dir{>}; 
(4,8); (4,-8) **\dir{-} ?(1)*\dir{>}; 
(8,0)*{\lambda}; 
(-4,-10)*{\scriptstyle j};
(4,-10)*{\scriptstyle i};
\endxy, \quad 
\xy 0;/r.17pc/:
(0,-4)*{\lcross{i}{j}};
(0,4)*{\rcross{}{}};
(8,0)*{\lambda}; 
\endxy 
= 
\xy 0;/r.17pc/:
(4,-8); (4,8) **\dir{-} ?(1)*\dir{>}; 
(-4,8); (-4,-8) **\dir{-} ?(1)*\dir{>}; 
(8,0)*{\lambda}; 
(-4,-10)*{\scriptstyle i};
(4,-10)*{\scriptstyle j};
\endxy. 
\]
\item Bubble relations:
We introduce the following 2-morphisms, called fake bubbles. 
When $\langle h_j, \lambda \rangle -1 < 0$, we inductively define 
\begin{align*}
  &\xy 0;/r.17pc/:
  (0,3)*{\rcap{j}};
  (0,-3)*{\lcup{}};
  (0,-4)*{\bullet};
  (0,-6)*{\scriptstyle \langle h_j, \lambda \rangle -1 + m};
  (9,0)*{\lambda};
  \endxy  \\
  &= \begin{cases}
  0 & \text {if $ m < 0$}, \\
  c_{j,\lambda} \id_{1_{\lambda}} & \text{if $m = 0$}, \\
  -c_{j,\lambda} \sum\limits_{a \geq 0, b \geq 1, a + b = m} \xy 0;/r.17pc/:
  (0,3)*{\rcap{j}};
  (0,-3)*{\lcup{}};
  (0,-4)*{\bullet};
  (-3,-6)*{\scriptstyle \langle h_j, \lambda \rangle -1 + a};
  (9,0)*{\lambda};
  (18,3)*{\lcap{j}};
  (18,-3)*{\rcup{}};
  (18,-4)*{\bullet};
  (21,-6)*{\scriptstyle -\langle h_j, \lambda \rangle -1 + b};
  \endxy & \text{if $0 < m < -\langle h_j, \lambda \rangle +1$}.   
  \end{cases}
\end{align*}
When $-\langle h_j, \lambda \rangle -1 <0$, we inductively define 
\begin{align*}
  &\xy 0;/r.17pc/:
  (0,3)*{\lcap{j}};
  (0,-3)*{\rcup{}};
  (0,-4)*{\bullet};
  (0,-6)*{\scriptstyle -\langle h_j, \lambda \rangle -1 + m};
  (9,0)*{\lambda};
  \endxy \\
  &= \begin{cases}
  0 & \text {if $m < 0$}, \\
  c_{j,\lambda}^{-1} \id_{1_{\lambda}} & \text{if $m= 0$}, \\
  -c_{j,\lambda}^{-1} \sum\limits_{a \geq 1, b \geq 0, a + b = m} \xy 0;/r.17pc/:
  (0,3)*{\rcap{j}};
  (0,-3)*{\lcup{}};
  (0,-4)*{\bullet};
  (-3,-6)*{\scriptstyle \langle h_j, \lambda \rangle -1 + a};
  (9,0)*{\lambda};
  (18,3)*{\lcap{j}};
  (18,-3)*{\rcup{}};
  (18,-4)*{\bullet};
  (21,-6)*{\scriptstyle -\langle h_j, \lambda \rangle -1 + b};
  \endxy & \text{if $0 < m < \langle h_j, \lambda \rangle +1$}. 
  \end{cases}
\end{align*}
We impose the following relations for $j \in J, \lambda \in \mathsf{P}$: 
\begin{align*}
\xy 0;/r.17pc/:
(0,3)*{\rcap{j}};
(0,-3)*{\lcup{}};
(0,-4)*{\bullet};
(0,-6)*{\scriptstyle \langle h_j, \lambda \rangle -1 + m};
(9,0)*{\lambda};
\endxy 
&= \begin{cases}
c_{j,\lambda} \id_{\lambda} & \text{if $m= 0$}, \\
0 & \text{if $m < 0$},  
\end{cases} \\
\xy 0;/r.17pc/:
(0,3)*{\lcap{j}};
(0,-3)*{\rcup{}};
(0,-4)*{\bullet};
(0,-6)*{\scriptstyle -\langle h_j, \lambda \rangle -1 + m};
(9,0)*{\lambda};
\endxy 
&= \begin{cases}
c_{j,\lambda}^{-1} \id_{\lambda} & \text{if $m= 0$}, \\
0 & \text{if $m < 0$}. 
\end{cases}
\end{align*}
Note that we have 
\[
  \sum_{a + b = m} \xy 0;/r.17pc/:
  (0,3)*{\rcap{j}};
  (0,-3)*{\lcup{}};
  (0,-4)*{\bullet};
  (-3,-6)*{\scriptstyle \langle h_j, \lambda \rangle -1 + a};
  (9,0)*{\lambda};
  (18,3)*{\lcap{j}};
  (18,-3)*{\rcup{}};
  (18,-4)*{\bullet};
  (21,-6)*{\scriptstyle -\langle h_j, \lambda \rangle -1 + b};
  \endxy = \delta_{m,0}. 
\]
\item Extended $\mathfrak{sl}_2$ relations ($j \in J, \lambda \in \mathsf{P}$):
\begin{align*}
\xy 0;/r.17pc/:
(-4,-8); (-4,8) **\dir{-} ?(1)*\dir{>};
(-4,-10)*{\scriptstyle j};
(4,8); (4,-8) **\dir{-} ?(1)*\dir{>};
(4,-10)*{\scriptstyle j};
(8,0)*{\lambda};
\endxy
&= 
- \xy 0;/r.17pc/:
(0,-4)*{\rcross{j}{j}};
(0,4)*{\lcross{}{}};
(8,0)*{\lambda}; 
\endxy
+ 
\sum_{\substack{a,b,c \geq 0 \\ a+b+c = \langle h_j, \lambda \rangle -1}} \xy 0;/r.17pc/:
(0,3)*{\lcap{j}};
(0,-3)*{\rcup{}};
(2,-4)*{\bullet};
(10,-6)*{\scriptstyle -\langle h_j, \lambda \rangle -1 + b}; 
(0,12)*{\lcup{}};
(0,-12)*{\rcap{}};
(4,17)*{\scriptstyle j};
(-4,-17)*{\scriptstyle j};
(2,11)*{\bullet};
(-2,-11)*{\bullet};
(3,9)*{\scriptstyle a};
(-3,-9)*{\scriptstyle c};
(7,4)*{\lambda};
\endxy, \\
\xy 0;/r.17pc/:
(-4,8); (-4,-8) **\dir{-} ?(1)*\dir{>};
(-4,-10)*{\scriptstyle j};
(4,-8); (4,8) **\dir{-} ?(1)*\dir{>};
(4,-10)*{\scriptstyle j};
(8,0)*{\lambda};
\endxy
&= 
- \xy 0;/r.17pc/:
(0,-4)*{\lcross{j}{j}};
(0,4)*{\rcross{}{}};
(8,0)*{\lambda}; 
\endxy
+ 
\sum_{\substack{a,b,c \geq 0 \\ a+b+c = -\langle h_j, \lambda \rangle -1}} \xy 0;/r.17pc/:
(0,3)*{\rcap{j}};
(0,-3)*{\lcup{}};
(2,-4)*{\bullet};
(10,-6)*{\scriptstyle \langle h_j, \lambda \rangle -1 + b}; 
(0,12)*{\rcup{}};
(0,-12)*{\lcap{}};
(-4,17)*{\scriptstyle j};
(4,-17)*{\scriptstyle j};
(2,11)*{\bullet};
(-2,-11)*{\bullet};
(3,9)*{\scriptstyle a};
(-3,-9)*{\scriptstyle c};
(7,4)*{\lambda};
\endxy. 
\end{align*}
\end{enumerate}
\end{definition}

The 2-category $\catquantum{\mathfrak{p}_J}^{\mathrm{op}}$ is depicted as follows: 
it consists of the same diagrams as $\catquantum{\mathfrak{p}_J}$, we read 1-morphisms from left to right, and 2-morphisms from bottom to top. 
The 2-category $\catquantum{\mathfrak{p}_J}^{\mathrm{co}}$ is depicted as follows:
it consists of the same diagrams as $\catquantum{\mathfrak{p}_J}$, we read 1-morphisms from right to left, and 2-morphisms from top to bottom. 

\begin{proposition} \label{prop:scalarshift}
Let $J \subset I$, and fix a choice of scalars $Q$. 
Let $C, C'$ be choices of bubble parameters compatible with $Q$. 
Let $b_{i,i'} \in \mathbf{k}^\times \ (i,i' \in I), d_{j,\lambda} \in \mathbf{k}^{\times} \ (j \in J, \lambda \in \mathsf{P})$, 
and assume $b_{i,i'} b_{i',i} = b_{i,i} = 1 \ (i, i' \in I)$. 
Then, there exists an isomorphisms of 2-categories $\catquantum{\mathfrak{p}_J; Q,C} \to \catquantum{\mathfrak{p}_J, Q, C'}$ given as follows: 
\begin{itemize}
\item On objects and 1-morphisms, it is the identity. 
\item On the generating 2-morphisms, it is defined by
\begin{align*}
&\xy 0;/r.17pc/: 
(0,0)*{\sdotd{i}}; (5,0)*{\lambda};
\endxy \mapsto \xy 0;/r.17pc/: 
(0,0)*{\sdotd{i}}; (5,0)*{\lambda};
\endxy, \quad \xy 0;/r.17pc/: 
(0,0)*{\sdotu{j}}; (5,0)*{\lambda}; 
\endxy 
\mapsto 
\xy 0;/r.17pc/: 
(0,0)*{\sdotu{j}}; (5,0)*{\lambda}; 
\endxy,  \\
&\xy 0;/r.17pc/: 
(0,0)*{\dcross{i}{i'}}; (9,0)*{\lambda}; 
\endxy \mapsto b_{i,i'} \xy 0;/r.17pc/: 
(0,0)*{\dcross{i}{i'}}; (9,0)*{\lambda}; 
\endxy, \\
&\xy 0;/r.17pc/: 
(0,0)*{\ucross{j}{j'}}; (9,0)*{\lambda}; 
\endxy \mapsto b_{j,j'} d_{j,\lambda+\alpha_j}^{-1}d_{j,\lambda+\alpha_j+\alpha_{j'}}d_{j',\lambda + \alpha_{j'}}d_{j',\lambda +\alpha_j + \alpha_{j'}}^{-1} \xy 0;/r.17pc/: 
(0,0)*{\ucross{j}{j'}}; (9,0)*{\lambda}; 
\endxy, \\
&\xy 0;/r.17pc/:
(0,0)*{\rcup{j}}; (9,0)*{\lambda}; 
\endxy \mapsto 
d_{j,\lambda+\alpha_j}^{-1} \xy 0;/r.17pc/:
(0,0)*{\rcup{j}}; (9,0)*{\lambda}; 
\endxy, \quad \xy 0;/r.17pc/:
(0,0)*{\rcap{j}}; (9,0)*{\lambda}; 
\endxy 
\mapsto 
d_{j,\lambda} \xy 0;/r.17pc/:
(0,0)*{\rcap{j}}; (9,0)*{\lambda}; 
\endxy, \\
&\xy 0;/r.17pc/:
(0,0)*{\lcup{j}}; (9,0)*{\lambda};  
\endxy 
\mapsto 
d_{j,\lambda}^{-1} c_{j,\lambda}{c'}_{j,\lambda}^{-1} \xy 0;/r.17pc/:
(0,0)*{\lcup{j}}; (9,0)*{\lambda};  
\endxy, \quad \xy 0;/r.17pc/:
(0,0)*{\lcap{j}}; (9,0)*{\lambda}; 
\endxy 
\mapsto 
d_{j,\lambda+\alpha_j}c_{j,\lambda}^{-1}c'_{j,\lambda} \xy 0;/r.17pc/:
(0,0)*{\lcap{j}}; (9,0)*{\lambda}; 
\endxy. \\
\end{align*}
\end{itemize}
\end{proposition}

Note that when $C = C'$, it gives a nontrivial automorphism of $\catquantum{\mathfrak{p}_J;Q,C}$. 

\begin{proof}
The well-definedness is straightforward. 
The inverse is induced from $b_{i,i'}^{-1}, d_{j,\lambda}^{-1}$. 
\end{proof}

In the rest of this section, we fix $J \subset I$, a choice of scalars $Q$, and a choice of bubble parameters $C$ compatible with $Q$. 
The following two propositions are straightforward. 

\begin{proposition}[Chevalley involution] \label{prop:chevalleyinvolution} 
If $J = I$, there exits an isomorphism of graded 2-categories $\omega \colon \catquantum{\mathfrak{g}} \to \catquantum{\mathfrak{g}}^{\mathrm{op}}$ given as follows (see the end of Definition \ref{def:catquantum} for the description of $\mathcal{U}_q(\mathfrak{g})^{\mathrm{op}}$). \index{$\omega$ on $\catquantum{\mathfrak{g}}$}
\begin{itemize}
\item On the objects, it is the identity. 
\item On the generating 1-morphisms, it is defined by \[
1_{\lambda -\alpha_i} F_i 1_{\lambda} \mapsto 1_{\lambda} E_i 1_{\lambda-\alpha_i}, 1_{\lambda+\alpha_i}E_i 1_{\lambda} \mapsto 1_{\lambda} F_i 1_{\lambda+\alpha_i}. 
\]
\item On the 2-morphisms, it is defined by 
\begin{align*}
  \xy 0;/r.17pc/: 
  (0,0)*{\sdotd{i}}; (5,0)*{\lambda};
  \endxy \mapsto \xy 0;/r.17pc/: 
  (0,0)*{\sdotu{i}}; (-5,0)*{\lambda};
  \endxy, \quad 
  & 
  \xy 0;/r.17pc/: 
  (0,0)*{\sdotu{i}}; (5,0)*{\lambda}; 
  \endxy 
  \mapsto 
  \xy 0;/r.17pc/: 
  (0,0)*{\sdotd{i}}; (-5,0)*{\lambda}; 
  \endxy,  \\
  \xy 0;/r.17pc/: 
  (0,0)*{\dcross{i}{i'}}; (9,0)*{\lambda}; 
  \endxy \mapsto \xy 0;/r.17pc/: 
  (0,0)*{\ucross{i'}{i}}; (-9,0)*{\lambda}; 
  \endxy, \quad 
  &
  \xy 0;/r.17pc/: 
  (0,0)*{\ucross{i}{i'}}; (9,0)*{\lambda}; 
  \endxy \mapsto \xy 0;/r.17pc/: 
  (0,0)*{\dcross{i'}{i}}; (-9,0)*{\lambda}; 
  \endxy, \\
  \xy 0;/r.17pc/:
  (0,0)*{\rcup{i}}; (9,0)*{\lambda}; 
  \endxy \mapsto \xy 0;/r.17pc/:
  (0,0)*{\rcup{i}}; (-9,0)*{\lambda}; 
  \endxy, \quad
  & 
  \xy 0;/r.17pc/:
  (0,0)*{\rcap{i}}; (9,0)*{\lambda}; 
  \endxy \mapsto \xy 0;/r.17pc/:
  (0,0)*{\rcap{i}}; (-9,0)*{\lambda}; 
  \endxy, \\
  \xy 0;/r.17pc/:
  (0,0)*{\lcup{i}}; (9,0)*{\lambda};  
  \endxy \mapsto \xy 0;/r.17pc/:
  (0,0)*{\lcup{i}}; (-9,0)*{\lambda};  
  \endxy, \quad
  & 
  \xy 0;/r.17pc/:
  (0,0)*{\lcap{i}}; (9,0)*{\lambda}; 
  \endxy 
  \mapsto 
  \xy 0;/r.17pc/:
  (0,0)*{\lcap{i}}; (-9,0)*{\lambda}; 
  \endxy. \\
  \end{align*}
\end{itemize}
\end{proposition}

\begin{proposition} \label{prop:involutionsigma}
There exists an isomorphim of 2-categories $\sigma \colon \catquantum{\mathfrak{p}_J} \to \catquantum{\mathfrak{p}_J}^{\mathrm{op}}$ given as follows (see the end of Definition \ref{def:catquantum} for the description of $\mathcal{U}_q(\mathfrak{g})^{\mathrm{op}}$). \index{$\sigma$ on $\catquantum{\mathfrak{p}_J}$}
\begin{itemize}
\item On the objects, it is $\lambda \mapsto -\lambda$. 
\item On the generating 1-morphisms, it is defined by 
\[  
1_{\lambda-\alpha_i} F_i 1_{\lambda} \mapsto 1_{-\lambda}F_i 1_{-\lambda+\alpha_i}, 1_{\lambda+\alpha_j}E_j 1_{\lambda} \mapsto 1_{-\lambda} E_j 1_{-\lambda-\alpha_j}. 
\]
\item On the 2-morphisms, it is defined by
\begin{align*}
\xy 0;/r.17pc/: 
(0,0)*{\sdotd{i}}; (5,0)*{\lambda};
\endxy \mapsto \xy 0;/r.17pc/: 
(0,0)*{\sdotd{i}}; (-5,0)*{-\lambda};
\endxy, \quad 
& 
\xy 0;/r.17pc/: 
(0,0)*{\sdotu{j}}; (5,0)*{\lambda}; 
\endxy 
\mapsto 
\xy 0;/r.17pc/: 
(0,0)*{\sdotu{j}}; (-5,0)*{-\lambda}; 
\endxy,  \\
\xy 0;/r.17pc/: 
(0,0)*{\dcross{i}{i'}}; (9,0)*{\lambda}; 
\endxy \mapsto (-1)^{\delta_{i,i'}} \left( \xy 0;/r.17pc/: 
(0,0)*{\dcross{i'}{i}}; 
%(-4,6)*{\scriptstyle i};
%(4,6)*{\scriptstyle i'};
(-9,0)*{-\lambda}; 
\endxy \right), \quad 
&
\xy 0;/r.17pc/: 
(0,0)*{\ucross{j}{j'}}; (9,0)*{\lambda}; 
\endxy \mapsto (-1)^{\delta_{j,j'}} \left( \xy 0;/r.17pc/: 
(0,0)*{\ucross{j'}{j}}; 
%(-4,6)*{\scriptstyle j};
%(4,6)*{\scriptstyle j'};
(-9,0)*{-\lambda}; 
\endxy \right), \\
\xy 0;/r.17pc/:
(0,0)*{\rcup{j}}; (9,0)*{\lambda}; 
\endxy \mapsto c_{j,\lambda}^{-1}c_{j,-\lambda}^{-1} \xy 0;/r.17pc/:
(0,0)*{\lcup{j}}; (9,0)*{-\lambda}; 
\endxy, \quad
& 
\xy 0;/r.17pc/:
(0,0)*{\rcap{j}}; (9,0)*{\lambda}; 
\endxy \mapsto c_{j,\lambda}c_{j,-\lambda} \xy 0;/r.17pc/:
(0,0)*{\lcap{j}}; (9,0)*{-\lambda}; 
\endxy, \\
\xy 0;/r.17pc/:
(0,0)*{\lcup{j}}; (9,0)*{\lambda};  
\endxy \mapsto \xy 0;/r.17pc/:
(0,0)*{\rcup{j}}; (9,0)*{-\lambda};  
\endxy, \quad
& 
\xy 0;/r.17pc/:
(0,0)*{\lcap{j}}; (9,0)*{\lambda}; 
\endxy 
\mapsto 
\xy 0;/r.17pc/:
(0,0)*{\rcap{j}}; (9,0)*{-\lambda}; 
\endxy. \\
\end{align*}
\end{itemize} 
\end{proposition}

\begin{proof}
It is straightforward. 
It is essentially the same as the involution of \cite[3E2]{MR4732757} after applying some scalar shifts of Proposition \ref{prop:scalarshift}. 
\end{proof}

\begin{proposition} \label{prop:involution}
There exists an isomorphim of 2-categories $\psi \colon \catquantum{\mathfrak{p}_J} \to \catquantum{\mathfrak{p}_J}^{\mathrm{co}}$ given as follows (see the end of Definition \ref{def:catquantum} for the description of $\mathcal{U}_q(\mathfrak{g})^{\mathrm{co}}$). \index{$\psi$ on $\catquantum{\mathfrak{p}_J}$}
\begin{itemize}
\item On the objects, it is the identity. 
\item On the generating 1-morphisms, it is defined by 
\[  
q^n F_i 1_{\lambda} \mapsto q^{-n} F_i 1_{\lambda}, q^n E_j 1_{\lambda} \mapsto q^{-n} E_j 1_{\lambda}. 
\]
\item On the 2-morphisms, it is defined by
\begin{align*}
\xy 0;/r.17pc/: 
(0,0)*{\sdotd{i}}; (5,0)*{\lambda};
\endxy \mapsto \xy 0;/r.17pc/: 
(0,0)*{\sdotd{i}}; (5,0)*{\lambda};
\endxy, \quad 
& 
\xy 0;/r.17pc/: 
(0,0)*{\sdotu{j}}; (5,0)*{\lambda}; 
\endxy 
\mapsto 
\xy 0;/r.17pc/: 
(0,0)*{\sdotu{j}}; (5,0)*{\lambda}; 
\endxy,  \\
\xy 0;/r.17pc/: 
(0,0)*{\dcross{i}{i'}}; (9,0)*{\lambda}; 
\endxy \mapsto \xy 0;/r.17pc/: 
(0,0)*{\dcross{i'}{i}}; (9,0)*{\lambda}; 
\endxy, \quad 
&
\xy 0;/r.17pc/: 
(0,0)*{\ucross{j}{j'}}; (9,0)*{\lambda}; 
\endxy \mapsto \xy 0;/r.17pc/: 
(0,0)*{\ucross{j'}{j}}; (9,0)*{\lambda}; 
\endxy, \\
\xy 0;/r.17pc/:
(0,0)*{\rcup{j}}; (9,0)*{\lambda}; 
\endxy \mapsto \xy 0;/r.17pc/:
(0,0)*{\lcap{j}}; (9,0)*{\lambda}; 
\endxy, \quad
& 
\xy 0;/r.17pc/:
(0,0)*{\rcap{j}}; (9,0)*{\lambda}; 
\endxy \mapsto \xy 0;/r.17pc/:
(0,0)*{\lcup{j}}; (9,0)*{\lambda}; 
\endxy, \\
\xy 0;/r.17pc/:
(0,0)*{\lcup{j}}; (9,0)*{\lambda};  
\endxy \mapsto \xy 0;/r.17pc/:
(0,0)*{\rcap{j}}; (9,0)*{\lambda};  
\endxy, \quad
& 
\xy 0;/r.17pc/:
(0,0)*{\lcap{j}}; (9,0)*{\lambda}; 
\endxy 
\mapsto 
\xy 0;/r.17pc/:
(0,0)*{\rcup{j}}; (9,0)*{\lambda}; 
\endxy. \\
\end{align*}
\end{itemize} 
\end{proposition}

\begin{proof}
It is \cite[3E4]{MR4732757}. 
\end{proof}

\begin{theorem} \label{thm:Rouquierver}
(1) The 2-category $\catquantum{\mathfrak{p}_J}$ is canonically isomorphic to the graded $\mathbf{k}$-linear 2-category defined as follows.
\begin{itemize}
\item The objects and 1-morphisms are the same as those of $\catquantum{\mathfrak{p}_J}$. 
\item The 2-morphisms are generated over $\mathbf{k}$ by composition of shifts of the decorated tangle-like diagrams: 
\begin{align*}
\xy 0;/r.17pc/: (0,0)*{\sdotd{i}}; (5,0)*{\lambda};
\endxy \colon q_i^2 F_i 1_{\lambda} \to F_i 1_{\lambda}, 
&
\xy 0;/r.17pc/: 
(0,0)*{\dcross{i}{i'}}; (9,0)*{\lambda}; 
\endxy \colon q^{-(\alpha_i,\alpha_{i'})} F_i F_{i'}1_{\lambda} \to F_{i'}F_i1_{\lambda}, \\
\xy 0;/r.17pc/:
(0,0)*{\rcup{j}}; (9,0)*{\lambda}; 
\endxy \colon q_j^{1+\langle h_j, \lambda \rangle} 1_{\lambda} \to F_jE_j 1_{\lambda}, \quad
& 
\xy 0;/r.17pc/:
(0,0)*{\rcap{j}}; (9,0)*{\lambda}; 
\endxy \colon q_j^{1-\langle h_j, \lambda \rangle} E_jF_j1_{\lambda} \to 1_{\lambda}, \\
\end{align*}
for $i, i' \in I, j \in J, \lambda \in \mathsf{P}$.  
\end{itemize}
We write
\[
\xy 0;/r.17pc/: 
(0,0)*{\slineu{j}}; 
(5,0)*{\lambda};
\endxy = \id_{E_j1_{\lambda}}, \quad \xy 0;/r.17pc/: 
(0,0)*{\slined{i}}; 
(5,0)*{\lambda}; 
\endxy = \id_{F_i1_{\lambda}}, \quad \xy 0;/r.17pc/: 
(0,0)*{\rcross{j}{i}}; (9,0)*{\lambda}; 
\endxy = \xy 0;/r.17pc/:
(0,0)*{\xybox{
(0,0)*{\dcross{}{}};
(-8,7)*{\rcap{}};
(8,-7)*{\rcup{}};
(4,8)*{\slined{}};
(-4,-8)*{\slined{i}};
(-12,-12); (-12,4) **\dir{-} ?(1)*\dir{>};
(-12,-14)*{\scriptstyle j};
(12,-4); (12,12) **\dir{-} ?(1)*\dir{>}; 
(17,0)*{\lambda};
}}
\endxy
\]
for $i \in I, j \in J, \lambda \in \mathsf{P}$. 
The following local relations are imposed on the 2-morphisms: 
\begin{itemize}
\item Right adjunction ($j \in J, \lambda \in \mathsf{P}$): 
\[
  \xy 0;/r.17pc/:
  (-4,3)*{\rcap{}};
  (-8,-4)*{\slineu{j}};
  (4,-3)*{\rcup{}};
  (8,4)*{\slineu{}};
  (12,0)*{\lambda};
  \endxy \quad 
  = 
  \xy 0;/r.17pc/:
  (0,-4)*{\sline{j}}; 
  (0,4)*{\slineu{}}; 
  (5,0)*{\lambda}; 
  \endxy, \quad 
  \xy 0;/r.17pc/:
  (-8,4)*{\slined{}};
  (-4,-3)*{\rcup{}};
  (4,3)*{\rcap{}};
  (8,-4)*{\slined{j}};
  (12,0)*{\lambda};
  \endxy \quad =
  \xy 0;/r.17pc/:
  (0,-4)*{\slined{j}}; 
  (0,4)*{\sline{}}; 
  (5,0)*{\lambda}; 
  \endxy. 
\]
\item Quadratic KLR (Definition \ref{def:catquantum} (4)).
\item Dot slide (Definition \ref{def:catquantum} (5)). 
\item Cubic KLR (Definition \ref{def:catquantum} (6)). 
\item Formal inverse: the following 2-morphisms are isomorphisms, that is, 
there are some additional as yet unnamed generators that serve as two-sided inverses ($i \in I, j \in J, i \neq j, \lambda \in \mathsf{P}$): 
\begin{align*}
& \xy 0;/r.17pc/:
(0,0)*{\rcross{j}{i}};  
\endxy \colon E_j F_i 1_{\lambda} \to F_i E_j 1_{\lambda}, \\
& \begin{bmatrix}
\xy 0;/r.17pc/:
(0,0)*{\rcross{j}{j}}; 
\endxy & \xy 0;/r.17pc/:
(0,0)*{\rcap{j}}; 
\endxy & \xy 0;/r.17pc/:
(0,0)*{\rcap{j}}; 
(2,.5)*{\bullet};
\endxy & \cdots & \xy 0;/r.17pc/:
(0,0)*{\rcap{j}}; 
(2,.5)*{\bullet};
(11,1)*{\scriptstyle \langle h_j, \lambda \rangle -1};
\endxy
\end{bmatrix}^\top \colon \\
& E_j F_j 1_\lambda \to F_j E_j 1_{\lambda} \oplus q_j^{\langle h_j, \lambda \rangle -1}1_{\lambda} \oplus q_j^{\langle h_j, \lambda \rangle -3}1_{\lambda} \oplus \cdots \oplus q_j^{-\langle h_j,\lambda \rangle +1}1_{\lambda}  \ \text{if $\langle h_j, \lambda \rangle \geq 0$}, \\
& \begin{bmatrix}
\xy 0;/r.17pc/:
(0,0)*{\rcross{j}{j}}; 
\endxy & \xy 0;/r.17pc/:
(0,0)*{\rcup{j}}; 
(-2,-0.5)*{\bullet};
(-12,-1)*{\scriptstyle -\langle h_j, \lambda \rangle - 1};
\endxy & \cdots & \xy 0;/r.17pc/:
(0,0)*{\rcup{j}};
(-2,-0.5)*{\bullet};
\endxy & \xy 0;/r.17pc/:
(0,0)*{\rcup{j}}; 
\endxy
\end{bmatrix}\colon \\
& E_j F_j 1_{\lambda} \oplus q_j^{-\langle h_j, \lambda \rangle -1}1_{\lambda} \oplus \cdots \oplus q_j^{\langle h_j, \lambda \rangle +3}1_{\lambda} \oplus q_j^{\langle h_j,\lambda \rangle +1}1_{\lambda} \to F_j E_j 1_{\lambda} \ \text{if $\langle h_j, \lambda \rangle \leq 0$}. 
\end{align*}
\end{itemize}
(2) The 2-category $\catquantum{\mathfrak{p}_J}$ is canonically isomorphic to the graded $\mathbf{k}$-linear 2-category defined as follows.
\begin{itemize}
\item The objects and 1-morphisms are the same as those of $\catquantum{\mathfrak{p}_J}$. 
\item The 2-morphisms are generated over $\mathbf{k}$ by composition of shifts of the decorated tangle-like diagrams: 
\begin{align*}
\xy 0;/r.17pc/: (0,0)*{\sdotd{i}}; (5,0)*{\lambda};
\endxy \colon q_i^2 F_i 1_{\lambda} \to F_i 1_{\lambda}, 
&
\xy 0;/r.17pc/: 
(0,0)*{\dcross{i}{i'}}; (9,0)*{\lambda}; 
\endxy \colon q^{-(\alpha_i,\alpha_{i'})} F_i F_{i'}1_{\lambda} \to F_{i'}F_i1_{\lambda}, \\
\xy 0;/r.17pc/:
(0,0)*{\lcup{j}}; (9,0)*{\lambda}; 
\endxy \colon q_j^{1-\langle h_j, \lambda \rangle} 1_{\lambda} \to E_jF_j 1_{\lambda}, \quad
& 
\xy 0;/r.17pc/:
(0,0)*{\lcap{j}}; (9,0)*{\lambda}; 
\endxy \colon q_j^{1+\langle h_j, \lambda \rangle} F_jE_j1_{\lambda} \to 1_{\lambda}, \\
\end{align*}
for $i, i' \in I, j \in J, \lambda \in \mathsf{P}$.  
\end{itemize}
We write
\[
\xy 0;/r.17pc/: 
(0,0)*{\slineu{j}}; 
(5,0)*{\lambda};
\endxy = \id_{E_j1_{\lambda}}, \quad \xy 0;/r.17pc/: 
(0,0)*{\slined{i}}; 
(5,0)*{\lambda}; 
\endxy = \id_{F_i1_{\lambda}}, \quad \xy 0;/r.17pc/: (0,0)*{\lcross{i}{j}}; (9,0)*{\lambda} \endxy = \xy 0;/r.17pc/:
(0,0)*{\xybox{
(0,0)*{\dcross{}{}};
(8,7)*{\lcap{}};
(-8,-7)*{\lcup{}};
(-4,8)*{\slined{}};
(4,-8)*{\slined{i}};
(12,-12); (12,4) **\dir{-} ?(1)*\dir{>};
(12,-14)*{\scriptstyle j};
(-12,-4); (-12,12) **\dir{-} ?(1)*\dir{>}; 
(17,0)*{\lambda};
}}
\endxy
\]
for $i \in I, j \in J, \lambda \in \mathsf{P}$. 
The following local relations are imposed on the 2-morphisms: 
\begin{itemize}
\item Left adjunction ($j \in J, \lambda \in \mathsf{P}$): 
\[
\xy 0;/r.17pc/:
(-8,4)*{\slineu{}};
(-4,-3)*{\lcup{}};
(4,3)*{\lcap{}};
(8,-4)*{\slineu{j}};
(12,0)*{\lambda};
\endxy \quad 
= 
\xy 0;/r.17pc/:
(0,-4)*{\sline{j}}; 
(0,4)*{\slineu{}}; 
(5,0)*{\lambda}; 
\endxy, \quad 
\xy 0;/r.17pc/:
(-4,3)*{\lcap{}};
(-8,-4)*{\slined{j}};
(4,-3)*{\lcup{}};
(8,4)*{\slined{}};
(12,0)*{\lambda};
\endxy \quad 
= 
\xy 0;/r.17pc/:
(0,-4)*{\slined{j}}; 
(0,4)*{\sline{}}; 
(5,0)*{\lambda}; 
\endxy. 
\]
\item Quadratic KLR (Definition \ref{def:catquantum} (4)).
\item Dot slide (Definition \ref{def:catquantum} (5)). 
\item Cubic KLR (Definition \ref{def:catquantum} (6)). 
\item Formal inverse: the following 2-morphisms are isomorphisms, 
that is, there are some additional as yet unnamed generators that serve as two-sided inverses ($i \in I, j \in J, i \neq j, \lambda \in \mathsf{P}$): 
\begin{align*}
& \xy 0;/r.17pc/:
(0,0)*{\lcross{i}{j}};  
\endxy \colon F_i E_j 1_{\lambda} \to E_j F_i 1_{\lambda}, \\
& \begin{bmatrix}
\xy 0;/r.17pc/:
(0,0)*{\lcross{j}{j}}; 
\endxy & \xy 0;/r.17pc/:
(0,0)*{\lcap{j}}; 
\endxy & \xy 0;/r.17pc/:
(0,0)*{\lcap{j}}; 
(-2,.5)*{\bullet};
\endxy & \cdots & \xy 0;/r.17pc/:
(0,0)*{\lcap{j}}; 
(-2,.5)*{\bullet};
(-11,3)*{\scriptstyle -\langle h_j, \lambda \rangle -1};
\endxy
\end{bmatrix}^\top \colon \\
& F_j E_j 1_\lambda \to E_j F_j 1_{\lambda} \oplus q_j^{-\langle h_j, \lambda \rangle -1}1_{\lambda} \oplus q_j^{-\langle h_j, \lambda \rangle -3}1_{\lambda} \oplus \cdots \oplus q_j^{\langle h_j,\lambda \rangle +1}1_{\lambda} \ \text{if $\langle h_j,\lambda \rangle \leq 0$}, \\
& \begin{bmatrix}
\xy 0;/r.17pc/:
(0,0)*{\lcross{j}{j}}; 
\endxy & \xy 0;/r.17pc/:
(0,0)*{\lcup{j}}; 
(2,-0.5)*{\bullet};
(12,-1)*{\scriptstyle \langle h_j, \lambda \rangle -1};
\endxy & \cdots & \xy 0;/r.17pc/:
(0,0)*{\lcup{j}};
(2,-0.5)*{\bullet};
\endxy & \xy 0;/r.17pc/:
(0,0)*{\lcup{j}}; 
\endxy
\end{bmatrix}\colon \\
& F_j E_j 1_{\lambda} \oplus q_j^{\langle h_j, \lambda \rangle -1}1_{\lambda} \oplus \cdots \oplus q_j^{-\langle h_j, \lambda \rangle +3}1_{\lambda} \oplus q_j^{-\langle h_j,\lambda \rangle +1}1_{\lambda} \to E_j F_j 1_{\lambda} \ \text{if $\langle h_j, \lambda \rangle \geq 0$}. 
\end{align*}
\end{itemize}
\end{theorem}

\begin{proof}
(1) 
Let $\mathcal{U}'_q(\mathfrak{p}_J)$ be the 2-category defined in the theorem. 

First, assume $J = I$, hence $\mathfrak{p}_J = \mathfrak{g}$.
Then, $\mathcal{U}'_q(\mathfrak{p}_J)$ is the 2-category of Rouquier \cite{rouquier20082kacmoodyalgebras}.  
Although Rouquier's generating 2-morphisms are given by upward 2-morphisms rather than downward 2-morphisms, 
it is equivalent to ours by applying adjunction of $F_i$ and $E_i$. 
%By the following equation, we see that our definition is equivalent to Rouquier's one:
\begin{comment}
\[
\xy 0;/r.17pc/:
(0,-4)*{\slined{j}}; 
(0,4)*{\sline{}}; 
(5,0)*{\lambda}; 
(0,0)*{\bullet};
\endxy \quad 
= 
\xy 0;/r.17pc/:
(-8,4)*{\slined{}};
(-4,-3)*{\rcup{}};
(4,3)*{\rcap{}};
(8,-4)*{\slined{j}};
(12,0)*{\lambda};
(0,0)*{\bullet};
\endxy, \quad \xy 0;/r.17pc/: 
(0,0)*{\dcross{j}{j'}}; 
(8,0)*{\lambda}
\endxy 
= 
\xy 0;/r.17pc/: 
(0,0)*{\xybox{(0,0)*{\ucross{}{}}; 
(-8,-7)*{\rcup{}}; 
(8,7)*{\rcap{}}; 
(-4,4); (20,4) **\crv{(-4,20) & (20,20)}; ?(1)*\dir{>};
(-20,-4); (4,-4) ** \crv{(-20,-20) & (4,-20)}; ?(1)*\dir{>}; 
(-20,22); (-20,-4) **\dir{-} ?(1)*\dir{>};
(-12,22); (-12,-4) **\dir{-} ?(1)*\dir{>};
(12,4); (12,-22) **\dir{-} ?(1)*\dir{>}+(0,-2)*{\scriptstyle j};
(20,4); (20,-22) **\dir{-} ?(1)*\dir{>}+(0,-2)*{\scriptstyle j'};
(25,0)*{\lambda}; 
(-25,0)*{}; 
}};
\endxy. 
\] 
\end{comment}
By \cite{MR3451390}, Rouquier's 2-category is isomorphic to Khovanov-Lauda's one. 
Moreover, \cite{MR3461059} shows that it is isomorphic to our $\catquantum{\mathfrak{g}}$.
(In \cite{MR3461059}, Khovanov-Lauda's 2-category is denoted by $\mathcal{U}_Q(\mathfrak{g})$, and our $\catquantum{\mathfrak{g}}$ is denoted by $\mathcal{U}_Q^{\mathrm{cyc}}(\mathfrak{g})$.)
We obtain an isomorphism $\mathcal{U}'(\mathfrak{g}) \to \catquantum{\mathfrak{g}}$ given as follows: 
\begin{align*}
\xy 0;/r.17pc/: (0,0)*{\sdotd{i}}; (5,0)*{\lambda};
\endxy 
\mapsto 
\xy 0;/r.17pc/: (0,0)*{\sdotd{i}}; (5,0)*{\lambda};
\endxy, 
&
\xy 0;/r.17pc/: 
(0,0)*{\dcross{i}{i'}}; (9,0)*{\lambda}; 
\endxy 
\mapsto 
t_{i,i'}t_{i',i}^{-1} \xy 0;/r.17pc/: 
(0,0)*{\dcross{i}{i'}}; (9,0)*{\lambda}; 
\endxy , \\
\xy 0;/r.17pc/:
(0,0)*{\rcup{j}}; (9,0)*{\lambda}; 
\endxy 
\mapsto 
c_{j,\lambda} \xy 0;/r.17pc/:
(0,0)*{\rcup{j}}; (9,0)*{\lambda}; 
\endxy, \quad
& 
\xy 0;/r.17pc/:
(0,0)*{\rcap{j}}; (9,0)*{\lambda}; 
\endxy 
\mapsto
c_{j,\lambda}^{-1} \xy 0;/r.17pc/:
(0,0)*{\rcap{j}}; (9,0)*{\lambda}; 
\endxy.  \\
\end{align*}
Here, the scalars are determined by \cite[(1.13)]{MR3451390} and \cite[(2.1)]{MR3461059}. 
Let $b_{i',i} = t_{i,i'}^{-1}t_{i',i} \ (i,i' \in I)$ and $d_{j,\lambda} = c_{j,\lambda}$. 
By Proposition \ref{prop:scalarshift}, they induce an automorphism $\catquantum{\mathfrak{g}} \to \catquantum{\mathfrak{g}}$. 
By post-composing it with the isomorphism $\mathcal{U}'_q(\mathfrak{g}) \to \catquantum{\mathfrak{g}}$ above, we obtain an isomorphism $\mathcal{U}'_q(\mathfrak{g}) \to \catquantum{\mathfrak{g}}$ given by
\begin{align*}
\xy 0;/r.17pc/: (0,0)*{\sdotd{i}}; (5,0)*{\lambda};
\endxy 
\mapsto 
\xy 0;/r.17pc/: (0,0)*{\sdotd{i}}; (5,0)*{\lambda};
\endxy, 
&
\xy 0;/r.17pc/: 
(0,0)*{\dcross{i}{i'}}; (9,0)*{\lambda}; 
\endxy 
\mapsto 
\xy 0;/r.17pc/: 
(0,0)*{\dcross{i}{i'}}; (9,0)*{\lambda}; 
\endxy , \\
\xy 0;/r.17pc/:
(0,0)*{\rcup{j}}; (9,0)*{\lambda}; 
\endxy 
\mapsto 
\xy 0;/r.17pc/:
(0,0)*{\rcup{j}}; (9,0)*{\lambda}; 
\endxy, \quad
& 
\xy 0;/r.17pc/:
(0,0)*{\rcap{j}}; (9,0)*{\lambda}; 
\endxy 
\mapsto 
\xy 0;/r.17pc/:
(0,0)*{\rcap{j}}; (9,0)*{\lambda}; 
\endxy. \\
\end{align*}
For general $J \subset I$, the proof is parallel: we only need to consider restricted 1-morphisms in the proof of \cite{MR3451390}. 

(2) is deduced from (1) by the involution $\psi$. 
\end{proof} 

\begin{comment}
\begin{corollary} \label{cor:trick}
Let $\lambda \in P, j \in J, G \in \catquantum{\mathfrak{p}_J}(\lambda,\lambda)$. 
For $f, g \in \HOM_{\catquantum{\mathfrak{p}_J}}(E_jF_j 1_{\lambda}, G)$, the following are equivalent: 
\begin{enumerate}
\item $f = g$, 
\item \[
f \circ \left(\xy 0;/r.12pc/: 
(0,0)*{\lcross{j}{j}};
\endxy \right) = g \circ \left(\xy 0;/r.12pc/:
(0,0)*{\lcross{j}{j}}; 
\endxy \right), f \circ \left( \xy 0;/r.12pc/:
(0,0)*{\lcup{j}};
(2,-0.5)*{\bullet};
(7,-1)*{\scriptstyle n};
\endxy \right) = g \circ \left( \xy 0;/r.12pc/:
(0,0)*{\lcup{j}};
(2,-0.5)*{\bullet};
(7,-1)*{\scriptstyle n};
\endxy \right) \ (n \in \mathbb{Z}_{\geq 0}), 
\]
\end{enumerate}
where $\circ$ denote the composition of 2-morphisms. 
\end{corollary}

\begin{proof}
It follows from the formal inverse of Theorem \ref{thm:Rouquierver} (2). 
\end{proof}
\end{comment}

\begin{theorem} \label{thm:KLRaction}
Let $\lambda \in \mathsf{P}$ and $\beta \in \mathsf{Q}_+$. 
Put $n = \height \beta$. 
There exists an isomorphism of graded $\mathbf{k}$-algebras
\[
R(\beta) \otimes \mathbf{k}[\{z_{j,m}\mid j\in J, m \geq 1\}] \to \END_{\catquantum{\mathfrak{p}_J}}\left(\bigoplus_{\nu \in I^{\beta}} F_{\nu_1} \cdots F_{\nu_n}1_{\lambda} \right) 
\]
with $\deg z_{j,m} = m(\alpha_j,\alpha_j)$, given by 
\begin{align*}
e(\nu) &\mapsto \text{the projection to $F_{\nu_1} \cdots F_{\nu_n}1_{\lambda}$}, \\
x_k e(\nu) &\mapsto \xy 0;/r.17pc/:
(0,0)*{\slined{\nu_1}}; 
\endxy \cdots \xy 0;/r.17pc/:
(0,0)*{\sdotd{\nu_k}}; 
\endxy \cdots \xy 0;/r.17pc/:
(0,0)*{\slined{\nu_n}}; 
\endxy, \\
\tau_k e(\nu) &\mapsto \xy 0;/r.17pc/:
(0,0)*{\slined{\nu_1}}; 
\endxy \cdots \xy 0;/r.17pc/:
(0,0)*{\dcross{\nu_k}{\nu_{k+1}}}; 
\endxy \cdots \xy 0;/r.17pc/:
(0,0)*{\slined{\nu_n}}; 
\endxy, \\
z_{j,m}e(\nu) &\mapsto F_{\nu_1} \cdots F_{\nu_n} \xy 0;/r.17pc/:
(0,3)*{\rcap{j}}; 
(0,-3)*{\lcup{}};
(0,-4.5)*{\bullet};
(0,-6.5)*{\scriptstyle \langle h_j, \lambda \rangle -1 +m};
(9,0)*{\lambda};
\endxy. 
\end{align*}
\begin{comment}
& \text{if $\langle h_j, \lambda \rangle \geq 0$}, \\
F_{\nu_1} \cdots F_{\nu_n} \xy 0;/r.17pc/:
(0,3)*{\lcap{j}}; 
(0,-3)*{\rcup{}};
(0,-4.5)*{\bullet};
(0,-6.5)*{\scriptstyle -\langle h_j, \lambda \rangle -1 +m};
(9,0)*{\lambda} ;
\endxy & \text{if $\langle h_j, \lambda \rangle < 0$}. 
\end{comment}
We also have another isomorphism by sending $z_{j,m}e(\nu)$ to 
\[
\xy 0;/r.17pc/:
(0,3)*{\rcap{j}}; 
(0,-3)*{\lcup{}};
(0,-4.5)*{\bullet};
(0,-6.5)*{\scriptstyle \langle h_j, \lambda-\beta \rangle -1 +m};
(-12,0)*{\lambda-\beta};
\endxy F_{\nu_1} \cdots F_{\nu_n}. 
\]

Assume $\beta \in \sum_{j \in J} \mathbb{Z}_{\geq 0} \alpha_j$. 
Then, there exists an isomorphism of graded $\mathbf{k}$-algebras 
\[
R(\beta) \otimes \mathbf{k}[\{z_{j,m}\mid j\in J, m \geq 1\}] \to \END_{\catquantum{\mathfrak{p}_J}}\left(\bigoplus_{\nu \in I^{\beta}} E_{\nu_n} \cdots E_{\nu_1}1_{\lambda} \right)
\]
with $\deg z_{j,m} = m(\alpha_j,\alpha_j)$, given by 
\begin{align*}
e(\nu) &\mapsto \text{the projection to $E_{\nu_n} \cdots E_{\nu_1}1_{\lambda}$}, \\
x_k e(\nu) &\mapsto \xy 0;/r.17pc/:
(0,0)*{\slineu{\nu_n}}; 
\endxy \cdots \xy 0;/r.17pc/:
(0,0)*{\sdotu{\nu_{k}}}; 
\endxy \cdots \xy 0;/r.17pc/:
(0,0)*{\slineu{\nu_1}}; 
\endxy, \\
\tau_k e(\nu) &\mapsto \xy 0;/r.17pc/:
(0,0)*{\slineu{\nu_n}}; 
\endxy \cdots \xy 0;/r.17pc/:
(0,0)*{\ucross{}{}}; 
(-5.5,-6)*{\scriptstyle \nu_{k+1}};
(6,-6)*{\scriptstyle \nu_{k}};
\endxy \cdots \xy 0;/r.17pc/:
(0,0)*{\slineu{\nu_1}}; 
\endxy, \\
z_{j,m}e(\nu) &\mapsto 
E_{\nu_n} \cdots E_{\nu_1} \xy 0;/r.17pc/:
(0,3)*{\rcap{j}}; 
(0,-3)*{\lcup{}};
(0,-4.5)*{\bullet};
(0,-6.5)*{\scriptstyle \langle h_j, \lambda \rangle -1 +m};
(9,0)*{\lambda} ;
\endxy. 
\end{align*}
We also have another isomorphism by sending $z_{j,m}e(\nu)$ to 
\[
\xy 0;/r.17pc/:
(0,3)*{\rcap{j}}; 
(0,-3)*{\lcup{}};
(0,-4.5)*{\bullet};
(0,-6.5)*{\scriptstyle \langle h_j, \lambda+ \beta \rangle -1 +m};
(-12,0)*{\lambda+\beta};
\endxy E_{\nu_n} \cdots E_{\nu_1}. 
\]
\end{theorem}

\begin{comment}
& \text{if $\langle h_j, \lambda \rangle \geq 0$}, \\
E_{\nu_n} \cdots E_{\nu_1} \xy 0;/r.17pc/:
(0,3)*{\lcap{j}}; 
(0,-3)*{\rcup{}};
(0,-4.5)*{\bullet};
(0,-6.5)*{\scriptstyle -\langle h_j, \lambda \rangle -1 +m};
(9,0)*{\lambda} ;
\endxy & \text{if $\langle h_j, \lambda \rangle < 0$}. 
\end{cases}
\end{comment}

A careful reader will notice that we only need the surjectivity in this paper. 

\begin{proof}
It is immediate from the definitions that the homomorphisms described in the theorem are well-defined. 
The surjectivity is proved by the same discussion as in \cite[Section 8]{MR2729010}: see also \cite[Proposition 3.11]{MR2628852}. 
The injectivity is proved when $J = I$ in \cite[Theorem 3.6]{webster2024unfurlingkhovanovlaudarouquieralgebras}, \cite[Theorem 3.5.3]{MR4186573}.
The general case is deduced from it by considering the canonical 2-functor $\catquantum{\mathfrak{p}_J} \to \catquantum{\mathfrak{g}}$. 
\end{proof}

\begin{definition}
Let $\lambda \in \mathsf{P}, \beta \in \mathsf{Q}_+, f \in R(\beta)$ and $\nu, \nu' \in I^{\beta}$. 
The image of $e(\nu') f e(\nu)$ in $\HOM_{\catquantum{\mathfrak{p}_J}} (F_{\nu_1} \cdots F_{\nu_n}1_{\lambda}, F_{\nu'_1} \cdots F_{\nu'_n}1_{\lambda})$ under the homomorphism given in Theorem \ref{thm:KLRaction} is depicted by 
\[
\xy 0;/r.15pc/:
(-6,8); (-6,-8) **\dir{-} ?(1)*\dir{>}; 
(-2,8); (-2,-8) **\dir{-} ?(1)*\dir{>}; 
(6,8); (6,-8) **\dir{-} ?(1)*\dir{>}; 
(0,0)*{\fcolorbox{black}{white}{\quad $f$ \quad}}; 
(-6,11)*{\scriptstyle \nu'_1}; 
(-2,11)*{\scriptstyle \nu'_2}; 
(6,11)*{\scriptstyle \nu'_n}; 
(-6,-10)*{\scriptstyle \nu_1}; 
(-2,-10)*{\scriptstyle \nu_2}; 
(6,-10)*{\scriptstyle \nu_n}; 
(1,8); (4,8) **@{.};
(1,-8); (4,-8) **@{.};
\endxy. 
\]
Similarly, the image of $e(\nu')fe(\nu)$ in $\HOM_{\catquantum{\mathfrak{p}_J}} (E_{\nu_n} \cdots E_{\nu_1}1_{\lambda}, E_{\nu'_n} \cdots E_{\nu'_1}1_{\lambda})$ is depicted by 
\[
\xy 0;/r.15pc/:
(-6,8); (-6,-8) **\dir{-} ?(0)*\dir{<}; 
(2,8); (2,-8) **\dir{-} ?(0)*\dir{<}; 
(6,8); (6,-8) **\dir{-} ?(0)*\dir{<}; 
(0,0)*{\fcolorbox{black}{white}{\quad $f$ \quad}}; 
(-6,11)*{\scriptstyle \nu'_n}; 
(2,11)*{\scriptstyle \nu'_2}; 
(6,11)*{\scriptstyle \nu'_1}; 
(-6,-10)*{\scriptstyle \nu_n}; 
(2,-10)*{\scriptstyle \nu_2}; 
(6,-10)*{\scriptstyle \nu_1}; 
(-1,8); (-4,8) **@{.};
(-1,-8); (-4,-8) **@{.};
\endxy. 
\]
\end{definition}

\begin{example}
\[
\xy 0;/r.15pc/:
(-6,8); (-6,-8) **\dir{-} ?(0)*\dir{<}; 
(0,8); (0,-8) **\dir{-} ?(0)*\dir{<}; 
(6,8); (6,-8) **\dir{-} ?(0)*\dir{<}; 
(0,0)*{\fcolorbox{black}{white}{\quad $\tau_1x_1$ \quad}}; 
(-6,11)*{\scriptstyle \nu_3}; 
(0,11)*{\scriptstyle \nu_1}; 
(6,11)*{\scriptstyle \nu_2}; 
(-6,-10)*{\scriptstyle \nu_3}; 
(0,-10)*{\scriptstyle \nu_2}; 
(6,-10)*{\scriptstyle \nu_1}; 
\endxy = \xy 0;/r.15pc/: (0,0)*{\xybox{
(0,0)*{\ucross{}{}}; 
(-4,-8)*{\slineu{\nu_2}};
(4,-8)*{\sdotu{\nu_1}};
(-12,-12); (-12,4) **\dir{-} ?(1)*\dir{>};
(-12,-14)*{\scriptstyle \nu_3}; 
}};
\endxy. 
\]
\end{example}

\begin{definition}
Let $J \subset I$. 
Fix a choice of scalars $Q$ and a choice of bubble parameters $C$ compatible with $Q$. 
The graded $\mathbf{k}$-linear 2-category $\dotcatquantum{\mathfrak{p}_J}$ has objects $\lambda \in \mathsf{P}$ and its Hom-category $\dot{\mathcal{U}}_q(\mathfrak{p}_J)(\lambda,\mu)$ is defined as the Karoubi envelope of $\mathcal{U}_q(\mathfrak{p}_J)(\lambda,\mu)$. \index{$\dotcatquantum{\mathfrak{p}_J}$}
\end{definition}

For a category $\mathcal{C}$, $X \in \mathcal{C}$ and an idempotent $e \in \mathcal{C}(X,X)$, 
the endomorphism $e$ of $X$ is a projection to a direct summand in the Karoubi envelop of $\mathcal{C}$. 
Let $eX$ denote this direct summand. 
Then, we have a canonical epimorphism $X \to eX$ and a canonical monomorphism $eX \to X$. 

\begin{definition} \label{def:dividedpower}
Let $i \in I, j\in J, \lambda \in \mathsf{P}$ and $n \in \mathbb{Z}_{\geq 1}$. 
We define 1-morphisms of $\dotcatquantum{\mathfrak{p}_J}$ 
\begin{align*}
E_j^{(n)}1_{\lambda} = q_j^{-n(n-1)/2}b_+(j^n) E_j^n 1_{\lambda},& \ F_i^{(n)}1_{\lambda} = q_i^{n(n-1)/2} b_-(i^n)F_i^n 1_{\lambda}, \\
E_j^{(n)'}1_{\lambda} = q_j^{-n(n-1)/2} b'_+(j^n)E_j^n 1_{\lambda},& \  F_i^{(n)'}1_{\lambda} = q_i^{n(n-1)/2} b'_-(i^n)F_i^n 1_{\lambda},
\end{align*} \index{$E_i^{(n)}, F_i^{(n)}, E_i^{(n)'}, F_i^{(n)'}$}
where the $R(n\alpha_i)$-action on $F_i^n$ (resp. the $R(n\alpha_j)$-action on $E_j^n$) is given by Theorem \ref{thm:KLRaction}. 
\end{definition}

\begin{comment}
The left-multiplication morphisms $E_j^{(n)} 1_{\lambda} \xrightarrow{b'_{+,n} \times} E_j^{(n)'}1_{\lambda}$ and $E_j^{(n)'} 1_{\lambda} \xrightarrow{b_{+,n} \times} E_j^{(n)}1_{\lambda}$ are mutually inverse isomorphisms since $b_{+,n}b'_{+,n} = b_{+,n}, b'_{+,n} b_{+,n} = b'_{+,n}$, see (\ref{eq:demazure}). 
Similarly, the morphism $F_i^{(n)}1_{\lambda} \to F_i^{(n)'} 1_{\lambda}$ induced from $F_i^n \xrightarrow{b_{-,n}} F_i^n$,
and the morphism $F_i^{(n)'}1_{\lambda} \to F_i^{(n)} 1_{\lambda}$ induced from $F_i^n \xrightarrow{b'_{-,n}} F_i^n$ give mutually inverse isomorphisms, 
since $b_{-,n} b'_{-,n} = b'_{-,n}, b'_{-,n}b_{-,n} = b_{-,n}$. 
\end{comment}

There are isomorphisms of 1-morphisms in $\dotcatquantum{\mathfrak{p}_J}$
\begin{align*}
E_j^n 1_{\lambda} &\simeq (E_j^{(n)}1_{\lambda})^{\oplus [n]_j!} \simeq (E_j^{(n)'}1_{\lambda})^{\oplus [n]_j!}, \\
F_i^n 1_{\lambda} &\simeq (F_i^{(n)}1_{\lambda})^{\oplus [n]_i!} \simeq (F_i^{(n)'}1_{\lambda})^{\oplus [n]_i!}. 
\end{align*}

\begin{definition} \label{def:unitcounit}
Let $j \in J, \lambda \in \mathsf{P}$. 
We define 2-morphisms in $\dotcatquantum{\mathfrak{p}_J}$ 
\begin{align*}
\varepsilon = \varepsilon_n \colon & q_j^{n(n+\langle h_j, \lambda \rangle)}F_j^{(n)} E_j^{(n)}1_{\lambda} \xrightarrow{\mathrm{can}} q_j^{n(n+\langle h_j, \lambda \rangle)}F_j^n E_j^n 1_{\lambda} \xrightarrow{\text{$n$-layer $\xy 0;/r.12pc/: (0,0)*{\lcap{j}}; \endxy$}} 1_{\lambda}, \index{$\varepsilon$} \\
\eta = \eta_n \colon & q_j^{-n(n+\langle h_j, \lambda \rangle)}1_{\lambda+n\alpha_j} \xrightarrow{\text{$n$-layer $\xy 0;/r.12pc/: (0,0)*{\lcup{j}}; \endxy$}} E_j^nF_j^n1_{\lambda+n\alpha_j} \xrightarrow{\text{can}} E_j^{(n)}F_j^{(n)}1_{\lambda+n\alpha_j}.  \index{$\eta$}
\end{align*} 
Similarly, we define 2-morphisms
\begin{align*}
\varepsilon' = \varepsilon'_n \colon & q_j^{n(n-\langle h_j, \lambda \rangle)}E_j^{(n)'} F_j^{(n)'}1_{\lambda} \xrightarrow{\mathrm{can}} q_j^{n(n-\langle h_j, \lambda \rangle)}E_j^n F_j^n 1_{\lambda} \xrightarrow{\text{$n$-layer $\xy 0;/r.12pc/: (0,0)*{\rcap{j}}; \endxy$}} 1_{\lambda}, \index{$\varepsilon'$} \\
\eta' = \eta'_n \colon & q_j^{-n(n-\langle h_j, \lambda \rangle)}1_{\lambda-n\alpha_j} \xrightarrow{\text{$n$-layer $\xy 0;/r.12pc/: (0,0)*{\rcup{j}}; \endxy$}} F_j^nE_j^n1_{\lambda-n\alpha_j} \xrightarrow{\text{can}} F_j^{(n)'}E_j^{(n)'}1_{\lambda-n\alpha_j}. \index{$\eta'$}
\end{align*}
\end{definition}

\begin{lemma} \label{lem:adjunction}
$(q_j^{n(n+\langle h_j,\lambda \rangle)}F_j^{(n)}1_{\lambda+n\alpha_j}, E_j^{(n)}1_{\lambda})$ is an adjoint pair with unit $\varepsilon$ and counit $\eta$. 
Similarly, $(q_j^{n(n-\langle h_j, \lambda \rangle)}E_j^{(n)'}1_{\lambda-n\alpha_j}, F_j^{(n)'}1_{\lambda})$ is an adjoint pair with unit $\varepsilon'$ and counit $\eta'$. 
\end{lemma}

\begin{proof}
We suppress grading shifts in this proof. 
We have to prove the unit-counit identities. 
We prove that the morphism
\[
E_j^{(n)} 1_{\lambda} \xrightarrow{\eta} E_j^{(n)}F_j^{(n)}E_j^{(n)}1_{\lambda} \xrightarrow{\varepsilon} E_j^{(n)}1_{\lambda}
\]
is the identity.  
By the definition, it is 
\begin{align*}
&E_j^{(n)}1_{\lambda} \xrightarrow{\text{$n$-layer $\xy 0;/r.12pc/: (0,0)*{\lcup{j}}; \endxy$}} E_j^nF_j^n E_j^{(n)}1_{\lambda} \xrightarrow{\text{can}} E_j^{(n)}F_j^{(n)}E_j^{(n)}1_{\lambda} \\
&\xrightarrow{\text{can}} E_j^{(n)}F_j^nE_j^n1_{\lambda} \xrightarrow{\text{$n$-layer $\xy 0;/r.12pc/: (0,0)*{\lcap{j}}; \endxy$}} E_j^{(n)}1_{\lambda}. 
\end{align*}
To simplify the picture, we describe computation in the case of $n = 3$. 
By precomposing $E_j^3 1_{\lambda} \xrightarrow{\text{can}}E_j^{(3)}1_{\lambda}$ and postcomposing $E_j^{(3)}1_{\lambda} \xrightarrow{\text{can}}E_j^3 1_{\lambda}$, we obtain
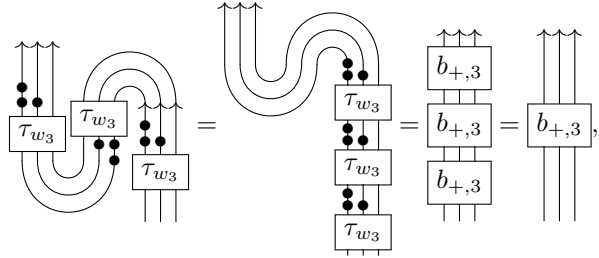
\begin{align*}
\xy 0;/r.12pc/: (0,0)*{\xybox{
(4,0); (12,0) **\crv{(4,6) & (12,6)}; 
(0,0); (16,0) **\crv{(0,12) & (16,12)};  
(-4,0); (20,0) **\crv{(-4,18) & (20,18)}; 
(-4,-15); (-12,-15) **\crv{(-4,-21) & (-12,-21)}; 
(0,-15); (-16,-15) **\crv{(0,-27) & (-16,-27)};  
(4,-15); (-20,-15) **\crv{(4,-33) & (-20,-33)}; 
(12,-31); (12,0) **\dir{-} ?(1)*\dir{>};
(16,-31); (16,0) **\dir{-} ?(1)*\dir{>};
(20,-31); (20,0) **\dir{-} ?(1)*\dir{>};
(-12,16); (-12,-15) **\dir{-} ?(0)*\dir{<};
(-16,16); (-16,-15) **\dir{-} ?(0)*\dir{<};
(-20,16); (-20,-15) **\dir{-} ?(0)*\dir{<};
(-4,0); (-4,-15) **\dir{-};
(0,0); (0,-15) **\dir{-};
(4,0); (4,-15) **\dir{-};
(12,-10)*{\bullet};
(12,-6)*{\bullet};
(16,-10)*{\bullet};
(16,-18)*{\fcolorbox{black}{white}{$\tau_{w_3}$}};
(0,-4)*{\fcolorbox{black}{white}{$\tau_{w_3}$}};
(4,-11)*{\bullet};
(4,-15)*{\bullet};
(0,-11)*{\bullet};
(-16,-8)*{\fcolorbox{black}{white}{$\tau_{w_3}$}};
(-16,0)*{\bullet};
(-20,0)*{\bullet};
(-20,4)*{\bullet};
}};
\endxy = \xy 0;/r.12pc/: (0,0)*{\xybox{
(4,0); (12,0) **\crv{(4,6) & (12,6)}; 
(0,0); (16,0) **\crv{(0,12) & (16,12)}; 
(-4,0); (20,0) **\crv{(-4,18) & (20,18)}; 
(-4,0); (-12,0) **\crv{(-4,-6) & (-12,-6)};
(0,0); (-16,0) **\crv{(0,-12) & (-16,-12)};
(4,0); (-20,0) **\crv{(4,-18) & (-20,-18)};
(12,-50); (12,0) **\dir{-}; 
(16,-50); (16,0) **\dir{-}; 
(20,-50); (20,0) **\dir{-}; 
(-12,16); (-12,0) **\dir{-} ?(0)*\dir{<};
(-16,16); (-16,0) **\dir{-} ?(0)*\dir{<};
(-20,16); (-20,0) **\dir{-} ?(0)*\dir{<};
(-4,0); (-4,0) **\dir{-} ?(1)*\dir{};
(0,0); (0,0) **\dir{-} ?(1)*\dir{};
(4,0); (4,0) **\dir{-} ?(1)*\dir{};
(12,0)*{\bullet};
(12,-3)*{\bullet};
(16,-3)*{\bullet};
(16,-10)*{\fcolorbox{black}{white}{$\tau_{w_3}$}};
(12,-17)*{\bullet};
(12,-20)*{\bullet};
(16,-20)*{\bullet};
(16,-27)*{\fcolorbox{black}{white}{$\tau_{w_3}$}};
(12,-34)*{\bullet};
(12,-37)*{\bullet};
(16,-37)*{\bullet};
(16,-44)*{\fcolorbox{black}{white}{$\tau_{w_3}$}}; 
}};
\endxy = \xy 0;/r.12pc/: (0,0)*{\xybox{
(12,-50); (12,0) **\dir{-} ?(1)*\dir{>}; 
(16,-50); (16,0) **\dir{-} ?(1)*\dir{>}; 
(20,-50); (20,0) **\dir{-} ?(1)*\dir{>}; 
(16,-10)*{\fcolorbox{black}{white}{$b_{+,3}$}}; 
(16,-25)*{\fcolorbox{black}{white}{$b_{+,3}$}}; 
(16,-40)*{\fcolorbox{black}{white}{$b_{+,3}$}}; 
}};
\endxy = \xy 0;/r.12pc/: (0,0)*{\xybox{
(12,-50); (12,0) **\dir{-} ?(1)*\dir{>}; 
(16,-50); (16,0) **\dir{-} ?(1)*\dir{>}; 
(20,-50); (20,0) **\dir{-} ?(1)*\dir{>}; 
(16,-25)*{\fcolorbox{black}{white}{$b_{+,3}$}}; 
}};
\endxy, 
\end{align*}
which coincides with the composition $E_j^3 1_{\lambda} \xrightarrow{\text{can}} E_j^{(3)}1_{\lambda} \xrightarrow{\text{can}} E_j^3 1_{\lambda}$. 
Since $E_j^3 1_{\lambda} \xrightarrow{\text{can}} E_j^{(3)} 1_{\lambda}$ is epi and $E_j^{(3)}1_{\lambda} \xrightarrow{\text{can}} E_j^3 1_{\lambda}$ is mono, the assertion follows. 
The remaining three identities can be proved in the same way. 
\end{proof}

\begin{theorem} \label{thm:EFrel}
Let $j \in J, \lambda \in P$ and $a, b \in \mathbb{Z}_{\geq 0}$. 
There are isomorphisms of 1-morphisms in $\dotcatquantum{\mathfrak{p}_J}$
\begin{align*}
F_j^{(b)} E_j^{(a)}1_{\lambda} &\simeq \bigoplus_{k=0}^{\min \{a,b\}} \begin{bmatrix} 
-a+b-\langle h_j,\lambda \rangle \\ 
k
\end{bmatrix}_j E_j^{(a-k)} F_j^{(b-k)} 1_{\lambda} \quad \text{if $-a+b-\langle h_j, \lambda \rangle \geq 0$}, \\
E_j^{(a)} F_j^{(b)}1_{\lambda} &\simeq \bigoplus_{k=0}^{\min \{a,b\}} \begin{bmatrix} 
a-b+\langle h_j,\lambda \rangle \\ 
k
\end{bmatrix}_j F_j^{(b-k)} E_j^{(a-k)} 1_{\lambda} \quad \text{if $a-b+\langle h_j, \lambda \rangle \geq 0$}. 
\end{align*}

Let $i \in I$ and assume $i \neq j$. 
Then, there is an isomorphism of 1-morphisms
\[
E_j F_i 1_{\lambda} \simeq F_i E_j 1_{\lambda}. 
\]
\end{theorem}

\begin{proof}
By considering the canonical 2-functor $\mathcal{U}_{q_j}(\mathfrak{sl}_2) \to \catquantum{\mathfrak{p}_J}$ associated with $j$, 
the first two isomorphisms are reduced to the isomorphisms in $\mathcal{U}_{q_j}(\mathfrak{sl}_2)$ proved in \cite[Lemma 4.14]{rouquier20082kacmoodyalgebras}, \cite[Theorem 9.6]{MR2729010}, \cite[Theorem 5.2.8]{MR2963085}. 
The last isomorphism follows from the mixed $EF$ relation (Definition \ref{def:catquantum} (7)). 
\end{proof}

\begin{corollary}[Triangular decomposition] \label{cor:triangular}
Let $\lambda, \mu \in \mathsf{P}$. 
Then, the additive category $\dotcatquantum{\mathfrak{p}_J}(\lambda,\mu)$ is generated by 
\begin{align*}
\{ &F_{i_1} \cdots F_{i_m} E_{j_1} \cdots E_{j_n}1_{\lambda} \ (\text{resp.} \ E_{j_1} \cdots E_{j_n} F_{i_1} \cdots F_{i_m} 1_{\lambda}) \mid \\ 
&m, n \geq 0, i_1, \ldots, i_m \in I, j_1, \ldots, j_n \in J, \\ 
&\mu = \lambda -\alpha_{i_1} - \cdots - \alpha_{i_m} + \alpha_{j_1} + \cdots + \alpha_{j_n} \}, 
\end{align*}
that is, every object of $\dotcatquantum{\mathfrak{p}_J}(\lambda,\mu)$ is a direct summand of a finite direct sum of the 1-morphisms listed above. 
\end{corollary}

\begin{proof}
It is immediate from Theorem \ref{thm:EFrel}
\end{proof}

\subsection{2-representations}

Let $\mathfrak{Lin}_{\mathbf{k}}$ be the 2-category of graded $\mathbf{k}$-linear categories.
For $\mathcal{A}, \mathcal{B} \in \mathfrak{Lin}_{\mathbf{k}}$, $F,G \in \mathfrak{Lin}_{\mathbf{k}}(\mathcal{A},\mathcal{B})$, a natural transformation $f \colon F \to G$ (a 2-morphisms in $\mathfrak{Lin}_{\mathbf{k}})$ and $X \in \mathcal{A}$, 
we write $fX \in \mathcal{B}(FX, GX)$ for the morphism given by $f$. 

When working in $\mathfrak{Lin}_{\mathbf{k}}^{\mathrm{op}}$, the following notation is convenient: 
for 
\[
F \in \mathfrak{Lin}_{\mathbf{k}}^{\mathrm{op}}(\mathcal{B},\mathcal{A}) = \mathfrak{Lin}_{\mathbf{k}}(\mathcal{A},\mathcal{B}),
\]
we write 
\begin{itemize}
\item $XF$ instead of $F(X)$ for $X \in \mathcal{A}$, 
\item $fF$ instead of $F(f)$ for $f \in \mathcal{A}(X,Y)$. 
\end{itemize}
Then, for $F \in \mathfrak{Lin}_{\mathbf{k}}^{\mathrm{op}}(\mathcal{B}, \mathcal{A}), G \in \mathfrak{Lin}_{\mathbf{k}}^{\mathrm{op}}(\mathcal{C},\mathcal{B})$, we have 
\[
(XF)G = X(FG) \ (X \in \mathcal{A}),\ (fF)G = f(FG) \ (f \in \mathcal{A}(X,Y)), 
\]
where $FG \in \mathfrak{Lin}_{\mathbf{k}}^{\mathrm{op}}(\mathcal{C}, \mathcal{A})$ is the composition of $G$ and $F$ in $\mathfrak{Lin}_{\mathbf{k}}^{\mathrm{op}}$. 
Similarly, for $F,G \in \mathfrak{Lin}_{\mathbf{k}}^{\mathrm{op}}(\mathcal{B},\mathcal{A})$, $f \colon F \to G$ and $X \in \mathcal{A}$, we write $Xf \in \mathcal{B}(XF,XG)$ for the morphism given by $f$. 

\begin{definition}
A left $\catquantum{\mathfrak{p}_J}$-module is a $\mathbf{k}$-linear 2-functor $\mathcal{V} \colon \catquantum{\mathfrak{p}_J} \to \mathfrak{Lin}_{\mathbf{k}}$.
This is equivalent to the data of 
\begin{itemize}
\item a family of graded $\mathbf{k}$-linear categories $\mathcal{V}_{\lambda} \ (\lambda \in \mathsf{P})$, 
\item a family of graded $\mathbf{k}$-linear functors $F_i \colon \mathcal{V}_{\lambda} \to \mathcal{V}_{\lambda -\alpha_i}, E_j \colon \mathcal{V}_{\lambda} \to \mathcal{V}_{\lambda + \alpha_j}$, 
\item a family of natural transformations corresponding to the generating 2-morphisms of $\catquantum{\mathfrak{p}_J}$ subject to the defining relations. 
\end{itemize}
We often identify $\mathcal{V}$ with the additive category $\bigoplus_{\lambda \in \mathsf{P}} \mathcal{V}_{\lambda}$. 

A right $\catquantum{\mathfrak{p}_J}$-module is a $\mathbf{k}$-linear 2-functor $\catquantum{\mathfrak{p}_J} \to \mathfrak{Lin}_{\mathbf{k}}^{\mathrm{op}}$. 

Modules over $\dotcatquantum{\mathfrak{p}_J}$ are defined in the same manner. 
\end{definition}

\begin{definition}
Let $\mathcal{V}, \mathcal{W} \colon \catquantum{\mathfrak{p}_J} \to \mathfrak{Lin}_{\mathbf{k}}$ be left $\catquantum{\mathfrak{p}_J}$-modules. 
A morphism of left $\catquantum{\mathfrak{p}_J}$-modules is a morphisms of 2-functors $\mathcal{V} \to \mathcal{W}$. 
This is equivalent to the data of 
\begin{itemize}
\item a family of graded $\mathbf{k}$-linear functors $\Theta_{\lambda} \colon \mathcal{V}_{\lambda} \to \mathcal{W}_{\lambda} \ (\lambda \in \mathsf{P})$, 
\item a family of natural isomorphisms $\Theta_{\mu} \mathcal{V}(G) \simeq \mathcal{W}(G) \Theta_{\lambda} \ (G \in \catquantum{\mathfrak{p}_J}(\lambda,\mu))$,
\end{itemize}
subject to some coherence conditions (\cite[Definition 2.3]{rouquier20082kacmoodyalgebras}). 
\end{definition}

\begin{definition}
Let $\Lambda \in \mathsf{P}$.
We define a left $\catquantum{\mathfrak{p}_J}$-module $\catquantum{\mathfrak{p}_J}1_{\Lambda}$ by 
\[
(\catquantum{\mathfrak{p}_J}1_{\Lambda})_{\lambda} = 1_{\lambda}\catquantum{\mathfrak{p}_J}1_{\Lambda} = \catquantum{\mathfrak{p}_J}(\Lambda,\lambda) \ (\lambda \in \mathsf{P}).   
\]
A left $\dotcatquantum{\mathfrak{p}_J}$-module $\dotcatquantum{\mathfrak{p}_J}1_{\Lambda}$, a right $\catquantum{\mathfrak{p}_J}$-module $1_{\Lambda}\catquantum{\mathfrak{p}_J}$, and a right $\dotcatquantum{\mathfrak{p}_J}$-module $1_{\Lambda}\dotcatquantum{\mathfrak{p}_J}$ are defined in the same manner. 
\end{definition}

\begin{lemma} \label{lem:universal2rep}
Let $\Lambda \in \mathsf{P}$. 
Let $\mathcal{V}$ be a left $\catquantum{\mathfrak{p}_J}$-module, and take $X_{\Lambda} \in \mathcal{V}_{\Lambda}$. 
Then, there exists a morphism of left $\catquantum{\mathfrak{p}_J}$-modules
\[
\catquantum{\mathfrak{p}_J}1_{\Lambda} \to \mathcal{V}, 1_{\Lambda} \mapsto X_{\Lambda}. 
\]
Moreover, such a morphism is unique up to equivalence. 
\end{lemma}

\begin{proof}
The morphism is given by 
\begin{align*}
1_{\lambda}\catquantum{\mathfrak{p}_J} 1_{\Lambda} &\to \mathcal{V}_{\lambda} \ (\lambda \in \mathsf{P}), \\
G &\mapsto GX_{\Lambda} \ (G \in 1_{\lambda}\catquantum{\mathfrak{p}_J}1_{\Lambda}), \\ 
f &\mapsto (fX_{\Lambda}\colon GX \to G'X) \ (G,G' \in 1_{\lambda}\catquantum{\mathfrak{p}_J}1_{\Lambda}, f \colon G \to G'). 
\end{align*}
\end{proof}

\subsection{Cyclotomic quiver Hecke algebras} \label{sec:cyclotomic}

Fix $J \subset I$, a choice of scalars $Q$, and a choice of bubble parameters $C$ compatible with $Q$. 
Kang-Kashiwara \cite{MR2995184, MR3709726} proved that, when $J = I$, the 2-category $\dotcatquantum{\mathfrak{g}}$ acts on the module category over the cyclotomic quiver Hecke algebra for every dominant weight $\Lambda$, 
which provides a categorification of the highest weight integrable module $V(\Lambda)$. 
In this section, we generalize this result to the $U_q(\mathfrak{p}_J)$-module $V_J(\Lambda)$ for any $J$-dominant $\Lambda \in \mathsf{P}$. 
For the particular case where $J = \{i\}$ and $\Lambda = 0$, 
most of the results of this section are established in \cite[section 4]{MR4285453}. 
Fix a $J$-dominant weight $\Lambda \in \mathsf{P}$. 

\begin{definition}
Let $\beta \in \mathsf{Q}_+$ and put $n = \height \beta$. 
The cyclotomic quiver Hecke algebras are defined by 
\begin{align*}
R^{J,\Lambda}(\beta) &= R(\beta)/(x_n^{\langle h_j, \Lambda \rangle}e(*,j) \ (j \in J)), \\
{}^{J,\Lambda}R(\beta) &= R(\beta)/(x_1^{\langle h_j, \Lambda \rangle}e(j,*) \ (j \in J)). 
\end{align*} \index{$R^{J,\Lambda}(\beta), {}^{J,\Lambda}R(\beta)$}
\end{definition}

Note that they only depend on the values $\langle h_j, \Lambda \rangle \ (j \in J)$. 
Note also that the involution $\sigma$ of $R(\beta)$ induces $R^{J,\Lambda}(\beta) \simeq {}^{J,\Lambda}R(\beta)$. 

\begin{definition} \label{def:generatingfunctors}
Let $\beta \in \mathsf{Q}_+, i \in I, j \in J$.
We define the following functors:
\begin{align*}
F_i \colon & \gMod{R^{J,\Lambda}(\beta)} \to \gMod{R^{J,\Lambda}(\beta+\alpha_i)}, \\
& X \mapsto R^{J,\Lambda}(\beta+\alpha_i)e(i, \beta) \otimes_{R^{J,\Lambda}(\beta)} X, \\
E_j \colon & \gMod{R^{J,\Lambda}(\beta)} \to \gMod{R^{J,\Lambda}(\beta-\alpha_j)}, \\
& X \mapsto q_j^{-1-\langle h_j,\Lambda -\beta \rangle} e(j,\beta-\alpha_j)X, \\
F_i^* \colon & \gMod{^{J,\Lambda}R(\beta)} \to \gMod{^{J,\Lambda}R(\beta+\alpha_i)}, \\
& X \mapsto {}^{J,\Lambda}R(\beta+\alpha_i)e(\beta,i) \otimes_{{}^{J,\Lambda}R(\beta)} X, \\
E_j^* \colon & \gMod{^{J,\Lambda}R(\beta)} \to \gMod{^{J,\Lambda}R(\beta-\alpha_j)}, \\
& X \mapsto q_j^{-1-\langle h_j, \Lambda -\beta \rangle} e(\beta-\alpha_j,j)X. 
\end{align*}
\end{definition}

\begin{remark} \label{rem:grading}
We have the following equalities of graded modules. 
\begin{align*}
E_j^{-\langle h_j, \Lambda-\beta \rangle} X &= e((-\langle h_j,\Lambda-\beta \rangle)\alpha_j,*)X, \\
(E_j^*)^{-\langle h_j,\Lambda - \beta \rangle} X &= e(*,(-\langle h_j,\Lambda -\beta \rangle )\alpha_j)X. 
\end{align*}
Here, the grading shifts appearing in the definition cancel out. 
\end{remark}

First, we study $\gMod{{}^{J,\Lambda}R} = \bigoplus_{\beta \in \mathsf{Q}_+} \gMod{{}^{J,\Lambda}R(\beta)}$.

\begin{proposition} \label{prop:cyclotomicR}
Let $j \in J, \beta \in \mathsf{Q}_+, X \in \gMod{{}^{J,\Lambda}R(\beta)}$ and put $n = \height \beta$. 
\begin{enumerate}
\item There exits a homomorphism of $(R(\beta+\alpha_j), R(\alpha_j))$-modules
\[
\mathsf{R} = \mathsf{R}_X \colon q^{(\alpha_j, 2\Lambda - \beta)} R(\alpha_j)\circ X \to X \circ R(\alpha_j) \ (X \in \gMod{^{J,\Lambda}R(\beta)})
\]
given by $u \boxtimes v \mapsto x_1^{\langle h_j, \Lambda \rangle} \tau_1 \tau_2 \cdots \tau_n (v \boxtimes u) \ (u \in R(\alpha_j), v \in X)$. 
Furthermore, this homomorphism is natural in $X$. 
\item There exists a homogeneous (not necessarily of degree zero) homomorphism of $(R(\beta+\alpha_j), R(\alpha_j))$-modules
\[
\mathsf{R}' = \mathsf{R}'_X \colon X \circ R(\alpha_j) \to R(\alpha_j) \circ X \ (X \in \gMod{^{J,\Lambda}R(\beta)})
\]
given by $v \boxtimes u \mapsto g_n \cdots g_1 (u\boxtimes v)$, where
\[
g_k = \sum_{\nu \in I^{\beta+\alpha_j}, \nu_k \neq \nu_{k+1}} \tau_k e(\nu) + \sum_{\nu \in I^{\beta + \alpha_j}, \nu_k = \nu_{k+1}} (x_{k+1} -x_k -(x_{k+1}-x_k)^2\tau_k) e(\nu). 
\]
Furthermore, this homomorphism is natural in $X$. 
\item The endomorphism $\mathsf{R}'_X \mathsf{R}_X$ coincides with $u \boxtimes v \mapsto A_{j,\beta,\Lambda}(u\boxtimes v)$, where 
\begin{align*}
A_{j,\beta,\Lambda} &= x_1^{\langle h_j,\Lambda \rangle} \sum_{\nu \in I^{\beta}} A_{j,\nu}, \\
A_{j,\nu} &= e(j,\nu) \prod_{1 \leq k \leq n, \nu_k \neq j} Q_{j,\nu_k}(x_1, x_{k+1}) \in R(\alpha_j) \otimes Z(\beta). 
\end{align*}
\item The endomorphism $\mathsf{R}_X \mathsf{R}'_X$ coincides with $v \boxtimes u \mapsto A'_{j,\beta,\Lambda}(v\boxtimes u)$, where
\begin{align*}
A'_{j,\beta, \Lambda} &= x_{n+1}^{\langle h_j,\Lambda \rangle} \sum_{\nu \in I^{\beta}} A'_{j,\nu}, \\
A'_{j,\nu} &= e(\nu,j) \prod_{1 \leq k \leq n, \nu_k \neq j} Q_{\nu_k,j}(x_k, x_{n+1}) \in Z(\beta) \otimes R(\alpha_j). 
\end{align*}
\end{enumerate}
\end{proposition}

\begin{proof}
(1) and (2) are direct generalizations of the results of \cite[Section 4.3]{MR2995184}. 

(3) is also a direct generalization of \cite[Theorem 4.15]{MR2995184}.
The key is the following equation for any $\nu \in I^{\beta}$:
\begin{align*}
&x_1^{\langle h_j, \Lambda \rangle} \tau_1 \cdots \tau_n g_n \cdots g_1 e(j,\nu) \equiv A_{j,\beta,\Lambda} e(j,\nu) \\
&\mod \sum_{j' \in J} R(\beta+\alpha_j)x_2^{\langle h_{j'},\Lambda \rangle} e(j,j',*) (e(j) \boxtimes R(\beta)e(\nu)).
\end{align*}
This is proved by the same inductive argument as that of \cite[Theorem 4.15]{MR2995184}. 

(4) We give two proofs.  

First proof: 
To begin with, we claim that $\mathsf{R}'_X$ is injective. 
Since $X$ is a projective limit of finite-dimensional $X/(R(\beta)X_{\geq d}) \ (d \in \mathbb{Z})$, we may assume that $X$ is finite dimensional.
By induction on the length of $X$, we may assume that $X$ is simple. 
Since $R(\alpha_j) \circ X$ is free over $R(\alpha_j)$ and $A_{j,\beta}$ is monic in $x_1$, 
(3) implies that $\mathsf{R}'_X \mathsf{R}_X$ is injective. 
In particular, $\mathsf{R}'_X$ is nonzero. 
Let $\renormalizedR{X}{R(\alpha_j)}$ be the normalized R-matrix introduced in \cite[p.1173]{MR3790066}. 
By \cite[Proposition 2.11]{MR3790066}, $\HOM_{(R(\beta+\alpha_j), R(\alpha_j))} (X \circ R(\alpha_j), R(\alpha_j) \circ X)$ is freely generated by $\renormalizedR{R(\alpha_j)}{X}$ over $R(\alpha_j)$. 
In addition, $\renormalizedR{X}{R(\alpha_j)}$ is injective by \cite[Lemma 3.12]{murata2025affinehighestweightstructures}. 
Since $X \circ R(\alpha_j)$ is free as an $R(\alpha_j)$-module, these two facts imply that every nonzero ungraded $(R(\beta+\alpha_j),R(\alpha_j))$-linear homomorphism $X \circ R(\alpha_j) \to R(\alpha_j) \circ X$ is injective. 
In particular, the nonzero homomorphism $\mathsf{R}'_X$ is injective. 

By (3), we have $\mathsf{R}'_X \mathsf{R}_X \mathsf{R}'_X = A_{j,\beta,\Lambda}\mathsf{R}'_X$. 
Since $\mathsf{R}'_X$ commutes with the action of $R(\alpha_j)\otimes Z(\beta)$ by naturality, $\mathsf{R}'_X \mathsf{R}_X \mathsf{R}'_X = \mathsf{R}'_X A'_{j,\beta,\Lambda}$. 
Since $\mathsf{R}'_X$ is injective, we deduce that $\mathsf{R}_X \mathsf{R}'_X = A'_{j,\beta,\Lambda}$. 

Second proof: 
We carry out a direct computation. 
We define for each $1 \leq k \leq n$ an element
\[
\varphi_k = \sum_{\nu \in I^{\beta+\alpha_j}, \nu_k \neq \nu_{k+1}} \tau_k e(\nu) + \sum_{\nu \in I^{\beta +\alpha_j}, \nu_k = \nu_{k+1}} ((x_{k+1}-x_k)\tau_k -1) e(\nu) \in R(\beta+\alpha_j). 
\]
They are called the intertwiners and satisfy the following relations (\cite[Lemma 1.5]{MR3748315}, \cite[Lemma 1.9]{MR3790066}): 
\begin{itemize}
\item $\varphi_k x_l = x_{s_k(l)} \varphi_k \ (1\leq l \leq n+1)$, 
\item $\varphi_k \varphi_{k+1} \varphi_k = \varphi_{k+1} \varphi_k \varphi_{k+1}$, 
\item $\varphi_{k+1} \varphi_k \tau_{k+1} = \tau_k \varphi_{k+1} \varphi_k$.  
\end{itemize}
The elements $g_k$ also satisfy similar relations (\cite[Lemma 4.12]{MR2995184}). 
Note that 
\[
g_k e(\nu)= \begin{cases}
\varphi_k e(\nu) & \text{if $\nu_k \neq \nu_{k+1}$}, \\
(x_k-x_{k+1}) \varphi_k e(\nu) & \text{if $\nu_k = \nu_{k+1}$}. 
\end{cases}
\]
From these, we deduce for any $\nu \in I^{\beta}$, 
\begin{equation} \label{eq:gvsphi}
g_n \cdots g_1 e(j,\nu) = \varphi_n \cdots \varphi_1 \prod_{1 \leq k \leq n, \nu_k = j} (x_{k+1}-x_1) e(j,\nu). 
\end{equation}

To prove (4), it suffices to verify the following equation for any $\nu \in I^{\beta}$: 
\begin{align} \label{eq:goal}
&g_n \cdots g_1 e(j,\nu) x_1^{\langle h_j, \Lambda \rangle} \tau_1 \cdots \tau_n \equiv A'_{j,\beta,\Lambda} e(\nu,j) \\
&\mod \varphi_n \cdots \varphi_1 x_1^{\langle h_j,\Lambda \rangle} e(j, \nu) (e(j) \boxtimes R(\beta)) (R(\beta) \boxtimes e(j)) e(\nu,j). \notag
\end{align}
Note that the part modded out annihilates $X \boxtimes R(\alpha_j)$. 
Since 
\[ 
g_n \cdots g_1 x_1^{\langle h_j,\Lambda \rangle} = x_{n+1}^{\langle h_j,\Lambda \rangle}g_n \cdots g_1, \ \varphi_n \cdots \varphi_1 x_1^{\langle h_j,\Lambda \rangle} = x_{n+1}^{\langle h_j,\Lambda \rangle}\varphi_n \cdots \varphi_1,
\]
we may assume $\langle h_j,\Lambda \rangle = 0$. 
We proceed by induction on $n$. 

Let $\nu = (\nu_1, \nu')$. 
Note that 
\begin{equation} \label{eq:getau} 
g_1 e(j,\nu) \tau_1 = \begin{cases}
 Q_{\nu_1,j} (x_1,x_2) e(\nu_1, j, \nu') & \text{if $\nu_1 \neq j$}, \\
 (x_2-x_1)\tau_1 e(j,j,\nu') = (1+\varphi_1)e(j,j,\nu) & \text{if $\nu_1 = j$}. 
\end{cases}
\end{equation}
In particular, the equation (\ref{eq:goal}) holds when $n = 1$. 
From now on, let $n \geq 2$.  

First, assume $\nu_1 \neq j$. 
Then, we have 
\begin{align*}
&g_n \cdots g_1 e(j,\nu) \tau_1 \cdots \tau_n \\
&= g_n \cdots g_2 Q_{\nu_1,j}(x_1,x_2) e(\nu_1, j, \nu') \tau_2 \cdots \tau_n \quad \text{by (\ref{eq:getau})} \\
&= Q_{\nu_1,j}(x_1,x_{n+1}) g_n \cdots g_2 e(\nu_1,j,\nu') \tau_2 \cdots \tau_n \\
&\equiv Q_{\nu_1,j}(x_1,x_{n+1}) (e(\nu_1) \boxtimes A'_{j,\nu'}) e(\nu,j) \\
&\mod Q_{\nu_1,j}(x_1, x_{n+1}) \varphi_n \cdots \varphi_2 e(\nu_1,j,\nu') (e(\nu_1,j) \boxtimes R(\beta-\alpha_{\nu_1})) \times \\
&\quad (e(\nu_1) \boxtimes R(\beta-\alpha_{\nu_1}) \boxtimes e(j)) e(\nu, j), 
\end{align*}
by the induction hypothesis. 
Note that  
\[
Q_{\nu_1,j}(x_1, x_{n+1}) (e(\nu_1) \boxtimes A'_{j,\nu'})e(\nu_1, \nu',j) = A'_{j,\nu} e(\nu,j). 
\]
In addition, we have 
\begin{align*}
&Q_{\nu_1,j}(x_1, x_{n+1}) \varphi_n \cdots \varphi_2 e(\nu_1,j,\nu') (e(\nu_1, j) \boxtimes R(\beta-\alpha_{\nu_1})) \times \\ 
&\quad (e(\nu_1) \boxtimes R(\beta-\alpha_{\nu_1}) \boxtimes e(j)) e(\nu, j) \\
&= \varphi_n \cdots \varphi_2 Q_{\nu_1,j}(x_1,x_2) e(\nu_1,j,\nu') (e(\nu_1, j) \boxtimes R(\beta-\alpha_{\nu_1})) \times \\ 
&\quad(e(\nu_1) \boxtimes R(\beta-\alpha_{\nu_1}) \boxtimes e(j)) e(\nu,j) \\
&= \varphi_n \cdots \varphi_2 \varphi_1^2 e(\nu_1, j, \nu') (e(\nu_1,j) \boxtimes R(\beta-\alpha_{\nu_1}))(e(\nu_1) \boxtimes R(\beta-\alpha_{\nu_1})\boxtimes e(j))e(\nu,j) \\
&= \varphi_n \cdots \varphi_1 e(j,\nu) (e(j,\nu_1) \boxtimes R(\beta-\alpha_{\nu_1})) \varphi_1 (e(\nu_1) \boxtimes R(\beta-\alpha_{\nu_1}) \boxtimes e(j)) e(\nu,j) \\
&\subset \varphi_n \cdots \varphi_1 e(j,\nu) (e(j) \boxtimes R(\beta)) (R(\beta) \boxtimes e(j)) e(\nu,j). 
\end{align*}
Hence, the equation (\ref{eq:goal}) holds. 

Next, assume $\nu_1 = j$. 
Then, we have
\begin{align} \label{eq:twoterm}
&g_n \cdots g_1 e(j,\nu) \tau_1 \cdots \tau_n \notag \\
&= g_n \cdots g_2 (1+\varphi_1) e(j, j, \nu') \tau_2 \cdots \tau_n \quad \text{by (\ref{eq:getau})} \notag \\
&= g_n \cdots g_2 e(j,j,\nu') \tau_2 \cdots \tau_n + g_n\cdots g_2 \varphi_1 e(j,j,\nu') \tau_2 \cdots \tau_n. 
\end{align}
By the induction hypothesis, the first term of (\ref{eq:twoterm}) is 
\begin{align*}
&(e(j) \boxtimes A'_{j,\nu'})e(j,j,\nu') \mod \varphi_n \cdots \varphi_2 e(j,j,\nu') (e(j,j) \boxtimes R(\beta-\alpha_j)) \times \\ 
&(e(j) \boxtimes R(\beta-\alpha_j) \boxtimes e(j)) e(\nu,j). 
\end{align*}
Note that $(e(j) \boxtimes A'_{j,\nu'}) e(j,j,\nu') = A'_{j,\nu} e(j,\nu)$, and 
\begin{align*}
&\varphi_n \cdots \varphi_2 e(j,j,\nu') (e(j,j) \boxtimes R(\beta-\alpha_j))(e(j) \boxtimes R(\beta-\alpha_j) \boxtimes e(j)) e(\nu,j) \\
&=\varphi_n \cdots \varphi_2 \varphi_1^2 e(j,j,\nu') (e(j,j)\boxtimes R(\beta-\alpha_j)) (e(j)\boxtimes R(\beta-\alpha_j)\boxtimes e(j)) e(\nu,j) \\
&=\varphi_n \cdots \varphi_2 \varphi_1 e(j,j,\nu') (e(j,j)\boxtimes R(\beta-\alpha_j)) \varphi_1 (e(j)\boxtimes R(\beta-\alpha_j)\boxtimes e(j)) e(\nu,j) \\
&\subset \varphi_n \cdots \varphi_1 e(j,\nu) (e(j)\boxtimes R(\beta))(R(\beta) \boxtimes e(j)) e(\nu,j).  
\end{align*}
On the other hand, the second term of (\ref{eq:twoterm}) is 
\begin{align*}
&\varphi_n \cdots \varphi_2 \varphi_1 \prod_{2 \leq k \leq n, \nu_k = j} (x_{k+1}-x_1) e(j,j,\nu')\tau_2 \cdots \tau_n e(\nu,j) \quad \text{by (\ref{eq:gvsphi})} \\
&\in \varphi_n \cdots \varphi_1 e(j,\nu) (R(\alpha_j)\boxtimes R(\beta)) e(\nu,j) \\
&\subset \varphi_n \cdots \varphi_1 e(j,\nu) (e(j) \boxtimes R(\beta))(R(\beta) \boxtimes e(j)) e(\nu,j). 
\end{align*}
Hence, the equation (\ref{eq:goal}) holds.
\end{proof}

Let $X \circ R(\alpha_j) \xrightarrow{\mathrm{can}} F_j^* (X)$ denote the canonical surjection. 

\begin{theorem} \label{thm:functorF*}
Let $\beta \in \mathsf{Q}_+$ and $X \in \gMod{{}^{J,\Lambda}R(\beta)}$. 
\begin{enumerate}
\item Let $i \in I \setminus J$.
Then, the canonical homomorphism $X \circ R(\alpha_i) \xrightarrow{\mathrm{can}} F_i^* X$ is an isomorphism. 
\item Let $j \in J$. 
Then, the following sequence is exact:
\[
0 \to  q^{(2\Lambda-\beta,\alpha_j)} R(\alpha_j) \circ X \xrightarrow{\mathsf{R}_X} X \circ R(\alpha_j) \xrightarrow{\mathrm{can}} F_j^* X \to 0. 
\]
\end{enumerate}
\end{theorem}
  
\begin{proof}
(1) It suffices to prove that $X \circ R(\alpha_i)$ is an $^{J,\Lambda}R(\beta)$-module. 
Let $j \in J$. 
Since $j \neq i$, we have an isomorphisms of $R(\alpha_j) \otimes R(\beta+\alpha_i-\alpha_j)$-modules
\[
e(j, \beta+\alpha_i-\alpha_j) (X \circ R(\alpha_i)) \simeq \Ind_{\beta-\alpha_j,\alpha_i} (e(j,\beta-\alpha_j)X \otimes R(\alpha_i)) 
\]
by considering the Mackey filtration (Proposition \ref{prop:Mackey}). 
Here, the action of $R(\alpha_j)$ on the right hand side is given by 
\[
x_1 \cdot (e(j,\beta-\alpha_j)u \otimes v) = x_1e(j,\beta-\alpha_j)u \otimes v \ (x_1 \in R(\alpha_j), u \in X, v \in R(\alpha_i)). 
\]
Since $x_1^{\langle h_i,\Lambda \rangle}e(j,\beta-\alpha_j)X = 0$, the above isomorphism implies $x_1^{\langle h_j, \Lambda \rangle}e(j,\beta+\alpha_i-\alpha_j) (X \circ R(\alpha_i)) = 0$. 
The assertion is proved. 

(2) 
The exactness at the middle and right terms is straightforward, see \cite[Section 4.3, (4.13)]{MR2995184}. 
It remains to prove that $\mathsf{R}_X$ is injective. 
By Proposition \ref{prop:cyclotomicR} (3), it suffices to verify that the multiplication by $A_{j,\beta,\Lambda}$ on $R(\alpha_j) \circ X$ is injective. 
Since $A_{j,\beta,\Lambda}$ is monic in $x_1 \in R(\alpha_j)$ and $R(\alpha_j) \circ X$ is free over $R(\alpha_j)$, the assertion follows.  
\end{proof}
 
\begin{theorem} \label{thm:functorE*}
Let $\beta \in \mathsf{Q}_+$ and $j \in J$.
Then, ${}^{J,\Lambda}R(\beta+\alpha_j)e(\beta,j)$ is a finitely-generated projective right ${}^{J,\Lambda}R(\beta)$-module, 
and $e(\beta,j) ({}^{J,\Lambda}R(\beta+\alpha_j))$ is a finitely-generated projective left ${}^{J,\Lambda}R(\beta)$-module, . 
\end{theorem}

\begin{proof}
Note that ${}^{J,\Lambda}R(\beta+\alpha_j)e(\beta,j)$ is isomorphic to $F_j^* ({}^{J,\Lambda}R(\beta))$. 
Applying theorem \ref{thm:functorF*} to $X ={}^{J,\Lambda}R(\beta)$, we obtain a short exact sequence of right $R(\alpha_j) \otimes {}^{J,\Lambda}R(\beta)$-modules
\[
0 \to q^{(\alpha_j, 2\Lambda-\beta)} R(\alpha_j) \circ {}^{J,\Lambda}R(\beta) \xrightarrow{\mathsf{R}} {}^{J,\Lambda}R(\beta) \circ R(\alpha_j) \to {}^{J,\Lambda}R(\beta+\alpha_j)e(\beta,j) \to 0. 
\] 
The left and the middle modules are finitely-generated free right $R(\alpha_j)\otimes {}^{J,\Lambda}R(\beta)$-modules.
By \cite[Lemma 4.18]{MR2995184}, the former assertion is reduced to finding an element $A \in R(\alpha_j) \otimes Z({}^{J,\Lambda}R(\beta))$ that 
\begin{itemize}
\item is monic as a polynomial in one variable $x_1 \in R(\alpha_j)$ with coefficients in $Z({}^{J,\Lambda}R(\beta))$, and 
\item annihilates ${}^{J,\Lambda}R(\beta+\alpha_j)e(\beta,j)$. 
\end{itemize}
By Proposition \ref{prop:cyclotomicR} (4), the condition is satisfied if we take $A = A'_{j,\beta,\Lambda}$. 
Hence, the former assertion is proved. 

By applying the antiautomorphism $\varphi$ that fixes all the generators, we deduce the latter assertion. 
\end{proof}

\begin{remark}
In \cite{MR2995184}, an element $A$ satisfying the condition in the proof above is obtained more easily (Lemma 4.3). 
The lemma does not extend to our setting and we needed to prove Proposition \ref{prop:cyclotomicR} (4) instead. 
\end{remark}

\begin{corollary} 
Let $i \in I, j \in J$. 
The functors $F_i^*, E_j^*$ are exact, and send finitely-generated modules (resp. projective modules) to finitely-generated ones (resp. projective ones). 
\end{corollary}

\begin{theorem} \label{thm:cyclotomic2rep}
The category $\gMod{{}^{J,\Lambda}R}$ is a right $\dotcatquantum{\mathfrak{p}_J}$-module as follows: 
\begin{itemize}
\item The category attached to an object $\lambda \in \mathsf{P}$ is $\gMod{{}^{J,\Lambda}R(\lambda + \Lambda)}$, where we regard ${}^{J,\Lambda}R(\lambda + \Lambda) = 0$ if $\lambda + \Lambda \not \in \mathsf{Q}_+$. 
\item %Let $\beta \in \mathsf{Q}_+$ and $X \in \gMod{{}^{J,\Lambda}R(\beta)}$. 
%Put $\lambda = -\Lambda + \beta$. 
Using the functors from Definition \ref{def:generatingfunctors}, the actions of the generating 1-morphisms are given as follows:
\[
X F_i = F_i^* (X) \ (i \in I),\ X E_j = E_j^*(X) \ (j \in J). 
\]
\item The actions of the generating 2-morphisms are given as follows ($X \in \gMod{{}^{J,\Lambda}R(\beta)}, \lambda = -\Lambda + \beta$): 
\begin{align*}
&\left[ X \xy 0;/r.17pc/: 
(0,0)*{\sdotd{i}};
\endxy \colon q^{(\alpha_i,\alpha_i)} X F_i \to X F_i \right] \ (i\in I) \\
&= \text{the right multiplication by $x_1 \in R(\alpha_i)$}, \\
&\left[ X \xy 0;/r.17pc/: 
(0,0)*{\dcross{i}{i'}};
\endxy \colon q^{-(\alpha_i,\alpha_{i'})} X F_i F_{i'} \to X F_{i'} F_i \right] \ (i, i' \in I) \\
&= \text{the right multiplication by $e(i,i')\tau_1 \in R(\alpha_i + \alpha_{i'})$}, \\
&\left[ X \xy 0;/r.17pc/: 
(0,0)*{\rcap{j}};
\endxy \colon q_j^{1 - \langle h_j, \lambda \rangle} XE_jF_j = q_j^{1 -\langle h_j,\lambda \rangle}F_j^*E_j^*(X) \to X \right] \ (j \in J) \\
&= \text{the canonical unit for the adjoint pair $(F_j^*, q_j^{1 + \langle h_j, \Lambda - \beta \rangle} E_j^*)$}, \\
&\left[ X \xy 0;/r.17pc/: 
(0,0)*{\rcup{j}};
\endxy \colon q_j^{1 + \langle h_j, \lambda \rangle} X \to XF_jE_j = E_j^*F_j^*(X) \right] \ (j \in J) \\
&= \text{the canonical counit for the adjoint pair $(F_j^*, q_j^{1 + \langle h_j, \Lambda - \beta -\alpha_j \rangle} E_j^*)$}. 
\end{align*}
\end{itemize}
The action restricts to the additive subcategory $\gproj{{}^{J,\Lambda}R} = \bigoplus_{\beta \in \mathsf{Q}_+} \gproj{{}^{J,\Lambda}R(\beta)}$. 
Moreover, ignoring grading shifts, we have
\begin{align*}
&\left[ X \xy 0;/r.17pc/: 
(0,0)*{\sdotu{j}};
\endxy  \colon XE_j \to XE_j \right] \ (j \in J) \\
&= \left[ e(*,j)X \xrightarrow{\text{the left multiplication by $x_{\height \beta} e(*,j)$}} e(*,j)X \right], \\
&\left[ X \xy 0;/r.17pc/: 
(0,0)*{\ucross{j}{j'}};
\endxy \colon XE_jE_{j'} \to XE_{j'}E_j \right] \ (j,j' \in J) \\
&= \left[ e(*,j',j)X \xrightarrow{\text{the left multiplication by $\tau_{\height \beta -1} e(*,j',j)$}} e(*,j,j')X \right].
\end{align*}
\end{theorem}

\begin{remark} \label{rem:restrictionorder}
Be careful about where each $E_j$ is applied. 
For instance, we have $X E_j E_{j'} = e(*,j',j) X$ up to grading shift. 
\end{remark}

\begin{proof}
To see that the action is well-defined, it suffices to check the relations listed in Theorem \ref{thm:Rouquierver} (1).
The adjunction and the KLR relations are immediate. 
The last set of relations, formal inverse, is proved by the same argument as that of \cite[Section 5]{MR2995184}. 

It follows from the definition that $\gproj{{}^{J,\Lambda}R}$ is invariant under $F_i \ (i \in I)$. 
Theorem \ref{thm:functorE*} shows that it is also stable under $E_j \ (j \in J)$. 
The last assertion is deduced from the defining actions of the downward 2-morphisms by applying the adjunction. 
\end{proof}

By Theorem \ref{thm:cyclotomic2rep}, 
both $K(\gMod{{}^{J,\Lambda}R})_{\mathbb{Q}(q)}$ and $K_{\oplus}(\gproj{{}^{J,\Lambda}R})_{\mathbb{Q}(q)}$ are right $U_q(\mathfrak{p}_J)$-modules. 

\begin{theorem} \label{thm:cyclotomiccategorification}
(1) We have an isomorphism of right $U_q(\mathfrak{p}_J)$-modules
\begin{align*}
K_{\oplus}(\gproj{{}^{J,\Lambda}R(\beta)})_{\mathbb{Q}(q)} &\simeq {}_JV(-\Lambda), \\
[{}^{J,\Lambda}R(0)] &\mapsto v_{-\Lambda}. 
\end{align*}
(2) The morphisms 
\begin{align*}
&K(\gmod{{}^{J,\Lambda}R})_{\mathbb{Q}(q)} \to K(\gMod{{}^{J,\Lambda}R})_{\mathbb{Q}(q)}, \\
&K_{\oplus}(\gproj{{}^{J,\Lambda}R})_{\mathbb{Q}(q)} \to K(\gMod{{}^{J,\Lambda}R})_{\mathbb{Q}(q)}
\end{align*}
induced by the inclusions are isomorphisms. 
Furthermore, the second one is a morphism of right $U_q(\mathfrak{p}_J)$-modules. 
\end{theorem}

\begin{proof}
(1)
The case where $J = \emptyset$ is \cite{MR2525917, MR2763732}. 
In this case, ${}_{\emptyset}V(-\Lambda) \simeq \quantum{-}$ with ${}_{\emptyset}V(-\Lambda)_{-\Lambda + \beta} \simeq \quantum{-}_{-\beta}$. 

Regarding the general case, we have surjective homomorphisms of right $U_q^-(\mathfrak{g})$-modules 
\begin{align*}
&\quantum{-} \to {}_JV(-\Lambda), u \mapsto v_{-\Lambda}u, \\
&K_{\oplus}(\gproj{R})_{\mathbb{Q}(q)} \to K_{\oplus}(\gproj{{}^{J,\Lambda}R})_{\mathbb{Q}(q)}, [P] \mapsto [{}^{J,\Lambda}R(\beta) \otimes_{R(\beta)} P]. 
\end{align*}
We claim that, through these two surjections, the isomorphism 
\[
\quantum{-} \simeq K_{\oplus}(\gproj{R})_{\mathbb{Q}(q)}
\]
induces a surjective homomorphism ${}_JV(-\Lambda) \to K_{\oplus}(\gproj{{}^{J,\Lambda}R})_{\mathbb{Q}(q)}$. 
By Lemma \ref{lem:presentation}, it suffices to prove that ${}^{J,\Lambda}R(0) F_j^{\langle h_j,\Lambda \rangle + 1} = 0$ for any $j \in J$. 
The left hand side is ${}^{J,\Lambda}R((\langle h_j,\Lambda \rangle +1)\alpha_j)$. 
It coincides with the cyclotomic quotient of the nil-Hecke algebra $R((\langle h_j,\Lambda \rangle +1)\alpha_j)$, which is known to be zero. 

It remains to prove that the induced surjective homomorphism is an isomorphism.
To see this, we prove that the dimensions of ${}_JV(\Lambda)_{-\Lambda + \beta}$ and of $K_{\oplus}(\gproj{{}^{J,\Lambda}R(\beta)})_{\mathbb{Q}(q)}$ are the same. 

We claim that a simple $R(\beta)$-module $L$ is a ${}^{J,\Lambda}R(\beta)$-module if and only if $\Res_{(\langle h_j,\Lambda \rangle +1)\alpha_j,*} L= 0$ for any $j \in J$. 
First, assume that $L$ is a simple ${}^{J,\Lambda}R(\beta)$-module. 
Let $j \in J$ and put $n = \langle h_j, \Lambda \rangle$. 
Then, $\Res_{(n+1)\alpha_j,\beta-(n+1)\alpha_j} L$ is a ${}^{J,\Lambda}R((n+1)\alpha_j) \otimes R(\beta-(n+1)\alpha_j)$-module. 
As is explained in the previous paragraph, ${}^{J,\Lambda}R((n+1)\alpha_j) = 0$, hence the restriction must be zero. 
Next, assume that $\Res_{(\langle h_j,\Lambda \rangle+1)\alpha_j,*} L = 0$ for any $j \in J$. 
Take $j \in J$ arbitrarily, and let $n \geq 0$ be the largest integer such that $\Res_{n\alpha_j,*} L \neq 0$. 
Take a simple $R(n\alpha_j) \otimes R(\beta-n\alpha_j)$-submodule of $\Res_{n\alpha_j,\beta-n\alpha_j}L$. 
It is of the form $L(j^n) \otimes L_0$, where $L_0$ is a simple $R(\beta-n\alpha_j)$-module that satisfies $e(j,*)L_0 = 0$. 
By the induction-restriction adjunction, we have a nonzero homomorphism $L(j^n) \circ L_0 \to L$, 
which is surjective since $L$ is simple. 
By considering the Mackey-filtration (Proposition \ref{prop:Mackey}), we have 
\[
\Res_{\alpha_j,\beta-\alpha_j}(L(j^n) \circ L_0) \simeq \Ind_{(n-1)\alpha_j,\beta-n\alpha_j} ((\Res_{\alpha_j,(n-1)\alpha_j}L(j^n)) \otimes L_0). 
\]
By the assumption, we have $n \leq \langle h_j,\Lambda \rangle$, hence $x_1^{\langle h_j,\Lambda \rangle}L(j^n) = 0$. 
It follows that $x_1^{\langle h_j,\Lambda \rangle}e(j,\beta-\alpha_j) (L(j^n) \circ L_0) = 0$.
Since we have a surjective homomorphism $L(j^n) \circ L_0 \to L$, it implies $x_1^{\langle h_j,\Lambda \rangle}e(j,\beta - \alpha_j)L = 0$. 
Therefore, $L$ is a ${}^{J,\Lambda}R(\beta)$-module.

On the other hand, the kernel of the homomorphism $\quantum{-} \to {}_JV(\Lambda)$ is $\sum_{j \in J} f_j^{\langle h_j,\Lambda \rangle +1} \quantum{-}$. 
By \cite[Theorem 7]{MR1115118}, this kernel is spanned by 
\[
\{ G(b) \mid b \in B(\infty), \text{there exists $j \in J$ such that $\varepsilon_j(b) \geq \langle h_j,\Lambda \rangle +1$}\}, 
\]
where $G(b)$ is the element of the global basis of $U_q^-(\mathfrak{g})$ corresponding to $b \in B(\infty)$. 

Using \cite{MR2822211}, we conclude that both ${}_JV(\Lambda)_{-\Lambda + \beta}$ and $K_{\oplus}(\gproj{{}^{J,\Lambda}R(\beta)})_{\mathbb{Q}(q)}$ have dimension 
\[
\lvert \{ b \in B(\infty)_{-\beta} \mid \text{$\varepsilon_j(b) \leq \langle h_j, \Lambda \rangle$ for any $j \in J$} \} \rvert. 
\]

(2) The morphism $K(\gmod{{}^{J,\Lambda}R})_{\mathbb{Q}(q)} \to K(\gMod{{}^{J,\Lambda}R})_{\mathbb{Q}(q)}$ is an isomorphism by the same reasoning as that of Theorem \ref{thm:categorification2}. 
\begin{comment}
Furthermore, the inverse is given by 
\[
[M] \mapsto \sum_{1 \leq s \leq r} (\qdim \HOM_{{}^{J,\Lambda}R(\beta)}(P_s,M))[L_s],
\]
where $\{L_1, \ldots, L_r \}$ is a complete set of isomorphism classes of simple graded ${}^{J,\Lambda}R(\beta)$-modules up to grading shift, 
and $P_s$ is the projective cover of $L_s$ in $\gMod{{}^{J,\Lambda}R(\beta)}$. 
\end{comment}
In particular, $K(\gmod{{}^{J,\Lambda}R(\beta)})_{\mathbb{Q}(q)}$, $K(\gMod{{}^{J,\Lambda}R(\beta)})_{\mathbb{Q}(q)}$ and $K(\gproj{{}^{J,\Lambda}R})_{\mathbb{Q}(q)}$ all have the same dimension. 

Hence, in order to show that te morphism 
\[
K_{\oplus}(\gproj{{}^{J,\Lambda}R(\beta)})_{\mathbb{Q}(q)} \to K(\gMod{{}^{J,\Lambda}R(\beta)})_{\mathbb{Q}(q)}
\]
is an isomorphism, 
it suffices to prove that it is injective.
\begin{comment} 
\begin{align*}
K_{\oplus}(\gproj{{}^{J,\Lambda}R(\beta)})_{\mathbb{Q}(q)} &\to K(\gmod{{}^{J,\Lambda}R(\beta)})_{\mathbb{Q}(q)}, \\
[P] \mapsto [M] \mapsto \sum_{1 \leq s \leq r} (\qdim \HOM_{{}^{J,\Lambda}R(\beta)}(P_s,P))[L_s],
\end{align*}
is injective. 
It is reduced to proving that the bilinear form of $K(\gproj{{}^{J,\Lambda}R(\beta)})_{\mathbb{Q}(q)}$ defined by
\begin{equation} \label{eq:bilinearform}
([P],[Q]) = \qdim \HOM_{{}^{J,\Lambda}R(\beta)}(P,Q) \ (P,Q \in \gproj{{}^{J,\Lambda}R(\beta)})
\end{equation}
is nondegenerate. 
It is further reduced to showing that 
\[
([P],[Q]) = (\overline{\chi'(P)},\chi'(Q)), 
\]
where $\chi' \colon K_{\oplus}(\gproj{{}^{J,\Lambda}R})_{\mathbb{Q}(q)} \to {}_JV(-\Lambda)$ is the isomorphism of (1) and the bilinear form in the right hand side is the one given in Theorem \ref{thm:bosonsimple'}. 
Note that $\overline{xf_i} = \overline{x}f_i$ for any $i \in I, x \in {}_JV(-\Lambda)$. 
\end{comment}

Let $\beta = \sum_{j \in i} k_j \alpha_j \in \mathsf{Q}_+$.   
For $i \in I \setminus J$ and $P \in \gproj{{}^{J,\Lambda}R(\beta)}$, we define 
\[
PE_i = (R(\alpha_i)/x_1R(\alpha_i)) \otimes_{R(\alpha_i)} \Res_{\beta-\alpha_i,\alpha_i}P. 
\]
By Theorem \ref{thm:functorF*}, $\Res_{\beta-\alpha_i,\alpha_i}P$ is a finitely-generated projective ${}^{J,\Lambda}R(\beta-\alpha_i) \otimes R(\alpha_i)$-module, 
hence $PE_i \in \gproj{{}^{J,\Lambda}R(\beta-\alpha_i)}$ and $E_i$ gives an additive endofuncntor of $\gproj{{}^{J,\Lambda}R}$. 

We claim that $E_i$ coincides with the right action of $e_i \in {B'}_q^J(\mathfrak{g})$ under the isomorphism $\chi' \colon K_{\oplus}(\gproj{{}^{J,\Lambda}R})_{\mathbb{Q}(q)} \simeq {}_JV(-\Lambda)$. 
Since ${}_JV(-\Lambda)$ is generated by $v_{-\Lambda}$ over $U_q^-(\mathfrak{g})$, 
it suffices to verify the commutation relation between $E_i$ and $F_j \ (j \in I)$. 
First, we verify the relation $f_ie_i = q_i^{-2}e_if_i + 1$. 
Let $P \in \gproj{{}^{J,\Lambda}R(\beta)}$. 
By applying $\Res_{\beta,\alpha_i}$ to the isomorphism $P \circ R(\alpha_i) \simeq PF_i$ of Theorem \ref{thm:functorF*} (1) and using the Mackey-filtration (Proposition \ref{prop:Mackey}), 
we obtain a short exact sequence
\[
0 \to P \otimes R(\alpha_i) \to \Res_{\beta,\alpha_i}(PF_i) \to q_i^{-2} \Ind_{\beta-\alpha_i,\alpha_i}(\Res_{\beta-\alpha_i,\alpha_i}P \otimes R(\alpha_i)) \to 0, 
\]
where $\Res_{\beta-\alpha_i,\alpha_i}P \otimes R(\alpha_i)$ is regarded as an $R(\beta-\alpha_i)\otimes R(\alpha_i)$-module by 
\[
(x \otimes y) (e(\beta-\alpha_i,i)u \otimes v) = (x \boxtimes e(i)) u \otimes yv 
\]
for $x \in R(\beta-\alpha_i),y\in R(\alpha_i), u\in P, v \in R(\alpha_i)$. 
Note that all the three terms are free over $R(\alpha_i)$. 
Applying $(R(\alpha_i)/x_1R(\alpha_i)) \otimes_{R(\alpha_i)} ?$, we obtain a short exact sequence 
\[
0 \to P \to PF_iE_i \to q_i^{-2}PE_iF_i \to 0, 
\]
which proves the desired relation. 

Next, let $j \in I \setminus J$ and assume that $j \neq i$. 
We verify the relation $f_je_i = q^{-(\alpha_i,\alpha_j)}e_if_j$. 
By applying $\Res_{\beta+\alpha_j-\alpha_i,\alpha_i}$ to the isomorphism $P \circ R(\alpha_j) \simeq PF_j$ of Theorem \ref{thm:functorF*} (1) and using the Mackey-filtration (Proposition \ref{prop:Mackey}), 
we obtain an isomorphism
\[
\Res_{\beta+\alpha_j-\alpha_i,\alpha_i}(PF_j) \simeq q^{-(\alpha_i,\alpha_j)} \Ind_{\beta-\alpha_i,\alpha_j}(\Res_{\beta-\alpha_i,\alpha_i}P \otimes R(\alpha_j)). 
\]
Applying $(R(\alpha_i)/x_1 R(\alpha_i)) \otimes_{R(\alpha_i)}?$, we obtain an isomorphism 
\[
PF_jE_i \simeq q^{-(\alpha_i,\alpha_j)}PE_iF_j, 
\]
which proves the desired relation. 

Finally, let $j \in J$. 
By applying $\Res_{\beta+\alpha_j-\alpha_i,\alpha_i}$ to the short exact sequence of Theorem \ref{thm:functorF*} (2) and using the Mackey-filtration (Proposition \ref{prop:Mackey}), 
we obtain a short exact sequence 
\begin{align*}
&0 \to q^{(\alpha_j,2\Lambda-\beta)} \Ind_{\alpha_j,\beta-\alpha_i} (R(\alpha_j) \otimes \Res_{\beta-\alpha_i,\alpha_i}P) \\ 
&\to q^{-(\alpha_i,\alpha_j)}\Ind_{\beta-\alpha_i,\alpha_j} (\Res_{\beta-\alpha_i,\alpha_i}P \otimes R(\alpha_j)) \to \Res_{\beta+\alpha_j-\alpha_i,\alpha_i} (PF_j) \to 0. 
\end{align*}
Note that all the three terms are free over $R(\alpha_i)$.
Applying $(R(\alpha_i)/x_1R(\alpha_i)) \otimes_{R(\alpha_i)} ?$, we obtain a short exact sequence 
\[
0 \to q^{(\alpha_j,2\Lambda-\beta)} R(\alpha_j) \circ PE_i \to q^{-(\alpha_i,\alpha_j)} PE_i \circ R(\alpha_j) \to PF_jE_i \to 0. 
\]
Comparing it with the short exact sequence of Theorem \ref{thm:functorF*} (2) for $X = PE_i$, we deduce that $q^{(\alpha_i,\alpha_j)}PF_jE_i \simeq PE_iF_j$.  
Now, the claim is proved. 

We have analogous endofunctors $M \mapsto ME_i \ (i \in I \setminus J)$ on $\gMod{{}^{J,\Lambda}R}$, 
and they make $K(\gMod{{}^{J,\Lambda}R})_{\mathbb{Q}(q)}$ a right ${B'}_q^J(\mathfrak{g})$-module. 
In addition, the morphism $K_{\oplus}(\gproj{{}^{J,\Lambda}R})_{\mathbb{Q}(q)} \to K(\gMod{{}^{J,\Lambda}R})_{\mathbb{Q}(q)}$ is a homomorphism of right ${B'}_q^J(\mathfrak{g})$-modules.
By Theorem \ref{thm:bosonsimple'} and (1), the module $K_{\oplus}(\gproj{{}^{J,\Lambda}R})_{\mathbb{Q}(q)}$ is simple. 
Hence, the injectivity of the morphism is equivalent to its nonvanishing.

To see that it is nonzero, consider the module $[{}^{J,\Lambda}R(0)] \in K_{\oplus}(\gproj{{}^{J,\Lambda}R(0)})$. 
Under the morphism 
\[
K_{\oplus}(\gproj{{}^{J,\Lambda}R(0)})_{\mathbb{Q}(q)} \to K(\gMod{{}^{J,\Lambda}R(0)})_{\mathbb{Q}(q)} \simeq K(\gmod{{}^{J,\Lambda}R(0)}),
\]
it is sent to $[{}^{J,\Lambda}R(0)]$. 
Since ${}^{J,\Lambda}R(0)$ is a one-dimensional simple module, this element is nonzero. 
The proof is complete.
\begin{comment}
We verify that the bilinear form (\ref{eq:bilinearform}) satisfies the formulas of Theorem \ref{thm:bosonsimple'}, which characterize the bilinear form on ${}_JV(-\Lambda)$. 
Since ${}^{J,\Lambda}R(0) \simeq \mathbf{k}$, we have 
\[
([{}^{J,\Lambda}R(0)], [{}^{J,\Lambda}R(0)]) = 1.  
\]
Let $i \in I, P \in \gproj{{}^{J,\Lambda}R(\beta)}$ and $Q \in \gproj{{}^{J,\Lambda}R(\beta+\alpha_i)}$. 
If $i \in J$, then we have 
\begin{align*}
&([PF_i], [Q]) \\ 
&= ([P], q_i^{1+\langle h_i,\Lambda - (\beta+\alpha_i) \rangle}[QE_i]) \quad \text{since $(F_i^*,q_i^{1 + \langle h_i,\Lambda - (\beta+\alpha_i)\rangle}E_i^*)$ is an adjoint pair} \\ 
&= q_i([P], [Q]t_i^{-1}e_i) \quad \text{since $[Q]$ is of weight $-\Lambda + \beta+\alpha_i$}. 
\end{align*}
If $i \not \in J$, we have 
\begin{align*}
&([PF_i],[Q]) \\ 
&= \qdim \HOM_{R(\beta)\otimes R(\alpha_i)}(P \otimes R(\alpha_i), \Res_{\beta,\alpha_i}Q) \quad \text{by the induction-restriction adjunction} \\
&= \qdim \HOM_{R(\beta)}(P,\Res_{\beta,\alpha_i}Q) \\
&= \frac{1}{1-q_i^2} ([P], [QE_i]) \quad \text{since $\Res_{\beta,\alpha_i}Q$ is free over $R(\alpha_i)$ by Theorem \ref{thm:functorF*}}. 
\end{align*}
The proof is complete. 
\end{comment}
\end{proof}

\begin{definition}
We define a right $\dotcatquantum{\mathfrak{p}_J}$-module 
\begin{align*}
&{}_J\mathcal{V}(-\Lambda) \index{${}_J\mathcal{V}(-\Lambda)$} \\
&= 1_{-\Lambda}\dotcatquantum{\mathfrak{p}_J}\bigg/ \sum_{j \in J} \left( 1_{-\Lambda} E_j \dotcatquantum{\mathfrak{p}_J} + \xy 0;/r.17pc/: 
(-5,3)*{\scriptstyle -\Lambda};
(0,0)*{\sdotd{j}};
(7,0)*{\scriptstyle \langle h_j,\Lambda \rangle};
\endxy \ \dotcatquantum{\mathfrak{p}_J} + \END (1_{-\Lambda})_{\geq 1} \dotcatquantum{\mathfrak{p}_J} \right)
\end{align*}
as follows: 
for each $\lambda \in P$, we define the additive category ${}_J\mathcal{V}(-\Lambda)_{\lambda}$ as a quotient category of $\dotcatquantum{\mathfrak{p}_J}(\lambda, -\Lambda)$ by an ideal generated by 
\begin{itemize}
\item morphisms that factor through objects of the form 
\[
\bigoplus_{j \in J} E_j G_j \ (G_j \in \dotcatquantum{\mathfrak{p}_J}(\lambda,-\Lambda-\alpha_j)), 
\]
\item morphisms in $\xy 0;/r.17pc/: 
(-5,3)*{\scriptstyle -\Lambda};
(0,0)*{\sdotd{j}};
(7,0)*{\scriptstyle \langle h_j,\Lambda \rangle};
\endxy \ \dotcatquantum{\mathfrak{p}_J}(\lambda,-\Lambda+\alpha_j)$, 
\item morphisms in $\END (1_{-\Lambda})_{\geq 1} \dotcatquantum{\mathfrak{p}_J}(\lambda,-\Lambda)$. 
\end{itemize}
The actions of the 1-morphisms and the 2-morphisms are given in a natural way.
\end{definition}

\begin{theorem} \label{thm:rightuniversality}
There exists an equivalence of right $\dotcatquantum{\mathfrak{p}_J}$-modules
\[
{}_J\mathcal{V}(-\Lambda) \simeq \gproj{{}^{J,\Lambda}R},\ 1_{-\Lambda} \mapsto {}^{J,\Lambda}R(0). 
\]
\end{theorem}

\begin{proof}

By Lemma \ref{lem:universal2rep} and Theorem \ref{thm:cyclotomic2rep}, we have a morphism of right $\dotcatquantum{\mathfrak{p}_J}$-modules
\[
1_{-\Lambda}\dotcatquantum{\mathfrak{p}_J} \to \gproj{{}^{J,\Lambda}R},\ 1_{-\Lambda} \mapsto {}^{J,\Lambda}R(0). 
\]
Since ${}^{J,\Lambda}R(0) E_j = 0, {}^{J,\Lambda}R(0) \xy 0;/r.12pc/: 
(0,0)*{\sdotd{j}};
(9,0)*{\scriptstyle \langle h_j,\Lambda \rangle};
\endxy = 0 \ (j \in J)$, and $\END_{{}^{J,\Lambda}R(0)}({}^{J,\Lambda}R(0)) \simeq \mathbf{k}$ is concentrated in degree zero, 
it induces a morphism of right $\dotcatquantum{\mathfrak{p}_J}$-modules
\[
{}_J\mathcal{V}(-\Lambda) \to \gproj{{}^{J,\Lambda}R}. 
\]
In other words, we have a family of $\mathbf{k}$-linear functors
\[
{}_J\mathcal{V}(-\Lambda)_{-\Lambda + \beta} \to \gproj{{}^{J,\Lambda}R(\beta)} \ (\beta \in \mathsf{Q}_+). 
\]
%Since ${}^{J,\Lambda}R(\beta) = \bigoplus_{\nu \in I^{\beta}} {}^{J,\Lambda}R(0)F_{\nu_1} \cdots F_{\nu_{\height \beta}}$ is a progenerator, 
%the functor above is essentially surjective. 
We first prove that this functor is fully-faithful. 
Put $n = \height \beta$. 
By the triangular decomposition (Corollary \ref{cor:triangular}), 
every objects in ${}_J\mathcal{V}(-\Lambda)_{-\Lambda +\beta}$ is a direct summand of a finite direct sum of 1-morphisms 
\[
1_{-\Lambda} F_{\nu_1} \cdots F_{\nu_n} \ (\nu \in I^{\beta}). 
\]
Hence, it suffices to prove that the following homomorphism is an isomorphism: 
\[
f\colon \END_{{}_J\mathcal{V}(-\Lambda)} \left(\bigoplus_{\nu \in I^{\beta}} 1_{-\Lambda}F_{\nu_1} \cdots F_{\nu_n} \right) \to \END_{{}^{J,\Lambda}R(\beta)} ({}^{J,\Lambda}R(\beta)) \simeq {}^{J,\Lambda}R(\beta). 
\]
Note that we have a canonical surjective homomorphism 
\[
\END_{\dotcatquantum{\mathfrak{p}_J}} \left(\bigoplus_{\nu \in I^{\beta}} 1_{-\Lambda}F_{\nu_1} \cdots F_{\nu_n}\right) \to \END_{{}_J\mathcal{V}(-\Lambda)} \left(\bigoplus_{\nu \in I^{\beta}} 1_{-\Lambda}F_{\nu_1} \cdots F_{\nu_n} \right). 
\] 
By Theorem \ref{thm:KLRaction} and the definition of ${}_J\mathcal{V}(-\Lambda)$, 
it induces a surjective homomorphism
\[
g\colon {}^{J,\Lambda}R(\beta) \to \END_{{}_J\mathcal{V}(-\Lambda)} \left(\bigoplus_{\nu \in I^{\beta}} 1_{-\Lambda}F_{\nu_1} \cdots F_{\nu_n}\right). 
\]
Since the composition $g \circ f$ is the identity, the homomorphism $f$ is an isomorphim. 

Next, we prove that the functor is essentially surjective. 
It suffices to prove that any indecomposable projective module $P \in \gproj{{}^{J,\Lambda}R(\beta)}$ lies in the image. 
$P$ is a projective cover of some simple module $L \in \gmod{{}^{J,\Lambda}R(\beta)}$. 
Let $\tilde{P}$ be the projective cover of $L$ in $\gMod{R(\beta)}$. 
There exists an idempotent $e \in R(\beta)$ and $d \in \mathbb{Z}$ such that $\tilde{P} \simeq q^dR(\beta)e$. 
(By quotienting out the ideal of $R(\beta)$ generated by homogeneous components whose degree is sufficiently large, we can argue in the same way as for finite-dimensional algebras.)
This idempotent gives a direct summand $G$ of $\bigoplus_{\nu \in I^{\beta}}1_{-\Lambda}F_{\nu_1} \cdots F_{\nu_n}$ in $\dotcatquantum{\mathfrak{p}_J}$. 
The image of $q^dG$ under the functor is 
\[
q^d ({}^{J,\Lambda}R(\beta)e) \simeq {}^{J,\Lambda}R(\beta) \otimes_{R(\beta)} \tilde{P} \simeq P.
\]
Hence, $P$ lies in the image. 

Since we have proved that the functor is fully-faithful and essentially surjective, it is an equivalence of categories. 
\end{proof}

By applying the autofucntor $\sigma_*$ on $\gMod{R}$, we deduce similar results for $\gMod{R^{J,\Lambda}}$. 
We include the precise statements for later use. 

\begin{proposition} \label{prop:cyclotomicR*}
Let $j \in J, \beta \in \mathsf{Q}_+, X \in \gMod{R^{J,\Lambda}(\beta)}$ and put $n = \height \beta$. 
\begin{enumerate}
\item There exists a homomorphism of $(R(\beta+\alpha_j), R(\alpha_j))$-modules
\[
\mathsf{R} = \mathsf{R}_X \colon q^{(\alpha_j, 2\Lambda - \beta)} X \circ R(\alpha_j) \to R(\alpha_j)\circ X \ (X \in \gMod{R^{J,\Lambda}(\beta)})
\]
given by $u \boxtimes v \mapsto x_{n+1}^{\langle h_j, \Lambda \rangle} \tau_n \tau_{n-1} \cdots \tau_1 (v \boxtimes u) \ (u \in X, v \in R(\alpha_j))$. 
Furthermore, this homomorphism is natural in $X$.
\item There exists a homogeneous homomorphism of $(R(\beta+\alpha_j), R(\alpha_j))$-modules
\[
\mathsf{R}' = \mathsf{R}'_X \colon R(\alpha_j) \circ X \to X \circ R(\alpha_j) \ (X \in \gMod{R^{J,\Lambda}(\beta)})
\]
given by $v \boxtimes u \mapsto g_1 \cdots g_n (u\boxtimes v)$, where
\[
g_k = \sum_{\nu \in I^{\beta+\alpha_j}, \nu_k \neq \nu_{k+1}} \tau_k e(\nu) + \sum_{\nu \in I^{\beta + \alpha_j}, \nu_k = \nu_{k+1}} (x_{k+1} -x_k -(x_{k+1}-x_k)^2\tau_k) e(\nu). 
\]
Furthermore, this homomorphism is natural in $X$. 
\item The endomorphism $\mathsf{R}'_X \mathsf{R}_X$ is given by 
\[
u \boxtimes v \mapsto x_{n+1}^{\langle h_j, \Lambda \rangle} \sum_{\nu \in I^{\beta}} e(\nu,j) \prod_{1 \leq k \leq n, \nu_k \neq j} Q_{\nu_k, j}(x_k, x_{n+1}) (u\boxtimes v). 
\]
\item The endomorphism $\mathsf{R}_X \mathsf{R}'_X$ is given by 
\[
v \boxtimes u \mapsto x_{1}^{\langle h_j, \Lambda \rangle} \sum_{\nu \in I^{\beta}} e(j,\nu) \prod_{1 \leq k \leq n, \nu_k \neq j} Q_{j,\nu_k}(x_1, x_{k+1}) (v\boxtimes u). 
\]
\end{enumerate}
\end{proposition}

\begin{theorem} \label{thm:functorF2}
Let $\beta \in \mathsf{Q}_+$ and $X \in \gMod{R^{J,\Lambda}(\beta)}$. 
\begin{enumerate}
\item Let $i \in I \setminus J$.
Then, the canonical homomorphism $R(\alpha_i) \circ X \xrightarrow{\mathrm{can}} F_i X$ is an isomorphism. 
\item Let $j \in J$. 
Then, the following sequence is exact:
\[
0 \to q^{(\alpha_j,2\Lambda-\beta)} X \circ R(\alpha_j) \xrightarrow{\mathsf{R}} R(\alpha_j) \circ X \xrightarrow{\mathrm{can}} F_j X \to 0. 
\]
\end{enumerate}
\end{theorem}

\begin{theorem}
Let $\beta \in \mathsf{Q}_+$ and $j \in J$. 
Then, $e(j,\beta)R^{J,\Lambda}(\beta+\alpha_j)$ is a finitely-generated projective left $R^{J,\Lambda}(\beta)$-module,
and $R^{J,\Lambda}(\beta+\alpha_j)e(j,\beta)$ is a finitely-generated projective right $R^{J,\Lambda}(\beta)$-module. 
\end{theorem}

\begin{corollary} 
Let $i \in I, j \in J$.  
The functors $F_i, E_j$ are exact, and send finitely-generated modules (resp. projective modules) to finitely-generated ones (resp. projective ones). 
\end{corollary}

\begin{theorem} \label{thm:cyclotomic2repleft}
The category $\gMod{R^{J,\Lambda}}$ is a left $\dotcatquantum{\mathfrak{p}_J}$-module as follows: 
\begin{itemize}
\item The category attached to $\lambda \in \mathsf{P}$ is $\lambda \mapsto \gMod{R^{J,\Lambda}(-\lambda + \Lambda)}$, where we regard $R^{J,\Lambda}(-\lambda + \Lambda) = 0$ if $-\lambda + \Lambda \not \in \mathsf{Q}_+$. 
\item  Using the functors from Definition \ref{def:generatingfunctors}, the actions of the generating 1-morphisms are given as follows:
\[
F_i X = F_i (X) \ (i \in I), E_j X = E_j(X) \ (j \in J). 
\]
\item The actions of the generating 2-morphisms are given as follows ($X \in \gMod{R^{J,\Lambda}(\beta)}, \lambda = \Lambda - \beta$): 
\begin{align*}
&\left[ \xy 0;/r.17pc/: 
(0,0)*{\sdotd{i}};
\endxy X \colon q^{(\alpha_i,\alpha_i)} F_i X \to F_i X \right] \ (i\in I) \\
&= \text{the right multiplication by $x_1 \in R(\alpha_i)$}, \\
&\left[ \xy 0;/r.17pc/: 
(0,0)*{\dcross{i}{i'}};
\endxy X \colon q^{-(\alpha_i,\alpha_{i'})} F_i F_{i'} X \to F_{i'} F_i X \right] \ (i, i' \in I) \\
&= \text{the right multiplication by $e(i,i')\tau_1 \in R(\alpha_i + \alpha_{i'})$}, \\
&\left[ \xy 0;/r.17pc/: 
(0,0)*{\lcap{j}};
\endxy X \colon q_j^{1 + \langle h_j, \lambda \rangle} F_j E_j X \to X \right] \ (j \in J) \\
&= \text{the canonical unit for the adjoint pair $(F_j, q_j^{1 + \langle h_j, \lambda \rangle} E_j)$}, \\
&\left[ \xy 0;/r.17pc/: 
(0,0)*{\lcup{j}};
\endxy X \colon q_j^{1 - \langle h_j, \lambda \rangle} X \to E_j F_j(X) \right] \ (j\in J) \\
&= \text{the canonical counit for the adjoint pair $(F_j, q_j^{1 + \langle h_j, \lambda-\alpha_j \rangle} E_j)$}. 
\end{align*}
\end{itemize}
The action restricts to the additive subcategory $\gproj{R^{J,\Lambda}} = \bigoplus_{\beta \in \mathsf{Q}_+} \gproj{R^{J,\Lambda}(\beta)}$. 
Moreover, ignoring grading shifts, we have 
\begin{align*}
&\left[ \xy 0;/r.17pc/: 
(0,0)*{\sdotu{j}};
\endxy X  \colon E_j X \to E_jX \right] \ (j \in J) \\
&= \left[ e(j,*)X \xrightarrow{\text{the left multiplication by $x_1 e(j,*)$}} e(j,*)X \right], \\
&\left[ \xy 0;/r.17pc/: 
(0,0)*{\ucross{j}{j'}};
\endxy X \colon D_jE_{j'}X \to E_{j'}E_jX \right] \ (j,j' \in J) \\
&= \left[ e(j',j,*)X \xrightarrow{\text{the left multiplication by $\tau_1 e(j',j,*)$}} e(j,j',*)X \right].
\end{align*}
\end{theorem}
  
Hence, both $K(\gMod{R^{J,\Lambda}})_{\mathbb{Q}(q)}$ and $K_{\oplus}(\gproj{R^{J,\Lambda}})_{\mathbb{Q}(q)}$ are left $U_q(\mathfrak{p}_J)$-modules. 

\begin{remark} \label{rem:LRchange}
Theorem \ref{thm:cyclotomic2rep} and \ref{thm:cyclotomic2repleft} are related in the following way: 
The automorphism $\sigma_*$ of $\gMod{R}$ induces an isomorphism
\[
\sigma_* \colon \gMod{R^{J,\Lambda}} \to \gMod{{}^{J,\Lambda}R}. 
\]
On the other hand, we have an isomorphism of 2-categories
\[
\sigma \colon \catquantum{\mathfrak{p}_J} \to \catquantum{\mathfrak{p}_J}^{\mathrm{op}}. 
\]
Through these isomorphisms, the left action of $\catquantum{\mathfrak{p}_J}$ on $\gMod{R^{J,\Lambda}}$ (Theorem \ref{thm:cyclotomic2repleft}) coincides with the left action of $\catquantum{\mathfrak{p}_J}^{\mathrm{op}}$ on $\gMod{{}^{J,\Lambda}R}$ (Theorem \ref{thm:cyclotomic2rep}). 
\end{remark}

\begin{theorem} \label{thm:cyclotomiccategorificationleft}
(1) We have an isomorphism of left $U_q(\mathfrak{p}_J)$-modules
\begin{align*}
K_{\oplus}(\gproj{R^{J,\Lambda}})_{\mathbb{Q}(q)} &\simeq V_J(\Lambda), \\
[R^{J,\Lambda}(0)] &\mapsto v_{\Lambda}.  
\end{align*}

(2) The homomorphisms 
\begin{align*}
&K(\gmod{R^{J,\Lambda}})_{\mathbb{Q}(q)} \to K(\gMod{R^{J,\Lambda}})_{\mathbb{Q}(q)},\\
&K_{\oplus}(\gproj{R^{J,\Lambda}})_{\mathbb{Q}(q)} \to K(\gMod{R^{J,\Lambda}})_{\mathbb{Q}(q)}
\end{align*}
induced by the inclusions are isomorphisms.  
\end{theorem}

\begin{definition}
We define a left $\dotcatquantum{\mathfrak{p}_J}$-module 
\[
\mathcal{V}_J(\Lambda) = \dotcatquantum{\mathfrak{p}_J}1_{\Lambda} \bigg/ \sum_{j \in J} \left( \dotcatquantum{\mathfrak{p}_J} E_j 1_{\Lambda} + \dotcatquantum{\mathfrak{p}_J} \xy 0;/r.17pc/: 
(5,2)*{\scriptstyle \Lambda};
(0,0)*{\sdotd{j}};
(-7,0)*{\scriptstyle \langle h_j,\Lambda \rangle};
\endxy + \dotcatquantum{\mathfrak{p}_J}\END (1_{\Lambda})_{\geq 1} \right)
\]
as follows: 
for each $\lambda \in P$, we define the additive category $\mathcal{V}_J(\Lambda)_{\lambda}$ as a quotient category of $\dotcatquantum{\mathfrak{p}_J}(\Lambda,\lambda)$ by an ideal generated by 
\begin{itemize}
\item morphisms that factor through objects of the form
\[
\bigoplus_{j \in J}  G_jE_j \ (G_j \in \dotcatquantum{\mathfrak{p}_J}(\Lambda+\alpha_j,\lambda)), 
\]
\item morphisms in $\dotcatquantum{\mathfrak{p}_J}(\Lambda-\alpha_j,\lambda) \xy 0;/r.17pc/: 
(5,2)*{\scriptstyle \Lambda};
(0,0)*{\sdotd{j}};
(-7,0)*{\scriptstyle \langle h_j,\Lambda \rangle};
\endxy $, 
\item morphisms in $ \dotcatquantum{\mathfrak{p}_J}(\Lambda,\lambda)\END (1_{\Lambda})_{\geq 1}$. 
\end{itemize}
The actions of the 1-morphisms and the 2-morphisms are given in a natural way.
\end{definition}

\begin{theorem} \label{thm:leftuniversality}
There exists an equivalence of left $\dotcatquantum{\mathfrak{p}_J}$-modules
\[
\mathcal{V}_J(\Lambda) \simeq \gproj{R^{J,\Lambda}},\ 1_{\Lambda}\mapsto R^{J,\Lambda}(0). 
\]
\end{theorem}

%%%%%%%%%%%%%%%%%%%%%%%%%%%%
\section{Reflection functors}

Throughout this section, we fix $i \in I$. 
We simply write $\mathfrak{p}_i = \mathfrak{p}_{\{i\}}$. 

\subsection{The algebras ${}_i R$ and $R_i$}

\begin{definition}
We define 
\[
{}_i R(\beta) = R(\beta)/\langle e(i,\beta-\alpha_i) \rangle,\ R_i(\beta) = R(\beta)/ \langle e(\beta-\alpha_i,i)\rangle. 
\]
\end{definition}

\begin{lemma}
The categories $\gMod{{}_i R}$ and $\gMod{R_i}$ are both closed under convolution product. 
\end{lemma}

\begin{proof}
We prove the assertion for $\gMod{{}_iR}$. 
It suffices to show that for $M \in \gMod{{}_iR(\beta)}, N \in \gMod{{}_iR(\gamma)}$, the restriction $e(i,*) (M \circ N)$ is zero. 
It follows from the Mackey-filtration (Proposition \ref{prop:Mackey}) using $e(i,*)M = 0, e(i,*) N = 0$. 
\end{proof}

Note that ${}_i R(\beta) = {}^{\{i\}, 0} R(\beta), R_i (\beta) = R^{\{i\},0}(\beta)$. 
Hence, Theorem \ref{thm:cyclotomic2rep} shows that $\gMod{{}_i R}$ is a right $\dotcatquantum{\mathfrak{p}_i}$-module, 
and Theorem \ref{thm:cyclotomic2repleft} shows that $\gMod{R_i}$ is a left $\dotcatquantum{\mathfrak{p}_i}$-module.
In particular, $E_i$ and $F_j \ (j \in I)$ act on these categories.
For $u \in M \in \gMod{R_i(\beta)}$, we define $E_iu \in E_iM$ as $E_iu = e(i,\beta-\alpha_i)u$, recalling the definition of $E_i^{(n)}$ (Definition \ref{def:dividedpower}). 
Similarly, we define $E_i^{(n)}u \in E_i^{(n)}M$ as $E_i^{(n)}u = b_+(i^n) e(i^n,\beta-n\alpha_i)u$.  
We use similar notations for other analogous cases. 

\begin{proposition} \label{prop:extremalequiv}
For $\beta \in \sum_{j \in I, j \neq i} \mathbb{Z}_{\geq 0}\alpha_j$, 
we have mutually quasi-inverse functors
\[
\gMod{R_i(\beta)} \xtofrom[E_i^{(-\langle h_i, \beta \rangle)}]{F_i^{(-\langle h_i, \beta \rangle)}} \gMod{R_i(s_i\beta)},
\]
with natural isomorphisms 
\[
\varepsilon \colon F_i^{(-\langle h_i, \beta \rangle)} E_i^{(-\langle h_i,\beta \rangle)} \to \Id, \eta \colon \Id \to F_i^{(-\langle h_i, \beta \rangle)} E_i^{(-\langle h_i,\beta \rangle)} \ \text{(Definition \ref{def:unitcounit})}. 
\]
In addition, they induce a monoidal equivalence
\[
\bigoplus_{\beta \in \sum_{j \neq i}\mathbb{Z}_{\geq 0}\alpha_j} \gMod{R_i(\beta)} \simeq \bigoplus_{\beta \in \sum_{j \neq i}\mathbb{Z}_{\geq 0}\alpha_j} \gMod{R_i(s_i\beta)}. 
\]
More precisely, it is monoidal with the canonical isomorphism $E_i^{(0)}(\mathbf{1}) \simeq \mathbf{1}$, 
and the natural isomorphism $\phi_+ \colon E_i^{(m)}X \circ E_i^{(n)}Y \to E_i^{(m+n)}(X \circ Y)$ given by  \index{$\phi_+$}
\begin{align*}
&E_i^{(m)} x \boxtimes E_i^{(n)} y \\ 
&\mapsto E_i^{(m+n)} (e(i^m) \boxtimes \tau_{w[\height \beta, n]} \boxtimes e(\gamma))(b_+(i^m)e(i^m,\beta)x \boxtimes b_+(i^n)e(i^n,\gamma)y), 
\end{align*}
where $X \in \gMod{R_i(s_i\beta)}, Y \in \gMod{R_i(s_i\gamma)}, \beta, \gamma \in \sum_{j \in I, j \neq i}\mathbb{Z}_{\geq 0}\alpha_j, m = -\langle h_i,\beta \rangle$ and $n = -\langle h_i, \gamma \rangle$. 
\end{proposition}

\begin{proof}
Regarding the former assertion, observe first that $E_i \gMod{R_i(\beta)} = 0$. 
Hence, Theorem \ref{thm:EFrel} shows $E_i^{(-\langle h_i,\beta \rangle)} F_i^{(-\langle h_i, \beta \rangle)} \simeq \Id$ on $\gMod{R_i(\beta)}$. 
%Similarly, we have $F_i^{(-\langle h_i, \beta \rangle)} E_i^{(-\langle h_i,\beta \rangle)} \simeq \Id$ on $\gMod{R_i(s_i\beta)}$. 

We claim that $F_i \gMod{R_i(s_i\beta)} = 0$. 
Note that $V_i(0)$ is an integrable $U_q(\mathfrak{p}_i)$-module, that is, the action of $e_i$ and $f_i$ are locally-nilpotent.
Hence, the set of weights of $V_i(0)$ is stable under $s_i$. 
It implies that the weight space $V_i(0)_{-s_i\beta - \alpha_i}$ of weight $-s_i\beta-\alpha_i = s_i(-\beta+\alpha_i)$ is zero, since $\beta -\alpha_i \not \in \mathsf{Q}_+$. 
By Theorem \ref{thm:cyclotomiccategorificationleft}, the claim follows. 
Hence, Theorem \ref{thm:EFrel} shows $F_i^{(-\langle h_i,\beta \rangle)} E_i^{(-\langle h_i,\beta \rangle)} \simeq \Id$ on $\gMod{R_i(s_i\beta)}$. 
By Lemma \ref{lem:adjunction}, these natural isomorphisms can be explicitly given by the unit $\varepsilon$ and the counit $\eta$.  
It proves the former assertion. 

Next, we prove the latter assertion. 
We first verify that if such a homomorphism $\phi_+$ (a priori ungraded) exists, it is of degree zero. 
Remark \ref{rem:grading} shows $E_i^m X = e(i^m,\beta)X$, hence $E_i^{(m)} X = q_i^{-m(m-1)/2}b_+(i^m)e(i^m,\beta)X$. 
Similarly, we have 
\begin{align*}
&E_i^{(n)}Y = q_i^{-n(n-1)/2}b_+(i^n)e(i^n,\gamma)Y, \\ 
&E_i^{(m+n)}(X \circ Y) = q_i^{-(m+n)(m+n-1)/2}b_+(i^{m+n})e(i^{m+n},\beta+\gamma)(X \circ Y).
\end{align*} 
Note that $b_+(i^m), b_+(i^n)$ and $b_+(i^{m+n})$ are of degree zero. 
Therefore, $\phi_+$ is of degree zero by the following identity: 
\begin{align*}
\left( \frac{(m+n)(m+n-1)}{2} - \frac{m(m-1)}{2} - \frac{n(n-1)}{2} \right) \frac{(\alpha_i,\alpha_i)}{2}&= mn \frac{(\alpha_i,\alpha_i)}{2} \\
&= \deg (\tau_{w[\height \beta, n]}e(\beta,i^n)). 
\end{align*}

Since $E_i^{m+1} X =0, E_i^{n+1}Y = 0$, the Mackey filtration (\ref{prop:Mackey}) of the module $\Res_{(m+n)\alpha_i, \beta+\gamma}(X \circ Y)$ is one-step, which gives an isomorphism of $R((m+n)\alpha_i) \otimes R(\beta+\gamma)$-modules
\begin{align*}
&(\Ind_{m\alpha_i,n\alpha_i} \otimes \Ind_{\beta,\gamma}) (\Res_{m\alpha_i,\beta}X \otimes \Res_{n\alpha_i,\gamma}Y) \simeq \Res_{(m+n)\alpha_i, \beta+\gamma} (X \circ Y), \\
&e(i^m,\beta)x \otimes e(i^n,\gamma)y \mapsto (e(i^m) \boxtimes e(i^n,\beta) \tau_{w[\height \beta, n]} \boxtimes e(\gamma))(x \boxtimes y)
\end{align*}
On the other hand, since $R(m\alpha_i)b_+(i^m)$ is a progenerator of $R(m\alpha_i)$ and 
\[
\END_{R(m\alpha_i)}(R(m\alpha_i)b_+(i^m)) \simeq Z(m\alpha_i),
\]
we have an isomorphism of $R(m\alpha_i) \otimes R(\beta)$-modules
\begin{align*}
R(m\alpha_i)b_+(i^m) \otimes_{Z(m\alpha_i)} E_i^{(m)}X &= R(m\alpha_i)b_+(i^m) \otimes_{Z(m\alpha_i)} b_+(i^m) E_i^m X \\ 
&\xrightarrow{\text{multiplication}} E_i^m X. 
\end{align*}
Similarly, we have an isomorphism $R(n\alpha_i)b_+(i^n) \otimes_{Z(n\alpha_i)} E_i^{(n)}Y \to E_i^nY$. 
Hence, there is an isomorphism of $R((m+n)\alpha_i) \otimes R(\beta + \gamma)$-modules
\begin{align*}
&(R(m\alpha_i)b_+(i^m) \circ R(n\alpha_i)b_+(i^n)) \otimes_{Z(m\alpha_i)\otimes Z(n\alpha_i)} (E_i^{(m)}X \circ E_i^{(n)}Y) \\ 
&\simeq \Res_{(m+n)\alpha_i,\beta+\gamma}(X\circ Y), \\
&(b_+(i^m) \boxtimes b_+(i^n)) \otimes (E_i^{(m)} x \boxtimes E_i^{(n)} y) \\
&\mapsto (e(i^m) \boxtimes e(i^n,\beta)\tau_{w[\height \beta,n]} \boxtimes e(\gamma)) (b_+(i^m) e(i^m,\beta)x \boxtimes b_+(i^n) e(i^n,\gamma)y). 
\end{align*}
By multiplying $b_+(i^{m+n})$ from the left, we obtain an isomorphism of $R(\beta+\gamma)$-modules. 
Note that 
\begin{align*}
&b_+(i^{m+n}) (R(m\alpha_i)b_+(i^m) \circ R(n\alpha_i)b_+(i^n)) \\ 
&\simeq b_+(i^{m+n}) R((m+n)\alpha_i) (b_+(i^m) \boxtimes b_+(i^n)) \\ 
&= b_+(i^{m+n}) \mathbf{k}[x_1, \ldots, x_{m+n}] (b_+(i^m) \boxtimes b_+(i^n)) \\
&\quad \text{since $b_+(i^{m+n}) \tau_k = 0$ for any $1 \leq k < m+n$} \\
&= b_+(i^{m+n}) (Z(m\alpha_i) \boxtimes Z(n\alpha_i)), 
\end{align*}
where the last equality follows from $\tau_{w_{m+n}}e(i^{m+n}) = \tau_{w[m,n]}(\tau_{w_m} \boxtimes \tau_{w_n})e(i^{m+n})$ and 
\[
\tau_{w_m} f(x_1, \ldots, x_m) b_+(i^m) = \tau_{w_m} \partial_{w_m}(f(x_1, \ldots, x_m)\mathbf{x}_m). 
\]
Hence, we have an isomorphism of $Z(m\alpha_i) \otimes Z(n\alpha_i)$-modules
\begin{align*}
&Z(m\alpha_i) \otimes Z(n\alpha_i) \simeq b_+(i^{m+n})(R(m\alpha_i) b_+(i^m) \circ R(n\alpha_i)b_+(i^n)), \\
&1 \otimes 1 \mapsto b_+(i^{m+n}) (b_+(i^m) \boxtimes b_+(i^n)) = b_+(i^{m+n}). 
\end{align*}
Finally, we obtain the isomorphism $\phi_+$ as desired. 

The unitality is obvious. 
To prove the associativity, take another module $Z \in \gMod{R_i(s_i\delta)}$ for some $\delta \in \sum_{j \in I, j \neq i} \mathbb{Z}_{\geq 0}\alpha_j$ and put $l = -\langle h_i,\delta \rangle$. 
The composition $E_i^{(m)}X \circ E_i^{(n)}Y \circ E_i^{(l)}Z \xrightarrow{\phi_+} E_i^{(m)}X \circ E_i^{(n+l)}(Y \circ Z) \xrightarrow{\phi_+} E_i^{(m+n+l)}(X \circ Y \circ Z)$ is given by 
\begin{align*}
&E_i^{(m)} x \boxtimes E_i^{(n)} y \boxtimes E_i^{(l)} z \\
&\mapsto E_i^{(m+n+l)} (e(i^m) \boxtimes e(i^{n+l},\beta)\tau_{w[\height \beta,n+l]} \boxtimes e(\gamma + \delta))  \times \\ 
&\quad (e(i^m,\beta, i^n) \boxtimes e(i^l,\gamma)\tau_{w[\height \gamma,l]} \boxtimes e(\delta)) \times \\ 
&\quad (b_+(i^m) e(i^m,\beta)x \boxtimes b_+(i^n) e(i^n,\gamma)y \boxtimes b_+(i^l) e(i^l,\delta)z). 
\end{align*} 
On the other hand, 
the composition $E_i^{(m)}X \circ E_i^{(n)}Y \circ E_i^{(l)}Z \xrightarrow{\phi_+} E_i^{(m+n)}(X \circ Y) \circ E_i^{(l)}Z \xrightarrow{\phi_+} E_i^{(m+n+l)}(X \circ Y \circ Z)$ is given by 
\begin{align*}
&E_i^{(m)} x \boxtimes E_i^{(n)} y \boxtimes E_i^{(l)} z \\
&\mapsto E_i^{(m+n+l)} (e(i^{m+n}) \boxtimes e(i^l,\beta+\gamma)\tau_{w[\height (\beta + \gamma),l]} \boxtimes e(\delta)) \times \\
&\quad (e(i^m) \boxtimes e(i^n,\beta)\tau_{w[\height \beta,n]} \boxtimes e(\gamma,i^l,\delta)) \times \\
&\quad (b_+(i^m) e(i^m,\beta)x \boxtimes b_+(i^n) e(i^n,\gamma)y \boxtimes b_+(i^l) e(i^l,\delta)z). 
\end{align*}
These two morphisms coincide by Lemma \ref{lem:tau}, which proves the associativity. 

\end{proof} 

\begin{remark} \label{rem:extension}
It is straightforward to extend this result to a monoidal equivalence
\[
\bigoplus_{\beta \in \mathsf{Q}_+} \{ X \in \gMod{R_i(\beta)} \mid E_iX = 0\} \simeq \bigoplus_{\beta \in \mathsf{Q}_+} \{ X \in \gMod{R_i(s_i\beta)} \mid F_iX = 0\}. 
\]
\end{remark}

\begin{definition}
Let $\beta, \gamma \in \sum_{j \in I, j \neq i}\mathbb{Z}_{\geq 0}\alpha_j$ and put $m = -\langle h_i, \beta \rangle, n = -\langle h_i,\gamma \rangle$. 
For $X \in \gMod{R_i(\beta)}, Y \in \gMod{R_i(\gamma)}$, we define $\phi_- \colon F_i^{(m+n)}(X \circ Y) \to F_i^{(m)}X \circ F_i^{(n)}Y$ as \index{$\phi_-$}
\begin{align*}
&F_i^{(m+n)}(X \circ Y) \xrightarrow{\eta_m \otimes \eta_n} F_i^{(m+n)}(E_i^{(m)}F_i^{(m)}X \circ E_i^{(n)}F_i^{(n)}Y) \\  
&\xrightarrow{\phi_+} F_i^{(m+n)}E_i^{(m+n)} (F_i^{(m)}X \circ F_i^{(n)}Y) \xrightarrow{\varepsilon_{m+n}} F_i^{(m)}X \circ F_i^{(n)}Y. 
\end{align*}
\end{definition}

It is straightforward to verify that $\phi_-$ is unital and associative.
Furthermore, the morphism
\begin{align*}
&E_i^{(m)}(X) \circ E_i^{(n)}(Y) \xrightarrow{\eta_{m+n}} E_i^{(m+n)}F_i^{(m+n)}(E_i^{(m)}(X) \circ E_i^{(n)}(Y)) \\ 
&\xrightarrow{\phi_-} E_i^{(m+n)}(F_i^{(m)}E_i^{(m)}X\circ F_i^{(n)}E_i^{(n)}Y) \xrightarrow{\varepsilon_m \otimes \varepsilon_n} E_i^{(m+n)}(X \circ Y)
\end{align*}
coincides with $\phi_+$. 

Note that the canonical homomorphism $R(m\alpha_i) \circ X \twoheadrightarrow F_i^mX$ induces 
\[
R(m\alpha_i)b_+(i^m) \circ X \twoheadrightarrow F_i^{(m)}X,
\] 
since the endomorphism $b_-(i^m)$ of $F_i^mX$ is given by the right multiplication by $\varphi(b_-(i^m)) = b_+(i^m)$. 
The image of $u \boxtimes v \in R(m\alpha_i)b_+(i^m) \boxtimes X \subset R(m\alpha_i)b_+(i^m) \circ X$ in $F_i^{(m)}X$ is also denoted by $u \boxtimes v$.  

\begin{lemma} \label{lem:phi-}
The isomorphism $\phi_- \colon F_i^{(m+n)}(X \circ Y) \to F_i^{(m)}X \circ F_i^{(n)}Y$ is given by 
\begin{align*}
&\phi_-(b_+(i^{m+n}) \boxtimes (u \boxtimes v)) \\ 
&= b_+(i^{m+n}) (e(i^m) \boxtimes \tau_{w[\height \beta,n]} \boxtimes e(\gamma)) ((b_+(i^m) \boxtimes u) \boxtimes (b_+(i^n)\boxtimes v)), 
\end{align*}
for $u \in X, v \in Y$. 
\end{lemma}

\begin{proof}
It is immediate from the definition. 
\end{proof}

As for $\gMod{{}_i R}$, the following parallel proposition holds. 

\begin{proposition} \label{prop:extremalequiv*}
For $\beta \in \sum_{j \in I, j \neq i} \mathbb{Z}_{\geq 0}\alpha_j$, 
we have mutually quasi-inverse functors
\[
\gMod{{}_iR(\beta)} \xtofrom[\times E_i^{(-\langle h_i, \beta \rangle)'}]{\times F_i^{(-\langle h_i, \beta \rangle)'}} \gMod{{}_iR(s_i\beta)}, 
\]
with natural isomorphisms 
\[
\varepsilon' \colon E_i^{(-\langle h_i, \beta \rangle)'} F_i^{(-\langle h_i,\beta \rangle)'} \to \Id, \eta' \colon \Id \to E_i^{(-\langle h_i, \beta \rangle)'} F_i^{(-\langle h_i,\beta \rangle)'} \ \text{(Definition \ref{def:unitcounit})}. 
\]
In addition, they induce a monoidal equivalence
\[
\bigoplus_{\beta \in \sum_{j \neq i}\mathbb{Z}_{\geq 0}\alpha_j} \gMod{{}_iR(\beta)} \simeq \bigoplus_{\beta \in \sum_{j \neq i}\mathbb{Z}_{\geq 0}\alpha_j} \gMod{{}_iR(s_i\beta)}. 
\]
More precisely, it is monoidal with the canonical isomorphism $\mathbf{1}E_i^{(0)} \simeq \mathbf{1}$, 
and the natural isomorphism $\phi'_+ \colon X E_i^{(m)'} \circ Y E_i^{(n)'} \to (X \circ Y)E_i^{(m+n)'}$ given by \index{$\phi_+'$}
\begin{align*}
&x E_i^{(m)'} \boxtimes y E_i^{(n)'} \mapsto \\ 
&[(e(\beta) \boxtimes \tau_{w[m,\height \gamma]} \boxtimes e(i^n)) ((e(\beta) \boxtimes e(i^m)b'_+(i^m))x \boxtimes (e(\gamma)\boxtimes e(i^n)b'_+(i^n))y)]E_i^{(m+n)'}. 
\end{align*}
\end{proposition}

\subsection{Main construction}

As in the previous section, the category $\gMod{R_i}$ is a right $\dotcatquantum{\mathfrak{p}_i}$-module by Theorem \ref{thm:cyclotomic2repleft} and $\gMod{{}_iR}$ is a right $\dotcatquantum{\mathfrak{p}_i}$-module by Theorem \ref{thm:cyclotomic2rep}. 
We begin our construction of reflection functors with the following definition, which is based on Propositions \ref{prop:extremalequiv} and \ref{prop:extremalequiv*}. 

\begin{definition} \label{def:anothersimple}
(1) For $j \in I \setminus \{i\}$, we define $M_j \in \gproj{R_i(s_i\alpha_j)}$ and $M'_j \in \gproj{{}_iR(s_i\alpha_j)}$ by \index{$M_j,M'_j$}
\[
M_j = F_i^{(-a_{i,j})}R(\alpha_j), \ M'_j = R(\alpha_j) F_i^{(-a_{i,j})'}. 
\]
(2) For $j \in I \setminus \{i\}$, we define an endomorphism
\[
\text{$y_j \in \END_{R_i(s_i\alpha_j)}(M_j)_{(\alpha_j,\alpha_j)}$ (resp. $y'_j \in \END_{{}_iR(s_i\alpha_j)}(M'_j)_{(\alpha_j,\alpha_j)}$)}\index{$y_j, y'_j$}
\] 
as the one obtained by applying $F_i^{(-a_{i,j})}$ from the left (resp. $F_i^{(-a_{i,j})'}$ from the right) to the endomorphism
\[
R(\alpha_j) \xrightarrow{\times x_1} R(\alpha_j). 
\]
(3) For $j, k \in I \setminus \{i \}$, we define homomorphisms
\begin{align*}\index{$\sigma_{j,k},\sigma'_{j,k}$ for $j,k \neq i$}
\sigma_{j,k} \in \HOM_{R_i(s_i(\alpha_j+\alpha_k))} (M_j \circ M_k, M_k \circ M_j)_{-(\alpha_j,\alpha_k)}, \\
\sigma'_{j,k} \in \HOM_{{}_iR(s_i(\alpha_j+\alpha_k))} (M'_j \circ M'_k, M'_k \circ M'_j)_{-(\alpha_j,\alpha_k)}, \\
\end{align*} 
as follows. 
Note that we have isomorphisms
\begin{align*}
M_j \circ M_k &= F_i^{(-a_{i,j})}R(\alpha_j) \circ F_i^{(-a_{i,k})}R(\alpha_k) \xrightarrow{\phi_-^{-1}} F_i^{(-a_{i,j}-a_{i,k})}(R(\alpha_j) \circ R(\alpha_k)), \\
M_k \circ M_j &= F_i^{(-a_{i,k})}R(\alpha_k) \circ F_i^{(-a_{i,j})}R(\alpha_j) \xrightarrow{\phi_-^{-1}} F_i^{(-a_{i,j}-a_{i,k})}(R(\alpha_k) \circ R(\alpha_j)), 
\end{align*}
and a homomorphism of left $R(\alpha_j+\alpha_k)$-modules
\[
R(\alpha_j) \circ R(\alpha_k) \simeq R(\alpha_j+\alpha_k)e(j,k) \xrightarrow{\times \tau_1} R(\alpha_j+\alpha_k)e(k,j) \simeq R(\alpha_k)\circ R(\alpha_j). 
\]
The homomorphism $\sigma_{j,k}$ is defined as the one obtained by applying $F_i^{(-a_{i,j}-a_{i,k})}$ to this homomorphism. 
The homomorphism $\sigma'_{j,k}$ is defined in the same way. 
\end{definition}

\begin{lemma} \label{lem:yaction}
Let $j \in I \setminus \{ i\}$, and put $n = -a_{i,j}$. 
Then, the endomorphism $y_j$ of $M_j$ coincides with the left action of $x_{n+1}$. 
Similarly, the endomorphism $y'_j$ of $M'_j$ coincides with the left action of $x_1$. 
\end{lemma}

\begin{proof}
We only prove the assertion for $M_j$, as the proof of the latter assertion is completely parallel. 
Note that $M_j$ is an $R_i(s_i\alpha_j)$-module, and $s_i \alpha_j = n\alpha_i + \alpha_j$. 
In $R_i(s_i\alpha_j)$, we have $e(*,i) = 0$, hence $e(n\alpha_i,j) = 1$. 
In addition, $\tau_n = 0$ and $x_{n+1}$ is central. 
Hence, the left action of $x_{n+1}$ on $M_j$ gives a left $R(s_i\alpha_j)$-module endomorphism of $M_j$.
Let $f$ denote this endomorphism.  
By Proposition \ref{prop:extremalequiv} and $M_j = F_i^{(n)}R(\alpha_j)$, 
it suffices to prove that $E_i^{(n)}y_j = E_i^{(n)}f$ as endomorphisms of $E_i^{(n)}M_j$. 
We identify 
\[
E_i^{(n)}M_j = E_i^{(n)}F_i^{(n)}R(\alpha_j) \xrightarrow{\eta^{-1}, \simeq} R(\alpha_j). 
\] 
By the definition of $y_j$, the endomorphism $E_i^{(n)}y_j$ is the left action of $x_1$ on $R(\alpha_j)$. 
On the other hand, the endomorphism $E_i^{(n)}f$ of $E_i^{(n)}f$ is given by the left action of $x_1$, 
since $f$ is the left action of $x_{n+1}$ and $E_i^{(n)}M_j = b_+(i^n)e(i^n,j)M_j$ up to grading shift. 
It follows that $E_i^{(n)}y_j = E_i^{(n)}f$. 

\begin{comment}
It follows that $E_i^n M_j = M_j$ as vector spaces. 
We regard $M_j$ as an $R(\alpha_j)$-module by letting $x_1 \in R(\alpha_j)$ act on $M_j$ as the left multiplication by $x_{n+1} \in R(s_i\alpha_j)$. 
Then, $E_i^n M_j = M_j$ as $R(\alpha_j)$-modules.  
Note that $E_i^n M_j \simeq P(i^n) \otimes_{Z(n\alpha_i)} E_i^{(n)}M_j$ as $R(\alpha_j)$-modules. 
Finally, by Proposition \ref{prop:extremalequiv}, $E_i^{(n)} M_j \simeq R(\alpha_j)$ as $R(\alpha_j)$-modules and $E_i^{(n)}y_j$ coincides with the left action of $x_1 \in R(\alpha_j)$. 
Hence, the lemma is proved. 
\end{comment}
\end{proof}

Let $j \in I$ and $X, Y \in \gMod{R_i}$. 
We have a canonical surjective homomorphism 
\[
R(\alpha_j) \circ (X \circ Y) \simeq (R(\alpha_j) \circ X) \circ Y \twoheadrightarrow F_jX \circ Y. 
\]
Since $F_jX \circ Y \in \gMod{R_i}$, it induces a surjective homomorphism 
\[
F_j(X \circ Y) \twoheadrightarrow F_jX \circ Y. 
\]
We have similar homomorphisms for the right action of $F_j$ on $\gMod{{}_iR}$. 

\begin{lemma} \label{lem:cansurj}
Let $j \in I$. 
(1) For $X, Y \in \gMod{R_i}$, the canonical surjective homomorphism $F_j (X \circ Y) \twoheadrightarrow F_jX \circ Y$ is natural in $X$ and $Y$, 
and commutes with the action of $R(\alpha_j)$ on $F_j$. 
If $j \neq i$, it is an isomorphism. 

(2) Let $k \in I$ and $X, Y\in \gMod{R_i}$. 
The following diagram commutes: 
\[
\begin{tikzcd}
F_jF_k(X \circ Y) \arrow[r] \arrow[d, "{\xy 0;/r.15pc/: 
(0,0)*{\dcross{j}{k}};
\endxy}"'] & F_j(F_kX \circ Y) \arrow[r] & F_jF_k(X \circ Y) \arrow[d,"{\xy 0;/r.15pc/: 
(0,0)*{\dcross{j}{k}};
\endxy}"] \\
F_kF_j(X \circ Y) \arrow[r] & F_k(F_j X \circ Y) \arrow[r] & F_kF_j(X \circ Y). 
\end{tikzcd}
\]
\end{lemma}

Similar assertions hold for the right action of $F_j$ on $\gMod{{}_iR}$. 

\begin{proof}
It is immediate from the definition and Theorem \ref{thm:functorF2}. 
\end{proof}

\begin{lemma} \label{lem:adjointSES}
(1) Let $\beta, \gamma \in \mathsf{Q}_+, X \in \gMod{R_i(\beta)}, Y \in \gMod{R_i(\gamma)}$. 
Then, there are short exact sequences in $\gMod{R_i}$
\begin{align*}
&0 \to q_i^{\langle h_i,\gamma \rangle}E_i X \circ Y \to E_i(X\circ Y) \to  X \circ E_i Y \to 0, \\
&0 \to q_i^{-\langle h_i, \beta \rangle}X \circ F_i Y \to F_i (X\circ Y) \to F_i X \circ Y \to 0,
\end{align*}
where the homomorphisms are given by
\begin{align*}
&q_i^{\langle h_i,\gamma \rangle} E_iX \circ Y \to E_i(X \circ Y),\ E_iu \boxtimes v \mapsto E_i(u \boxtimes v), \\
&E_i(X \circ Y) \to X \circ E_iY, \ E_i(u \boxtimes v) \mapsto 0, \ E_i\tau_1 \cdots \tau_{\height \beta} (u \boxtimes v) \mapsto u \boxtimes E_iv, \\
&q_i^{-\langle h_i,\beta \rangle} X \circ F_iY \to F_i(X \circ Y),\ u \boxtimes (e(i) \boxtimes v) \mapsto \tau_{\height \beta} \cdots \tau_1 (e(i) \boxtimes (u\boxtimes v)), \\
&F_i(X \circ Y) \to F_iX \circ Y, \ e(i) \boxtimes (u \boxtimes v) \mapsto (e(i)\boxtimes u) \boxtimes v. 
\end{align*}
These homomorphisms are natural in $X$ and $Y$, and commute with the action of $R(\alpha_i)$ on $E_i$ or $F_i$. 

(2) Let $\beta, \gamma \in \mathsf{Q}_+, X \in \gMod{{}_iR(\beta)}, Y \in \gMod{{}_iR(\gamma)}$. 
Then, there are short exact sequences in $\gMod{R_i}$
\begin{align*}
&0 \to q_i^{\langle h_i,\beta \rangle}X \circ YE_i \to (X\circ Y)E_i \to X E_i \circ  Y \to 0, \\
&0 \to q_i^{-\langle h_i, \gamma \rangle}X F_i \circ Y \to (X\circ Y) F_i \to X \circ YF_i \to 0, 
\end{align*}
where the homomorphisms are given by 
\begin{align*}
&q_i^{\langle h_i,\beta \rangle} X \circ YE_i \to (X \circ Y)E_i,\ u \boxtimes vE_i \mapsto (u \boxtimes v)E_i, \\
&(X \circ Y)E_i \to XE_i \circ Y,\ (u\boxtimes v)E_i \mapsto 0, \ [\tau_{\height (\beta+\gamma)} \cdots \tau_{\height \beta +1} (u \boxtimes v)]E_i \mapsto uE_i \boxtimes v, \\
&q_i^{-\langle h_i,\gamma \rangle}XF_i \circ Y \to (X \circ Y)F_i,\ (u \boxtimes e(i)) \boxtimes v \mapsto \tau_{\height \beta+1} \cdots \tau_{\height (\beta + \gamma)}((u \boxtimes v) \boxtimes e(i)), \\
&(X \circ Y)F_i \to X \circ YF_i,\ (u \boxtimes v) \boxtimes e(i) \mapsto u \boxtimes (v \boxtimes e(i)). 
\end{align*}
These homomorphisms are natural in $X$ and $Y$, and commute with the action of $R(\alpha_i)$ on $E_i$ or $F_i$. 
\end{lemma}

We simply refer to the homomorphisms above as canonical homomorphisms. 
Note that the canonical surjections for $F_i$ on $\gMod{R_i}$ in Lemma \ref{lem:adjointSES} coincides with those in Lemma \ref{lem:cansurj}. 

\begin{proof}
(1)
The first short exact sequence is a special case of the Mackey-filtration (Proposition \ref{prop:Mackey}).
To prove the second one, consider the following sequence of injective homomorphisms from Proposition \ref{prop:cyclotomicR*}: 
\[
q_i^{-\langle h_i, \beta+\gamma \rangle} X\circ Y \circ R(\alpha_i) \xrightarrow{\id_X \otimes \mathsf{R}_Y} q_i^{-\langle h_i, \beta \rangle} X \circ R(\alpha_i) \circ Y \xrightarrow{\mathsf{R}_X \otimes \id_Y} R(\alpha_i) \circ X \circ Y. 
\]
It yields a short exact sequence
\[
0 \to q_i^{-\langle h_i,\beta \rangle}\Cok (\id_X \otimes \mathsf{R}_Y) \to \Cok ((\mathsf{R}_X \otimes \id_Y)(\id_X \otimes \mathsf{R}_Y)) \to \Cok (\mathsf{R}_X \otimes \id_Y) \to 0. 
\]
Note that we have $(\mathsf{R}_X \otimes \id_Y) (\id_X \otimes \mathsf{R}_Y) = \mathsf{R}_{X\circ Y}$ by definition.   
By Theorem \ref{thm:functorF2} (2), the short exact sequence above is the desired one. 
Since $\mathsf{R}_X$ (resp. $\mathsf{R}_Y$) is natural in $X$ (resp. $Y$) and is right $R(\alpha_i)$-linear (Proposition \ref{prop:cyclotomicR}), 
the naturality and the commutativity with the action of $R(\alpha_i)$ also follow. 

The proof of (2) is analogous to (1). 
\end{proof}

The homomorphisms of Lemma \ref{lem:adjointSES} are compatible with convolution products as follows. 

\begin{lemma} \label{lem:EFvsmultiplication}
For $X \in \gMod{R_i(\alpha)}, Y \in \gMod{R_i(\beta)}, Z \in \gMod{R_i(\gamma)}$, the following diagrams commute (we suppress degree shifts): 
\[
\begin{tikzcd}
E_iX \circ Y \circ Z \arrow[rr] \arrow[rd] && E_i(X \circ Y \circ Z ) \\
& E_i (X \circ Y) \circ Z \arrow[ru] &
\end{tikzcd}
\]
\[
\begin{tikzcd}
E_i(X \circ Y \circ Z) \arrow[rr] \arrow[rd] && X \circ Y \circ E_iZ \\
& X \circ E_i (Y \circ Z) \arrow[ru] &
\end{tikzcd}
\]
\[
\begin{tikzcd}
X \circ Y \circ F_iZ \arrow[rr] \arrow[rd] && F_i(X \circ Y \circ Z ) \\
& X \circ F_i(Y \circ Z) \arrow[ru] &
\end{tikzcd}
\]
\[
\begin{tikzcd}
F_i(X \circ Y \circ Z) \arrow[rr] \arrow[rd] && F_iX \circ Y \circ Z \\
& F_i(X \circ Y) \circ Z \arrow[ru] &
\end{tikzcd}
\]
\end{lemma}

\begin{proof}
It is immediate from the definition of morphisms. 
\end{proof}

\begin{lemma} \label{lem:naturality}
Let $X \in \gMod{R_i(\beta)}, Y \in \gMod{R_i(\gamma)}$. We suppress degree shifts. 
\begin{enumerate}
\item The canonical homomorphism $X \circ F_iY \hookrightarrow F_i(X \circ Y)$ commutes with $\xy 0;/r.12pc/: (0,0)*{\sdotd{i}}; \endxy$. 
\item The canonical homomorphism $X \circ F_iF_iY \hookrightarrow F_iF_i(X \circ Y)$ commute with $\xy 0;/r.12pc/: (0,0)*{\dcross{i}{i}}; \endxy$. 
\item The canonical homomorphisms $E_iE_iX \circ Y \hookrightarrow E_iE_i(X \circ Y)$ and $E_iE_i(X \circ Y) \twoheadrightarrow X \circ E_iE_iY$ commute with $\xy 0;/r.12pc/: (0,0)*{\dcross{i}{i}}; \endxy$. 
\end{enumerate}
\end{lemma}

\begin{proof}
(1) It follows from the fact that the homomorphism is induced from $\mathsf{R}_X \colon X \circ (\alpha_i) \to R(\alpha_i) \circ X$, 
which is right $R(\alpha_i)$-linear (Proposition \ref{prop:cyclotomicR*}). 

(2) Note that the homomorphism is induced from 
\[
X \circ R(2\alpha_i) \to R(2\alpha_i) \circ X, \ u \otimes v \mapsto \tau_{\height \beta} \cdots \tau_1 \tau_{\height \beta + 1} \cdots \tau_2 (v \otimes u). 
\]
By \cite[Proposition 2.12]{MR3771147}, it commutes with the right multiplication of $R(2\alpha_i)$, hence the assertion holds. 

(3) Put $V = \Res_{2\alpha_i,\beta+\gamma-2\alpha_i} (X \circ Y)$ and consider its Mackey-filtration as in Proposition \ref{prop:Mackey}. 
Then, the injective homomorphism $E_iE_iX \circ Y \to E_iE_i(X \circ Y)$ coincides with the embedding $F_{\leq e} V \to V$, 
and the surjective homomorphism $E_iE_i(X \circ Y) \to X \circ E_iE_iY$ coincides with the quotient $V \to V / F_{< w_{[\height \beta, 2]}}V$. 
In particular, they are $R(2\alpha_i)$-linear, hence the assertion follows. 
\end{proof}

\begin{lemma} \label{lem:sigmaji}
Let $j \in I \setminus \{i\}, \beta \in \mathsf{Q}_+$. 
We suppress degree shifts here. 

(1)  Let $X \in \gMod{R_i(\beta)}$. 
The endomorphism $Q_{i,j}\left(\xy 0:/r.10pc/: (0,0)*{\sdotu{i}}; \endxy, y_j\right)$ of $E_i(X \circ M_j)$ factors through the canonical injective homomorphism $(E_i X) \circ M_j \to E_i (X\circ M_j)$ of Lemma \ref{lem:adjointSES}. 

(2) Let $X \in \gMod{{}_iR(\beta)}$. 
The endomorphism $Q_{j,i}\left(y_j,\xy 0:/r.10pc/: (0,0)*{\sdotu{i}}; \endxy \right)$ of $(M'_j \circ X)E_i$ factors through the canonical injective homomorphism $M'_j \circ (X E_i) \to (M'_j \circ X)E_i$ of Lemma \ref{lem:adjointSES}.  
\end{lemma}

\begin{proof} 
(1) By Lemma \ref{lem:adjointSES}, it suffices to prove that the composite homomorphism
\[
E_i(X\circ M_j) \xrightarrow{Q_{i,j}\left(\xy 0:/r.10pc/: (0,0)*{\sdotu{i}}; \endxy, y_j\right)} E_i(X\circ M_j) \xrightarrow{\mathrm{can}} X \circ (E_i M_j)
\]
is zero. 
Since the canonical surjection $E_i(X \circ M_j) \to X \circ E_iM_j$ is natural in $M_j$ and is commute with the action of $R(\alpha_i)$ on $E_i$, the composite morphisms above is equal to
\[
E_i(X \circ M_j) \xrightarrow{\mathrm{can}} X \circ (E_i M_j) \xrightarrow{\id_X \otimes Q_{i,j}\left(\xy 0:/r.10pc/: (0,0)*{\sdotu{i}}; \endxy, y_j\right)} X \circ (E_i M_j). 
\]
Hence, it suffices to prove the endomorphism $Q_{i,j}\left(\xy 0:/r.13pc/: (0,0)*{\sdotu{i}}; \endxy, y_j\right)$ of $E_i M_j$ is zero. 
Recall that $M_j$ is an $R_i(s_i\alpha_j)$-module, and $s_i \alpha_j = -a_{i,j}\alpha_i + \alpha_j$. 
Put $n = -a_{i,j}$. 
In $R_i(s_i\alpha_j)$, we have $e(*,i) = 0$, hence $e(n\alpha_i,j) = 1$. 
Thus, 
\[
\text{$Q_{i,j}(x_n, x_{n+1}) = \tau_n^2 = \tau_n e(*,j,i) \tau_n = 0$ in $R_i(s_i\alpha_j)$.}
\]
It follows that $Q_{i,j}(x_n,x_{n+1})M_j = 0$.
By Lemma \ref{lem:yaction} and \cite[Lemma 4.2]{MR2995184}, we deduce $Q_{i,j}\left(\xy 0:/r.13pc/: (0,0)*{\sdotu{i}}; \endxy, y_j\right)E_iM_j = Q_{i,j}(x_1,x_{n+1})e(i,*)M_j = 0$. 

The proof of (2) is analogous to (1). 
\end{proof}

\begin{definition} \label{def:anothertauij}
Let $j \in I \setminus \{i\}$.
(1) Let $X \in \gMod{R_i(\beta)}$. \index{$\sigma_{i,j}, \sigma_{j,i}$ for $j \neq i$}
We define $\sigma_{i,j} = \sigma_{i,j}(X) \colon q^{-(\alpha_i,\alpha_j)} E_iX \circ M_j \to E_i(X \circ M_j)$ to be the canonical injection as in Lemma \ref{lem:adjointSES}. 
We define $\sigma_{j,i} = \sigma_{j,i}(X) \colon q^{-(\alpha_i,\alpha_j)} E_i(X \circ M_j) \to E_i X \circ M_j$ to be the homomorphism characterized by 
\[
\text{$\sigma_{i,j} \sigma_{j,i} = Q_{i,j}\left(\xy 0:/r.10pc/: (0,0)*{\sdotu{i}}; \endxy, y_j\right)$ in $\END (E_iX \circ M_j)$},  
\]
whose existence is guaranteed by Lemma \ref{lem:sigmaji}. 
Note that both $\sigma_{i,j}$ and $\sigma_{j,i}$ are natural in $X$, and commute with the action of $R(\alpha_i)$ on $E_i$ and endomorphisms of $M_j$. 

(2) For $X \in \gMod{{}_iR(\beta)}$, we define 
\begin{align*}
\sigma'_{j,i} = \sigma'_{j,i}(X) \colon q^{-(\alpha_i,\alpha_j)} M'_j \circ X E_i \to (M'_j \circ X)E_i,  \\
\sigma'_{i,j} = \sigma'_{i,j}(X) \colon q^{-(\alpha_i,\alpha_j)} (M'_j \circ X)E_i \to M'_j \circ X E_i, 
\end{align*} \index{$\sigma'_{j,i}, \sigma'_{i,j}$ for $j \neq i$}
in the same manner. 
\end{definition}

The main results of this paper are the following theorems.
The proofs will be given in the subsequent sections. 

\begin{theorem} \label{thm:anotheraction}
(1) The category $\gMod{R_i}$ is a right $\dotcatquantum{\mathfrak{p}_i}$-module as follows: 
\begin{itemize}
\item The category attached to $\lambda \in \mathsf{P}$ is $\gMod{R_i(s_i\lambda)}$. 
\item The actions of the generating 1-morphisms are given by 
\[
X F_i = E_i X, X E_i = F_i X, X F_j = X \circ M_j \ (j \neq i). 
\]
\item The actions of the generating 2-morphisms are given as follows $(X \in \gMod{R_i(\beta)})$: 
\begin{align*}
&\left[ X \xy 0;/r.15pc/: 
(0,0)*{\sdotd{i}};
\endxy \colon q^{(\alpha_i,\alpha_i)} X F_i \to X F_i \right] = \left[ \xy 0;/r.15pc/: 
(0,0)*{\sdotu{i}};
\endxy X \colon q^{(\alpha_i,\alpha_i)} E_iX \to E_iX \right] \\
&= \left[ q^{(\alpha_i,\alpha_i)}e(i,*)X \xrightarrow{\text{the left multiplication by $x_1 \in R(\alpha_i)$}} e(i,*)X \right],  \\
&\left[ X \xy 0;/r.15pc/: 
(0,0)*{\sdotd{j}};
\endxy \colon q^{(\alpha_j,\alpha_j)} X F_j \to X F_j \ (j \neq i) \right] \\
&= \left[ \id_X \otimes y_j \colon q^{(\alpha_j,\alpha_j)} X \circ M_j \to X \circ M_j \right], \\
&\left[ X \xy 0;/r.15pc/: 
(0,0)*{\dcross{i}{i}};
\endxy \colon q^{-(\alpha_i,\alpha_i)} X F_i F_i \to X F_i F_i \right] = \left[ \xy 0;/r.15pc/: 
(0,0)*{\ucross{i}{i}};
\endxy X \colon q^{-(\alpha_i,\alpha_i)} E_iE_i X \to E_i E_i X \right] \\
&= \left[ q^{-(\alpha_i,\alpha_i)} e(2\alpha_i,*)X \xrightarrow{\text{the left multiplication by $\tau_1 \in R(2\alpha_i)$}} e(2\alpha_i,*)X \right], \\
&\left[ X \xy 0;/r.15pc/: 
(0,0)*{\dcross{i}{j}};
\endxy \colon q^{-(\alpha_i,\alpha_j)} X F_i F_j \to X F_j F_i \right] \ (j \neq i) \\
&= \left[ \sigma_{i,j} \colon q^{-(\alpha_i,\alpha_j)} (E_iX) \circ M_j \to E_i (X \circ M_j) \right], \\
&\left[ X \xy 0;/r.15pc/: 
(0,0)*{\dcross{j}{i}};
\endxy \colon q^{-(\alpha_i,\alpha_j)} X F_j F_i \to X F_i F_j \right] \ (j \neq i) \\
&= \left[ \sigma_{j,i} \colon q^{-(\alpha_i,\alpha_j)} E_i (X \circ M_j) \to (E_iX) \circ M_j \right], \\
&\left[ X \xy 0;/r.15pc/: 
(0,0)*{\dcross{j}{k}};
\endxy \colon q^{-(\alpha_j,\alpha_k)} X F_j F_k \to X F_k F_j \right] \ (j,k \neq i) \\
&= \left[ \id_X \otimes \sigma_{j,k} \colon q^{-(\alpha_j,\alpha_k)} X \circ M_j \circ M_k \to X \circ M_k \circ M_j \right], \\
&\left[ X \xy 0;/r.15pc/: 
(0,0)*{\lcap{i}};
\endxy \colon q_i^{1 + \langle h_i, s_i\beta \rangle} X F_i E_i \to X \right] = \left[c_{i,\beta}^{-1}\xy 0;/r.15pc/: 
(0,0)*{\lcap{i}};
\endxy X \colon q_i^{1 + \langle h_i, -\beta \rangle} F_i E_i X \to X \right] \\ 
&= \text{$c_{i,\beta}^{-1}$-multiple of the canonical unit for the adjoint pair $(F_i, q_i^{1 + \langle h_j, -\beta \rangle} E_i)$}, \\
&\left[ X \xy 0;/r.15pc/: 
(0,0)*{\lcup{i}};
\endxy \colon q_i^{1 - \langle h_i, s_i\beta \rangle} X \to X E_i F_i\right] = \left[ c_{i,\beta}\xy 0;/r.15pc/: 
(0,0)*{\lcup{i}};
\endxy X \colon q_i^{1 - \langle h_i, -\beta \rangle} X \to E_i F_i X \right] \\
&= \text{$c_{i,\beta}$-multiple of the canonical counit for the adjoint pair $(F_i, q_i^{1 + \langle h_i, -\beta -\alpha_i \rangle} E_i)$}. 
\end{align*}
\end{itemize}
The action restricts to the additive subcategory $\gproj{R_i}$. 
Moreover, we have 
\begin{align*}
%&X  \xy 0;/r.17pc/:
%(0,0)*{\lcross{i}{i}}; 
%\endxy = \xy 0;/r.17pc/:
%(0,0)*{\lcross{i}{i}}; 
%\endxy X, \\
&\left[ X \xy 0;/r.15pc/: 
(0,0)*{\rcap{i}};
\endxy \colon q_i^{1 - \langle h_i, s_i\beta \rangle} X E_i F_i \to X \right] = \left[ c_{i,-\beta}^{-1} \xy 0;/r.15pc/: 
(0,0)*{\rcap{i}};
\endxy X \colon q_i^{1 - \langle h_i, -\beta \rangle} E_i F_i X \to X \right], \\ 
&\left[ X \xy 0;/r.15pc/: 
(0,0)*{\rcup{i}};
\endxy \colon q_i^{1 + \langle h_i, s_i\beta \rangle} X \to X F_i E_i \right] = \left[ c_{i,-\beta} \xy 0;/r.15pc/: 
(0,0)*{\rcup{i}};
\endxy X \colon q_i^{1 + \langle h_i, -\beta \rangle} X \to F_i E_i X \right]. \\
\end{align*}

(2) The category $\gMod{{}_iR}$ is a left $\dotcatquantum{\mathfrak{p}_i}$-module as follows: 
\begin{itemize}
\item The category attached to $\lambda \in \mathsf{P}$ is $\gMod{R_j(-s_i\lambda)}$. 
\item The actions of the generating 1-morphisms are given by 
\[
F_i X = X E_i, E_i X = X F_i , F_j X = M_j \circ X \ (j \neq i). 
\]
\item The actions of the generating 2-morphisms are given as follows $(X \in \gMod{{}_iR(\beta)}, n = \height \beta)$: 
\begin{align*}
&\left[ \xy 0;/r.15pc/: 
(0,0)*{\sdotd{i}};
\endxy X \colon q^{(\alpha_i,\alpha_i)} F_i X \to F_i X \right] = \left[ X \xy 0;/r.15pc/: 
(0,0)*{\sdotu{i}};
\endxy \colon q^{(\alpha_i,\alpha_i)} X E_i \to XE_i\right] \\
&= \left[q^{(\alpha_i,\alpha_i)} e(*,i)X \xrightarrow{\text{the left multiplication by $x_n$}} e(*,i)X \right], \\
&\left[ \xy 0;/r.15pc/: 
(0,0)*{\sdotd{j}};
\endxy X \colon q^{(\alpha_j,\alpha_j)} F_j X \to F_j X \right] \ (j \neq i) \\
&= \left[y'_j \otimes \id_X \colon q^{(\alpha_j,\alpha_j)} M'_j \circ X \to M'_j \circ X \right], \\
&\left[ \xy 0;/r.15pc/: 
(0,0)*{\dcross{i}{i}};
\endxy X \colon q^{-(\alpha_i,\alpha_i)} F_i F_i X \to F_i F_i X \right] = \left[ X \xy 0;/r.15pc/: 
(0,0)*{\ucross{i}{i}};
\endxy  \colon q^{-(\alpha_i,\alpha_i)} X E_iE_i \to X E_i E_i \right] \\
&= \left[ q^{-(\alpha_i,\alpha_i)} e(*,2\alpha_i)X \xrightarrow{\text{the left multiplication by $\tau_{n-1}$}} e(*,2\alpha_i)X \right], \\
&\left[ \xy 0;/r.15pc/: 
(0,0)*{\dcross{i}{j}};
\endxy X \colon q^{-(\alpha_i,\alpha_j)} F_i F_j X \to F_j F_i X \right] \ (j \neq i) \\
&= \left[ \sigma'_{i,j} \colon q^{-(\alpha_i,\alpha_j)} (M'_j \circ X) E_i \to M'_j \circ (XE_i) \right], \\
&\left[ \xy 0;/r.15pc/: 
(0,0)*{\dcross{j}{i}};
\endxy X \colon q^{-(\alpha_i,\alpha_j)} F_j F_i X \to F_i F_j X \right] \ (j \neq i) \\
&= \left[ \sigma'_{j,i} \colon q^{-(\alpha_i,\alpha_j)} M'_j \circ (X E_i) \to (M'_j \circ X) E_i \right], \\
&\left[ \xy 0;/r.15pc/: 
(0,0)*{\dcross{j}{k}};
\endxy X \colon q^{-(\alpha_j,\alpha_k)} F_j F_k X \to F_k F_j X \right] \ (j,k \neq i) \\
&= \left[ \sigma'_{j,k} \otimes \id_X \colon q^{-(\alpha_j,\alpha_k)} M'_j \circ M'_k \circ X \to M'_k \circ M'_j \circ X \right], \\
&\left[ \xy 0;/r.15pc/: 
(0,0)*{\rcap{i}};
\endxy X \colon q_i^{1 - \langle h_i, -s_i\beta \rangle} E_i F_i X \to X \right] = \left[ c_{i,-\beta} X \xy 0;/r.15pc/: 
(0,0)*{\rcap{i}};
\endxy \colon q_i^{1 - \langle h_i, \beta \rangle} X E_i F_i \to X \right] \\ 
&= \text{$c_{i,-\beta}$-multiple of the canonical unit for the adjoint pair $(F_i^*, q_i^{1-\langle h_i,\beta \rangle}E_i^*)$}, \\
&\left[ \xy 0;/r.15pc/: 
(0,0)*{\rcup{i}};
\endxy X \colon q_i^{1 + \langle h_i, -s_i\beta \rangle} X \to F_i E_i X \right] = \left[ c_{i,-\beta}^{-1} X \xy 0;/r.15pc/: 
(0,0)*{\rcup{i}};
\endxy \colon q_i^{1 + \langle h_i, \beta \rangle} X \to X F_i E_i \right] \\
&= \text{$c_{i,-\beta}^{-1}$-multiple of the canonical counit for the adjoint pair $(F_i^*, q_i^{1-\langle h_i,\beta+\alpha_i \rangle}E_i^*)$}. \\
\end{align*}
\end{itemize}
The action restricts to the additive subcategory $\gproj{{}_i R}$.
Moreover, we have 
\begin{align*}
&\left[ \xy 0;/r.15pc/: 
(0,0)*{\lcap{i}};
\endxy X \colon q_i^{1 + \langle h_i, -s_i\beta \rangle} F_i E_i X \to X  \right]= \left[ c_{i,\beta}X \xy 0;/r.15pc/: 
(0,0)*{\lcap{i}};
\endxy \colon q_i^{1 + \langle h_i, \beta \rangle} X F_i E_i \to X \right], \\ 
&\left[\xy 0;/r.15pc/: 
(0,0)*{\lcup{i}};
\endxy X \colon q_i^{1 - \langle h_i, -s_i\beta \rangle} X \to E_i F_i X\right] =\left[ c_{i,\beta}^{-1} X \xy 0;/r.15pc/: 
(0,0)*{\lcup{i}};
\endxy \colon q_i^{1 - \langle h_i, \beta \rangle} X \to X E_i F_i \right]. \\
\end{align*}
\end{theorem}

Note that $R_i(0)E_i=0, R_i(0)F_i = R_i(\alpha_i) = 0$, and the graded algebra 
\[
\END_{R_i(0)}(R_i(0)) \simeq \mathbf{k}
\]
is concentrated in degree zero. 
By Theorem \ref{thm:rightuniversality}, there exists a morphism of right $\dotcatquantum{\mathfrak{p}_i}$-modules
\[
\gproj{{}_iR} \to \gproj{R_i},\ \mathbf{1} = {}_iR(0) \mapsto \mathbf{1} = R_i(0). 
\]
It uniquely extends to a right exact functor of right $\dotcatquantum{\mathfrak{p}_i}$-modules
\[
S_i \colon \gMod{{}_iR} \to \gMod{R_i},  \index{$S_i$: reflection functor}
\] 
given by 
\[
S_i(X) = S_i({}_iR(\beta)) \otimes_{{}_iR(\beta)} X \ (X \in \gMod{{}_iR(\beta)}). 
\]
Similarly, Theorem \ref{thm:leftuniversality} shows that there exists a morphism of left $\dotcatquantum{\mathfrak{p}_i}$-modules
\[
\gproj{R_i} \to \gproj{{}_iR}, \ \mathbf{1} = R_i(0) \mapsto \mathbf{1} = {}_iR(0),  
\] 
which uniquely extends to a right exact functor of left $\dotcatquantum{\mathfrak{p}_i}$-modules
\[
S'_i \colon \gMod{R_i} \to \gMod{{}_iR}. \index{$S'_i$: reflection functor}
\] 
We call them the reflection functors. 

\begin{theorem} \label{thm:reflectionfunctor}
$S_i$ and $S'_i$ are mutually quasi-inverse, and give a graded monoidal equivalence $\gMod{R_i} \simeq \gMod{{}_iR}$. 
\end{theorem}

\begin{remark} \label{rem:LRchange2}
$S_i$ and $S'_i$ are related in the following manner. 
The automorphism $\sigma_*$ of $\gMod{R}$ induces an isomorphism of monoidal categories
\[
\sigma_* \colon \gMod{{}_iR} \to \gMod{R_i}. 
\] \index{$\sigma_*$: twist by $\sigma$}
On the other hand, we have an isomorphism of 2-categories
\[
\sigma \colon \catquantum{\mathfrak{p}_i} \to \catquantum{\mathfrak{p}_i}^{\mathrm{op}}. 
\]
Through these isomorphisms, the left action of $\catquantum{\mathfrak{p}_i}$ on $\gMod{{}_iR}$ (Theorem \ref{thm:anotheraction} (1)) coincides with the left action of $\catquantum{\mathfrak{p}_i}^{\mathrm{op}}$ on $\gMod{R_i}$ (Theorem \ref{thm:anotheraction} (1)). 
Combined with Remark \ref{rem:LRchange}, the following diagram of equivalences commutes up to natural isomorphism: 
\begin{equation*}
  \centering
\begin{tikzcd}
\gMod{R_i} \arrow[r,"S_i'"]\arrow[d,"\sigma_*"] & \gMod{{}_iR} \arrow[d, "\sigma_*"] \\
\gMod{{}_iR} \arrow[r,"S_i"] & \gMod{R_i} \\
\end{tikzcd}
\end{equation*}
This is a categorical lift of the formula \cite[37.2.4]{MR2759715}. 
\end{remark}

\subsection[Proof of Theorem]{Proof of Theorem \ref{thm:anotheraction}}

In this section, we prove Theorem \ref{thm:anotheraction}. 
We only address (1), as (2) is completely parallel to it: 
in fact, (2) is reduced to (1) by Remark \ref{rem:LRchange2}. 
In order to prove that the $\dotcatquantum{\mathfrak{p}_i}$-action is well-defined, it suffices to verify the relations listed in Theorem \ref{thm:Rouquierver} (2).
We suppress degree shifts. 
As in the previous section, $\gMod{R_i}$ is a left $\dotcatquantum{\mathfrak{p}_i}$-module by Theorem \ref{thm:cyclotomic2repleft}. 
Let $\beta \in \mathsf{Q}_+, X \in \gMod{R_i(\beta)}$. 

\subsubsection{Left adjunction}

It is immediate from the definition. 

\subsubsection{Quadratic KLR}

We need to prove
\[
X \xy 0;/r.17pc/:
(0,4)*{\dcross{}{}};
(0,-4)*{\dcross{j}{k}};
\endxy 
= X Q_{j,k} \left( 
\xy 0;/r.17pc/: 
(-4,8); (-4,-8) **\dir{-} ?(1)*\dir{>}+(0,-2)*{\scriptstyle j}; 
(4,8); (4,-8) **\dir{-} ?(1)*\dir{>}+(0,-2)*{\scriptstyle k};
(-4,0)*{\bullet};
\endxy, \quad 
\xy 0;/r.17pc/:
(-4,8); (-4,-8) **\dir{-} ?(1)*\dir{>}+(0,-2)*{\scriptstyle j}; 
(4,8); (4,-8) **\dir{-} ?(1)*\dir{>}+(0,-2)*{\scriptstyle k};
(4,0)*{\bullet}; 
\endxy 
\right)
\]
for any $j, k \in I$. 

If $j = k = i$, we need to prove
\[
\xy 0;/r.17pc/:
(0,4)*{\ucross{}{}};
(0,-4)*{\ucross{i}{i}};
\endxy X = 0. 
\]
It follows from relations (1), (3) and (4) in Definition \ref{def:catquantum} and $Q_{i,i}(u,v) = 0$. 

If $j \neq i, k = i$, it follows from $\sigma_{i,j} \sigma_{j,i} = Q_{i,j}\left(\xy 0:/r.10pc/: (0,0)*{\sdotu{i}}; \endxy, y_j\right)$ in $\End(E_i(X \circ M_j))$, see Definition \ref{def:anothertauij} (1). 

If $j = i, k \neq i$, then we have 
\[
\sigma_{i,k} \sigma_{k,i} \sigma_{i,k} = Q_{i,k}\left(\xy 0:/r.10pc/: (0,0)*{\sdotu{i}}; \endxy, y_k \right) \sigma_{i,k} = \sigma_{i,k} Q_{i,k}\left(\xy 0:/r.10pc/: (0,0)*{\sdotu{i}}; \endxy, y_k \right) \colon (E_i X) \circ M_j \to E_i (X \circ M_j). 
\] 
Since $\sigma_{i,k}$ is injective, we deduce $\sigma_{k,i} \sigma_{i,k} = Q_{i,k}(\xy 0:/r.10pc/: (0,0)*{\sdotu{i}}; \endxy, y_k)$. 

Finally, assume $j,k \neq i$. 
The composition
\[
R(\alpha_j+\alpha_k)e(j,k) \xrightarrow{\times \tau_1} R(\alpha_j+\alpha_k)e(k,j) \xrightarrow{\times \tau_1} R(\alpha_j+\alpha_k)e(j,k)
\]
coincides with the right multiplication by $Q_{j,k}(x_1,x_2)$. 
Hence, the monoidal equivalence of Proposition \ref{prop:extremalequiv} shows that $\sigma_{k,j}\sigma_{j,k} = Q_{j,k}(y_j,y_k)$ as an endomorphism of $M_j \circ M_k$. 
It proves the relation. 

\subsubsection{Dot slides}

We need to prove
\begin{align*}
  X \xy 0;/r.15pc/: 
  (0,0)*{\dcross{j}{k}}; 
  (2,1.8)*{\bullet}; 
  \endxy 
  - 
  X \xy 0;/r.15pc/: 
  (0,0)*{\dcross{j}{k}}; 
  (-2,-1.3)*{\bullet};
  \endxy
  &= 
  X \xy 0;/r.15pc/: 
  (0,0)*{\dcross{j}{k}}; 
  (2,-1.3)*{\bullet}; 
  \endxy
  -
  X \xy 0;/r.15pc/: 
  (0,0)*{\dcross{j}{k}}; 
  (-2,1.8)*{\bullet}; 
  \endxy = \begin{cases}
  X \xy 0;/r.15pc/: 
  (-4,0)*{\slined{j}};
  (4,0)*{\slined{k}}; 
  \endxy & \text {if $j=k$}, \\
  0 & \text {if $j \neq k$},
  \end{cases}
  \end{align*}
for any $j, k \in I$. 

If $j = k = i$, it is equivalent to proving 
\[
  \xy 0;/r.15pc/: 
  (0,0)*{\ucross{i}{i}}; 
  (-2,1.8)*{\bullet}; 
  \endxy X
  - 
  \xy 0;/r.15pc/: 
  (0,0)*{\ucross{i}{i}}; 
  (2,-1.3)*{\bullet};
  \endxy X
  = 
  \xy 0;/r.15pc/: 
  (0,0)*{\ucross{i}{i}}; 
  (-2,-1.3)*{\bullet}; 
  \endxy X 
  -
  \xy 0;/r.15pc/: 
  (0,0)*{\ucross{i}{i}}; 
  (2,1.8)*{\bullet}; 
  \endxy X 
  = 
  \xy 0;/r.15pc/: 
  (-4,0)*{\slineu{i}};
  (4,0)*{\slineu{i}}; 
  \endxy X. 
\]
It follows from relations (1), (2), (3) and (5) in Definition \ref{def:catquantum}. 
Alternatively, Theorem \ref{thm:cyclotomic2repleft} shows that the relation is equivalent to the equation
\[
x_2\tau_1e(i,i,*) - \tau_1x_1 e(i,i,*) = \tau_1x_2 e(i,i,*) - x_1 \tau_1e(i,i,*) = e(i,i,*)
\]
as an endomorphism of $X$. 
(One must be careful about where each $E_i$ is applied: see Remark \ref{rem:restrictionorder}.) 

If $j = i, k\neq i$ (resp. $j \neq i, k = i$), it follows from the fact that $\sigma_{i,k}$ (resp. $\sigma_{j,i}$) commutes with the action of $R(\alpha_i)$ on $E_i$ and endomorphisms of $M_j$ (resp. $M_k$). 

If $j,k \neq i$, it follows from the relation in $R(\alpha_j+\alpha_k)$ using the monoidal equivalence of Proposition \ref{prop:extremalequiv}. 

\subsubsection{Cubic KLR}

We need to prove 
\[
X \xy 0;/r.17pc/: 
(4,-8)*{\dcross{k}{l}}; 
(-4,0)*{\dcross{}{}};
(4,8)*{\dcross{}{}};
(-8,-8)*{\slined{j}};
(8,0)*{\slined{}};
(-8,8)*{\slined{}};
\endxy
- 
X \xy 0;/r.17pc/: 
(-4,-8)*{\dcross{j}{k}}; 
(4,0)*{\dcross{}{}};
(-4,8)*{\dcross{}{}};
(8,-8)*{\slined{l}};
(-8,0)*{\slined{}};
(8,8)*{\slined{}};
\endxy
= 
X \overline{Q}_{j,k,l}\left( 
\xy 0;/r.17pc/: 
(-4,0)*{\slined{j}}; 
(0,0)*{\slined{k}};
(4,0)*{\slined{l}}; 
(-4,0)*{\bullet}; 
\endxy, 
\xy 0;/r.17pc/: 
(-4,0)*{\slined{j}}; 
(0,0)*{\slined{k}};
(4,0)*{\slined{l}}; 
(0,0)*{\bullet}; 
\endxy, 
\xy 0;/r.17pc/: 
(-4,0)*{\slined{j}}; 
(0,0)*{\slined{k}};
(4,0)*{\slined{l}}; 
(4,0)*{\bullet}; 
\endxy
\right) 
\]
for any $j,k,l \in I$. 
The calculation splits into the following ten cases: 

Case 1. $j=k=l = i$. 
It follows from 
\[
\xy 0;/r.17pc/: 
(-4,-8)*{\ucross{i}{i}}; 
(4,0)*{\ucross{}{}};
(-4,8)*{\ucross{}{}};
(8,-8)*{\slineu{i}};
(-8,0)*{\slineu{}};
(8,8)*{\slineu{}};
\endxy X
-
\xy 0;/r.17pc/: 
(4,-8)*{\ucross{i}{i}}; 
(-4,0)*{\ucross{}{}};
(4,8)*{\ucross{}{}};
(-8,-8)*{\slineu{i}};
(8,0)*{\slineu{}};
(-8,8)*{\slineu{}};
\endxy X =0.  
\]

Case 2. $j = k = i, l \neq i$. 
Since the right hand side is zero, it suffices to prove that the following two homomorphisms coincide: 
\begin{align*}
&E_i E_i X \circ M_l \xrightarrow{\sigma_{i,l}} E_i (E_i X \circ M_l) 
\xrightarrow{\sigma_{i,l}} E_i E_i (X \circ M_l)  
\xrightarrow{\xy 0;/r.12pc/: (0,0)*{\ucross{i}{i}} \endxy} E_iE_i (X \circ M_l), \\
&E_iE_i X \circ M_l \xrightarrow{\xy 0;/r.12pc/: (0,0)*{\ucross{i}{i}} \endxy} E_iE_i X \circ M_l 
\xrightarrow{\sigma_{i,l}} E_i(E_i X \circ M_l) 
\xrightarrow{\sigma_{i,l}} E_iE_i (X \circ M_l).   
\end{align*}
It follows from Lemma \ref{lem:naturality} (3). 

Case 3. $j \neq i, k = l = i$. 
Since the right hand side is zero, we have to prove that the following two homomorphisms coincide. 
\begin{align*}
&E_i E_i (X \circ M_j) \xrightarrow{\xy 0;/r.12pc/: (0,0)*{\ucross{i}{i}} \endxy} E_i E_i (X \circ M_j) 
\xrightarrow{\sigma_{j,i}} E_i (E_i X \circ M_j) 
\xrightarrow{\sigma_{j,i}} E_iE_i X \circ M_j, \\
&E_iE_i (X \circ M_j) \xrightarrow{\sigma_{j,i}} E_i(E_i X \circ M_j) 
\xrightarrow{\sigma_{j,i}} E_iE_i X \circ M_j 
\xrightarrow{\xy 0;/r.12pc/: (0,0)*{\ucross{i}{i}} \endxy} E_iE_i X \circ M_j. 
\end{align*}
By embedding $E_i (E_i X \circ M_j)$ and $E_iE_i X \circ M_j$ into $E_i E_i(X \circ M_j)$ using $\sigma_{i,j}$, 
it suffices to prove that the following two homomorphisms coincide:
\begin{align*}
&E_i E_i (X \circ M_j) \xrightarrow{\xy 0;/r.12pc/: (0,0)*{\ucross{i}{i}} \endxy} E_iE_i (X\circ M_j) \\
&\xrightarrow{Q_{i,j} \left(\xy 0;/r.12pc/: (-3,0)*{\slineu{i}}; (3,0)*{\sdotu{i}}; \endxy, y_j\right)} E_iE_i(X \circ M_j) \xrightarrow{Q_{i,j} \left(\xy 0;/r.12pc/: (-3,0)*{\sdotu{i}}; (3,0)*{\slineu{i}}; \endxy, y_j\right)} E_iE_i(X \circ M_j), \\
&E_iE_i (X\circ M_j) \xrightarrow{Q_{i,j} \left(\xy 0;/r.12pc/: (-3,0)*{\slineu{i}}; (3,0)*{\sdotu{i}}; \endxy, y_j\right)} E_iE_i(X \circ M_j) \\
&\xrightarrow{Q_{i,j} \left(\xy 0;/r.12pc/: (-3,0)*{\sdotu{i}}; (3,0)*{\slineu{i}}; \endxy, y_j\right)} E_iE_i(X \circ M_j) \xrightarrow{\xy 0;/r.12pc/: (0,0)*{\ucross{i}{i}} \endxy} E_i E_i (X \circ M_j). 
\end{align*}
Here, we used Definition \ref{def:anothertauij} (1), 
the naturality of the embedding (Lemma \ref{lem:adjointSES}), 
and Lemma \ref{lem:naturality} (3). 
Note that the composition of $Q_{i,j} \left(\xy 0;/r.12pc/: (-3,0)*{\slineu{i}}; (3,0)*{\sdotu{i}}; \endxy, y_j\right)$ and $Q_{i,j} \left(\xy 0;/r.12pc/: (-3,0)*{\sdotu{i}}; (3,0)*{\slineu{i}}; \endxy, y_j\right)$ is symmetric for $\xy 0;/r.12pc/: (-3,0)*{\slineu{i}}; (3,0)*{\sdotu{i}}; \endxy$ and $\xy 0;/r.12pc/: (-3,0)*{\sdotu{i}}; (3,0)*{\slineu{i}}; \endxy$, 
since $Q_{i,j}(u,v)Q_{i,j}(u',v)$ is symmetric for $u$ and $u'$. 
Hence, it commutes with $\xy 0;/r.12pc/: (0,0)*{\ucross{i}{i}} \endxy$, and the two homomorphisms coincide.

Case 4. $j = i, k \neq i, l = i$. 
Note that 
\begin{align*}
  &X \xy 0;/r.15pc/: 
  (4,-8)*{\dcross{k}{i}}; 
  (-4,0)*{\dcross{}{}};
  (4,8)*{\dcross{}{}};
  (-8,-8)*{\slined{i}};
  (8,0)*{\slined{}};
  (-8,8)*{\slined{}}; 
  \endxy \\
  &= \bigg[ E_i (E_i X \circ M_k) \xrightarrow{\sigma_{k,i}} E_iE_i X \circ M_k \\
  &\quad \xrightarrow{\xy 0;/r.12pc/: (0,0)*{\ucross{i}{i}} \endxy} E_i E_i X \circ M_k \xrightarrow{\sigma_{i,k}} E_i(E_i X \circ M_k) \bigg], \\
  &X \xy 0;/r.15pc/: 
  (-4,-8)*{\dcross{i}{k}}; 
  (4,0)*{\dcross{}{}};
  (-4,8)*{\dcross{}{}};
  (8,-8)*{\slined{i}};
  (-8,0)*{\slined{}};
  (8,8)*{\slined{}};
  \endxy \\
  &= \bigg[ E_i(E_i X \circ M_k) \xrightarrow{\sigma_{i,k}} E_iE_i (X \circ M_k) \\
  &\xrightarrow{\xy 0;/r.12pc/: (0,0)*{\ucross{i}{i}} \endxy} E_iE_i (X\circ M_k) \xrightarrow{\sigma_{k,i}} E_i(E_i X \circ M_k)\bigg].
\end{align*}
By embedding all the modules into $E_i E_i (X \circ M_k)$, it suffices to prove 
\begin{align*}
&\bigg[ 
E_i E_i (X \circ M_k) \xrightarrow{Q_{i,k} \left(\xy 0;/r.12pc/: (-3,0)*{\sdotu{i}}; (3,0)*{\slineu{i}}; \endxy, y_k\right)} E_i E_i (X \circ M_k) \\
&\quad \xrightarrow{\xy 0;/r.12pc/: (0,0)*{\ucross{i}{i}}; \endxy} E_iE_i (X \circ M_k) \xrightarrow{\id} E_iE_i (X \circ M_k)\bigg] \\
&- \bigg[E_i E_i (X \circ M_k) \xrightarrow{\id} E_i E_i (X \circ M_k) \\
&\quad \xrightarrow{\xy 0;/r.12pc/: (0,0)*{\ucross{i}{i}}; \endxy} E_iE_i (X \circ M_k) \xrightarrow{Q_{i,k} \left(\xy 0;/r.12pc/: (-3,0)*{\slineu{i}}; (3,0)*{\sdotu{i}}; \endxy, y_k\right)} E_iE_i (X \circ M_k)\bigg] \\
&= \bigg[E_i E_i (X \circ M_k) \xrightarrow{\overline{Q}_{i,k,i}\left( \xy 0;/r.12pc/: (-3,0)*{\slineu{i}}; (3,0)*{\sdotu{i}}; \endxy, y_k, \xy 0;/r.12pc/: (-3,0)*{\sdotu{i}}; (3,0)*{\slineu{i}}; \endxy \right)} E_iE_i(X \circ M_k)\bigg]. 
\end{align*}
Here, we used Definition \ref{def:anothertauij},
the naturality of the embedding (Lemma \ref{lem:adjointSES}), 
Lemma \ref{lem:EFvsmultiplication}, 
and Lemma \ref{lem:naturality} (3). 
The equality is a consequence of the following equation in the nil-Hecke algebra $R(2\alpha_i)$: for a polynomial $f(u) \in \mathbf{k}[u]$, we have 
\[
\tau_1 f(x_2) - f(x_1) \tau_1 = \partial_1(f(x_2)) = \frac{f(x_1) - f(x_2)}{x_1 - x_2}. 
\]
See also Remark \ref{rem:restrictionorder}. 

Case 5. $j = k \neq i, l = i$. 
Since the right hand side is zero, we need to prove that the following two homomorphisms coincide: 
\begin{align*}
&E_i (X \circ M_k \circ M_k) \xrightarrow{\sigma_{k,i}} E_i (X \circ M_k) \circ M_k \\
&\xrightarrow{\sigma_{k,i}} E_i X \circ M_k \circ M_k  
\xrightarrow{\sigma_{k,k}} E_iX \circ M_k \circ M_k, \\
&E_i (X \circ M_k \circ M_k) \xrightarrow{\sigma_{k,k}} E_i(X \circ M_k \circ M_k) \\
&\xrightarrow{\sigma_{k,i}} E_i(X \circ M_k) \circ M_k 
\xrightarrow{\sigma_{k,i}} E_iX \circ M_k \circ M_k.  
\end{align*}
By embedding all the modules into $E_i (X \circ M_k \circ M_k)$ as before, it suffices to prove that the following two homomorphisms coincide: 
\begin{align*}
&E_i (X \circ M_k \circ M_k) \xrightarrow{Q_{i,k} \left(\xy 0;/r.12pc/: (0,0)*{\sdotu{i}}; \endxy, \id_{M_k} \otimes y_k\right)} E_i (X\circ M_k \circ M_k)  \\
&\xrightarrow{Q_{i,k} \left(\xy 0;/r.12pc/: (0,0)*{\sdotu{i}}; \endxy, y_k \otimes \id_{M_k}\right)} E_i(X \circ M_k \circ M_k) 
\xrightarrow{\sigma_{k,k}} E_i(X \circ M_k \circ M_k), \\
&E_i(X\circ M_k \circ M_k) \xrightarrow{\sigma_{k,k}} E_i(X \circ M_k \circ M_k) \\
&\xrightarrow{Q_{i,k} \left(\xy 0;/r.12pc/: (0,0)*{\sdotu{i}}; \endxy, \id_{M_k} \otimes y_k\right)} E_i(X \circ M_k \circ M_k) 
\xrightarrow{Q_{i,k} \left(\xy 0;/r.12pc/: (0,0)*{\sdotu{i}}; \endxy, y_k \otimes \id_{M_k}\right)} E_i(X \circ M_k \circ M_k).   
\end{align*}
Note that the composition of $Q_{i,k} \left(\xy 0;/r.12pc/: (0,0)*{\sdotu{i}}; \endxy, \id_{M_k} \otimes y_k\right)$ and $Q_{i,k} \left(\xy 0;/r.12pc/: (0,0)*{\sdotu{i}}; \endxy, y_k \otimes \id_{M_k}\right)$ is symmetric in $\id_{M_k} \otimes y_k$ and $y_k \otimes \id_{M_k}$. 
We need to show that it commutes with $\sigma_{k,k}$. 
By Proposition \ref{prop:extremalequiv}, it follows from the fact that, in $R(2\alpha_k)$, the element $\tau_1$ commutes with symmetric polynomials of $x_1$ and $x_2$. 

Case 6. $j = l \neq i, k = i$. 
Note that 
\begin{align*}
  &X \xy 0;/r.15pc/: 
  (4,-8)*{\dcross{i}{j}}; 
  (-4,0)*{\dcross{}{}};
  (4,8)*{\dcross{}{}};
  (-8,-8)*{\slined{j}};
  (8,0)*{\slined{}};
  (-8,8)*{\slined{}}; 
  \endxy \\
  &= \bigg[E_i (X \circ M_j) \circ M_j \xrightarrow{\sigma_{i,j}} E_i(X \circ M_j \circ M_j) \\
  &\quad \xrightarrow{\sigma_{j,j}} E_i (X \circ M_j \circ M_j) 
  \xrightarrow{\sigma_{j,i}} E_i(X \circ M_j) \circ M_j\bigg], \\
  &X \xy 0;/r.15pc/: 
  (-4,-8)*{\dcross{j}{i}}; 
  (4,0)*{\dcross{}{}};
  (-4,8)*{\dcross{}{}};
  (8,-8)*{\slined{j}};
  (-8,0)*{\slined{}};
  (8,8)*{\slined{}};
  \endxy \\
  &= \bigg[E_i (X \circ M_j) \circ M_j \xrightarrow{\sigma_{j,i}} E_iX \circ M_j \circ M_j \\
  & \quad \xrightarrow{\sigma_{j,j}} E_i X \circ M_j \circ M_j
  \xrightarrow{\sigma_{i,j}} E_i(X \circ M_j) \circ M_j \bigg].
\end{align*}
By embedding all the modules into $E_i(X \circ M_j \circ M_j)$ as before, it suffices to prove 
\begin{align*}
&\bigg[ E_i (X \circ M_j \circ M_j) \xrightarrow{\id} E_i (X \circ M_j \circ M_j) \\
&\quad \xrightarrow{\sigma_{j,j}} E_i (X \circ M_j \circ M_j) 
\xrightarrow{Q_{i,j} \left(\xy 0;/r.12pc/: (0,0)*{\sdotu{i}}; \endxy, \id_{M_j} \otimes y_j\right)} E_i (X \circ M_j \circ M_j) \bigg] \\
&-\bigg[E_i (X \circ M_j \circ M_j) \xrightarrow{Q_{i,j} \left(\xy 0;/r.12pc/: (0,0)*{\sdotu{i}}; \endxy, y_j \otimes \id_{M_j}\right)} E_i (X \circ M_j \circ M_j) \\
&\quad \xrightarrow{\sigma_{j,j}} E_i (X \circ M_j \circ M_j) 
\xrightarrow{\id} E_i (X \circ M_j \circ M_j) \bigg] \\
&= 
\left[
E_i (X \circ M_j \circ M_j) \xrightarrow{\overline{Q}_{j,i,j}\left(y_j \otimes \id_{M_j},\xy 0;/r.12pc/: (0,0)*{\sdotu{i}}; \endxy, \id_{M_j}\otimes y_j \right)} E_i (X \circ M_j \circ M_j) 
\right]. 
\end{align*}
By the equivalence of Propopsition \ref{prop:extremalequiv}, 
it is reduced to the formula in $R(2\alpha_j)$ used in Case 3. 

Case 7. $j = i, k,l \neq i$. 
Since the right hand side is zero, it suffices to prove the following two homomorphisms coincide: 
\begin{align*}
&E_i X \circ M_k \circ M_l \xrightarrow{\sigma_{k,l}} E_i X \circ M_l \circ M_k \\
&\xrightarrow{\sigma_{i,l}} E_i (X \circ M_l) \circ M_k  
\xrightarrow{\sigma_{i,k}} E_i(X \circ M_l \circ M_k), \\
&E_i X \circ M_k \circ M_l \xrightarrow{\sigma_{i,k}} E_i(X \circ M_k) \circ M_l \\
&\xrightarrow{\sigma_{i,l}} E_i(X \circ M_k \circ M_l)
\xrightarrow{\sigma_{k,l}} E_i(X \circ M_l \circ M_k).  
\end{align*}
Note that the compositions of $\sigma_{i,l}$ and $\sigma_{i,k}$ coincides with the canonical injection $E_iX \circ M_l \circ M_k \to E_i(X \circ M_l \circ M_k)$ by Lemma \ref{lem:EFvsmultiplication}. 
Hence, the assertion follows from the naturality (Lemma \ref{lem:adjointSES}).  

Case 8. $j,k,l$ are distinct and $k = i$. 
Since the right hand side is zero, it suffices to prove the following two homomorphisms coincide: 
\begin{align*}
&E_i (X \circ M_j) \circ M_l \xrightarrow{\sigma_{i,l}} E_i (X \circ M_j \circ M_l) \\
&\xrightarrow{\sigma_{j,l}} E_i (X \circ M_l \circ M_j)  
\xrightarrow{\sigma_{j,i}} E_i(X \circ M_l) \circ M_j, \\
&E_i (X \circ M_j) \circ M_l \xrightarrow{\sigma_{j,i}} E_iX \circ M_j \circ M_l \\
&\xrightarrow{\sigma_{j,l}} E_iX \circ M_l \circ M_j
\xrightarrow{\sigma_{i,l}} E_i(X \circ M_l) \circ M_j.  
\end{align*}
By embedding as before $E_i(X \circ M_j) \circ M_l$ into $E_i(X \circ M_j \circ M_l)$, 
$E_i (X \circ M_l) \circ M_j$ into $E_i(X \circ M_l \circ M_j)$, 
$E_iX \circ M_j \circ M_l$ into $E_i(X \circ M_j \circ M_l)$, and 
$E_i X \circ M_l \circ M_j$ into $E_i (X \circ M_l \circ M_j)$, 
it suffices to prove that the following two homomorphisms coincide: 
\begin{align*}
&E_i (X \circ M_j \circ M_l) \xrightarrow{\id} E_i (X \circ M_j \circ M_l) \\
&\xrightarrow{\sigma_{j,l}} E_i (X \circ M_l \circ M_j)  
\xrightarrow{Q_{i,j}\left(\xy 0:/r.10pc/: (0,0)*{\sdotu{i}}; \endxy, y_j\right)} E_i(X \circ M_l \circ M_j), \\
&E_i (X \circ M_j \circ M_l) \xrightarrow{Q_{i,j}\left(\xy 0:/r.10pc/: (0,0)*{\sdotu{i}}; \endxy, y_j\right)} E_i(X \circ M_j \circ M_l) \\
&\xrightarrow{\sigma_{j,l}} E_i(X \circ M_l \circ M_j)
\xrightarrow{\id} E_i(X \circ M_l \circ M_j).  
\end{align*}
It follows from the relation $e(j,l)\tau_1x_2 = e(j,l)x_1\tau_1$ in $R(\alpha_j + \alpha_l)$
using the equivalence of Proposition \ref{prop:extremalequiv}. 

Case 9. $j,k,l$ are distinct and $l = i$. 
Since the right hand side is zero, we need to prove the following two homomorphisms coincide: 
\begin{align*}
&E_i(X \circ M_j \circ M_k) \xrightarrow{\sigma_{k,i}} E_i (X \circ M_j) \circ M_k \\
&\xrightarrow{\sigma_{j,i}} E_i X \circ M_j \circ M_k  
\xrightarrow{\sigma_{j,k}} E_iX \circ M_k \circ M_j, \\
&E_i (X \circ M_j \circ M_k) \xrightarrow{\sigma_{j,k}} E_i(X \circ M_k \circ M_j) \\
&\xrightarrow{\sigma_{j,i}} E_i(X \circ M_k) \circ M_j
\xrightarrow{\sigma_{k,i}} E_iX \circ M_k \circ M_j.  
\end{align*}
By embedding as before $E_i(X \circ M_j) \circ M_k$ into $E_i(X \circ M_j \circ M_k)$, 
$E_i X \circ M_ j \circ M_k$ into $E_i(X \circ M_j \circ M_k)$, 
$E_i X \circ M_k \circ M_j$ into $E_i(X \circ M_k \circ M_j)$, and 
$E_i (X \circ M_k)\circ M_j$ into $E_i(X \circ M_k \circ M_j)$, 
it suffices to prove that the following two homomorphisms coincide: 
\begin{align*}
&E_i(X \circ M_j \circ M_k) \xrightarrow{Q_{i,k}\left(\xy 0:/r.10pc/: (0,0)*{\sdotu{i}}; \endxy, y_k\right)} E_i (X \circ M_j \circ M_k) \\
&\xrightarrow{Q_{i,j}\left(\xy 0:/r.10pc/: (0,0)*{\sdotu{i}}; \endxy, y_j\right)} E_i (X \circ M_j \circ M_k)  
\xrightarrow{\sigma_{j,k}} E_i (X \circ M_k \circ M_j), \\
&E_i (X \circ M_j \circ M_k) \xrightarrow{\sigma_{j,k}} E_i(X \circ M_k \circ M_j) \\
&\xrightarrow{Q_{i,j}\left(\xy 0:/r.10pc/: (0,0)*{\sdotu{i}}; \endxy, y_j\right)} E_i(X \circ M_k \circ M_j)
\xrightarrow{Q_{i,k}\left(\xy 0:/r.10pc/: (0,0)*{\sdotu{i}}; \endxy, y_k\right)} E_i(X \circ M_k \circ M_j).  
\end{align*}
It follows from the relations $x_2 \tau_1 = \tau_1 x_1, x_1 \tau_1 = \tau_1 x_2$ in $R(\alpha_j + \alpha_k)$
using the equivalence of Proposition \ref{prop:extremalequiv}. 

Case 10. $j,k,l \neq i$. 
It follows from the relation $e(j,k,l)\tau_2 \tau_1 \tau_2 = e(j,k,l) \tau_1 \tau_2 \tau_1$ in $R(\alpha_j+\alpha_k+\alpha_l)$ using Proposition \ref{prop:extremalequiv}. 

We have now completed the verification of the cubic KLR relation. 

\subsubsection{Formal inverse} \label{subsub:formalinverse}

We need to verify three axioms. 
As we are only concerned with whether certain homomorphisms are isomorphisms, we disregard scalar multiples when they are irrelevant. 

(1) We prove 
\[
 X \xy 0;/r.15pc/:
  (0,0)*{\xybox{
  (0,0)*{\dcross{}{}};
  (8,7)*{\lcap{}};
  (-8,-7)*{\lcup{}};
  (-4,8)*{\slined{}};
  (4,-8)*{\slined{j}};
  (12,-12); (12,4) **\dir{-} ?(1)*\dir{>};
  (12,-14)*{\scriptstyle i};
  (-12,-4); (-12,12) **\dir{-} ?(1)*\dir{>}; 
  }}
  \endxy \colon X F_j E_i \to X E_i F_j
\]
is an isomorphism for any $j \neq i$. 
By definition, it is equal to some scalar multiple of 
\[
F_i (X \circ M_j) \xrightarrow{\xy 0;/r.12pc/: (0,0)*{\lcup{i}}; \endxy} F_i (E_iF_i X \circ M_j) \xrightarrow{\sigma_{i,j}} F_iE_i (F_i X \circ M_j) \xrightarrow{\xy 0;/r.12pc/: (0,0)*{\lcap{i}}; \endxy} F_i X \circ M_j, 
\]
which is denoted by $f$. 

\begin{lemma} \label{lem:canonicalsurj}
The homomorphism $f$ coincides with the canonical surjection of Lemma \ref{lem:adjointSES}. 
\end{lemma}

\begin{proof}
Let $x \in X, v \in M_j$. 
The element $e(i) \boxtimes (x \boxtimes v) \in F_i(X \circ M_j)$ is sent to 
\begin{align*}
&e(i) \boxtimes (x \boxtimes v) \\ 
&\mapsto e(i)\boxtimes (E_i(e(i) \boxtimes x) \boxtimes v)  \in F_i(E_iF_iX \circ M_j) \\ 
&\mapsto e(i) \boxtimes E_i((e(i) \boxtimes x) \boxtimes v) \in F_iE_i(F_iX \circ M_j) \\
&\mapsto (e(i) \boxtimes x) \boxtimes v \in F_iX \circ M_j.  
\end{align*}
Hence, the assertion follows. 
\end{proof}

Note that $F_i M_j = 0$, see the proof of Proposition \ref{prop:extremalequiv}. 
Hence, the short exact sequence of Lemma \ref{lem:adjointSES} (1) shows that the canonical surjection $f$ is an isomorphism. 

(2) Assume $\langle h_i, s_i\beta \rangle \leq 0$. We prove that
\begin{align*}
&X \begin{bmatrix}
    \xy 0;/r.15pc/:
    (0,0)*{\xybox{0;/r.12pc/:
    (0,0)*{\dcross{}{}};
    (8,7)*{\lcap{}};
    (-8,-7)*{\lcup{}};
    (-4,8)*{\slined{}};
    (4,-8)*{\slined{i}};
    (12,-12); (12,4) **\dir{-} ?(1)*\dir{>};
    (12,-14)*{\scriptstyle i};
    (-12,-4); (-12,12) **\dir{-} ?(1)*\dir{>}; 
    }}; 
    \endxy & \xy 0;/r.17pc/:
    (0,0)*{\lcap{i}}; 
    \endxy & \xy 0;/r.17pc/:
    (0,0)*{\lcap{i}}; 
    (-2,.5)*{\bullet};
    \endxy & \cdots & \xy 0;/r.17pc/:
    (0,0)*{\lcap{i}}; 
    (-2,.5)*{\bullet};
    (-13,3)*{\scriptstyle -\langle h_i, s_i \beta \rangle -1};
    \endxy
    \end{bmatrix}^\top \\
    &\colon X F_i E_i \to X E_i F_i \oplus X^{\oplus -\langle h_i, s_i\beta \rangle}
\end{align*}
is an isomorphism.
By definition, this homomorphism is 
\begin{align*}
&\begin{bmatrix}
      \xy 0;/r.15pc/:
      (0,0)*{\xybox{0;/r.12pc/:
        (0,0)*{\ucross{}{}};
        (8,-7)*{\lcup{}};
        (-8,7)*{\lcap{}};
        (-4,-8)*{\slineu{i}};
        (4,8)*{\slineu{}};
        (12,12); (12,-4) **\dir{-} ?(1)*\dir{>};
        (-12,-14)*{\scriptstyle i};
        (-12,4); (-12,-12) **\dir{-} ?(1)*\dir{>}; 
      }}; 
      \endxy & \xy 0;/r.17pc/:
      (0,0)*{\lcap{i}}; 
      \endxy & \xy 0;/r.17pc/:
      (0,0)*{\lcap{i}}; 
      (-2,.5)*{\bullet};
      \endxy & \cdots & \xy 0;/r.17pc/:
      (0,0)*{\lcap{i}}; 
      (-2,.5)*{\bullet};
      (-13,3)*{\scriptstyle -\langle h_i, -\beta \rangle -1};
      \endxy
      \end{bmatrix}^\top X \\
    &\colon F_i E_i X \to E_i F_i X \oplus X^{\oplus -\langle h_i, -\beta \rangle}, 
\end{align*}
up to scalar multiples. 
By Definition \ref{def:catquantum} (3) and Theorem \ref{thm:Rouquierver} (2) Formal inverse, it is an isomorphism. 

(3) The remaining axiom for the case $\langle h_i, s_i\beta \rangle \geq 0$ is proved in the same way as (2) above. 

We have now verified all the relations for $\catquantum{\mathfrak{p}_i}$. 
Hence, $\gMod{R_i}$ is a right $\dotcatquantum{\mathfrak{p}_i}$-module. 

\subsubsection{Restriction to $\gproj{R_i}$} 
We prove that the subcategory $\gproj{R_i} \subset \gMod{R_i}$ is stable under the right action of $E_i$ and $F_j \ (j \in I)$. 
The assertion for $E_i$ (resp. $F_i$) follows from Theorem \ref{thm:cyclotomic2repleft}, 
since its right action is given by the left action of $F_i$ (resp. $E_i$). 
It remains to show that, for any $P \in \gproj{R_i(\beta)}$ and $j \in I \setminus \{i \}$, 
the module $PF_j = P \circ M_j$ is projective over $R_i(\beta + s_i\alpha_j)$. 

Let $X \in \gMod{R_i(\beta + s_i\alpha_j)}$. 
Then, $\Res_{\beta,s_i\alpha_j} X$ is an $R(\beta) \otimes R_i(s_i\alpha_j)$-module. 
We claim that it is an $R_i(\beta) \otimes R_i(s_i\alpha_j)$-module. 
Since we have 
\[
(\Res_{\beta-\alpha_i,\alpha_i}\otimes \Id) \Res_{\beta,s_i\alpha_j} = (\Id \otimes \Res_{\alpha_i,s_i\alpha_j})\Res_{\beta-\alpha_i,s_i\alpha_j + \alpha_i}, 
\]
it is enough to show that $\Res_{\beta-\alpha_i,s_i\alpha_j+\alpha_i} X = 0$. 
Note that it is an $R(\beta -\alpha_i) \otimes R_i(s_i\alpha_j + \alpha_i)$-module. 
However, we have $R_i(s_i\alpha_j + \alpha_i) = 0$ as shown in the proof of Proposition \ref{prop:extremalequiv}. 
Hence, the claim is proved. 

Note that $M_j$ is a projective $R_i(s_i\alpha_j)$-module by Proposition \ref{prop:extremalequiv}. 
We compute
\begin{align*}
&\Ext_{R_i(\beta+s_i\alpha_j)}^1 (P \circ M_j, X) \\
&\simeq \Ext_{R(\beta + s_i\alpha_j)}^1 (P \circ M_j,X) \quad \text{since $\gMod{R_i} \subset \gMod{R}$ is closed under extension} \\
&\simeq \Ext_{R(\beta) \otimes R(s_i\alpha_j)}^1 (P \otimes M_j, \Res_{\beta,s_i\alpha_j} X) \quad \text{by the induction-restriction adjunction} \\
&\simeq \Ext_{R_i(\beta) \otimes R_i(s_i\alpha_j)}^1 (P \otimes M_j, \Res_{\beta,s_i\alpha_j} X) \\ 
&\quad \text{since $\gMod{R_i} \subset \gMod{R}$ is closed under extension} \\
&= 0 \quad \text{since $P, M_j \in \gproj{R_i}$}. 
\end{align*}
Therefore, $P \circ M_j \in \gproj{R_i}$. 

\subsubsection{The other 2-morphisms}
We prove the two equalities of 2-morphisms stated at the end of Theorem \ref{thm:anotheraction} (1). 
Note that it is only relevant to $E_i$ and $F_i$. 
Roughly speaking, we prove them by relating the left action of $\catquantum{\mathfrak{sl}_2}$ and the right one using an autiautomorphism of $\catquantum{\mathfrak{sl}_2}$. 

Fix $\beta \in \mathsf{Q}_+$.  
Let $C' = (c'_{i,n})_{n \in \mathbb{Z}}, C'' = (c''_{i,n})_{n \in \mathbb{Z}}$ be choices of bubble parameters for $\mathfrak{sl}_2$ satisfying
\[
c'_{i,\langle h_i, -\beta \rangle} = c_{i,-\beta},\ c''_{i,\langle h_i,s_i\beta \rangle} = c_{i,s_i\beta} = c_{i,\beta}. 
\]
Note that we have 
\[
c'_{i,\langle h_i, -\beta \rangle + 2n} = c_{i,-\beta},\ c''_{i,\langle h_i,s_i\beta \rangle+2n} = c_{i,\beta} \ (n \in \mathbb{Z}). 
\]
Let 
\[
\mathcal{U}'_q(\mathfrak{sl}_2) = \bigoplus_{n \in \mathbb{Z}} \mathcal{U}_q(\mathfrak{sl}_2, C')1_{\langle h_i, -\beta \rangle + 2n}, \ \mathcal{U}''_q(\mathfrak{sl}_2) = \bigoplus_{n \in \mathbb{Z}} \catquantum{\mathfrak{sl}_2, C''}1_{\langle h_i, s_i\beta \rangle + 2n}.
\] 
Let $\iota' \colon \mathcal{U}'_q(\mathfrak{sl}_2) \to \catquantum{\mathfrak{p}_i}$ and $\iota'' \colon \mathcal{U}''_q(\mathfrak{sl}_2) \to \catquantum{\mathfrak{p}_i}$ be the 2-functors given by 
\[
\iota' (\langle h_i, -\beta \rangle + 2n) = -\beta + n\alpha_i,\ \iota'' (\langle h_i, s_i\beta \rangle + 2n) = s_i\beta + n\alpha_i,
\]
and sending 1-morphisms and 2-morphisms in $\mathcal{U}_q' (\mathfrak{sl}_2)$ or in $\mathcal{U}_q''(\mathfrak{sl}_2)$ to those depicted by the same symbol or diagrams in $\mathcal{U}_q(\mathfrak{p}_i)$. 
Let $\xi \colon \mathcal{U}'_q(\mathfrak{sl}_2) \to \mathcal{U}''_q(\mathfrak{sl}_2)$ be the isomorphism of Proposition \ref{prop:scalarshift} given by setting
\[
b_{i,i} = 1,\ d_{i,m} = c'_{i,m} = c_{i,-\beta} \ (m \in \mathbb{Z}). 
\] 
Let $\omega \colon \mathcal{U}''_q(\mathfrak{sl}_2) \to \mathcal{U}''_q(\mathfrak{sl}_2)^{\mathrm{op}}$ be the Chevalley involution (Proposition \ref{prop:chevalleyinvolution}). 
Then, we have the following morphisms of 2-categories:
\[
\begin{tikzcd}
\mathcal{U}'_q(\mathfrak{sl}_2) \arrow[r,"\iota'"] \arrow[d,"\xi"] & \catquantum{\mathfrak{p}_i} \arrow[rd] & \\
\mathcal{U}''_q(\mathfrak{sl}_2) \arrow[d,"\omega"] & & \mathfrak{Lin}_{\mathbf{k}} \\
\mathcal{U}''_q(\mathfrak{sl}_2)^{\mathrm{op}} \arrow[r, "\iota''"] & \catquantum{\mathfrak{p}_i}^{\mathrm{op}} \arrow[ru] & 
\end{tikzcd}
\]
where the right-top arrow is given by the left action of $\catquantum{\mathfrak{p}_i}$ on $\gMod{R_i}$, 
and the right-bottom arrow is given by the right action of $\catquantum{\mathfrak{p}_i}$ on $\gMod{R_i}$. 

We claim that this diagram is commutative. 
By Theorem \ref{thm:Rouquierver} (2), it suffices to prove that the images of $E_i, F_i$ and the following generating 2-morphisms of $\mathcal{U}'_q(\mathfrak{sl}_2)$ coincides in $\mathfrak{End}(\gMod{R_i})$. 
\[
\xy 0;/r.15pc/: 
(0,0)*{\sdotu{i}};
\endxy, \xy 0;/r.15pc/: 
(0,0)*{\ucross{i}{i}};
\endxy, \xy 0;/r.15pc/: 
(0,0)*{\lcup{i}};
\endxy, \xy 0;/r.15pc/: 
(0,0)*{\lcap{i}};
\endxy.
\]
%(The generating 2-morphisms differ by applications of adjunction.)
It is a consequence of the following equalities obtained directly from the statement of Theorem \ref{thm:anotheraction} (1). 
\begin{align*}
& F_iX = X E_i, \ E_iX = XF_i, \\
& \xy 0;/r.15pc/: 
(0,0)*{\sdotu{i}};
\endxy X = X \xy 0;/r.15pc/: 
(0,0)*{\sdotd{i}};
\endxy, \\
& \xy 0;/r.15pc/: 
(0,0)*{\ucross{i}{i}};
\endxy X= X \xy 0;/r.15pc/: 
(0,0)*{\dcross{i}{i}};
\endxy, \\
&\xy 0;/r.15pc/: 
(0,0)*{\lcap{i}};
\endxy X = c_{i,\beta} X \xy 0;/r.15pc/: 
(0,0)*{\lcap{i}};
\endxy = c_{i,\beta+n\alpha_i} X \xy 0;/r.15pc/: 
(0,0)*{\lcap{i}};
\endxy, \\ 
&\xy 0;/r.15pc/: 
(0,0)*{\lcup{i}};
\endxy X = c_{i,\beta}^{-1} X \xy 0;/r.15pc/: 
(0,0)*{\lcup{i}};
\endxy = c_{i,\beta+n\alpha_i}^{-1} X \xy 0;/r.15pc/: 
(0,0)*{\lcup{i}};
\endxy, 
\end{align*}
where $n \in \mathbb{Z}, X \in \gMod{R_i(\beta + n\alpha_i)}$. 

Considering the images of $\xy 0;/r.15pc/: 
(0,0)*{\rcap{i}};
\endxy$ and $\xy 0;/r.15pc/: 
(0,0)*{\rcup{i}};
\endxy$ in $\mathfrak{Lin}_{\mathbf{k}}$ under the above commutative diagram, 
we obtain the desired equalities of natural homomorphisms.

The proof of Theorem \ref{thm:anotheraction} (1) is now complete.

\subsection{Monoidality}
In this section, we prove a part of Theorem \ref{thm:reflectionfunctor}: the functors $S_i$ and $S'_i$ are monoidal. 
Note that $\gMod{R_i}$ is both a left $\dotcatquantum{\mathfrak{p}_i}$-module by Theorem \ref{thm:cyclotomic2repleft} and a right $\dotcatquantum{\mathfrak{p}_i}$-module by Theorem \ref{thm:anotheraction} (1), 
while $\gMod{{}_iR}$ is both a right $\dotcatquantum{\mathfrak{p}_i}$-module by Theorem \ref{thm:cyclotomic2rep} and a left $\dotcatquantum{\mathfrak{p}_i}$-module by Theorem \ref{thm:anotheraction} (2).  

Note that there are canonical isomorphisms
\begin{align*}
&(X \circ Y) \circ Z \simeq X \circ Y \circ Z \simeq X \circ (Y \circ Z), \ (x \boxtimes y)\boxtimes z \mapsto x \boxtimes y \boxtimes z \mapsto x \boxtimes (y \boxtimes z) \\
&X \circ \mathbf{1} \simeq X \simeq \mathbf{1} \circ X, \ x \boxtimes 1 \mapsto x \mapsto 1 \boxtimes x, 
\end{align*}
for $X, Y, Z \in \gMod{R}$. 
By the definition of $S_i$, we also have 
\begin{align*}
S_i (XF_i) \simeq E_i S_i(X), \ S_i(XE_i) \simeq F_i S_i(X),\ S_i (XF_j) \simeq S_i(X) \circ M_j \ (j \neq i), 
\end{align*}
for $X \in \gMod{{}_iR}$, and 
\[
S_i(\mathbf{1}) = S_i({}_iR(0)) = R_i(0) = \mathbf{1}. 
\]
For $j \neq i$ and $X, Y \in \gMod{{}_iR}$, Theorem \ref{thm:functorF*} shows 
\[
\text{$X F_j \simeq X \circ R(\alpha_j)$, hence $X \circ YF_j \simeq (X \circ Y) F_j$.} 
\]
In addition, Lemma \ref{lem:adjointSES} shows that there are short exact sequences 
\begin{align*}
0 \to q_i^{-\langle h_i, \beta \rangle}(X F_i) \circ Y \hookrightarrow (X \circ Y) F_i \twoheadrightarrow X \circ (Y F_i) \to 0, \\
0 \to q_i^{\langle h_i, \beta' \rangle}(E_i X') \circ Y' \hookrightarrow E_i (X' \circ Y') \twoheadrightarrow X' \circ (E_i Y') \to 0.
\end{align*}
for $X \in \gMod{{}_iR(\alpha)}, Y \in \gMod{{}_iR(\beta)}, X' \in \gMod{R_i(\alpha')}, Y' \in \gMod{R_i(\beta')}$. 
In what follows, we use these natural homomorphisms without explicitly referring to them. 
In addition, we suppress degree shifts. 

The monoidality of $S_i$ is proved in the following proposition: 

\begin{proposition} \label{prop:monoidality}
There exist natural isomorphisms
\[
\theta = \theta_{\alpha,\beta}(X,Y) \colon S_i(X) \circ S_i(Y) \to S_i(X \circ Y), 
\] \index{$\theta \colon S_i(X) \circ S_i(Y) \xrightarrow{\sim} S_i(X \circ Y)$}
for $\alpha,\beta \in \mathsf{Q}_+, X \in \gMod{{}_iR(\alpha)}, Y \in \gMod{{}_iR(\beta)}$, 
that make the following diagrams commutative ($j \in I \setminus \{i\}$): 
\[
\begin{tikzcd}
\mathbf{1} \circ S_i(X) \arrow[r,-,"\sim"] \arrow[d,equal]\arrow[rd,phantom,"(1)"] & S_i(X) \arrow[d,-,"\sim" sloped] &  S_i(X) \circ \mathbf{1} \arrow[r,-,"\sim"]\arrow[d,equal]\arrow[rd,phantom,"(2)"] & S_i(X) \arrow[d,-,"\sim" sloped] \\
S_i(\mathbf{1}) \circ S_i(X) \arrow[r,"\theta"] & S_i(\mathbf{1} \circ X)  &  S_i(X) \circ S_i(\mathbf{1}) \arrow[r,"\theta"] & S_i(X \circ \mathbf{1}) 
\end{tikzcd}
\]
\[
\begin{tikzcd}
S_i (XF_i \circ Y) \arrow[r] & S_i((X\circ Y)F_i) \arrow[r] \arrow[d, -, "\sim" sloped] & S_i(X \circ YF_i)  \\
S_i (XF_i) \circ S_i(Y) \arrow[d,-,"\sim" sloped]\arrow[u,"\theta"]\arrow[r,phantom,"(3)"] & E_i S_i(X \circ Y) \arrow[r,phantom,"(4)"]  & S_i(X) \circ S_i(YF_i) \arrow[d,-,"\sim" sloped]\arrow[u,"\theta"] \\
E_i S_i(X) \circ S_i(Y) \arrow[r, hookrightarrow] & E_i(S_i(X) \circ S_i(Y)) \arrow[r,twoheadrightarrow]\arrow[u,"\theta"] & S_i(X) \circ E_iS_i(Y) 
\end{tikzcd}
\]
\[
\begin{tikzcd}
S_i(X) \circ S_i(Y F_j) \arrow[r,"\theta"]\arrow[d,-,"\sim" sloped]\arrow[rd,phantom,"(5)"] & S_i(X \circ Y F_j) \arrow[rd,-,"\sim" sloped] \\
S_i(X) \circ S_i(Y) \circ M_j \arrow[r,"\theta"] & S_i(X \circ Y) \circ M_j \arrow[r,-,"\sim"] & S_i((X \circ Y)F_j)  
\end{tikzcd}
\]
\[
\begin{tikzcd}
S_i(X) \circ S_i(Y) \circ S_i(Z) \arrow[r,"\theta"]\arrow[d,"\theta"]\arrow[rd,phantom,"(6)"] & S_i(X \circ Y) \circ S_i(Z) \arrow[d,"\theta"] \\
S_i(X) \circ S_i(Y \circ Z) \arrow[r,"\theta"] & S_i(X \circ Y \circ Z). 
\end{tikzcd}
\]
\end{proposition}

The rest of this section is devoted to the proof of this proposition. 
We first construct natural isomorphisms $\theta = \theta_{\alpha,\beta}$ that satisfy (2) -- (5) by induction on $\height \beta$. 
Then, we prove that they also satisfy (1) and (6). 
In the proof, we repeatedly use Lemma \ref{lem:adjointSES}, \ref{lem:EFvsmultiplication} and \ref{lem:naturality}, without explicitly referring to them. 

Since both $S_i(?) \circ S_i(?)$ and $S_i(? \circ ?)$ are right exact functors, 
we have 
\begin{align*}
&S_i(X) \circ S_i(Y) \simeq S_i({}_iR(\alpha)) \circ S_i({}_iR(\beta)) \otimes_{{}_iR(\alpha) \otimes {}_iR(\beta)} (X \otimes Y), \\
&S_i(X \circ Y) \simeq S_i({}_iR(\alpha) \circ {}_iR(\beta)) \otimes_{{}_iR(\alpha) \otimes {}_iR(\beta)} (X \otimes Y). 
\end{align*}
Hence, giving a natural isomorphism $\theta_{\alpha,\beta}$ is equivalent to giving an isomorphism of $({}_iR(\alpha+\beta), {}_iR(\alpha)\otimes {}_iR(\beta))$-modules
\[
S_i({}_iR(\alpha)) \circ S_i({}_iR(\beta)) \to S_i({}_iR(\alpha) \circ {}_iR(\beta)). 
\]
Moreover, we have ${}_iR(\beta) = \bigoplus_{\nu \in I^{\beta}} P(\nu)$, where
\[
P(\nu) = {}_iR(\beta)e(\nu) = {}_iR(0) F_{\nu_1} \cdots F_{\nu_{\height \beta}}. 
\]
Therefore, giving a natural isomorphism $\theta_{\alpha,\beta}$ is equivalent to giving isomorphisms
\begin{equation} \label{eq:naturaliso}
\theta \colon S_i(X) \circ S_i(Y F_j) \to S_i(X \circ YF_j) \ (j \in I)
\end{equation}
that is natural in $X \in \gMod{{}_iR(\alpha)}$ and $Y \in \gMod{{}_iR(\beta-\alpha_j)}$, and make the following diagrams commutative:
\[
\begin{tikzcd}
S_i(X) \circ S_i(Y F_j) \arrow[r,"{\xy 0;/r.12pc/: (0,0)*{\sdotd{j}}; \endxy}"]\arrow[d,"\theta"]\arrow[rd,phantom,"(7)"] & S_i(X) \circ S_i(YF_j) \arrow[d,"\theta"] \\
S_i(X \circ YF_j) \arrow[r,"{\xy 0;/r.12pc/: (0,0)*{\sdotd{j}}; \endxy}"'] & S_i(X \circ YF_j)
\end{tikzcd} 
\]
for $j \in I, X \in \gMod{{}_iR(\alpha)}, Y \in \gMod{{}_iR(\beta-\alpha_j)}$, and 
\[
\begin{tikzcd}
S_i(X) \circ S_i(YF_jF_k) \arrow[r,"{\xy 0;/r.12pc/: (0,0)*{\dcross{j}{k}}; \endxy}"]\arrow[d,"\theta"]\arrow[rd,phantom,"(8)"] & S_i(X) \circ S_i(YF_kF_j) \arrow[d,"\theta"] \\
S_i(X \circ YF_jF_k) \arrow[r,"{\xy 0;/r.12pc/: (0,0)*{\dcross{j}{k}}; \endxy}"'] & S_i(X \circ YF_kF_j) \\
\end{tikzcd}
\]
for $j,k \in I, X \in \gMod{{}_iR(\alpha)}, Y \in \gMod{{}_iR(\beta-\alpha_j-\alpha_k)}$. 

\begin{comment}
\begin{remark}
More rigorously, we consider the following. 
Let ${}_i\mathcal{C} = \bigoplus_{\beta \in \mathsf{Q}_+}{}_i\mathcal{C}(\beta)$ be the full subcategory of $\gproj{{}_iR}$ consisting objects
\[
P(\nu) = \mathbf{1} F_{\nu_1} \cdots F_{\nu_n} 
\]
indexed by $\nu \in \bigcup_{n \geq 0} I^n$. 
The functor $S_i \colon {}_i\mathcal{C} \to \gproj{R_i}$ is inductively defined by 
\[
S_i(P(\nu,j)) = \begin{cases}
E_iS_i(P(\nu)) & \text{if $j = i$},
S_i(P(\nu))\circ M_j & \text{if $j \neq i$}. 
\end{cases}
\]

\end{remark}
\end{comment}

\subsubsection{The initial step}
We define $\theta_{\alpha,0}$ for any $\alpha \in \mathsf{Q}_+$. 
Since ${}_iR(0) = \mathbf{k} = \mathbf{1}$, giving a natural homomorphism $\theta_{\alpha,0}$ is equivalent to giving a natural homomorphism 
\[ 
\theta_{\alpha,0} (?,\mathbf{1}) \colon S_i(?) \circ S_i(\mathbf{1}) \to S_i(? \circ \mathbf{1}). 
\] 
We define it to be the unique isomorphism that makes the diagram (2) commutative. 

It makes Diagram (3) commutative when $\beta = 0$. 
In fact, it follows from the commutative diagram below: 
\[
\begin{tikzcd}
S_i(XF_i \circ \mathbf{1}) \arrow[ddd,bend right=50]\arrow[d,-,"\sim" sloped]\arrow[rd,phantom,"(2)"] & S_i(XF_i) \circ S_i(\mathbf{1}) \arrow[l,"\theta"']\arrow[r,-,"\sim"]\arrow[d,equal] & E_iS_i(X) \circ S_i(\mathbf{1}) \arrow[ddd,bend left=60]\arrow[d,equal] \\
S_i(XF_i) \arrow[r,-,"\sim"]\arrow[d,equal] & S_i(XF_i) \circ \mathbf{1} \arrow[r,-,"\sim"] & E_iS_i(X) \circ \mathbf{1} \arrow[d] \\
S_i(XF_i) \arrow[r,-,"\sim"] & E_iS_i(X) \arrow[r,-,"\sim"]\arrow[rd,phantom,"(2)"] & E_i(S_i(X) \circ \mathbf{1}) \\
S_i((X\circ \mathbf{1})F_i)\arrow[u,-,"\sim"' sloped] \arrow[r,-,"\sim"] & E_i(S_i(X \circ \mathbf{1})) \arrow[u,-,"\sim"' sloped] & E_i(S_i(X) \circ S_i(\mathbf{1})) \arrow[l,"\theta"']\arrow[u,equal]. 
\end{tikzcd}
\]

\subsubsection{The induction step: construction}
Let $\beta \in \mathsf{Q}_+ \setminus \{0\}$.
Assume that $\theta_{\alpha,\beta'}$ is constructed for every $\alpha,\beta' \in \mathsf{Q}_+$ that satisfies $\height \beta' < \height \beta$, 
and they make the diagrams (2) -- (5) and (7), (8) commutative as long as all the $\theta$ involved have already been constructed. 
In order to construct a natural isomorphism $\theta_{\alpha,\beta}$, 
we give data of isomorphisms \eqref{eq:naturaliso}. 

For $j \neq i, X \in \gMod{{}_iR(\alpha)}$ and $Y \in \gMod{{}_iR(\beta - \alpha_j)}$, 
we define 
\begin{equation} \label{eq:defisoforj}
\theta \colon S_i(X) \circ S_i(YF_j) \to S_i(X \circ YF_j) 
\end{equation}
to be the unique isomorphism that makes Diagram (5) commutative. 
(Note that we already have $\theta_{\alpha,\beta-\alpha_j} \colon S_i(X) \circ S_i(Y) \to S_i(X \circ Y)$ by the induction hypothesis.)

For $j = i, X \in \gMod{{}_iR(\alpha)}$ and $Y \in \gMod{{}_iR(\beta -\alpha_i)}$, 
we already have the commutative diagram (3) by the induction hypothesis. 
Note that the following sequences appearing in Diagram (3) (4) are exact: 
\begin{align*}
&0 \to E_iS_i(X) \circ S_i(Y) \to E_i(S_i(X) \circ S_i(Y)) \to S_i(X) \circ E_iS_i(Y) \to 0, \\
&0 \to S_i(XF_i \circ Y) \to S_i((X\circ Y)F_i) \to S_i(X \circ YF_i) \to 0. 
\end{align*}
In fact, the first one is exact by Lemma \ref{lem:adjointSES}. 
The exactness at the left term in the second sequence follows from the exactness of the first one by Diagram (3). 
The remaining exactness follows from Lemma \ref{lem:adjointSES} by applying the right exact functor $S_i$.  
Hence, we can define 
\begin{equation} \label{eq:defisofori}
\theta \colon S_i(X) \circ S_i(YF_i) \to S_i(X \circ YF_i) 
\end{equation}
to be the unique isomorphism that makes Diagram (4) commutative. 

By construction these isomorphisms $\theta$ are natural in $X$ and $Y$. 

\subsubsection{The induction step: naturality}
We verify that the isomorphisms $\theta_{\alpha,\beta} \colon S_i(X)\circ S_i(YF_j) \to S_i(X\circ YF_j)$ defined in (\ref{eq:defisoforj}) and (\ref{eq:defisofori}) make Diagram (7) and (8) commutative. 

The commutativity of (7) for $j = i$: 
Consider the commutative diagram (4) that is used to define $\theta_{\alpha,\beta} \colon S_i(X) \circ S_i(YF_i) \to S_i(X \circ YF_i)$ in (\ref{eq:defisofori}). 
Each of the six modules has an endomorphism induced from $\xy 0;/r.12pc/: (0,0)*{\sdotd{i}}; \endxy$ or $\xy 0;/r.12pc/: (0,0)*{\sdotu{i}}; \endxy$. 
All the homomorphisms in the diagram except $\theta_{\alpha,\beta}$ commute with these endomorphisms. 
Here, we used the equality
\[
S_i(Y) \xy 0;/r.12pc/: (0,0)*{\sdotd{i}}; \endxy = \xy 0;/r.12pc/: (0,0)*{\sdotu{i}}; \endxy S_i(Y)
\]
in Theorem \ref{thm:anotheraction} for example. 
Hence, $\theta_{\alpha,\beta}$ also commutes with $\xy 0;/r.12pc/: (0,0)*{\sdotd{i}}; \endxy$. 

The commutativity of (7) for $j \neq i$: 
Consider the commutative diagram (5) that is used to define $\theta_{\alpha,\beta} \colon S_i(X) \circ S_i(YF_j) \to S_i(X \circ YF_j)$ in (\ref{eq:defisoforj}). 
Each of the five modules has an endomorphism induced from $\xy 0;/r.12pc/: (0,0)*{\sdotd{j}}; \endxy$ or $z_j \in \END(M_j)$. 
All the homomorphisms in the diagram except $\theta_{\alpha,\beta}$ commute with these endomorphisms, hence $\theta_{\alpha,\beta}$ also does. 
 
The proof of the commutativity of (8) splits into the following four cases: 

Case 1. $j = k = i$. 
Consider the following diagram: 
\begin{equation*}
\centering
\begin{tikzpicture}[auto]
\node (a) at (0,4,0) {$E_iE_i(S_i(X) \circ S_i(Y))$}; 
\node (b) at (4,4,0) {$E_iE_iS_i(X \circ Y)$};
\node (c) at (8,4,0) {$S_i((X \circ Y)F_iF_i)$};
\node (d) at (0,0,0) {$E_iE_i(S_i(X) \circ S_i(Y))$};
\node (e) at (4,0,0) {$E_iE_iS_i(X \circ Y)$};
\node (f) at (8,0,0) {$S_i((X \circ Y)F_iF_i)$};
\node (g) at (0,4,4.2) {$S_i(X) \circ E_iE_iS_i(Y)$}; 
\node (h) at (4,4,4.2) {$S_i(X) \circ S_i(Y F_iF_i)$};
\node (i) at (8,4,4.2) {$S_i(X \circ YF_iF_i)$};
\node (j) at (0,0,4.2) {$S_i(X) \circ E_iE_iS_i(Y)$};
\node (k) at (4,0,4.2) {$S_i(X) \circ S_i(YF_iF_i)$};
\node (l) at (8,0,4.2) {$S_i(X \circ YF_iF_i)$};
\draw[->] (a) -- node {$\scriptstyle \theta$} (b);
\draw (b) -- node {$\scriptstyle \sim$} (c);
\draw[->] (d) -- node[pos=.3] {$\scriptstyle \theta$} (e);
\draw (e) -- node {$\scriptstyle \sim$} (f);
\draw[->] (a) -- (d);
\draw[->] (b) -- (e);
\draw[->] (c) -- (f);
\draw[->>] (a) -- (g);
\draw[->>] (c) -- (i);
\draw[->>] (d) -- (j);
\draw[->>] (f) -- (l);
\draw[cross] (g) -- node {$\scriptstyle \sim$} (h);
\draw[cross,->, line width=1.3pt] (h) -- node {$\scriptstyle \theta$} (i);
\draw (j) -- node {$\scriptstyle \sim$} (k);
\draw[->,line width=1.3pt] (k) -- node {$\scriptstyle \theta$} (l);
\draw[->] (g) -- (j);
\draw[cross,->,line width=1.3pt] (h) -- (k);
\draw[cross,->,line width=1.3pt] (i) -- (l);
\end{tikzpicture}
\end{equation*}
where each vertical homomorphism is induced from $\xy 0;/r.12pc/: (0,0)*{\dcross{i}{i}}; \endxy$ or $\xy 0;/r.12pc/: (0,0)*{\ucross{i}{i}}; \endxy$. 
We need to prove that the thick square commutes. 
It is easy to see that all the side faces except the thick one are commutative (there are five such faces).
Hence, it suffices to prove that the top and the bottom faces commute, 
which follows from the commutative diagram below:
\begin{equation*}
\begin{tikzcd}
E_iE_i(S_i(X) \circ S_i(Y)) \arrow[r,twoheadrightarrow]\arrow[d,"\theta"] & E_i(S_i(X) \circ E_iS_i(Y)) \arrow[r,twoheadrightarrow]\arrow[d,-,"\sim" sloped]\arrow[rd,phantom,"\text{naturality}"] & S_i(X) \circ E_iE_iS_i(Y) \arrow[d,-,"\sim" sloped] \\
E_iE_iS_i(X \circ Y) \arrow[d,-,"\sim" sloped]\arrow[r,phantom,"(4)"] & E_i(S_i(X) \circ S_i(Y F_i)) \arrow[d,"\theta"]\arrow[r,twoheadrightarrow] & S_i(X) \circ E_iS_i(YF_i) \arrow[d,-,"\sim"sloped] \\
E_iS_i((X \circ Y)F_i) \arrow[r,twoheadrightarrow]\arrow[d,-,"\sim" sloped]\arrow[rd,phantom,"\text{naturality}"] & E_iS_i(X \circ Y F_i) \arrow[d,-,"\sim" sloped]\arrow[r,phantom,"(4)"] & S_i(X) \circ S_i(Y F_iF_i) \arrow[d,"\theta"] \\
S_i((X\circ Y) F_iF_i) \arrow[r,twoheadrightarrow] & S_i((X\circ YF_i)F_i) \arrow[r,twoheadrightarrow] & S_i(X \circ Y F_iF_i)
\end{tikzcd} 
\end{equation*}

Case 2. $j = i, k \neq i$. 
Consider the following diagram: 
\begin{equation*}
\centering
\scalebox{0.8}{
\begin{tikzpicture}[auto]
\node (a) at (0,4,0) {$E_i(S_i(X) \circ S_i(Y)) \circ M_k$}; 
\node (b) at (5,4,0) {$E_iS_i(X \circ Y) \circ M_k$};
\node (c) at (9,4,0) {$S_i((X \circ Y)F_iF_k)$};
\node (d) at (0,0,0) {$E_i(S_i(X) \circ S_i(Y) \circ M_k)$};
\node (e) at (5,0,0) {$E_i(S_i(X \circ Y) \circ M_k)$};
\node (f) at (9,0,0) {$S_i((X \circ Y)F_kF_i)$};
\node (g) at (0,4,5) {$S_i(X) \circ E_iS_i(Y) \circ M_k$}; 
\node (h) at (5,4,5) {$S_i(X) \circ S_i(YF_iF_k)$};
\node (i) at (9,4,5) {$S_i(X \circ YF_iF_k)$};
\node (j) at (0,0,5) {$S_i(X) \circ E_i(S_i(Y) \circ M_k)$};
\node (k) at (5,0,5) {$S_i(X) \circ S_i(Y F_kF_i)$};
\node (l) at (9,0,5) {$S_i(X \circ YF_kF_i)$};
\draw[->] (a) -- node {$\scriptstyle \theta$} (b);
\draw (b) -- node {$\scriptstyle \sim$} (c);
\draw[->] (d) -- node[pos=.3] {$\scriptstyle \theta$} (e);
\draw (e) -- node[pos=.8] {$\scriptstyle \sim$} (f);
\draw[->] (a) -- (d);
\draw[->] (b) -- (e);
\draw[->] (c) -- (f);
\draw[->>] (a) -- (g);
\draw[->>] (c) -- (i);
\draw[->>] (d) -- (j);
\draw[->>] (f) -- (l);
\draw[cross] (g) -- node {$\scriptstyle \sim$} (h);
\draw[cross,->, line width=1.3pt] (h) -- node {$\scriptstyle \theta$} (i);
\draw (j) -- node {$\scriptstyle \sim$} (k);
\draw[->,line width=1.3pt] (k) -- node {$\scriptstyle \theta$} (l);
\draw[->] (g) -- (j);
\draw[cross,->,line width=1.3pt] (h) -- (k);
\draw[cross,->,line width=1.3pt] (i) -- (l);
\end{tikzpicture}
}
\end{equation*}
where each vertical homomorphism is induced from $\xy 0;/r.12pc/: (0,0)*{\dcross{i}{k}}; \endxy$ or $\sigma_{i,k}$.
We need to prove that the thick square commutes. 
It is easy to see that all the side faces except the thick one and the left-most one commute (there are four such faces). 
Hence, it is enough to show the commutativity of the left-most, top, and bottom faces. 

The commutativity of the left-most face follows from the lemma below: 
\begin{lemma} \label{lem:commute}
Let $\alpha',\beta',\gamma' \in \mathsf{Q}_+$ and 
\[
L \in \gMod{R_i(\alpha')},\  M \in \gMod{R_i(\beta')},\ N \in \gMod{R_i(\gamma')}.
\] 
The following diagram commutes: 
\[
\begin{tikzcd}
E_i(L \circ M) \circ N \arrow[r,hookrightarrow,"i"]\arrow[d,twoheadrightarrow,"p"] & E_i(L \circ M \circ N) \arrow[d,twoheadrightarrow, "p'"] \\
L \circ E_iM \circ N \arrow[r,hookrightarrow,"i'"] & L \circ E_i(M \circ N)
\end{tikzcd}
\]
\end{lemma}

\begin{proof}
Note that $E_i(L \circ M) \circ N$ is generated as an $R(\alpha' + \beta' +\gamma' - \alpha_i)$-module by $E_i(l \boxtimes m) \boxtimes n$ and $E_i\tau_1 \cdots \tau_a (l \boxtimes m) \boxtimes n$ ($l \in L, m \in M, n \in N$), where $a = \height \alpha'$. 
Hence, it suffices to verify that $p'i$ and $i'p$ coincide on these generating elements. 
It follows from the calculation below: 
\begin{align*}
&p'i(E_i(l \boxtimes m) \boxtimes n) = p'(E_i(l \boxtimes m \boxtimes n)) = 0, \\
&i'p(E_i(l \boxtimes m) \boxtimes n) = p(0) = 0, \\
&p'i(E_i\tau_1 \cdots \tau_a (l\boxtimes m) \boxtimes n) = p'(E_i\tau_1 \cdots \tau_a (l \boxtimes m \boxtimes n)) = l \boxtimes E_i(m \boxtimes n), \\
&i'p(E_i\tau_1 \cdots \tau_a (l\boxtimes m) \boxtimes n) = i'(l \boxtimes E_im \boxtimes n) = l \boxtimes E_i(m \boxtimes n). 
\end{align*}
\end{proof}

The commutativity of the top face follows from the commutative diagram below: 
\begin{equation} \label{pic:1}
\begin{tikzcd}
E_i(S_i(X)\circ S_i(Y)) \circ M_k \arrow[r,twoheadrightarrow]\arrow[d,"\theta"] & S_i(X) \circ E_iS_i(Y) \circ M_k \arrow[d,-,"\sim" sloped] & \\
E_iS_i(X \circ Y) \circ M_k \arrow[d,-,"\sim" sloped]\arrow[r,phantom,"(4)"] & S_i(X) \circ S_i(Y F_i) \circ M_k \arrow[d,"\theta"]\arrow[r,-,"\sim" sloped]\arrow[rd,phantom,"(5)"] & S_i(X) \circ S_i(Y F_iF_k) \arrow[d,"\theta"] \\
S_i((X \circ Y)F_i) \circ M_k \arrow[d,-,"\sim" sloped]\arrow[r,twoheadrightarrow]\arrow[rd,phantom,"\text{naturality}"] & S_i(X \circ YF_i) \circ M_k \arrow[d,-,"\sim"sloped] & S_i(X \circ Y F_iF_k) \\
S_i((X\circ Y)F_iF_k) \arrow[r,twoheadrightarrow] & S_i((X \circ YF_i)F_k) \arrow[ru,-,"\sim" sloped] &  \\
\end{tikzcd}
\end{equation}

The commutativity of the bottom face follows from the commutative diagram below: 
\begin{equation} \label{pic:2}
\adjustbox{scale=0.9,center}{
\begin{tikzcd}
E_i(S_i(X) \circ S_i(Y) \circ M_k) \arrow[rr,twoheadrightarrow] \arrow[d,"\theta"]\arrow[rd,-,"\sim" sloped]\arrow[rrd,phantom,"\text{naturality}"]\arrow[rdd,phantom,"(5)"] & & S_i(X) \circ E_i(S_i(Y) \circ M_k) \arrow[d,-,"\sim" sloped] \\
E_i(S_i(X \circ Y) \circ M_k) \arrow[d,-,"\sim" sloped] & E_i(S_i(X) \circ S_i(Y F_k)) \arrow[d,"\theta"]\arrow[r,twoheadrightarrow] & S_i(X) \circ E_iS_i(YF_k) \arrow[d,-,"\sim" sloped] \\
E_iS_i((X \circ Y)F_k) \arrow[r,-,"\sim" sloped]\arrow[d,-,"\sim" sloped]\arrow[rd,phantom,"\text{naturality}"] &E_iS_i(X \circ YF_k) \arrow[d,-,"\sim" sloped]\arrow[r,phantom,"(4)"] & S_i(X) \circ S_i(Y F_kF_i) \arrow[d,"\theta"] \\
S_i((X \circ Y)F_kF_i) \arrow[r,-,"\sim" sloped] & S_i((X \circ YF_k)F_i)\arrow[r,twoheadrightarrow] & S_i(X \circ YF_kF_i)\\ 
\end{tikzcd}
}
\end{equation}

Case 3. $j \neq i, k = i$. 
Consider the following diagram. 
\begin{equation*}
\centering
\scalebox{0.8}{
\begin{tikzpicture}[auto]
\node (a) at (-1,4,0) {$E_i(S_i(X) \circ S_i(Y) \circ M_j)$}; 
\node (b) at (4,4,0) {$E_i(S_i(X \circ Y) \circ M_j)$};
\node (c) at (8,4,0) {$S_i((X \circ Y)F_jF_i)$};
\node (d) at (-1,0,0) {$E_i(S_i(X) \circ S_i(Y)) \circ M_j$};
\node (e) at (4,0,0) {$E_iS_i(X \circ Y) \circ M_j$};
\node (f) at (8,0,0) {$S_i((X \circ Y)F_iF_j)$};
\node (g) at (-1,4,5) {$S_i(X) \circ E_i(S_i(Y) \circ M_j)$}; 
\node (h) at (4,4,5) {$S_i(X) \circ S_i(Y F_jF_i)$};
\node (i) at (8,4,5) {$S_i(X \circ YF_jF_i)$};
\node (j) at (-1,0,5) {$S_i(X) \circ E_iS_i(Y) \circ M_j$};
\node (k) at (4,0,5) {$S_i(X) \circ S_i(YF_iF_j)$};
\node (l) at (8,0,5) {$S_i(X \circ YF_iF_j)$};
\draw[->] (a) -- node {$\scriptstyle \theta$} (b);
\draw (b) -- node {$\scriptstyle \sim$} (c);
\draw[->] (d) -- node[pos=.3] {$\scriptstyle \theta$} (e);
\draw (e) -- node[pos=.8] {$\scriptstyle \sim$} (f);
\draw[->] (a) -- (d);
\draw[->] (b) -- (e);
\draw[->] (c) -- (f);
\draw[->>] (a) -- (g);
\draw[->>] (c) -- (i);
\draw[->>] (d) -- (j);
\draw[->>] (f) -- (l);
\draw[cross] (g) -- node {$\scriptstyle \sim$} (h);
\draw[cross,->, line width=1.3pt] (h) -- node {$\scriptstyle \theta$} (i);
\draw (j) -- node {$\scriptstyle \sim$} (k);
\draw[->,line width=1.3pt] (k) -- node {$\scriptstyle \theta$} (l);
\draw[->] (g) -- (j);
\draw[cross,->,line width=1.3pt] (h) -- (k);
\draw[cross,->,line width=1.3pt] (i) -- (l);
\end{tikzpicture}
}
\end{equation*}
where each vertical homomorphism is induced from $\xy 0;/r.12pc/: (0,0)*{\dcross{j}{i}}; \endxy$ or $\sigma_{j,i}$. 
We need to prove that the thick square commutes. 
The top and bottom faces are commutative by (\ref{pic:1}) and (\ref{pic:2}). 
It is easy to see that all the side faces except the thick and the left-most one are commutative (there are four such faces). 
Hence, it suffices to prove that the left-most face commutes. 
 
Consider the following diagram:
\begin{equation*}
\adjustbox{scale=0.9,center}{
\begin{tikzcd}
E_i(S_i(X) \circ S_i(Y) \circ M_j) \arrow[r,"\sigma_{j,i}"]\arrow[d,twoheadrightarrow]\arrow[rr,bend left, "{Q_{i,j}\left(\xy 0;/r.12pc/: (0,0)*{\sdotu{i}}; \endxy,y_j\right)}"] & E_i(S_i(X) \circ S_i(Y)) \circ M_j \arrow[r,hookrightarrow]\arrow[d,twoheadrightarrow] & E_i(S_i(X)\circ S_i(Y) \circ M_j) \arrow[d,twoheadrightarrow] \\
S_i(X) \circ E_i(S_i(Y) \circ M_j) \arrow[r,"\sigma_{j,i}"]\arrow[rr,bend right, "{Q_{i,j}\left(\xy 0;/r.12pc/: (0,0)*{\sdotu{i}}; \endxy,y_j\right)}"] & S_i(X) \circ E_iS_i(Y) \circ M_j \arrow[r,hookrightarrow] & S_i(X) \circ E_i(S_i(Y) \circ M_j)
\end{tikzcd}
}
\end{equation*}
The right square commutes by Lemma \ref{lem:commute}, 
and the top and bottom triangles commute by Definition \ref{def:anothertauij}. 
Hence, the left square is commutative as desired. 

Case 4. $j,k \neq i$. 
Consider the following diagram: 
\begin{equation*}
  \centering
  \scalebox{0.9}{
  \begin{tikzpicture}[auto]
  \node (a) at (-2,4,0) {$S_i(X) \circ S_i(Y) \circ M_j \circ M_k$};
  \node (b) at (4,4,0) {$S_i(X \circ Y) \circ M_j \circ M_k$};
  \node (c) at (8,4,0) {$S_i((X \circ Y)F_jF_k)$};
  \node (d) at (-2,0,0) {$S_i(X) \circ S_i(Y) \circ M_k \circ M_j$};
  \node (e) at (4,0,0) {$S_i(X \circ Y) \circ M_k \circ M_j$};
  \node (f) at (8,0,0) {$S_i((X \circ Y)F_kF_j)$};
  \node (h) at (4,4,5) {$S_i(X) \circ S_i(Y F_jF_k)$};
  \node (i) at (8,4,5) {$S_i(X \circ YF_jF_k)$};
  \node (k) at (4,0,5) {$S_i(X) \circ S_i(YF_kF_j)$};
  \node (l) at (8,0,5) {$S_i(X \circ YF_kF_j)$};
  \draw[->] (a) -- node {$\scriptstyle \theta$} (b);
  \draw (b) -- node {$\scriptstyle \sim$} (c);
  \draw[->] (d) -- node[pos=.3] {$\scriptstyle \theta$} (e);
  \draw (e) -- node[pos=.8] {$\scriptstyle \sim$} (f);
  \draw[-.] (a) -- (d);
  \draw[->] (b) -- (e);
  \draw[->] (c) -- (f);
  \draw (c) -- node[sloped] {$\scriptstyle \sim$} (i);
  \draw (f) -- node[sloped] {$\scriptstyle \sim$} (l);
  \draw[cross,->, line width = 1.3pt] (h) -- node {$\scriptstyle \theta$} (i);
  \draw[->,line width = 1.3pt] (k) -- node {$\scriptstyle \theta$} (l);
  \draw[cross,->,line width = 1.3pt] (h) -- (k);
  \draw[cross,->,line width = 1.3pt] (i) -- (l);
  \draw (a) -- node[sloped] {$\scriptstyle \sim$} (h);
  \draw (d) -- node[sloped] {$\scriptstyle \sim$} (k);
  %\draw (b) -- node[sloped] {$\scriptstyle \sim$} (h);
  %\draw (e) -- node[sloped] {$\scriptstyle \sim$} (k);
  \end{tikzpicture}
  }
  \end{equation*}
where each vertical homomorphism is induced from $\xy 0;/r.12pc/: (0,0)*{\dcross{j}{k}}; \endxy$ or $\sigma_{j,k}$. 
We need to prove that the thick square commutes. 
It is easy to see that the side faces except the thick one are commutative (there are four such faces). 
The commutativity of the top pentagon follows from the commutative diagram below: 
\begin{equation*}
\begin{tikzcd}
S_i(X) \circ S_i(Y) \circ M_j \circ M_k \arrow[rd,-,"\sim" sloped] \arrow[d,"\theta"] & & \\ 
S_i(X \circ Y) \circ M_j \circ M_k \arrow[d,-,"\sim" sloped]\arrow[r,phantom,"(5)"] & S_i(X) \circ S_i(Y F_j) \circ M_k \arrow[rd,-,"\sim" sloped]\arrow[d,"\theta"] & \\
S_i((X \circ Y)F_j) \circ M_k \arrow[r,-,"\sim" sloped]\arrow[d,-,"\sim" sloped]\arrow[rd,phantom,"\text{naturality}"] & S_i(X \circ YF_j) \circ M_k \arrow[d,-,"\sim" sloped]\arrow[r,phantom,"(5)"] & S_i(X) \circ S_i(Y F_jF_k) \arrow[d,"\theta"] \\
S_i((X \circ Y)F_jF_k) \arrow[r,-,"\sim" sloped] & S_i((X \circ YF_j)F_k)\arrow[r,-,"\sim" sloped] & S_i(X \circ Y F_jF_k) \\
\end{tikzcd}
\end{equation*}
The commutativity of the bottom pentagon is also deduced from the diagram above with $j$ and $k$ swapped. 

This completes the proof of the commutativity of Diagram (7) and (8). 

\subsubsection{The induction step: commutativity of (3)}
We verify that the natural isomorphism $\theta_{\alpha,\beta}$ defined in (\ref{eq:defisoforj}) and (\ref{eq:defisofori}) makes Diagram (3) commutative 
for $X \in \gMod{{}_iR(\alpha)}, Y \in \gMod{{}_iR(\beta)}$. 
We may assume $Y = Y'F_j$ for some $j \in I$ and $Y' \in \gMod{{}_iR(\beta - \alpha_j)}$. 

Case 1. $j = i$. 
Consider the following diagram: 
\begin{equation} \label{pic:3}
\centering
\scalebox{0.55}{
\begin{tikzpicture}[auto]
 \node (a) at (-5.5,0) {$E_iS_i(X) \circ E_iS_i(Y')$};
 \node (b) at (-2,0) {$E_iS_i(X) \circ S_i(Y' F_i)$};
 \node (c) at  (2,0) {$E_i(S_i(X) \circ S_i(Y'F_i))$};
 \node (d) at (6,0) {$E_i(S_i(X) \circ E_iS_i(Y'))$};
 \node (e) at (-5.5,2) {$S_i(XF_i) \circ E_iS_i(Y')$};
 \node (f) at (-2,2) {$S_i(XF_i)\circ S_i(Y'F_i)$};
 \node (g) at (2,2) {$E_iS_i(X \circ Y'F_i)$};
 \node (h) at (-2,4) {$S_i(XF_i \circ Y'F_i)$};
 \node (i) at (2,4) {$S_i((X \circ Y'F_i)F_i)$};
 \node (j) at (-9.5,4) {$E_i(E_iS_i(X) \circ S_i(Y'))$};
 \node (k) at  (-9.5,6) {$E_i(S_i(X F_i) \circ S_i(Y'))$};
 \node (l) at (-9.5,8) {$E_iS_i(X F_i \circ Y')$};
 \node (m) at  (-9.5,10) {$S_i((XF_i \circ Y') F_i)$};
 \node (n) at (9,4) {$E_iE_i(S_i(X) \circ S_i(Y'))$};
 \node (o) at (9,6) {$E_iE_iS_i(X \circ Y')$};
 \node (p) at (9,8) {$E_iS_i((X \circ Y')F_i)$};
 \node (q) at (9,10) {$S_i((X \circ Y')F_iF_i)$};
 \node (r) at (0,8) {$E_iE_i(S_i(X) \circ S_i(Y'))$};
 \node (s) at (0,10) {$E_iE_iS_i(X \circ Y')$};
 \node (t) at (0,12) {$E_iS_i((X \circ Y')F_i)$};
 \node (u) at (0,14) {$S_i((X \circ Y')F_iF_i)$};
 \draw (a) -- node {$\scriptstyle \sim$} (b);
 \draw[{Hooks[right]}->, line width =1.5pt] (b) -- (c);
 \draw (c) -- node {$\scriptstyle \sim$} (d);
 \draw (a) -- node[sloped] {$\scriptstyle \sim$} (e);
 \draw[line width =1.5pt] (b) -- node[sloped] {$\scriptstyle \sim$} (f);
 \draw[->,line width =1.5pt] (c) -- node {$\scriptstyle \theta$} (g);
 \draw (e) -- node {$\scriptstyle \sim$} (f);
 \draw[->,line width =1.5pt] (f) -- node {$\scriptstyle \theta$} (h);
 \draw[line width =1.5pt] (g) -- node[sloped] {$\scriptstyle \sim$} (i);
 \draw[->,line width =1.5pt] (h) -- (i);
 \draw (j) -- node[sloped] {$\scriptstyle \sim$} (k);
 \draw[->] (k) -- node {$\scriptstyle \theta$} (l);
 \draw (l) -- node[sloped] {$\scriptstyle \sim$} (m);
 \draw[->>] (j) -- (a);
 \draw[->] (n) -- node {$\scriptstyle \theta$} (o);
 \draw (o) -- node[sloped] {$\scriptstyle \sim$} (p);
 \draw (p) -- node[sloped] {$\scriptstyle \sim$} (q);
 \draw[->>] (n) -- (d);
 \draw[->] (r) -- node {$\scriptstyle \theta$} (s);
 \draw (s) -- node[sloped] {$\scriptstyle \sim$} (t);
 \draw (t) -- node[sloped] {$\scriptstyle \sim$} (u);
 \draw[->] (m) -- (u);
 \draw[->] (l) -- (t);
 \draw[{Hooks[right]}->] (j) -- (r);
 \draw[->] (r) -- node[pos=0.2] {$\xy 0;/r.12pc/: (0,0)*{\ucross{i}{i}}; \endxy$} (n);
 \draw[->] (s) -- node[pos=0.3] {$\xy 0;/r.12pc/: (0,0)*{\ucross{i}{i}}; \endxy$} (o);
 \draw[->] (u) -- node {$\xy 0;/r.12pc/: (0,0)*{\dcross{i}{i}}; \endxy$} (q);
 \draw[cross,->>] (k) -- (e);
 \draw[cross,->>] (m) -- (h);
 \draw[cross,->>] (p) -- (g);
 \draw[cross,->>] (q) -- (i);
\end{tikzpicture}
}
\end{equation}
We need to prove that the thick face commutes. 
It is easy to see that all the side faces except the thick one are commutative (there are nine such faces). 
Hence, it suffices to prove the top and bottom faces commute. 

In order to prove the commutativity of the top face, consider the following diagram: 
\begin{equation*}
\centering
\scalebox{0.8}{
\begin{tikzpicture}[auto]
 \node (a) at (-2,0) {$XF_i \circ Y'F_i$}; 
 \node (b) at (2,0) {$(X \circ Y'F_i)F_i$};
 \node (c) at (-3,2) {$(XF_i \circ Y') F_i$};
 \node (d) at (3,2) {$(X \circ Y')F_iF_i$};
 \node (e) at (0,3) {$(X\circ Y')F_iF_i$};
 \node (f) at (-4,-2) {$X \circ R(\alpha_i) \circ Y'F_i$};
 \node (g) at (4,-2) {$X \circ Y'F_i \circ R(\alpha_i)$};
 \node (h) at (-6,3) {$X \circ R(\alpha_i) \circ Y' \circ R(\alpha_i)$};
 \node (i) at (6,3) {$X \circ Y' \circ R(\alpha_i) \circ R(\alpha_i)$};
 \node (j) at (0,5) {$X \circ Y' \circ R(\alpha_i) \circ R(\alpha_i)$};
 \draw[{Hooks[right]}->] (a) -- (b);
 \draw[->>] (c) -- (a);
 \draw[->>] (d) -- (b);
 \draw[{Hooks[right]->}] (c) -- (e);
 \draw[->] (e) -- node {$\xy 0;/r.12pc/: (0,0)*{\dcross{i}{i}}; \endxy$} (d);
 \draw[->] (f) -- node {$\mathsf{R}_{Y'F_i}$} (g);
 \draw[->>] (h) -- (f);
 \draw[->>] (i) -- (g);
 \draw[->] (h) -- node {$\mathsf{R}_{Y'}$} (j);
 \draw[->] (j) -- node {$f$} (i);
 \draw[->>] (f) -- (a);
 \draw[->>] (g) -- (b);
 \draw[->>] (h) -- (c);
 \draw[->>] (i) -- (d);
 \draw[->>] (j) -- (e);
\end{tikzpicture}
}
\end{equation*}
where $f$ is induced by
\[
R(\alpha_i) \circ R(\alpha_i) = R(2\alpha_i) \xrightarrow{\times \tau_1} R(2\alpha_i) = R(\alpha_i) \circ R(\alpha_i),  
\]
and $\mathsf{R}_{Y'}, \mathsf{R}_{Y'F_i}$ are given in Proposition \ref{prop:cyclotomicR}. 
It suffices to prove that the pentagon in the center commutes. 
It is easy to see that all the squares commute (there are five in total). 
Recall that the homomorphism $\mathsf{R}_{Y'} \colon R(\alpha_i) \circ Y' \to Y' \circ R(\alpha_i)$ is given by $u \boxtimes v \mapsto \tau_1 \cdots \tau_{\height \beta -1} (v\boxtimes u)$, 
and $\mathsf{R}_{Y'F_i} \colon R(\alpha_i) \circ Y'F_i \to Y'F_i \circ R(\alpha_i)$ is given by $(u \boxtimes v) \mapsto \tau_1 \cdots \tau_{\height \beta}(v\boxtimes u)$. 
Therefore, the outer pentagon commutes, which proves the assertion. 

Next, we prove that the bottom face of Diagram (\ref{pic:3}) commutes. 
Note that the composition
\begin{align*}
E_iS_i(X) \circ E_iS_i(Y') \simeq E_iS_i(X) \circ S_i(Y'F_i) &\hookrightarrow E_i(S_i(X) \circ S_i(Y'F_i)) \\
&\simeq E_i(S_i(X) \circ E_iS_i(Y'))
\end{align*}
appearing in the bottom face coincides with the canonical injective homomorphism $E_iS_i(X) \circ E_iS_i(Y') \hookrightarrow E_i(S_i(X) \circ E_iS_i(Y'))$. 
Hence, in order to prove the commutativity of the bottom face, it suffices to verify that the following diagram commutes:
\begin{equation} \label{pic:4}
  \centering
  \begin{tikzpicture}[auto]
   \node (a) at (-2,0) {$E_iM \circ E_iN$}; 
   \node (b) at (2,0) {$E_i(M \circ E_iN)$};
   \node (c) at (-3,2) {$E_i(E_iM \circ N)$};
   \node (d) at (3,2) {$E_iE_i(M \circ N)$};
   \node (e) at (0,3) {$E_iE_i(M \circ N)$};
   \draw[{Hooks[right]}->] (a) -- node {$b'$} (b);
   \draw[->>] (c) --node {$a'$} (a);
   \draw[->>] (d) -- node {$c$} (b);
   \draw[{Hooks[right]->}] (c) -- node {$a$} (e);
   \draw[->] (e) -- node {$b \coloneq \xy 0;/r.12pc/: (0,0)*{\ucross{i}{i}}; \endxy$} (d);
  \end{tikzpicture}
  \end{equation}
where $M \in \gMod{R_i(\alpha')}, N \in \gMod{R_i(\beta')}, \alpha',\beta' \in \mathsf{Q}_+$. 
Let $m=\height \alpha', n=\height \beta'$. 
If $m \leq 1$, we have $E_iM = 0$ since $R_i(\alpha_i) = 0$. Hence, the assertion is obvious in this case. 
Assume $m \geq 2$. 
Note that we have 
\[
E_iM \circ N = \bigoplus_{w \in \mathfrak{S}_{m+n-1}^{m-1,n}} \tau_w(E_iM \boxtimes N). 
\]
Let $w \in \mathfrak{S}_{m+n-1}^{m-1,n}, u \in M, v \in N$. 
It suffices to prove 
\begin{equation} \label{eq:1}
cba(E_i\tau_w(E_iu \boxtimes v)) = b'a'(E_i \tau_w(E_iu \boxtimes v)). 
\end{equation}

The left hand side is 
\begin{align}
cb(E_i\tau_w E_i(u\boxtimes v)) &=cb (E_i E_i (e(i) \boxtimes \tau_w) (u \boxtimes v)) \notag \\ 
&= c(E_iE_i \tau_1(e(i) \boxtimes \tau_w)(u \boxtimes v)). \label{eq:2}
\end{align}
If $w(1) = 1$, we have 
\[
E_i\tau_1(e(i) \boxtimes \tau_w)(u\boxtimes v) = \tau_w E_i (\tau_1 u \boxtimes v) = \tau_wE_i(e(i,*)\tau_1u \boxtimes v) \in E_i(M \circ N).
\]
Here, we used the assumption $m \geq 2$.
Hence, the image of this element under the canonical surjection $E_i(M \circ N) \to M \circ E_iN$ is zero. 
It follows that the left hand side of (\ref{eq:2}) is zero in this case.  
Assume $w (1) \neq 1$. 
Since $w \in \mathfrak{S}_{m+n-1}^{m-1,n}$, we have $w(m) = 1$. 
Hence, there exists $x \in \mathfrak{S}_{m+n-2}^{m-1,n-1}$ such that $w = (e_1 \star x)s_1 \cdots s_{m-1}$ (see Section \ref{sec:notation} for the notation). 
We have $s_1(e_1 \star w) = s_1(e_2 \star x)s_2 \cdots s_m = (e_2 \star x)s_1 \cdots s_m$.  
Note that this is an element of $\mathfrak{S}_{m+n}^{m,n}$, and is of length $\ell(w) + 1 = \ell(x) + m$. 
Hence, Lemma \ref{lem:tau} shows that $e(i,i,*)\tau_1 (e(i) \boxtimes \tau_w) = (e(i,i) \boxtimes \tau_x)\tau_1 \cdots \tau_m$ in $R(\alpha' + \beta')$.  
Therefore, (\ref{eq:2}) is 
\begin{align*}
c(E_i(e(i) \boxtimes \tau_x) E_i\tau_1\cdots \tau_m(u\boxtimes v)) &= E_i(e(i) \boxtimes \tau_x) (u \boxtimes E_iv) \\
&= \tau_x E_i( u \boxtimes E_iv). 
\end{align*}

Next, we compute the right hand side of (\ref{eq:1}). 
If $w (1) = 1$, there exists $w' \in \mathfrak{S}_{m+n-2}^{m-1,n-1}$ such that $w = e_1 \star w'$. 
We have 
\[
a'(E_i\tau_w(E_iu \boxtimes v)) = a'(\tau_{w'}E_i(E_iu \boxtimes v)) = \tau_{w'}a'(E_i(E_iu \boxtimes v)) = 0.
\] 
Therefore, the right hand side of (\ref{eq:1}) is zero in this case. 
Assume $w (1) \neq 1$. 
Then $w(m) = 1$, and we can write $w = (e_1 \star x) s_1 \cdots s_{m-1}$ as in the previous paragraph. 
Hence, the right hand side is 
\begin{align*}
b'a'(E_i(e(i) \boxtimes \tau_x)\tau_1 \cdots \tau_{m-1}(E_i u\boxtimes v)) &= b'a' (\tau_x E_i\tau_1 \cdots \tau_{m-1}(E_i u \boxtimes v)) \\
&= b'(\tau_x(E_i u\boxtimes E_i v)) \\
&= \tau_x E_i(u \boxtimes E_iv). 
\end{align*}
Therefore, Diagram (\ref{pic:4}) commutes. 

Case 2. $j \neq i$. 
Consider the following diagram: 
\begin{equation} \label{pic:5}
\centering 
\scalebox{0.65}{
\begin{tikzpicture}[auto]
  \node (b) at (-2,-2) {$E_iS_i(X) \circ S_i(Y' F_j)$};
  \node (c) at  (2,-2) {$E_i(S_i(X) \circ S_i(Y'F_j))$};
  \node (f) at (-2,0) {$S_i(XF_i)\circ S_i(Y'F_j)$};
  \node (g) at (2,2) {$E_iS_i(X \circ Y'F_j)$};
  \node (h) at (-2,4) {$S_i(XF_i \circ Y'F_j)$};
  \node (i) at (2,4) {$S_i((X \circ Y'F_j)F_i)$};
  \node (j) at (-8,4) {$E_iS_i(X) \circ S_i(Y') \circ M_j$};
  \node (k) at  (-8,6) {$S_i(X F_i) \circ S_i(Y') \circ M_j$};
  \node (l) at (-8,8) {$S_i(X F_i \circ Y') \circ M_j$};
  \node (m) at  (-8,10) {$S_i((XF_i \circ Y') F_j)$};
  \node (n) at (8,4) {$E_i(S_i(X) \circ S_i(Y') \circ M_j)$};
  \node (o) at (8,6) {$E_i(S_i(X \circ Y') \circ M_j)$};
  \node (p) at (8,8) {$E_iS_i((X \circ Y')F_j)$};
  \node (q) at (8,10) {$S_i((X \circ Y')F_jF_i)$};
  \node (r) at (0,8) {$E_i(S_i(X) \circ S_i(Y')) \circ M_j$};
  \node (s) at (0,10) {$E_iS_i(X \circ Y') \circ M_j$};
  \node (t) at (0,12) {$S_i((X \circ Y')F_i)\circ M_j$};
  \node (u) at (0,14) {$S_i((X \circ Y')F_iF_j)$};
  \draw[{Hooks[right]}->, line width = 1.5pt] (b) -- (c);
 \draw[line width = 1.5pt] (b) -- node[sloped] {$\scriptstyle \sim$} (f);
 \draw[->,line width = 1.5pt] (c) -- node {$\scriptstyle \theta$} (g);
 \draw[->,line width = 1.5pt] (f) -- node {$\scriptstyle \theta$} (h);
 \draw[line width = 1.5pt] (g) -- node[sloped] {$\scriptstyle \sim$} (i);
 \draw[->,line width = 1.5pt] (h) -- (i);
 \draw (j) -- node[sloped] {$\scriptstyle \sim$} (k);
 \draw[->] (k) -- node {$\scriptstyle \theta$} (l);
 \draw (l) -- node[sloped] {$\scriptstyle \sim$} (m);
 \draw (j) -- node[sloped] {$\scriptstyle \sim$} (b);
 \draw[->] (n) -- node {$\scriptstyle \theta$} (o);
 \draw (o) -- node[sloped] {$\scriptstyle \sim$} (p);
 \draw (p) -- node[sloped] {$\scriptstyle \sim$} (q);
 \draw (n) -- node[sloped] {$\scriptstyle \sim$} (c);
 \draw[->] (r) -- node {$\scriptstyle \theta$} (s);
 \draw (s) -- node[sloped] {$\scriptstyle \sim$} (t);
 \draw (t) -- node[sloped] {$\scriptstyle \sim$} (u);
 \draw[->] (m) -- (u);
 \draw[->] (l) -- (t);
 \draw[{Hooks[right]}->] (j) -- (r);
 \draw[{Hooks[right]}->] (r) -- node[pos=0.2] {$\sigma_{i,j}$} (n);
 \draw[{Hooks[right]}->] (s) -- node[pos=0.3] {$\sigma_{i,j}$} (o);
 \draw[->] (u) -- node {$\xy 0;/r.12pc/: (0,0)*{\dcross{i}{j}}; \endxy$} (q);
 \draw[cross] (k) -- node[sloped] {$\scriptstyle \sim$} (f);
 \draw[cross] (m) -- node[sloped] {$\scriptstyle \sim$} (h);
 \draw[cross] (p) -- node[sloped] {$\scriptstyle \sim$} (g);
 \draw[cross] (q) -- node[sloped] {$\scriptstyle \sim$} (i);
\end{tikzpicture}
}
\end{equation}
We need to prove the thick face commutes. 
It is easy to see that all the other side faces (there are eight in total) and the bottom face are commutative.
The top face also commutes by the same reasoning as in Case 1. 
Hence, the assertion follows. 

Now, we have obtained a family of natural isomorphisms $\theta_{\alpha,\beta} \ (\alpha, \beta \in \mathsf{Q}_+)$ that makes Diagrams (2) -- (5) commutative. 

\subsubsection{Commutativity of (1)}
Let $\alpha \in \mathsf{Q}_+$ and $X \in \gMod{{}_iR(\alpha)}$. 
We prove that Diagram (1) commutes by induction on $\height \alpha$. 

First, assume $\alpha = 0$. 
Since $R_i(0) = \mathbf{k}$, we may assume $X = \mathbf{1}$. 
Then, Diagram (1) coincides with Diagram (2), and the assertion follows. 

Next, assume $\alpha \neq 0$. 
We may assume $X = X' F_j$ for some $j \in I$ and $X' \in \gMod{{}_iR(\alpha-\alpha_j)}$.

Case 1. $j = i$. 
Consider the following diagram: 
\begin{equation*}
\begin{tikzcd}
& \mathbf{1}  \circ S_i(X'F_i) \arrow[r,-,"\sim",line width=1.3pt] \arrow[d,equal,line width=1.3pt]\arrow[ldddd,bend right = 30,-,"\sim" sloped] & S_i(X'F_i) \arrow[d,-,"\sim"sloped,line width=1.3pt]\arrow[dd,bend left= 60,-,"\sim"sloped] \arrow[dddd,bend left= 80,-,"\sim"sloped] \\
& S_i(\mathbf{1}) \circ S_i(X'F_i) \arrow[r,"\theta",line width=1.3pt]\arrow[d,-,"\sim"sloped] & S_i(\mathbf{1} \circ X'F_i) \\
& S_i(\mathbf{1}) \circ E_iS_i(X') \arrow[r,phantom,"(4)"] & S_i((\mathbf{1}\circ X')F_i) \arrow[u,twoheadrightarrow]\arrow[d,-,"\sim"sloped] \\
& E_i(S_i(\mathbf{1}) \circ S_i(X')) \arrow[u,twoheadrightarrow]\arrow[r,"\theta"]\arrow[d,equal] & E_iS_i(\mathbf{1} \circ X')\arrow[d,-,"\sim"sloped] \\
\mathbf{1} \circ E_iS_i(X') \arrow[ruu,bend left = 20,equal]\arrow[rr,bend right = 30,-,"\sim"] & E_i(\mathbf{1} \circ S_i(X')) \arrow[r,-,"\sim"]\arrow[l] & E_iS_i(X')
\end{tikzcd}
\end{equation*}
We need to prove that the thick diagram commutes. 
Using the induction hypothesis, it is easy to see that all the other inner diagrams and the outer diagram commute. 
Hence, the assertion follows. 

Case 2. $j \neq i$. 
Consider the following diagram: 
\begin{equation*}
\begin{tikzcd}
\mathbf{1} \circ S_i(X'F_j) \arrow[r,-,"\sim",line width=1.3pt]\arrow[ddd,bend right = 70,"\sim" sloped]\arrow[d,equal,line width=1.3pt] & S_i(X'F_j)\arrow[d,-,"\sim"sloped,line width=1.3pt]\arrow[dd,-,"\sim" sloped, bend left=70]\arrow[rddd,-,"\sim"sloped,bend left=35] & \\
S_i(\mathbf{1}) \circ S_i(X'F_j) \arrow[r,"\theta",line width=1.3pt]\arrow[d,-,"\sim"sloped] & S_i(\mathbf{1} \circ S'F_j)\arrow[d,-,"\sim"sloped] & \\
S_i(\mathbf{1}) \circ S_i(X') \circ M_j\arrow[d,equal]\arrow[rd,"\theta"]\arrow[ru,phantom,"(5)"] & S_i(\mathbf{1} \circ X')F_j \arrow[d,-,"\sim"sloped] & \\
\mathbf{1} \circ S_i(X') \circ M_j\arrow[rr,bend right=20,-,"\sim"sloped] & S_i(\mathbf{1} \circ X') \circ M_j \arrow[r,-,"\sim"sloped] & S_i(X') \circ M_j
\end{tikzcd}
\end{equation*}
We need to prove that the thick diagram commutes.
Using the induction hypothesis, it is easy to see that all the other inner diagrams and the outer diagram commute. 
Hence, the assertion follows. 

\subsubsection{Commutativity of (6)} 
Let $\alpha, \beta, \gamma \in \mathsf{Q}_+, X \in \gMod{{}_iR(\alpha)}, Y \in \gMod{{}_iR(\beta)}, Z \in \gMod{{}_iR(\gamma)}$. 
We prove that Diagram (6) commutes by induction on $\height \gamma$. 

First, assume $\gamma = 0$. 
Then, ${}_iR(\gamma) = \mathbf{k} = \mathbf{1}$ and we may assume $Z = \mathbf{1}$. 
The assertion follows from the commutative diagram below.  
\begin{equation*}
\centering
\begin{tikzcd}
S_i(X) \circ S_i(Y) \circ S_i(\mathbf{1}) \arrow[rrr,"\theta"]\arrow[ddd,"\theta"] &[-1cm] & &[-1cm] S_i(X \circ Y) \circ S_i(\mathbf{1}) \arrow[ddd,"\theta"] \\
\arrow[rd, phantom,"(2)"] & S_i(X) \circ S_i(Y) \circ \mathbf{1} \arrow[r,"\theta"] \arrow[d,-,"\sim"sloped]\arrow[lu,equal] & S_i(X \circ Y) \circ \mathbf{1} \arrow[d,-,"\sim"sloped]\arrow[ru,equal] & \arrow[ld,phantom,"(2)"] \\
& S_i(X)\circ S_i(Y) \arrow[r,"\theta"]\arrow[ld,-,"\sim" sloped] & S_i(X \circ Y) \arrow[rd,-,"\sim" sloped] & \\
S_i(X) \circ S_i(Y \circ \mathbf{1}) \arrow[rrr,"\theta"] & & & S_i(X \circ Y \circ \mathbf{1}) \\
\end{tikzcd}
\end{equation*}
 
Next, assume $\gamma \neq 0$. 
We may assume $Z = Z'F_j$ for some $j \in I$ and $Z' \in \gMod{{}_iR(\gamma - \alpha_j)}$. 

Case 1. $j = i$. 
It follows from the commutative diagram below: 
\begin{equation*}
\centering
\scalebox{0.6}{
\begin{tikzpicture}[auto]
\node (a1) at (-7,3) {$S_i(X) \circ S_i(Y) \circ S_i(Z'F_i)$};
\node (a2) at (-5,2) {$S_i(X)\circ S_i(Y) \circ E_iS_i(Z')$};
\node (a3) at (-3.3,1) {$S_i(X) \circ E_i(S_i(Y) \circ S_i(Z'))$};
\node (a4) at (-1,0) {$E_i(S_i(X) \circ S_i(Y) \circ S_i(Z'))$};
\node (b1) at (10,3) {$S_i(X \circ Y) \circ S_i(Z'F_i)$};
\node (b2) at (8,2) {$S_i(X \circ Y) \circ E_iS_i(Z')$};
\node (b4) at (4,0) {$E_i(S_i(X \circ Y)\circ S_i(Z'))$};
\node (c1) at (-7,-4) {$S_i(X) \circ S_i(Y \circ Z'F_i)$};
\node (c2) at (-5,-3) {$S_i(X) \circ S_i((Y\circ Z')F_i)$};
\node (c3) at (-3.3,-2) {$S_i(X) \circ E_iS_i(Y \circ Z')$};
\node (c4) at (-1,-1) {$E_i(S_i(X) \circ S_i(Y \circ Z'))$};
\node (d1) at (10,-4) {$S_i(X \circ Y \circ Z'F_i)$};
\node (d2) at (8,-3) {$S_i(X \circ (Y \circ Z')F_i)$};
\node (d3) at (6,-2) {$S_i((X \circ Y \circ Z')F_i)$};
\node (d4) at (4,-1) {$E_iS_i(X \circ Y \circ Z')$};
\draw[->] (a1) -- node {$\scriptstyle \theta$} (b1); 
\draw[->] (a1) -- node {$\scriptstyle \theta$} (c1); 
\draw[->] (b1) -- node {$\scriptstyle \theta$} (d1); 
\draw[->] (c1) -- node {$\scriptstyle \theta$} (d1); 
\draw[->] (a2) -- node {$\scriptstyle \theta$} (b2); 
\draw[->] (a4) -- node {$\scriptstyle \theta$} (b4); 
\draw[->] (a3) -- node {$\scriptstyle \theta$} (c3); 
\draw[->] (a4) -- node {$\scriptstyle \theta$} (c4); 
\draw[->] (b4) -- node {$\scriptstyle \theta$} (d4); 
\draw[->] (c2) -- node {$\scriptstyle \theta$} (d2); 
\draw[->] (c4) -- node {$\scriptstyle \theta$} (d4); 
\draw[->>] (a4) -- (a3);
\draw[->>] (a3) -- (a2);
\draw (a2) -- node[sloped] {$\scriptstyle \sim$} (a1);
\draw[->>] (b4) -- (b2);
\draw (b2) -- node[sloped] {$\scriptstyle \sim$} (b1);
\draw[->>] (c4) -- (c3);
\draw (c3) -- node[sloped] {$\scriptstyle \sim$} (c2);
\draw[->>] (c2) -- (c1);
\draw (d4) -- node[sloped] {$\scriptstyle \sim$} (d3);
\draw[->>] (d3) -- (d2);
\draw[->>] (d2) -- (d1); 
\draw (-5,-0.5) node {(4)};
\draw (1.5,1) node {(4)};
\draw (1.5,-2) node {(4)};
\draw (7,-0.5) node {(4)};
\end{tikzpicture} 
}
\end{equation*}
Note that the central square commutes by the induction hypothesis. 

Case 2. $j \neq i$. 
It follows from the commutative diagram below: 
\begin{equation*}
\centering
\scalebox{0.6}{
\begin{tikzpicture}[auto]
\node (a1) at (-7,3) {$S_i(X) \circ S_i(Y) \circ S_i(Z'F_j)$};
\node (a4) at (-1,0) {$S_i(X) \circ S_i(Y) \circ S_i(Z') \circ M_j$};
\node (b1) at (10,3) {$S_i(X \circ Y) \circ S_i(Z'F_j)$};
\node (b4) at (4,0) {$S_i(X \circ Y)\circ S_i(Z') \circ M_j$};
\node (c1) at (-7,-4) {$S_i(X) \circ S_i(Y \circ Z'F_j)$};
\node (c2) at (-5,-3) {$S_i(X) \circ S_i((Y\circ Z')F_j)$};
\node (c4) at (-1,-1) {$S_i(X) \circ S_i(Y \circ Z')\circ M_j$};
\node (d1) at (10,-4) {$S_i(X \circ Y \circ Z'F_j)$};
\node (d2) at (8,-3) {$S_i(X \circ (Y \circ Z')F_j)$};
\node (d3) at (6,-2) {$S_i((X \circ Y \circ Z')F_j)$};
\node (d4) at (4,-1) {$S_i(X \circ Y \circ Z') \circ M_j$};
\draw[->] (a1) -- node {$\scriptstyle \theta$} (b1);
\draw[->] (a4) -- node {$\scriptstyle \theta$} (b4);
\draw[->] (a1) -- node {$\scriptstyle \theta$} (c1);
\draw[->] (a4) -- node {$\scriptstyle \theta$} (c4);
\draw[->] (b1) -- node {$\scriptstyle \theta$} (d1);
\draw[->] (b4) -- node {$\scriptstyle \theta$} (d4);
\draw[->] (c1) -- node {$\scriptstyle \theta$} (d1);
\draw[->] (c2) -- node {$\scriptstyle \theta$} (d2);
\draw[->] (c4) -- node {$\scriptstyle \theta$} (d4);
\draw (a1) -- node[sloped] {$\scriptstyle \sim$} (a4);
\draw (b1) -- node[sloped] {$\scriptstyle \sim$} (b4);
\draw (c1) -- node[sloped] {$\scriptstyle \sim$} (c2);
\draw (c2) -- node[sloped] {$\scriptstyle \sim$} (c4);
\draw (d1) -- node[sloped] {$\scriptstyle \sim$} (d2);
\draw (d2) -- node[sloped] {$\scriptstyle \sim$} (d3);
\draw (d3) -- node[sloped] {$\scriptstyle \sim$} (d4);
\draw (-4,-0.5) node {(5)};
\draw (1.5,-2) node {(5)};
\draw (7,-0.5) node {(5)};
\end{tikzpicture}
}
\end{equation*}
Note that the central square commutes by the induction hypothesis. 

Now, the proof of Proposition \ref{prop:monoidality} is complete. 

\subsubsection{Additional commutative diagram} 

\begin{proposition} \label{prop:additionalcd}
Let $j \in I \setminus \{i\}$ and $X \in \gMod{{}_iR}$. 
The following diagram commutes: 
\begin{equation*}
\begin{tikzcd}
S_i(X) \circ S_i(R(\alpha_j)) \arrow[r,"\theta"]\arrow[d,-,"\sim"sloped] & S_i(X \circ R(\alpha_j)) \arrow[d,-,"\sim"sloped] \\
S_i(X) \circ M_j \arrow[r,-,"\sim"] & S_i(XF_j)
\end{tikzcd}
\end{equation*}
where the left vertical arrow is the morphism obtained by
\[
S_i(R(\alpha_j)) \simeq S_i(\mathbf{1}F_j) \simeq S_i(\mathbf{1})\circ M_j = \mathbf{1} \circ M_j \simeq M_j. 
\]
\end{proposition}

\begin{proof}
It follows from the commutative diagram below: 
\begin{equation*}
\begin{tikzcd}
S_i(X) \circ S_i(R(\alpha_j)) \arrow[r,"\theta"]\arrow[d,-,"\sim"sloped] & S_i(X \circ R(\alpha_j)) \arrow[d,-,"\sim"sloped] & \\
S_i(X) \circ S_i(\mathbf{1}F_j) \arrow[r,"\theta"]\arrow[d,-,"\sim"sloped]\arrow[rd,phantom,"(5)"] & S_i(X \circ \mathbf{1}F_j) \arrow[rd,-,"\sim"sloped] & \\
S_i(X) \circ S_i(\mathbf{1}) \circ M_j \arrow[r,"\theta"]\arrow[d,equal]\arrow[rd,phantom,"(2)"] & S_i(X \circ \mathbf{1}) \circ M_j \arrow[r,-,"\sim"]\arrow[d,-,"\sim"sloped] & S_i((X \circ \mathbf{1}) F_j) \arrow[d,-,"\sim"sloped] \\
S_i(X) \circ \mathbf{1} \circ M_j \arrow[r,-,"\sim"] & S_i(X) \circ M_j \arrow[r,-,"\sim"] & S_i(X F_j) 
\end{tikzcd}
\end{equation*}
\end{proof}

\subsubsection{Monoidality of $S_i'$}
Note that we have $S'_i(\mathbf{1}) = \mathbf{1}$. 
The following proposition shows the monoidality of $S_i'$:  

\begin{proposition} \label{prop:monoidality*}
There exist natural isomorphisms
\[
\theta'(X,Y) = \theta'_{\alpha,\beta}(X,Y) \colon S_i'(X) \circ S_i'(Y) \to S_i'(X \circ Y), 
\]
for $\alpha,\beta \in \mathsf{Q}_+, X \in \gMod{R_i(\alpha)}, Y \in \gMod{R_i(\beta)}$
that make the following diagrams commutative \ ($j \in I \setminus \{i\}$): 
\[
\begin{tikzcd}
\mathbf{1} \circ S_i'(X) \arrow[r,-,"\sim"] \arrow[d,"\epsilon'"]\arrow[rd,phantom,"(1)'"] & S_i'(X) \arrow[d,-,"\sim" sloped] &  S_i'(X) \circ \mathbf{1} \arrow[r,-,"\sim"]\arrow[d,"\epsilon'"]\arrow[rd,phantom,"(2)'"] & S_i'(X) \arrow[d,-,"\sim" sloped] \\
S_i'(\mathbf{1}) \circ S_i'(X) \arrow[r,"\theta'"] & S_i'(\mathbf{1} \circ X)  &  S_i'(X) \circ S_i'(\mathbf{1}) \arrow[r,"\theta'"] & S_i'(X \circ \mathbf{1}) 
\end{tikzcd}
\]
\[
\begin{tikzcd}
S_i' (X \circ F_iY) \arrow[r] & S_i'(F_i(X\circ Y)) \arrow[r] \arrow[d, -, "\sim" sloped] & S_i'(F_iX \circ Y)  \\
S_i' (X) \circ S_i'(F_iY) \arrow[d,-,"\sim" sloped]\arrow[u,"\theta'"]\arrow[r,phantom,"(3)'"] & S_i'(X \circ Y)E_i \arrow[r,phantom,"(4)'"]  & S_i'(F_iX) \circ S_i'(Y) \arrow[d,-,"\sim" sloped]\arrow[u,"\theta'"] \\
S_i(X) \circ S_i'(Y)E_i \arrow[r, hookrightarrow] & (S_i'(X) \circ S_i'(Y))E_i \arrow[r,twoheadrightarrow]\arrow[u,"\theta'"] & S_i'(X)E_i \circ S_i'(Y) 
\end{tikzcd}
\]
\[
\begin{tikzcd}
S_i'(F_jX) \circ S_i'(Y) \arrow[r,"\theta'"]\arrow[d,-,"\sim" sloped]\arrow[rd,phantom,"(5)'"] & S_i'(F_jX \circ Y) \arrow[rd,-,"\sim" sloped] \\
M'_j \circ S_i'(X) \circ S_i'(Y) \arrow[r,"\theta'"] & M'_j \circ S_i'(X \circ Y) \arrow[r,-,"\sim"] & S_i'(F_j(X \circ Y))  
\end{tikzcd}
\]
\[
\begin{tikzcd}
S_i'(X) \circ S_i'(Y) \circ S_i'(Z) \arrow[r,"\theta'"]\arrow[d,"\theta'"]\arrow[rd,phantom,"(6)'"] & S_i'(X \circ Y) \circ S_i'(Z) \arrow[d,"\theta'"] \\
S_i'(X) \circ S_i'(Y \circ Z) \arrow[r,"\theta'"] & S_i'(X \circ Y \circ Z). 
\end{tikzcd}
\]
\end{proposition}

\begin{proof}
By Remark \ref{rem:LRchange2}, it is equivalent to Proposition \ref{prop:monoidality}. 
\end{proof}

\subsection{Quasi-inverse}
In this section, we complete the proof of Theorem \ref{thm:reflectionfunctor}: 
two functors $S_i$ and $S'_i$ are mutually quasi-inverse. 
The following proposition is the key. 

\begin{proposition} \label{prop:linear}
The functor $S_i \colon \gMod{{}_iR} \to \gMod{R_i}$ is a morphism of left $\dotcatquantum{\mathfrak{p}_i}$-modules: 
there exits a family of natural isomorphisms 
\[
\kappa_j^- = \kappa_{j,\beta}^-(X) \colon S_i(F_jX) \to F_jS_i(X), \kappa_i^+ = \kappa_{i,\beta}^+(X) \colon S_i(E_iX) \to E_iS_i(X), 
\] \index{$\kappa_j^-, \kappa_i^+$}
for $j \in I, \beta \in \mathsf{Q}_+, X \in \gMod{{}_iR(\beta)}$ that commute with the left action of 2-morphisms of $\dotcatquantum{\mathfrak{p}_i}$. 
Similarly, the functor $S_i' \colon \gMod{R_i} \to \gMod{{}_iR}$ is a morphism of right $\dotcatquantum{\mathfrak{p}_i}$-modules. 
\end{proposition}

Once the proposition is established, Theorem \ref{thm:reflectionfunctor} is deduced as follows. 

\begin{proof}[Proof of Theorem \ref{thm:reflectionfunctor}]
By Proposition \ref{prop:linear}, the endofunctor $S_i S_i'$ of $\gMod{R_i}$ is left $\catquantum{\mathfrak{p}_i}$-linear. 
In addition, we have $S_iS_i'(\mathbf{1}) = \mathbf{1}$ and $S_iS_i'(f) = f$ for any endomorphism $f$ of $\mathbf{1}$. 
Since $\gproj{R_i}$ is generated by $\mathbf{1}$ as a left $\dotcatquantum{\mathfrak{p}_i}$-module (Theorem \ref{thm:leftuniversality}), 
$S_i S_i'$ is naturally isomorphic to the identity functor. 
Similarly, $S_i' S_i$ is naturally isomorphic to the identity functor, which completes the proof. 
\end{proof}

The rest of this section is devoted to the proof of Proposition \ref{prop:linear}. 
We only address the assertion for $S_i$, since the one for $S_i'$ is equivalent to it by Remark \ref{rem:LRchange2}. 

\subsubsection{Construction of natural isomorphisms}
We define the natural isomorphisms $\kappa_j^- \ (j \in I)$ and $\kappa_i^+$ in the statement of Proposition \ref{prop:linear}. 
Let $\beta \in \mathsf{Q}_+, X \in \gMod{{}_iR(\beta)}$. 
By Theorem \ref{thm:anotheraction}, we have natural isomorphisms
\begin{align*}
S_i(F_iX) &= S_i(X E_i) \simeq S_i(X)E_i \simeq F_iS_i(X), \\
S_i(E_iX) &= S_i(X F_i) \simeq S_i(X) F_i \simeq E_iS_i(X). 
\end{align*}
We define $\kappa_i^- \colon S_i(F_iX) \to F_iS_i(X)$ and $\kappa_i^+ \colon S_i(E_iX) \to E_iS_i(X)$ as the compositions above. 

Let $j \neq i$. 
Note that, under the equality $MF_i^n = E_i^n M$ for $M \in \gMod{R_i(\beta)}$, we have 
\[
M \ \xy 0;/r.12pc/:
(0,10); (0,-10) **\dir{-} ?(1)*\dir{>};
(0,0)*{\fcolorbox{black}{white}{$b'_-(i^n)$}};
(0,-13)*{\scriptstyle i^n};
\endxy = \xy 0;/r.12pc/:
(0,10); (0,-10) **\dir{-} ?(0)*\dir{<};
(0,0)*{\fcolorbox{black}{white}{$b'_-(i^n)$}};
(0,-13)*{\scriptstyle i^n};
\endxy \ M. 
\]
It yields a natural isomorphism $MF_i^{(n)'} \simeq q_i^{n(n-1)/2} b'_-(i^n) E_i^n M$. 
Combined with Lemma \ref{lem:projisom}, we obtain an isomorphism 
\[
MF_i^{(n)'} \simeq q_i^{n(n-1)/2}b'_-(i^n)E_i^n M \simeq q_i^{-n(n-1)/2} b_+(i^n)E_i^n M = E_i^{(n)}M. 
\]

\begin{lemma} \label{lem:directsummand}
The canonical morphisms $MF_i^n \to MF_i^{(n)'}$ (resp. $MF_i^{(n)'} \to MF_i^n$) and $E_i^nM \to E_i^{(n)}M$ (resp. $E_i^{(n)}M \to E_i^nM$) coincide under the isomorphisms above. 
\end{lemma}

\begin{proof}
It follows from the definition. 
\end{proof}

We define an isomorphism $\zeta_j \colon S_i(M'_j) \to R(\alpha_j)$ as \index{$\zeta_j$}
\begin{align*}
S_i(M'_j) &= S_i(R(\alpha_j)F_i^{(-a_{i,j})'}) \simeq S_i(R(\alpha_j))F_i^{(-a_{i,j})'} \simeq M_j F_i^{(-a_{i,j})'} \\ 
&\simeq E_i^{(-a_{i,j})}M_j \quad \text{the isomorphism described above} \\
&=E_i^{(-a_{i,j})}F_i^{(-a_{i,j})} R(\alpha_j) \xrightarrow{\eta^{-1}} R(\alpha_j),
\end{align*} 
where $\iota^{-1}$ is an isomorphism by Proposition \ref{prop:extremalequiv}. 
Let $\varrho \colon \mathsf{Q} \to \mathbb{Z}$ be the group homomorphism defined by $\varrho(\alpha_k) = \delta_{k,i}$.
We define a natural isomorphism 
\[
\kappa_j^- \colon S_i(F_jX) \to F_jS_i(X)
\]
as the $t_{i,j}^{\varrho(\beta)}$-multiple of the composition
\[
S_i(F_jX) \simeq S_i(M'_j \circ X) \xrightarrow{\theta^{-1}} S_i(M'_j) \circ S_i(X) \xrightarrow{\zeta_j} R(\alpha_j) \circ S_i(X) \simeq F_jS_i(X). 
\]

\subsubsection{Naturality}

We prove that the natural isomorphisms $\kappa_j^- \colon S_i(F_jX) \to F_jS_i(X)$ and $\kappa_i^+ \colon S_i(E_iX) \to E_iS_i(X)$ commute with the action of generating 2-morphisms of $\catquantum{\mathfrak{p}_i}$. 

Under the identification $S_i(F_iX) \simeq S_i(XE_i) \simeq S_i(X)E_i \simeq F_iS_i(X)$ and $S_i(E_iX) \simeq S_i(XF_i) \simeq S_i(X)F_i \simeq E_iS_i(X)$, we have 
\begin{align} \label{eq:3}
S_i\left(\xy 0;/r.12pc/: (0,0)*{\sdotd{i}}; \endxy X \right) &= S_i\left(X \xy 0;/r.12pc/: (0,0)*{\sdotu{i}}; \endxy \right) = S_i(X) \xy 0;/r.12pc/: (0,0)*{\sdotu{i}}; \endxy = \xy 0;/r.12pc/: (0,0)*{\sdotd{i}}; \endxy S_i(X), \\
S_i\left(\xy 0;/r.12pc/: (0,0)*{\sdotu{i}}; \endxy X\right) &= S_i\left(X \xy 0;/r.12pc/: (0,0)*{\sdotd{i}}; \endxy\right) = S_i(X) \xy 0;/r.12pc/: (0,0)*{\sdotd{i}}; \endxy = \xy 0;/r.12pc/: (0,0)*{\sdotu{i}}; \endxy S_i(X), \notag \\
S_i\left(\xy 0;/r.12pc/: (0,0)*{\dcross{i}{i}}; \endxy X\right) &= S_i \left(X \xy 0;/r.12pc/: (0,0)*{\ucross{i}{i}}; \endxy \right) = S_i(X) \xy 0;/r.12pc/: (0,0)*{\ucross{i}{i}}; \endxy = \xy 0;/r.12pc/: (0,0)*{\dcross{i}{i}}; \endxy S_i(X), \notag \\
S_i(\xy 0;/r.12pc/: (0,0)*{\ucross{i}{i}}; \endxy X) &= S_i\left(X \xy 0;/r.12pc/: (0,0)*{\dcross{i}{i}}; \endxy \right) = S_i(X)\xy 0;/r.12pc/: (0,0)*{\dcross{i}{i}}; \endxy = \xy 0;/r.12pc/: (0,0)*{\ucross{i}{i}}; \endxy S_i(X), \notag\\
S_i\left(\xy 0;/r.12pc/: (0,0)*{\rcap{i}}; \endxy X\right) &= c_{i,-\beta} S_i\left(X \xy 0;/r.12pc/: (0,0)*{\rcap{i}}; \endxy \right) = c_{i,-\beta}S_i(X) \xy 0;/r.12pc/: (0,0)*{\rcap{i}}; \endxy \notag\\ 
&= c_{i,-\beta}c_{i,-s_i\beta}^{-1}\xy 0;/r.12pc/: (0,0)*{\rcap{i}}; \endxy S_i(X) = \xy 0;/r.12pc/: (0,0)*{\rcap{i}}; \endxy S_i(X), \notag\\ 
S_i(\xy 0;/r.12pc/: (0,0)*{\rcup{i}}; \endxy X) &= c_{i,-\beta}^{-1} S_i\left( X \xy 0;/r.12pc/: (0,0)*{\rcup{i}}; \endxy \right) = c_{i,-\beta}^{-1} S_i(X) \xy 0;/r.12pc/: (0,0)*{\rcup{i}}; \endxy \notag\\ 
&= c_{i,-\beta}^{-1}c_{i,-s_i\beta} \xy 0;/r.12pc/: (0,0)*{\rcup{i}}; \endxy S_i(X) = \xy 0;/r.12pc/: (0,0)*{\rcup{i}}; \endxy S_i(X), \notag \\
S_i\left(\xy 0;/r.12pc/: (0,0)*{\lcap{i}}; \endxy X\right) &= c_{i,\beta} S_i\left(X \xy 0;/r.12pc/: (0,0)*{\lcap{i}}; \endxy \right) = c_{i,\beta} S_i(X) \xy 0;/r.12pc/: (0,0)*{\lcap{i}}; \endxy \notag\\
&= c_{i,\beta}c_{i,s_i\beta}^{-1} \xy 0;/r.12pc/: (0,0)*{\lcap{i}}; \endxy S_i(X) = \xy 0;/r.12pc/: (0,0)*{\lcap{i}}; \endxy S_i(X), \notag \\ 
S_i\left(\xy 0;/r.12pc/: (0,0)*{\lcup{i}}; \endxy X\right) &= c_{i,\beta}^{-1} S_i\left( X \xy 0;/r.12pc/: (0,0)*{\lcup{i}}; \endxy \right) = c_{i,\beta}^{-1} S_i(X) \xy 0;/r.12pc/: (0,0)*{\lcup{i}}; \endxy \notag\\
&= c_{i,\beta}^{-1}c_{i,s_i\beta} \xy 0;/r.12pc/: (0,0)*{\lcup{i}}; \endxy S_i(X) = \xy 0;/r.12pc/: (0,0)*{\lcup{i}}; \endxy S_i(X). \notag 
\end{align}

\begin{comment}
Be careful of the scalar multiple included in $\kappa$ and $\eta$: 
for instance, we used the fact that the isomorphism $S_i(E_iF_iX) \xrightarrow{\eta} E_iS_i(F_iX) \xrightarrow{\kappa} E_iF_iS_i(X)$ coincides with the canonical isomorphism $S_i(E_iF_iX) \simeq E_iF_iS_i(X)$. 
\end{comment}

For $j \neq i$, the endomorphism $S_i(y'_j)$ of $S_i(M'_j)$ coincides with the endomorphism of $R(\alpha_j)$ given by the multiplication by $x_1$ through the isomorphism $\zeta_j \colon S_i'(M'_j) \to R(\alpha_j)$ by definition.
Hence, we have 
\[
S_i\left(\xy 0;/r.12pc/: (0,0)*{\sdotd{j}}; \endxy X \right) = \xy 0;/r.12pc/: (0,0)*{\sdotd{j}}; \endxy S_i(X). 
\]

From now on, we freely use the canonical isomorphisms
\begin{align*}
&S_i(F_iX) \simeq S_i(XE_i) \simeq S_i(X)E_i \simeq F_iS_i(X), \\
&S_i(E_iX) \simeq S_i(XF_i) \simeq S_i(X) F_i \simeq E_iS_i(X). \\
%&F_j Y \simeq R(\alpha_j) \circ Y. 
\end{align*}
In addition, scalar multiples of these isomorphisms by some $c \in \mathbf{k}$ are denoted simply by $c$: 
for instance, the $c$-multiple of the isomorphism $S_i(XE_i) \to S_i(X)E_i$ is written as $S_i(XE_i) \xrightarrow{c} S_i(X)E_i$. 
The isomorphism 
\[
S_i(M'_j) \circ Y \xrightarrow{\zeta_j} R(\alpha_j) \circ Y \simeq F_jY
\]
is also denoted by $\zeta_j$. 

It remains to deal with the 2-morphisms $\xy 0;/r.12pc/: (0,0)*{\dcross{j}{k}}; \endxy$ for $(j,k) \in I^2 \setminus \{(i,i)\}$.  
The computation splits into three cases, treated in separate sections below. 

\subsubsection{Case 1}
$\xy 0;/r.12pc/: (0,0)*{\dcross{j}{i}}; \endxy$ for $j \neq i$.
Let $X \in \gMod{{}_iR(\beta)}$. 
Consider the following diagram:
\begin{comment}
\begin{equation} 
\begin{tikzcd}
& & F_i(S_i(M'_j \circ X)) \arrow[d,-,"\sim" sloped] \\
S_i((M'_j \circ X)E_i) \arrow[r,-,"\sim" sloped]\arrow[d,-,"\sim" sloped] & S_i(M'_j \circ X)E_i \arrow[ru,-,"\sim" sloped]\arrow[d,-,"\sim" sloped] & F_i(S_i(M'_j) \circ S_i(X)) \arrow[u,"\theta"]\arrow[d,-,"\sim" sloped] \\
S_i((F_jX)E_i) \arrow[r,-,"\sim" sloped]\arrow[d,-,"\sim" sloped] & S_i(F_jX) E_i \arrow[d,-,"\sim" sloped] & F_i(R(\alpha_j)\circ S_i(X)) \arrow[d,-,"\sim" sloped] \\
S_i(F_iF_jX) \arrow[r,"\kappa", color=red] & F_iS_i(F_jX) \arrow[r,"\kappa",color=red]\arrow[ruuu,bend left = 8,-,"\sim" sloped] & F_iF_jS_i(X) \\
S_i(F_jF_iX) \arrow[r,"\kappa",color=red]\arrow[u,"{\xy 0;/r.12pc/: (0,0)*{\dcross{j}{i}}; \endxy}",color=red]\arrow[dd,-,"\sim" sloped] & F_jS_i(F_iX) \arrow[r,"\kappa",color=red]\arrow[d,-,"\sim" sloped] & F_jF_iS_i(X) \arrow[u,"{\xy 0;/r.12pc/: (0,0)*{\dcross{j}{i}}; \endxy}",color=red]\arrow[d,-,"\sim" sloped] \\
 & R(\alpha_j) \circ S_i(F_iX)\arrow[r,"\kappa"]\arrow[d,-,"\sim" sloped] & R(\alpha_j) \circ F_iS_i(X) \arrow[d,-,"\sim" sloped] \\
S_i(M'_j\circ F_iX) \arrow[d,-,"\sim" sloped] & S_i(M'_j) \circ S_i(F_iX) \arrow[l,"\theta"]\arrow[r,"\kappa"]\arrow[d,-,"\sim" sloped] & S_i(M'_j) \circ F_iS_i(X) \arrow[d,-,"\sim" sloped] \\
S_i(M'_j\circ XE_i) \arrow[uuuuuu,bend left = 60,"\sigma'_{j,i}"] & S_i(M'_j) \circ S_i(XE_i) \arrow[l,"\theta"] \arrow[r,-,"\sim" sloped] & S_i(M'_j) \circ S_i(X)E_i \\
\end{tikzcd}
\end{equation}
\end{comment}

\begin{equation} \label{pic:6}
\begin{tikzcd}
& & F_i(S_i(M'_j \circ X)) \arrow[d,-,"\sim" sloped] \\
S_i((M'_j \circ X)E_i) \arrow[r,-,"\sim"]\arrow[d,-,"\sim" sloped] & S_i(M'_j \circ X)E_i \arrow[ru,-,"\sim" sloped]\arrow[d,-,"\sim" sloped] & F_i(S_i(M'_j) \circ S_i(X)) \arrow[u,"\theta"]\arrow[d,"t_{i,j}^{\varrho(\beta)}\zeta_j"] \\
S_i((F_jX)E_i) \arrow[r,-,"\sim"]\arrow[d,-,"\sim" sloped] & S_i(F_jX) E_i \arrow[d,-,"\sim" sloped] & F_i(R(\alpha_j)\circ S_i(X)) \arrow[d,-,"\sim" sloped] \\
S_i(F_iF_jX) \arrow[r,"\kappa_i^-",line width=1pt] & F_iS_i(F_jX) \arrow[r,"\kappa_j^-",line width=1pt]\arrow[ruuu,bend left = 8,-,"\sim" sloped] & F_iF_jS_i(X) \\
S_i(F_jF_iX) \arrow[r,"\kappa_j^-",line width=1pt]\arrow[u,"{\xy 0;/r.12pc/: (0,0)*{\dcross{j}{i}}; \endxy}",line width=1pt]\arrow[d,-,"\sim" sloped] & F_jS_i(F_iX) \arrow[r,"\kappa_i^-",line width=1pt]\arrow[d,-,"\sim" sloped] & F_jF_iS_i(X) \arrow[u,"{\xy 0;/r.12pc/: (0,0)*{\dcross{j}{i}}; \endxy}",line width=1pt]\arrow[d,-,"\sim" sloped] \\
S_i(F_j(XE_i)) \arrow[dd,-,"\sim" sloped]\arrow[r,"\kappa_j^-"] & F_jS_i(XE_i) \arrow[r,-,"\sim"]\arrow[d,-,"\sim" sloped] & F_j(S_i(X)E_i) \\
& R(\alpha_j) \circ S_i(XE_i) & \\
S_i(M'_j \circ XE_i) \arrow[uuuuuu,bend left = 60,"\sigma'_{j,i}"] & S_i(M'_j) \circ S_i(XE_i) \arrow[l,"\theta"]\arrow[u,"t_{i,j}^{\varrho(\beta-\alpha_i)}\zeta_j"] & \\
\end{tikzcd}
\end{equation}

We need to prove that the thick diagram commutes. 
It is easy to see that all the other inner diagrams are commutative. % using $c_{i,-(\beta + s_i\alpha_j)} = c_{i,-(\beta +\alpha_j)}$.  
Hence, it suffices to prove the commutativity of the outer diagram.  
We prove it by induction on $\height \beta$. 

If $\beta = 0$, then $F_iX = XE_i =0$ and the assertion is trivial. 
Assume $\beta \neq 0$. 
Since it is enough to show the assertion for projective modules, we may assume $X = X'F_k$ for some $k \in I$ and $X' \in \gMod{{}_iR(\beta-\alpha_k)}$.

The case $k =i$. 
Consider the diagram of Figure \ref{fig:1} and \ref{fig:2}.  

\begin{figure} %\label{pic:7}
\scalebox{0.95}{
\rotatebox{90}{
\begin{tikzpicture}[auto]
\node (a) at (0,1) {$S_i(M'_j \circ X'F_iE_i)$};
\node (b) at (4,1) {$S_i(M'_j) \circ S_i(X'F_iE_i)$};
\node (c) at (8,1) {$F_j S_i(X'F_iE_i)$};
\node (c') at (12,1) {$F_j F_iS_i(X'F_i)$};
\node (d) at (16,1) {$F_jF_iE_iS_i(X')$};
\node (e) at (0,-0.5) {$S_i((M'_j \circ X'F_i)E_i)$};
\node (f) at (4,-0.5) {$F_iS_i(M'_j \circ X'F_i)$};
\node (g) at (8,-0.5) {$F_i(S_i(M'_j) \circ S_i(X'F_i))$};
\node (g') at (12,-0.5) {$F_iF_j S_i(X'F_i)$}; 
\node (h) at (16,-0.5) {$F_iF_jE_iS_i(X')$};
\node (a1) at (0,2.5) {$S_i(M'_j \circ X')$};
\node (b1) at (4,2.5) {$S_i(M'_j) \circ S_i(X')$};
\node (c1) at (8,2.5) {$F_j S_i(X')$};
\node (d1) at (16,2.5) {$F_jS_i(X')$};
\node (e1) at (0,-2) {$S_i((M'_j\circ X')F_iE_i)$};
\node (f1) at (4,-2) {$F_iS_i((M'_j \circ X')F_i)$};
\node (g1) at (8,-2) {$F_i(S_i(M'_j) \circ E_iS_i(X'))$};
\node (h1) at (16,-2) {$F_iF_jE_iS_i(X')$};
\node (f2) at (4,-3) {$F_iE_iS_i(M'_j \circ X')$};
\node (g2) at (8,-3) {$F_iE_i(S_i(M'_j) \circ S_i(X'))$};
\node (h2) at (16,-3) {$F_iE_iF_jS_i(X')$};
\node (e3) at (0,-4.5) {$S_i(M'_j \circ X')$};
\node (f3) at (4,-4.5) {$S_i(M'_j\circ X')$};
\node (g3) at (8,-4.5) {$S_i(M'_j) \circ S_i(X')$};
\node (h3) at (16,-4.5) {$F_jS_i(X')$};
\draw[<-,line width=1.3pt] (a) -- node {$\scriptstyle \theta$} (b);
\draw[->,line width=1.3pt] (b) -- node {$\scriptstyle t_{i,j}^{\varrho(\beta-\alpha_i)}\zeta_j$} (c);
\draw[line width=1.3pt] (c) -- node {$\scriptstyle \sim$} (c');
\draw (c') -- node {$\scriptstyle \sim$} (d);
\draw[->,line width=1.3pt] (a) -- node {$\scriptstyle \sigma'_{j,i}$} (e);
\draw[->] (d) -- node {$\xy 0;/r.12pc/: (0,0)*{\dcross{j}{i}}; \endxy$} (h);
\draw[->,line width=1.3pt] (c') -- node {$\xy 0;/r.12pc/: (0,0)*{\dcross{j}{i}}; \endxy$} (g');
\draw[line width=1.3pt] (e) -- node {$\scriptstyle \sim$} (f);
\draw[<-,line width=1.3pt] (f) -- node {$\scriptstyle \theta$} (g); 
\draw[->,line width=1.3pt] (g) -- node {$\scriptstyle t_{i,j}^{\varrho(\beta)}\zeta_j$} (g');
\draw (g') -- node {$\scriptstyle \sim$} (h);
\draw[->] (a1) -- node {$\xy 0;/r.12pc/: (0,0)*{\rcup{i}};(-2,-0.5)*{\bullet}; (-6,-1)*{\scriptstyle n}; \endxy$} (a);
\draw[->] (b1) -- node {$\xy 0;/r.12pc/: (0,0)*{\rcup{i}};(-2,-0.5)*{\bullet}; (-6,-1)*{\scriptstyle n}; \endxy$} (b);
\draw[->] (c1) -- node {$\xy 0;/r.12pc/: (0,0)*{\rcup{i}};(-2,-0.5)*{\bullet}; (-6,-1)*{\scriptstyle n}; \endxy$} (c);
\draw[->] (d1) -- node[swap] {$c_{i,-\beta}\xy 0;/r.12pc/: (0,0)*{\rcup{i}}; (-2,-0.5)*{\bullet}; (-6,-1)*{\scriptstyle n};\endxy$} (d);
\draw[<-] (a1) -- node {$\scriptstyle \theta$} (b1);
\draw[->] (b1) -- node {$\scriptstyle t_{i,j}^{\varrho(\beta-\alpha_i)}\zeta_j$} (c1);
\draw[double distance=2pt] (c1) -- (d1); 
\draw[->] (e1) -- node[sloped] {$\scriptstyle \sim$} (e);
\draw[->] (f1) -- node[sloped] {$\scriptstyle \sim$} (f);
\draw (g1) -- node[sloped] {$\scriptstyle \sim$} (g);
\draw (f2) -- node[sloped] {$\scriptstyle \sim$} (f1);
\draw[->] (g2) -- (g1);
\draw[<-] (f2) -- node {$\scriptstyle \theta$} (g2);
\draw[double distance=2pt] (h1) -- (h);
\draw[->] (h2) -- node {$\xy 0;/r.10pc/: (0,0)*{\rcross{i}{j}}; \endxy$} (h1); 
\draw[->] (g2) -- node {$\scriptstyle t_{i,j}^{\varrho(\beta)}\zeta_j$} (h2);
\draw[->] (e3) -- node {$\xy 0;/r.12pc/: (0,0)*{\rcup{i}};(-2,-0.5)*{\bullet}; (-6,-1)*{\scriptstyle n}; \endxy$} (e1); 
\draw[->] (f3) -- node {$c_{i,-(\beta+\alpha_j)} \xy 0;/r.12pc/: (0,0)*{\rcup{i}};(-2,-0.5)*{\bullet}; (-6,-1)*{\scriptstyle n}; \endxy$} (f2);
\draw[->] (g3) -- node {$c_{i,-(\beta+\alpha_j)} \xy 0;/r.12pc/: (0,0)*{\rcup{i}};(-2,-0.5)*{\bullet}; (-6,-1)*{\scriptstyle n}; \endxy$} (g2);
\draw[->] (h3) -- node {$c_{i,-(\beta+\alpha_j)} \xy 0;/r.12pc/: (0,0)*{\rcup{i}};(-2,-0.5)*{\bullet}; (-6,-1)*{\scriptstyle n}; \endxy$} (h2);
\draw[double distance=2pt] (e3) -- (f3);
\draw[<-] (f3) -- node {$\scriptstyle \theta$} (g3);
\draw[->] (g3) -- node {$\scriptstyle t_{i,j}^{\varrho(\beta)}\zeta_j$} (h3); 
\draw[double distance=2pt] (a1) to[bend right=70] (e3);
\draw[->] (d1) to[bend left=70] node {$\scriptstyle t_{i,j} \id$} (h3);
\draw (e1) -- node {$\scriptstyle \sim$} (f1);
\draw[->] (g1) -- node {$\scriptstyle t_{i,j}^{\varrho(\beta)}\zeta_j$} (h1);
\node (A) at (12,1.8) {(A)};
\node (B) at (-1.5,0) {(B)};
\node (C) at (5.7,-2) {(C)};
\node (D) at (2,-3.3) {(D)};
\node (E) at (12,-2.3) {(E)};
\node (F) at (17,-1.2) {(F)}; 
\end{tikzpicture}
}
}
\caption{} \label{fig:1}
\end{figure}

\begin{figure} %\label{pic:8}
\scalebox{0.9}{
\rotatebox{90}{
\begin{tikzpicture}[auto]
\node (a) at (0,1) {$S_i(M'_j \circ X'F_iE_i)$};
\node (b) at (4,1) {$S_i(M'_j) \circ S_i(X'F_iE_i)$};
\node (c) at (8,1) {$F_j S_i(X'F_iE_i)$};
\node (c') at (12,1) {$F_j F_iS_i(X'F_i)$};
\node (d) at (16,1) {$F_jF_iE_iS_i(X')$};
\node (e) at (0,-1) {$S_i((M'_j \circ X'F_i)E_i)$};
\node (f) at (4,-1) {$F_iS_i(M'_j \circ X'F_i)$};
\node (g) at (8,-1) {$F_i(S_i(M'_j) \circ S_i(X'F_i))$};
\node (g') at (12,-1) {$F_iF_j S_i(X'F_i)$}; 
\node (h) at (16,-1) {$F_iF_jE_iS_i(X')$};
\node (a1) at (0,2.5) {$S_i(M'_j \circ X'E_iF_i)$};
\node (b1) at (4,2.5) {$S_i(M'_j) \circ S_i(X'E_iF_i)$};
\node (c1) at (8,2.5) {$F_j S_i(X'E_iF_i)$};
\node (d1) at (16,2.5) {$F_jE_iF_iS_i(X')$};
\node (e1) at (0,-2) {$S_i((M'_j\circ X')F_iE_i)$};
\node (f1) at (4,-2) {$F_iS_i((M'_j \circ X')F_i)$};
\node (g1) at (8,-2) {$F_i(S_i(M'_j) \circ E_iS_i(X'))$};
\node (h1) at (16,-2) {$F_iF_jE_iS_i(X')$};
\node (f2) at (4,-3) {$F_iE_iS_i(M'_j \circ X')$};
\node (g2) at (8,-3) {$F_iE_i(S_i(M'_j) \circ S_i(X'))$};
\node (h2) at (16,-3) {$F_iE_iF_jS_i(X')$};
\node (e3) at (0,-4.5) {$S_i((M'_j \circ X')E_iF_i)$};
\node (f3) at (4,-4.5) {$E_iF_iS_i(M'_j\circ X')$};
\node (g3) at (8,-4.5) {$E_iF_i(S_i(M'_j) \circ S_i(X'))$};
\node (h3) at (16,-4.5) {$E_iF_iF_jS_i(X')$};
\node (a2) at (0,3.5) {$S_i((M'_j \circ X'E_i)F_i)$};
\node (a3) at (0,5) {$E_iS_i(M'_j \circ X' E_i)$};
\node (b2) at (4,3.5) {$S_i(M'_j) \circ E_iS_i(X'E_i)$};
\node (b3) at (4,5) {$E_i(S_i(M'_j) \circ S_i(X'E_i))$};
\node (c2) at (8,3.5) {$F_jE_iS_i(X'E_i)$};
\node (c3) at (8,5) {$E_iF_jS_i(X'E_i)$};
\node (d3) at (16,5) {$E_iF_jF_iS_i(X')$};
\draw[<-,line width=1.3pt] (a) -- node {$\scriptstyle \theta$} (b);
\draw[->,line width=1.3pt] (b) -- node {$\scriptstyle t_{i,j}^{\varrho(\beta-\alpha_i)}\zeta_j$} (c);
\draw[line width=1.3pt] (c) -- node {$\scriptstyle \sim$} (c');
\draw (c') -- node {$\scriptstyle \sim$} (d);
\draw[->,line width=1.3pt] (a) -- node {$\scriptstyle \sigma'_{j,i}$} (e);
\draw[->] (d) -- node {$\xy 0;/r.12pc/: (0,0)*{\dcross{j}{i}}; \endxy$} (h);
\draw[->,line width=1.3pt] (c') -- node {$\xy 0;/r.12pc/: (0,0)*{\dcross{j}{i}}; \endxy$} (g');
\draw[line width=1.3pt] (e) -- node {$\scriptstyle \sim$} (f);
\draw[<-,line width=1.3pt] (f) -- node {$\scriptstyle \theta$} (g); 
\draw[->,line width=1.3pt] (g) -- node {$\scriptstyle t_{i,j}^{\varrho(\beta)}\zeta_j$} (g');
\draw (g') -- node {$\scriptstyle \sim$} (h);
\draw[->] (a1) -- node {$\xy 0;/r.12pc/: (0,0)*{\rcross{i}{i}}; \endxy$} (a);
\draw[->] (b1) -- node {$\xy 0;/r.12pc/: (0,0)*{\rcross{i}{i}}; \endxy$} (b);
\draw[->] (c1) -- node {$\xy 0;/r.12pc/: (0,0)*{\rcross{i}{i}}; \endxy$} (c);
\draw[->] (d1) -- node {$\xy 0;/r.12pc/: (0,0)*{\rcross{i}{i}}; \endxy$} (d);
\draw[<-] (a1) -- node {$\scriptstyle \theta$} (b1);
\draw[->] (b1) -- node {$\scriptstyle t_{i,j}^{\varrho(\beta-\alpha_i)}\zeta_j$} (c1);
\draw (c2) -- node {$\scriptstyle \sim$} (d1); 
\draw[->] (e1) -- node[sloped] {$\scriptstyle \sim$} (e);
\draw[->] (f1) -- node[sloped] {$\scriptstyle \sim$} (f);
\draw (g1) -- node[sloped] {$\scriptstyle \sim$} (g);
\draw (f2) -- node[sloped] {$\scriptstyle \sim$} (f1);
\draw[->] (g2) -- (g1);
\draw[<-] (f2) -- node {$\scriptstyle \theta$} (g2);
\draw[double distance=2pt] (h1) -- (h);
\draw[->] (h2) -- node {$\xy 0;/r.10pc/: (0,0)*{\rcross{i}{j}}; \endxy$} (h1); 
\draw[->] (g2) -- node {$\scriptstyle t_{i,j}^{\varrho(\beta)}\zeta_j$} (h2);
\draw[->] (e3) -- node {$\xy 0;/r.12pc/: (0,0)*{\rcross{i}{i}}; \endxy$} (e1); 
\draw[->] (f3) -- node {$\xy 0;/r.12pc/: (0,0)*{\rcross{i}{i}}; \endxy$} (f2);
\draw[->] (g3) -- node {$\xy 0;/r.12pc/: (0,0)*{\rcross{i}{i}}; \endxy$} (g2);
\draw[->] (h3) -- node {$\xy 0;/r.12pc/: (0,0)*{\rcross{i}{i}}; \endxy$} (h2);
\draw (e3) -- node {$\scriptstyle \sim$} (f3);
\draw[<-] (f3) -- node {$\scriptstyle \theta$} (g3);
\draw[->] (g3) -- node {$\scriptstyle t_{i,j}^{\varrho(\beta)}\zeta_j$} (h3); 
\draw[->] (a2) -- node[sloped] {$\scriptstyle \sim$} (a1);
\draw (a3) -- node[sloped] {$\scriptstyle \sim$} (a2);
\draw[<-] (a3) -- node {$\scriptstyle \theta$} (b3);
\draw (b2) -- node[sloped] {$\scriptstyle \sim$} (b1);
\draw[->] (b3) -- (b2);
\draw[->] (b3) -- node {$\scriptstyle t_{i,j}^{\varrho(\beta-\alpha_i)}\zeta_j$} (c3);
\draw (c3) -- node {$\scriptstyle \sim$} (d3);
\draw[->] (d3) -- node {$\xy 0;/r.12pc/: (0,0)*{\rcross{i}{j}}; \endxy$} (d1);
\draw[->] (c3) -- node {$\xy 0;/r.10pc/: (0,0)*{\rcross{i}{j}}; \endxy$} (c2);
\draw (c2) -- node[sloped] {$\scriptstyle \sim$} (c1);
\draw[->] (a2) to[bend right=70] node {$\sigma'_{j,i}$} (e3); 
\draw[->] (d3) to[bend left=70] node {$\xy 0;/r.12pc/: (0,0)*{\dcross{j}{i}}; \endxy$} (h3);
\draw[->] (g1) -- node {$\scriptstyle t_{i,j}^{\varrho(\beta)}\zeta_j$} (h1);
\draw[->] (b2) -- node {$\scriptstyle t_{i,j}^{\varrho(\beta-\alpha_i)}\zeta_j$} (c2);
\draw (e1) -- node {$\scriptstyle \sim$} (f1);
\node (C) at (5.7,-2) {(C)};
\node (E) at (12,-2.3) {(E)};
\node (G) at (2,3.5) {(G)};
\node (E') at (6,4.5) {(E')};
\node (H) at (12,1.8) {(H)};
\node (I) at (-1,0) {(I)};
\node (J) at (18,0) {(J)};
\node (H') at (2,-3) {(H')};
\end{tikzpicture}
}
}
\caption{} \label{fig:2}
\end{figure}

The homomorphisms $(M'_j \circ X')F_i \to M'_j \circ X'F_i$ and $(M'_j \circ X'E_i)F_i \to M'_j \circ X'E_iF_i$ are isomorphisms by Lemma \ref{lem:adjointSES} (2) and $M'_j F_i = 0$ (cf. the proof of Proposition \ref{prop:extremalequiv}). 
The thick diagrams in Figure \ref{fig:1} and \ref{fig:2} are the outer diagram of (\ref{pic:6}). 
Note that, for any $\lambda \in \mathsf{P}$, the homomorphism 
\[
X' \begin{bmatrix}
\xy 0;/r.17pc/:
(0,0)*{\rcross{i}{i}}; 
\endxy & \xy 0;/r.17pc/:
(0,0)*{\rcup{i}}; 
(-2,-0.5)*{\bullet};
(-8,-1)*{\scriptstyle N - 1};
\endxy & \cdots & \xy 0;/r.17pc/:
(0,0)*{\rcup{i}};
(-2,-0.5)*{\bullet};
\endxy & \xy 0;/r.17pc/:
(0,0)*{\rcup{i}}; 
\endxy
\end{bmatrix} \colon X'E_iF_i \oplus X'^{\oplus N} \to X'F_iE_i
\]
is a split epimorphism for sufficiently large $N$, by Theorem \ref{thm:Rouquierver} and \ref{thm:cyclotomic2rep}.  
%Also note that $S_i$ preserves the surjectivity since it is right exact. 
Hence, it suffices to prove the commutativity of the inner diagrams other than the thick ones, and the outer diagrams in Figure \ref{fig:1} and \ref{fig:2}. 

Commutativity of (A). 
We may disregard $F_j$. 
We identify $S_i(X'F_iE_i) \simeq F_iS_i(X'F_i) \simeq F_iE_iS_i(X')$. 
By Theorem \ref{thm:anotheraction}, we have 
\begin{align*}
S_i\left(X' \xy 0;/r.12pc/: (0,0)*{\rcup{i}};(-2,-0.5)*{\bullet}; (-6,-1)*{\scriptstyle n}; \endxy\right) &= S_i(X')\xy 0;/r.12pc/: (0,0)*{\rcup{i}};(-2,-0.5)*{\bullet}; (-6,-1)*{\scriptstyle n}; \endxy = c_{i,-s_i(\beta-\alpha_i)} \ \xy 0;/r.12pc/: (0,0)*{\rcup{i}};(-2,-0.5)*{\bullet}; (-6,-1)*{\scriptstyle n}; \endxy S_i(X') = c_{i,-\beta} \ \xy 0;/r.12pc/: (0,0)*{\rcup{i}};(-2,-0.5)*{\bullet}; (-6,-1)*{\scriptstyle n}; \endxy S_i(X'). 
\end{align*}

Commutativity of (B).
We may disregard $S_i$. 
Let $u \in M_j', v \in X'$. 
Under the homomorphism $M'_j \circ X' \xrightarrow{\xy 0;/r.12pc/: (0,0)*{\rcup{i}};(-2,-0.5)*{\bullet}; (-6,-1)*{\scriptstyle n}; \endxy} M'_j \circ X'F_iE_i \xrightarrow{\sigma'_{j,i}} (M'_j \circ X'F_i)E_i$,
the element $u \boxtimes v$ is mapped as follows: %sent to $u \boxtimes (v \boxtimes x_1^n e_(i)) e(*,i)$, following 
\begin{align*}
u \boxtimes v \mapsto u \boxtimes (v \boxtimes x_1^ne(i))E_i \mapsto (u\boxtimes (v \boxtimes x_1^n e(i)))E_i.  
\end{align*}
On the other hand, under the homomorphism $M'_j \circ X' \xrightarrow{\xy 0;/r.12pc/: (0,0)*{\rcup{i}};(-2,-0.5)*{\bullet}; (-6,-1)*{\scriptstyle n}; \endxy} (M'_j \circ X')F_iE_i \to (M'_j \circ X'F_i)E_i$, 
the element $u \boxtimes v$ is mapped as follows: %also sent to $(u \boxtimes (v \boxtimes x_1^n e(i)))e(*,i)$, following
\begin{align*}
u \boxtimes v \mapsto ((u\boxtimes v) \boxtimes x_1^ne(i))E_i \mapsto (u\boxtimes (v \boxtimes x_1^n e(i)))E_i. 
\end{align*}

Commutativity of (C) follows from Proposition \ref{prop:monoidality} (4). 

Commutativity of (D). 
We identify $S_i((M'_j \circ X') F_iE_i) \simeq F_i S_i((M'_j \circ X')F_i) \simeq F_iE_i S_i(M'_j \circ X')$. 
By Theorem \ref{thm:anotheraction}, we have 
\begin{align*}
S_i \left((M'_j \circ X')\xy 0;/r.12pc/: (0,0)*{\rcup{i}};(-2,-0.5)*{\bullet}; (-6,-1)*{\scriptstyle n}; \endxy \right) &= S_i(M'_j \circ X') \ \xy 0;/r.12pc/: (0,0)*{\rcup{i}};(-2,-0.5)*{\bullet}; (-6,-1)*{\scriptstyle n}; \endxy \\
&= c_{i,-(\alpha_j + s_i(\beta-\alpha_i))} \ \xy 0;/r.12pc/: (0,0)*{\rcup{i}};(-2,-0.5)*{\bullet}; (-6,-1)*{\scriptstyle n}; \endxy S_i(M'_j \circ X') \\ 
&= c_{i,-(\beta + \alpha_j)} \ \xy 0;/r.12pc/: (0,0)*{\rcup{i}};(-2,-0.5)*{\bullet}; (-6,-1)*{\scriptstyle n}; \endxy S_i(M'_j \circ X'). 
\end{align*}

Commutativity of (E). 
We may disregard the leftmost $F_i$ and identify $S_i(M'_j)$ with $R(\alpha_j)$ through $\zeta_j$. 
Note that $\xy 0;/r.12pc/: (0,0)*{\rcross{i}{j}}; \endxy S_i(X')$ is an isomorphism whose inverse is 
\[
\xy 0;/r.12pc/: (0,0)*{\lcross{j}{i}}; \endxy S_i(X') = \xy 0;/r.17pc/:
(0,0)*{\xybox{
(0,0)*{\dcross{}{}};
(8,7)*{\lcap{}};
(-8,-7)*{\lcup{}};
(-4,8)*{\slined{}};
(4,-8)*{\slined{j}};
(12,-12); (12,4) **\dir{-} ?(1)*\dir{>};
(12,-14)*{\scriptstyle i};
(-12,-4); (-12,12) **\dir{-} ?(1)*\dir{>}; 
}}\endxy S_i(X'), 
\]
by Definition \ref{def:catquantum} (7). 
Let $u \in R(\alpha_j), v \in S_i(X')$ and $1 \leq n \leq \height s_i(\beta-\alpha_i)$. 
We compute the images of $E_i \tau_n \cdots \tau_1 (u \boxtimes v) \in E_i(R(\alpha_j)\circ S_i(X'))$ in $E_iF_jS_i(X')$. 
If $n = 0$, $E_i(u\boxtimes v) = 0$ since $u \in R(\alpha_j)$ and $j \neq i$.  
Assume $n \neq 0$. 
Under the homomorphism $E_i(R(\alpha_j) \circ S_i(X')) \to R(\alpha_j) \circ E_iS_i(X') \to F_jE_iS_i(X')$, 
the element $E_i \tau_n \cdots \tau_1 (u \boxtimes v)$ mapped as follows: 
\begin{align*}
E_i\tau_n \cdots \tau_1 (u \boxtimes v) &= \tau_{n-1} \cdots \tau_1 E_i(\tau_1 (u \boxtimes v)) \\
&\mapsto \tau_{n-1}\cdots \tau_1 (u \boxtimes E_iv) \in R(\alpha_j) \circ E_iS_i(X') \\
&\mapsto \tau_{n-1} \cdots \tau_1 (u \boxtimes E_iv) \in F_jE_iS_i(X'). 
\end{align*}
Then, under the homomorphism $\xy 0;/r.12pc/: (0,0)*{\lcross{j}{i}}; \endxy S_i(X')$, 
it is sent to $E_i \tau_n \cdots \tau_1 (u \boxtimes v)$, following
\begin{align} \label{eq:5}
\tau_{n-1} \cdots\tau_1 (u \boxtimes E_iv) &\mapsto E_i(e(i) \boxtimes \tau_{n-1} \cdots \tau_1(u \boxtimes E_iv)) \in E_iF_iF_jE_iS_i(X') \\  
&= E_i \tau_n \cdots \tau_2 (e(i) \boxtimes u \boxtimes E_iv) \notag \\
&\mapsto E_i \tau_n \cdots \tau_2 \tau_1 (u \boxtimes e(i) \boxtimes E_iv) \in E_iF_jF_iE_iS_i(X') \notag \\
&\mapsto E_i \tau_n \cdots \tau_1 (u \boxtimes e(i,*)v) \in E_iF_j S_i(X') \notag \\
&= E_ie(i,*) \tau_n \cdots \tau_1 (u \boxtimes v)  \notag \\
&= E_i \tau_n \cdots \tau_1 (u \boxtimes v). \notag
\end{align}
It coincides with the image of $E_i\tau_n \cdots \tau_1 (u \boxtimes v)$ under the homomorphism $E_i(R(\alpha_j) \circ S_i(X')) \to E_iF_jS_i(X')$, hence (E) commutes. 
Commutativity of (E') is proved in the same way. 

Commutativity of (F): By Definition \ref{def:catquantum}, we compute
\begin{align*}
\xy 0;/r.12pc/: 
(0,0)*{\dcross{}{}};
(8,-7)*{\rcup{}};
(-4,-8)*{\slined{j}};
(12,-4); (12,4) **\dir{-} ?(1)*\dir{>};
(12,6)*{\scriptstyle i};
(6,-7.5)*{\bullet};
(2,-8)*{\scriptstyle n};
\endxy &=  \xy 0;/r.12pc/: 
(0,0)*{\dcross{}{}};
(8,-7)*{\rcup{}};
(-4,-8)*{\slined{j}};
(12,-4); (12,4) **\dir{-} ?(1)*\dir{>};
(12,6)*{\scriptstyle i};
(-2,1.7)*{\bullet};
(-5,1.7)*{\scriptstyle n};
\endxy = \xy 0;/r.12pc/:
(0,0)*{\dcross{}{}};
(-8,7)*{\rcap{}};
(8,-7)*{\rcup{}};
(4,8)*{\slined{}};
(-16,1)*{\rcup{}};
(-20,8)*{\slined{}};
(-4,-8)*{\slined{j}};
(12,-4); (12,12) **\dir{-} ?(1)*\dir{>}; 
(12,14)*{\scriptstyle i};
(-18,0.5)*{\bullet};
(-22,0)*{\scriptstyle n};
\endxy = \xy 0;/r.12pc/:
(0,0)*{\rcross{}{}};
(-8,-7)*{\rcup{}};
(-12,0)*{\slined{}};
(4,-8)*{\slined{j}};
(-10,-7.5)*{\bullet};
(-14,-8)*{\scriptstyle n};
(4,6)*{\scriptstyle i}; 
\endxy. 
\end{align*}
Using $t_{i,j}c_{i,-(\beta+\alpha_j)} = c_{i,-\beta}$, the assertion is proved. 

Commutativity of (G) follows from Proposition \ref{prop:monoidality} (4). 

Commutativity of (H). 
Recall that 
\[
\xy 0;/r.12pc/:
(0,0)*{\rcross{i}{i}}; 
\endxy = \xy 0;/r.12pc/:
(0,0)*{\xybox{
(0,0)*{\dcross{}{}};
(-8,7)*{\rcap{}};
(8,-7)*{\rcup{}};
(4,8)*{\slined{}};
(-4,-8)*{\slined{i}};
(-12,-12); (-12,4) **\dir{-} ?(1)*\dir{>};
(-12,-14)*{\scriptstyle i};
(12,-4); (12,12) **\dir{-} ?(1)*\dir{>}; 
}}\endxy =  \xy 0;/r.12pc/:
  (0,0)*{\xybox{
  (0,0)*{\ucross{}{}};
  (-8,-7)*{\rcup{}};
  (8,7)*{\rcap{}};
  (4,-8)*{\slineu{i}};
  (-4,8)*{\slineu{}};
  (-12,12); (-12,-4) **\dir{-} ?(1)*\dir{>};
  (12,-14)*{\scriptstyle i};
  (12,4); (12,-12) **\dir{-} ?(1)*\dir{>}; 
  }} \endxy. 
\]
Hence, we have
\begin{align*}
S_i\left(X' \xy 0;/r.12pc/:
(0,0)*{\rcross{i}{i}}; 
\endxy\right) &= S_i(X')\xy 0;/r.12pc/:
(0,0)*{\xybox{
(0,0)*{\dcross{}{}};
(-8,7)*{\rcap{}};
(8,-7)*{\rcup{}};
(4,8)*{\slined{}};
(-4,-8)*{\slined{i}};
(-12,-12); (-12,4) **\dir{-} ?(1)*\dir{>};
(-12,-14)*{\scriptstyle i};
(12,-4); (12,12) **\dir{-} ?(1)*\dir{>}; 
}}\endxy \\
&= c_{i,-s_i(\beta-\alpha_i)}^{-1}c_{i,-s_i(\beta-\alpha_i)} \xy 0;/r.12pc/:
  (0,0)*{\xybox{
  (0,0)*{\ucross{}{}};
  (-8,-7)*{\rcup{}};
  (8,7)*{\rcap{}};
  (4,-8)*{\slineu{i}};
  (-4,8)*{\slineu{}};
  (-12,12); (-12,-4) **\dir{-} ?(1)*\dir{>};
  (12,-14)*{\scriptstyle i};
  (12,4); (12,-12) **\dir{-} ?(1)*\dir{>}; 
  }} \endxy S_i(X') \\
&= \xy 0;/r.12pc/:
(0,0)*{\rcross{i}{i}};
\endxy S_i(X'). 
\end{align*}
Commutativity of (H') is proved in the same way. 

Commutativity of (I).
We may disregard $S_i$. 
Let $u \in M'_j, v \in X'$. 
Since $(M'_j \circ X'E_i)F_i$ is generated by $(M'_j \boxtimes X'E_i) \boxtimes e(i)$ as an $R(s_i\alpha_j + \beta - \alpha_i)$-module, 
it suffices to compute the images of $(u\boxtimes vE_i) \boxtimes e(i)$ in $(M'_j \circ X'F_i)E_i$.  
Under the homomorphism through $M'_j \circ X'E_iF_i$, it is mapped as follows: 
\begin{align*}
&(u \boxtimes vE_i) \boxtimes e(i) \\
&\mapsto u \boxtimes (vE_i \boxtimes e(i)) \in M'_j \circ X'E_iF_i \\
&\mapsto u \boxtimes (\tau_{\height \beta-1}(v \boxtimes e(i)))E_i \in M'_j \circ X'F_iE_i \quad \text{by the same computation as (\ref{eq:5})} \\
&\mapsto [u \boxtimes \tau_{\height \beta-1}(v \boxtimes e(i))]E_i \in (M'_j \circ X'F_i)E_i. 
\end{align*}
Under the other homomorphism, it is sent as follows: 
\begin{align*}
&(u \boxtimes vE_i) \boxtimes e(i) \\
&\mapsto (u \boxtimes v)E_i \boxtimes e(i) \in (M'_j \circ X')E_iF_i \\
&\mapsto [\tau_{\height \beta -a_{i,j}}((u\boxtimes v) \boxtimes e(i))]E_i \in (M'_j \circ X')F_iE_i \\
&\quad \text{by the same computation as (\ref{eq:5})}\\
&\mapsto [\tau_{\height \beta -a_{i,j}}(u \boxtimes (v \boxtimes e(i)))]E_i \in (M'_j \circ X'F_i)E_i. 
\end{align*}
Hence, (I) commutes. 

Commutativity of (J) follows from a computation based on Definition \ref{def:catquantum}: 
\begin{equation} \label{eq:6}
\xy 0;/r.12pc/: (0,0)*{\xybox{
(0,0)*{\rcross{i}{j}};
(12,4); (12,-4) **\dir{-} ?(1)*\dir{>};
(12,-6)*{\scriptstyle i};
(8,8)*{\rcross{}{}};
(-4,12); (-4,4) **\dir{-} ?(1)*\dir{>};
(0,16)*{\dcross{}{}};
(12,12); (12,20) **\dir{-} ?(1)*\dir{>};
}};
\endxy = \xy 0;/r.12pc/: (0,0)*{\xybox{ 
(0,0)*{\dcross{}{}};
(-4,-8)*{\slined{j}};
(8,-7)*{\rcup{}};
(-8,7)*{\rcap{}};
(-12,-12); (-12,4) **\dir{-} ?(1)*\dir{>};
(12,-4); (12,4) **\dir{-} ?(1)*\dir{>};
(16,7)*{\rcap{}};
(24,0)*{\dcross{}{}};
(20,-8)*{\slined{i}};
(32,-7)*{\rcup{}};
(36,-4); (36,20) **\dir{-} ?(1)*\dir{>};
(4,4)*{};(12,12)*{} **\crv{(4,7) & (12,9)}?(0)*\dir{<};
(28,4)*{};(20,12)*{} **\crv{(28,7) & (20,9)}?(0)*\dir{<};
(16,16)*{\dcross{}{}};
(-12,-14)*{\scriptstyle i}; 
}};
\endxy = \xy 0;/r.12pc/: (0,0)*{\xybox{
(0,0)*{\dcross{}{}};
(-8,7)*{\rcap{}};
(8,-8)*{\dcross{}{}};
(16,-15)*{\rcup{}};
(20,-12); (20,12) **\dir{-} ?(1)*\dir{>};
(12,0)*{\slined{}};
(8,8)*{\dcross{}{}};
(-12,-20); (-12,4) **\dir{-} ?(1)*\dir{>};
(-4,-4); (-4,-20) **\dir{-} ?(1)*\dir{>};
(4,-16)*{\slined{i}};
(-12,-22)*{\scriptstyle i};
(-4,-22)*{\scriptstyle j};
}};
\endxy = \xy 0;/r.12pc/: (0,0)*{\xybox{
(-4,-8)*{\dcross{j}{i}}; 
(4,0)*{\dcross{}{}};
(-4,8)*{\dcross{}{}};
(-8,0)*{\slined{}};
(8,8)*{\slined{}};
(12,-7)*{\rcup{}};
(16,-4); (16,20) **\dir{-} ?(1)*\dir{>};
(-12,15)*{\rcap{}};
(-16,-12); (-16,12) **\dir{-} ?(1)*\dir{>};
(0,16)*{\slined{}};
(8,16)*{\slined{}};
(-16,-14)*{\scriptstyle i};
}};
\endxy = \xy 0;/r.12pc/: (0,0)*{\xybox{
(0,0)*{\dcross{j}{i}};
(-12,0)*{\slineu{i}};
(-8,8)*{\rcross{}{}};
(4,8)*{\slined{}};
(-12,16)*{\slined{}};
(0,16)*{\rcross{}{}};
}};
\endxy. 
\end{equation}

Commutativity of the other inner diagrams and the outer diagram of Figure \ref{pic:1} are easily verified. 
Commutativity of the other inner diagrams of Figure \ref{pic:2} are also easy. 
Commutativity of the outer diagram of Figure \ref{pic:2} follows from the induction hypothesis
using $\varrho(\beta) - \varrho(\beta-\alpha_i) = 1 = \varrho(\beta-\alpha_i) - \varrho(\beta-2\alpha_i)$. 

The case $k \neq i$. 
Consider the diagrams of Figure \ref{fig:3},
where 
\begin{align*}
f_1 = F_j\left(S_i(X') \xy 0;/r.12pc/: (0,0)*{\rcross{i}{k}}; \endxy \right),\ & f_2 = S_i(M'_j \circ X') \xy 0;/r.12pc/: (0,0)*{\rcross{i}{k}}; \endxy,\\
f_3 = (S_i(M'_j) \circ S_i(X'))\xy 0;/r.12pc/: (0,0)*{\rcross{i}{k}}; \endxy,\ & f_4 = F_jS_i(X') \xy 0;/r.12pc/: (0,0)*{\rcross{i}{k}}; \endxy. 
\end{align*}

\begin{figure} %\label{pic:9}
\scalebox{0.9}{
\rotatebox{90}{
\begin{tikzpicture}[auto]
\node (a) at (0,1) {$S_i(M'_j \circ X'F_kE_i)$};
\node (b) at (4,1) {$S_i(M'_j) \circ S_i(X'F_kE_i)$};
\node (c) at (8,1) {$F_j S_i(X'F_kE_i)$};
\node (c') at (12,1) {$F_j F_iS_i(X'F_k)$};
\node (d) at (16,1) {$F_jF_i(S_i(X')\circ M_k)$};
\node (e) at (0,-0.5) {$S_i((M'_j \circ X'F_k)E_i)$};
\node (f) at (4,-0.5) {$F_iS_i(M'_j \circ X'F_k)$};
\node (g) at (8,-0.5) {$F_i(S_i(M'_j) \circ S_i(X'F_k))$};
\node (g') at (12,-0.5) {$F_iF_j S_i(X'F_k)$}; 
\node (h) at (16,-0.5) {$F_iF_j(S_i(X') \circ M_k)$};
\draw[<-,line width=1.3pt] (a) -- node {$\scriptstyle \theta$} (b);
\draw[->,line width=1.3pt] (b) -- node {$\scriptstyle t_{i,j}^{\varrho(\beta-\alpha_i)}\zeta_j$} (c);
\draw[line width=1.3pt] (c) -- node {$\scriptstyle \sim$} (c');
\draw (c') -- node {$\scriptstyle \sim$} (d);
\draw[->,line width=1.3pt] (a) -- node {$\scriptstyle \sigma'_{j,i}$} (e);
\draw[->] (d) -- node {$\xy 0;/r.12pc/: (0,0)*{\dcross{j}{i}}; \endxy$} (h);
\draw[->,line width=1.3pt] (c') -- node {$\xy 0;/r.12pc/: (0,0)*{\dcross{j}{i}}; \endxy$} (g');
\draw[line width=1.3pt] (e) -- node {$\scriptstyle \sim$} (f);
\draw[<-,line width=1.3pt] (f) -- node {$\scriptstyle \theta$} (g); 
\draw[->,line width=1.3pt] (g) -- node {$\scriptstyle t_{i,j}^{\varrho(\beta)}\zeta_j$} (g');
\draw (g') -- node {$\scriptstyle \sim$} (h);
\node (a1) at (0,2.5) {$S_i(M'_j \circ X'E_iF_k)$};
\node (b1) at (4,2.5) {$S_i(M'_j)\circ S_i(X'E_iF_k)$};
\node (c1) at (8,2.5) {$F_jS_i(X'E_iF_k)$};
\node (d1) at (16,2.5) {$F_j(F_iS_i(X')\circ M_k)$};
\node (a2) at (0,4) {$S_i((M'_j \circ X'E_i)F_k)$}; 
\node (b2) at (4,4) {$S_i(M'_j \circ X'E_i) \circ M_k$};
\node (c2) at (8,4) {$S_i(M'_j) \circ S_i(X'E_i) \circ M_k$};
\node (c'2) at (12,4) {$F_jS_i(X'E_i) \circ M_k$};
\node (d2) at (16,4) {$F_jF_iS_i(X') \circ M_k$};
\node (e1) at (0,-2) {$S_i((M'_j\circ X')F_kE_i)$};
\node (f1) at (4,-2) {$F_iS_i((M'_j \circ X')F_k)$};
\node (g1) at (8,-2) {$F_i(S_i(M'_j \circ X') \circ M_k)$};
\node (g'1) at (12,-2) {$F_i(S_i(M'_j)\circ S_i(X') \circ M_k)$};
\node (h1) at (16,-2) {$F_i(F_jS_i(X') \circ M_k)$};
\node (e2) at (0,-3.5) {$S_i((M'_j \circ X')E_iF_k)$};
\node (f2) at (4,-3.5) {$S_i((M'_j \circ X')E_i) \circ M_k$};
\node (g2) at (8,-3.5) {$F_iS_i(M'_j \circ X') \circ M_k$};
\node (g'2) at (12,-3.5) {$F_i(S_i(M'_j) \circ S_i(X')) \circ M_k$};
\node (h2) at (16,-3.5) {$F_iF_jS_i(X') \circ M_k$};
\draw[->] (a1) -- node {$\xy 0;/r.12pc/: (0,0)*{\rcross{i}{k}}; \endxy$} (a);
\draw[->] (b1) -- node {$\xy 0;/r.12pc/: (0,0)*{\rcross{i}{k}}; \endxy$} (b);
\draw[->] (c1) -- node {$\xy 0;/r.12pc/: (0,0)*{\rcross{i}{k}}; \endxy$} (c);
\draw[->] (d1) -- node {$f_1$} (d);
\draw[<-] (a1) -- node {$\scriptstyle \theta$} (b1);
\draw[->] (b1) -- node {$\scriptstyle t_{i,j}^{\varrho(\beta-\alpha_i)}\zeta_j$} (c1);
\draw (c1) -- node {$\scriptstyle \sim$} (d1);
\draw (a2) -- node[sloped] {$\scriptstyle \sim$} (a1);
\draw (c2) -- node[sloped] {$\scriptstyle \sim$} (b1);
\draw (d2) -- node[sloped] {$\scriptstyle \sim$} (d1);
\draw (a2) -- node {$\scriptstyle \sim$} (b2);
\draw[<-] (b2) -- node {$\scriptstyle \theta$} (c2);
\draw[->] (c2) -- node {$\scriptstyle t_{i,j}^{\varrho(\beta-\alpha_i)}\zeta_j$} (c'2);
\draw (c'2) -- node {$\scriptstyle \sim$} (d2);
\draw[->] (e1) -- node[sloped] {$\scriptstyle \sim$} (e);
\draw[->] (f1) -- node[sloped] {$\scriptstyle \sim$} (f); 
\draw (g'1) -- node[sloped] {$\scriptstyle \sim$} (g);
\draw (h1) -- node[sloped] {$\scriptstyle \sim$} (h);
\draw (e1) -- node {$\scriptstyle \sim$} (f1);
\draw (f1) -- node {$\scriptstyle \sim$} (g1);
\draw[<-] (g1) -- node {$\scriptstyle \theta$} (g'1);
\draw[->] (e2) -- node {$\xy 0;/r.12pc/: (0,0)*{\rcross{i}{k}}; \endxy$} (e1);
\draw[->] (g2) -- node {$f_2$} (g1);
\draw[->] (g'2) -- node {$f_3$} (g'1);
\draw[->] (h2) -- node {$f_4$} (h1);
\draw (e2) -- node {$\scriptstyle \sim$} (f2);
\draw (f2) -- node {$\scriptstyle \sim$} (g2);
\draw[<-] (g2) -- node {$\scriptstyle \theta$} (g'2);
\draw[->] (g'2) -- node {$\scriptstyle t_{i,j}^{\varrho(\beta)}\zeta_j$} (h2);
\draw[->] (d2) to[bend left = 70] node {$\xy 0;/r.12pc/: (0,0)*{\dcross{j}{i}};\endxy$} (h2);
\draw[->] (a2) to[bend right = 70] node {$\sigma'_{j,i}$} (e2); 
\draw[->] (g'1) -- node {$\scriptstyle t_{i,j}^{\varrho(\beta)}\zeta_j$} (h1);
\node (K) at (3,3.2) {(K)};
\node (L) at (-1.3, 0.2) {(L)};
\node (M) at (17.5,0.2) {(M)};
\node (N) at (7,-1.2) {(N)};
%\draw[smooth,->] plot coordinates {(-1.5,3.8) (-3,2.5) (-3,-4) (0,-5) (3,-3.7)};
\end{tikzpicture}
}
}
\caption{} \label{fig:3}
\end{figure}

We need to prove that the thick diagram commutes. 
Note that the homomorphism $S_i(M'_j \circ X'E_iF_k) \xrightarrow{\xy 0;/r.12pc/: (0,0)*{\rcross{i}{k}}; \endxy} S_i(M'_j \circ X'F_kE_i)$ is an isomorphism by Definition \ref{def:catquantum} (7).
Hence, it suffices to prove that all the other inner diagrams and the outer diagram commute. 

Commutativity of (L). 
We may disregard $S_i$. 
Let $u \in M'_j, v \in X'$. 
Since $(M'_j \circ X'E_i)F_k$ is generated by $(M'_j \boxtimes X'E_i) \boxtimes e(k)$ as an $R(\beta-\alpha_i + s_i\alpha_j)$-module, 
it suffices to compute the images of $(u \boxtimes vE_i) \boxtimes e(k)$ in $(M'_j \circ X'F_k)E_i$. 
Under the homomorphism $(M'_j \circ X'E_i)F_k \to M'_j \circ X'E_iF_k \xrightarrow{\xy 0;/r.12pc/: (0,0)*{\rcross{i}{k}}; \endxy} M'_j \circ X'F_kE_i \xrightarrow{\sigma'_{j,i}} (M'_j \circ X'F_k)E_i$, 
it is sent to $[u \boxtimes \tau_{\height \beta-1}(v \boxtimes e(k))]E_i$, following
\begin{align*}
&(u \boxtimes vE_i)\boxtimes e(k) \\
&\mapsto u \boxtimes (vE_i \boxtimes e(k)) \\
&\mapsto u \boxtimes [\tau_{\height \beta -1} (v \boxtimes e(k))]E_i \quad \text{by the same computation as (\ref{eq:5})} \\
&\mapsto [u \boxtimes \tau_{\height \beta -1} (v \boxtimes e(k))]E_i. 
\end{align*}
On the other hand, under the homomorphism 
\[
(M'_j \circ X'E_i)F_k \xrightarrow{\sigma'_{j,i}} (M'_j \circ X')E_iF_k \xrightarrow{\xy 0;/r.12pc/: (0,0)*{\rcross{i}{k}}; \endxy} (M'_j \circ X')F_kE_i \to (M'_j \circ X'F_k)E_i,
\] 
it is also sent to $[u \boxtimes \tau_{\height \beta -1}(v \boxtimes e(k))]E_i$, following
\begin{align*}
&(u \boxtimes vE_i) \boxtimes e(k) \\
&\mapsto (u \boxtimes v)E_i \boxtimes e(k) \\
&\mapsto [\tau_{\height \beta + \height s_i\alpha_j -1}((u \boxtimes v) \boxtimes e(k))]E_i \quad \text{by the same computation as (\ref{eq:5})} \\
&\mapsto [\tau_{\height \beta + \height s_i\alpha_j -1}(u \boxtimes (v \boxtimes e(k)))]E_i \\
&= [u \boxtimes \tau_{\height \beta -1}(v \boxtimes e(k))]E_i. 
\end{align*}
Hence, (L) commutes. 

The commutativity of (K) and (N) follows from Proposition \ref{prop:monoidality} (5).

Commutativity of (M). 
By Definition \ref{def:catquantum} (7), $f_1$ and $f_4$ are isomorphisms whose inverses are given by 
\[
f_1^{-1} = F_j\left(S_i(X') \xy 0;/r.12pc/: (0,0)*{\lcross{k}{i}}; \endxy \right),\ f_4^{-1} = F_jS_i(X') \xy 0;/r.12pc/: (0,0)*{\lcross{k}{i}}; \endxy. 
\]
We use the following lemma.

\begin{lemma} \label{lem:formula}
Let $\gamma \in \mathsf{Q}_+, Y \in \gMod{R_i(\gamma)}$. 
Then, the homomorphism $F_i(Y \circ M_k) \xrightarrow{Y \xy 0;/r.12pc/: (0,0)*{\lcross{k}{i}}; \endxy} F_iY \circ M_k$ coincides with the $t_{i,k}^{-1}$-multiple of the canonical homomorphism given in Lemma \ref{lem:adjointSES}. 
Similarly, the homomorphism $(M'_k \circ Z)F_i \xrightarrow{\xy 0;/r.12pc/: (0,0)*{\rcross{i}{k}}; \endxy Z} M'_k \circ ZF_i$ coincides with the $t_{i,k}^{-1}$-multiple of the canonical homomorphism for $Z \in \gMod{{}_iR(\gamma)}$. 
\end{lemma}

\begin{proof}
Let $u \in Y, v \in M_k$. 
Since $F_i(Y \circ M_k)$ is generated by $e(i) \boxtimes (Y \boxtimes M_k)$ as an $R_i(\alpha_i +\gamma + s_i\alpha_k)$-module, 
it suffices to compute the image of $e(i) \boxtimes (u \boxtimes v)$. 
Note that 
\begin{align*}
Y \xy 0;/r.12pc/: (0,0)*{\lcross{k}{i}}; \endxy &= Y \xy 0;/r.12pc/:
(0,0)*{\xybox{
(0,0)*{\dcross{}{}};
(8,7)*{\lcap{}};
(-8,-7)*{\lcup{}};
(-4,8)*{\slined{}};
(4,-8)*{\slined{k}};
(12,-12); (12,4) **\dir{-} ?(1)*\dir{>};
(12,-14)*{\scriptstyle i};
(-12,-4); (-12,12) **\dir{-} ?(1)*\dir{>}; 
}}\endxy \\
&= \bigg[ F_i(Y \circ M_k) \xrightarrow{c_{i,\gamma} \xy 0;/r.12pc/: (0,0)*{\lcup{i}}; \endxy} F_i(E_iF_iY \circ M_k) \\
&\quad \xrightarrow{\sigma_{i,k}} F_iE_i(F_iY \circ M_k) \xrightarrow{c_{i,\alpha_i + \gamma + s_i\alpha_k}^{-1} \xy 0;/r.12pc/: (0,0)*{\lcap{i}}; \endxy} F_iY \circ M_k \bigg] \\
&= t_{i,k}^{-1} \bigg[ F_i(Y \circ M_k) \xrightarrow{\xy 0;/r.12pc/: (0,0)*{\lcup{i}}; \endxy} F_i(E_iF_iY \circ M_k) \\
&\quad \xrightarrow{\sigma_{i,k}} F_iE_i(F_iY \circ M_k) \xrightarrow{ \xy 0;/r.12pc/: (0,0)*{\lcap{i}}; \endxy} F_iY \circ M_k \bigg]. \\
\end{align*}
Under this homomorphism, the element $e(i) \boxtimes (u \boxtimes v)$ is sent to $t_{i,k}^{-1} (e(i) \boxtimes u) \boxtimes v$, following
\[
e(i) \boxtimes (u \boxtimes v) \mapsto e(i) \boxtimes (E_i(e(i) \boxtimes u) \boxtimes v) \mapsto e(i) \boxtimes E_i((e(i) \boxtimes u) \boxtimes v ) \mapsto (e(i) \boxtimes u) \boxtimes v. 
\]
The lemma is proved . 
\end{proof}

Let $u \in S_i(X'), v \in M_k$. 
Under the homomorphism $F_jF_i(S_i(X') \circ M_k) \xrightarrow{f_1^{-1}} F_j(F_iS_i(X') \circ M_k) \xrightarrow{\sim} F_jF_iS_i(X') \circ M_k \xrightarrow{\xy 0;/r.12pc/: (0,0)*{\dcross{j}{i}};\endxy} F_iF_jS_i(X') \circ M_k$, 
the element $e(j) \boxtimes e(i) \boxtimes (u \boxtimes v)$ is sent to $t_{i,k}^{-1}\tau_1 (e(i) \boxtimes e(j) \boxtimes u) \boxtimes v$, following 
\begin{align*}
e(j) \boxtimes e(i) \boxtimes (u \boxtimes v) &\mapsto t_{i,k}^{-1} e(j) \boxtimes ((e(i) \boxtimes u)\boxtimes v) \quad \text{by Lemma \ref{lem:formula}} \\
&\mapsto t_{i,k}^{-1} (e(j) \boxtimes e(i) \boxtimes u) \boxtimes v \\ 
&\mapsto t_{i,k}^{-1} \tau_1 (e(i) \boxtimes e(j) \boxtimes u) \boxtimes v. 
\end{align*}
On the other hand, under the homomorphism 
\begin{align*}
&F_j F_i(S_i(X') \circ M_k) \xrightarrow{\xy 0;/r.12pc/: (0,0)*{\dcross{j}{i}};\endxy} F_iF_j(S_i(X') \circ M_k) \\
&\to F_i(F_j S_i(X') \circ M_k) \xrightarrow{f_4^{-1}} F_iF_jS_i(X') \circ M_k,
\end{align*} 
the element $e(j) \boxtimes e(i) \boxtimes (u \boxtimes v)$ is also sent to $t_{i,k}^{-1} \tau_1 (e(i) \boxtimes e(j) \boxtimes u) \boxtimes v$, following
\begin{align*}
e(j) \boxtimes e(i) \boxtimes (u \boxtimes v) &\mapsto \tau_1 (e(i) \boxtimes e(j) \boxtimes (u \boxtimes v))\\
&\mapsto \tau_1 (e(i) \boxtimes ((e(j)\boxtimes u) \boxtimes v)) \\
&\mapsto t_{i,k}^{-1} \tau_1 ((e(i) \boxtimes e(j) \boxtimes u) \boxtimes v) \quad \text{by Lemma \ref{lem:formula}}. \\
\end{align*}
Hence, (M) commutes. 

Commutativity of the outer diagram of Figure \ref{fig:3} follows from the induction hypothesis, using 
\[
\varrho(\beta) - \varrho(\beta-\alpha_i) = 1 = \varrho(\beta-\alpha_k) - \varrho(\beta-\alpha_k-\alpha_i). 
\]

Now, Case 1 is complete. 

\subsubsection{Case 2}
$\xy 0;/r.12pc/: (0,0)*{\dcross{i}{j}}; \endxy$ for $j \neq i$. 
Let $X \in \gMod{{}_iR(\beta)}$.
Consider the following diagram: 
\begin{equation} \label{pic:10}
\begin{tikzcd}
& & F_i(S_i(M'_j \circ X)) \arrow[d,-,"\sim" sloped]\\
S_i((M'_j \circ X)E_i) \arrow[dddddd,bend right=60, "\sigma'_{i,j}"] \arrow[r,-,"\sim"]\arrow[d,-,"\sim" sloped] & S_i(M'_j \circ X)E_i \arrow[ru,-,"\sim" sloped]\arrow[d,-,"\sim" sloped] & F_i(S_i(M'_j) \circ S_i(X)) \arrow[u,"\theta"]\arrow[d,"t_{i,j}^{\varrho(\beta)}\zeta_j"] \\
S_i((F_jX)E_i) \arrow[r,-,"\sim" sloped]\arrow[d,-,"\sim" sloped] & S_i(F_jX) E_i \arrow[d,-,"\sim" sloped] & F_i(R(\alpha_j)\circ S_i(X)) \arrow[d,-,"\sim" sloped] \\
S_i(F_iF_jX) \arrow[r,"\kappa_i^-", line width=1pt]\arrow[d,"{\xy 0;/r.12pc/: (0,0)*{\dcross{i}{j}}; \endxy}",line width=1pt] & F_iS_i(F_jX) \arrow[r,"\kappa_j^-",line width=1.3pt]\arrow[ruuu,bend left = 8,-,"\sim" sloped] & F_iF_jS_i(X) \arrow[d,"{\xy 0;/r.12pc/: (0,0)*{\dcross{i}{j}}; \endxy}",line width=1pt] \\
S_i(F_jF_iX) \arrow[r,"\kappa_j^-",line width=1pt]\arrow[d,-,"\sim" sloped] & F_jS_i(F_iX) \arrow[r,"\kappa_i^-",line width=1pt]\arrow[d,-,"\sim" sloped] & F_jF_iS_i(X) \arrow[d,-,"\sim" sloped] \\
S_i(F_j(XE_i)) \arrow[dd,-,"\sim" sloped]\arrow[r,"\kappa"] & F_jS_i(XE_i) \arrow[r,-,"\sim"]\arrow[d,-,"\sim" sloped] & F_j(S_i(X)E_i) \\
& R(\alpha_j) \circ S_i(XE_i)  & \\
S_i(M'_j \circ XE_i) & S_i(M'_j) \circ S_i(XE_i) \arrow[l,"\theta"] \arrow[u,"t_{i,j}^{\varrho(\beta-\alpha_i)}\zeta_j"] & \\
\end{tikzcd}
\end{equation}
Note that it is almost the same as (\ref{pic:6}) except three morphisms.
We need to prove that the thick diagram commutes, and it is reduced to proving the commutativity of the outer diagram as in Case 1.
We prove it by induction on $\height \beta$. 
Most of the argument is parallel to Case 1. 

If $\beta = 0$, then $F_j(S_i(X)E_i) =0$ and the assertion is trivial. 
Assume $\beta \neq 0$. 
As before, we may assume $X = X'F_k$ for some $k \in I$ and $X' \in \gMod{{}_iR(\beta - \alpha_k)}$. 

The case $k = i$. 
Consider diagrams in Figure \ref{fig:4} and \ref{fig:5},
where 
\begin{align*}
g_1 =& \bigg[(M'_j \circ X')E_iF_i \xrightarrow{\xy 0;/r.12pc/: (0,0)*{\rcross{i}{i}}; \endxy} (M'_j \circ X')F_iE_i \\ 
&\quad \to (M'_j \circ X'F_i)E_i \xrightarrow{\sigma'_{i,j}} M'_j \circ X'F_iE_i\bigg] \\
&- \bigg[(M'_j \circ X')E_iF_i \xrightarrow{\sigma'_{i,j}} (M'_j \circ X'E_i)F_i \\ 
&\quad \to M'_j \circ X'E_iF_i \xrightarrow{\xy 0;/r.12pc/: (0,0)*{\rcross{i}{i}}; \endxy} M'_j \circ X'F_iE_i \bigg], \\
g_2 =& \bigg[E_iF_iF_jS_i(X') \xrightarrow{\xy 0;/r.12pc/: (0,0)*{\rcross{i}{i}}; \endxy} F_iE_iF_jS_i(X') \\
&\quad \xrightarrow{\xy 0;/r.12pc/: (0,0)*{\rcross{i}{j}}; \endxy} F_iF_jE_iS_i(X') \xrightarrow{\xy 0;/r.12pc/: (0,0)*{\dcross{i}{j}}; \endxy} F_jF_iE_iS_i(X') \bigg] \\
&- \bigg[E_iF_iF_jS_i(X') \xrightarrow{\xy 0;/r.12pc/: (0,0)*{\dcross{i}{j}}; \endxy} E_iF_jF_iS_i(X') \\
&\quad \xrightarrow{\xy 0;/r.12pc/: (0,0)*{\rcross{i}{j}}; \endxy} F_jE_iF_iS_i(X') \xrightarrow{\xy 0;/r.12pc/: (0,0)*{\rcross{i}{i}}; \endxy} F_jF_iE_iS_i(X') \bigg].
\end{align*}

\begin{figure} %\label{pic:11}
\scalebox{0.9}{
\rotatebox{90}{
\begin{tikzpicture}[auto]
\node (a) at (0,1) {$S_i(M'_j \circ X'F_iE_i)$};
\node (b) at (4,1) {$S_i(M'_j) \circ S_i(X'F_iE_i)$};
\node (c) at (8,1) {$F_j S_i(X'F_iE_i)$};
\node (c') at (12,1) {$F_j F_iS_i(X'F_i)$};
\node (d) at (16,1) {$F_jF_iE_iS_i(X')$};
\node (e) at (0,-0.5) {$S_i((M'_j \circ X'F_i)E_i)$};
\node (f) at (4,-0.5) {$F_iS_i(M'_j \circ X'F_i)$};
\node (g) at (8,-0.5) {$F_i(S_i(M'_j) \circ S_i(X'F_i))$};
\node (g') at (12,-0.5) {$F_iF_j S_i(X'F_i)$}; 
\node (h) at (16,-0.5) {$F_iF_jE_iS_i(X')$};
\node (a2) at (0,4) {$S_i(M'_j \circ X')$};
\node (b2) at (4,4) {$S_i(M'_j) \circ S_i(X')$};
\node (c2) at (8,4) {$F_j S_i(X')$};
\node (d2) at (16,4) {$F_jS_i(X')$};
\node (e1) at (0,-2) {$S_i((M'_j\circ X')F_iE_i)$};
\node (f1) at (4,-2) {$F_iS_i((M'_j \circ X')F_i)$};
\node (g1) at (8,-2) {$F_i(S_i(M'_j) \circ E_iS_i(X'))$};
\node (h1) at (16,-2) {$F_iF_jE_iS_i(X')$};
\node (f2) at (4,-3) {$F_iE_iS_i(M'_j \circ X')$};
\node (g2) at (8,-3) {$F_iE_i(S_i(M'_j) \circ S_i(X'))$};
\node (h2) at (16,-3) {$F_iE_iF_jS_i(X')$};
\node (e3) at (0,-4.5) {$S_i(M'_j \circ X')$};
\node (f3) at (4,-4.5) {$S_i(M'_j\circ X')$};
\node (g3) at (8,-4.5) {$S_i(M'_j) \circ S_i(X')$};
\node (h3) at (16,-4.5) {$F_jS_i(X')$};
\node (a1) at (0,2.5) {$S_i(M'_j \circ X'F_iE_i)$};
\node (b1) at (4,2.5) {$S_i(M'_j) \circ S_i(X'F_iE_i)$};
\node (c1) at (8,2.5) {$F_j S_i(X'F_iE_i)$};
\node (c'1) at (12,2.5) {$F_j F_iS_i(X'F_i)$};
\node (d1) at (16,2.5) {$F_jF_iE_iS_i(X')$};
\draw[<-,line width=1.3pt] (a) -- node {$\scriptstyle \theta$} (b);
\draw[->,line width=1.3pt] (b) -- node {$\scriptstyle t_{i,j}^{\varrho(\beta-\alpha_i)}\zeta_j$} (c);
\draw[line width=1.3pt] (c) -- node {$\scriptstyle \sim$} (c');
\draw (c') -- node {$\scriptstyle \sim$} (d);
\draw[<-,line width=1.3pt] (a) -- node {$\scriptstyle \sigma'_{i,j}$} (e);
\draw[<-] (d) -- node {$\xy 0;/r.12pc/: (0,0)*{\dcross{i}{j}}; \endxy$} (h);
\draw[<-,line width=1.3pt] (c') -- node {$\xy 0;/r.12pc/: (0,0)*{\dcross{i}{j}}; \endxy$} (g');
\draw[line width=1.3pt] (e) -- node {$\scriptstyle \sim$} (f);
\draw[<-,line width=1.3pt] (f) -- node {$\scriptstyle \theta$} (g); 
\draw[->,line width=1.3pt] (g) -- node {$\scriptstyle t_{i,j}^{\varrho(\beta)}\zeta_j$} (g');
\draw (g') -- node {$\scriptstyle \sim$} (h);
\draw[->] (a2) -- node {$\xy 0;/r.12pc/: (0,0)*{\rcup{i}};(-2,-0.5)*{\bullet}; (-6,-1)*{\scriptstyle n}; \endxy$} (a1);
\draw[->] (b2) -- node {$\xy 0;/r.12pc/: (0,0)*{\rcup{i}};(-2,-0.5)*{\bullet}; (-6,-1)*{\scriptstyle n}; \endxy$} (b1);
\draw[->] (c2) -- node {$\xy 0;/r.12pc/: (0,0)*{\rcup{i}};(-2,-0.5)*{\bullet}; (-6,-1)*{\scriptstyle n}; \endxy$} (c1);
\draw[->] (d2) -- node[swap] {$c_{i,-\beta}\xy 0;/r.12pc/: (0,0)*{\rcup{i}}; (-2,-0.5)*{\bullet}; (-6,-1)*{\scriptstyle n};\endxy$} (d1);
\draw[<-] (a2) -- node {$\scriptstyle \theta$} (b2);
\draw[->] (b2) -- node {$\scriptstyle t_{i,j}^{\varrho(\beta-\alpha_i)}\zeta_j$} (c2);
\draw[double distance=2pt] (c2) -- (d2); 
\draw[->] (e1) -- node[sloped] {$\scriptstyle \sim$} (e);
\draw[->] (f1) -- node[sloped] {$\scriptstyle \sim$} (f);
\draw (g1) -- node[sloped] {$\scriptstyle \sim$} (g);
\draw (f2) -- node[sloped] {$\scriptstyle \sim$} (f1);
\draw[->] (g2) -- (g1);
\draw[<-] (f2) -- node {$\scriptstyle \theta$} (g2);
\draw[double distance=2pt] (h1) -- (h);
\draw[->] (h2) -- node {$\xy 0;/r.10pc/: (0,0)*{\rcross{i}{j}}; \endxy$} (h1); 
\draw[->] (g2) -- node {$\scriptstyle t_{i,j}^{\varrho(\beta)}\zeta_j$} (h2);
\draw[->] (e3) -- node {$\xy 0;/r.12pc/: (0,0)*{\rcup{i}};(-2,-0.5)*{\bullet}; (-6,-1)*{\scriptstyle n}; \endxy$} (e1); 
\draw[->] (f3) -- node {$c_{i,-(\beta+\alpha_j)}\xy 0;/r.12pc/: (0,0)*{\rcup{i}};(-2,-0.5)*{\bullet}; (-6,-1)*{\scriptstyle n}; \endxy$} (f2);
\draw[->] (g3) -- node {$c_{i,-(\beta+\alpha_j)}\xy 0;/r.12pc/: (0,0)*{\rcup{i}};(-2,-0.5)*{\bullet}; (-6,-1)*{\scriptstyle n}; \endxy$} (g2);
\draw[->] (h3) -- node {$c_{i,-(\beta+\alpha_j)}\xy 0;/r.12pc/: (0,0)*{\rcup{i}};(-2,-0.5)*{\bullet}; (-6,-1)*{\scriptstyle n}; \endxy$} (h2);
\draw[double distance=2pt] (e3) -- (f3);
\draw[<-] (f3) -- node {$\scriptstyle \theta$} (g3);
\draw[->] (g3) -- node {$\scriptstyle t_{i,j}^{\varrho(\beta)}\zeta_j$} (h3); 
\draw[double distance=2pt] (a2) to[bend right=70] (e3);
\draw[->] (h3) to[bend right=70] node {$\scriptstyle t_{i,j}^{-1}\id$} (d2);
\draw (e1) -- node {$\scriptstyle \sim$} (f1);
\draw[->] (g1) -- node {$\scriptstyle t_{i,j}^{\varrho(\beta)}\zeta_j$} (h1);
\draw[<-] (a1) -- node {$\scriptstyle \theta$} (b1);
\draw[->] (b1) -- node {$\scriptstyle t_{i,j}^{\varrho(\beta-\alpha_i)}\zeta_j$} (c1);
\draw (c1) -- node {$\scriptstyle \sim$} (c'1);
\draw (c'1) -- node {$\scriptstyle \sim$} (d1);
\draw[->] (a1) -- node[swap] {$Q_{i,j}\left(\xy 0;/r.12pc/: (0,0)*{\sdotu{i}}; \endxy, y_j \right)$} (a);
\draw[->] (b1) -- node[swap] {$Q_{i,j}\left(\xy 0;/r.12pc/: (0,0)*{\sdotu{i}}; \endxy, y_j \right)$} (b);
\draw[->] (c1) -- node[swap] {$Q_{i,j}\left(\xy 0;/r.12pc/: (0,0)*{\sdotu{i}}; \endxy, \xy 0;/r.12pc/: (0,0)*{\sdotd{j}}; \endxy \right)$} (c);
\draw[->] (c'1) -- node[swap] {$Q_{i,j}\left(\xy 0;/r.12pc/: (0,0)*{\sdotd{i}}; \endxy, \xy 0;/r.12pc/: (0,0)*{\sdotd{j}}; \endxy \right) S_i(X'F_i)$} (c');
\draw[->] (d1) --node[swap] {$Q_{i,j}\left(\xy 0;/r.12pc/: (0,0)*{\sdotd{i}}; \endxy, \xy 0;/r.12pc/: (0,0)*{\sdotd{j}}; \endxy \right)E_iS_i(X')$}  (d);
\node (O) at (-1.5,0) {(O)};
\node (P) at (17,-1.2) {(P)}; 
\end{tikzpicture}
}
}
\caption{} \label{fig:4}
\end{figure}

\begin{figure} %\label{pic:12}
\scalebox{0.85}{
\rotatebox{90}{
\begin{tikzpicture}[auto]
\node (a) at (0,1) {$S_i(M'_j \circ X'F_iE_i)$};
\node (b) at (4,1) {$S_i(M'_j) \circ S_i(X'F_iE_i)$};
\node (c) at (8,1) {$F_j S_i(X'F_iE_i)$};
\node (c') at (12,1) {$F_j F_iS_i(X'F_i)$};
\node (d) at (16,1) {$F_jF_iE_iS_i(X')$};
\node (e) at (0,-1) {$S_i((M'_j \circ X'F_i)E_i)$};
\node (f) at (4,-1) {$F_iS_i(M'_j \circ X'F_i)$};
\node (g) at (8,-1) {$F_i(S_i(M'_j) \circ S_i(X'F_i))$};
\node (g') at (12,-1) {$F_iF_j S_i(X'F_i)$}; 
\node (h) at (16,-1) {$F_iF_jE_iS_i(X')$};
\node (a1) at (0,2.5) {$S_i(M'_j \circ X'E_iF_i)$};
\node (b1) at (4,2.5) {$S_i(M'_j) \circ S_i(X'E_iF_i)$};
\node (c1) at (8,2.5) {$F_j S_i(X'E_iF_i)$};
\node (d1) at (16,2.5) {$F_jE_iF_iS_i(X')$};
\node (e1) at (0,-2) {$S_i((M'_j\circ X')F_iE_i)$};
\node (f1) at (4,-2) {$F_iS_i((M'_j \circ X')F_i)$};
\node (g1) at (8,-2) {$F_i(S_i(M'_j) \circ E_iS_i(X'))$};
\node (h1) at (16,-2) {$F_iF_jE_iS_i(X')$};
\node (f2) at (4,-3) {$F_iE_iS_i(M'_j \circ X')$};
\node (g2) at (8,-3) {$F_iE_i(S_i(M'_j) \circ S_i(X'))$};
\node (h2) at (16,-3) {$F_iE_iF_jS_i(X')$};
\node (e3) at (0,-4.5) {$S_i((M'_j \circ X')E_iF_i)$};
\node (f3) at (4,-4.5) {$E_iF_iS_i(M'_j\circ X')$};
\node (g3) at (8,-4.5) {$E_iF_i(S_i(M'_j) \circ S_i(X'))$};
\node (h3) at (16,-4.5) {$E_iF_iF_jS_i(X')$};
\node (a2) at (0,3.5) {$S_i((M'_j \circ X'E_i)F_i)$};
\node (a3) at (0,4.5) {$E_iS_i(M'_j \circ X' E_i)$};
\node (b2) at (4,3.5) {$S_i(M'_j) \circ E_iS_i(X'E_i)$};
\node (b3) at (4,4.5) {$E_i(S_i(M'_j) \circ S_i(X'E_i))$};
\node (c2) at (8,3.5) {$F_jE_iS_i(X'E_i)$};
\node (c3) at (8,4.5) {$E_iF_jS_i(X'E_i)$};
\node (d3) at (16,4.5) {$E_iF_jF_iS_i(X')$};
\draw[<-,line width=1.3pt] (a) -- node {$\scriptstyle \theta$} (b);
\draw[->,line width=1.3pt] (b) -- node {$\scriptstyle t_{i,j}^{\varrho(\beta-\alpha_i)}\zeta_j$} (c);
\draw[line width=1.3pt] (c) -- node {$\scriptstyle \sim$} (c');
\draw (c') -- node {$\scriptstyle \sim$} (d);
\draw[<-,line width=1.3pt] (a) -- node {$\scriptstyle \sigma'_{i,j}$} (e);
\draw[<-] (d) -- node {$\xy 0;/r.12pc/: (0,0)*{\dcross{i}{j}}; \endxy$} (h);
\draw[<-,line width=1.3pt] (c') -- node {$\xy 0;/r.12pc/: (0,0)*{\dcross{i}{j}}; \endxy$} (g');
\draw[line width=1.3pt] (e) -- node {$\scriptstyle \sim$} (f);
\draw[<-,line width=1.3pt] (f) -- node {$\scriptstyle \theta$} (g); 
\draw[->,line width=1.3pt] (g) -- node {$\scriptstyle t_{i,j}^{\varrho(\beta)}\zeta_j$} (g');
\draw (g') -- node {$\scriptstyle \sim$} (h);
\draw[->] (a1) -- node {$\xy 0;/r.12pc/: (0,0)*{\rcross{i}{i}}; \endxy$} (a);
\draw[->] (b1) -- node {$\xy 0;/r.12pc/: (0,0)*{\rcross{i}{i}}; \endxy$} (b);
\draw[->] (c1) -- node {$\xy 0;/r.12pc/: (0,0)*{\rcross{i}{i}}; \endxy$} (c);
\draw[->] (d1) -- node {$\xy 0;/r.12pc/: (0,0)*{\rcross{i}{i}}; \endxy$} (d);
\draw[<-] (a1) -- node {$\scriptstyle \theta$} (b1);
\draw[->] (b1) -- node {$\scriptstyle t_{i,j}^{\varrho(\beta-\alpha_i)}\zeta_j$} (c1);
\draw (c2) -- node[sloped] {$\scriptstyle \sim$} (d1); 
\draw[->] (e1) -- node[sloped] {$\scriptstyle \sim$} (e);
\draw[->] (f1) -- node[sloped] {$\scriptstyle \sim$} (f);
\draw (g1) -- node[sloped] {$\scriptstyle \sim$} (g);
\draw (f2) -- node[sloped] {$\scriptstyle \sim$} (f1);
\draw[->] (g2) -- (g1);
\draw[<-] (f2) -- node {$\scriptstyle \theta$} (g2);
\draw[double distance=2pt] (h1) -- (h);
\draw[->] (h2) -- node {$\xy 0;/r.10pc/: (0,0)*{\rcross{i}{j}}; \endxy$} (h1); 
\draw[->] (g2) -- node {$\scriptstyle t_{i,j}^{\varrho(\beta)}\zeta_j$} (h2);
\draw[->] (e3) -- node {$\xy 0;/r.12pc/: (0,0)*{\rcross{i}{i}}; \endxy$} (e1); 
\draw[->] (f3) -- node {$\xy 0;/r.12pc/: (0,0)*{\rcross{i}{i}}; \endxy$} (f2);
\draw[->] (g3) -- node {$\xy 0;/r.12pc/: (0,0)*{\rcross{i}{i}}; \endxy$} (g2);
\draw[->] (h3) -- node {$\xy 0;/r.12pc/: (0,0)*{\rcross{i}{i}}; \endxy$} (h2);
\draw (e3) -- node {$\scriptstyle \sim$} (f3);
\draw[<-] (f3) -- node {$\scriptstyle \theta$} (g3);
\draw[->] (g3) -- node {$\scriptstyle t_{i,j}^{\varrho(\beta)}\zeta_j$} (h3); 
\draw[->] (a2) -- node[sloped] {$\scriptstyle \sim$} (a1);
\draw (a3) -- node[sloped] {$\scriptstyle \sim$} (a2);
\draw[<-] (a3) -- node {$\scriptstyle \theta$} (b3);
\draw (b2) -- node[sloped] {$\scriptstyle \sim$} (b1);
\draw[->] (b3) -- (b2);
\draw[->] (b3) -- node {$\scriptstyle t_{i,j}^{\varrho(\beta-\alpha_i)}\zeta_j$} (c3);
\draw (c3) -- node {$\scriptstyle \sim$} (d3);
\draw[->] (d3) -- node {$\xy 0;/r.12pc/: (0,0)*{\rcross{i}{j}}; \endxy$} (d1);
\draw[->] (c3) -- node {$\xy 0;/r.10pc/: (0,0)*{\rcross{i}{j}}; \endxy$} (c2);
\draw (c2) -- node[sloped] {$\scriptstyle \sim$} (c1);
\draw[<-] (a2) to[bend right=80] node {$\sigma'_{i,j}$} (e3); 
\draw[<-] (d3) to[bend left=80] node {$\xy 0;/r.12pc/: (0,0)*{\dcross{i}{j}}; \endxy$} (h3);
\draw[->] (g1) -- node {$\scriptstyle t_{i,j}^{\varrho(\beta)}\zeta_j$} (h1);
\draw[->] (b2) -- node {$\scriptstyle t_{i,j}^{\varrho(\beta-\alpha_i)}\zeta_j$} (c2);
\draw (e1) -- node {$\scriptstyle \sim$} (f1);
\draw[->] (e3) to[bend left=60] node {$\scriptstyle S_i(g_1)$} (a);
\draw[->] (h3) to[bend right=60] node {$\scriptstyle g_2$} (d);
%\node (Q) at (-1,0) {(Q)};
%\node (R) at (18,0) {(R)};
\end{tikzpicture}
}
}
\caption{} \label{fig:5}
\end{figure}

The thick diagrams in Figure \ref{fig:4} and \ref{fig:5} are the outer diagram of (\ref{pic:10}). 
Note that \ref{fig:4} and \ref{fig:5} are quite similar to \ref{fig:1} and \ref{fig:2} respectively. 
In order to prove that the outer diagram of (\ref{pic:10}) commutes, it suffices to verify the following assertions: 
\begin{itemize}
\item (O), (P) and the outer diagram of Figure \ref{fig:4} commute. 
\item In Figure \ref{fig:5}, the outer diagram and the rectangular diagram containing curved edges $S_i(g_1)$ and $g_2$ are commutative. 
\end{itemize}

\begin{remark}
In Figure \ref{fig:5}, the leftmost two inner diagrams and the rightmost two inner diagrams are not commutative.
\end{remark}

Commutativity of (O). 
We may disregard $S_i$.
Since $M'_j \circ X'F_iE_i \xrightarrow{\sigma'_{j,i}} (M'_j \circ X'F_i)E_i$ is injective, it suffices to prove the commutativity after postcomposing $\sigma'_{j,i}$. 
Let $u \in M'_j, v \in X'$. 
Note that we have the following commutative diagram by Lemma \ref{lem:adjointSES} and : 
\begin{equation*}
\begin{tikzcd}
M'_j \circ X'F_iE_i \arrow[r,"\sigma'_{j,i}"]\arrow[d,"{Q_{i,j}\left(\xy 0;/r.12pc/: (0,0)*{\sdotu{i}}; \endxy, y_j \right)}"'] & (M'_j \circ X'F_i)E_i \arrow[d,"{Q_{i,j}\left(\xy 0;/r.12pc/: (0,0)*{\sdotu{i}}; \endxy, y_j \right)}"] \\
M'_j \circ X'F_iE_i \arrow[r,"\sigma'_{j,i}"] & (M'_j \circ X'F_i)E_i
\end{tikzcd}
\end{equation*}
Hence, under the homomorphism $M'_j \circ X' \xrightarrow{\xy 0;/r.12pc/: (0,0)*{\rcup{i}};(-2,-0.5)*{\bullet}; (-6,-1)*{\scriptstyle n}; \endxy} M'_j \circ X'F_iE_i \xrightarrow{Q_{i,j}\left(\xy 0;/r.12pc/: (0,0)*{\sdotu{i}}; \endxy, y_j \right)} M'_j\circ X'F_iE_i \xrightarrow{\sigma'_{j,i}} (M'_j \circ X'F_i) E_i$, 
the element $u \boxtimes v$ is sent to $Q_{i,j}\left(\xy 0;/r.12pc/: (0,0)*{\sdotu{i}}; \endxy, y_j \right) [u \boxtimes (v \boxtimes x_1^ne(i))]E_i$, following
\begin{align*}
u \boxtimes v &\mapsto u \boxtimes (v \boxtimes x_1^ne(i))E_i \\ 
&\mapsto [u \boxtimes (v \boxtimes x_1^n e(i))]E_i \in (M'_j \circ X'F_i)E_i \\
&\mapsto  Q_{i,j}\left(\xy 0;/r.12pc/: (0,0)*{\sdotu{i}}; \endxy, y_j \right) [u \boxtimes (v \boxtimes x_1^ne(i))]E_i. 
\end{align*}
On the other hand, under the homomorphism $M'_j \circ X' \xrightarrow{\xy 0;/r.12pc/: (0,0)*{\rcup{i}};(-2,-0.5)*{\bullet}; (-6,-1)*{\scriptstyle n}; \endxy} (M'_j \circ X')F_iE_i \to (M'_j \circ X'F_i)E_i \xrightarrow{\sigma'_{i,j}} M'_j \circ X'F_iE_i \xrightarrow{\sigma'_{j,i}} (M'_j \circ X'F_i)E_i$, 
it is also sent to $Q_{i,j}\left(\xy 0;/r.12pc/: (0,0)*{\sdotu{i}}; \endxy, y_j \right) [u \boxtimes (v \boxtimes e(i))]E_i$, following 
\begin{align*}
u \boxtimes v &\mapsto ((u \boxtimes v)\boxtimes x_1^ne(i))E_i \\
&\mapsto [u \boxtimes (v\boxtimes x_1^ne(i))]E_i \\
&\mapsto Q_{i,j}\left(\xy 0;/r.12pc/: (0,0)*{\sdotu{i}}; \endxy, y_j \right) [u \boxtimes (v \boxtimes x_1^ne(i))]E_i, 
\end{align*}
since the composition $\sigma'_{j,i} \circ \sigma'_{i,j}$ is $ Q_{i,j}\left(\xy 0;/r.12pc/: (0,0)*{\sdotu{i}}; \endxy, y_j \right)$ (Definition \ref{def:anothertauij}). 
 
Commutativity of (P) follows from 
\begin{align*}
\xy 0;/r.12pc/: (0,0)*{\xybox{
(0,0)*{\rcross{}{}};
(-8,-7)*{\rcup{}};
(4,-8)*{\slined{j}};
(-12,0)*{\slined{}};
(-8,8)*{\dcross{}{}};
(4,8)*{\slineu{}};
(4,14)*{\scriptstyle i};
(-10,-7.5)*{\bullet};
(-14,-8)*{\scriptstyle n};
}};
\endxy &= \xy 0;/r.12pc/: (0,0)*{\xybox{
(0,0)*{\dcross{}{}};
(-4,-8)*{\slined{j}};
(8,-7)*{\rcup{}};
(-8,7)*{\rcap{}};
(-16,1)*{\rcup{}};
(-20,.5)*{\bullet};
(-24,0)*{\scriptstyle n};
(12,-4); (12,20) **\dir{-} ?(1)*\dir{>};
(-20,4)*{};(-12,12)*{} **\crv{(-20,7) & (-12,9)}?(0)*\dir{<};
(4,4)*{};(-4,12)*{} **\crv{(4,7) & (-4,9)}?(0)*\dir{<};
(-8,16)*{\dcross{}{}};
(12,22)*{\scriptstyle i}; 
}};
\endxy = \xy 0;/r.12pc/: (0,0)*{\xybox{
(0,0)*{\dcross{}{}};
(0,8)*{\dcross{}{}};
(-4,-8)*{\slined{j}};
(8,-7)*{\rcup{}};
(12,-4); (12,12) **\dir{-} ?(1)*\dir{>};
(12,14)*{\scriptstyle i};
(-4,4)*{\bullet};
(-8,4)*{\scriptstyle n};
}};
\endxy = \xy 0;/r.12pc/: (0,0)*{\xybox{
(0,0)*{\dcross{}{}};
(0,8)*{\dcross{}{}};
(-4,-8)*{\slined{j}};
(8,-7)*{\rcup{}};
(12,-4); (12,12) **\dir{-} ?(1)*\dir{>};
(12,14)*{\scriptstyle i}; 
(6,-7.5)*{\bullet};
(4,-9)*{\scriptstyle n};
}};
\endxy \\
&= \left[ F_jS_i(X') \xrightarrow{\xy 0;/r.12pc/: (0,0)*{\rcup{i}};(-2,-0.5)*{\bullet}; (-6,-1)*{\scriptstyle n}; \endxy} F_jF_iE_iS_i(X') \xrightarrow{Q_{i,j}\left(\xy 0;/r.12pc/: (0,0)*{\sdotd{i}}; \endxy, \xy 0;/r.12pc/: (0,0)*{\sdotd{j}}; \endxy \right)E_iS_i(X')} F_jF_iE_iS_i(X') \right]. 
\end{align*}

Commutativity of the outer diagram of Figure \ref{fig:4} follows from $\rho(\beta) - \rho(\beta-\alpha_i) = 1$.
Commutativity of the outer diagram of Figure \ref{fig:5} follows from the induction hypothesis. 

Commutativity of the rectangular diagram containing curved edges $S_i(g_1)$ and $g_2$ in Figure \ref{fig:5}. 
Consider Figure \ref{fig:6}, 
where 
\begin{align*}
h_1 &= \overline{Q}_{i,j,i}\left(S_i(M'_j \circ X'F_iE_i) \xy 0;/r.12pc/: (0,0)*{\sdotu{i}}; \endxy F_i, S_i(y_j \circ X'F_iE_i)E_iF_i, S_i(M'_j \circ X'F_i\xy 0;/r.12pc/: (0,0)*{\sdotu{i}}; \endxy)E_iF_i \right), \\
h_2 &= \overline{Q}_{i,j,i}\left(E_i \xy 0;/r.12pc/: (0,0)*{\sdotd{i}}; \endxy S_i(M'_j\circ X'F_iE_i), E_iF_iS_i(y_j \circ X'F_iE_i), E_iF_iS_i(M'_j \circ X'F_i \xy 0;/r.12pc/: (0,0)*{\sdotu{i}}; \endxy) \right), \\
h_3 &= \overline{Q}_{i,j,i}\bigg(E_i \xy 0;/r.12pc/: (0,0)*{\sdotd{i}}; \endxy(S_i(M'_j)\circ S_i(X'F_iE_i)), E_iF_i(S_i(y_j)\circ S_i(X'F_iE_i)), \\
&\quad E_iF_i(S_i(M'_j) \circ S_i(X'F_i\xy 0;/r.12pc/: (0,0)*{\sdotu{i}}; \endxy)) \bigg), \\
h_4 &= \overline{Q}_{i,j,i}\left(E_i\xy 0;/r.12pc/: (0,0)*{\sdotd{i}}; \endxy F_jS_i(X'F_iE_i), E_iF_i\xy 0;/r.12pc/: (0,0)*{\sdotd{j}}; \endxy S_i(X'F_iE_i), E_iF_iF_jS_i(X'F_i\xy 0;/r.12pc/: (0,0)*{\sdotu{i}}; \endxy ) \right), \\
h_5 &= \overline{Q}_{i,j,i}\left(E_i\xy 0;/r.12pc/: (0,0)*{\sdotd{i}}; \endxy F_jF_iE_iS_i(X'), E_iF_i\xy 0;/r.12pc/: (0,0)*{\sdotd{j}}; \endxy F_iE_iS_i(X'), E_iF_iF_j\xy 0;/r.12pc/: (0,0)*{\sdotd{i}}; \endxy E_iS_i(X')  \right). 
\end{align*}

\begin{figure} %\label{pic:13}
\scalebox{0.9}{
\rotatebox{90}{
\begin{tikzpicture}[auto]
\node (a) at (0,0) {$S_i(M'_j \circ X'F_iE_i)$};
\node (b) at (9,0) {$S_i(M'_j) \circ S_i(X'F_iE_i)$};
\node (c) at (13.5,0) {$F_j S_i(X'F_iE_i)$};
%\node (d) at (13.5,0) {$F_j F_iS_i(X'F_i)$};
\node (e) at (18,0) {$F_jF_iE_iS_i(X')$};
\node (a1) at (0,-1.5) {$S_i((M'_j \circ X'F_iE_i)E_iF_i)$};
\node (b1) at (4.5,-1.5) {$E_iF_iS_i(M'_j \circ X'F_iE_i)$};
\node (c1) at (9,-1.5) {$E_iF_i(S_i(M'_j) \circ S_i(X'F_iE_i))$};
\node (d1) at (13.5,-1.5) {$E_iF_iF_jS_i(X'F_iE_i)$};
\node (e1) at (18,-1.5) {$E_iF_iF_jF_iE_iS_i(X')$};
\node (a2) at (0,-3) {$S_i((M'_j \circ X'F_iE_i)E_iF_i)$};
\node (b2) at (4.5,-3) {$E_iF_iS_i(M'_j \circ X'F_iE_i)$};
\node (c2) at (9,-3) {$E_iF_i(S_i(M'_j) \circ S_i(X'F_iE_i))$};
\node (d2) at (13.5,-3) {$E_iF_iF_jS_i(X'F_iE_i)$};
\node (e2) at (18,-3) {$E_iF_iF_jF_iE_iS_i(X')$};
\node (a3) at (0,-4.5) {$S_i((M'_j \circ X')E_iF_i)$};
\node (b3) at (4.5,-4.5) {$E_iF_iS_i(M'_j \circ X')$};
\node (c3) at (9,-4.5) {$E_iF_i(S_i(M'_j) \circ S_i(X'))$};
\node (d3) at (13.5,-4.5) {$E_iF_iF_jS_i(X')$};
\node (e3) at (18,-4.5) {$E_iF_iF_jS_i(X')$};
\draw[<-] (a) -- node {$\scriptstyle c_{i,-(\alpha_j+s_i(\beta-\alpha_i))}^{-1}\theta$} (b);
\draw[->] (b) -- node {$\scriptstyle \zeta_j$} (c);
%\draw[->,color=red] (c) -- node {$\scriptstyle $} (d);
\draw[-] (c) -- node {$\scriptstyle \sim$} (e);
\draw[->] (a1) -- node {$\xy 0;/r.12pc/: (0,0)*{\rcap{i}}; \endxy$} (a); 
\draw[->] (c1) -- node {$\xy 0;/r.12pc/: (0,0)*{\rcap{i}}; \endxy$} (b);
\draw[->] (d1) -- node {$\xy 0;/r.12pc/: (0,0)*{\rcap{i}}; \endxy$} (c);
\draw[->] (e1) -- node {$\xy 0;/r.12pc/: (0,0)*{\rcap{i}}; \endxy$} (e);
\draw[-] (a1) -- node {$\scriptstyle \sim$} (b1);
\draw[<-] (b1) -- node {$\scriptstyle \theta$} (c1);
\draw[->] (c1) -- node {$\scriptstyle \zeta_j$} (d1);
\draw[-] (d1) -- node {$\scriptstyle \sim$} (e1);
\draw[->] (a2) -- node {$\scriptstyle h_1$} (a1);
\draw[->] (b2) -- node {$\scriptstyle h_2$} (b1);
\draw[->] (c2) -- node {$\scriptstyle h_3$} (c1);
\draw[->] (d2) -- node {$\scriptstyle h_4$} (d1);
\draw[->] (e2) -- node {$\scriptstyle h_5$} (e1);
\draw[-] (a2) -- node {$\scriptstyle \sim$} (b2); 
\draw[<-] (b2) -- node {$\scriptstyle \theta$} (c2);
\draw[->] (c2) -- node {$\scriptstyle \zeta_j$} (d2);
\draw[-] (d2) -- node {$\scriptstyle \sim$} (e2);
\draw[->] (a3) -- node {$\xy 0;/r.12pc/: (0,0)*{\rcup{i}}; \endxy$} (a2); 
\draw[->] (b3) -- node {$\xy 0;/r.12pc/: (0,0)*{\rcup{i}}; \endxy$} (b2);
\draw[->] (c3) -- node {$\xy 0;/r.12pc/: (0,0)*{\rcup{i}}; \endxy$} (c2);
\draw[->] (d3) -- node {$\xy 0;/r.12pc/: (0,0)*{\rcup{i}}; \endxy$} (d2);
\draw[->] (e3) -- node {$\xy 0;/r.12pc/: (0,0)*{\rcup{i}}; \endxy$} (e2);
\draw[->] (a3) -- node {$\scriptstyle \sim$} (b3); 
\draw[<-] (b3) -- node {$\scriptstyle \theta$} (c3);
\draw[->] (c3) -- node {$\scriptstyle \zeta_j$} (d3);
\draw[->] (d3) -- node {$\scriptstyle c_{i,-(\beta-\alpha_i)}$} (e3);
\draw[->] (a3) to[bend left = 70] node {$S_i(g_1)$} (a);
\draw[->] (e3) to[bend right=70] node {$g_2$} (e); 
\node (Q) at (-1,-2.2) {(Q)}; 
\node (R) at (19,-2.2) {(R)};
\end{tikzpicture} 
}
}
\caption{} \label{fig:6}
\end{figure}

Since
\[
c_{i,-(\alpha_j+s_i(\beta-\alpha_i))}/c_{i,-(\beta-\alpha_i)} = t_{i,j}^{-1} = t_{i,j}^{\varrho(\beta-\alpha_i)-\varrho(\beta)}, 
\]
it suffices to prove that the outer diagram of Figure \ref{fig:6} commutes.
It is further reduced to proving that all the inner diagrams of Figure \ref{fig:6} commute. 

Commutativity of (Q). 
By identifying 
\begin{align*}
&(M'_j \circ X')E_iF_i = E_iF_iF_jX', \ (M'_j \circ X')F_iE_i = F_iE_iF_jX', \\
&M'_j \circ X'F_iE_i = F_jF_iE_iX',\  M'_j \circ X'E_iF_i = F_jE_iF_iX' \\
&(M'_j \circ X'E_i)F_i = E_iF_jF_iX',
\end{align*}
we have 
\begin{align*}
g_1 &= t_{i,j} \left( \xy 0;/r.12pc/: (0,0)*{\xybox{
(0,0)*{\rcross{i}{i}};
(12,4); (12,-4) **\dir{-} ?(1)*\dir{>};
(12,-6)*{\scriptstyle j};
(8,8)*{\rcross{}{}};
(-4,12); (-4,4) **\dir{-} ?(1)*\dir{>};
(0,16)*{\dcross{}{}};
(12,12); (12,20) **\dir{-} ?(1)*\dir{>};
}}; 
\endxy  - \xy 0;/r.12pc/: (0,0)*{\xybox{
(0,0)*{\dcross{i}{j}};
(-12,0)*{\slineu{i}};
(-8,8)*{\rcross{}{}};
(4,8)*{\slined{}};
(-12,16)*{\slined{}};
(0,16)*{\rcross{}{}};
}};
\endxy \right) X' \quad \text{by Lemma \ref{lem:formula}} \\
&= t_{i,j} \left( \xy 0;/r.12pc/: (0,0)*{\xybox{
(0,0)*{\dcross{}{}};
(-8,7)*{\rcap{}};
(8,-8)*{\dcross{}{}};
(16,-15)*{\rcup{}};
(20,-12); (20,12) **\dir{-} ?(1)*\dir{>};
(12,0)*{\slined{}};
(8,8)*{\dcross{}{}};
(-12,-20); (-12,4) **\dir{-} ?(1)*\dir{>};
(-4,-4); (-4,-20) **\dir{-} ?(1)*\dir{>};
(4,-16)*{\slined{j}};
(-12,-22)*{\scriptstyle i};
(-4,-22)*{\scriptstyle i};
}};
\endxy - \xy 0;/r.12pc/: (0,0)*{\xybox{
(-4,-8)*{\dcross{i}{j}}; 
(4,0)*{\dcross{}{}};
(-4,8)*{\dcross{}{}};
(-8,0)*{\slined{}};
(8,8)*{\slined{}};
(12,-7)*{\rcup{}};
(16,-4); (16,20) **\dir{-} ?(1)*\dir{>};
(-12,15)*{\rcap{}};
(-16,-12); (-16,12) **\dir{-} ?(1)*\dir{>};
(0,16)*{\slined{}};
(8,16)*{\slined{}};
(-16,-14)*{\scriptstyle i};
}};
\endxy \right) X' \quad \text{by computing as in (\ref{eq:6})} \\
&= t_{i,j} \bigg[
(M'_j \circ X')E_iF_i \xrightarrow{c_{i,-(\beta -\alpha_i)}^{-1} \xy 0;/r.12pc/: (0,0)*{\rcup{i}}; \endxy} (M'_j \circ X'F_iE_i)E_iF_i \\ 
&\quad \xrightarrow{h_1} (M'_j \circ X'F_iE_i)E_iF_i \quad \xrightarrow{c_{i,-(\beta+s_i\alpha_j-\alpha_i)}\xy 0;/r.12pc/: (0,0)*{\rcap{i}}; \endxy} M'_j \circ X'F_iE_i \bigg] \\
%& \quad \text{where $h_1 = \overline{Q}_{i,j,i}\left((M'_j \circ X'E_iF_i) \xy 0;/r.12pc/: (0,0)*{\sdotu{i}}; \endxy F_i, (z_j \circ X'F_iE_i)E_iF_i, (M'_j \circ X'F_i\xy 0;/r.12pc/: (0,0)*{\sdotu{i}}; \endxy)E_iF_i \right)$}\\
&= \bigg[(M'_j \circ X')E_iF_i \xrightarrow{ \xy 0;/r.12pc/: (0,0)*{\rcup{i}}; \endxy} (M'_j \circ X'F_iE_i)E_iF_i \\
&\quad \xrightarrow{h_1} (M'_j \circ X'F_iE_i)E_iF_i \xrightarrow{\xy 0;/r.12pc/: (0,0)*{\rcap{i}}; \endxy} M'_j \circ X'F_iE_i \bigg].
\end{align*}

Commutativity of (R) follows from
\begin{align*}
g_2 &= \left( \xy 0;/r.12pc/: (0,0)*{\xybox{
(0,0)*{\rcross{i}{i}};
(12,4); (12,-4) **\dir{-} ?(1)*\dir{>};
(12,-6)*{\scriptstyle j};
(8,8)*{\rcross{}{}};
(-4,12); (-4,4) **\dir{-} ?(1)*\dir{>};
(0,16)*{\dcross{}{}};
(12,12); (12,20) **\dir{-} ?(1)*\dir{>};
}}; 
\endxy  - \xy 0;/r.12pc/: (0,0)*{\xybox{
(0,0)*{\dcross{i}{j}};
(-12,0)*{\slineu{i}};
(-8,8)*{\rcross{}{}};
(4,8)*{\slined{}};
(-12,16)*{\slined{}};
(0,16)*{\rcross{}{}};
}};
\endxy \right) S_i(X') \\
&= \left( \xy 0;/r.12pc/: (0,0)*{\xybox{
(0,0)*{\dcross{}{}};
(-8,7)*{\rcap{}};
(8,-8)*{\dcross{}{}};
(16,-15)*{\rcup{}};
(20,-12); (20,12) **\dir{-} ?(1)*\dir{>};
(12,0)*{\slined{}};
(8,8)*{\dcross{}{}};
(-12,-20); (-12,4) **\dir{-} ?(1)*\dir{>};
(-4,-4); (-4,-20) **\dir{-} ?(1)*\dir{>};
(4,-16)*{\slined{j}};
(-12,-22)*{\scriptstyle i};
(-4,-22)*{\scriptstyle i};
}};
\endxy - \xy 0;/r.12pc/: (0,0)*{\xybox{
(-4,-8)*{\dcross{i}{j}}; 
(4,0)*{\dcross{}{}};
(-4,8)*{\dcross{}{}};
(-8,0)*{\slined{}};
(8,8)*{\slined{}};
(12,-7)*{\rcup{}};
(16,-4); (16,20) **\dir{-} ?(1)*\dir{>};
(-12,15)*{\rcap{}};
(-16,-12); (-16,12) **\dir{-} ?(1)*\dir{>};
(0,16)*{\slined{}};
(8,16)*{\slined{}};
(-16,-14)*{\scriptstyle i};
}};
\endxy \right)S_i(X') \quad \text{by computing as in (\ref{eq:6})} \\
&= \bigg[E_iF_iF_jS_i(X') \xrightarrow{\xy 0;/r.12pc/: (0,0)*{\rcup{i}}; \endxy} E_iF_iF_jF_iE_iS_i(X')  \\ 
&\quad \xrightarrow{h_5} E_iF_iF_jF_iE_iS_i(X') \xrightarrow{\xy 0;/r.12pc/: (0,0)*{\rcap{i}}; \endxy} F_jF_iE_iS_i(X')\bigg].
\end{align*}

It is easy to verify the commutativity of the other inner diagrams of Figure \ref{fig:6}. 

The case $k \neq i$. 
Consider Figure \ref{fig:7},
where $f_1, f_2, f_3, f_4$ are the same as those of Figure \ref{fig:3}. 
\begin{figure}% \label{pic:14}
\scalebox{0.95}{
\rotatebox{90}{
\begin{tikzpicture}[auto]
\node (a) at (0,1) {$S_i(M'_j \circ X'F_kE_i)$};
\node (b) at (4,1) {$S_i(M'_j) \circ S_i(X'F_kE_i)$};
\node (c) at (8,1) {$F_j S_i(X'F_kE_i)$};
\node (c') at (12,1) {$F_j F_iS_i(X'F_k)$};
\node (d) at (16,1) {$F_jF_i(S_i(X')\circ M_k)$};
\node (e) at (0,-0.5) {$S_i((M'_j \circ X'F_k)E_i)$};
\node (f) at (4,-0.5) {$F_iS_i(M'_j \circ X'F_k)$};
\node (g) at (8,-0.5) {$F_i(S_i(M'_j) \circ S_i(X'F_k))$};
\node (g') at (12,-0.5) {$F_iF_j S_i(X'F_k)$}; 
\node (h) at (16,-0.5) {$F_iF_j(S_i(X') \circ M_k)$};
\draw[<-,line width=1.3pt] (a) -- node {$\scriptstyle \theta$} (b);
\draw[->,line width=1.3pt] (b) -- node {$\scriptstyle t_{i,j}^{\varrho(\beta-\alpha_i)}\zeta_j$} (c);
\draw[line width=1.3pt] (c) -- node {$\scriptstyle \sim$} (c');
\draw (c') -- node {$\scriptstyle \sim$} (d);
\draw[<-,line width=1.3pt] (a) -- node {$\scriptstyle \sigma'_{i,j}$} (e);
\draw[<-] (d) -- node {$\xy 0;/r.12pc/: (0,0)*{\dcross{i}{j}}; \endxy$} (h);
\draw[<-,line width=1.3pt] (c') -- node {$\xy 0;/r.12pc/: (0,0)*{\dcross{i}{j}}; \endxy$} (g');
\draw[line width=1.3pt] (e) -- node {$\scriptstyle \sim$} (f);
\draw[<-,line width=1.3pt] (f) -- node {$\scriptstyle \theta$} (g); 
\draw[->,line width=1.3pt] (g) -- node {$\scriptstyle t_{i,j}^{\varrho(\beta)}\zeta_j$} (g');
\draw (g') -- node {$\scriptstyle \sim$} (h);
\node (a1) at (0,2.5) {$S_i(M'_j \circ X'E_iF_k)$};
\node (b1) at (4,2.5) {$S_i(M'_j)\circ S_i(X'E_iF_k)$};
\node (c1) at (8,2.5) {$F_jS_i(X'E_iF_k)$};
\node (d1) at (16,2.5) {$F_j(F_iS_i(X')\circ M_k)$};
\node (a2) at (0,4) {$S_i((M'_j \circ X'E_i)F_k)$}; 
\node (b2) at (4,4) {$S_i(M'_j \circ X'E_i) \circ M_k$};
\node (c2) at (8,4) {$S_i(M'_j) \circ S_i(X'E_i) \circ M_k$};
\node (c'2) at (12,4) {$F_jS_i(X'E_i) \circ M_k$};
\node (d2) at (16,4) {$F_jF_iS_i(X') \circ M_k$};
\node (e1) at (0,-2) {$S_i((M'_j\circ X')F_kE_i)$};
\node (f1) at (4,-2) {$F_iS_i((M'_j \circ X')F_k)$};
\node (g1) at (8,-2) {$F_i(S_i(M'_j \circ X') \circ M_k)$};
\node (g'1) at (12,-2) {$F_i(S_i(M'_j)\circ S_i(X') \circ M_k)$};
\node (h1) at (16,-2) {$F_i(F_jS_i(X') \circ M_k)$};
\node (e2) at (0,-3.5) {$S_i((M'_j \circ X')E_iF_k)$};
\node (f2) at (4,-3.5) {$S_i((M'_j \circ X')E_i) \circ M_k$};
\node (g2) at (8,-3.5) {$F_iS_i(M'_j \circ X') \circ M_k$};
\node (g'2) at (12,-3.5) {$F_i(S_i(M'_j) \circ S_i(X')) \circ M_k$};
\node (h2) at (16,-3.5) {$F_iF_jS_i(X') \circ M_k$};
\draw[->] (a1) -- node {$\xy 0;/r.12pc/: (0,0)*{\rcross{i}{k}}; \endxy$} (a);
\draw[->] (b1) -- node {$\xy 0;/r.12pc/: (0,0)*{\rcross{i}{k}}; \endxy$} (b);
\draw[->] (c1) -- node {$\xy 0;/r.12pc/: (0,0)*{\rcross{i}{k}}; \endxy$} (c);
\draw[->] (d1) -- node {$f_1$} (d);
\draw[<-] (a1) -- node {$\scriptstyle \theta$} (b1);
\draw[->] (b1) -- node {$\scriptstyle t_{i,j}^{\varrho(\beta-\alpha_i)}\zeta_j$} (c1);
\draw (c1) -- node {$\scriptstyle \sim$} (d1);
\draw (a2) -- node[sloped] {$\scriptstyle \sim$} (a1);
\draw (c2) -- node[sloped] {$\scriptstyle \sim$} (b1);
\draw (d2) -- node[sloped] {$\scriptstyle \sim$} (d1);
\draw (a2) -- node {$\scriptstyle \sim$} (b2);
\draw[<-] (b2) -- node {$\scriptstyle \theta$} (c2);
\draw[->] (c2) -- node {$\scriptstyle t_{i,j}^{\varrho(\beta-\alpha_i)}\zeta_j$} (c'2);
\draw (c'2) -- node {$\scriptstyle \sim$} (d2);
\draw[->] (e1) -- node[sloped] {$\scriptstyle \sim$} (e);
\draw[->] (f1) -- node[sloped] {$\scriptstyle \sim$} (f); 
\draw (g'1) -- node[sloped] {$\scriptstyle \sim$} (g);
\draw (h1) -- node[sloped] {$\scriptstyle \sim$} (h);
\draw (e1) -- node {$\scriptstyle \sim$} (f1);
\draw (f1) -- node {$\scriptstyle \sim$} (g1);
\draw[<-] (g1) -- node {$\scriptstyle \theta$} (g'1);
\draw[->] (e2) -- node {$\xy 0;/r.12pc/: (0,0)*{\rcross{i}{k}}; \endxy$} (e1);
\draw[->] (g2) -- node {$f_2$} (g1);
\draw[->] (g'2) -- node {$f_3$} (g'1);
\draw[->] (h2) -- node {$f_4$} (h1);
\draw (e2) -- node {$\scriptstyle \sim$} (f2);
\draw (f2) -- node {$\scriptstyle \sim$} (g2);
\draw[<-] (g2) -- node {$\scriptstyle \theta$} (g'2);
\draw[->] (g'2) -- node {$\scriptstyle t_{i,j}^{\varrho(\beta)}\zeta_j$} (h2);
\draw[<-] (d2) to[bend left = 70] node {$\xy 0;/r.12pc/: (0,0)*{\dcross{i}{j}};\endxy$} (h2);
\draw[<-] (a2) to[bend right = 70] node {$\sigma'_{i,j}$} (e2); 
\draw[->] (g'1) -- node {$\scriptstyle t_{i,j}^{\varrho(\beta)}$} (h1); 
\node (S) at (-1.3, 0.2) {(S)};
\node (T) at (17.5,0.2) {(T)};
%\draw[smooth,->] plot coordinates {(-1.5,3.8) (-3,2.5) (-3,-4) (0,-5) (3,-3.7)};
\end{tikzpicture}
}
}
\caption{} \label{fig:7}
\end{figure}
We need to prove that the thick diagram commutes.
The commutativity of the other inner diagrams except (S) and (T) was already proved when we considered Figure \ref{fig:3}. 
In addition, the outer diagram commutes by the induction hypothesis. 
Hence, it suffices to prove that (S) and (T) commute. 

Commutativity of (S). 
We may disregard $S_i$. 
Consider the following diagram: 
\begin{equation*}
\begin{tikzcd} 
(M'_j \circ X')E_iF_k \arrow[r,"\sigma'_{i,j}",line width=1pt]\arrow[d, "{\xy 0;/r.12pc/: (0,0)*{\rcross{i}{k}}; \endxy}",line width=1pt]\arrow[rr,bend left=40, "{Q_{i,j}\left(\xy 0;/r.12pc/: (0,0)*{\sdotu{i}}; \endxy, y_j \right)}"'] & (M'_j \circ X'E_i) F_k \arrow[r,"\sigma'_{j,i}"]\arrow[d,line width=1pt] & (M'_j \circ X') E_iF_k \arrow[d,"{\xy 0;/r.12pc/: (0,0)*{\rcross{i}{k}}; \endxy}"] \\
(M'_j \circ X')F_kE_i \arrow[d,line width=1pt] & M'_j \circ X'E_iF_k \arrow[d,"{\xy 0;/r.12pc/: (0,0)*{\rcross{i}{k}}; \endxy}",line width=1pt] & (M'_j \circ X')F_kE_i \arrow[d] \\
(M'_j \circ X'F_k) E_i \arrow[r,"\sigma'_{i,j}",line width=1pt] \arrow[rr,bend right=40, "{Q_{i,j}\left(\xy 0;/r.12pc/: (0,0)*{\sdotu{i}}; \endxy, y_j \right)}"] & M'_j \circ X'F_k E_i \arrow[r,"\sigma'_{j,i}"] & (M'_j \circ X'F_k)E_i  
\end{tikzcd}
\end{equation*}
We need to prove the commutativity of the thick diagram. 
The upper triangle and the lower triangle commute by Definition \ref{def:anothertauij}. 
The right square coincides with (L) in Figure \ref{fig:3}, hence it is commutative. 
It is immediate to see that the outer diagram also commutes, using 
\[
\xy 0;/r.17pc/: 
(0,0)*{\rcross{i}{k}}; 
(2,1.7)*{\bullet}; 
\endxy =  \xy 0;/r.17pc/: 
(0,0)*{\rcross{i}{k}}; 
(-2,-1.5)*{\bullet}; 
\endxy. 
\]
Since $\sigma'_{j,i}$ is injective, the thick diagram commutes. 

Commutativity of (T). 
Recall that 
\[
f_1^{-1} = F_j\left(S_i(X') \xy 0;/r.12pc/: (0,0)*{\lcross{k}{i}}; \endxy \right),\ f_4^{-1} = F_jS_i(X') \xy 0;/r.12pc/: (0,0)*{\lcross{k}{i}}; \endxy. 
\]
Let $u \in S_i(X'), v \in M_k$. 
Under the homomorphism $F_iF_j(S_i(X') \circ M_k) \xrightarrow{\xy 0;/r.12pc/: (0,0)*{\dcross{i}{j}}; \endxy} F_jF_i(S_i(X') \circ M_k) \xrightarrow{f_1^{-1}} F_j(F_iS_i(X') \circ M_k) \to F_jF_iS_i(X') \circ M_k$, 
the element $e(i,j) \boxtimes (u \boxtimes v)$ is sent to $ t_{i,k}^{-1} \tau_1 ((e(j,i) \boxtimes u) \boxtimes v)$, following 
\begin{align*}
e(i,j) \boxtimes (u \boxtimes v) &\mapsto \tau_1e(j,i) \boxtimes (u \boxtimes v) \\
&\mapsto t_{i,k}^{-1} \tau_1 (e(j) \boxtimes((e(i) \boxtimes u) \boxtimes v)) \quad \text{by Lemma \ref{lem:formula}} \\ 
&\mapsto t_{i,k}^{-1} \tau_1 ((e(j,i) \boxtimes u) \boxtimes v). 
\end{align*}
On the other hand, under the homomorphism $F_iF_j(S_i(X') \circ M_k) \to F_i(F_jS_i(X') \circ M_k) \xrightarrow{f_4^{-1}} F_iF_jS_i(X') \circ M_k \xrightarrow{\xy 0;/r.12pc/: (0,0)*{\dcross{i}{j}}; \endxy} F_jF_iS_i(X') \circ M_k$, 
the element $e(i,j) \boxtimes (u \boxtimes v)$ is also sent to $ t_{i,k}^{-1} \tau_1 ((e(j,i) \boxtimes u) \boxtimes v)$, following 
\begin{align*}
e(i,j) \boxtimes (u \boxtimes v) &\mapsto e(i) \boxtimes ((e(j)\boxtimes u) \boxtimes v) \\
&\mapsto t_{i,k}^{-1} (e(i,j) \boxtimes u) \boxtimes v \quad \text{by Lemma \ref{lem:formula}} \\
&\mapsto t_{i,k}^{-1} (\tau_1 (e(j,i) \boxtimes u)) \boxtimes v. 
\end{align*}
Now, Case 2 is complete. 

\subsubsection{Case 3} 
$\xy 0;/r.12pc/: (0,0)*{\dcross{j}{k}}; \endxy$ for $j,k \neq i$.
Let $X \in \gMod{{}_iR(\beta)}$. 
Consider the following diagram: 
\begin{equation} \label{pic:15}
\adjustbox{scale=0.8,center}{
\begin{tikzcd}
S_i(M'_j \circ M'_k) \circ S_i(X) \arrow[d,"\theta"]\arrow[ddddddd,bend right=70, "\sigma'_{j,k}"]\arrow[rd,phantom,"\text{Proposition \ref{prop:monoidality} (6)}"] && \\
S_i(M'_j \circ M'_k \circ X) \arrow[d,-, "\sim" sloped]\arrow[ddddd, bend right=60,"\sigma'_{j,k}"] & S_i(M'_j) \circ S_i(M'_k \circ X) \arrow[l,"\theta"]\arrow[d,-,"\sim" sloped] & S_i(M'_j) \circ S_i(M'_k) \circ S_i(X) \arrow[l,"\theta"]\arrow[d,"t_{i,k}^{\varrho(\beta)}\zeta_k"]\arrow[llu,"\theta", bend right=10] \\
S_i(M'_j \circ F_kX)\arrow[d,-,"\sim" sloped] & S_i(M'_j) \circ S_i(F_kX) \arrow[l,"\theta"]\arrow[r,"\kappa_k^-"]\arrow[d,"t_{i,j}^{\varrho(\beta+s_i\alpha_k)}\zeta_j"] & S_i(M'_j) \circ F_kS_i(X)\arrow[d,"t_{i,j}^{\varrho(\beta+s_i\alpha_k)}\zeta_j"] \\
S_i(F_jF_kX) \arrow[r,"\kappa_j^-", line width=1pt]\arrow[d,"{\xy 0;/r.12pc/: (0,0)*{\dcross{j}{k}}; \endxy}",line width=1pt] & F_jS_i(F_kX) \arrow[r,"\kappa_k^-",line width=1pt] & F_jF_kS_i(X) \arrow[d,"{\xy 0;/r.12pc/: (0,0)*{\dcross{j}{k}}; \endxy}",line width=1pt] \\
S_i(F_kF_jX) \arrow[r,"\kappa_k^-",line width=1pt] & F_kS_i(F_jX) \arrow[r,"\kappa_j^-",line width=1pt] & F_kF_jS_i(X) \\
S_i(M'_k \circ F_jX) \arrow[u,-,"\sim" sloped] & S_i(M'_k) \circ S_i(F_jX) \arrow[l,"\theta"]\arrow[r,"\kappa_j^-"]\arrow[u,"t_{i,k}^{\varrho(\beta+s_i\alpha_j)}\zeta_k"] & S_i(M'_k) \circ F_jS_i(X)\arrow[u,"t_{i,k}^{\varrho(\beta+s_i\alpha_j)}\zeta_k"] \\
S_i(M'_k \circ M'_j \circ X) \arrow[u,-,"\sim" sloped] & S_i(M'_k) \circ S_i(M'_j \circ X) \arrow[l,"\theta"]\arrow[u,-,"\sim" sloped] & S_i(M'_k) \circ S_i(M'_j) \circ S_i(X) \arrow[l,"\theta"]\arrow[u,"t_{i,j}^{\varrho(\beta)}\zeta_j"]\arrow[lld,"\theta", bend left=10] \\
S_i(M'_k \circ M'_j) \circ S_i(X) \arrow[u,"\theta"] \arrow[ru,phantom,"\text{Proposition \ref{prop:monoidality} (6)}"] &&
\end{tikzcd}
}
\end{equation}
We need to prove that the thick diagram commutes.
It is easy to see that all the other inner diagrams commute. 
Hence, it suffices to prove the commutativity of the outer diagram. 

We may assume $X = \mathbf{1}$. 
Put $m = -a_{i,j}, n = - a_{i,k}$. 
There is a canonical isomorphism
\begin{align*}
&M_j \circ M_k = F_i^{(m)}R(\alpha_j) \circ M_k \xleftarrow{\text{can}} F_i^{(m)} (R(\alpha_j) \circ M_k) \\ 
&= F_i^{(m)}F_jM_k = F_i^{(m)}F_jF_i^{(n)}R(\alpha_k) = F_i^{(m)}F_jF_i^{(n)}F_k \mathbf{1}, 
\end{align*}
where the second homomorphism is an isomorphism since $F_iM_k = 0$, see the proof of Proposition \ref{prop:extremalequiv}. 
We have a similar isomorphism $M_k \circ M_j \simeq F_i^{(n)}F_kF_i^{(m)}F_j\mathbf{1}$ by interchanging $j$ and $k$,
and isomorphisms $M'_j \circ M'_k \simeq \mathbf{1} F_jF_i^{(m)'}F_kF_i^{(n)'}, M'_k \circ M'_j \simeq \mathbf{1}F_kF_i^{(n)'}F_jF_i^{(m)'}$. 
We will freely use these isomorphisms.  

\begin{lemma} \label{lem:sigmaR}
(1) The isomorphisms above make the following diagram commute: 
\begin{equation*}
\begin{tikzcd}
M_j \circ M_k \arrow[d,"\sigma_{j,k}"]\arrow[r,-,"\sim"] & F_i^{(m)}F_jF_i^{(n)}F_k\mathbf{1} \arrow[r,"\mathrm{can}"] & F_i^mF_jF_i^nF_k\mathbf{1} \arrow[d,"t_{i,j}^{-n}\mathsf{R}"] \\
M_k \circ M_j \arrow[r,-,"\sim"] & F_i^{(n)}F_kF_i^{(m)}F_j \mathbf{1} & F_i^nF_kF_i^mF_j\mathbf{1} \arrow[l,"\mathrm{can}"], 
\end{tikzcd}
\end{equation*}
where 
\[
\mathsf{R} = \xy 0;/r.12pc/: 
(-4,-8)*{};(12,8)*{} **\crv{(-4,-2) & (12,2)}?(0)*\dir{<};
(-12,-8)*{};(4,8)*{} **\crv{(-12,-2) & (4,2)}?(0)*\dir{<};
(4,-8)*{};(-12,8)*{} **\crv{(4,-2) & (-12,2)}?(0)*\dir{<};
(12,-8)*{};(-4,8)*{} **\crv{(12,-2) & (-4,2)}?(0)*\dir{<};
(-12,-11)*{\scriptstyle i^m};
(-4,-11)*{\scriptstyle j};
(4,-11)*{\scriptstyle i^n};
(12,-11)*{\scriptstyle k};
\endxy \mathbf{1}. 
\]
(2) The isomorphisms above make the following diagram commute: 
\begin{equation*}
\begin{tikzcd}
M'_j \circ M'_k \arrow[d,"\sigma'_{j,k}"]\arrow[r,-,"\sim"] & \mathbf{1}F_jF_i^{(m)'}F_kF_i^{(n)'}\arrow[r,"\mathrm{can}"] & \mathsf{1}F_jF_i^mF_kF_i^n \arrow[d,"(-1)^{mn}t_{i,k}^{-m}\mathsf{R}'"] \\
M'_k \circ M'_j \arrow[r,-,"\sim"] & \mathbf{1}F_kF_i^{(n)'}F_jF_i^{(m)'} & \mathbf{1}F_kF_i^nF_jF_i^m \arrow[l,"\mathrm{can}"], 
\end{tikzcd}
\end{equation*}
where 
\[
\mathsf{R}' = \mathbf{1} \xy 0;/r.12pc/: 
(-4,-8)*{};(12,8)*{} **\crv{(-4,-2) & (12,2)}?(0)*\dir{<};
(-12,-8)*{};(4,8)*{} **\crv{(-12,-2) & (4,2)}?(0)*\dir{<};
(4,-8)*{};(-12,8)*{} **\crv{(4,-2) & (-12,2)}?(0)*\dir{<};
(12,-8)*{};(-4,8)*{} **\crv{(12,-2) & (-4,2)}?(0)*\dir{<};
(-12,-11)*{\scriptstyle j};
(-4,-11)*{\scriptstyle i^m};
(4,-11)*{\scriptstyle k};
(12,-11)*{\scriptstyle i^n};
\endxy
\]
\end{lemma}

\begin{proof}
(1) By Proposition \ref{prop:extremalequiv} and Definition \ref{def:anothersimple}, we have isomorphisms
\begin{align*}
M_j \circ M_k \simeq F_i^{(m)}R(\alpha_j) \circ F_i^{(n)}R(\alpha_k) \xleftarrow{\phi_-} F_i^{(m+n)}(R(\alpha_j)\circ R(\alpha_k)), \\
M_k \circ M_j \simeq F_i^{(n)}R(\alpha_k) \circ F_i^{(m)}R(\alpha_j) \xleftarrow{\phi_-} F_i^{(m+n)}(R(\alpha_k)\circ R(\alpha_j)), 
\end{align*}
under which the homomorphism $\sigma_{j,k}$ coincides with the one obtained by applying $F_i^{(m+n)}$ to
\[
R(\alpha_j) \circ R(\alpha_k) \simeq R(\alpha_j+\alpha_k)(j,k) \xrightarrow{\times \tau_1} R(\alpha_j+\alpha_k)e(k,j) \simeq R(\alpha_k)\circ R(\alpha_j). 
\]
To prove the commutativity, we compute the image of $b_+(i^{m+n}) \boxtimes (e(j) \boxtimes e(k)) \in F_i^{(m+n)}(R(\alpha_j)\circ R(\alpha_k))$. 
Under the homomorphism $F_i^{(m+n)}(R(\alpha_j)\circ R(\alpha_k)) \xrightarrow{\times \tau_1} F_i^{(m+n)}(R(\alpha_k)\circ R(\alpha_j))$, 
it is sent to $b_+(i^{m+n}) \boxtimes \tau_1e(k,j)$. 
On the other hand, under the homomorphism 
\begin{align*}
&F_i^{(m+n)}(R(\alpha_j) \circ R(\alpha_k)) \xrightarrow{\phi_-} F_i^{(m)}R(\alpha_j) \circ F_i^{(n)}R(\alpha_k) \simeq F_i^{(m)}F_jF_i^{(n)}F_k \mathbf{1} \\ 
&\xrightarrow{\text{can}} F_i^mF_jF_i^nF_k\mathbf{1} \xrightarrow{\mathsf{R}} F_i^nF_kF_i^mF_j \mathbf{1} \xrightarrow{\text{can}} F_i^{(n)}F_kF_i^{(m)}F_j \mathbf{1} \\
&\simeq F_i^{(n)}R(\alpha_k) \circ F_i^{(m)}R(\alpha_j) \xrightarrow{\phi_-^{-1}} F_i^{(n+m)}(R(\alpha_k) \circ R(\alpha_j)), 
\end{align*}
it is mapped as follows:
\begin{align*}
&b_+(i^{m+n}) \boxtimes e(j,k) \notag \\
&\mapsto b_+(i^{m+n}) (b_+(i^m) \boxtimes \tau_{w[1,n]}(e(j) \boxtimes b_+(i^n)) \boxtimes e(k) ) \in F_i^mF_jF_i^nF_k\mathbf{1} \notag \\ 
&\quad \text{by Lemma \ref{lem:phi-}} \notag \\
&\mapsto b_+(i^{m+n}) (b_+(i^m) \boxtimes \tau_{w[1,n]}(e(j)\boxtimes b_+(i^n)) \boxtimes e(k)) \tau_{w[n+1,m+1]} \in F_i^nF_kF_i^mF_j\mathbf{1} \notag \\
&\mapsto b_+(i^{m+n}) (b_+(i^m) \boxtimes \tau_{w[1,n]}(e(j)\boxtimes b_+(i^n)) \boxtimes e(k)) \tau_{w[n+1,m+1]}\times \notag \\
&\quad (b_+(i^n) \boxtimes e(k) \boxtimes b_+(i^m) \boxtimes e(j)) \in F_i^{(n)}F_kF_i^{(m)}F_j \mathbf{1} \quad \text{since $\varphi(b_-(i^n)) = b_+(i^n)$} \notag \\
%&= b_+(i^{m+n})(b_+(i^m) \boxtimes \tau_{w}[1,n] (e(j) \boxtimes b_+(i^n)) \boxtimes e(k)) \tau_{w[n+1,m+1]}(b_+(i^n) \boxtimes \tau_{w[1,m]}(e(k) \boxtimes b_+(i^m)) \boxtimes e(j)) \notag \\
&= b_+(i^{m+n}) (b_+(i^m) \boxtimes (b_+(i^n)\boxtimes e(j))\tau_{w[1,n]}\tau_{w[n,1]}  \boxtimes e(k))\tau_{w[n,m]} \times \notag \\
&\quad \tau_{m+n+1}\tau_{m+n} \cdots \tau_{n+1} \times (b_+(i^n) \boxtimes e(k) \boxtimes b_+(i^m) \boxtimes e(j)) \notag \\
&= b_+(i^{m+n}) \left(b_+(i^m) \boxtimes \left((b_+(i^n) \boxtimes e(j)) \prod_{1 \leq k \leq n} Q_{i,j}(x_k, x_{n+1})\right) \boxtimes e(k)\right) \times \notag\\  
&\quad \tau_{w[n,m]}\tau_{m+n+1}\tau_{m+n} \cdots \tau_{n+1} (b_+(i^n) \boxtimes e(k) \boxtimes b_+(i^m) \boxtimes e(j)) \notag \\
&= t_{i,j}^n b_+(i^{m+n})\tau_{m+n+1}\tau_{m+n} \cdots \tau_{n+1} (b_+(i^n) \boxtimes e(k) \boxtimes b_+(i^m) \boxtimes e(j)) \notag \\
&\quad \text{by (\ref{eq:computation}) below} \\
&= t_{i,j}^n \tau_{m+n+1} b_+(i^{m+n}) \tau_{m+n} \cdots \tau_{n+1} (b_+(i^n) \boxtimes e(k) \boxtimes b_+(i^m) \boxtimes e(j)) \\
&\mapsto t_{i,j}^n \tau_{m+n+1} ((b_+(i^{m+n})) \boxtimes e(k) \boxtimes e(j)) \in F_i^{(n+m)}(R(\alpha_k)\circ R(\alpha_j)) \notag \\
&\quad \text{by Lemma \ref{lem:phi-}} \\
&= t_{i,j}^n (b_+(i^{m+n}) \boxtimes \tau_1 e(k,j)). 
\end{align*} 
We used the following formula in $R(s_i(\alpha_j+\alpha_k))$: 
\begin{align} \label{eq:computation} 
&b_+(i^{m+n}) \left(b_+(i^m) \boxtimes \left((b_+(i^n) \boxtimes e(j)) \prod_{1 \leq k \leq n} Q_{i,j}(x_k, x_{n+1})\right) \boxtimes e(k)\right) \tau_{w[n,m]} \\
&= t_{i,j}^n b_+(i^{m+n}). \notag
\end{align}
It is computed as follows: the left hand side is 
\begin{align*}
&b_+(i^{m+n}) \left(\mathbf{x}_m \boxtimes \left((\mathbf{x}_n \boxtimes e(j)) \prod_{1 \leq k \leq n} Q_{i,j}(x_k, x_{n+1})\right) \boxtimes e(k)\right) \times \\ 
&(\tau_{w_m}e(i^m) \boxtimes \tau_{w_n}e(i^n) \boxtimes e(j,k)) \tau_{w[n,m]} \\
&\quad \text{since $\prod_{1 \leq k \leq n} Q_{i,j}(x_k,x_{n+1})$ is symmetric in $x_1, \ldots, x_n$} \\
&= b_+(i^{m+n}) \left(\mathbf{x}_m \boxtimes \left(\mathbf{x}_n \prod_{1 \leq k \leq n} Q_{i,j}(x_k, x_{n+1})\right) \boxtimes e(k)\right)(\tau_{w_{m+n}}e(i^{m+n}) \boxtimes e(j,k)) \\
&= b_+(i^{m+n}) \partial_{w_{m+n}}\left(\mathbf{x}_m \boxtimes \left(\mathbf{x}_n \prod_{1 \leq k \leq n} Q_{i,j}(x_k, x_{n+1})\right) \right) e(i^{m+n},j,k) \quad \text{by (\ref{eq:demazure}).}
\end{align*} 
Note that the leading term of $\mathbf{x}_m \boxtimes \left(\mathbf{x}_n \prod_{1 \leq k \leq n} Q_{i,j}(x_k, x_{n+1}) \right)$ for variables in $x_1, \ldots, x_{m+n}$ is
\[
(\mathbf{x}_m \boxtimes \mathbf{x_n}) t_{i,j}^n (x_1 \cdots x_n)^m = t_{i,j}^n \mathbf{x}_{m+n}. 
\]
We have $\partial_{w_{m+n}}(\mathbf{x}_{m+n}) = 1$ and the monomials in $x_1, \ldots, x_{m+n}$ of lower degree are killed by $\partial_{w_{m+n}}$.
Hence, the left hand side of (\ref{eq:computation}) is $t_{i,j}^n b_+(i^{m+n})$.
(1) is proved. 

(2) follows from (1) by Remark \ref{rem:LRchange}. 
The sign $(-1)^{mn}$ is deduced from 
\begin{align*}
&\sigma_* (\sigma_{k,j}) = (-1)^{\delta_{j,k}} \sigma'_{j,k}, \\
&\sigma_* \left({\xy 0;/r.12pc/: 
(-4,-8)*{};(12,8)*{} **\crv{(-4,-2) & (12,2)}?(0)*\dir{<};
(-12,-8)*{};(4,8)*{} **\crv{(-12,-2) & (4,2)}?(0)*\dir{<};
(4,-8)*{};(-12,8)*{} **\crv{(4,-2) & (-12,2)}?(0)*\dir{<};
(12,-8)*{};(-4,8)*{} **\crv{(12,-2) & (-4,2)}?(0)*\dir{<};
(-12,-11)*{\scriptstyle i^n};
(-4,-11)*{\scriptstyle k};
(4,-11)*{\scriptstyle i^m};
(12,-11)*{\scriptstyle j};
\endxy} \mathbf{1} \right) = (-1)^{mn + \delta_{j,k}} \mathbf{1} {\xy 0;/r.12pc/: 
(-4,-8)*{};(12,8)*{} **\crv{(-4,-2) & (12,2)}?(0)*\dir{<};
(-12,-8)*{};(4,8)*{} **\crv{(-12,-2) & (4,2)}?(0)*\dir{<};
(4,-8)*{};(-12,8)*{} **\crv{(4,-2) & (-12,2)}?(0)*\dir{<};
(12,-8)*{};(-4,8)*{} **\crv{(12,-2) & (-4,2)}?(0)*\dir{<};
(-12,-11)*{\scriptstyle j};
(-4,-11)*{\scriptstyle i^m};
(4,-11)*{\scriptstyle k};
(12,-11)*{\scriptstyle i^n};
\endxy}. 
\end{align*}
\end{proof}
 
\begin{corollary} \label{cor:sigma}
We have 
\begin{align*}
&\sigma_{j,k}((b_+(i^m)\boxtimes e(j)) \boxtimes (b_+(i^n) \boxtimes e(k))) \\  
&= t_{i,j}^{-n} (b_+(i^m) \boxtimes e(j) \boxtimes b_+(i^n) \boxtimes e(k))\tau_{w[n+1,m+1]}(b_+(i^n) \boxtimes e(k) \boxtimes b_+(i^m) \boxtimes e(j)). 
\end{align*}
\end{corollary}

We return to the proof that the outer diagram in (\ref{pic:15}) commutes.
Consider the diagram in Figure \ref{fig:8}, 
where $\mathsf{S}' \colon \mathbf{1}F_jF_i^{(m)'} F_kF_i^{(n)'} \to \mathbf{1}F_kF_i^{(n)'}F_jF_i^{(m)'}$ is the composition
\begin{align*}
&\mathbf{1}F_j F_i^{(m)'}F_kF_i^{(n)'} \xrightarrow{\text{can}} \mathbf{1}F_jF_i^mF_kF_i^n \xrightarrow{\mathsf{R}'} \mathbf{1} F_kF_i^n F_jF_i^m \xrightarrow{\text{can}} \mathbf{1}F_kF_i^{(n)'}F_jF_i^{(m)'}. 
\end{align*} 

\begin{figure} %\label{pic:16}
\centering
\scalebox{0.9}{
\rotatebox{90}{
\begin{tikzpicture}[auto]
\node (a) at (0,0) {$S_i(M'_j \circ M'_k$)}; 
\node (b) at (3.6,0) {$S_i(M'_j \circ R(\alpha_k)F_i^{(n)'})$}; 
\node (c) at (8,0) {$S_i((M'_j \circ R(\alpha_k))F_i^{(n)'})$}; 
\node (d) at (12,0) {$S_i(M'_j F_kF_i^{(n)'})$};
\node (e) at (16,0) {$S_i(\mathbf{1}F_jF_i^{(m)'}F_kF_i^{(n)'})$};
\node (a1) at (0,1.5) {$S_i(M'_j) \circ S_i(M'_k)$};
\node (b1) at (3.6,1.5) {$S_i(M'_j)\circ S_i(R(\alpha_k)F_i^{(n)'})$};
\node (c1) at (8,1.5) {$E_i^{(n)}S_i(M'_j \circ R(\alpha_k))$};
\node (d1) at (12,1.5) {$E_i^{(n)}S_i(M'_j F_k)$};
\node (e1) at (16,1.5) {$E_i^{(n)}S_i(\mathbf{1}F_jF_i^{(m)'}F_k)$};
\node (a2) at (0,3) {$R(\alpha_j) \circ S_i(M'_k)$};
\node (b2) at (3.6,3) {$S_i(M'_j) \circ E_i^{(n)}S_i(R(\alpha_k))$};
\node (c2) at (8,3) {$E_i^{(n)}(S_i(M'_j) \circ S_i(R(\alpha_k)))$};
\node (d2) at (12,3) {$E_i^{(n)} (S_i(M'_j)\circ M_k)$};
\node (e2) at (16,3) {$E_i^{(n)}(S_i(\mathbf{1}F_jF_i^{(m)'}) \circ M_k)$};
\node (a3) at (0,4.5) {$R(\alpha_j)\circ R(\alpha_k)$};
\node (b3) at (3.6,4.5) {$S_i(R(\alpha_j)F_i^{(m)'}) \circ E_i^{(n)}M_k$};
\node (c3) at (8,4.5) {$E_i^{(m)}S_i(R(\alpha_j)) \circ E_i^{(n)}M_k$};
\node (d3) at (12,4.5) {$E_i^{(m)}M_j \circ E_i^{(n)}M_k$};
\node (e3) at (16,4.5) {$E_i^{(n)} (E_i^{(m)}M_j\circ M_k)$};
\draw (a) -- node {$\scriptstyle \sim$} (b);
\draw[->] (c) -- (b);
\draw (c) -- node {$\scriptstyle \sim$} (d);
\draw (d) -- node {$\scriptstyle \sim$} (e);
\draw[->,line width=1.3pt] (a1) -- node{$\scriptstyle \theta$} (a);
\draw[->] (b1) -- node {$\scriptstyle \theta$} (b);
\draw (c1) -- node[sloped] {$\scriptstyle \sim$} (c);
\draw (d1) -- node[sloped] {$\scriptstyle \sim$} (d);
\draw (e1) -- node[sloped] {$\scriptstyle \sim$} (e);
\draw (a1) -- node {$\scriptstyle \sim$} (b1);
\draw (c1) -- node {$\scriptstyle \sim$} (d1);
\draw (d1) -- node {$\scriptstyle \sim$} (e1);
\draw[->,line width=1.3pt] (a1) -- node {$\scriptstyle \zeta_j$} (a2);
\draw[-] (b1) -- node[sloped] {$\scriptstyle \sim$} (b2);
\draw[->] (c2) -- node {$\scriptstyle \theta$} (c1);
\draw (d2) -- node[sloped] {$\scriptstyle \sim$} (d1);
\draw (e2) -- node[sloped] {$\scriptstyle \sim$} (e1);
\draw[<-] (b2) -- (c2);
\draw (c2) -- node {$\scriptstyle \sim$} (d2);
\draw (d2) -- node {$\scriptstyle \sim$} (e2);
\draw[->,line width=1.3pt] (a2) -- node {$\scriptstyle \zeta_k$} (a3);
\draw (b2) -- node[sloped] {$\scriptstyle \sim$} (b3);
\draw (e2) -- node[sloped] {$\scriptstyle \sim$} (e3);
\draw (b3) -- node {$\scriptstyle \sim$} (c3);
\draw (c3) -- node {$\scriptstyle \sim$} (d3);
\draw[<-] (d3) -- node {$\scriptstyle \sim$} (e3); 
\draw (a3) to[bend left=40] node {$\scriptstyle \sim$} (d3);
\node (f) at (0,-1.5) {$S_i(M'_k \circ M'_j$)}; 
\node (g) at (3.6,-1.5) {$S_i(M'_k \circ R(\alpha_j)F_i^{(m)'})$}; 
\node (h) at (8,-1.5) {$S_i((M'_k \circ R(\alpha_j))F_i^{(m)'})$}; 
\node (i) at (12,-1.5) {$S_i(M'_k F_jF_i^{(m)'})$};
\node (j) at (16,-1.5) {$S_i(\mathbf{1}F_kF_i^{(n)'}F_jF_i^{(m)'})$};
\node (f1) at (0,-3) {$S_i(M'_k) \circ S_i(M'_j)$};
\node (g1) at (3.6,-3) {$S_i(M'_k)\circ S_i(R(\alpha_j)F_i^{(m)'})$};
\node (h1) at (8,-3) {$E_i^{(m)}S_i(M'_k \circ R(\alpha_j))$};
\node (i1) at (12,-3) {$E_i^{(m)}S_i(M'_k F_j)$};
\node (j1) at (16,-3) {$E_i^{(m)}S_i(\mathbf{1}F_kF_i^{(n)'}F_j)$};
\node (f2) at (0,-4.5) {$R(\alpha_k) \circ S_i(M'_j)$};
\node (g2) at (3.6,-4.5) {$S_i(M'_k) \circ E_i^{(m)}S_i(R(\alpha_j))$};
\node (h2) at (8,-4.5) {$E_i^{(m)}(S_i(M'_k) \circ S_i(R(\alpha_j)))$};
\node (i2) at (12,-4.5) {$E_i^{(m)} (S_i(M'_k)\circ M_j)$};
\node (j2) at (16,-4.5) {$E_i^{(m)}(S_i(\mathbf{1}F_kF_i^{(n)'}) \circ M_j)$};
\node (f3) at (0,-6) {$R(\alpha_k)\circ R(\alpha_j)$};
\node (g3) at (3.6,-6) {$S_i(R(\alpha_k)F_i^{(n)'}) \circ E_i^{(m)}M_j$};
\node (h3) at (8,-6) {$E_i^{(n)}S_i(R(\alpha_k)) \circ E_i^{(m)}M_j$};
\node (i3) at (12,-6) {$E_i^{(n)}M_k \circ E_i^{(m)}M_j$};
\node (j3) at (16,-6) {$E_i^{(m)} (E_i^{(n)}M_k\circ M_j)$};
\draw (f) -- node {$\scriptstyle \sim$} (g);
\draw[->] (h) -- (g);
\draw (h) -- node {$\scriptstyle \sim$} (i);
\draw (i) -- node {$\scriptstyle \sim$} (j);
\draw[->,line width=1.3pt] (f1) -- node{$\scriptstyle \theta$} (f);
\draw[->] (g1) -- node {$\scriptstyle \theta$} (g);
\draw (h1) -- node[sloped] {$\scriptstyle \sim$} (h);
\draw (i1) -- node[sloped] {$\scriptstyle \sim$} (i);
\draw (j1) -- node[sloped] {$\scriptstyle \sim$} (j);
\draw (f1) -- node {$\scriptstyle \sim$} (g1);
\draw (h1) -- node {$\scriptstyle \sim$} (i1);
\draw (i1) -- node {$\scriptstyle \sim$} (j1);
\draw[->,line width=1.3pt] (f1) -- node {$\scriptstyle \zeta_k$} (f2);
\draw[-] (g1) -- node[sloped] {$\scriptstyle \sim$} (g2);
\draw[->] (h2) -- node {$\scriptstyle \theta$} (h1);
\draw (i2) -- node[sloped] {$\scriptstyle \sim$} (i1);
\draw (j2) -- node[sloped] {$\scriptstyle \sim$} (j1);
\draw[<-] (g2) -- (h2);
\draw (h2) -- node {$\scriptstyle \sim$} (i2);
\draw (i2) -- node {$\scriptstyle \sim$} (j2);
\draw[->,line width=1.3pt] (f2) -- node {$\scriptstyle \zeta_j$} (f3);
\draw (g2) -- node[sloped] {$\scriptstyle \sim$} (g3);
\draw (j2) -- node[sloped] {$\scriptstyle \sim$} (j3);
\draw (g3) -- node {$\scriptstyle \sim$} (h3);
\draw (h3) -- node {$\scriptstyle \sim$} (i3);
\draw[<-] (i3) -- node {$\scriptstyle \sim$} (j3); 
\draw (f3) to[bend right=40] node {$\scriptstyle \sim$} (i3);
\draw[<-,line width=1.3pt] (f) -- node {$\scriptstyle \sigma'_{j,k}$} (a);
\draw[<-] (j) -- node {$\scriptstyle (-1)^{mn}t_{i,k}^{-m}\mathsf{S}'$} (e);
\draw[->, line width=1.3pt] (a3) to[bend right=50] node {$\scriptstyle \times t_{i,j}^nt_{i,k}^{-m}\tau_1$} (f3); 
\node (U) at (6.1,1.5) {(U)};
\node (U') at (6.1,-3) {(U')};
\node (sigmaR) at (8,-0.7) {Lemma \ref{lem:sigmaR}(2)};
\node (zeta) at (6,5.5) {Definition of $\zeta_j$ and $\zeta_k$};
\node (zeta') at (6,-7) {Definition of $\zeta'_j$ and $\zeta'_k$}; 
\node (V) at (10,3.7) {(V)};
\node (V') at (10,-5.2) {(V')};
\node (W) at (10,2.2) {(W)};
\node (W') at (10,-3.7) {(W')};
\end{tikzpicture}
}
}
\caption{} \label{fig:8}
\end{figure}

The commutativity of the outer diagram of (\ref{pic:15}) is reduced to the commutativity of the thick diagram in Figure \ref{fig:8} since 
\[
t_{i,k}^{\varrho(\beta)-\varrho(\beta+s_i\alpha_j)}t_{i,j}^{\varrho(\beta+s_i\alpha_k) - \varrho(\beta)}= t_{i,j}^n t_{i,k}^{-m}. 
\]
It is easy to see that all the inner diagrams of Figure \ref{fig:8} except (U), (U'), (V), (V'), (W) and (W') are commutative. 
Hence, it suffices to prove that these six diagrams and the outer diagram commute. 

Commutativity of (U) and (U') follows from Proposition \ref{prop:monoidality} (4).

Commutativity of (V). 
Identifying 
\[
S_i(M'_j) \simeq S_i(R(\alpha_j)F_i^{(m)'}) \simeq E_i^{(m)}S_i(R(\alpha_j)) \simeq E_i^{(m)}M_j \simeq S_i(\mathbf{1}F_jF_i^{(m)'}),
\] 
it is the same as the following commutative diagram (Lemma \ref{lem:adjointSES}): 
\begin{equation}
\begin{tikzcd}
S_i(M'_j) \circ E_i^{(n)}M_k  \arrow[d,-,"\sim"sloped] & E_i^{(n)}(S_i(M'_j) \circ M_k) \arrow[l]\arrow[d,-,"\sim"sloped] \\
S_i(M'_j) \circ E_i^{(n)}S_i(R(\alpha_k)) & E_i^{(n)}(S_i(M'_j) \circ S_i(R(\alpha_k))) \arrow[l].  
\end{tikzcd}
\end{equation}
Commutativity of (V') is analogous. 

Commutativity of (W) and (W') follows from Proposition \ref{prop:additionalcd}. 

Commutativity of the outer diagram of Figure \ref{fig:8}. 
Consider the following diagram: 
\begin{equation} \label{pic:17}
\adjustbox{scale=0.9,center}{
\begin{tikzcd}
S_i(\mathbf{1}F_jF_i^{(m)'}F_kF_i^{(n)'}) \arrow[r,-,"\sim"]\arrow[d,"\text{can}"]\arrow[rd,phantom,"(X)"]\arrow[ddd,bend right=70,"(-1)^{mn}t_{i,k}^{-m} \mathsf{S}'"'] & E_i^{(n)}(E_i^{(m)}M_j \circ M_k) \arrow[d,"\text{can}"]\arrow[r,"\sim"] & E_i^{(m)}M_j \circ E_i^{(n)}M_k\arrow[d,-,"\sim"sloped] \\
S_i(\mathbf{1}F_jF_i^mF_kF_i^n) \arrow[r,-,"\sim"]\arrow[d,"(-1)^{mn}t_{i,k}^{-m}\mathsf{R}'"] & E_i^n(E_i^mM_j \circ M_k) \arrow[d,"(-1)^{mn}t_{i,k}^{-m} \mathsf{T}"] & R(\alpha_j) \circ R(\alpha_k) \arrow[d,"\times t_{i,j}^nt_{i,k}^{-m}\tau_1"]\arrow[d,phantom,bend right=80, "(Z)"] \\
S_i(\mathbf{1}F_kF_i^nF_jF_i^m) \arrow[r,-,"\sim"]\arrow[d,"\text{can}"] & E_i^m(E_i^nM_k \circ M_j) \arrow[d,"\text{can}"] & R(\alpha_k) \circ R(\alpha_j) \arrow[d,-,"\sim"sloped] \\
S_i(\mathbf{1}F_kF_i^{(n)'}F_jF_i^{(m)'}) \arrow[r,-,"\sim"]\arrow[ru,phantom,"(Y)"] & E_i^{(m)}(E_i^{(n)}M_k \circ M_j) \arrow[r,"\sim"] & E_i^{(n)}M_k \circ E_i^{(m)}M_j  \\
\end{tikzcd}
}
\end{equation}
where $\mathsf{T}$ is the composition
\begin{align*}
&E_i^n(E_i^mM_j\circ M_k) \xrightarrow{\sigma_{i,k}} E_i^nE_i(E_i^{m-1}M_j \circ M_k) \xrightarrow{\sigma_{i,k}} \cdots \\
&\xrightarrow{\sigma_{i,k}} E_i^nE_i^m(M_j \circ M_k) \xrightarrow{\xy 0;/r.12pc/: (0,0)*{\ucross{i^n}{i^m}}; \endxy \boxtimes \sigma_{j,k}} E_i^mE_i^n(M_k \circ M_j) \\
&\xrightarrow{\sigma_{j,i}} E_i^mE_i^{n-1} (E_iM_k \circ M_j) \xrightarrow{\sigma_{j,i}} \cdots \xrightarrow{\sigma_{j,i}} E_i^m(E_i^nM_k \circ M_j). 
\end{align*}
The outer diagram of (\ref{pic:17}) is the outer diagram of Figure \ref{fig:8}. 
It is easy to see the commutativity of the inner diagrams of (\ref{pic:17}) except (X),(Y) and (Z). 
Hence, it remains to verify these three diagrams commute.

Diagrams (X) and (Y) commute by Lemma \ref{lem:directsummand}. 
We prove the commutativity of (Z). 
Recall the isomorphism $\phi_+ \colon E_i^{(n)}M_k \circ E_i^{(m)}M_j \to E_i^{(n+m)}(M_k \circ M_j)$ of Proposition \ref{prop:extremalequiv}. 
Postcompositing $\phi_+$, we compute the two images of $e(j) \boxtimes e(k) \in R(\alpha_j) \circ R(\alpha_k)$ in $E_i^{(n+m)}(M_k \circ M_j)$.

Under the homomorphism through $R(\alpha_k) \circ R(\alpha_j)$, it is sent to $t_{i,j}^n t_{i,k}^{-m}$-multiple of 
\begin{align*}
&e(j) \boxtimes e(k) \\ 
&\mapsto (e(j) \boxtimes e(k)) \tau_1 = \tau_1 (e(k) \boxtimes e(j)) \in R(\alpha_k) \circ R(\alpha_j) \\
&\mapsto \tau_1 (E_i^{(n)}(b_+(i^n)\boxtimes e(k)) \boxtimes E_i^{(m)}(b_+(i^m)\boxtimes e(j))) \in E_i^{(n)}M_k \circ E_i^{(m)}M_j \\
&\mapsto \tau_1 E_i^{(m+n)} \tau_{n+m} \cdots \tau_{n+1}((b_+(i^n) \boxtimes e(k)) \boxtimes (b_+(i^m)\boxtimes e(j))) \in E_i^{(n+m)}(M_k \circ M_j). 
\end{align*}
To compute the image of $e(j) \boxtimes e(k)$ under the other homomorphism, we use the following lemma:

\begin{lemma}
(1) The following diagram commutes: 
\begin{equation*}
\begin{tikzcd}
E_i^{(m)}(E_i^{(n)}M_k \circ M_j) \arrow[r,"\sim"]\arrow[d,"\mathrm{can}"] & E_i^{(n)}M_k \circ E_i^{(m)}M_j \arrow[r,"\phi_+"] & E_i^{(n+m)}(M_k \circ M_j) \\
E_i^m(E_i^nM_k \circ M_j) \arrow[r,"{\text{$n$-times $\sigma_{i,j}$}}"] & E_i^{m+n} (M_k \circ M_j) \arrow[ru,"\mathrm{can}"] & 
\end{tikzcd}
\end{equation*}

(2) The composition
\begin{align*}
&E_i^mE_i^n(M_k \circ M_j) \xrightarrow{\text{$n$-times $\sigma_{j,i}$}} E_i^m (E_i^nM_k \circ M_j) \xrightarrow{\text{can}} E_i^{(m)}(E_i^{(n)}M_k \circ M_j) \\
&\xrightarrow{\text{can}} E_i^m(E_i^nM_k \circ M_j) \xrightarrow{\text{$n$-times $\sigma_{i,j}$}} E_i^mE_i^n(M_k \circ M_j)
\end{align*}
coincides with the endomorphism 
\[ 
\left(\xy 0;/r.12pc/: 
(0,-10); (0,10) **\dir{-} ?(1)*\dir{>};
(30,-10); (30,10) **\dir{-} ?(1)*\dir{>};
(0,0)*{\fcolorbox{black}{white}{$b_+(i^m)$}};
(30,0)*{\fcolorbox{black}{white}{$b_+(i^n)$}};
(0,-13)*{\scriptstyle i^m};
(30,-13)*{\scriptstyle i^n};
\endxy \right) \prod_{1 \leq k \leq n} Q_{i,j} \left( \xy 0;/r.12pc/: 
(-1,-0.5)*{\slineu{i^{m+k-1}}}; 
(8,0)*{\sdotu{i}}; 
(16,0)*{\slineu{i^{n-k}}}; 
\endxy, y_j \right).
\]
\end{lemma}

\begin{proof}
(1) follows from the explicit formula of $\phi_+$ in Proposition \ref{prop:extremalequiv}. 

(2) follows from Definition \ref{def:anothertauij} (1). 
\end{proof}

By this lemma, we can compute the another image of $e(j) \boxtimes e(k)$ using $(-1)^{mn}t_{i,k}^{-m}$-multiple of the homomorphism 
\begin{align*}
&R(\alpha_j) \circ R(\alpha_k) \simeq E_i^{(m)} M_j \circ E_i^{(n)}M_k \xleftarrow[]{\sim} E_i^{(n)}(E_i^{(m)}M_j \circ M_k) \xrightarrow{\text{can}} E_i^n(E_i^mM_j \circ M_k) \\
&\xrightarrow{\text{$n$-times $\sigma_{i,k}$}} E_i^nE_i^m(M_j \circ M_k) \xrightarrow{\xy 0;/r.12pc/: (0,0)*{\ucross{i^n}{i^m}}; \endxy \boxtimes \sigma_{j,k}} E_i^mE_i^n(M_k \circ M_j) \\
&\xrightarrow{\left(\xy 0;/r.12pc/: 
(0,-10); (0,10) **\dir{-} ?(1)*\dir{>};
(30,-10); (30,10) **\dir{-} ?(1)*\dir{>};
(0,0)*{\fcolorbox{black}{white}{$b_+(i^m)$}};
(30,0)*{\fcolorbox{black}{white}{$b_+(i^n)$}};
(0,-13)*{\scriptstyle i^m};
(30,-13)*{\scriptstyle i^n};
\endxy \right) \prod_{1 \leq k \leq n} Q_{i,j} \left( \xy 0;/r.12pc/: 
(-1,-0.5)*{\slineu{i^{m+k-1}}}; 
(8,0)*{\sdotu{i}}; 
(16,0)*{\slineu{i^{n-k}}}; 
\endxy, y_j \right)} E_i^mE_i^n(M_k \circ M_j) \\
&\xrightarrow{\text{can}} E_i^{(n+m)}(M_k\circ M_j). 
\end{align*}
Namely, $e(j) \boxtimes e(k)$ is sent to $(-1)^{mn}t_{i,k}^{-m}$-multiple of 
\allowdisplaybreaks
\begin{align} \notag
&e(j) \boxtimes e(k) \notag \\
&\mapsto E_i^{(m)}(b_+(i^m) \boxtimes e(j)) \boxtimes E_i^{(n)}(b_+(i^n) \boxtimes e(k)) \in E_i^{(m)}M_j \circ E_i^{(n)}M_k \notag \\
&\mapsto E_i^{(n)} \tau_n \cdots \tau_2 \tau_1 (E_i^{(m)}(b_+(i^m) \boxtimes e(j)) \boxtimes (b_+(i^n)\boxtimes e(k))) \in E_i^{(n)}(E_i^{(m)}M_j \circ M_k) \notag \\ 
&\mapsto E_i^n b_+(i^n) \tau_n \cdots \tau_1 (E_i^m(b_+(i^m) \boxtimes e(j)) \boxtimes (b_+(i^n) \boxtimes e(k))) \in E_i^n(E_i^mM_j\circ M_k) \notag \\
&\mapsto E_i^n b_+(i^n) \tau_n \cdots \tau_1 E_i^m ((b_+(i^m) \boxtimes e(j)) \boxtimes (b_+(i^n) \boxtimes e(k))) \in E_i^nE_i^m(M_j \circ M_k) \notag \\
&= E_i^nE_i^m \tau_{m+n} \cdots \tau_{m+1} ((b_+(i^m) \boxtimes e(j)) \boxtimes (b_+(i^n)\boxtimes e(k))) \notag \\
&\quad \text{since $b_+(i^n)\tau_n \cdots \tau_1 (e(j) \boxtimes b_+(i^n)) = \tau_n \cdots \tau_1 (e(j) \boxtimes b_+(i^n))$} \notag \\
&\mapsto t_{i,j}^{-n} E_i^mE_i^n \tau_{w[m,n]} \tau_{m+n} \cdots \tau_{m+1} ((b_+(i^m) \boxtimes e(j)) \boxtimes (b_+(i^n) \boxtimes e(k))) \tau_{w[n+1,m+1]} \times \notag \\
&\quad ((b_+(i^n) \boxtimes e(k)) \boxtimes (b_+(i^m)\boxtimes e(j))) \in E_i^mE_i^n(M_k\circ M_j) \quad \text{by Corollary \ref{cor:sigma}} \notag \\
&\mapsto t_{i,j}^{-n} E_i^mE_i^n (b_+(i^n) \boxtimes b_+(i^m)) \left( \prod_{1 \leq k \leq n} Q_{i,j}(x_k, y_j)\right)  \tau_{w[m,n]} \tau_{m+n} \cdots \tau_{m+1} \times \notag  \\
&\quad ((b_+(i^m) \boxtimes e(j)) \boxtimes (b_+(i^n) \boxtimes e(k)))\tau_{w[n+1,m+1]}((b_+(i^n) \boxtimes e(k)) \boxtimes (b_+(i^m)\boxtimes e(j))) \notag \\ 
&\in E_i^mE_i^n(M_k\circ M_j) \notag \\ 
&\mapsto t_{i,j}^{-n} E_i^{(m+n)}b_+(i^{m+n}) (b_+(i^n) \boxtimes b_+(i^m)) \left( \prod_{1 \leq k \leq n} Q_{i,j}(x_k, y_j)\right)  \tau_{w[m,n]} \times \notag\\ 
&\quad \tau_{m+n} \cdots \tau_{m+1} ((b_+(i^m) \boxtimes e(j)) \boxtimes (b_+(i^n) \boxtimes e(k))) \tau_{w[n+1,m+1]} \times \notag \\ 
&((b_+(i^n) \boxtimes e(k)) \boxtimes (b_+(i^m)\boxtimes e(j))) \in E_i^{(m+n)}(M_k\circ M_j) \notag \\
&= (-1)^{mn} E_i^{(m+n)}b_+(i^{m+n}) \tau_{m+n} \cdots \tau_{m+1}((b_+(i^m) \boxtimes e(j)) \boxtimes (b_+(i^n) \boxtimes e(k))) \times \notag \\
&\quad \tau_{w[n+1,m+1]} ((b_+(i^n) \boxtimes e(k)) \boxtimes (b_+(i^m)\boxtimes e(j))) \notag \\
&\quad \text{by the computation (\ref{eq:computation1}) below} \notag \\
&= (-1)^{mn}  E_i^{(m+n)}b_+(i^{m+n}) \tau_{m+n} \cdots \tau_{m+1}((b_+(i^m) \boxtimes e(j)) \boxtimes (b_+(i^n) \boxtimes e(k)))\times  \notag \\
&\quad  \tau_{m+1}\cdots \tau_{m+n} \tau_{w[n,m]} \tau_{m+n+1} \tau_{n+m} \cdots \tau_{n+1} ((b_+(i^n) \boxtimes e(k)) \boxtimes (b_+(i^m)\boxtimes e(j)))  \notag \\
&= (-1)^{mn} E_i^{(m+n)}b_+(i^{m+n}) \tau_{m+n} \cdots \tau_{m+1}\tau_{m+1}\cdots \tau_{m+n} \times \notag \\ 
&\quad (b_+(i^m) \boxtimes b_+(i^n) \boxtimes e(j) \boxtimes e(k))\tau_{w[n,m]} \tau_{m+n+1} \tau_{n+m} \cdots \tau_{n+1} \times  \notag \\
&\quad ((b_+(i^n) \boxtimes e(k)) \boxtimes (b_+(i^m)\boxtimes e(j))) \notag \\
&= (-1)^{mn}  E_i^{(m+n)}b_+(i^{m+n}) \left( \prod_{1 \leq k \leq n} Q_{i,j}(x_{m+k}, x_{m+n+1})\right) (\mathbf{x}_m \boxtimes \mathbf{x}_n \boxtimes e(j) \boxtimes e(k)) \times  \notag \\
&\quad \tau_{w_{m+n}} \tau_{m+n+1} \tau_{n+m} \cdots \tau_{n+1} ((b_+(i^n) \boxtimes e(k)) \boxtimes (b_+(i^m)\boxtimes e(j))) \notag \\
&= (-1)^{mn} E_i^{(m+n)}b_+(i^{m+n}) \partial_{w_{m+n}} \left( (\mathbf{x}_m \boxtimes \mathbf{x}_n) \prod_{1 \leq k \leq n} Q_{i,j}(x_{m+k}, x_{m+n+1})\right) \times \notag \\
&\quad \tau_{m+n+1} \tau_{n+m} \cdots \tau_{n+1} ((b_+(i^n) \boxtimes e(k)) \boxtimes (b_+(i^m)\boxtimes e(j))) \notag \\
&= (-1)^{mn} t_{i,j}^n E_i^{(m+n)}b_+(i^{m+n}) \tau_{m+n+1} \tau_{n+m} \cdots \tau_{n+1} \times \notag \\ 
&\quad ((b_+(i^n) \boxtimes e(k)) \boxtimes (b_+(i^m)\boxtimes e(j)))\text{by the computation (\ref{eq:computation2}) below}.  \notag 
\end{align}
We used the following formulas: 
\begin{align}
&b_+(i^{m+n})(b_+(i^n) \boxtimes b_+(i^m)) \left( \prod_{1 \leq k \leq n} Q_{i,j}(x_k, y_j)\right) \tau_{w[m,n]} = (-1)^{mn} t_{i,j}^n, \label{eq:computation1} \\  
&\partial_{w_{m+n}} \left( (\mathbf{x}_m \boxtimes \mathbf{x_n}) \prod_{1 \leq k \leq n} Q_{i,j}(x_{m+k}, x_{m+n+1})\right) = t_{i,j}^n. \label{eq:computation2} 
\end{align}
They are proved as follows. 
The left hand side of (\ref{eq:computation1}) is 
\begin{align*}
&b_+(i^{m+n}) \left((\mathbf{x}_n \boxtimes \mathbf{x}_m) \prod_{1 \leq k \leq n} Q_{i,j}(x_k, y_j)\right) \tau_{w_{m+n}} \\
&\quad \text{since $\left( \prod_{1 \leq k \leq n} Q_{i,j}(x_k, y_j)\right)$ is symmetric in $x_1, \ldots, x_n$} \\ 
&= b_+(i^{m+n}) \partial_{w_{m+n}} \left( (\mathbf{x}_n \boxtimes \mathbf{x}_m) \prod_{1 \leq k \leq n} Q_{i,j}(x_k, y_j)\right). 
\end{align*}
The leading term of $\left( (\mathbf{x}_n \boxtimes \mathbf{x}_m) \prod_{1 \leq k \leq n} Q_{i,j}(x_k, y_j)\right)$ in variables $x_1, \ldots, x_{m+n}$ is
\[
t_{i,j}^n (\mathbf{x}_n \boxtimes \mathbf{x}_m)(x_1 \cdots x_n)^m = t_{i,j}^n w_{[m,n]}(\mathbf{x}_{m+n}). 
\]
Hence, (\ref{eq:computation1}) follows. 
Similarly, since the leading term of 
\[
(\mathbf{x}_m \boxtimes \mathbf{x_n}) \prod_{1 \leq k \leq n} Q_{i,j}(x_{m+k}, x_{m+n+1})
\]
in variables $x_1, \ldots,x_{m+n}$ is $t_{i,j}^n \mathbf{x}_{m+n}$, 
(\ref{eq:computation2}) follows. 

Since $t_{i,j}^nt_{i,k}^{-m} = (-1)^{mn}t_{i,k}^{-m} (-1)^{mn}t_{i,j}^n$, 
the two images of $e(j) \boxtimes e(k)$ coincide. 
This completes the proof of Case 3, 
and of the assertion that the natural isomorphisms $\kappa_j^- \colon S_i(F_j X) \to F_jS_i(X)$ and $\kappa_i^+ \colon S_i(E_iX) \to E_iS_i(X)$ commute with the $\dotcatquantum{\mathfrak{p}_i}$-action. 
Therefore, Proposition \ref{prop:linear} is proved. 

\section{Properties of reflection functors}

\subsection{Categorification of $T_i$}

Let $i \in I$. 
Recall the isomorphism $\chi \colon K(\gMod{R})_{\mathbb{Q}(q)} \to U_q^-(\mathfrak{g})$ from Theorem \ref{thm:categorification2}. 

\begin{theorem} \label{thm:braidcategorification}
(1) The homomorphisms 
\[
K(\gmod{R_i})_{\mathbb{Q}(q)} \to K(\gMod{R_i})_{\mathbb{Q}(q)}, \ K(\gmod{{}_iR})_{\mathbb{Q}(q)} \to K(\gMod{{}_iR})_{\mathbb{Q}(q)}
\]
induced by the inclusions are isomorphisms. 

(2) The homomorphisms $K(\gMod{R_i})_{\mathbb{Q}(q)}, K(\gMod{{}_iR})_{\mathbb{Q}(q)} \to K(\gMod{R})_{\mathbb{Q}(q)}$ induced by the inclusions are injective, 
and $\chi \colon K(\gMod{R})_{\mathbb{Q}(q)} \xrightarrow{\sim} U_q^-(\mathfrak{g})$ restricts to isomorphisms
\[
K(\gMod{R_i})_{\mathbb{Q}(q)} \xrightarrow{\sim} U_i,\ K(\gMod{{}_iR})_{\mathbb{Q}(q)} \xrightarrow{\sim} {}_iU. 
\]

(3) The isomorphisms of (2) are both homomorphisms of left $U_q(\mathfrak{p}_i)$-modules and of right $U_q(\mathfrak{p}_i)$-modules, 
for the module structures described in Proposition \ref{prop:actionlift}, \ref{prop:bimodule}, Theorem \ref{thm:cyclotomic2rep}, \ref{thm:cyclotomic2repleft} and \ref{thm:anotheraction}. 

(4) The following diagram commutes: 
\begin{equation*}
\begin{tikzcd}
K(\gMod{{}_iR})_{\mathbb{Q}(q)}\arrow[r,"S_i"]\arrow[d,"\chi"] & K(\gMod{R_i})_{\mathbb{Q}(q)} \arrow[d,"\chi"] \\ 
{}_iU \arrow[r,"T_i"] & {U_i}. 
\end{tikzcd}
\end{equation*}
\end{theorem}

\begin{proof}
(1) It is a special case of Theorem \ref{thm:cyclotomiccategorification} and \ref{thm:cyclotomiccategorificationleft}. 

(2) Using (1) and Theorem \ref{thm:categorification2}, the injectivity follows from the injectivity of $K(\gmod{R_i}) \to K(\gmod{R})$ and $K(\gmod{{}_iR}) \to K(\gmod{R})$. 
By Lemma \ref{lem:res} and Lemma \ref{lem:boson}, we have $\chi(\gmod{R}) \subset \Ker r_i = U_i$.  
By Lemma \ref{lem:isom} and Theorem \ref{thm:cyclotomiccategorificationleft}, each weight spaces of $K(\gMod{R}_i)$ and of $U_i$ have the same dimension. 
Hence, the assertion follows. 

\begin{comment}
Now, recall that the isomorphism $K(\gproj{R_i})_{\mathbb{Q}(q)} \simeq V_i(0)$ of Theorem \ref{thm:cyclotomiccategorificationleft},
which is induced from $K(\gproj{R})_{\mathbb{Q}(q)} \simeq U_q^-(\mathfrak{g})$ 
through the canonical surjections $K(\gproj{R})_{\mathbb{Q}(q)} \twoheadrightarrow K(\gproj{R_i})_{\mathbb{Q}(q)}$ and $U_q^-(\mathfrak{g}) \twoheadrightarrow V_i(0)$. 
Note that the composition 
\[
K(\gproj{R_i})_{\mathbb{Q}(q)} \simeq K(\gMod{R_i})_{\mathbb{Q}(q)} \to K(\gMod{R})_{\mathbb{Q}(q)}\simeq K(\gproj{R})_{\mathbb{Q}(q)} \to K(\gproj{R_i})_{\mathbb{Q}(q)}
\]
is the identity, and $U_i \hookrightarrow U_q^-(\mathfrak{g}) \to V_i(0)$ is an isomorphism (Lemma \ref{lem:isom}). 
Combined with the discussion in the previous paragraph, we deduce that $\chi (K(\gproj{R_i})_{\mathbb{Q}(q)}) = U_i$. 
\end{comment}
The proof for ${}_iR$ is similar. 

(3) By construction, the isomorphism $K(\gproj{R_i})_{\mathbb{Q}(q)} \simeq V_i(0)$ is left $U_q(\mathfrak{p}_i)$-linear. 
Note that, for $j \in I \setminus \{i\}$, 
\[
\chi(M_j) = \chi (F_i^{(-a_{i,j})}R(\alpha_j)) = \ad_{f_i}^{(-a_{i,j})}\chi(R(\alpha_j)) = u_j, 
\]
where the last equality follows from Lemma \ref{lem:uj}. 
It implies that the isomorphism is also right $U_q(\mathfrak{p}_i)$-linear. 

The proof for ${}_iR$ is similar. 

(4) It follows from the facts below: 
\begin{itemize}
\item Both $S_i$ and $T_i$ are right $U_q(\mathfrak{p}_i)$-linear. 
\item $K(\gMod{{}_iR})_{\mathbb{Q}(q)} \simeq {}_iU \simeq {}_iV(0)$ is generated by $\mathbf{1}$ as a right $U_q(\mathfrak{p}_i)$-module. 
\item $S_i(\mathbf{1}) = \mathbf{1}, \chi(\mathbf{1}) = 1, T_i(1) = 1$. 
\end{itemize}
\end{proof}

Recall the bilinear form $(,)$ and the $\mathbb{Q}$-linear automorphism $c$ of $U_q^-(\mathfrak{g})$ from Section \ref{subsec:categorification}. 
Let $B(\infty)$ be the crystal basis of $U_q^-(\mathfrak{g})$, 
and let $\mathsf{B} = \{G(b) \mid b \in B(\infty)\}$ be the canonical basis. 
Let $\mathsf{B}^* = \{G^*(b) \mid b \in B(\infty) \}$ be the basis of $U_q^-(\mathfrak{g})$ adjoint to $\mathsf{B}$ with respect to the bilinear form $(,)$,
which is called the dual canonical basis. 
For $b \in B(\infty)$, it is known that $\overline{G(b)} = G(b)$. 
Hence, $c(G^*(b)) = G^*(b)$ by definition. 

\begin{lemma} \label{lem:candTi}
We have $c(U_i) = U_i, c({}_iU) = {}_iU$, and 
\[
c T_i (u) = T_i c(u) \ (u \in {}_iU). 
\]
\end{lemma}

\begin{proof}
It follows from the fact that $c(G^*(b)) = G^*(b) \ (b \in B(\infty))$, and 
\begin{itemize}
\item that ${}_iU$ (resp. $U_i$) is spanned by $\mathsf{B}^* \cap {}_iU$ (resp. $\mathsf{B}^* \cap U_i$) as a $\mathbb{Q}(q)$-vector space \cite[Proposition 4.14]{MR2914878},  
\item that $T_i (\mathsf{B}^* \cap {}_iU) = \mathsf{B}^* \cap U_i$ (\cite[Theorem 4.23]{MR2914878}; note that when the weight of $b$ is $-\sum_{j \in I} k_j \alpha_j$, their $G^{\text{up}}(b)$ is $\prod_{j\in I}(1-q_j^2)^{-k_j}$-multiple of our $G^*(b)$), and 
\item that $c(q) = q^{-1}$. 
\end{itemize}
\end{proof}

\begin{proposition} \label{prop:DandSi}
If $L \in \gmod{{}_iR}$ is a self-dual simple module, then $S_i(L) \in \gmod{R_i}$ is also a self-dual simple module. 
\end{proposition}

\begin{proof}
Since $S_i$ is an equivalence, $S_i(L)$ is simple. 
By Lemma \ref{lem:twist} and Lemma \ref{lem:candTi}, we have 
\[
\chi(DS_i(L)) = cT_i(\chi(L)) = T_i c(\chi(L)) = \chi(S_iDL) = \chi(S_iL). 
\]
Hence, $S_i(L)$ is self-dual. 
\end{proof}

\subsection{Reflection functors and standard modules} \label{sec:stratification}

\begin{lemma} \label{lem:std}
Let $i \in I, w\in W$ and assume that $ws_i > w$. 
Fix a reduced expression $\underline{w} = (i_1, \ldots, i_l)$ of $w$.  
Let $m \in \mathbb{Z}_{\geq 0}$.
\begin{enumerate}
\item[(1)] $S_{i_1} \cdots S_{i_l} L(i^m)$ and $S_{i_1} \cdots S_{i_l} P(i^m)$ are well-defined, that is, the compositions of the functors make sense. 
\end{enumerate}
We define 
\[
L(\underline{w},i^m) = S_{i_1} \cdots S_{i_l} L(i^m), \ \Delta(\underline{w},i^m) = S_{i_1} \cdots S_{i_l} P(i^m). 
\]
\begin{enumerate}
\item[(2)] $L(\underline{w},i^m)$ is self-dual simple, and we have 
\begin{align*}
\qdim \HOM_{R(mw\alpha_i)} (\Delta(\underline{w},i^m), L(\underline{w},i^m)) = 1, \\
\qdim \EXT_{R(mw\alpha_i)}^1 (\Delta(\underline{w},i^m), L(\underline{w},i^m)) = 0. 
\end{align*}
\end{enumerate}
\end{lemma}

\begin{proof}
(1) Using Theorem \ref{thm:braidcategorification}, it follows from \cite[Lemma 40.1.2]{MR2759715}. 

(2) By Proposition \ref{prop:DandSi}, $L(\underline{w},i^c)$ is self-dual simple. 
Note that the categories $\gMod{R_j}$ and $\gMod{{}_jR}$ are closed under subquotient and extension in $\gMod{R}$. 
Hence, the equalities are reduced to the case $w = e$, which are trivial.  
\end{proof}

Note that $L(\underline{w},i^m)$ is simple, hence it is determined by its character
\[
\chi (L(\underline{w},i^m)) = T_{i_1} \cdots T_{i_l} \chi(L(i^m)). 
\]
Since $T_j \ (j \in I)$ satisfy the braid relations, the character is independent of the choice of $\underline{w}$. 
Furthermore, $\Delta(\underline{w},i^m)$ is the unique module whose composition factors are grading shifts of $L(\underline{w},i^m)$ and that satisfy Lemma \ref{lem:std} (2). 
Therefore, both $L(\underline{w},i^m)$ and $\Delta(\underline{w},i^m)$ are independent of $\underline{w}$.
This leads to the following definition. 

\begin{definition} \label{def:rootmodule}
Let $w \in W, i \in I, m \in \mathbb{Z}_{\geq 0}$ and assume that $ws_i > w$. 
We define $\Delta(w,i^m) = \Delta(\underline{w},i^m)$ and $L(w, i^m) = L(\underline{w}, i^m)$, where $\underline{w}$ is a reduced word of $w$. \index{$\Delta(w,i^m), L(w,i^m)$}

Similarly, we define
\begin{align*}
\Delta'(w, i^m) = S'_{i_1}\cdots S'_{i_l} P(i^m),\ L'(w,i^m) = S'_{i_1}\cdots S'_{i_l} L(i^m).  
\end{align*} \index{$\Delta'(w,i^m), L'(w,i^m)$} 
\end{definition}

\begin{remark} \label{rem:determinantialmodule}
These modules are related to determinantial modules \cite{MR3771147} and their affinizations \cite{MR4359265} as follows. 
The simple module $L(w,i)$ coincides with the determinantial module $M = M(ws_i\Lambda_i,w\Lambda_i)$, 
since their characters are the same unipotent quantum minor $D(ws_i\Lambda_i,w\Lambda_i)$, see \cite[Section 5.3]{murata2025affinehighestweightstructures}. 
For $m \geq 0$, the module $L(w,i^m)$ is isomorphic to $q_i^{m(m-1)/2} M^{\circ m}$, since $L(i^m) \simeq q_i^{m(m-1)/2} L(i)^{\circ m}$. 
Let $\widehat{M} = \widehat{M}(ws_i\Lambda_i,w\Lambda_i)$ be the affinization of $M$ as in \cite[Section 3.6]{murata2025affinehighestweightstructures}. 
By \cite[Definition 4.27]{murata2025affinehighestweightstructures}, for $m \in \mathbb{Z}_{\geq 0}$, there exists $\widehat{M}^{\circ (m)} \in \gMod{R(mw\alpha_i)}$ such that 
\[
(\widehat{M}^{\circ (m)})^{\oplus [m]_i!} \simeq \widehat{M}^{\circ m}. 
\]
By the proof of \cite[Lemma 5.15]{murata2025affinehighestweightstructures}, we have 
\begin{align*} 
\qdim \HOM_{R(w\alpha_i)} (\widehat{M}^{\circ (m)}, q_i^{m(m-1)/2}M^{\circ m}) = 1, \\
\qdim \EXT_{R(w\alpha_i)}^1 (\widehat{M}^{\circ (m)}, q_i^{m(m-1)/2}M^{\circ m}) = 0. 
\end{align*}
Hence, $\Delta(w,i^m)$ is isomorphic to $\widehat{M}^{\circ (m)}$. 
\end{remark}

\begin{lemma} \label{lem:std2}
Let $w \in W, i,j \in I$ and $m \in \mathbb{Z}_{\geq 0}$. 
If $w\alpha_i = \alpha_j$, we have 
\[
L(w,i^m) \simeq L(j^m), \ \Delta(w,i^m) \simeq P(j^m). 
\]
\end{lemma}

\begin{proof}
By Lemma \ref{lem:std}, $L(w,i^m)$ is a self-dual simple $R(m\alpha_j)$-module. 
Since $L(j^m)$ is the unique self-dual simple $R(m\alpha_j)$-module, we must have $L(w,i^m) \simeq L(j^m)$. 
By Lemma \ref{lem:std}, $\Delta(w,i^m)$ is the projective cover of $L(j^m)$, that is, $P(j^m)$.
\end{proof}

\begin{definition}
Let $w,v\in W$, and $\beta \in \mathsf{Q}_+$. 

(1) We define 
\[
{}_wR(\beta) = R(\beta)/{}_wI,\ R_v(\beta) = R(\beta)/I_v,\ {}_wR_v(\beta) = R(\beta)/{}_wI_v, 
\] \index{${}_wR_v(\beta)$}
where
\begin{align*}
{}_wI &= \langle e(\nu) \mid \text{$\nu \in I^{\beta}, \alpha_{\nu_1}+ \cdots + \alpha_{\nu_k} \not \in w\mathsf{Q}_+$ for some $1 \leq k \leq \height \beta$} \rangle, \\
I_v &= \langle e(\nu) \mid \text{$\nu \in I^{\beta}, \alpha_{\nu_{\height \beta}}+ \cdots + \alpha_{\nu_k} \not \in v\mathsf{Q}_+$ for some $1 \leq k \leq \height \beta$} \rangle, \\
{}_wI_v &= {}_wI + I_v. 
\end{align*}

(2) We define 
\[
R_{w,v}(\beta) = R(\beta)/ (I_{w,e} + I_v), 
\] \index{$R_{w,v}(\beta)$}
where 
\begin{align*}
I_{w,e} = \langle e(\nu) \mid \nu \in I^{\beta}, \alpha_{\nu_1} + \cdots + \alpha_{\nu_k} \not \in w\mathsf{Q}_- \text{for some $1 \leq k \leq \height \beta$} \rangle.  \\
\end{align*}
\end{definition}

Note that ${}_wR(\beta) = {}_wR_e(\beta), R_v(\beta) = {}_eR_v(\beta)$ and ${}_eR_e(\beta) = R(\beta)$.
By the Mackey-filtration (Proposition \ref{prop:Mackey}), the categories 
\[
\gMod{{}_wR_v} = \bigoplus_{\beta \in \mathsf{Q}_+} \gMod{{}_wR_v(\beta)}, \ \gMod{R_{w,v}} = \bigoplus_{\beta \in \mathsf{Q}_+} \gMod{R_{w,v}(\beta)}
\]
are closed under convolution products. 

Let $v \in W$. 
Fix a reduced word $\underline{v} = (i_1, \ldots, i_m)$ of $v$. 
%From now on, we usually omit $\underline{v}$ from the notation. 
For $1 \leq k \leq m$ and $c \in \mathbb{Z}_{\geq 0}$, we define
\[
\beta_{\underline{v},k} = s_{i_1} \cdots s_{i_{k-1}}\alpha_{i_k}, \  \Delta_{\underline{v},k}^{(c)} = \Delta(s_{i_1}\cdots s_{i_{k-1}}, i_k^c).  
\]

\begin{definition}
We work in the setting above. 
Let $\beta \in \mathsf{Q}_+$. 
\begin{enumerate}
\item Let $S_v(\beta)$ be a complete set of representatives of simple graded $R_v(\beta)$-modules up to isomorphism and grading shift. 
\item For $S \in S_v(\beta)$, let $\Delta_v(S)$ be the projective cover of $S$ in $\gMod{R_v(\beta)}$. 
\item Let $\Lambda_{\underline{v}}(\beta)$ (resp. $\overline{\Lambda}_{\underline{v}}(\beta)$) be the set of triples $(\mathbf{c},\gamma,S)$ (resp. pairs $(\mathbf{c},\gamma)$) of 
\begin{itemize}
\item $\mathbf{c}= (\mathbf{c}_1, \ldots, \mathbf{c}_m) \in \mathbb{Z}_{\geq 0}^m$, 
\item $\gamma \in \mathsf{Q}_+$, and 
\item $S \in S(\gamma)$,   
\end{itemize} satisfying 
\[
\sum_{1 \leq k \leq m} \mathbf{c}_k \beta_{\underline{v},k} + \gamma = \beta. 
\]
Note that $\gamma$ is determined by $\mathbf{c}$ by this equation. 
\item We define a map $\rho_{\underline{v}} \colon \Lambda_{\underline{v}}(\beta) \to \overline{\Lambda}_{\underline{v}}(\beta)$ as the projection. 
\item We define a partial order $\leq$ on $\overline{\Lambda}_{\underline{v}}(\beta)$ as follows: for $\lambda = (\mathbf{c}, \gamma), \mu = (\mathbf{d},\delta) \in \overline{\Lambda}_{\underline{v}}(\beta)$, $\lambda \leq \mu$ if the following two conditions hold. 
\begin{itemize}
\item $\mathbf{c} \leq \mathbf{d}$ in the lexicographic order, that is, either (i) there exists $1 \leq k \leq m$ such that $\mathbf{c}_p = \mathbf{d}_p \ (1 \leq p \leq k-1)$ and $\mathbf{c}_k < \mathbf{d}_k$, or (ii) $\mathbf{c} = \mathbf{d}$ holds. 
\item If $\delta = 0$, then $\gamma = 0$ and $\mathbf{c} \leq \mathbf{d}$ in the lexicographic order from the right, that is, either (i) there exists $1 \leq k \leq m$ such that $\mathbf{c}_p = \mathbf{d}_p \ (k+1 \leq p \leq m)$ and $\mathbf{c}_k < \mathbf{d}_k$, or (ii) $\mathbf{c} = \mathbf{d}$ holds.
\end{itemize}
\item For $\lambda = (\mathbf{c},\gamma,S) \in \Lambda_{\underline{v}}(\beta)$, we define 
\[
\Delta_{\underline{v}}(\lambda) = \Delta_v(S) \circ \Delta_{\underline{v},m}^{(\mathbf{c}_m)} \circ \cdots \circ \Delta_{\underline{v},1}^{(\mathbf{c}_1)}. 
\]
Let $L_{\underline{v}}(\lambda)$ be the head of $\Delta_{\underline{v}}(\lambda)$. 
\item Let $\Lambda'_{\underline{v}}(\beta)$ be the subset of $\Lambda_{\underline{v}}(\beta)$ consisting of elements $(\mathbf{c},\gamma = 0,S = \mathbf{1})$.  
\end{enumerate}
\end{definition}

Note that $\rho_{\underline{v}}(\Lambda'_{\underline{v}}(\beta))$ is an ideal of the partially ordered set $\overline{\Lambda}_{\underline{v}}(\beta)$, that is, it is downward closed. 
Also note that $\rho_{\underline{v}}$ is injective on $\Lambda'_{\underline{v}}(\beta)$. 
By Remark \ref{rem:determinantialmodule}, our $\Delta_{\underline{v}}(\lambda)$ coincides with the module defined in \cite[Definition 5.14]{murata2025affinehighestweightstructures}. 

\begin{theorem} \label{thm:stratification}
We work in the setting above. 
\begin{enumerate}
\item The category $\gMod{R(\beta)}$ is stratified in the sense of Kleshchev \cite{MR3335289}, with respect to $\rho_{\underline{v}} \colon \Lambda_{\underline{v}}(\beta) \to \overline{\Lambda}_{\underline{v}}(\beta)$, the partial order $\leq$ on $\overline{\Lambda}_{\underline{v}}(\beta)$, 
and the standard modules $\Delta_{\underline{v}}(\lambda) \ (\lambda \in \Lambda_{\underline{v}}(\beta))$.  
\item The category $\gMod{R_{v,e}(\beta)}$ is an affine highest weight category in the sense of Kleshchev \cite{MR3335289}, with respect to $\rho_{\underline{v}} \colon \Lambda'_{\underline{v}}(\beta) \to \overline{\Lambda}_{\underline{v}}(\beta)$, the partial order $\leq$, and the standard modules $\Delta_{\underline{v}}(\lambda) \ (\lambda \in \Lambda'_{\underline{v}}(\beta))$. 
\end{enumerate}
\end{theorem}

\begin{proof}
(1) is \cite[Theorem 5.18]{murata2025affinehighestweightstructures}, and (2) is \cite[Theorem 5.21]{murata2025affinehighestweightstructures}. 
\end{proof}

\begin{corollary} \label{cor:criteria}
We work in the setting above. 
(1) For $\lambda = (\mathbf{c}, \gamma,S) \in \Lambda_{\underline{v}}(\beta)$, the simple module $L_{\underline{v}}(\lambda)$ belongs to $\gMod{R_v(\beta)}$ if and only if $\mathbf{c}= 0$. 
When this holds, we have $L_{\underline{v}} \simeq S$. 

(2) For $M \in \gMod{R(\beta)}$, the following statements are equivalent: 
\begin{itemize}
  \item $M \in \gMod{R_v(\beta)}$. 
  \item For any $\lambda (\mathbf{c},\gamma,S) \in \Lambda_{\underline{v}}(\beta)$ with $\mathbf{c} \neq 0$, we have $\HOM_{R(\beta)}(\Delta_{\underline{v}}(\lambda),M) =0$. 
\end{itemize}
\end{corollary}

\begin{proof}
By Theorem \ref{thm:stratification}, $\{ L_{\underline{v}}(\lambda) \mid \lambda \in \Lambda_{\underline{v}}(\beta) \}$ is a complete set of representatives of simple $R(\beta)$-modules up to isomorphism and grading shift. 
By definition, we have $L_{\underline{v}}(0, \beta, S) = S$ for any $S \in S_v(\beta)$. 
Since $S_v(\beta)$ is a complete set of representatives of simple $R_v(\beta)$-modules, (1) follows.

As for (2), note that for $\lambda = (\mathbf{c},\gamma,S), \mu = (\mathbf{d},\delta,T) \in  \Lambda_{\underline{v}}(\beta)$, if $\rho_{\underline{v}}(\lambda) \leq \rho_{\underline{v}}(\mu)$ and $\mathbf{d} = 0$, then $\mathbf{c} = 0$. 
Hence, the assertion follows from standard argument by Theorem \ref{thm:stratification} and (1). 
\end{proof}

Next, let $w \in W$ and fix a reduced expression $\underline{w} = (i_1, \ldots, i_m)$ of $w$. 
%We usually omit $\underline{w}$ from the notation. 
For $1 \leq k \leq m$ and $c \in \mathbb{Z}_{\geq 0}$, we define
\[
{}_{\underline{w}}\beta_k = s_{i_1} \cdots s_{i_{k-1}}\alpha_{i_k}, \ {}_{\underline{w}}{\Delta}_k^{(c)}= {\Delta'}(s_{i_1}\cdots s_{i_{k-1}}, i_k^c).  
\]

\begin{definition}
We work in the setting above. 
Let $\beta \in \mathsf{Q}_+$. 
\begin{enumerate}
\item Let ${}_wS(\beta)$ be a complete set of representatives of simple graded ${}_wR(\beta)$-modules up to isomorphism and grading shift. 
\item For $S \in {}_wS(\beta)$, let ${}_w\Delta(S)$ be the projective cover of $S$ in $\gMod{{}_wR(\beta)}$. 
\item Let ${}_{\underline{w}}\Lambda(\beta)$ (resp. ${}_{\underline{w}}\overline{\Lambda}(\beta)$) be the set of triples $(\mathbf{c},\gamma,S)$ (resp. pairs $(\mathbf{c},\gamma)$) of 
\begin{itemize}
\item $\mathbf{c}= (\mathbf{c}_1, \ldots, \mathbf{c}_m) \in \mathbb{Z}_{\geq 0}^m$, 
\item $\gamma \in \mathsf{Q}_+$, and 
\item $S \in {}_wS(\gamma)$,   
\end{itemize} satisfying 
\[
\sum_{1 \leq k \leq m} \mathbf{c}_k\, {}_{\underline{w}}\beta_k + \gamma = \beta. 
\]
Note that $\gamma$ is determined by $\mathbf{c}$ by this equation. 
\item We define a map ${}_{\underline{w}}\rho \colon {}_{\underline{w}}\Lambda(\beta) \to {}_{\underline{w}}\overline{\Lambda}(\beta)$ as the projection. 
\item We define a partial order $\leq$ on ${}_{\underline{w}}\overline{\Lambda}(\beta)$ as follows: for $\lambda = (\mathbf{c}, \gamma), \mu = (\mathbf{d},\delta) \in {}_{\underline{w}}\overline{\Lambda}(\beta)$, $\lambda \leq \mu$ if the following two conditions hold. 
\begin{itemize}
\item $\mathbf{c} \leq \mathbf{d}$ in the lexicographic order, that is, either (i) there exists $1 \leq k \leq m$ such that $\mathbf{c}_p = \mathbf{d}_p \ (1 \leq p \leq k-1)$ and $\mathbf{c}_k < \mathbf{d}_k$, or (ii) $\mathbf{c} = \mathbf{d}$ holds. 
\item If $\delta = 0$, then $\gamma = 0$ and $\mathbf{c} \leq \mathbf{d}$ in the lexicographic order from the right, that is, either (i) there exists $1 \leq k \leq m$ such that $\mathbf{c}_p = \mathbf{d}_p \ (k+1 \leq p \leq m)$ and $\mathbf{c}_k < \mathbf{d}_k$, or (ii) $\mathbf{c} = \mathbf{d}$ holds.
\end{itemize}
\item For $\lambda = (\mathbf{c},\gamma,S) \in {}_{\underline{w}}\Lambda(\beta)$, we define 
\[
{}_{\underline{w}}\Delta(\lambda) = {}_{\underline{w}}{\Delta}_1^{(\mathbf{c}_1)} \circ \cdots \circ {}_{\underline{w}}{\Delta}_m^{(\mathbf{c}_m)} \circ {}_w\Delta(S). 
\]
\end{enumerate}
\end{definition}

\begin{theorem}\label{thm:stratification'}
In the setting above, the category $\gMod{R(\beta)}$ is stratified in the sense of Kleshchev \cite{MR3335289}, with respect to ${}_{\underline{w}}\rho \colon {}_{\underline{w}}\Lambda(\beta) \to {}_{\underline{w}}\overline{\Lambda}(\beta)$, 
the partial order $\leq$ on ${}_{\underline{w}}\overline{\Lambda}(\beta)$, 
and the standard modules ${}_{\underline{w}}\Delta(\lambda) \ (\lambda \in {}_{\underline{w}}\Lambda(\beta))$.  
\end{theorem}

\begin{proof}
It follows from Theorem \ref{thm:stratification} by applying the involution $\sigma$. 
\end{proof}

The following corollary is similar to Corollary \ref{cor:criteria}

\begin{corollary} \label{cor:criteria'}
In the setting above, for $M \in \gMod{R(\beta)}$, the following statements are equivalent: 
\begin{itemize}
  \item $M \in \gMod{{}_wR(\beta)}$. 
  \item For any $\lambda (\mathbf{c},\gamma,S) \in {}_{\underline{w}}\Lambda(\beta)$ with $\mathbf{c} \neq 0$, we have $\HOM_{R(\beta)}({}_{\underline{w}}\Delta(\lambda),M) =0$. 
\end{itemize}
\end{corollary}

\subsection{Equivalences between various subcategories}

\begin{comment}
\begin{lemma} \label{lem:characterization} 
In the setting above, let $X \in \gMod{{}_iR(\beta)}, Y \in \gMod{R_i(\beta)}$. 
\begin{enumerate}
\item[(1)] $X \in \gMod{{}_{s_iw}R}$ if and only if $\HOM(\Delta'(\gamma) \circ {}_iR(\beta-\gamma), X) = 0$ for any $\gamma \in \Phi_+ \cap s_iw\Phi_- \setminus \{\alpha_i\}$. 
\item[(2)] $X \in \gMod{R_v}$ if and only if $\HOM ({}_iR(\beta-\gamma) \circ \Delta(\gamma), X) = 0$ for any $\gamma \in \Phi_+ \cap v\Phi_-$.    
\item[(3)] $Y \in \gMod{{}_wR}$ if and only if $\HOM(\Delta'(\gamma) \circ R_i(\beta-\gamma), Y) = 0$ for any $\gamma \in \Phi_+ \cap w \Phi_-$. 
\item[(4)] $Y \in \gMod{R_{s_iv}}$ if and only if $\HOM(R_i(\beta-\gamma)\circ \Delta(\gamma), Y) - 0$ for any $\gamma \in \Phi_+ \cap v\Phi_- \setminus \{\alpha_i\}$. 
\end{enumerate}
\end{lemma}
\end{comment}

\begin{theorem} \label{thm:categorificatoinwRv}
Let $w, v \in W$. 
The homomorphism $K(\gmod{{}_wR_v})_{\mathbb{Q}(q)} \to K(\gMod{{}_wR_v})_{\mathbb{Q}(q)}$ induced by the inclusion is an isomorphism.
Furthermore, the isomorphism $\chi \colon K(\gMod{R})_{\mathbb{Q}(q)} \to U_q^-(\mathfrak{g})$ restricts to an isomorphism
\[
K(\gMod{{}_wR_v})_{\mathbb{Q}(q)} \simeq T_{w^{-1}}^{-1}U_q^-(\mathfrak{g}) \cap U_q^-(\mathfrak{g}) \cap T_vU_q^-(\mathfrak{g}). 
\]
\end{theorem} 

\begin{proof}
The former assertion is proved in the same way as Theorem \ref{thm:categorification2}. 

When $w = e$, the latter assertion is \cite[Theorem 5.27]{murata2025affinehighestweightstructures}. 
By applying $\sigma$, we deduce the assertion for the case where $v = e$. 
In general, we have 
\begin{align*}
&K(\gMod{{}_wR_v})_{\mathbb{Q}(q)} \\
&= K(\gMod{{}_wR})_{\mathbb{Q}(q)} \cap K(\gMod{R_v})_{\mathbb{Q}(q)} \\ 
&\quad \text{since simple modules give a compatible basis} \\
&\simeq (T_{w^{-1}}^{-1}U_q^-(\mathfrak{g}) \cap U_q^-(\mathfrak{g})) \cap (T_vU_q^-(\mathfrak{g}) \cap U_q^-(\mathfrak{g})) \quad \text{by the discussion above} \\
&= T_{w^{-1}}^{-1}U_q^-(\mathfrak{g}) \cap U_q^-(\mathfrak{g}) \cap T_vU_q^-(\mathfrak{g}). 
\end{align*}
\end{proof}

Let $w,v \in W$ and $i \in I$.
Assume that $s_iw > w$ and $s_iv <v$. 
Note that 
\[
\gMod{{}_{s_iw}R_{s_iv}} \subset \gMod{{}_iR}, \ \gMod{{}_wR_v} \subset \gMod{R_i}. 
\]

\begin{theorem} \label{thm:variousequiv}
In the setting above, the equivalence $S_i \colon \gMod{{}_iR} \to \gMod{R_i}$ restricts to an equivalence 
\[
\gMod{{}_{s_iw}R_{s_iv}} \simeq \gMod{{}_wR_v}. 
\]
\end{theorem}

\begin{proof}
Let $X \in \gMod{{}_{s_iw}R_{s_iv}(\beta)}$. 
We prove that $S_i(X) \in \gMod{{}_wR_v}$. 
We may assume $S_i(X) \neq 0$. 

First, we prove that $S_i(X) \in \gMod{{}_wR}$. 
We fix a reduced word $\underline{w} = (i_1,\ldots,i_m)$ of $w$, 
and use Theorem \ref{thm:stratification'} and the notation there. 
Let $\lambda = (\mathbf{c},\gamma,S) \in {}_{\underline{w}}\Lambda(s_i\beta)$ be an arbitrary element such that 
\[
\HOM_{R(s_i\beta)}({}_{\underline{w}}\Delta(\lambda),S_i(X)) \neq 0.
\] 
Suppose $\mathbf{c} \neq 0$ and let $1 \leq k \leq m$ be the least integer such that $\mathbf{c}_k \neq 0$. 
By the induction-restriction adjunction, it follows that 
\[
\HOM_{R(\mathbf{c}_k\, {}_{\underline{w}}\beta_k)}({}_{\underline{w}}{\Delta}_k^{(\mathbf{c}_k)}, \Res_{\mathbf{c}_k\, {}_{\underline{w}}\beta_k,*}S_i(X)) \neq 0.
\] 
Since $\Res_{\mathbf{c}_k\, {}_{\underline{w}}\beta_k,*} S_i(X)$ is an $R(\mathbf{c}_k\, {}_{\underline{w}}\beta_k) \otimes R_i(s_i\beta - \mathbf{c}_k\, {}_{\underline{w}}\beta_k)$-module, 
we have 
\[
\HOM_{R(s_i\beta)} ({}_{\underline{w}}{\Delta}_k^{(\mathbf{c}_k)}\circ R_i(s_i\beta - \mathbf{c}_k\, {}_{\underline{w}}\beta_k), S_i(X)) \neq 0.
\] 
Applying $S'_i$, we deduce that 
\[
\HOM_{R(\beta)} (S'_i({}_{\underline{w}}{\Delta}_k^{(\mathbf{c}_k)})\circ S'_i(R_i(s_i\beta-\mathbf{c}_k\, {}_{\underline{w}}\beta_k)), X) \neq 0.
\] 
Hence, $\Res_{\mathbf{c}_k s_i\, {}_{\underline{w}}\beta_k,*}X \neq 0$. 
Since $(s_iw)^{-1}s_i\, {}_{\underline{w}}\beta_k \in \mathsf{Q}_-$, it contradicts the assumption $X \in \gMod{{}_{s_iw}R}$. 
Therefore, $\mathbf{c} = 0$ and ${}_{\underline{w}}\Delta(\lambda) \in \gMod{{}_wR(\beta)}$. 
Since $\lambda$ is an arbitrary element satisfying $\HOM_{R(s_i\beta)}({}_{\underline{w}}\Delta(\lambda),S_i(X)) \neq 0$, 
we deduce from Corollary \ref{cor:criteria'} that $S_i(X) \in \gMod{{}_wR}$ as desired. 

Next, we prove that $S_i(X) \in \gMod{R_v}$.
We fix a reduced word $\underline{v} = (j_1 = i, j_2,\ldots,j_n)$ of $v$,
and use Theorem \ref{thm:stratification} (1) and the notation there. 
Let $\lambda = (\mathbf{c},\gamma,S) \in \Lambda_{\underline{v}}(s_i\beta)$ be an arbitrary element such that 
\[
\HOM_{R(s_i\beta)}(\Delta_{\underline{v}}(\lambda),S_i(X)) \neq 0.
\] 
Suppose $\mathbf{c} \neq 0$ and let $1 \leq k \leq m$ be the least integer such that $\mathbf{c}_k \neq 0$.
Since $S_i(X) \in \gMod{R_i}$ and $j_1 = i$, we have $k \geq 2$. 
Then, $\Delta_{\underline{v}}(\lambda) \simeq \Delta_{\underline{v}}(\lambda') \circ \Delta_{\underline{v},k}^{(\mathbf{c}_k)}$, 
where $\lambda'$ is the same as $\lambda$ except that the component $\mathbf{c}_k$ is zero. 
Note that both $\Delta_{\underline{v}}(\lambda')$ and $\Delta_{\underline{v},k}^{(\mathbf{c}_k)}$ belong to $\gMod{R_i}$. 
Applying $S'_i$, we deduce 
\[
\HOM_{R(\beta)}(S'_i(\Delta_{\underline{v}}(\lambda')) \circ S'_i(\Delta_{\underline{v},k}^{(\mathbf{c}_k)}), X) \neq 0.
\]
Hence, $\Res_{\beta - \mathbf{c}_k s_i\beta_{\underline{v},k},\mathbf{c}_k s_i\beta_{\underline{v},k}} X \neq 0$. 
Since $(s_iw)^{-1}(s_i\beta_{\underline{v},k}) \in \mathsf{Q}_-$, it contradicts the assumption $X \in \gMod{R_{s_iv}}$.
Therefore, $\mathbf{c} = 0$ and $\Delta_{\underline{v}}(\lambda) \in \gMod{R_v(\beta)}$. 
Since $\lambda$ is an arbitrary element satisfying $\HOM_{R(s_i\beta)}(\Delta_{\underline{v}}(\lambda),S_i(X)) \neq 0$, 
we deduce from Corollary \ref{cor:criteria} that $S_i(X) \in \gMod{R_v}$ as desired. 

We can prove $S'_i (\gMod{{}_wR_v}) \subset \gMod{{}_{s_iw}R_{s_iv}}$ in the same manner. 
The theorem is proved. 
\end{proof}

\begin{comment}
\begin{theorem}
Let $w,v \in W$, and assume $\ell(w^{-1}v) = \ell(w) + \ell(v)$. 
The homomorphism $K(\gmod{{}_wR_v})_{\mathbb{Q}(q)} \to K(\gMod{{}_wR_v})_{\mathbb{Q}(q)}$ induced by the inclusion is an isomorphism.
Furthermore, the isomorphism $\chi \colon K(\gMod{R})_{\mathbb{Q}(q)} \to U_q^-(\mathfrak{g})$ restricts to an isomorphism
\[
K(\gMod{{}_wR_v})_{\mathbb{Q}(q)} \simeq T_{w^{-1}}^{-1}U_q^-(\mathfrak{g}) \cap T_vU_q^-(\mathfrak{g})
\]
\end{theorem} 

\begin{proof}
The former assertion is proved in the same manner as Theorem \ref{thm:braidcategorification} (2). 

When $w = e$, the latter assertion is \cite[Theorem 5.27]{murata2025affinehighestweightstructures}. 
In general, it is deduced from Theorem \ref{thm:braidcategorification}, Theorem \ref{thm:variousequiv} and the fact that 
\[
T_{w^{-1}}^{-1}(U_q^-(\mathfrak{g}) \cap T_{w^{-1}v}U_q^-(\mathfrak{g})) = T_{w^{-1}}^{-1}U_q^-(\mathfrak{g}) \cap T_vU_q^-(\mathfrak{g}). 
\]
\end{proof}
\end{comment}

\begin{theorem}
Let $w, v \in W$. 
The homomorphism $K(\gmod{R_{w,v}})_{\mathbb{Q}(q)} \to K(\gMod{R_{w,v}})_{\mathbb{Q}(q)}$ induced by the inclusion is an isomorphism. 
Furthermore, the isomorphism $\chi \colon K(\gMod{R})_{\mathbb{Q}(q)} \to U_q^-(\mathfrak{g})$ restricts to an isomorphism
\[
K(\gMod{R_{w,v}})_{\mathbb{Q}(q)} \simeq T_w(U_q^0(\mathfrak{g})U_q^+(\mathfrak{g})) \cap U_q^-(\mathfrak{g}) \cap T_vU_q^-(\mathfrak{g}). 
\]
\end{theorem}

\begin{proof}
The former assertion is proved in the same manner as Theorem \ref{thm:categorification2}. 

When $v = e$, the latter assertion is \cite[Theorem 5.26]{murata2025affinehighestweightstructures}. 
In general, we have
\begin{align*}
&K(\gMod{R_{w,v}})_{\mathbb{Q}(q)} \\
&= K(\gMod{R_{w,e}})_{\mathbb{Q}(q)} \cap K(\gMod{R_v})_{\mathbb{Q}(q)} \\ 
&\quad \text{since simple modules give a compatible basis} \\
&\simeq (T_w(U_q^0(\mathfrak{g})U_q^+(\mathfrak{g})) \cap U_q^-(\mathfrak{g})) \cap (T_vU_q^-(\mathfrak{g}) \cap U_q^-(\mathfrak{g})) \\
&\quad \text{by Theorem \ref{thm:categorificatoinwRv} and the case where $v = e$} \\
&= T_w(U_q^0(\mathfrak{g})U_q^+(\mathfrak{g})) \cap U_q^-(\mathfrak{g}) \cap T_vU_q^-(\mathfrak{g}). 
\end{align*}
\end{proof}

Let $w,v \in W$ and $i \in I$.
Assume that $s_iw > w$ and $s_iv > v$. 
Note that 
\[
\gMod{R_{w,v}} \subset \gMod{{}_iR}, \gMod{R_{s_iw,s_iv}} \subset \gMod{R_i}. 
\]

\begin{theorem}
In the setting above, the equivalence $S_i \colon \gMod{{}_iR} \to \gMod{R_i}$ restricts to an equivalence 
\[
\gMod{R_{w,v}} \simeq \gMod{R_{s_iw,s_iv}}. 
\]
\end{theorem}

\begin{proof}
First, we prove that $S_i(\gMod{R_{w,v}}) \subset \gMod{R_{s_iw,s_iv}}$. 
Let $X \in \gMod{R_{w,v}(\beta)}$. 
By Theorem \ref{thm:variousequiv}, we have $S_i(X) \in \gMod{R_{s_iv}}$. 
Hence, it suffices to prove $S_i(X) \in \gMod{R_{s_iw,e}}$. 
We fix a reduced word $\underline{w}$ of $w$. 
By Theorem \ref{thm:stratification} (2), 
$\gMod{R_w(\beta)}$ is generated by $\Delta_{\underline{w}}(\lambda) \ (\lambda \in \Lambda'_{\underline{w}}(\beta))$ as a Serre subcategory of $\gMod{R(\beta)}$. 
Since $X \in \gMod{R_w(\beta)}$, it follows that $S_i(X)$ belongs to the Serre subcategory generated by $S_i(\Delta_{\underline{w}}(\lambda)) \ (\lambda \in \Lambda'_{\underline{w}}(\beta))$. 
For any $\lambda = (\mathbf{c},0,\mathbf{1}) \in \Lambda'_{\underline{w}}(\beta)$, we have 
\[
S_i(\Delta_{\underline{w}}(\lambda)) \simeq \Delta(s_is_{i_1} \cdots s_{i_{m-1}}, i_m^{\mathbf{c}_m}) \circ \cdots \circ \Delta(s_is_{i_1},i_2^{\mathbf{c}_2}) \circ \Delta(s_i,i_1^{\mathbf{c}_1}). 
\]
By using Theorem \ref{thm:stratification} (2) for $s_iw$, we see that $S_i(\Delta_{\underline{w}}(\lambda)) \in \gMod{R_{s_iw,e}}$. 
hence, we obtain $S_i(X) \in \gMod{R_{s_iw,e}}$. 

Next, we prove $S_i' (\gMod{R_{s_iw,s_iv}}) \subset \gMod{R_{w,v}}$.
Let $X \in \gMod{R_{s_iw,s_iv}(\beta)}$. 
By Theorem \ref{thm:variousequiv}, we have $S_i'(X) \in \gMod{R_v}$. 
Hence, it suffices to prove $S_i'(X) \in \gMod{R_{w,e}}$. 
We fix a reduced word $\underline{s_iw} = (i_1 = i, i_2, \ldots,i_m)$ of $s_iw$, and use Theorem \ref{thm:stratification} (2). 
Note that for $\lambda = (\mathbf{c},0,\mathbf{1}) \in \Lambda'_{\underline{w}}(\beta)$ with $\mathbf{c}_1 \neq 0$, 
we have 
\[
\HOM_{R(\beta)}(\Delta_{\underline{s_iw}}(\lambda), Y) = 0 \ (Y \in \gMod{R_{s_iw,s_i}(\beta)} \subset \gMod{R_i}), 
\]
by the induction-restriction adjunction. 
On the other hand, if $\mathbf{c}_1 = 0$, then 
\[
\Delta_{\underline{s_iw}}(\lambda) \simeq  S_i(\Delta(s_{i_2} \cdots s_{i_{m-1}},i_m^{\mathbf{c}_m})) \circ \cdots \circ S_i(\Delta(s_{i_2},i_3^{\mathbf{c}_3})) \circ S_i(\Delta(e,i_2^{\mathbf{c}_2})), 
\]
which belongs to $\gMod{R_{s_iw,s_i}(s_i\beta)}$ by Theorem \ref{thm:stratification} (2) and the fact $R_{e,s_i}(s_i\beta) = R_i(s_i\beta)$. 
Also note that the subset of $\Lambda'_{\underline{s_iw}}(\beta)$ consisting of elements $(\mathbf{c}, 0, \mathbf{1})$ with $\mathbf{c}_1 = 0$ is an ideal for the partial order. 
These observations imply that the category $\gMod{R_{s_iw,s_i}(\beta)}$ is generated by $\Delta_{\underline{s_iw}}(\lambda) \ (\text{$\lambda = (\mathbf{c},0,\mathbf{1}) \in \Lambda_0(\beta)$ satisfying $\mathbf{c}_1 =0$})$ as a Serre subcategory of $\gMod{R(\beta)}$. 
Hence, $S_i'(X)$ belongs to the Serre subcategory generated by $S_i'(\Delta_{\underline{s_iw}}(\lambda))$ for these $\lambda$. 
Since 
\[
S_i'(\Delta_{\underline{s_iw}}(\lambda)) = \Delta(s_{i_2} \cdots s_{i_{m-1}},i_m^{\mathbf{c}_m}) \circ \cdots \circ \Delta(s_{i_2},i_3^{\mathbf{c}_3}) \circ \Delta(e,i_2^{\mathbf{c}_2}) \in \gMod{R_{w,e}}, 
\]
we obtain $S_i'(X) \in \gMod{R_{w,e}}$. 
\end{proof}

\begin{comment}
\subsubsection{Reflection functors and adjoint action}

\begin{proposition}
Let $M \in \gMod{{}_iR(\beta)}$ and assume that $ME_i = 0$. 
Then, $T_i(M) \simeq F_i^{(-\langle h_i,\beta \rangle)}M$. 
\end{proposition}

\begin{proof}
Put $n = -\langle h_i,\beta \rangle$. 
By Remark \ref{rem:extension}, we have 
\begin{align*}
T_i(M) \simeq T_i(MF_i^{(n)}E_i^{(n)}) \simeq 
\end{align*}
\end{proof}
\end{comment}

\section{Braid relations} \label{chap:braidrel}
In this section, we prove that the functors $S_i \ (i \in I)$ satisfy the braid relations.
Note that we can prove that the functors $S'_i \ (i \in I)$ also satisfy the braid relations by applying the involution $\sigma$. 

\subsection{The algebras ${}_JR$ and $R_J$}

Let $J \subset I$. 

\begin{definition}
For $\beta \in \mathsf{Q}_+$, we define 
\[
{}_JR(\beta) = R(\beta)/\langle e(j,\beta-\alpha_j) \ (j \in J) \rangle, \ R_J(\beta) = R(\beta)/\langle e(\beta-\alpha_j,j) \ (j \in J) \rangle. 
\] \index{$R_J, {}_JR$}
\end{definition}
Let $\gMod{{}_JR} = \bigoplus_{\beta \in \mathsf{Q}_+} \gMod{{}_JR(\beta)}, \gMod{R_J} = \bigoplus_{\beta \in \mathsf{Q}_+} \gMod{R_J(\beta)}$. 
We regard these categories as Serre subcategories of $\gMod{R}$ by inflation. 
Note that ${}_JR(\beta) = {}^{J,0}R(\beta), R_J(\beta) = R^{J,0}(\beta)$. 
In particular, ${}_{\{i\}}R(\beta) = {}_i R(\beta), R_{\{i\}}(\beta) = R_i(\beta)$. 

\begin{lemma}
The categories $\gMod{{}_JR}$ and $\gMod{R_J}$ of $\gMod{R}$ are both closed under convolution products. 
\end{lemma}

\begin{proof}
It follows from considering the Mackey-filtration (Proposition \ref{prop:Mackey}). 
\end{proof}

\begin{lemma} \label{lem:intersubsection}
We have 
\[
\gMod{{}_JR} = \bigcap_{j \in J} \gMod{{}_jR}, \gMod{R_J} = \bigcap_{j \in J} \gMod{R_j},  
\]
as subcategories of $\gMod{R}$. 
\end{lemma}

\begin{proof}
It follows from the definition. 
\end{proof}

Recall that $\gMod{{}_JR}$ is a right $\mathcal{U}_q(\mathfrak{p}_J)$-module and that $\gMod{{}_jR}$ is a right $\mathcal{U}_q(\mathfrak{p}_j)$-module by Theorem \ref{thm:cyclotomic2rep}.

\begin{lemma} \label{lem:resaction}
Let $j \in J$. 
Then, the subcategory $\gMod{{}_JR} \subset \gMod{{}_jR}$ is stable under the right actions of $F_j, E_j$ in $\catquantum{\mathfrak{p}_j}$, 
and the restricted actions on $\gMod{{}_JR}$ coincide with the right actions of $F_j, E_j$ in $\catquantum{\mathfrak{p}_J}$ respectively. 
Similarly, the subcategory $\gMod{R_J} \subset \gMod{R_j}$ is stable under the left actions of $F_j, E_j$ in $\catquantum{\mathfrak{p}_j}$, 
and the restricted actions on $\gMod{R_J}$ coincide with the left actions of $F_j, E_j$ in $\catquantum{\mathfrak{p}_J}$ respectively. 
\end{lemma}

\begin{proof}
Let $M \in \gMod{{}_JR}$ and consider the module $MF_j$ obtained by applying $F_j$ in $\mathcal{U}_q(\mathfrak{p}_j)$. 
By definition, we have $M F_j \in \gMod{{}_jR}$. 
For any $k \in J \setminus \{j \}$, $M \circ R(\alpha_j)$ belongs to $\gMod{{}_kR}$ since $M, R(\alpha_j) \in \gMod{{}_kR}$. 
Since $MF_j$ is a quotient of $M \circ R(\alpha_j)$, it follows that $MF_j \in \gMod{{}_kR}$.  
By Lemma \ref{lem:intersubsection}, we deduce that $MF_j \in \gMod{{}_JR}$. 
As for $E_j$, the assertion is immediate from the definition. 

The proof for $\gMod{R_J}$ is similar. 
\end{proof}

\begin{comment}
Recall that $\gMod{{}_jR}$ is a left $\dotcatquantum{\mathfrak{p}_j}$-module, and $\gMod{R_j}$ is a right $\dotcatquantum{\mathfrak{p}_j}$-module by Theorem \ref{thm:anotheraction}. 

\begin{lemma}
Let $j \in J$. 
The subcategory $\gMod{{}_JR} \subset \gMod{{}_jR}$ is closed under the left action of 
\end{lemma}
\end{comment}

\subsection{Statement}

Let $J \subset I$ be a subset consisting of two elements.
Rename these two elements as $1$ and $2$. 
Let $\mathsf{A}_J = (a_{i,j})_{i,j \in J}$, and assume that it is of finite type: 
it is one of the types $A_1 \times A_1, A_2, B_2$ or $G_2$. 
We define accordingly $h$ to be $2, 3, 4$ or $6$. 
For $1 \leq k \leq h$, put 
\[
i_k = \begin{cases}
1 & \text{if $k$ is odd}, \\
2 & \text{if $k$ is even}. 
\end{cases}
\]

Let $w_J$ be the longest element of the Weyl group associated with $\mathsf{A}_J$, regarded as an element of $W$. 
Note that we have $R_J(\beta) = R_{w_J}(\beta)$ and ${}_JR(\beta) = {}_{w_J}R(\beta)$.
We define $\overline{1} = 2, \overline{2} = 1$. 
We have 
\[
w_J = s_{i_h} \cdots s_{i_1} = s_{\overline{i_h}} \cdots s_{\overline{i_1}}
\]
For $1 \leq k \leq l \leq h$, we define 
\[
S_{[l,k]} = S_{i_l} \cdots S_{i_k}, \ S_{\overline{[l,k]}} = S_{\overline{i_l}} \cdots S_{\overline{i_k}}. 
\]
They are functors defined on certain subcategories. 

\begin{theorem} \label{thm:braidrel}
The following functors are well-defined equivalences: 
\[
S_{[h,1]}, S_{\overline{[h,1]}} \colon \gMod{{}_JR} \rightrightarrows \gMod{R_J}. 
\]
Furthermore, they are naturally isomorphic to each other.  
\end{theorem}

Since ${}_JR(\beta) = {}_{w_J}R(\beta)$ and $R_J(\beta) = R_{w_J}(\beta)$, the former assertion is a consequence of Theorem \ref{thm:variousequiv}. 
\begin{comment}
Regarding the latter assertion, first note that these two equivalences restrict to two equivalences 
\begin{equation} \label{eq:twoequivalences}
\gproj{{}_JR} \rightrightarrows \gproj{R_J}. 
\end{equation}
It suffices to prove that these two restricted functors are naturally isomorphic to each other. 
By Theorem \ref{thm:rightuniversality}, $\gproj{{}_JR} = \gproj{{}^{J,0}R} \simeq {}_J\mathcal{V}(0)$ is a right $\dotcatquantum{\mathfrak{p}_J}$-module generated by $\mathbf{1}$. 
Through the two equivalences (\ref{eq:twoequivalences}), $\gproj{R_J}$ inherits two left $\dotcatquantum{\mathfrak{p}_J}$-actions. 
Hence, the latter assertion is reduced to showing that these two actions on $\gproj{R_J}$ coincide. 
It is enough to prove it for generating 2-morphisms. 
More rigorously, we consider the full subcategory of $\gproj{{}_JR}$ consisting of objects 
\[
\{ q^n \mathbf{1} \Theta \mid n \in \mathbb{Z}, \text{$\Theta$ is a word in $E_1, E_2, F_i \ (i \in I)$} \}
\]
and the action of the generating 2-morphisms on these objects. 
\end{comment}
The rest of this section is devoted to the proof of the latter assertion. 
We usually suppress degree shifts. 

We define $1^*, 2^* \in \{ 1,2\}$ by $w_J \alpha_1 = -\alpha_{1^*}, w_J \alpha_2 = - \alpha_{2^*}$. 
Note that $1^* = i_{h-1}, 2^* = i_h$. 
By Lemma \ref{lem:std2}, there exists isomorphisms 
\begin{equation} \label{eq:fixedisom}
S_{[h,2]}R(\alpha_1) \simeq R(\alpha_{1^*}), \ S_{[h-1,1]}R(\alpha_2) \simeq R(\alpha_{2^*}). 
\end{equation}
We fix such isomorphisms in the discussion below. 
Note that these fixed isomorphisms also determine isomorphisms 
\begin{equation} \label{eq:fixedisom2}
S_{\overline{[h,2]}} R(\alpha_2) \simeq R(\alpha_{2^*}),\ S_{\overline{[h-1,1]}} R(\alpha_1) \simeq R(\alpha_{1^*}),
\end{equation}
since $S_{\overline{[h,2]}} = S_{[h-1,1]}, S_{\overline{[h-1,1]}} = S_{[h,2]}$. 

In this proof, we adopt the following notation.
For $i \in \{1, 2\}$, let $F_i$ and $E_i$ denote the functor given by the left action of $\catquantum{\mathfrak{p}_i}$ on $\gMod{R_i}$ or the right action of $\catquantum{\mathfrak{p}_i}$ on $\gMod{{}_iR}$. 
By Lemma \ref{lem:resaction}, these functors coincide with those given by the left action of $U_q(\mathfrak{p}_J)$ on $\gMod{R_J}$ or the right action of $U_q(\mathfrak{p}_J)$ on $\gMod{{}_JR}$. 
When we apply these $F_i$ or $E_i$ to some module $M \in \gMod{R}$, we implicitly assume that $M$ belongs to the subcategory $\gMod{R_i}$ or $\gMod{{}_iR}$. 

For $i \in \{1,2\}$, 
let $\tilde{F}_i$ denote the functor given by 
\[
\text{$\tilde{F}_iM = R(\alpha_i) \circ M$ (or $M \tilde{F}_i = M \circ R(\alpha_i)$).} 
\]
In diagrams, $\id_{\tilde{F}_i}$ is depicted as a downward strand labeled $\tilde{i}$. 

For $i \in I \setminus \{1,2\}$, let $F_i$ denote the functor given by 
\[
\text{$F_iM = R(\alpha_i) \circ M$ (or $M F_i = M \circ R(\alpha_i)$).} 
\]
It coincides with the left action of $\catquantum{\mathfrak{p}_J}$ on $\gMod{R_J}$ or its right action on $\gMod{{}_JR}$.

Recall that $S_1$ and $S_2$ are monoidal for the natural isomorphisms given in Proposition \ref{prop:monoidality} denoted by $\theta$. 
Hence, the functors $S_{[l,k]}$ and $S_{\overline{[l,k]}}$ are also monoidal for the natural isomorphisms obtained by compositions of $\theta$. 
These natural isomorphisms will be denoted by $\theta$ as well: 
\[
S_{[l,k]}(X) \circ S_{[l,k]}(Y) \xrightarrow{\theta} S_{[l,k]}(X \circ Y), \ S_{\overline{[l,k]}}(X) \circ S_{\overline{[l,k]}}(Y) \xrightarrow{\theta} S_{\overline{[l,k]}}(X \circ Y). 
\]

\subsection{Construction of natural isomorphisms}

In this section, we define natural isomorphisms 
\begin{align*}
&S_{[h,1]}(MF_i) \simeq E_{i^*}S_{[h,1]}(M), \ S_{[h,1]}(ME_i) \simeq F_{i^*}S_{[h,1]}(M) \ (i \in \{1,2\}), \\
&S_{[h,1]}(MF_i) \simeq S_{[h,1]}(M) \circ \Delta(w_J\alpha_i) \ (i \in I \setminus \{1,2\}), 
\end{align*}
for $M \in \gMod{{}_JR}$. 

Let $M \in \gMod{{}_JR(\beta)}$. 
We construct a natural isomorphism $S_{[h,1]}(ME_1) \simeq F_{1^*}S_{[h,1]}(M)$ as follows. 
Note that $S_1(ME_1) \simeq F_1S_1(M)$, which belongs to $\gMod{{}_{s_1w_J}R_{s_1}}$ by Theorem \ref{thm:variousequiv}. 
Hence, $S_{[h,1]}(ME_1) \simeq S_{[h,2]}(F_1S_1(M))$. 
On the other hand, we have a natural isomorphism
\[
S_{[h,2]}(R(\alpha_1) \circ S_1(M)) \xleftarrow[]{\theta} S_{[h,2]}(R(\alpha_1)) \circ S_{[h,2]}S_1(M) \overset{(\ref{eq:fixedisom})}{\simeq}  R(\alpha_{1^*}) \circ S_{[h,2]}S_1(M).  
\]
Note that $R(\alpha_1) \in \gMod{{}_{s_1w_J}R}$ since $(s_1w_J)^{-1}\alpha_1 = \alpha_{1^*} \in \mathsf{Q}_+$. 

\begin{lemma} \label{lem:braidnat}
Let $M \in \gMod{{}_JR}$. 
The isomorphism $S_{[h,2]}(R(\alpha_1) \circ S_1(M)) \simeq R(\alpha_{1^*}) \circ S_{[h,2]}S_1(M)$ above 
induces a natural isomorphism 
\[
S_{[h,2]}(F_1S_1(M)) \simeq F_{1^*}S_{[h,2]}S_1(M),
\]
through the canonical surjections
\[
S_{[h,2]}(R(\alpha_1) \circ S_1(M)) \twoheadrightarrow S_{[h,2]}(F_1S_1(M)), \ R(\alpha_{1^*}) \circ S_{[h,2]}S_1(M) \twoheadrightarrow F_{1^*}S_{[h,2]}S_1(M). 
\]
\end{lemma}

\begin{proof}
By Theorem \ref{thm:functorF2}, we have a short exact sequence 
\[
0 \to S_1(M) \circ R(\alpha_1) \xrightarrow{\mathsf{R}_{S_1(M)}} R(\alpha_1) \circ S_1(M) \to F_1 S_1(M) \to 0, 
\]
where $\mathsf{R}_{S_1(M)}$ is homogeneous of degree $-(\alpha_1,s_1\beta)$. 
By applying $S_{[h,2]}$ to this short exact sequence and using the isomorphism (\ref{eq:fixedisom}), we obtain a short exact sequence 
\begin{align} \label{eq:SES}
0 &\to S_{[h,2]} S_1(M) \circ R(\alpha_{1^*}) \xrightarrow{S_{[h,2]}(\mathsf{R}_{S_1(M)})} R(\alpha_{1^*}) \circ S_{[h,2]} S_1(M) \\
&\to S_{[h,2]} (F_1S_1(M)) \to 0,  \notag
\end{align}
where $S_{[h,2]}(\mathsf{R}_{S_1(M)})$ is homogeneous of degree $-(\alpha_1,s_1\beta) = -(\alpha_{1^*}, w_J\beta)$. 
On the other hand, Theorem \ref{thm:functorF2} yields a short exact sequence 
\[
0 \to S_{[h,2]}S_1(M) \circ R(\alpha_{1^*}) \xrightarrow{\mathsf{R}_{S_{[h,2]}S_1(M)}} R(\alpha_{1^*}) \circ S_{[h,2]}S_1(M) \to F_{1^*}S_{[h,2]}S_1(M) \to 0, 
\]
where $\mathsf{R}_{S_{[h,2]}S_1(M)}$ is homogeneous of degree $-(\alpha_{1^*},w_J\beta)$. 
These two short exact sequences show that 
\[
\chi(S_{[h,2]}(F_1S_1(M))) = \chi(F_{1^*}S_{[h,2]}S_1(M)). 
\]

Since $S_{[h,2]}(F_1S_1(M)) \simeq S_{[h,1]}(ME_1) \in \gMod{R_J} \subset \gMod{R_{1^*}}$, 
the surjective homomorphism $R(\alpha_{1^*}) \circ S_{[h,2]}S_1(M) \to S_{[h,2]}(F_1S_1(M))$ in (\ref{eq:SES}) induces a surjective homomorphism 
\[
F_{1^*} S_{[h,2]}S_1(M) \to S_{[h,2]}(F_1S_1(M)). 
\]
Since $\chi(S_{[h,2]}(F_1S_1(M))) = \chi(F_{1^*}S_{[h,2]}S_1(M))$, it must be an isomorphism. 

\begin{comment}
Note that $S_{[h,2]}(\mathsf{R}_{S_1(M)})$ is natural in $M \in \gMod{{}_JR}$, hence it is natural in $S_{[h,2]}S_1(M) \in \gMod{R_J}$. 
Moreover, it commutes with the action of $\END_R(R(\alpha_1)) \simeq R(\alpha_1)$, hence commutes with the action of $\END_R(R(\alpha_{1^*})) \simeq R(\alpha_{1^*})$. 
Since we have 
\begin{align*}
&\HOM_{(R(w_J\beta + \alpha_{1^*}), R_J(w_J\beta)\otimes R(\alpha_{1^*}))}(R_J(w_J\beta) \circ R(\alpha_{1^*}), R(\alpha_{1^*}) \circ R_J(w_J\beta)) \\
&\simeq \HOM_{(R(w_J\beta)\otimes R(\alpha_{1^*}), R_J(w_J\beta) \otimes R(\alpha_{1^*}))} (R_J(w_j\beta) \otimes R(\alpha_{1^*}), q^{-(\alpha_{1^*},\beta)R_J(w_J\beta) \otimes R(\alpha_{1^*})}) \\
&\quad \text{by the induction-restriction adjunction and the Mackey-filtration} \\ 
&\simeq Z(R_J(w_J\beta)) \otimes R(\alpha_{1^*}). 
\end{align*}
\end{comment}
\end{proof}

By composing the isomorphism $S_1(ME_1)\simeq F_1S_1(M)$ and the one in Lemma \ref{lem:braidnat}, 
we obtain a natural isomorphism
\begin{equation} \label{eq:isomE1}
S_{[h,1]}(ME_1) \simeq S_{[h,2]}(F_1S_1(M)) \simeq F_{1^*}S_{[h,2]}S_1(M)= F_{1^*}S_{[h,1]}(M). 
\end{equation}

Similarly, the isomorphism
\begin{align*}
&S_{[h-1,1]}(M \circ R(\alpha_2)) \xleftarrow[]{\theta} S_{[h-1,1]}(M) \circ S_{[h-1,1]}(R(\alpha_2)) \overset{(\ref{eq:fixedisom})}{\simeq} S_{[h-1,1]}(M) \circ R(\alpha_{2^*}) \\
\end{align*}
induces an isomorphism 
\begin{align*}
&S_{[h-1,1]}(MF_2) \simeq S_{[h-1,1]}(M) F_{2^*}, \\
\end{align*}
which yields an isomorphism
\begin{equation}\label{eq:isomF2}
S_{[h,1]}(MF_2) \simeq S_{i_h}(S_{[h-1,1]}(M) F_{2^*}) \simeq E_{2^*} S_{[h,1]}(M). 
\end{equation}

Recall that $S_{[h,1]}\colon \gMod{{}_JR} \to \gMod{R_J}$ is an equivalence. 
By the uniqueness of the adjoint functors, we have natural isomorphisms
\begin{equation} \label{eq:isomF1}
S_{[h,1]}(MF_1) \simeq E_{1^*}S_{[h,1]}(M), \ S_{[h,1]}(ME_2) \simeq F_2S_{[h,1]}(M), 
\end{equation}
such that 
\begin{equation} \label{eq:adjunctioncorrespondence}
S_{[h,1]}(M \xy 0;/r.12pc/: 
(0,0)*{\lcap{i}};
\endxy) = \xy 0;/r.12pc/: 
(0,0)*{\lcap{i^*}};
\endxy S_{[h,1]}(M), \ S_{[h,1]}(M \xy 0;/r.12pc/: 
(0,0)*{\lcup{i}};
\endxy) = \xy 0;/r.12pc/: 
(0,0)*{\lcup{i^*}};
\endxy S_{[h,1]}(M) \ (i \in \{1,2\}), 
\end{equation}
under the isomorphisms (\ref{eq:isomE1}), (\ref{eq:isomF2}) and (\ref{eq:isomF1}). 

By interchanging $1$ and $2$, we can define isomorphisms based on (\ref{eq:fixedisom2})
\begin{equation} \label{eq:isombar}
S_{\overline{[h,1]}}(MF_i) \simeq E_{i^*} S_{\overline{[h,1]}}(M), \ S_{\overline{[h,1]}}(ME_i) \simeq F_{i^*}S_{\overline{[h,1]}}(M) \ (i \in \{1,2\}), 
\end{equation}
under which we have 
\begin{align*}
S_{\overline{[h,1]}}(M \xy 0;/r.12pc/: 
(0,0)*{\lcap{i}};
\endxy) = \xy 0;/r.12pc/: 
(0,0)*{\lcap{i^*}};
\endxy S_{\overline{[h,1]}}(M), \ S_{\overline{[h,1]}}(M \xy 0;/r.12pc/: 
(0,0)*{\lcup{i}};
\endxy) = \xy 0;/r.12pc/: 
(0,0)*{\lcup{i^*}};
\endxy S_{\overline{[h,1]}}(M) \ (i \in \{1,2\}). 
\end{align*}

Let $i \in I \setminus \{1, 2\}$. 
We fix isomorphisms
\begin{equation}\label{eq:fixedisom3}
S_{[h,1]}R(\alpha_i) \simeq \Delta(w_J,i) \simeq S_{\overline{[h,1]}}R(\alpha_i), 
\end{equation}
see Definition \ref{def:rootmodule}. 
Since $MF_i \simeq M \circ R(\alpha_i)$ for $M \in \gMod{{}_JR}$ and the functors $S_{[h,1]}$ and $S_{\overline{[h,1]}}$ are monoidal, these isomorphisms yield isomorphisms
\begin{equation}\label{eq:isomFi}
S_{[h,1]}(MF_i) \simeq  S_{[h,1]}(M) \circ \Delta(w_J,i), \ S_{\overline{[h,1]}}(MF_i) \simeq S_{\overline{[h,1]}}(M) \circ \Delta(w_J,i), 
\end{equation}
using $\theta$. 

The isomorphisms (\ref{eq:isomE1}), (\ref{eq:isomF2}), (\ref{eq:isomF1}), (\ref{eq:isombar}) and (\ref{eq:isomFi}) inductively induce isomorphisms 
\begin{equation} \label{eq:natisom}
S_{[h,1]} (\mathbf{1}F_{\nu}) \simeq S_{\overline{[h,1]}}(\mathbf{1}F_{\nu}) \ \left(\nu \in \bigsqcup_{n \geq 0} I^n\right). 
\end{equation}

\begin{proposition} \label{prop:braidrel}
There exist $b_{i,j} \in \mathbf{k}^{\times} \ (i, j \in I)$ such that 
\begin{itemize}
\item $b_{i,j} b_{j,i} = b_{i,i} = 1 \ (i,j \in I)$, 
\item Identifying the isomorphisms of (\ref{eq:natisom}), we have \[
S_{[h,1]} \left(M \xy
0;/r.12pc/:
(0,0)*{\sdotd{i}}; 
\endxy \right) = S_{\overline{[h,1]}} \left( M \xy 0;/r.12pc/: 
(0,0)*{\sdotd{i}}; 
\endxy \right),\ S_{[h,1]}\left( M \xy 
0;/r.12pc/: 
(0,0)*{\dcross{i}{j}}; 
\endxy\right) = b_{i,j} S_{\overline{[h,1]}}\left( M \xy 
0;/r.12pc/: 
(0,0)*{\dcross{i}{j}}; 
\endxy\right) 
\] 
for $i, j \in I, M \in \gMod{{}_JR}$. 
\end{itemize}
\end{proposition}

Assuming this proposition, we can prove Theorem \ref{thm:braidrel} as follows. 

\begin{proof}[Proof of Thoeorem \ref{thm:braidrel}]
Let $\xi$ be the automorphism of $\catquantum{\mathfrak{p}_J}$ in Proposition \ref{prop:scalarshift} given by $b_{i,j}$ above and $d_{i,\lambda} = 1 \ (i \in J, \lambda \in \mathsf{P})$. 
It induces an autoequivalence of $\gMod{{}_JR}$ denoted by the same letter $\xi$. 
Explicitly, it is given by 
\begin{align*}
&\xi(\mathbf{1}F_{\nu}) = \mathbf{1}F_{\nu}, \\
&\xi \left(\mathbf{1}\xy 0;/r.12pc/:
(0,0)*{\slined{\nu_1}}; 
\endxy \cdots \xy 0;/r.12pc/:
(0,0)*{\sdotd{\nu_k}}; 
\endxy \cdots \xy 0;/r.12pc/:
(0,0)*{\slined{\nu_n}}; 
\endxy \right) = \mathbf{1} \xy 0;/r.12pc/:
(0,0)*{\slined{\nu_1}}; 
\endxy \cdots \xy 0;/r.12pc/:
(0,0)*{\sdotd{\nu_k}}; 
\endxy \cdots \xy 0;/r.12pc/:
(0,0)*{\slined{\nu_n}}; 
\endxy, \\
&\xi \left(\mathbf{1} \xy 0;/r.12pc/:
(0,0)*{\slined{\nu_1}}; 
\endxy \cdots \xy 0;/r.12pc/:
(0,0)*{\dcross{\nu_k}{\nu_{k+1}}}; 
\endxy \cdots \xy 0;/r.12pc/:
(0,0)*{\slined{\nu_n}}; 
\endxy\right) = b_{\nu_k,\nu_{k+1}} \mathbf{1} \xy 0;/r.12pc/:
(0,0)*{\slined{\nu_1}}; 
\endxy \cdots \xy 0;/r.12pc/:
(0,0)*{\dcross{\nu_k}{\nu_{k+1}}}; 
\endxy \cdots \xy 0;/r.12pc/:
(0,0)*{\slined{\nu_n}}; 
\endxy. 
\end{align*}
Note that $\bigoplus_{\nu \in I^{\beta}}\mathbf{1}F_{\nu} \simeq {}_JR(\beta)$, 
and its endomorphism ring as a left ${}_JR(\beta)$-module is isomorphic to ${}_JR(\beta)$. 
Hence, Proposition \ref{prop:braidrel} implies that $S_{[h,1]}$ is naturally isomorphic to $S_{\overline{[h,1]}}\xi$. 
It remains to prove that the functor $\xi$ is naturally isomorphic to the identity functor. 

Let $\beta \in \mathsf{Q}_+$, and put $n = \height \beta$. 
For $\nu \in I^{\beta}$ and $w \in \mathfrak{S}_n$, we define 
\[
b'_{w,\nu} = \prod_{1 \leq k < l \leq n, w(k)>w(l)} b_{\nu_k, \nu_l}. 
\]
Note that we have $b'_{s_kw, \nu} = b_{(w\nu)_k,(w\nu)_{k+1}} b'_{w,\nu}$ since $b_{i,j} b_{j,i} = 1$.
It implies 
\begin{equation} \label{eq:multiplicative}
b'_{wv, \nu} = b'_{w,v\nu}b'_{v,\nu}
\end{equation}
for any $w, v \in \mathfrak{S}_n$ and $\nu \in I^{\beta}$. 

We claim that if $w \in \mathfrak{S}_n$ and $\nu \in I^{\beta}$ satisfy $w\nu = \nu$, then $b'_{w,\nu} = 1$. 
By (\ref{eq:multiplicative}), we may assume $w = (k,l)$, a transposition of $k$ and $l$ for some $1 \leq k < l \leq n$ satisfying $\nu_k = \nu_l$. 
Put $i = \nu_k = \nu_l$. 
Then, 
\[
b'_{w, \nu} = b_{i,i} \prod_{k < m < l}(b_{i,\nu_m} b_{\nu_m,i}) = 1. 
\] 
The claim is proved. 

Fix $\nu^{\beta} \in I^{\beta}$. 
For $\nu \in I^{\beta}$, we choose $w \in \mathfrak{S}_n$ satisfying $w\nu^{\beta} = \nu$, and define 
\[
b''_{\nu} = b'_{w, \nu^{\beta}}. 
\]
By (\ref{eq:multiplicative}) and the claim above, it is independent of the choice of $w$. 
Then, we have a natural isomorphism $\Id \to \xi$ between endofunctors of $\gMod{{}_JR(\beta)}$ given by 
\[
\mathbf{1} F_{\nu} \xrightarrow{b''_{\nu}\id} \mathbf{1} F_{\nu} \ (\nu \in I^{\beta}). 
\]
In fact, (\ref{eq:multiplicative}) shows that the following diagram commutes: 
\[
\begin{tikzcd}
\mathbf{1} F_{\nu} \arrow[r,"b''_{\nu}\id"]\arrow[d,"{\mathbf{1} \xy 0;/r.12pc/:
(0,0)*{\slined{\nu_1}}; 
\endxy \cdots \xy 0;/r.12pc/:
(0,0)*{\dcross{\nu_k}{\nu_{k+1}}}; 
\endxy \cdots \xy 0;/r.12pc/:
(0,0)*{\slined{\nu_n}}; 
\endxy}"'] & \mathbf{1} F_{\nu} \arrow[d,"{b_{\nu_k,\nu_{k+1}}\mathbf{1} \xy 0;/r.12pc/:
(0,0)*{\slined{\nu_1}}; 
\endxy \cdots \xy 0;/r.12pc/:
(0,0)*{\dcross{\nu_k}{\nu_{k+1}}}; 
\endxy \cdots \xy 0;/r.12pc/:
(0,0)*{\slined{\nu_n}}; 
\endxy}"] \\
\mathbf{1} F_{s_k\nu} \arrow[r,"b''_{s_k\nu}\id"] & \mathbf{1} F_{s_k\nu}
\end{tikzcd}
\]
Therefore, the theorem is proved. 
\end{proof}

We will prove Proposition \ref{prop:braidrel} in the subsequent section. 

\subsection[Proof of Proposition]{Proof of Proposition {\ref{prop:braidrel}}}

The isomorphism $S_{[h,2]}R(\alpha_1) \simeq R(\alpha_{1^*})$ of (\ref{eq:fixedisom}) induces an isomorphism of graded algebras
\begin{align*}
&R(\alpha_1) \simeq \END_{R(\alpha_1)}(R(\alpha_1)) \xrightarrow{S_{[h,2]}} \END_{R(\alpha_{1^*})}(S_{[h,2]}R(\alpha_1)) \\ 
&\simeq \END_{R(\alpha_{1^*})}(R(\alpha_{1^*})) \simeq R(\alpha_{1^*}).  
\end{align*}
Note that the homogeneous components of degree $(\alpha_1, \alpha_1)$ of these graded algebras are one-dimensional. 
Hence, there exists $g_1 \in \mathbf{k}^{\times}$ such that $x_1 \in R(\alpha_1)$ is sent to $g_1 x_1 \in R(\alpha_{1^*})$ through the isomorphism above. 

Similarly, the isomorphism $S_{[h-1,1]}R(\alpha_2) \simeq R(\alpha_{2^*})$ of (\ref{eq:fixedisom}) induces an isomorphism $R(\alpha_2) \simeq R(\alpha_{2^*})$, 
and there exists $g_2 \in \mathbf{k}^{\times}$ such that $x_1 \in R(\alpha_2)$ is sent to $g_2x_1 \in R(\alpha_{2^*})$ under the isomorphism above. 

\begin{lemma} \label{lem:g1}
Assume $(\alpha_2,\alpha_2) \geq (\alpha_1,\alpha_1)$. 
\begin{enumerate}
\item If $\mathsf{A}_J$ is of type $A_2$, then $1^* = 2, 2^* = 1$ and $g_1 = -\frac{t_{2,1}}{t_{1,2}}, g_2 = -\frac{t_{1,2}}{t_{2,1}}$. 
\item If $\mathsf{A}_J$ is of type $B_2$, then $1^* = 1, 2^* = 2$ and $g_1 = -1, g_2 = 1$. 
\item If $\mathsf{A}_J$ is of type $A_1 \times A_1$ or $G_2$, then $1^* = 1, 2^* = 2$ and $g_1 = g_2 = 1$.  
\end{enumerate}
\end{lemma}

\begin{proof}
It will be proved in the subsequent section. 
\end{proof}

Let $\beta \in \mathsf{Q}_+$ and $M \in \gMod{{}_JR(\beta)}$. 

\begin{lemma} \label{lem:sdotd1}
Let $i \in \{1, 2\}$. 
Identifying the isomorphisms 
\[
S_{[h,1]}(MF_i) \simeq E_{i^*}S_{[h,1]}(M), \ S_{\overline{[h,1]}}(MF_i) \simeq E_{i^*}S_{\overline{[h,1]}}(M)
\]
of (\ref{eq:isomE1}), (\ref{eq:isomF2}), (\ref{eq:isomF1}) and (\ref{eq:isombar}), 
we have 
\[
S_{[h,1]}\left(M \xy
0;/r.12pc/:
(0,0)*{\sdotd{i}}; 
\endxy\right) = g_i \xy
0;/r.12pc/:
(0,0)*{\sdotu{i^*}}; 
\endxy S_{[h,1]}(M), \ S_{\overline{[h,1]}}\left(M \xy
0;/r.12pc/:
(0,0)*{\sdotd{i}}; 
\endxy \right) = g_i \xy
0;/r.12pc/:
(0,0)*{\sdotu{i^*}}; 
\endxy S_{\overline{[h,1]}}(M). 
\]
\end{lemma}

\begin{proof}
Recall that the isomorphism $S_{[h,1]}(MF_2) \simeq E_{2^*}S_{[h,1]}(M)$ is induced from 
\[
S_{[h-1,1]}(M \circ R(\alpha_2)) \simeq S_{[h-1,1]}(M) \circ R(\alpha_{2^*}), \ S_{2^*}(S_{[h-1,1]}(M)F_{2^*}) \simeq E_{2^*}S_{[h,1]}(M).
\] 
Since $S_{2^*}\left(S_{[h-1,1]}M \xy
0;/r.12pc/:
(0,0)*{\sdotd{2}}; 
\endxy\right) = \xy
0;/r.12pc/:
(0,0)*{\sdotu{2^*}}; 
\endxy S_{[h,1]}(M)$, we deduce $S_{[h,1]}\left(M \xy
0;/r.12pc/:
(0,0)*{\sdotd{2}}; 
\endxy \right) = g_2 \xy
0;/r.12pc/:
(0,0)*{\sdotu{2^*}}; 
\endxy S_{[h,1]}(M)$.

Similarly, we have $S_{[h,1]}\left(M \xy
0;/r.12pc/:
(0,0)*{\sdotu{1}}; 
\endxy\right) = g_1 \xy
0;/r.12pc/:
(0,0)*{\sdotd{1^*}}; 
\endxy S_{[h,1]}(M)$. 
Hence, we obtain 
\begin{align*}
S_{[h,1]}\left(M \xy
0;/r.12pc/:
(0,0)*{\sdotd{1}}; 
\endxy\right) &= S_{[h,1]}\left(M \xy 0;/r.12pc/:
(-4,3)*{\lcap{}};
(-8,-4)*{\slined{1}};
(4,-3)*{\lcup{}};
(8,4)*{\slined{}};
(0,0)*{\bullet};
\endxy\right) \\
&= g_1 \xy 0;/r.12pc/:
(-4,3)*{\lcap{}};
(-8,-4)*{\slined{1^*}};
(4,-3)*{\lcup{}};
(8,4)*{\slined{}};
(0,0)*{\bullet};
\endxy S_{[h,1]}(M) \quad \text{by (\ref{eq:adjunctioncorrespondence}) and the discussion above}\\
&= g_1 \xy
0;/r.12pc/:
(0,0)*{\sdotd{1^*}}; 
\endxy S_{[h,1]}(M). 
\end{align*}

By interchanging $1$ and $2$, we deduce the remaining equalities. 
\end{proof}

\begin{lemma} \label{lem:dcross11}
Let $i \in \{1,2\}$. 
Identifying the isomorphisms of (\ref{eq:isomE1}), (\ref{eq:isomF2}), (\ref{eq:isomF1}) and (\ref{eq:isombar}), we have 
\begin{enumerate}
\item $S_{[h,1]}\left(M \xy 0;/r.12pc/:
(0,0)*{\dcross{i}{i}};
\endxy \right) = g_i^{-1} \xy 0;/r.12pc/:
(0,0)*{\ucross{i^*}{i^*}};
\endxy S_{[h,1]} (M)$,
\item $S_{\overline{[h,1]}}\left(M \xy 0;/r.12pc/:
(0,0)*{\dcross{i}{i}};
\endxy \right) = g_i^{-1} \xy 0;/r.12pc/:
(0,0)*{\ucross{i^*}{i^*}};
\endxy S_{\overline{[h,1]}} (M)$, 
\item $S_{[h,1]}\left(M \xy 0;/r.12pc/:
(0,0)*{\lcross{i}{i}};
\endxy \right) = g_i^{-1} \xy 0;/r.12pc/:
(0,0)*{\lcross{i^*}{i^*}};
\endxy S_{[h,1]} (M)$,
\item $S_{\overline{[h,1]}}\left(M \xy 0;/r.12pc/:
(0,0)*{\lcross{i}{i}};
\endxy \right) = g_i^{-1} \xy 0;/r.12pc/:
(0,0)*{\lcross{i^*}{i^*}};
\endxy S_{\overline{[h,1]}} (M)$.
\end{enumerate}
\end{lemma}

\begin{proof} 
(1) 
We have an isomorphism 
\begin{align*}
S_{[h-1,1]}(R(\alpha_2) \circ R(\alpha_2)) &\xleftarrow{\theta} S_{[h-1,1]}(R(\alpha_2)) \circ S_{[h-1,1]}(R(\alpha_2)) \\
&\overset{(\ref{eq:fixedisom})}{\simeq} R(\alpha_{2^*}) \circ R(\alpha_{2^*}). 
\end{align*}
It induces an isomorphism of graded algebras
\begin{align} \label{eq:isomdcross11}
&R(2\alpha_2) \simeq \END_{R(2\alpha_2)} (R(\alpha_2) \circ R(\alpha_2)) \\ 
&\overset{S_{[h-1,1]}}{\simeq} \END_{R(2\alpha_{2^*})}(R(\alpha_{2^*}) \circ R(\alpha_{2^*})) \simeq R(2\alpha_{2^*}), \notag
\end{align}
since $R(\alpha_2) \circ R(\alpha_2) = R(2\alpha_2), R(\alpha_{2^*}) \circ R(\alpha_{2^*}) = R(2\alpha_{2^*})$. 
Under this isomorphism, $x_1 \in R(2\alpha_2)$ is sent to $g_2 x_1 \in R(2\alpha_{2^*})$, and $x_2$ is sent to $g_2x_2 \in R(2\alpha_{2^*})$. 

Put $d = (\alpha_2,\alpha_2) = (\alpha_{2^*}, \alpha_{2^*})$.
Note that both $R(2\alpha_2)_{-d}$ and $R(2\alpha_2)_{-d}$ are one-dimensional and generated by $\tau_1$ over $\mathbf{k}$. 
Hence, $\tau_1 \in R(2\alpha_2)$ is sent to $a\tau_1 \in R(2\alpha_{2^*})$ for some $a \in \mathbf{k}^{\times}$. 
Note that we have a relation $x_2\tau_1 - \tau_1 x_1 = 1$ in both $R(2\alpha_2)$ and $R(2\alpha_{2^*})$. 
Since the isomorphism (\ref{eq:isomdcross11}) is an algebra isomorphism, we must have $a = g_2^{-1}$. 

Since $S_{2^*}\left(S_{[h-1,1]}(M)\xy 0;/r.12pc/:
(0,0)*{\dcross{2^*}{2^*}};
\endxy\right) = \xy 0;/r.12pc/:
(0,0)*{\ucross{2^*}{2^*}};
\endxy S_{[h,1]} (M)$, we deduce that 
\[
S_{[h,1]}\left(M \xy 0;/r.12pc/:
(0,0)*{\dcross{2}{2}};
\endxy \right) = g_2^{-1} \xy 0;/r.12pc/:
(0,0)*{\ucross{2^*}{2^*}};
\endxy S_{[h,1]} (M). 
\]
Similarly, we have 
\[
S_{[h,1]}\left(M \xy 0;/r.12pc/:
(0,0)*{\ucross{1}{1}};
\endxy \right) = g_1^{-1} \xy 0;/r.12pc/:
(0,0)*{\dcross{1^*}{1^*}};
\endxy S_{[h,1]} (M). 
\]
By applying adjunction based on the formula (\ref{eq:adjunctioncorrespondence}), we obtain
\[
S_{[h,1]}\left(M \xy 0;/r.12pc/:
(0,0)*{\dcross{1}{1}};
\endxy \right) = g_1^{-1} \xy 0;/r.12pc/:
(0,0)*{\ucross{1^*}{1^*}};
\endxy S_{[h,1]} (M). 
\]

(2) follows from (1) by interchanging $1$ and $2$. 

(3) We have 
\begin{align*}
S_{[h,1]}\left(M \xy 0;/r.12pc/:
(0,0)*{\lcross{i}{i}};
\endxy \right) &= S_{[h,1]}\left(M \xy 0;/r.12pc/:
(0,0)*{\xybox{
(0,0)*{\dcross{}{}};
(8,7)*{\lcap{}};
(-8,-7)*{\lcup{}};
(-4,8)*{\slined{}};
(4,-8)*{\slined{i}};
(12,-12); (12,4) **\dir{-} ?(1)*\dir{>};
(12,-14)*{\scriptstyle i};
(-12,-4); (-12,12) **\dir{-} ?(1)*\dir{>}; 
}}\endxy \right) \\
&= g_i^{-1} \xy 0;/r.12pc/:
  (0,0)*{\xybox{
  (0,0)*{\ucross{}{}};
  (8,-7)*{\lcup{}};
  (-8,7)*{\lcap{}};
  (-4,-8)*{\slineu{i^*}};
  (4,8)*{\slineu{}};
  (12,12); (12,-4) **\dir{-} ?(1)*\dir{>};
  (-12,-14)*{\scriptstyle i^*};
  (-12,4); (-12,-12) **\dir{-} ?(1)*\dir{>}; 
  }}\endxy S_{[h,1]}(M) \quad \text{by (1) and (\ref{eq:adjunctioncorrespondence})} \\
&= g_i^{-1}\xy 0;/r.12pc/:
(0,0)*{\lcross{i^*}{i^*}};
\endxy S_{[h,1]}(M). 
\end{align*}

(4) is parallel to (3). 
\end{proof}
 
\begin{lemma}  \label{lem:unitcounit}
Let $i \in \{1,2\}$. 
Identifying the isomorphisms of (\ref{eq:isomE1}), (\ref{eq:isomF2}), (\ref{eq:isomF1}) and (\ref{eq:isombar}), we have 
\begin{enumerate}
\item $S_{[h,1]}(M \xy 0;/r.12pc/: 
(0,0)*{\rcap{i}};
\endxy) = c_{i,\beta}c_{i^*,-w_J\beta}^{-1}g_i^{-\langle h_i,\beta \rangle +1}\xy 0;/r.12pc/: 
(0,0)*{\rcap{i^*}};
\endxy S_{[h,1]}(M)$, 
\item $S_{[h,1]}(M \xy 0;/r.12pc/: 
(0,0)*{\rcup{i}};
\endxy) = c_{i,\beta}^{-1}c_{i^*,-w_J\beta}g_i^{\langle h_i,\beta\rangle + 1} \xy 0;/r.12pc/: 
(0,0)*{\rcup{i^*}};
\endxy S_{[h,1]}(M)$, 
\item $S_{\overline{[h,1]}}(M \xy 0;/r.12pc/: 
(0,0)*{\rcap{i}};
\endxy) = c_{i,\beta}c_{i^*,-w_J\beta}^{-1}g_i^{-\langle h_i,\beta \rangle +1}\xy 0;/r.12pc/: 
(0,0)*{\rcap{i^*}};
\endxy S_{\overline{[h,1]}}(M)$, 
\item $S_{\overline{[h,1]}}(M \xy 0;/r.12pc/: 
(0,0)*{\rcup{i}};
\endxy) = c_{i,\beta}^{-1}c_{i^*,-w_J\beta}g_i^{\langle h_i,\beta\rangle + 1}\xy 0;/r.12pc/: 
(0,0)*{\rcup{i^*}};
\endxy S_{\overline{[h,1]}}(M)$. 
\end{enumerate}
\end{lemma}

\begin{proof}
(1) (2) First, assume $\langle h_i,\beta \rangle \geq 0$. 
Then, the homomorphism 
\[
M \begin{bmatrix}
\xy 0;/r.12pc/:
(0,0)*{\lcross{i}{i}}; 
\endxy & \xy 0;/r.12pc/:
(0,0)*{\lcup{i}}; 
(2,-0.5)*{\bullet};
(14,-1)*{\scriptstyle \langle h_i, \beta \rangle -1};
\endxy & \cdots & \xy 0;/r.12pc/:
(0,0)*{\lcup{i}};
(2,-0.5)*{\bullet};
\endxy & \xy 0;/r.12pc/:
(0,0)*{\lcup{i}}; 
\endxy
\end{bmatrix}\colon MF_i E_i \oplus M^{\oplus \langle h_i,\beta \rangle} \to ME_i F_i 
\]
is an isomorphism by Theorem \ref{thm:Rouquierver} (2). 
Since $S_{[h,1]}\left(M \xy 0;/r.12pc/:
(0,0)*{\lcup{i}}; 
(2,-0.5)*{\bullet};
(6,-1)*{\scriptstyle p};
\endxy\right) = g_i^p \  \xy 0;/r.12pc/:
(0,0)*{\lcup{}}; 
(0,-4)*{\scriptstyle i^*}; 
(2,-0.5)*{\bullet};
(6,-1)*{\scriptstyle p}; 
\endxy \ S_{[h,1]}(M)$ ((\ref{eq:adjunctioncorrespondence}) and Lemma \ref{lem:sdotd1}) and $S_{[h,1]}\left(M \xy 0;/r.12pc/:
(0,0)*{\lcross{i}{i}}; 
\endxy \right) = g_i^{-1} \xy 0;/r.12pc/:
(0,0)*{\lcross{i^*}{i^*}}; 
\endxy S_{[h,1]}(M)$ (Lemma \ref{lem:dcross11}), (1) is reduced to proving 
\begin{align*}
&S_{[h,1]}\left(M \xy 0;/r.12pc/: (0,0)*{\xybox{
(0,3)*{\rcap{}}; 
(0,-4)*{\lcross{i}{i}};
}}; 
\endxy\right) = c_{i,\beta}c_{i^*,-w_J\beta}^{-1}g_i^{-\langle h_i,\beta \rangle} \xy 0;/r.12pc/: (0,0)*{\xybox{
(0,3)*{\rcap{}}; 
(0,-4)*{\lcross{i^*}{i^*}};
}};
\endxy S_{[h,1]}(M), \\
&S_{[h,1]}\left(M \xy 0;/r.12pc/:
(0,3)*{\rcap{}}; 
(0,-3)*{\lcup{i}}; 
(2,-3.5)*{\bullet};
(6,-4)*{\scriptstyle p};
(9,2)*{\scriptstyle \beta}; 
\endxy \right) = c_{i,\beta}c_{i^*,-w_J\beta}^{-1}g_i^{-\langle h_i,\beta \rangle + 1 + p} \xy 0;/r.12pc/:
(0,3)*{\rcap{}}; 
(0,-3)*{\lcup{}};
(0,-7)*{\scriptstyle i^*}; 
(2,-3.5)*{\bullet};
(6,-4)*{\scriptstyle p};
(12,2)*{\scriptstyle -w_J\beta}
\endxy \ S_{[h,1]}(M), 
\end{align*}
for $0 \leq p \leq \langle h_i,\beta \rangle -1$. 
Note that $\langle h_{i^*}, -w_J\beta \rangle = \langle h_i, \beta \rangle \geq 0$. 
By \cite[Lemma 3.1]{MR3461059}, the left hand side of the first equality is 
\[
-\delta_{\langle h_i,\beta \rangle, 0} c_{i,\beta} S_{[h,1]} \left(M \xy 0;/r.12pc/:
(0,0)*{\lcap{i}};
\endxy \right) \overset{(\ref{eq:adjunctioncorrespondence})}{=} -\delta_{\langle h_i,\beta \rangle, 0} c_{i,\beta} \xy 0;/r.12pc/:
(0,0)*{\lcap{i^*}};
\endxy S_{[h,1]} (M),
\]
and the right hand side is 
\[
c_{i,\beta}c_{i^*,-w_J\beta}^{-1}g_i^{-\langle h_i,\beta \rangle} \delta_{\langle h_i,\beta \rangle, 0}(-c_{i^*,-w_J\beta}) \xy 0;/r.12pc/:
(0,0)*{\lcap{i^*}};
\endxy S_{[h,1]}(M) = -\delta_{\langle h_i,\beta \rangle,0} c_{i,\beta} \xy 0;/r.12pc/:
(0,0)*{\lcap{i^*}};
\endxy S_{[h,1]} (M). 
\]
Hence, the first equality holds. 
For $0 \leq p \leq \langle h_i,\beta \rangle -1$, Definition \ref{def:catquantum} (8) shows
\begin{align*}
&S_{[h,1]}\left(M \xy 0;/r.12pc/:
(0,3)*{\rcap{}}; 
(0,-3)*{\lcup{i}}; 
(2,-3.5)*{\bullet};
(6,-4)*{\scriptstyle p};
\endxy \right) = \delta_{p,\langle h_i,\beta \rangle -1} c_{i,\beta} \id_{S_{[h,1]}(M)}, \\
&\xy 0;/r.12pc/:
(0,3)*{\rcap{}}; 
(0,-3)*{\lcup{}}; 
(0,-7)*{\scriptstyle i^*}; 
(2,-3.5)*{\bullet};
(6,-4)*{\scriptstyle p};
\endxy S_{[h,1]}(M) = \delta_{p,\langle h_i,\beta \rangle -1} c_{i^*,-w_J\beta} \id_{S_{[h,1]}(M)}. 
\end{align*}
If $p = \langle h_i,\beta \rangle -1$, then $g_i^{-\langle h_i,\beta \rangle + 1 +p} = 1$. 
Hence, the second equality holds, and (1) is proved in this case. 
By the uniqueness of adjoint functors, (2) also follows in this case. 

Next, assume $\langle h_i,\beta \rangle < 0$. 
Then, the homomorphism 
\[
M \begin{bmatrix}
\xy 0;/r.12pc/:
(0,0)*{\lcross{i}{i}}; 
\endxy & \xy 0;/r.12pc/:
(0,0)*{\lcap{i}}; 
\endxy & \xy 0;/r.12pc/:
(0,0)*{\lcap{i}}; 
(-2,.5)*{\bullet};
\endxy & \cdots & \xy 0;/r.12pc/:
(0,0)*{\lcap{i}}; 
(-2,.5)*{\bullet};
(-14,3)*{\scriptstyle -\langle h_j, \beta \rangle -1};
\endxy
\end{bmatrix}^\top \colon M F_i E_i \to M E_i F_i \oplus M^{\oplus -\langle h_i,\beta \rangle} 
\]
is an isomorphism by Theorem \ref{thm:Rouquierver} (2). 
Since $S_{[h,1]}\left(M \ \xy 0;/r.12pc/:
(0,0)*{\lcap{}}; 
(0.5,3)*{\scriptstyle i};
(-2,.5)*{\bullet};
(-4.5,3)*{\scriptstyle p};
\endxy\right) = g_i^p \ \xy 0;/r.12pc/:
(0,0)*{\lcap{}};
(.5,3)*{\scriptstyle i^*};
(-2,.5)*{\bullet};
(-4.5,3)*{\scriptstyle p};
\endxy S_{[h,1]}(M)$ ((\ref{eq:adjunctioncorrespondence}) and Lemma \ref{lem:sdotd1}) and $S_{[h,1]}\left(M \xy 0;/r.12pc/:
(0,0)*{\lcross{i}{i}}; 
\endxy\right) = g_i^{-1} \xy 0;/r.12pc/:
(0,0)*{\lcross{i^*}{i^*}}; 
\endxy S_{[h,1]}(M)$ (Lemma \ref{lem:dcross11}), (2) is reduced to proving
\begin{align*}
&S_{[h,1]}\left(M \xy 0;/r.12pc/:
(0,4)*{\lcross{}{}}; 
(0,-3)*{\rcup{i}};
\endxy \right) = c_{i,\beta}^{-1}c_{i^*,-w_J\beta} g_i^{\langle h_i,\beta \rangle} \xy 0;/r.12pc/:
(0,4)*{\lcross{}{}}; 
(0,-3)*{\rcup{i^*}};
\endxy S_{[h,1]}(M), \\
&S_{[h,1]}\left(M \ \xy 0;/r.12pc/:
(0,3)*{\lcap{i}}; 
(-2,3.5)*{\bullet};
(-6,6)*{\scriptstyle p};
(0,-3)*{\rcup{}};
\endxy \right) = c_{i,\beta}^{-1}c_{i^*, -w_J\beta} g_i^{\langle h_i,\beta \rangle + 1 + p}  \xy 0;/r.12pc/:
(0,3)*{\lcap{}};
(0.5,6)*{\scriptstyle i^*}; 
(-2,3.5)*{\bullet};
(-4.5,6)*{\scriptstyle p};
(0,-3)*{\rcup{}};
\endxy S_{[h,1]}(M) \ (0 \leq p \leq -\langle h_i,\beta \rangle -1). 
\end{align*}
Note that $\langle h_i,\beta \rangle = \langle h_i,-w_J\beta \rangle < 0$. 
By \cite[Lemma 3.1]{MR3461059}, both sides of the first equality are zero. 
For $0 \leq p \leq -\langle h_i,\beta \rangle -1$, Definition \ref{def:catquantum} (8) shows
\begin{align*}
&S_{[h,1]}\left(M \xy 0;/r.12pc/:
(0,3)*{\lcap{}};
(0.5,6)*{\scriptstyle i};
(-2,3.5)*{\bullet};
(-4.5,6)*{\scriptstyle p};
(0,-3)*{\rcup{}};
\endxy \right) = \delta_{p,-\langle h_i,\beta \rangle -1} c_{i,\beta}^{-1}\id_{S_{[h,1]}(M)}, \\
& \xy 0;/r.12pc/:
(0,3)*{\lcap{}};
(0.5,6)*{\scriptstyle i^*}; 
(-2,3.5)*{\bullet};
(-4.5,6)*{\scriptstyle p};
(0,-3)*{\rcup{}};
\endxy S_{[h,1]}(M) = \delta_{p,-\langle h_i,\beta \rangle -1} c_{i^*,-w_J\beta}^{-1} \id_{S_{[h,1]}(M)}. 
\end{align*}
Hence, the second equality holds, and (2) is proved in this case. 
By the uniqueness of adjoint, (1) also follows. 

\begin{comment}
By \cite[Proposition 3.3]{MR3461059}, we have 
\begin{align*}
\xy 0;/r.12pc/:
(0,3)*{\lcap{}}; 
(0,-4)*{\rcross{i}{i}};
(2,3.5)*{\bullet};
(12,4)*{\scriptstyle -\langle h_i,\beta \rangle};
(9,-4)*{\scriptstyle \beta};
\endxy = c_{i,\beta}^{-1} \xy 0;/r.12pc/: 
(0,0)*{\rcap{i}};
\endxy, \ \xy 0;/r.12pc/:
(0,3)*{\lcap{}}; 
(0,-4)*{\rcross{i^*}{i^*}};
(2,3.5)*{\bullet};
(12,4)*{\scriptstyle -\langle h_i,\beta \rangle};
(9,-4)*{\scriptstyle -w_J\beta};
\endxy = c_{i,-w_J\beta}^{-1} \xy 0;/r.12pc/: 
(0,0)*{\rcap{i}};
\endxy. 
\end{align*}
Hence, 
\begin{align*}
&S_{[h,1]}(M \xy 0;/r.12pc/: 
(0,0)*{\rcap{i}};
\endxy) = c_{i,\beta} S_{[h,1]}\left(M \xy 0;/r.12pc/:
(0,3)*{\lcap{}}; 
(0,-4)*{\rcross{i}{i}};
(2,3.5)*{\bullet};
(12,4)*{\scriptstyle -\langle h_i,\beta \rangle};
(9,-4)*{\scriptstyle \beta};
\endxy \right) = c_{i,\beta}g_i^{-\langle h_i,\beta \rangle -1} \xy 0;/r.12pc/:
(0,3)*{\lcap{}}; 
(0,-4)*{\rcross{i^*}{i^*}};
(2,3.5)*{\bullet};
(12,4)*{\scriptstyle -\langle h_i,\beta \rangle};
(9,-4)*{\scriptstyle -w_J\beta};
\endxy S_{[h,1]}(M) \\ 
&= c_{i,\beta}g_i^{-\langle h_i,\beta \rangle -1}c_{i^*,-w_J\beta}^{-1} \xy 0;/r.12pc/: 
(0,0)*{\rcap{i^*}};
\endxy S_{[h,1]}(M). 
\end{align*}
\end{comment}

(3) and (4) follow from (1) and (2) by interchanging 1 and 2. 
\end{proof}

\begin{lemma} \label{lem:dcross12}
Identifying the isomorphisms of (\ref{eq:isomE1}), (\ref{eq:isomF2}), (\ref{eq:isomF1}) and (\ref{eq:isombar}), we have 
\begin{enumerate}
\item $S_{[h,1]}\left(M \xy 0;/r.12pc/:
(0,0)*{\lcross{2}{1}};
\endxy \right) = t_{1,2}^{-1}t_{2^*,1^*}^{-1} \xy 0;/r.12pc/:
(0,0)*{\lcross{1^*}{2^*}};
\endxy S_{[h,1]}(M)$, 
\item $S_{[h,1]}\left(M \xy 0;/r.12pc/:
(0,0)*{\rcross{1}{2}};
\endxy \right) = t_{1,2}t_{2^*,1^*} \xy 0;/r.12pc/:
(0,0)*{\rcross{2^*}{1^*}};
\endxy S_{[h,1]}(M)$, 
\item  $S_{[h,1]}\left(M \xy 0;/r.12pc/:
(0,0)*{\dcross{1}{2}};
\endxy \right) = t_{1,2}^{-1}t_{2^*,1^*}^{-1} \xy 0;/r.12pc/:
(0,0)*{\ucross{2^*}{1^*}};
\endxy S_{[h,1]}(M)$,
\item $S_{[h,1]}\left(M \xy 0;/r.12pc/:
(0,0)*{\dcross{2}{1}};
\endxy \right) = t_{1,2}^2t_{1^*,2^*}^{-1}t_{2^*,1^*}g_1^{-a_{1,2}} \xy 0;/r.12pc/:
(0,0)*{\ucross{1^*}{2^*}};
\endxy S_{[h,1]}(M)$, 
\item $S_{\overline{[h,1]}}\left(M \xy 0;/r.12pc/:
(0,0)*{\dcross{2}{1}};
\endxy \right) = t_{2,1}^{-1}t_{1^*,2^*}^{-1} \xy 0;/r.12pc/:
(0,0)*{\ucross{1^*}{2^*}};
\endxy S_{\overline{[h,1]}}(M)$,
\item $S_{\overline{[h,1]}}\left(M \xy 0;/r.12pc/:
(0,0)*{\dcross{1}{2}};
\endxy \right) = t_{2,1}^2t_{2^*,1^*}^{-1}t_{1^*,2^*}g_2^{-a_{2,1}} \xy 0;/r.12pc/:
(0,0)*{\ucross{2^*}{1^*}};
\endxy S_{\overline{[h,1]}}(M)$. 
\end{enumerate}
\end{lemma}

\begin{proof}
(1) Note that the following diagram commutes: 
\begin{equation*}
\begin{tikzcd}
S_1(MF_2E_1) \arrow[r,"{\xy 0;/r.12pc/:
(0,0)*{\lcross{2}{1}};
\endxy}"] & S_1(ME_1F_2) \\
S_1(M\tilde{F}_2E_1) \arrow[u,twoheadrightarrow,"\text{can}"]\arrow[r,"{\xy 0;/r.12pc/:
(0,0)*{\lcross{}{1}};
(-4,-7)*{\scriptstyle \tilde{2}}
\endxy}"'] & S_1(ME_1 \tilde{F}_2). \arrow[u,twoheadrightarrow,"\text{can}"]
\end{tikzcd}
\end{equation*}

Identifying 
\[
S_1(ME_1\tilde{F}_2) \simeq F_1S_1(M) \circ S_1(R(\alpha_2)),\ S_1(M\tilde{F}_2E_1) \simeq F_1(S_1(M) \circ S_1(R(\alpha_2)))
\]
based on the monoidality of $S_1$, 
we have 
\begin{align*}
&S_1\left(M \xy 0;/r.12pc/:
(0,0)*{\lcross{}{1}};
(-4,-6.5)*{\scriptstyle \tilde{2}}
\endxy \right) \\ 
&= S_1\left(M \xy 0;/r.12pc/:
(0,0)*{\xybox{
(0,0)*{\dcross{}{}};
(8,7)*{\lcap{}};
(-8,-7)*{\lcup{}};
(-4,8)*{\slined{}};
(4,-8)*{\slined{}};
(4,-14.5)*{\scriptstyle \tilde{2}};
(12,-12); (12,4) **\dir{-} ?(1)*\dir{>};
(12,-14)*{\scriptstyle 1};
(-12,-4); (-12,12) **\dir{-} ?(1)*\dir{>}; 
}}\endxy \right) \\
&= [ F_1(S_1(M) \circ S_1(R(\alpha_2))) \xrightarrow{c_{1,s_1\beta} \xy 0;/r.12pc/: (0,0)*{\lcup{1}}; \endxy} F_1(E_1F_1S_1(M) \circ S_1(R(\alpha_2)))  \\
&\quad \xrightarrow{\sigma_{1,2}} F_1E_1(F_1S_1(M) \circ S_1(R(\alpha_2))) \xrightarrow{c_{1,\alpha_1+s_1\beta+s_1\alpha_2}^{-1}\xy 0;/r.12pc/: (0,0)*{\lcap{1}}; \endxy} F_1S_1(M) \circ S_1(R(\alpha_2))] \\
&\quad \text{by Theorem \ref{thm:anotheraction}}, 
\end{align*} 
which coincides with $c_{1,s_1\beta}c_{1,\alpha_1+s_1\beta+s_1\alpha_2}^{-1} = t_{1,2}^{-1}$-multiple of the canonical surjection 
\[
F_1(S_1(M) \circ S_1R(\alpha_2)) \to F_1S_1(M) \circ S_1R(\alpha_2)
\]
by Lemma \ref{lem:canonicalsurj}. 
Moreover, it is an isomorphism whose inverse is $S_1\left(M \xy 0;/r.12pc/:
(0,0)*{\rcross{1}{}};
(4,-6.5)*{\scriptstyle \tilde{2}}
\endxy\right)$. 
Hence, we have a commutative diagram 
\begin{equation*}
\begin{tikzcd}
R(\alpha_1) \circ S_1(M) \circ S_1(R(\alpha_2)) \arrow[d,twoheadrightarrow,"\text{can}"]\arrow[r,"t_{1,2}^{-1}"] & R(\alpha_1) \circ S_1(M) \circ S_1(R(\alpha_2)) \arrow[d,twoheadrightarrow,"\text{can}"] \\ 
F_1(S_1(M) \circ S_1(R(\alpha_2))) \arrow[r,"{S_1 \left(M \xy 0;/r.12pc/:
(0,0)*{\lcross{}{1}};
(-4,-6.5)*{\scriptstyle \tilde{2}}; 
\endxy \right)}"'] & F_1S_1(M) \circ S_1(R(\alpha_2)). 
\end{tikzcd}
\end{equation*}

Similarly, putting $M' = S_{[h,1]}(M)$ and 
identifying 
\begin{align*}
&S'_{2^*} (\tilde{F}_{1^*}E_{2^*}M') \simeq S'_{2^*}(R(\alpha_{1^*})) \circ S'_{2^*}(M')F_{2^*}, \\
&S'_{2^*}(E_{2^*}\tilde{F}_{1^*}M') \simeq (S'_{2^*}(R(\alpha_{1^*})) \circ S'_{2^*}(M'))F_{2^*}, 
\end{align*}
we have 
\begin{align*}
&S'_{2^*} \left( \xy 0;/r.12pc/:
(0,0)*{\rcross{2^*}{}};
(4,-6.5)*{\scriptstyle \tilde{1^*}}
\endxy M' \right) = S'_{2^*} \left( \xy 0;/r.12pc/:
(0,0)*{\xybox{
(0,0)*{\dcross{}{}};
(-8,7)*{\rcap{}};
(8,-7)*{\rcup{}};
(4,8)*{\slined{}};
(-4,-8)*{\slined{}};
(-12,-12); (-12,4) **\dir{-} ?(1)*\dir{>};
(-12,-14)*{\scriptstyle 2^*};
(-4,-15)*{\scriptstyle \tilde{1}^*}; 
(12,-4); (12,12) **\dir{-} ?(1)*\dir{>}; 
}}
\endxy M' \right), 
\end{align*}
which coincides with $c_{2^*,-s_{2^*}w_J\beta}^{-1}c_{2^*,-s_{2^*}(w_J\beta -\alpha_{2^*}+\alpha_{1^*})} = t_{2^*,1^*}^{-1}$-multiple of the canonical surjection 
\[
(S'_{2^*}(R(\alpha_{1^*})) \circ S'_{2^*}(M'))F_{2^*} \xrightarrow{\text{can}} S'_{2^*}(R(\alpha_{1^*})) \circ S'_{2^*}(M')F_{2^*}. 
\]
Moreover, it is an isomorphism whose inverse is $S'_{2^*}\left( \xy 0;/r.12pc/:
(0,0)*{\lcross{}{2^*}};
(-4,-6.5)*{\scriptstyle \tilde{1^*}}
\endxy M' \right)$. 
Hence, we have a commutative diagram
\[
\begin{tikzcd}
S'_{2^*}(R(\alpha_{1^*})) \circ S_{[h-1,1]}(M) \circ R(\alpha_{2^*}) \arrow[r,"t_{2^*,1^*}"]\arrow[d,twoheadrightarrow,"{\text{can}}"] & S'_{2^*}(R(\alpha_{1^*})) \circ S_{[h-1,1]}(M) \circ R(\alpha_{2^*}) \arrow[d,twoheadrightarrow,"{\text{can}}"] \\
S'_{2^*}(R(\alpha_{1^*})) \circ S'_{2^*}(M')F_{2^*} \arrow[r,"{S'_{2^*}\left( \xy 0;/r.12pc/:
(0,0)*{\lcross{}{2^*}};
(-4,-6.5)*{\scriptstyle \tilde{1^*}}
\endxy M' \right)}"'] & (S'_{2^*}(R(\alpha_{1^*})) \circ S'_{2^*}(M'))F_{2^*}. 
\end{tikzcd}
\]

Using the isomorphism $S_{[h-1,2]}(S_1(R(\alpha_2))) \simeq R(\alpha_{2^*})$ of (\ref{eq:fixedisom}), we have an isomorphism 
\[
S_{[h-1,2]}(R(\alpha_1) \circ S_1(M) \circ S_1(R(\alpha_2))) \xleftarrow{\theta} S_{[h-1,2]}(R(\alpha_1)) \circ S_{[h-1,1]}(M) \circ R(\alpha_{2^*}). 
\] 
Note that the isomorphism $S_{[h,2]}(R(\alpha_1)) \simeq R(\alpha_{1^*})$ of (\ref{eq:fixedisom}) induces an isomorphism $S_{[h-1,2]} (R(\alpha_1)) = S'_{2^*}S_{[h,2]}(R(\alpha_1)) \simeq S'_{2^*}(R(\alpha_{1^*}))$. 
Combining these together, we deduce (1). 

(2) The morphism is the inverse of that of (1). 

(3) We have 
\begin{align*}
&S_{[h,1]}\left(M \xy 0;/r.12pc/:
(0,0)*{\dcross{1}{2}};
\endxy \right) = S_{[h,1]} \left(M \xy 0;/r.12pc/:
  (0,0)*{\xybox{
  (0,0)*{\lcross{}{}};
  (8,-7)*{\lcup{}};
  (-8,7)*{\lcap{}};
  (-4,-8)*{\slined{2}};
  (4,8)*{\slined{}};
  (12,12); (12,-4) **\dir{-} ?(1)*\dir{>};
  (-12,-14)*{\scriptstyle 1};
  (-12,4); (-12,-12) **\dir{-} ?(1)*\dir{>}; 
  }}\endxy \right) \\
&= t_{1,2}^{-1}t_{2^*,1^*}^{-1} \xy 0;/r.12pc/:
(0,0)*{\xybox{
(0,0)*{\lcross{}{}};
(8,7)*{\lcap{}};
(-8,-7)*{\lcup{}};
(-4,8)*{\slineu{}};
(4,-8)*{\slineu{2^*}};
(12,-12); (12,4) **\dir{-} ?(1)*\dir{>};
(12,-14)*{\scriptstyle 1^*};
(-12,-4); (-12,12) **\dir{-} ?(1)*\dir{>}; 
}}\endxy S_{[h,1]}(M) \quad \text{by (\ref{eq:adjunctioncorrespondence}) and (1)} \\
&= t_{1,2}^{-1}t_{2^*,1^*}^{-1} \xy 0;/r.12pc/:
(0,0)*{\ucross{2^*}{1^*}};
\endxy S_{[h,1]}(M). 
\end{align*}

(4) We have 
\begin{align*}
&S_{[h,1]}\left(M \xy 0;/r.12pc/:
(0,0)*{\dcross{2}{1}};
\endxy \right) = S_{[h,1]}\left(M \xy 0;/r.12pc/:
  (0,0)*{\xybox{
  (0,0)*{\rcross{}{}};
  (-8,-7)*{\rcup{}};
  (8,7)*{\rcap{}};
  (4,-8)*{\slined{2}};
  (-4,8)*{\slined{}};
  (-12,12); (-12,-4) **\dir{-} ?(1)*\dir{>};
  (12,-14)*{\scriptstyle 1};
  (12,4); (12,-12) **\dir{-} ?(1)*\dir{>}; 
  }}\endxy \right) \\
&= t_{1,2}t_{2^*,1^*} c_{1,\beta}^{-1}c_{1^*,-w_J\beta}g_1^{\langle h_1,\beta \rangle + 1} c_{1,\beta+\alpha_1+\alpha_2} c_{1^*,-w_J(\beta + \alpha_1 + \alpha_2)}^{-1}g_1^{-\langle h_1, \beta + \alpha_1+\alpha_2 \rangle +1} \times \\
&\quad \xy 0;/r.12pc/:
(0,0)*{\xybox{
(0,0)*{\rcross{}{}};
(-8,7)*{\rcap{}};
(8,-7)*{\rcup{}};
(4,8)*{\slineu{}};
(-4,-8)*{\slineu{2^*}};
(-12,-12); (-12,4) **\dir{-} ?(1)*\dir{>};
(-12,-14)*{\scriptstyle 1^*};
(12,-4); (12,12) **\dir{-} ?(1)*\dir{>}; 
}}\endxy S_{[h,1]}(M) \quad \text{by Lemma \ref{lem:unitcounit} and (2)} \\
&= t_{1,2}^2t_{1^*,2^*}^{-1}t_{2^*,1^*}g_1^{-a_{1,2}} \xy 0;/r.12pc/:
(0,0)*{\ucross{1^*}{2^*}};
\endxy S_{[h,1]}(M). 
\end{align*}

(5) follows from (3) by interchanging $1$ and $2$. 

(6) follows from (4) by interchanging $2$ and $1$. 
\end{proof}

\begin{lemma}
Put $b_{1,2} = t_{1,2}^{-1}t_{2,1}^{-2}t_{1^*,2^*}^{-1}g_2^{a_{2,1}}$ and $b_{2,1} = t_{2,1}t_{1,2}^2t_{2^*,1^*}g_1^{-a_{1,2}}$ 
Then, we have $b_{1,2}b_{2,1} = 1$ and 
\[
S_{[h,1]}\left( M \xy 
0;/r.12pc/: 
(0,0)*{\dcross{1}{2}}; 
\endxy\right) = b_{1,2} S_{\overline{[h,1]}}\left( M \xy 
0;/r.12pc/: 
(0,0)*{\dcross{1}{2}}; 
\endxy\right), S_{[h,1]}\left( M \xy 
0;/r.12pc/: 
(0,0)*{\dcross{2}{1}}; 
\endxy\right) = b_{2,1} S_{\overline{[h,1]}}\left( M \xy 
0;/r.12pc/: 
(0,0)*{\dcross{2}{1}}; 
\endxy\right), 
\]
under the identification (\ref{eq:natisom}). 
\end{lemma}

\begin{proof}
We may assume $(\alpha_2,\alpha_2) \geq (\alpha_1,\alpha_1)$. 
By definition, we have 
\[
b_{1,2}b_{2,1} = t_{1,2}t_{2,1}^{-1}t_{1^*,2^*}^{-1}t_{2^*,1^*}g_2^{a_{2,1}}g_1^{-a_{1,2}}.
\]
We compute it using Lemma \ref{lem:g1}. 
If $\mathsf{A}_J$ is of type $A_2$, we have 
\[
b_{1,2}b_{2,1} = t_{1,2}^2t_{2,1}^{-2} \left(-\frac{t_{1,2}}{t_{2,1}}\right)^{-1} \left( -\frac{t_{2,1}}{t_{1,2}}\right) =1. 
\]
If $\mathsf{A}_J$ is of type $B_2$, we have 
\[
b_{1,2}b_{2,1} = t_{1,2}t_{2,1}^{-1}t_{1,2}^{-1}t_{2,1} 1^{-1} (-1)^2 = 1. 
\]
Otherwise, we have 
\[
b_{1,2}b_{2,1} = t_{1,2}t_{2,1}^{-1}t_{1,2}^{-1}t_{2,1} = 1. 
\]
The first assertion is proved. 

The remaining assertions follow from Lemma \ref{lem:dcross12}. 
\end{proof}

Let $i \in I \setminus \{1,2\}$.
Put $n = \height (w_J\alpha_i)$. 
Since $\Delta(w_J,i) \in \gMod{R_J(w_J\alpha_i)}$ and $w_J\alpha_i \in \alpha_i + \mathbb{Z}\alpha_1 + \mathbb{Z}\alpha_2$, 
the idempotent $e(w_J\alpha_i-\alpha_i,i)$ acts on $\Delta(w_J,i)$ as the identity and $\tau_{n-1}\Delta(w_J,i) = 0$. 
Hence, th action of $x_n$ on $\Delta(w_J,i)$ yields an $R(w_J\alpha_i)$-module endomorphism, which is denoted by $z_i$. 

\begin{lemma} \label{lem:sdotdi}
Let $i \in I \setminus \{1,2\}$.
Identifying the isomorphism (\ref{eq:isomFi}), the endomorphism $S_{[h,1]}\left(M \xy 0;/r.12pc/:
(0,0)*{\sdotd{i}};
\endxy \right)$ coincides with the endomorphism 
\[
S_{[h,1]}(M) \circ \Delta(w_J,i) \xrightarrow{z_i} S_{[h,1]}(M) \circ \Delta(w_J, i). 
\]
Similarly, $S_{\overline{[h,1]}}\left(M \xy 0;/r.12pc/:
(0,0)*{\sdotd{i}};
\endxy \right)$ coincides with the endomorphism 
\[
S_{\overline{[h,1]}}(M) \circ \Delta(w_J,i) \xrightarrow{z_i} S_{\overline{[h,1]}}(M) \circ \Delta(w_J,i). 
\]
In particular, we have $S_{[h,1]}\left(M \xy 0;/r.12pc/:
(0,0)*{\sdotd{i}};
\endxy \right) = S_{\overline{[h,1]}}\left(M \xy 0;/r.12pc/:
(0,0)*{\sdotd{i}};
\endxy \right)$. 
\end{lemma}

\begin{proof}
We may assume $M = \mathbf{1}$. 
For $1 \leq p \leq h$, we inductively prove that the endomorphisms of $S_{[p,1]}(R(\alpha_i))$ induced by the multiplication by $x_1$ on $R(\alpha_i)$ coincides with the multiplication by $x_{\height (s_{i_p} \cdots s_{i_1}\alpha_i)}$.
Put $w_p = s_{i_p} \cdots s_{i_1}, n_p = \height (w_p\alpha_i)$. 
Since $\chi(L(w_{p-1},i))$ is the unipotent quantum minor 
\[
D(w_{p-1}s_i\Lambda_i,w_{p-1}\Lambda_i) = D(w_{p-1}s_i\Lambda_i,\Lambda_i) \quad \text{(Remark \ref{rem:determinantialmodule})}
\]
and $r_1(D(w_{p-1} \Lambda_i,\Lambda_i)) = r_2(D(w_{p-1}\Lambda_i,\Lambda_i)) = 0$ by \cite[Lemma 9.1.5]{MR3758148}, 
the idempotent $e(w_{p-1}\alpha_i-\alpha_i,i)$ acts on $\Delta(w_{p-1},i)$ as the identity. 
Hence, $\Delta(w_{p-1},i)$ is a quotient of $X F_i$ for some $X \in \gMod{{}_{i_p}R(w_{p-1}\alpha_i - \alpha_i)}$, 
and the following diagram commutes by the induction hypothesis: 
\begin{equation*}
\begin{tikzcd}
XF_i \arrow[d,twoheadrightarrow]\arrow[r,"{\xy 0;/r.12pc/: (0,0)*{\sdotd{i}}; \endxy}"] & XF_i \arrow[d,twoheadrightarrow] \\
\Delta(w_{p-1},i) \arrow[r,"{S_{[p,1]}\left( \mathbf{1}\xy 0;/r.12pc/: (0,0)*{\sdotd{i}}; \endxy \right)}"'] & \Delta(w_{p-1},i). 
\end{tikzcd}
\end{equation*}
Applying $S_{i_p}$, we obtain a commutative diagram 
\begin{equation*}
\begin{tikzcd}
S_{i_p}(X)\circ S_{i_p}(R(\alpha_i)) \arrow[d,twoheadrightarrow]\arrow[r,"{S_{i_p}\left( \mathbf{1} \xy 0;/r.12pc/: (0,0)*{\sdotd{i}}; \endxy\right)}"] & S_{i_p}(X) \circ S_{i_p}(R(\alpha_i)) \arrow[d,twoheadrightarrow] \\
\Delta(w_p,i) \arrow[r,"{S_{[p,1]}\left(\mathbf{1} \xy 0;/r.12pc/: (0,0)*{\sdotd{i}}; \endxy\right)}"'] & \Delta(w_p,i). 
\end{tikzcd}
\end{equation*}
By Lemma \ref{lem:yaction}, the endomorphism $S_{i_p}\left(\xy 0;/r.12pc/: (0,0)*{\sdotd{i}}; \endxy\right)$ of $S_{i_p}(R(\alpha_i))$ is the left action of $x_{1-a_{i_p,i}}$. 
It implies that the endomorphism $S_{[p,1]}\left(\xy 0;/r.12pc/: (0,0)*{\sdotd{i}}; \endxy\right)$ of $\Delta(w_p, i)$ is the left action of $x_{n_p}$. 
In fact, the discussion above shows that these two endomorphisms coincide on the image of 
\[
S_{i_p}(X) \boxtimes S_{i_p}(R(\alpha_i)) \hookrightarrow S_{i_p}(X)\circ S_{i_p}(R(\alpha_i)) \twoheadrightarrow \Delta(w_p,i).
\] 
Since $\Delta(w_p,i)$ is generated by this image as an $R(w_p\alpha_i)$-module, these two endomorphisms coincide on the whole $\Delta(w_p,i)$. 

The proof for $S_{\overline{[h,1]}}$ is parallel.
\end{proof}

\begin{lemma}
Let $i \in I \setminus \{1,2\}$. 
Identifying the isomorphisms of (\ref{eq:isomE1}), (\ref{eq:isomF2}), (\ref{eq:isomF1}), (\ref{eq:isombar}) and (\ref{eq:isomFi}), we have 
\begin{enumerate}
\item $S_{[h,1]}\left(M \xy 0;/r.12pc/:
(0,0)*{\lcross{i}{1}};
\endxy \right)$ is the $t_{1,i}^{-1}$-multiple of the canonical isomorphism 
\[
F_{1^*}(S_{[h,1]}(M) \circ \Delta(w_J,i)) \to F_{1^*}S_{[h,1]}(M)\circ \Delta(w_J,i).
\]  
In addition, $S_{[h,1]} \left(M \xy 0;/r.12pc/:
(0,0)*{\rcross{1}{i}};
\endxy \right)$ is the $t_{1,i}$-multiple of the inverse of this canonical isomorphism. 
\item $S_{[h,1]}\left(M \xy 0;/r.12pc/:
(0,0)*{\dcross{1}{i}};
\endxy \right)$ is the $t_{1,i}^{-1}$-multiple of the canonical injective homomorphism 
\[
E_{1^*}S_{[h,1]}(M) \circ \Delta(w_J,i) \to E_{1^*}(S_{[h,1]}(M) \circ \Delta(w_J,i)). 
\] 
\item $S_{\overline{[h,1]}}\left(M \xy 0;/r.12pc/:
(0,0)*{\dcross{2}{i}};
\endxy \right)$ is the $t_{2,i}^{-1}$-multiple of the canonical injective homomorphism 
\[
E_{2^*}S_{\overline{[h,1]}}(M) \circ \Delta(w_J,i) \to E_{2^*}(S_{\overline{[h,1]}}(M) \circ \Delta(w_J,i)). 
\]
\item There exists $b_{2,i} \in \mathbf{k}^{\times}$ independent of $M$ such that 
\[
S_{[h,1]}\left(M \xy 0;/r.12pc/:
(0,0)*{\dcross{2}{i}};
\endxy \right) = b_{2,i}S_{\overline{[h,1]}}\left(M \xy 0;/r.12pc/:
(0,0)*{\dcross{2}{i}};
\endxy \right).
\] 
Moreover, we have 
\[
S_{[h,1]}\left(M \xy 0;/r.12pc/:
(0,0)*{\dcross{i}{2}};
\endxy \right) = b_{2,i}^{-1}S_{\overline{[h,1]}}\left(M \xy 0;/r.12pc/:
(0,0)*{\dcross{i}{2}};
\endxy \right).
\] 
\item There exists $b_{1,i} \in \mathbf{k}^{\times}$ independent of $M$ such that 
\[
S_{[h,1]}\left(M \xy 0;/r.12pc/:
(0,0)*{\dcross{1}{i}};
\endxy \right) = b_{1,i} S_{\overline{[h,1]}}\left(M \xy 0;/r.12pc/:
(0,0)*{\dcross{1}{i}};
\endxy \right).
\] 
Moreover, we have 
\[
S_{[h,1]}\left(M \xy 0;/r.12pc/:
(0,0)*{\dcross{i}{1}};
\endxy \right) = b_{1,i}^{-1}S_{\overline{[h,1]}}\left(M \xy 0;/r.12pc/:
(0,0)*{\dcross{i}{1}};
\endxy \right).
\] 
\end{enumerate}
\end{lemma}

\begin{proof}
(1) The former assertion is proved by the same argument as that of Lemma \ref{lem:dcross12} (1). 
The latter one follows by taking the inverse of the morphism. 

(2) By (\ref{eq:adjunctioncorrespondence}) and (1), we have 
\begin{align*}
&S_{[h,1]}\left(M \xy 0;/r.12pc/:
(0,0)*{\dcross{1}{i}};
\endxy \right) = S_{[h,1]} \left(M \xy 0;/r.12pc/:
  (0,0)*{\xybox{
  (0,0)*{\lcross{}{}};
  (8,-7)*{\lcup{}};
  (-8,7)*{\lcap{}};
  (-4,-8)*{\slined{i}};
  (4,8)*{\slined{}};
  (12,12); (12,-4) **\dir{-} ?(1)*\dir{>};
  (-12,-14)*{\scriptstyle 1};
  (-12,4); (-12,-12) **\dir{-} ?(1)*\dir{>}; 
  }}\endxy \right) \\
&= [E_{1^*}S_{[h,1]}(M) \circ \Delta(w_J,i) \xrightarrow{\xy 0;/r.12pc/:
(0,0)*{\lcup{i^*}};
\endxy} E_{1^*}F_{1^*}(E_{1^*}S_{[h,1]}(M) \circ \Delta(w_J,i)) \\
& \quad \xrightarrow{t_{1,i}^{-1} \text{can}} E_{1^*}(F_{1^*}E_{1^*}S_{[h,1]}(M) \circ \Delta(w_J,i)) \xrightarrow{\xy 0;/r.12pc/:
(0,0)*{\lcap{i^*}};
\endxy} E_{1^*}(S_{[h,1]}(M) \circ \Delta(w_J,i))]. 
\end{align*}
Let $u \in S_{[h,1]}(M), v \in \Delta(w_J,i)$. 
Under $S_{[h,1]}\left(M \xy 0;/r.12pc/:
(0,0)*{\dcross{1}{i}};
\endxy \right)$, the element $E_{1^*}u \boxtimes v \in E_{1^*}S_{[h,1]}(M ) \circ \Delta(w_J,i)$ is sent to 
\begin{align*}
&E_{1^*}u \boxtimes v \mapsto E_{1^*}(e(1^*) \boxtimes (E_{1^*}u \boxtimes v)) \\
&\mapsto t_{1,i}^{-1} E_{1^*}((e(1^*) \boxtimes E_{1^*}u) \boxtimes v) \mapsto t_{1,i}^{-1} E_{1^*}(u \boxtimes v). 
\end{align*}
Hence, the assertion follows. 

(3) It follows from (2) by interchanging $1$ and $2$. 

(4) Since $\mathbf{1} \xy 0;/r.12pc/: (0,0)*{\dcross{\tilde{2}}{i}}; \endxy \colon R(\alpha_2) \circ R(\alpha_i) \to R(\alpha_i) \circ R(\alpha_2)$ is of degree $-(\alpha_2,\alpha_i)$, 
the homomorphism $S_{[h-1,1]}\left(\mathbf{1} \xy 0;/r.12pc/:
(0,0)*{\dcross{\tilde{2}}{i}};
\endxy\right)$ is of degree $-(\alpha_2,\alpha_i)$. 
We have isomorphisms
\begin{align*}
S_{[h-1,1]}(R(\alpha_2) \circ R(\alpha_i)) &\xleftarrow[]{\theta} S_{[h-1,1]}(R(\alpha_2)) \circ S_{[h-1,1]}(R(\alpha_i)) \\
&\overset{(\ref{eq:fixedisom})}{\simeq} R(\alpha_{2^*})\circ \Delta(s_{2^*}w_J,i), \\
S_{[h-1,1]}(R(\alpha_i) \circ R(\alpha_2)) &\xleftarrow[]{\theta} S_{[h-1,1]}(R(\alpha_i)) \circ S_{[h-1,1]}(R(\alpha_2)) \\
&\overset{(\ref{eq:fixedisom})}{\simeq} \Delta(s_{2^*}w_J,i) \circ R(\alpha_{2^*}). 
\end{align*}
We identify them below. 
Note that the space
\begin{align*}
&\HOM_{R(s_{2^*}w_J(\alpha_2+\alpha_i))} (R(\alpha_{2^*}) \circ \Delta(s_{2^*}w_J,i), \Delta(s_{2^*}w_J,i) \circ R(\alpha_{2^*}))_{-(\alpha_2,\alpha_i)} \\ 
&\simeq \HOM_{R(\alpha_2+\alpha_i)}(R(\alpha_2) \circ R(\alpha_i),R(\alpha_i) \circ R(\alpha_2))_{-(\alpha_2,\alpha_i)}
\end{align*}
is one-dimensional since $2 \neq i$. 
Since $\Delta(s_{2^*}w_J,i) \simeq S_{[h-1,1]}(R(\alpha_i))$ belongs to $\gMod{{}_{2^*}R}$, 
we have an injective homomorphism of degree $-(\alpha_{2^*}, s_{2^*}w_J\alpha_i) = -(\alpha_2,\alpha_i)$
\[
\mathsf{R} \colon R(\alpha_{2^*}) \circ \Delta(s_{2^*}w_J,i) \to \Delta(s_{2^*}w_J,i) \circ R(\alpha_{2^*}), 
\]
by Proposition \ref{prop:cyclotomicR}.
Therefore, there exists $b_{2,i} \in \mathbf{k}^{\times}$ such that $S_{[h-1,1]}\left(\mathbf{1} \xy 0;/r.12pc/:
(0,0)*{\dcross{}{i}};
(-4,-7)*{\scriptstyle \tilde{2}};
\endxy\right) = b_{2,i} \mathsf{R}$. 
It follows that $S_{[h-1,1]}\left(M \xy 0;/r.12pc/:
(0,0)*{\dcross{2}{i}};
\endxy \right)$ is the $b_{2,i}$-multiple of the canonical injective homomorphism
\begin{equation}\label{eq:canonicalinj}
S_{[h-1,1]}(M) F_{2^*} \circ \Delta(s_{2^*}w_J,i) \to (S_{[h-1,1]}(M) \circ \Delta(s_{2^*}w_J,i))F_{2^*}. 
\end{equation}
Identifying
\begin{align*}
S_{2^*}(S_{[h-1,1]}(M)F_{2^*} \circ \Delta(s_{2^*}w_J,i)) &\xleftarrow{\theta} S_{2^*}(S_{[h-1,1]}(M)F_{2^*}) \circ S_{2^*}(Delta(s_{2^*}w_J,i)) \\
&\simeq E_{2^*}S_{[h,1]}(M) \circ \Delta(w_J,i) \\
%&\overset{(\ref{eq:isomF2}), (\ref{eq:isomFi})}{\simeq} S_{[h,1]}(MF_2F_i), \\
S_{2^*}((S_{[h-1,1]}(M) \circ \Delta(s_{2^*}w_J,i))F_{2^*}) &\simeq E_{2^*}S_{2^*}(S_{[h-1,1]}(M) \circ \Delta(s_{2^*}w_J,i)) \\
&\xleftarrow{\theta} E_{2^*} (S_{[h,1]}(M) \circ \Delta(w_J,i)) \\
%&\overset{(\ref{eq:isomF2}), (\ref{eq:isomFi})}{\simeq} S_{[h,1]}(MF_iF_2), 
\end{align*}
the homomorphism obtained by applying $S_{2^*}$ to (\ref{eq:canonicalinj}) coincides with the canonical injective homomorphism: it follows from Proposition \ref{prop:monoidality} (3). 
Hence, the former assertion follows. 

To prove the latter assertion, note that we have 
\begin{align*}
&S_{[h,1]}\left(M \xy 0;/r.12pc/: (0,0)*{\xybox{
(0,0)*{\dcross{}{}};
(0,-8)*{\dcross{i}{2}};
}};
\endxy \right) = S_{[h,1]}\left( M Q_{i,2}\left(\xy 0;/r.12pc/:
(0,0)*{\sdotd{i}};
\endxy, \xy 0;/r.12pc/:
(0,0)*{\sdotd{2}};
\endxy\right) \right) \\
&\overset{(\ref{eq:natisom})}{=} S_{\overline{[h,1]}}\left(M Q_{i,2}\left(\xy 0;/r.12pc/:
(0,0)*{\sdotd{i}};
\endxy, \xy 0;/r.12pc/:
(0,0)*{\sdotd{2}};
\endxy\right) \right) \quad \text{by Lemma \ref{lem:sdotd1} and Lemma \ref{lem:sdotdi}} \\ 
&= S_{\overline{[h,1]}}\left(M \xy 0;/r.12pc/: (0,0)*{\xybox{
(0,0)*{\dcross{}{}};
(0,-8)*{\dcross{i}{2}};
}};
\endxy \right)
\end{align*}
Since $S_{[h,1]}\left(M \xy 0;/r.12pc/:
(0,0)*{\dcross{2}{i}};
\endxy \right) = b_{2,i}S_{\overline{[h,1]}}\left(M \xy 0;/r.12pc/:
(0,0)*{\dcross{2}{i}};
\endxy \right)$ and they are injective by the former assertion, the latter assertion follows. 

(5) follows from (4) by interchanging $1$ and $2$. 

\begin{comment}
(3) We have 
\begin{align*}
&S_{[h,1]}\left(M \xy 0;/r.12pc/:
(0,0)*{\dcross{i}{1}};
\endxy \right) = S_{[h,1]}\left(M \xy 0;/r.12pc/:
  (0,0)*{\xybox{
  (0,0)*{\rcross{}{}};
  (-8,-7)*{\rcup{}};
  (8,7)*{\rcap{}};
  (4,-8)*{\slined{i}};
  (-4,8)*{\slined{}};
  (-12,12); (-12,-4) **\dir{-} ?(1)*\dir{>};
  (12,-14)*{\scriptstyle 1};
  (12,4); (12,-12) **\dir{-} ?(1)*\dir{>}; 
  }}\endxy \right) \\
&= [E_{1^*}(S_{[h,1]}(M) \circ \Delta(w_J\alpha_i)) \xrightarrow{} ]

c_{1,\beta}^{-1}c_{1^*,-w_J\beta}g_1^{\langle h_1,\beta \rangle + 1} c_{1,\beta+\alpha_1+\alpha_i} c_{1^*,-w_J(\beta + \alpha_1 + \alpha_i)}^{-1}g_1^{-\langle h_1, \beta + \alpha_1+\alpha_i \rangle +1} 
\end{align*}
\end{comment}
\end{proof}

\begin{lemma}\label{lem:dcrossii2}
Let $i \in I \setminus \{1,2\}$. 
Identifying the isomorphisms (\ref{eq:isomFi}), we have 
\[
S_{[h,1]}\left(M \xy 0;/r.12pc/:
(0,0)*{\dcross{i}{i}};
\endxy\right) = S_{\overline{[h,1]}}\left(M \xy 0;/r.12pc/:
(0,0)*{\dcross{i}{i}};
\endxy \right). 
\]
\end{lemma}

\begin{proof}
We may assume $M = \mathbf{k}$. 
Note that the isomorphism (\ref{eq:fixedisom3}) induces an algebra isomorphism
\[
\END_{R(2w_J\alpha_i)}(\Delta(w_J,i) \circ \Delta(w_J,i)) \simeq R(2\alpha_i). 
\]
In $R(2\alpha_i)$, the element $\tau_1$ is characterized as the homogeneous element of degree $-(\alpha_i,\alpha_i)$ satisfying 
\[
\tau_1 x_2 - x_1 \tau_1 = 1. 
\]
Hence, the assertion follows from Lemma \ref{lem:sdotdi}. 
\end{proof}

\begin{lemma} \label{lem:dcrossij2}
Let $i, j \in I \setminus \{1,2\}$, and assume $i \neq j$. 
Identifying the isomorphisms (\ref{eq:isomFi}), there exists $b_{i,j} \in \mathbf{k}^{\times}$ such that 
\[
S_{[h,1]}\left(M \xy 0;/r.12pc/:
(0,0)*{\dcross{i}{j}};
\endxy\right) = b_{i,j} S_{\overline{[h,1]}}\left(M \xy 0;/r.12pc/:
(0,0)*{\dcross{i}{j}};
\endxy \right) \ (M \in \gMod{{}_JR}). 
\]
Moreover, they satisfy $b_{i,j} b_{j,i} = 1$.
\end{lemma}

\begin{proof}
We may assume $M = \mathbf{1}$. 
Since $\xy 0;/r.12pc/: (0,0)*{\dcross{i}{j}}; \endxy \colon R(\alpha_i) \circ R(\alpha_j) \to R(\alpha_j) \circ R(\alpha_i)$ is of degree $-(\alpha_i,\alpha_j)$, 
both $S_{[h,1]}\left(\mathbf{1} \xy 0;/r.12pc/:
(0,0)*{\dcross{i}{j}};
\endxy\right)$ and $S_{\overline{[h,1]}}\left(\mathbf{1} \xy 0;/r.12pc/:
(0,0)*{\dcross{i}{j}};
\endxy \right)$ are of degree $-(\alpha_i,\alpha_j)$. 
Note that the space
\begin{align*}
&\HOM_{R(w_J(\alpha_i+\alpha_j))}(\Delta(w_J,i)\circ \Delta(w_J,j), \Delta(w_J,j) \circ \Delta(w_J,i))_{-(\alpha_i,\alpha_j)} \\ 
&\simeq \HOM_{R(\alpha_i+\alpha_j)}(R(\alpha_i) \circ R(\alpha_j),R(\alpha_j) \circ R(\alpha_i))_{-(\alpha_i,\alpha_j)}
\end{align*}
is one-dimensional since $i \neq j$. 
Hence, there exists $b_{i,j} \in \mathbf{k}^{\times}$ such that 
\[
S_{[h,1]}\left(\mathbf{1} \xy 0;/r.12pc/:
(0,0)*{\dcross{i}{j}};
\endxy\right) = b_{i,j} S_{\overline{[h,1]}}\left(\mathbf{1} \xy 0;/r.12pc/:
(0,0)*{\dcross{i}{j}};
\endxy \right)
\]
The former assertion follows. 

Regarding the latter assertion, note that we have 
\[
S_{[h,1]}\left(\mathbf{1} \xy 0;/r.12pc/: (0,0)*{\xybox{
(0,0)*{\dcross{i}{j}};
(0,8)*{\dcross{}{}};
}};
\endxy\right) = b_{i,j}b_{j,i} S_{\overline{[h,1]}}\left(\mathbf{1} \xy 0;/r.12pc/: (0,0)*{\xybox{
(0,0)*{\dcross{i}{j}};
(0,8)*{\dcross{}{}};
}};
\endxy \right)
\]
On the other hand, since 
\[ 
\xy 0;/r.12pc/: (0,0)*{\xybox{
(0,0)*{\dcross{i}{j}};
(0,8)*{\dcross{}{}};
}};
\endxy = Q_{i,j}\left(\xy 0;/r.12pc/:
(0,0)*{\sdotd{i}};
\endxy,  \xy 0;/r.12pc/:
(0,0)*{\sdotd{j}};
\endxy\right), 
\]
Lemma \ref{lem:sdotdi} shows
\[
S_{[h,1]}\left(\mathbf{1} \xy 0;/r.12pc/: (0,0)*{\xybox{
(0,0)*{\dcross{i}{j}};
(0,8)*{\dcross{}{}};
}};
\endxy\right) = S_{\overline{[h,1]}}\left(\mathbf{1} \xy 0;/r.12pc/: (0,0)*{\xybox{
(0,0)*{\dcross{i}{j}};
(0,8)*{\dcross{}{}};
}};
\endxy \right)
\]
Since 
\[
\END_{R(w_J(\alpha_i+\alpha_j))} (\Delta(w_J,i) \circ \Delta(w_J,j)) \simeq e(i,j)R(\alpha_i+\alpha_j)e(i,j), 
\]
both 
\[
\text{$S_{[h,1]}\left(\mathbf{1} \xy 0;/r.12pc/: (0,0)*{\xybox{
(0,0)*{\dcross{i}{j}};
(0,8)*{\dcross{}{}};
}};
\endxy\right)$ and $S_{\overline{[h,1]}}\left(\mathbf{1} \xy 0;/r.12pc/: (0,0)*{\xybox{
(0,0)*{\dcross{i}{j}};
(0,8)*{\dcross{}{}};
}};
\endxy \right)$}
\]
are nonzero. 
Therefore, we must have $b_{i,j} b_{j,i} = 1$. 
\end{proof}

Now, we have completed the proof of Proposition \ref{prop:braidrel} assuming Lemma \ref{lem:g1}.  

\subsection[Proof of Lemma]{Proof of Lemma \ref{lem:g1}}

\subsubsection{The module $\Delta(s_2,1)$}

\begin{lemma}[{\cite[Lemma 4.2]{MR2995184}}] \label{lem:nilhecke} 
Let $n \geq 1$, and let $M$ be a module over the nil-Hecke algebra $R(n\alpha_i)$. 
Assume that 
\[
f = f(x_1, \ldots,x_n) \in \END_{R(n\alpha_i)}(M)[x_1,\ldots,x_n]
\]
acts on $M$ by zero. 
Then, $s_k(f)$ and $\partial_k(f)$ also acts on $M$ by zero for any $1 \leq k \leq n-1$. 
\end{lemma}

\begin{lemma} \label{lem:Delta}
Put $n = -a_{2,1}$. 
Note that $\Delta(s_2,1) \simeq S_2(R(\alpha_1))$ is an $R(s_2\alpha_1) = R(\alpha_1 + n\alpha_2)$-module. 
We have 
\begin{enumerate}
\item $e(2^n,1)$ acts on $\Delta(s_2,1)$ as the identity. 
\item $(t_{2,1}x_k^{-a_{2,1}} + t_{1,2}x_{n+1}^{-a_{1,2}}) \Delta(s_2,1) = 0$ for $1 \leq k \leq n$, 
\item $(x_{k_1}^r + \cdots + x_{k_{n+1-r}}^r) \Delta(s_2,1) = 0$ for any $1 \leq r \leq n-1$ and $1 \leq k_1 < \cdots < k_{n+1-r} \leq n$.  
\item The canonical surjective homomorphism $R(n\alpha_2) \circ R(\alpha_1) \to F_2^{(n)}R(\alpha_1) = \Delta(s_2,1)$ induces a surjective homomorphism $R(n\alpha_2)F_1 \to \Delta(s_2,1)$. 
\end{enumerate}
Similar assertions hold for the modules $\Delta(s_1,2), \Delta'(s_2,1), \Delta'(s_1,2)$. 
\end{lemma}

\begin{proof}
(1) By definition, $\Delta(s_2,1) \simeq F_2^{(-a_{2,1})}(R(\alpha_1)) \in \gMod{R_2(s_2\alpha_1)}$, 
that is, $e(s_2\alpha_1-\alpha_2,2) \Delta(s_2,1) = 0$. 
(1) follows. 

(2) By (1), we have 
\begin{align*}
0 &= \tau_n^2 e(2^n,1) \Delta(s_2,1) = Q_{2,1}(x_n,x_{n+1}) \Delta(s_2,1) \\
&= (t_{2,1}x_n^{-a_{2,1}} + t_{1,2}x_{n+1}^{-a_{1,2}}) \Delta(s_2,1).
\end{align*}

(3) Note that, for $1 \leq r \leq n-1$, we have 
\begin{align*}
&\partial_{r} \partial_{r+1} \cdots \partial_{n-2} \partial_{n-1}(t_{2,1}x_n^{-a_{2,1}} + t_{1,2}x_{n+1}^{-a_{1,2}}) \\
&= t_{2,1} (x_r^r + x_{r+1}^r + \cdots + x_n^r). 
\end{align*}
Since the nil-Hecke algebra $R(n\alpha_2)$ acts on $\Delta(s_2,1) = e(2^n,1)\Delta(s_2,1)$, 
the assertion follow from Lemma \ref{lem:nilhecke}. 

(4) By (1), we have $\Delta(s_2,1) \in \gMod{{}_1R}$. Hence, the assertion follows. 
\end{proof}

\begin{lemma} \label{lem:quotient}
Assume $(\alpha_2,\alpha_2) \geq (\alpha_1,\alpha_1)$. 
Put $n = -a_{1,2}$. 
Then, the canonical surjective homomorphism 
\[
R((n-1)\alpha_1) \circ R(\alpha_1) \circ R(\alpha_2) \simeq R(n\alpha_1) \circ R(\alpha_2) \to F_1^{(n)}R(\alpha_2) = \Delta(s_1,2)
\]
induces a surjective homomorphism 
\[
R((n-1)\alpha_1) \circ R(\alpha_1)F_2 \to \Delta(s_1,2).  
\]
\end{lemma}

\begin{proof}
By Lemma \ref{lem:adjointSES}, we have a short exact sequence
\begin{align*}
0 &\to R((n-1)\alpha_1)F_2 \circ R(\alpha_1) \to (R((n-1)\alpha_1) \circ R(\alpha_1)) F_2 \\
&\to R((n-1)\alpha_1) \circ R(\alpha_1)F_2 \to 0. 
\end{align*}
By Lemma \ref{lem:Delta} (4), the canonical surjective homomorphism $R(n\alpha_1) \circ R(\alpha_2) \to \Delta(s_1,2)$ factors through $R(n\alpha_1) F_2 \simeq (R((n-1)\alpha_1)\circ R(\alpha_1))F_2$. 
On the other hand, we have 
\begin{align*}
&\HOM_{R(s_1\alpha_2)}(R((n-1)\alpha_1)F_2 \circ R(\alpha_1), \Delta(s_2,1)) \\
&\simeq \HOM_{R(s_1\alpha_2)} (R((n-1)\alpha_1)F_2 \otimes R(\alpha_1), \Res_{(n-1)\alpha_1+\alpha_2,\alpha_1}\Delta(s_1,2)) \\
&\quad \text{by induction-restriction adjunction}\\
&= 0 \quad \text{by Lemma \ref{lem:Delta} (1)}. 
\end{align*}
Hence, the lemma is proved. 
\end{proof}

\begin{comment}
\begin{proof}
Note that $a_{2,1} = -1, s_2\alpha_1 = \alpha_1+ \alpha_2$. 
We have a surjective homomorphism $\Delta(s_1,2) \to L(s_1,2)$. 
Since $e(1^n,2)$ acts on $\Delta(s_1,2)$ by the identity (Lemma \ref{lem:Delta}),
both $\Res_{(n-1)\alpha_1, \alpha_1+\alpha_2} \Delta(s_1,2)$ and $\Res_{(n-1)\alpha_1,\alpha_1+\alpha_2}L(s_1\alpha_2)$ are $R((n-1)\alpha_1) \otimes {}_2R(\alpha_1+\alpha_2)$-modules. 
Since $L'(s_2,1)$ is the unique simple ${}_2R(\alpha_1+\alpha_2)$-module, 
we have an injective homomorphism $L(1^{n-1}) \otimes L'(s_2,1) \to \Res_{(n-1)\alpha_1,\alpha_1+ \alpha_2} L(s_1\alpha_2)$. 
We obtain a nonzero homomorphism
\[
R((n-1)\alpha_1)\otimes \Delta'(s_2,1) \twoheadrightarrow L((n-1)\alpha_1) \otimes L'(s_2,1) \hookrightarrow \Res_{(n-1)\alpha_1,\alpha_1+\alpha_2}L(s_1,2).
\]
Since $R((n-1)\alpha_1) \otimes \Delta'(s_2,1)$ is projective as an $R((n-1)\alpha_1) \otimes {}_2R(\alpha_1+\alpha_2)$-module, 
this homomorphism lifts to a homomorphism
\[
R((n-1)\alpha_1) \otimes \Delta'(s_2,1) \to \Res_{(n-1)\alpha_1,\alpha_1+\alpha_2} \Delta(s_1,2). 
\]
By the induction-restriction adjunction, we obtain a nonzero homomorphism 
\[
R((n-1)\alpha_1) \circ \Delta'(s_2,1) \to \Delta(s_1,2) 
\]
that lifts $R((n-1)\alpha_1) \circ \Delta'(s_2,1) \to L(s_1,2)$. 
Since the head of $\Delta(s_1,2)$ is the simple module $L(s_1,2)$, 
the homomorphism $R((n-1)\alpha_1) \circ \Delta'(s_2,1) \to \Delta(s_1,2)$ above is surjective. 
Since $\Delta'(s_2,1) = S'_2(R(\alpha_1)) \simeq R(\alpha_1)F_2$, the assertion follows.  
\end{proof}
\end{comment}

In the computations below, we freely use these lemmas. 
We assume that $(\alpha_2,\alpha_2) \geq (\alpha_1,\alpha_1)$. 

\subsubsection{The $A_1 \times A_1$ case} 
It is trivial.  

\subsubsection{The $A_2$ case}

First, we compute $g_1$. 
Note that $s_2\alpha_1 = \alpha_1 + \alpha_2$. 
By Lemma \ref{lem:yaction}, the endomorphism of $\Delta(s_2,1) \simeq S_2(R(\alpha_1))$ induced by $\mathbf{1} \xy 0;/r.12pc/: (0,0)*{\sdotd{}}; (0,-7)*{\scriptstyle \tilde{1}}; \endxy$ is the left-multiplication by $x_2$. 
Recall that $e(2,1)$ acts on $\Delta(s_2,1)$ as the identity, and $(t_{2,1}x_1 + t_{1,2}x_2) \Delta(s_2,1) = 0$. 
It follows that the endomorphism of $\Delta(s_2,1)$ above coincides with the left multiplication by $-\frac{t_{2,1}}{t_{1,2}}x_1$. 

On the other hand, the isomorphism $S_1(\Delta(s_2,1)) \simeq \Delta(s_1s_2,1) \simeq R(\alpha_2)$ of Lemma \ref{lem:std} (3) gives an isomorphism $\Delta(s_2,1) \simeq S'_1(R(\alpha_2))$.
By Lemma \ref{lem:yaction}, the endomorphism of $S_1'(R(\alpha_2))$ induced by $\mathbf{1} \xy 0;/r.12pc/: (0,0)*{\sdotd{}}; (0,-7)*{\scriptstyle \tilde{1}}; \endxy$ coincides with the left multiplication by $x_1$. 
Therefore, the endomorphism of $\Delta(s_1s_2,1) \simeq S_1S_2(R(\alpha_1))$ induced by $\mathbf{1} \xy 0;/r.12pc/: (0,0)*{\sdotd{}}; (0,-7)*{\scriptstyle \tilde{1}}; \endxy$ coincides with $-\frac{t_{2,1}}{t_{1,2}}$-multiple of $\mathbf{1} \xy 0;/r.12pc/: (0,0)*{\sdotd{}}; (0,-7)*{\scriptstyle \tilde{2}}; \endxy$. 
Hence, $g_1 = -\frac{t_{2,1}}{t_{1,2}}$. 

By interchanging $1$ and $2$, we deduce $g_2 = -\frac{t_{1,2}}{t_{2,1}}$. 

\subsubsection{The $B_2$ case}

First, we compute $g_1$. 
Note that $s_2\alpha_1 = \alpha_1 + \alpha_2$. 
By Lemma \ref{lem:yaction}, the endomorphism of $\Delta(s_2,1) \simeq S_2(R(\alpha_1))$ induced by $\mathbf{1} \xy 0;/r.12pc/: (0,0)*{\sdotd{}}; (0,-7)*{\scriptstyle \tilde{1}}; \endxy$ is the left-multiplication by $x_2$. 
Furthermore, $\Delta(s_2,1)$ is a quotient of $R(\alpha_2)F_1$, and the endomorphism above is induced by $R(\alpha_2) \xy 0;/r.12pc/: (0,0)*{\sdotd{1}}; \endxy$. 

Note that $s_1s_2\alpha_1 = \alpha_1 + \alpha_2$. 
We have $S_1(R(\alpha_2)F_1) \simeq E_1 S_1(R(\alpha_2)) \simeq E_1 \Delta(s_1,2)$ and
\[
S_1 \left(R(\alpha_2)  \xy 0;/r.12pc/: (0,0)*{\sdotd{1}}; \endxy \right) =  \xy 0;/r.12pc/: (0,0)*{\sdotu{1}}; \endxy \Delta(s_1,2). 
\]
Recall that $\Delta(s_1,2)$ is an $R(2\alpha_1+\alpha_2)$-module, on which $e(1,1,2)$ acts as the identity. 
Hence, $E_1\Delta(s_1,2) = \Delta(s_1,2)$ as a vector space. 
The endomorphism of $E_1\Delta(s_1,2)$ above coincides with the left action of $x_1$ on $\Delta(s_1,2)$. 
Since $(x_1 + x_2) \Delta(s_1,2) = 0$, it also coincides with the left multiplication by $-x_2$ on $\Delta(s_2,1)$. 
Hence, the endomorphism $S_1S_2 \left( \mathbf{1}\xy 0;/r.12pc/: (0,0)*{\sdotd{1}}; \endxy\right)$ of $\Delta(s_1s_2,1)$ is the left multiplication by $-x_1$. 

Note that $S_1S_2(R(\alpha_1)) \simeq S'_2 S_2S_1S_2(R(\alpha_1)) \overset{(\ref{eq:fixedisom})}{\simeq} S'_2 (R(\alpha_1)) \simeq \Delta'(s_2,1)$. 
It is a quotient of $R(\alpha_1)F_2= \mathbf{1}\tilde{F}_1F_2$, 
and the endomorphism $S_1S_2 \left( \mathbf{1}\xy 0;/r.12pc/: (0,0)*{\sdotd{1}}; \endxy\right)$ is induced by $-\mathbf{1} \xy 0;/r.12pc/: (0,0)*{\sdotd{}}; (0,-7)*{\scriptstyle \tilde{1}}; \endxy F_2$. 
Applying $S_2$, we deduce that $\Delta(s_2s_1s_2,1) \overset{(\ref{eq:fixedisom})}{\simeq} R(\alpha_1)$ is a quotient of $S_2(R(\alpha_1)F_2) \simeq E_2\Delta(s_2,1)$, 
and the endomorphism $S_2S_1S_2\left(\mathbf{1} \xy 0;/r.12pc/: (0,0)*{\sdotd{}}; (0,-7)*{\scriptstyle \tilde{1}}; \endxy\right)$ of $R(\alpha_1)$ coincides with the one induced from the left multiplication by $-x_2$ on $\Delta(s_2,1)$ (Lemma \ref{lem:yaction}). 
It implies that $g_1 = -1$. 

Next, we compute $g_2$. 
Note that $s_1\alpha_2 = 2\alpha_1 + \alpha_2$. 
By Lemma \ref{lem:yaction}, the endomorphism $S_1(\mathbf{1}\xy 0;/r.12pc/: (0,0)*{\sdotd{}}; (0,-7)*{\scriptstyle \tilde{2}}; \endxy)$ of $\Delta(s_1,2) \simeq S_1(R(\alpha_2))$ coincides with the action of $x_3$. 
Furthermore, we have a surjective homomorphism $R(\alpha_1) \circ R(\alpha_1) F_2 \to \Delta(s_1,2)$, 
and the endomorphism above is induced by $R(\alpha_1) \circ R(\alpha_1) \xy 0;/r.12pc/: (0,0)*{\sdotd{2}}; \endxy$. 

Applying $S_2$, we obtain a surjective homomorphism $\Delta(s_2,1) \circ E_2\Delta(s_2,1) \to \Delta(s_2s_1,2)$,
and the endomorphism $S_2S_1\left(\mathbf{1} \xy 0;/r.12pc/: (0,0)*{\sdotd{}}; (0,-7)*{\scriptstyle \tilde{2}}; \endxy\right)$ coincides with the one induced by $\Delta(s_2,1) \circ \xy 0;/r.12pc/: (0,0)*{\sdotu{2}}; \endxy \Delta(s_2,1)$. 
We have $\Delta(s_2,1) = e(2,1)\Delta(s_2,1)$, 
and $\xy 0;/r.12pc/: (0,0)*{\sdotu{2}}; \endxy \Delta(s_2\alpha_1)$ coincides with the action of $x_1$ on $\Delta(s_2,1)$. 
Furthermore, it equals to the action of $-\frac{t_{1,2}}{t_{2,1}}x_2^2$ on $\Delta(s_2,1)$, since $(t_{2,1}x_1 + t_{1,2}x_2^2) \Delta(s_2,1) = 0$. 

$\Delta(s_2s_1,2) \overset{(\ref{eq:fixedisom})}{\simeq} S'_1(R(\alpha_2)) \simeq \Delta'(s_1,2) \simeq R(\alpha_2)F_1^{(2)}$ is a quotient of $R(\alpha_2)F_1F_1$, 
and the endomorphism $S_2S_1\left(\mathbf{1}\xy 0;/r.12pc/: (0,0)*{\sdotd{}}; (0,-7)*{\scriptstyle \tilde{2}}; \endxy \right)$ of $\Delta(s_2s_1,2)$ is induced by $-\frac{t_{1,2}}{t_{2,1}}$-multiple of $R(\alpha_2)F_1 \xy 0;/r.12pc/: (0,0)*{\sdotd{1}}; (3,0)*{\scriptstyle 2}; \endxy$. 

Applying $S_1$, the module $R(\alpha_2) \overset{(\ref{eq:fixedisom})}{\simeq} \Delta(s_1s_2s_1,2)$ is a quotient of $E_1E_1\Delta(s_1,2)$, 
and the endomorphism above becomes $-\frac{t_{1,2}}{t_{2,1}}$-multiple of $\xy 0;/r.12pc/: (0,0)*{\sdotu{1}}; (3,0)*{\scriptstyle 2}; \endxy E_1 \Delta(s_1,2)$, 
namely, the action of $x_2^2$ on $\Delta(s_1,2)$. 
Since $(t_{1,2}x_2^2 + t_{2,1}x_3)\Delta(s_1,2) = 0$, it coincides with the action of $x_3$ on $\Delta(s_1,2)$. 
Therefore, the endomorphism $S_1S_2S_1\left(\mathbf{1} \xy 0;/r.12pc/: (0,0)*{\sdotd{}}; (0,-7)*{\scriptstyle \tilde{2}}; \endxy\right)$ of $R(\alpha_2) \simeq \Delta(s_1s_2s_1\alpha_2)$ coincides with the action of $x_1$.  
It means that $g_2 = 1$. 

\subsubsection{The $G_2$ case}

First, we compute $g_1$.
Note that $s_2\alpha_1 = \alpha_1 + \alpha_2$. 
By Lemma \ref{lem:yaction}, the endomorphism $S_2\left(\mathbf{1} \xy 0;/r.12pc/: (0,0)*{\sdotd{}}; (0,-7)*{\scriptstyle \tilde{1}}; \endxy\right)$ of $\Delta(s_2,1) \simeq S_2(R(\alpha_1))$ is the left-multiplication by $x_2$. 
Furthermore, $\Delta(s_2,1)$ is a quotient of $R(\alpha_2)F_1$, and the endomorphism above is induced by $R(\alpha_2) \xy 0;/r.12pc/: (0,0)*{\sdotd{1}}; \endxy$. 

Note that $s_1s_2\alpha_1 = 2\alpha_1 + \alpha_2$. 
We have $S_1(R(\alpha_2)F_1) \simeq E_1 \Delta(s_1,2)$ and 
\[
S_1 \left(R(\alpha_2)  \xy 0;/r.12pc/: (0,0)*{\sdotd{1}}; \endxy \right) =  \xy 0;/r.12pc/: (0,0)*{\sdotu{1}}; \endxy \Delta(s_1,2). 
\]
Recall that $\Delta(s_1,2)$ is an $R(3\alpha_1+\alpha_2)$-module, on which $e(1,1,1,2)$ acts as the identity. 
Hence, $E_1\Delta(s_1,2) = \Delta(s_1,2)$ as a vector space. 
The endomorphism of $E_1\Delta(s_1,2)$ above coincides with the left action of $x_1$ on $\Delta(s_1,2)$. 
Since $(x_1 + x_2 + x_3) \Delta(s_1,2) = 0$, it also coincides with the action of $-(x_2 + x_3)$ on $\Delta(s_1,2)$, 
namely, the left action of $-(x_1 + x_2)$ on $E_1\Delta(s_1,2)$.  
Note that $\Delta(s_1s_2,1) \simeq S_1(\Delta(s_2,1))$ is a quotient of $S_1(R(\alpha_2)F_1) \simeq E_1\Delta(s_1,2)$.
Hence, the endomorphism $S_1S_2\left(\mathbf{1}\xy 0;/r.12pc/: (0,0)*{\sdotd{}}; (0,-7)*{\scriptstyle \tilde{1}}; \endxy\ \right)$ of $\Delta(s_1s_2,1)$ coincides with the action of $-(x_1 + x_2)$. 

Since $e(1,1,2)$ acts on $E_1\Delta(s_1,2)$ as the identity, 
it also acts on $\Delta(s_1s_2,1)$ as the identity. 
Since the head of $\Delta(s_1s_2,1)$ is simple $L(s_1s_2,1)$, 
there exists a surjective homomorphism $R(\alpha_1) \circ R(\alpha_1)\circ R(\alpha_2) \to \Delta(s_1s_2,1)$. 
By the same argument as the proof of Lemma \ref{lem:quotient}, 
it induces a surjective homomorphism $(\mathbf{1} \tilde{F}_1) \circ (\mathbf{1}\tilde{F}_1F_2) \simeq R(\alpha_1) \circ R(\alpha_1)F_2 \to \Delta(s_1s_2,1)$. 
By the previous paragraph, the endomorphism $S_1S_2 \left(\mathbf{1} \xy 0;/r.12pc/: (0,0)*{\sdotd{}}; (0,-7)*{\scriptstyle \tilde{1}}; \endxy \right)$ of $\Delta(s_1s_2,1)$ 
is induced by the following endomorphism of $R(\alpha_1) \circ R(\alpha_1)F_2$: 
\[
-\left( \mathbf{1} \xy 0;/r.12pc/: (0,0)*{\sdotd{}}; (0,-7)*{\scriptstyle \tilde{1}}; \endxy \right) \circ (\mathbf{1} \tilde{F}_1F_2) - (\mathbf{1}\tilde{F}_1) \circ \left(\mathbf{1} \xy 0;/r.12pc/: (0,0)*{\sdotd{}}; (0,-7)*{\scriptstyle \tilde{1}}; \endxy F_2 \right). 
\]

Applying $S_2$, we obtain a surjective homomorphism $\Delta(s_2,1) \circ E_2\Delta(s_2,1) \to \Delta(s_2s_1s_2,1)$, 
and the endomorphism $S_2S_1S_2\left(\mathbf{1} \xy 0;/r.12pc/: (0,0)*{\sdotd{}}; (0,-7)*{\scriptstyle \tilde{1}}; \endxy \right)$ of $\Delta(s_2s_1s_2,1)$ is induced by the endomorphism of $\Delta(s_2,1) \circ E_2\Delta(s_2,1)$ given by 
\[
u \boxtimes E_2v \mapsto -x_2u \boxtimes E_2v - u \boxtimes E_2(x_2v) \ (u,v \in \Delta(s_2,1)), 
\]
by Lemma \ref{lem:yaction}.
Note that $E_2\Delta(s_2,1) \simeq E_2F_2R(\alpha_1) \simeq R(\alpha_1)$ and $\Delta(s_2,1) \simeq R(\alpha_2)F_1$.
Since $\Delta(s_2s_1s_2,1) \in \gMod{{}_1R}$, 
the surjective homomorphism $R(\alpha_2)F_1 \circ R(\alpha_1) \simeq \Delta(s_2,1) \circ E_2\Delta(s_2,1) \to \Delta(s_2s_1s_2,1)$ factors through $R(\alpha_2)F_1F_1$. 
The endomorphism of $\Delta(s_2s_1s_2,1)$ above coincides with the one induced by the following endomorphism of $R(\alpha_2)F_1F_1$: 
\[ 
- R(\alpha_2) \xy 0;/r.12pc/: (0,0)*{\sdotd{1}}; (4,0)*{\slined{1}}; \endxy - R(\alpha_2)\xy 0;/r.12pc/: (0,0)*{\slined{1}}; (4,0)*{\sdotd{1}}; \endxy.
\]

Applying $S_1$, we obtain a surjective homomorphism 
\[
E_1E_1\Delta(s_1,2) \to \Delta(s_1s_2s_1s_2,1),
\] 
and the endomorphism $S_1S_2S_1S_2\left(\mathbf{1} \xy 0;/r.12pc/: (0,0)*{\sdotd{}}; (0,-7)*{\scriptstyle \tilde{1}}; \endxy \right)$ is the one induced by the following endomorphism of $E_1E_1\Delta(s_1,2)$:  
\[
- \xy 0;/r.12pc/: (0,0)*{\sdotu{1}}; (4,0)*{\slineu{1}}; \endxy \Delta(s_1,2) - \xy 0;/r.12pc/: (0,0)*{\slineu{1}}; (4,0)*{\sdotu{1}}; \endxy \Delta(s_1,2). 
\]
Note that $E_1E_1\Delta(s_1,2) = \Delta(s_1,2)$ as vector space, and this endomorphism of $E_1E_1\Delta(s_1,2)$ coincides with the action of $-(x_1 + x_2)$ on $\Delta(s_1,2)$.  
Since $(x_1 + x_2 + x_3) \Delta(s_1,2) = 0$, it also coincides with the action of $x_3$ on $\Delta(s_1,2)$, namely, the action of $x_1$ on $E_1E_1\Delta(s_1,2)$. 
Hence, the endomorphism $S_1S_2S_1S_2\left(\mathbf{1}\xy 0;/r.12pc/: (0,0)*{\sdotd{}}; (0,-7)*{\scriptstyle \tilde{1}}; \endxy  \right)$ coincides with the action of $x_1$. 

Note that $\Delta(s_1s_2s_1s_2,1) \overset{(\ref{eq:fixedisom})}{\simeq} S'_2(R(\alpha_1))$, 
and the endomorphism above is given by $S_2' \left(\xy 0;/r.12pc/: (0,0)*{\sdotd{}}; (0,-7)*{\scriptstyle \tilde{1}}; \endxy \mathbf{1} \right)$ by Lemma \ref{lem:yaction}. 
Applying $S_2$, it becomes an endomorphisms of $\Delta(s_2s_1s_2s_1s_2,1) \simeq R(\alpha_1)$,
which coincides with the action of $x_1$ on $R(\alpha_1)$. 
Hence, $g_1= 1$.

Next, we compute $g_2$. 
Note that $s_1\alpha_2 = 3\alpha_1 + \alpha_2$.  
By Lemma \ref{lem:yaction}, the endomorphism of $\Delta(s_1,2) \simeq S_1(R(\alpha_2))$ given by $S_1 \left(\mathbf{1}\xy 0;/r.12pc/: (0,0)*{\sdotd{}}; (0,-7)*{\scriptstyle \tilde{2}}; \endxy \right)$ coincides with the action of $x_4$. 
Furthermore, we have a surjective homomorphism $R(\alpha_1) \circ R(\alpha_1) \circ R(\alpha_1) F_2 \to \Delta(s_1,2)$, 
and the endomorphism above is induced by the endomorphism $R(\alpha_1) \circ R(\alpha_1) \circ R(\alpha_1) \xy 0;/r.12pc/: (0,0)*{\sdotd{2}}; \endxy$ of $R(\alpha_1) \circ R(\alpha_1) \circ R(\alpha_1)F_2$. 

Applying $S_2$, we obtain a surjective homomorphism $\Delta(s_2,1) \circ \Delta(s_2,1) \circ E_2\Delta(s_2,1) \to \Delta(s_2s_1,2)$,
and the endomorphism $S_2S_1\left(\mathbf{1}\xy 0;/r.12pc/: (0,0)*{\sdotd{}}; (0,-7)*{\scriptstyle \tilde{2}}; \endxy\right)$ coincides with the one induced by the endomorphism $\Delta(s_2,1) \circ \Delta(s_2,1) \circ \xy 0;/r.12pc/: (0,0)*{\sdotu{2}}; \endxy \Delta(s_2,1)$ of $\Delta(s_2,1) \circ \Delta(s_2,1) \circ E_2\Delta(s_2,1)$. 
We have $\Delta(s_2,1) = e(2,1)\Delta(s_2,1)$, 
and $\xy 0;/r.12pc/: (0,0)*{\sdotu{2}}; \endxy \Delta(s_2,1)$ coincides with the action of $x_1$ on $\Delta(s_2,1)$. 
Furthermore, it equals to the action of $-\frac{t_{1,2}}{t_{2,1}}x_2^3$ on $\Delta(s_2,1)$. 

Note that $E_2\Delta(s_2,1) \simeq R(\alpha_1)$. 
Since $\Delta(s_2s_1,2) \in \gMod{{}_{s_1s_2s_1s_2}R_{s_2s_1}}$ (Theorem \ref{thm:variousequiv}) and $(s_2s_1)^{-1} (s_2\alpha_1) = -\alpha_1 \not \in \mathsf{Q}_+$, 
we have $\Res_{s_2\alpha_1+\alpha_1,s_2\alpha_1}\Delta(s_2s_1,2) = 0$. 
Hence, the surjective homomorphism $\Delta(s_2,1) \circ \Delta(s_2,1) \circ R(\alpha_1) \to \Delta(s_2s_1,2)$ factors through $\Delta(s_2,1) \circ \Delta(s_2,1) F_1$ by Proposition \ref{thm:functorF2}. 
The endomorphism $S_2S_1\left( \mathbf{1}\xy 0;/r.12pc/: (0,0)*{\sdotd{}}; (0,-7)*{\scriptstyle \tilde{2}}; \endxy\right)$ of $\Delta(s_2s_1,2)$ coincides with the one induced by the endomorphism $-\frac{t_{1,2}}{t_{2,1}} \Delta(s_2,1) \circ \Delta(s_2,1) \xy 0;/r.12pc/: (0,0)*{\sdotd{1}}; (3,0)*{\scriptstyle 3}; \endxy$ of $\Delta(s_2,1) \circ \Delta(s_2,1)F_1$. 

Applying $S_1$, we obtain a surjective homomorphism $\Delta(s_1s_2,1) \circ E_1\Delta(s_1s_2,1) \to \Delta(s_1s_2s_1,2)$, 
and the endomorphism $S_1S_2S_1\left(\mathbf{1} \xy 0;/r.12pc/: (0,0)*{\sdotd{}}; (0,-7)*{\scriptstyle \tilde{2}}; \endxy  \right)$ coincides with the one induced by $-\frac{t_{1,2}}{t_{2,1}}$-multiple of $\Delta(s_1s_2,1) \circ \xy 0;/r.12pc/: (0,0)*{\sdotu{1}}; (3,0)*{\scriptstyle 3};  \endxy \Delta(s_1s_2,1)$. 
In the computation of $g_1$, we have seen that $e(1,1,2)$ acts on $\Delta(s_1s_2,1)$ by the identity. 
Hence, $E_1\Delta(s_1s_2,1) = \Delta(s_1s_2,1)$ as a vector space, 
and the endomorphism $\xy 0;/r.12pc/: (0,0)*{\sdotu{1}}; (3,0)*{\scriptstyle 3}; \endxy \Delta(s_1s_2,1)$ coincides with the action of $x_1^3$ on $\Delta(s_1s_2,1)$. 
By the same argument as the proof of Lemma \ref{lem:Delta}, we have $(t_{1,2}x_1^3 + t_{2,1}x_3)\Delta(s_1s_2,1)$. 
Therefore, the endomorphism $-\frac{t_{1,2}}{t_{2,1}} \xy 0;/r.12pc/: (0,0)*{\sdotu{1}}; (3,0)*{\scriptstyle 3}; \endxy \ \Delta(s_1s_2,1)$ coincides with the action of $x_3$ on $\Delta(s_1s_2,1)$, 
that is, the action of $x_2$ on $E_1\Delta(s_1s_2,1)$. 

Since $e(1,1,2)\Delta(s_1s_2,1) = \Delta(s_1s_2,1)$ as a vector space, we have a canonical surjective homomorphism $(E_1\Delta(s_1s_2,1))E_2F_2 \to E_1\Delta(s_1s_2,1)$. 
Put $M = (E_1\Delta(s_1s_2,1))E_2$, which is an $R(\alpha_1)$-module. 
The action of $x_2$ on $E_1\Delta(s_1s_2,1)$ is induced by the endomorphism $M \xy 0;/r.12pc/: (0,0)*{\sdotd{2}}; \endxy$. 
To summarize, we have a surjective homomorphism $\Delta(s_1s_2,1) \circ MF_2 \to \Delta(s_1s_2s_1,2)$, 
and the endomorphism $S_1S_2S_1\left(\mathbf{1} \xy 0;/r.12pc/: (0,0)*{\sdotd{}}; (0,-7)*{\scriptstyle \tilde{2}}; \endxy \right)$ is induced by $\Delta(s_1s_2,1) \circ M \xy 0;/r.12pc/: (0,0)*{\sdotd{2}}; \endxy$. 

Applying $S_2$, we obtain a surjective homomorphism $\Delta(s_2s_1s_2,1) \circ E_2 S_2(M) \to \Delta(s_2s_1s_2s_1,2)$, 
and the endomorphism $S_2S_1S_2S_1\left(\mathbf{1} \xy 0;/r.12pc/: (0,0)*{\sdotd{}}; (0,-7)*{\scriptstyle \tilde{2}}; \endxy  \right)$ coincides with the one induced by $\Delta(s_2s_1s_2,1) \circ \xy 0;/r.12pc/: (0,0)*{\sdotu{2}}; \endxy S_2(M)$. 
Since $M$ is and $R(\alpha_1)$-module, $S_2(M)$ is an $R(\alpha_1+\alpha_2)$-module on which $e(2,1)$ acts as the identity. 
Hence, $E_2S_2(M) = S_2(M)$ as a vector space, 
and the endomorphism $\xy 0;/r.12pc/: (0,0)*{\sdotu{2}}; \endxy S_2(M)$ coincides with the action of $x_1$ on $S_2(M)$.
Furthermore, since $(t_{2,1}x_1 + t_{1,2}x_2^3) S_2(M) = \tau_1^2 S_2(M) = 0$, it also coincides with the action of $-\frac{t_{1,2}}{t_{2,1}}x_2^3$ on $S_2(M)$, 
that is, the action of $-\frac{t_{1,2}}{t_{2,1}}x_1^3$ on $E_2S_2(M)$. 
Hence, the endomorphism $S_2S_1S_2S_1\left(\mathbf{1} \xy 0;/r.12pc/: (0,0)*{\sdotd{}}; (0,-7)*{\scriptstyle \tilde{2}}; \endxy  \right)$ is the actionn of $-\frac{t_{1,2}}{t_{2,1}}x_2^3$. 

Note that $\Delta(s_2s_1s_2s_1,2) \overset{(\ref{eq:fixedisom})}{\simeq} S_1'(R(\alpha_2))$ and $e(2,1,1,1)$ acts on it as the identity. 
The action of $-\frac{t_{1,2}}{t_{2,1}}x_4^3$ on $\Delta(s_2s_1s_2s_1,2)$ coincides with the action of $x_1$.
Therefore, the endomorphism $S_2S_1S_2S_1 \left(\mathbf{1} \xy 0;/r.12pc/: (0,0)*{\sdotd{}}; (0,-7)*{\scriptstyle \tilde{2}}; \endxy \right)$ of $\Delta(s_2s_1s_2s_1,2)$ coincides with the action of $x_1$. 
By Lemma \ref{lem:yaction}, it also coincides with $S_1'\left( \xy 0;/r.12pc/: (0,0)*{\sdotd{}}; (0,-7)*{\scriptstyle \tilde{2}}; \endxy \mathbf{1} \right)$. 
Applying $S_1$, we see that the endomorphism $S_1S_2S_1S_2S_1\left(\mathbf{1} \xy 0;/r.12pc/: (0,0)*{\sdotd{}}; (0,-7)*{\scriptstyle \tilde{2}}; \endxy \right)$ of $\Delta(s_1s_2s_1s_2s_1,2) \simeq R(\alpha_2)$ coincides with the action of $x_1$. 
It indicates that $g_2 = 1$.

\bibliographystyle{amsalpha}
\bibliography{library.bib}

@article {MR4717658,
    AUTHOR = {Kashiwara, Masaki and Kim, Myungho and Oh, Se-jin and Park,
              Euiyong},
     TITLE = {Affinizations, {R}-matrices and reflection functors},
   JOURNAL = {Adv. Math.},
  FJOURNAL = {Advances in Mathematics},
    VOLUME = {443},
      YEAR = {2024},
     PAGES = {Paper No. 109598, 83},
      ISSN = {0001-8708,1090-2082},
   MRCLASS = {18M05 (16D90 81R10)},
  MRNUMBER = {4717658},
       DOI = {10.1016/j.aim.2024.109598},
       URL = {https://doi.org/10.1016/j.aim.2024.109598},
}

@article {MR4359265,
    AUTHOR = {Kashiwara, Masaki and Kim, Myungho and Oh, Se-Jin and Park,
              Euiyong},
     TITLE = {Localizations for quiver {H}ecke algebras},
   JOURNAL = {Pure Appl. Math. Q.},
  FJOURNAL = {Pure and Applied Mathematics Quarterly},
    VOLUME = {17},
      YEAR = {2021},
    NUMBER = {4},
     PAGES = {1465--1548},
      ISSN = {1558-8599,1558-8602},
   MRCLASS = {18M05 (16G20 20C08 81R10)},
  MRNUMBER = {4359265},
MRREVIEWER = {Mee\ Seong\ Im},
       DOI = {10.4310/PAMQ.2021.v17.n4.a8},
       URL = {https://doi.org/10.4310/PAMQ.2021.v17.n4.a8},
}

@article {MR3790066,
    AUTHOR = {Kashiwara, Masaki and Park, Euiyong},
     TITLE = {Affinizations and {R}-matrices for quiver {H}ecke algebras},
   JOURNAL = {J. Eur. Math. Soc. (JEMS)},
  FJOURNAL = {Journal of the European Mathematical Society (JEMS)},
    VOLUME = {20},
      YEAR = {2018},
    NUMBER = {5},
     PAGES = {1161--1193},
      ISSN = {1435-9855,1435-9863},
   MRCLASS = {17B37 (16T20 81R10)},
  MRNUMBER = {3790066},
MRREVIEWER = {Serge\ M.\ Skryabin},
       DOI = {10.4171/JEMS/785},
       URL = {https://doi.org/10.4171/JEMS/785},
}

@article {MR3771147,
    AUTHOR = {Kashiwara, Masaki and Kim, Myungho and Oh, Se-jin and Park,
              Euiyong},
     TITLE = {Monoidal categories associated with strata of flag manifolds},
   JOURNAL = {Adv. Math.},
  FJOURNAL = {Advances in Mathematics},
    VOLUME = {328},
      YEAR = {2018},
     PAGES = {959--1009},
      ISSN = {0001-8708,1090-2082},
   MRCLASS = {16D90 (18D10 81R10)},
  MRNUMBER = {3771147},
MRREVIEWER = {Vanessa\ Miemietz},
       DOI = {10.1016/j.aim.2018.02.013},
       URL = {https://doi.org/10.1016/j.aim.2018.02.013},
}

@article {MR3758148,
    AUTHOR = {Kang, Seok-Jin and Kashiwara, Masaki and Kim, Myungho and Oh,
              Se-jin},
     TITLE = {Monoidal categorification of cluster algebras},
   JOURNAL = {J. Amer. Math. Soc.},
  FJOURNAL = {Journal of the American Mathematical Society},
    VOLUME = {31},
      YEAR = {2018},
    NUMBER = {2},
     PAGES = {349--426},
      ISSN = {0894-0347,1088-6834},
   MRCLASS = {13F60 (16Gxx 17B37 18D10 81R50)},
  MRNUMBER = {3758148},
MRREVIEWER = {Fan\ Qin},
       DOI = {10.1090/jams/895},
       URL = {https://doi.org/10.1090/jams/895},
}

@article {MR3748315,
    AUTHOR = {Kang, Seok-Jin and Kashiwara, Masaki and Kim, Myungho},
     TITLE = {Symmetric quiver {H}ecke algebras and {R}-matrices of quantum
              affine algebras},
   JOURNAL = {Invent. Math.},
  FJOURNAL = {Inventiones Mathematicae},
    VOLUME = {211},
      YEAR = {2018},
    NUMBER = {2},
     PAGES = {591--685},
      ISSN = {0020-9910,1432-1297},
   MRCLASS = {17B37 (16T25 20C08 81R50)},
  MRNUMBER = {3748315},
MRREVIEWER = {Jun\ Hu},
       DOI = {10.1007/s00222-017-0754-0},
       URL = {https://doi.org/10.1007/s00222-017-0754-0},
}

@article {MR3352041,
    AUTHOR = {Kang, Seok-Jin and Kashiwara, Masaki and Kim, Myungho},
     TITLE = {Symmetric quiver {H}ecke algebras and {$R$}-matrices of
              quantum affine algebras, {II}},
   JOURNAL = {Duke Math. J.},
  FJOURNAL = {Duke Mathematical Journal},
    VOLUME = {164},
      YEAR = {2015},
    NUMBER = {8},
     PAGES = {1549--1602},
      ISSN = {0012-7094,1547-7398},
   MRCLASS = {17B67 (17B10 17B37)},
  MRNUMBER = {3352041},
MRREVIEWER = {Kyu-Hwan\ Lee},
       DOI = {10.1215/00127094-3119632},
       URL = {https://doi.org/10.1215/00127094-3119632},
}

@article {MR2995184,
    AUTHOR = {Kang, Seok-Jin and Kashiwara, Masaki},
     TITLE = {Categorification of highest weight modules via
              {K}hovanov-{L}auda-{R}ouquier algebras},
   JOURNAL = {Invent. Math.},
  FJOURNAL = {Inventiones Mathematicae},
    VOLUME = {190},
      YEAR = {2012},
    NUMBER = {3},
     PAGES = {699--742},
      ISSN = {0020-9910,1432-1297},
   MRCLASS = {17B67 (17B37)},
  MRNUMBER = {2995184},
MRREVIEWER = {Volodymyr\ Mazorchuk},
       DOI = {10.1007/s00222-012-0388-1},
       URL = {https://doi.org/10.1007/s00222-012-0388-1},
}

@inproceedings {MR1159265,
    AUTHOR = {Kashiwara, Masaki},
     TITLE = {Crystallizing the {$q$}-analogue of universal enveloping
              algebras},
 BOOKTITLE = {Proceedings of the {I}nternational {C}ongress of
              {M}athematicians, {V}ol. {I}, {II} ({K}yoto, 1990)},
     PAGES = {791--797},
 PUBLISHER = {Math. Soc. Japan, Tokyo},
      YEAR = {1991},
      ISBN = {4-431-70047-1},
   MRCLASS = {17B37 (17B10 17B67)},
  MRNUMBER = {1159265},
MRREVIEWER = {Kailash\ C.\ Misra},
}

@article {MR1115118,
    AUTHOR = {Kashiwara, M.},
     TITLE = {On crystal bases of the {$Q$}-analogue of universal enveloping
              algebras},
   JOURNAL = {Duke Math. J.},
  FJOURNAL = {Duke Mathematical Journal},
    VOLUME = {63},
      YEAR = {1991},
    NUMBER = {2},
     PAGES = {465--516},
      ISSN = {0012-7094,1547-7398},
   MRCLASS = {17B37 (17B10 17B67)},
  MRNUMBER = {1115118},
MRREVIEWER = {Kailash\ C.\ Misra},
       DOI = {10.1215/S0012-7094-91-06321-0},
       URL = {https://doi.org/10.1215/S0012-7094-91-06321-0},
}

@article {MR3694676,
    AUTHOR = {McNamara, Peter J.},
     TITLE = {Representations of {K}hovanov-{L}auda-{R}ouquier algebras
              {III}: symmetric affine type},
   JOURNAL = {Math. Z.},
  FJOURNAL = {Mathematische Zeitschrift},
    VOLUME = {287},
      YEAR = {2017},
    NUMBER = {1-2},
     PAGES = {243--286},
      ISSN = {0025-5874,1432-1823},
   MRCLASS = {16W50 (16T05)},
  MRNUMBER = {3694676},
MRREVIEWER = {Volodymyr\ Mazorchuk},
       DOI = {10.1007/s00209-016-1825-4},
       URL = {https://doi.org/10.1007/s00209-016-1825-4},
}

@article {MR3205728,
    AUTHOR = {Brundan, Jonathan and Kleshchev, Alexander and McNamara, Peter
              J.},
     TITLE = {Homological properties of finite-type
              {K}hovanov-{L}auda-{R}ouquier algebras},
   JOURNAL = {Duke Math. J.},
  FJOURNAL = {Duke Mathematical Journal},
    VOLUME = {163},
      YEAR = {2014},
    NUMBER = {7},
     PAGES = {1353--1404},
      ISSN = {0012-7094,1547-7398},
   MRCLASS = {17B37},
  MRNUMBER = {3205728},
MRREVIEWER = {Aleksandr\ A.\ Mikhalev},
       DOI = {10.1215/00127094-2681278},
       URL = {https://doi.org/10.1215/00127094-2681278},
}

@article {MR2628852,
    AUTHOR = {Khovanov, Mikhail and Lauda, Aaron D.},
     TITLE = {A categorification of quantum {${\rm sl}(n)$}},
   JOURNAL = {Quantum Topol.},
  FJOURNAL = {Quantum Topology},
    VOLUME = {1},
      YEAR = {2010},
    NUMBER = {1},
     PAGES = {1--92},
      ISSN = {1663-487X,1664-073X},
   MRCLASS = {17B37 (18D05)},
  MRNUMBER = {2628852},
MRREVIEWER = {Volodymyr\ V.\ Lyubashenko},
       DOI = {10.4171/QT/1},
       URL = {https://doi.org/10.4171/QT/1},
}

@article {MR2729010,
    AUTHOR = {Lauda, Aaron D.},
     TITLE = {A categorification of quantum {${\rm sl}(2)$}},
   JOURNAL = {Adv. Math.},
  FJOURNAL = {Advances in Mathematics},
    VOLUME = {225},
      YEAR = {2010},
    NUMBER = {6},
     PAGES = {3327--3424},
      ISSN = {0001-8708,1090-2082},
   MRCLASS = {17B37 (16T20 18D05 18E15)},
  MRNUMBER = {2729010},
MRREVIEWER = {Adriana\ M.\ Criscuolo},
       DOI = {10.1016/j.aim.2010.06.003},
       URL = {https://doi.org/10.1016/j.aim.2010.06.003},
}

@article {MR3451390,
    AUTHOR = {Brundan, Jonathan},
     TITLE = {On the definition of {K}ac-{M}oody 2-category},
   JOURNAL = {Math. Ann.},
  FJOURNAL = {Mathematische Annalen},
    VOLUME = {364},
      YEAR = {2016},
    NUMBER = {1-2},
     PAGES = {353--372},
      ISSN = {0025-5831,1432-1807},
   MRCLASS = {18D05 (17B10)},
  MRNUMBER = {3451390},
MRREVIEWER = {Ram\'{o}n\ Gonz\'{a}lez Rodr\'{\i}guez},
       DOI = {10.1007/s00208-015-1207-y},
       URL = {https://doi.org/10.1007/s00208-015-1207-y},
}

@article {MR2914878,
    AUTHOR = {Kimura, Yoshiyuki},
     TITLE = {Quantum unipotent subgroup and dual canonical basis},
   JOURNAL = {Kyoto J. Math.},
  FJOURNAL = {Kyoto Journal of Mathematics},
    VOLUME = {52},
      YEAR = {2012},
    NUMBER = {2},
     PAGES = {277--331},
      ISSN = {2156-2261,2154-3321},
   MRCLASS = {17B37 (13F60 20G42)},
  MRNUMBER = {2914878},
MRREVIEWER = {Xueqing\ Chen},
       DOI = {10.1215/21562261-1550976},
       URL = {https://doi.org/10.1215/21562261-1550976},
}

@article {MR2822211,
    AUTHOR = {Lauda, Aaron D. and Vazirani, Monica},
     TITLE = {Crystals from categorified quantum groups},
   JOURNAL = {Adv. Math.},
  FJOURNAL = {Advances in Mathematics},
    VOLUME = {228},
      YEAR = {2011},
    NUMBER = {2},
     PAGES = {803--861},
      ISSN = {0001-8708,1090-2082},
   MRCLASS = {17B37},
  MRNUMBER = {2822211},
MRREVIEWER = {Peter\ W.\ Tingley},
       DOI = {10.1016/j.aim.2011.06.009},
       URL = {https://doi.org/10.1016/j.aim.2011.06.009},
}

@book {MR2759715,
    AUTHOR = {Lusztig, George},
     TITLE = {Introduction to quantum groups},
    SERIES = {Modern Birkh\"auser Classics},
      NOTE = {Reprint of the 1994 edition},
 PUBLISHER = {Birkh\"auser/Springer, New York},
      YEAR = {2010},
     PAGES = {xiv+346},
      ISBN = {978-0-8176-4716-2},
   MRCLASS = {17B37 (16T05 17-02 17B35)},
  MRNUMBER = {2759715},
       DOI = {10.1007/978-0-8176-4717-9},
       URL = {https://doi.org/10.1007/978-0-8176-4717-9},
}

@article {MR3335289,
    AUTHOR = {Kleshchev, Alexander S.},
     TITLE = {Affine highest weight categories and affine quasihereditary
              algebras},
   JOURNAL = {Proc. Lond. Math. Soc. (3)},
  FJOURNAL = {Proceedings of the London Mathematical Society. Third Series},
    VOLUME = {110},
      YEAR = {2015},
    NUMBER = {4},
     PAGES = {841--882},
      ISSN = {0024-6115,1460-244X},
   MRCLASS = {16G99 (16P40 16W50)},
  MRNUMBER = {3335289},
MRREVIEWER = {Volodymyr\ Mazorchuk},
       DOI = {10.1112/plms/pdv004},
       URL = {https://doi.org/10.1112/plms/pdv004},
}

@article {MR2066942,
    AUTHOR = {Beck, Jonathan and Nakajima, Hiraku},
     TITLE = {Crystal bases and two-sided cells of quantum affine algebras},
   JOURNAL = {Duke Math. J.},
  FJOURNAL = {Duke Mathematical Journal},
    VOLUME = {123},
      YEAR = {2004},
    NUMBER = {2},
     PAGES = {335--402},
      ISSN = {0012-7094,1547-7398},
   MRCLASS = {17B37 (17B67)},
  MRNUMBER = {2066942},
MRREVIEWER = {David\ Hernandez},
       DOI = {10.1215/S0012-7094-04-12325-2X},
       URL = {https://doi.org/10.1215/S0012-7094-04-12325-2X},
}

@article {MR2525917,
    AUTHOR = {Khovanov, Mikhail and Lauda, Aaron D.},
     TITLE = {A diagrammatic approach to categorification of quantum groups.
              {I}},
   JOURNAL = {Represent. Theory},
  FJOURNAL = {Representation Theory. An Electronic Journal of the American
              Mathematical Society},
    VOLUME = {13},
      YEAR = {2009},
     PAGES = {309--347},
      ISSN = {1088-4165},
   MRCLASS = {17B37},
  MRNUMBER = {2525917},
MRREVIEWER = {Fan\ Xu},
       DOI = {10.1090/S1088-4165-09-00346-X},
       URL = {https://doi.org/10.1090/S1088-4165-09-00346-X},
}

@misc{rouquier20082kacmoodyalgebras,
      title={2-{K}ac-{M}oody algebras}, 
      author={Raphael Rouquier},
      year={2008},
      eprint={0812.5023},
      archivePrefix={arXiv},
      primaryClass={math.RT},
      url={https://arxiv.org/abs/0812.5023}, 
      note={arXiv:0812.5023}
}

@article {MR3165425,
    AUTHOR = {Kato, Syu},
     TITLE = {Poincar\'e-{B}irkhoff-{W}itt bases and
              {K}hovanov-{L}auda-{R}ouquier algebras},
   JOURNAL = {Duke Math. J.},
  FJOURNAL = {Duke Mathematical Journal},
    VOLUME = {163},
      YEAR = {2014},
    NUMBER = {3},
     PAGES = {619--663},
      ISSN = {0012-7094,1547-7398},
   MRCLASS = {17B37 (16T20)},
  MRNUMBER = {3165425},
MRREVIEWER = {Sonia\ Natale},
       DOI = {10.1215/00127094-2405388},
       URL = {https://doi.org/10.1215/00127094-2405388},
}

@article {MR4216698,
    AUTHOR = {Kato, Syu},
     TITLE = {On the monoidality of {S}aito reflection functors},
   JOURNAL = {Int. Math. Res. Not. IMRN},
  FJOURNAL = {International Mathematics Research Notices. IMRN},
      YEAR = {2020},
    NUMBER = {22},
     PAGES = {8600--8623},
      ISSN = {1073-7928,1687-0247},
   MRCLASS = {16G20 (16E35 16W50 17B67 18M99)},
  MRNUMBER = {4216698},
MRREVIEWER = {Guiyu\ Yang},
       DOI = {10.1093/imrn/rny233},
       URL = {https://doi.org/10.1093/imrn/rny233},
}

@misc{mcnamara2020representationtheorygeometricextension,
      title={Representation theory of geometric extension algebras}, 
      author={Peter J. McNamara},
      year={2020},
      eprint={1701.07949},
      archivePrefix={arXiv},
      primaryClass={math.RT},
      url={https://arxiv.org/abs/1701.07949}, 
}

@article {MR3612467,
    AUTHOR = {Kleshchev, Alexander and Muth, Robert},
     TITLE = {Stratifying {KLR} algebras of affine {ADE} types},
   JOURNAL = {J. Algebra},
  FJOURNAL = {Journal of Algebra},
    VOLUME = {475},
      YEAR = {2017},
     PAGES = {133--170},
      ISSN = {0021-8693,1090-266X},
   MRCLASS = {17B67 (05E10 16G30 17B22 20C30 20G43)},
  MRNUMBER = {3612467},
       DOI = {10.1016/j.jalgebra.2016.07.006},
       URL = {https://doi.org/10.1016/j.jalgebra.2016.07.006},
}

@article {MR2963085,
    AUTHOR = {Khovanov, Mikhail and Lauda, Aaron D. and Mackaay, Marco and
              Sto\v{s}i\'{c}, Marko},
     TITLE = {Extended graphical calculus for categorified quantum {${\rm
              sl}(2)$}},
   JOURNAL = {Mem. Amer. Math. Soc.},
  FJOURNAL = {Memoirs of the American Mathematical Society},
    VOLUME = {219},
      YEAR = {2012},
    NUMBER = {1029},
     PAGES = {vi+87},
      ISSN = {0065-9266,1947-6221},
      ISBN = {978-0-8218-8977-0},
   MRCLASS = {81R50 (05Axx 18D05)},
  MRNUMBER = {2963085},
MRREVIEWER = {Fan\ Xu},
       DOI = {10.1090/S0065-9266-2012-00665-4},
       URL = {https://doi.org/10.1090/S0065-9266-2012-00665-4},
}

@article {MR2551762,
    AUTHOR = {Brundan, Jonathan and Kleshchev, Alexander},
     TITLE = {Blocks of cyclotomic {H}ecke algebras and {K}hovanov-{L}auda
              algebras},
   JOURNAL = {Invent. Math.},
  FJOURNAL = {Inventiones Mathematicae},
    VOLUME = {178},
      YEAR = {2009},
    NUMBER = {3},
     PAGES = {451--484},
      ISSN = {0020-9910,1432-1297},
   MRCLASS = {20C08 (16G99 20F55)},
  MRNUMBER = {2551762},
MRREVIEWER = {Andrew\ Mathas},
       DOI = {10.1007/s00222-009-0204-8},
       URL = {https://doi.org/10.1007/s00222-009-0204-8},
}

@article {MR2763732,
    AUTHOR = {Khovanov, Mikhail and Lauda, Aaron D.},
     TITLE = {A diagrammatic approach to categorification of quantum groups
              {II}},
   JOURNAL = {Trans. Amer. Math. Soc.},
  FJOURNAL = {Transactions of the American Mathematical Society},
    VOLUME = {363},
      YEAR = {2011},
    NUMBER = {5},
     PAGES = {2685--2700},
      ISSN = {0002-9947,1088-6850},
   MRCLASS = {17B37 (16T20)},
  MRNUMBER = {2763732},
MRREVIEWER = {Volodymyr\ Mazorchuk},
       DOI = {10.1090/S0002-9947-2010-05210-9},
       URL = {https://doi.org/10.1090/S0002-9947-2010-05210-9},
}

@article {MR3461059,
    AUTHOR = {Beliakova, Anna and Habiro, Kazuo and Lauda, Aaron D. and
              Webster, Ben},
     TITLE = {Cyclicity for categorified quantum groups},
   JOURNAL = {J. Algebra},
  FJOURNAL = {Journal of Algebra},
    VOLUME = {452},
      YEAR = {2016},
     PAGES = {118--132},
      ISSN = {0021-8693,1090-266X},
   MRCLASS = {18D05 (17B37)},
  MRNUMBER = {3461059},
MRREVIEWER = {Robert\ Harold\ McRae},
       DOI = {10.1016/j.jalgebra.2015.11.041},
       URL = {https://doi.org/10.1016/j.jalgebra.2015.11.041},
}

@article {MR3300416,
    AUTHOR = {Cautis, Sabin and Lauda, Aaron D.},
     TITLE = {Implicit structure in 2-representations of quantum groups},
   JOURNAL = {Selecta Math. (N.S.)},
  FJOURNAL = {Selecta Mathematica. New Series},
    VOLUME = {21},
      YEAR = {2015},
    NUMBER = {1},
     PAGES = {201--244},
      ISSN = {1022-1824,1420-9020},
   MRCLASS = {17B37 (14F05 18D05)},
  MRNUMBER = {3300416},
MRREVIEWER = {Pedro\ Vaz},
       DOI = {10.1007/s00029-014-0162-x},
       URL = {https://doi.org/10.1007/s00029-014-0162-x},
}

@article {MR4186573,
    AUTHOR = {Dupont, Benjamin},
     TITLE = {Rewriting modulo isotopies in {K}hovanov-{L}auda-{R}ouquier's
              categorification of quantum groups},
   JOURNAL = {Adv. Math.},
  FJOURNAL = {Advances in Mathematics},
    VOLUME = {378},
      YEAR = {2021},
     PAGES = {Paper No. 107524, 75},
      ISSN = {0001-8708,1090-2082},
   MRCLASS = {17B37 (16T20 18N10 68Q42)},
  MRNUMBER = {4186573},
MRREVIEWER = {Kenneth\ A.\ Brown},
       DOI = {10.1016/j.aim.2020.107524},
       URL = {https://doi.org/10.1016/j.aim.2020.107524},
}

@misc{webster2024unfurlingkhovanovlaudarouquieralgebras,
      title={Unfurling {K}hovanov-{L}auda-{R}ouquier algebras}, 
      author={Ben Webster},
      year={2016},
      eprint={1603.06311},
      archivePrefix={arXiv},
      primaryClass={math.RT},
      url={https://arxiv.org/abs/1603.06311}, 
      note={arXiv:1603.06311}
}

@book {MR4285453,
    AUTHOR = {Vera, Laurent},
     TITLE = {Categorified {Q}uantum {G}roups and {B}raid {G}roup {A}ctions},
      NOTE = {Thesis (Ph.D.)--University of California, Los Angeles},
 PUBLISHER = {ProQuest LLC, Ann Arbor, MI},
      YEAR = {2021},
     PAGES = {177},
      ISBN = {979-8515-20267-5},
   MRCLASS = {99-05},
  MRNUMBER = {4285453},
       URL =
              {http://gateway.proquest.com/openurl?url_ver=Z39.88-2004&rft_val_fmt=info:ofi/fmt:kev:mtx:dissertation&res_dat=xri:pqm&rft_dat=xri:pqdiss:28541953},
}

@article {MR4732757,
    AUTHOR = {Abram, Michael T. and Lamberto-Egan, Laffite and Lauda, Aaron
              D. and Rose, David E. V.},
     TITLE = {Categorification of the internal braid group action for
              quantum groups, {I}: 2-functoriality},
   JOURNAL = {Pacific J. Math.},
  FJOURNAL = {Pacific Journal of Mathematics},
    VOLUME = {328},
      YEAR = {2024},
    NUMBER = {1},
     PAGES = {1--75},
      ISSN = {0030-8730,1945-5844},
   MRCLASS = {17B37 (20F36)},
  MRNUMBER = {4732757},
       DOI = {10.2140/pjm.2024.328.1},
       URL = {https://doi.org/10.2140/pjm.2024.328.1},
}

@article {MR1265471,
    AUTHOR = {Saito, Yoshihisa},
     TITLE = {P{BW} basis of quantized universal enveloping algebras},
   JOURNAL = {Publ. Res. Inst. Math. Sci.},
  FJOURNAL = {Kyoto University. Research Institute for Mathematical
              Sciences. Publications},
    VOLUME = {30},
      YEAR = {1994},
    NUMBER = {2},
     PAGES = {209--232},
      ISSN = {0034-5318,1663-4926},
   MRCLASS = {17B37 (17B10)},
  MRNUMBER = {1265471},
MRREVIEWER = {Kailash\ C.\ Misra},
       DOI = {10.2977/prims/1195166130},
       URL = {https://doi.org/10.2977/prims/1195166130},
}

@misc{murata2025affinehighestweightstructures,
      title={Affine highest weight structures on module categories over quiver Hecke algebras}, 
      author={Haruto Murata},
      year={2024},
      eprint={2412.12903},
      archivePrefix={arXiv},
      primaryClass={math.RT},
      url={https://arxiv.org/abs/2412.12903}, 
      note={arXiv:2412.12903},
}

@article {MR5045632,
    AUTHOR = {Kashiwara, Masaki and Kim, Myungho and Oh, Se-jin and Park,
              Euiyong},
     TITLE = {Reflection functors on quiver {H}ecke algebras},
   JOURNAL = {J. Algebra},
  FJOURNAL = {Journal of Algebra},
    VOLUME = {697},
      YEAR = {2026},
     PAGES = {170--242},
      ISSN = {0021-8693,1090-266X},
   MRCLASS = {18M05 (16D90 16G20 20C08 81R10)},
  MRNUMBER = {5045632},
       DOI = {10.1016/j.jalgebra.2026.02.012},
       URL = {https://doi-org.utokyo.idm.oclc.org/10.1016/j.jalgebra.2026.02.012},
}

@article {MR4609778,
    AUTHOR = {Hernandez, David},
     TITLE = {Representations of shifted quantum affine algebras},
   JOURNAL = {Int. Math. Res. Not. IMRN},
  FJOURNAL = {International Mathematics Research Notices. IMRN},
      YEAR = {2023},
    NUMBER = {13},
     PAGES = {11035--11126},
      ISSN = {1073-7928,1687-0247},
   MRCLASS = {17B37 (13F60)},
  MRNUMBER = {4609778},
MRREVIEWER = {Xin\ Tang},
       DOI = {10.1093/imrn/rnac149},
       URL = {https://doi.org/10.1093/imrn/rnac149},
}

@article {MR3376147,
    AUTHOR = {De Commer, Kenny and Neshveyev, Sergey},
     TITLE = {Quantum flag manifolds as quotients of degenerate quantized
              universal enveloping algebras},
   JOURNAL = {Transform. Groups},
  FJOURNAL = {Transformation Groups},
    VOLUME = {20},
      YEAR = {2015},
    NUMBER = {3},
     PAGES = {725--742},
      ISSN = {1083-4362,1531-586X},
   MRCLASS = {20G42 (14M15 17B37)},
  MRNUMBER = {3376147},
MRREVIEWER = {Christian\ Ohn},
       DOI = {10.1007/s00031-015-9324-y},
       URL = {https://doi-org.utokyo.idm.oclc.org/10.1007/s00031-015-9324-y},
}

@misc{hoshino2025semisimplemodulecategoriesfusion,
      title={Semisimple module categories with fusion rules of the compact full flag manifold type}, 
      author={Mao Hoshino},
      year={2025},
      eprint={2510.12057},
      archivePrefix={arXiv},
      primaryClass={math.QA},
      url={https://arxiv.org/abs/2510.12057}, 
      note={arXiv:2510.12057}
}

@misc{kwon2025infinitelevelfockspacescrystal,
      title={Infinite-level Fock spaces, crystal bases, and tensor product of extremal weight modules of type $A_{+\infty}$}, 
      author={Jae-Hoon Kwon and Soo-Hong Lee},
      year={2025},
      eprint={2501.07941},
      archivePrefix={arXiv},
      primaryClass={math.RT},
      url={https://arxiv.org/abs/2501.07941}, 
      note={arXiv:2501.07941}
}

@misc{murata2026quantumimaginaryschurweylduality,
      title={Quantum imaginary Schur-Weyl duality}, 
      author={Haruto Murata},
      year={2026},
      eprint={2607.01026},
      archivePrefix={arXiv},
      primaryClass={math.RT},
      url={https://arxiv.org/abs/2607.01026}, 
      note={arXiv:2607.01026}
}

@article {MR3709726,
    AUTHOR = {Webster, Ben},
     TITLE = {Knot invariants and higher representation theory},
   JOURNAL = {Mem. Amer. Math. Soc.},
  FJOURNAL = {Memoirs of the American Mathematical Society},
    VOLUME = {250},
      YEAR = {2017},
    NUMBER = {1191},
     PAGES = {v+141},
      ISSN = {0065-9266,1947-6221},
      ISBN = {978-1-4704-2650-7; 978-1-4704-4206-4},
   MRCLASS = {57M27 (17B10 18D05 57M25)},
  MRNUMBER = {3709726},
MRREVIEWER = {Stefan\ K.\ Friedl},
       DOI = {10.1090/memo/1191},
       URL = {https://doi-org.utokyo.idm.oclc.org/10.1090/memo/1191},
}

@article {MR2373155,
    AUTHOR = {Chuang, Joseph and Rouquier, Rapha\"{e}l},
     TITLE = {Derived equivalences for symmetric groups and {$\mathfrak{sl}_2$}-categorification},
   JOURNAL = {Ann. of Math. (2)},
  FJOURNAL = {Annals of Mathematics. Second Series},
    VOLUME = {167},
      YEAR = {2008},
    NUMBER = {1},
     PAGES = {245--298},
      ISSN = {0003-486X,1939-8980},
   MRCLASS = {20C08 (17B10 18E15 20C33)},
  MRNUMBER = {2373155},
MRREVIEWER = {Bogdan\ Ion},
       DOI = {10.4007/annals.2008.167.245},
       URL = {https://doi-org.utokyo.idm.oclc.org/10.4007/annals.2008.167.245},
}

@article {MR1035415,
    AUTHOR = {Lusztig, G.},
     TITLE = {Canonical bases arising from quantized enveloping algebras},
   JOURNAL = {J. Amer. Math. Soc.},
  FJOURNAL = {Journal of the American Mathematical Society},
    VOLUME = {3},
      YEAR = {1990},
    NUMBER = {2},
     PAGES = {447--498},
      ISSN = {0894-0347,1088-6834},
   MRCLASS = {17B35 (16A64)},
  MRNUMBER = {1035415},
MRREVIEWER = {Ya.\ S.\ So\u ibel\cprime man},
       DOI = {10.2307/1990961},
       URL = {https://doi.org/10.2307/1990961},
}

\printindex

\end{document}